\documentclass{amsbook}

\usepackage{comment}
\excludecomment{chapter-one}
\excludecomment{chapter-two}
\excludecomment{chapter-three}
\excludecomment{chapter-four}
\excludecomment{chapter-five}
\excludecomment{chapter-six}
\excludecomment{chapter-seven}
\excludecomment{chapter-eight}
\excludecomment{chapter-nine}
\excludecomment{chapter-ten}
\excludecomment{chapter-appendix}

\includecomment{chapter-one}
\includecomment{chapter-two}
\includecomment{chapter-three}
\includecomment{chapter-four}
\includecomment{chapter-five}
\includecomment{chapter-six}
\includecomment{chapter-seven}
\includecomment{chapter-appendix}

\usepackage[mathscr]{eucal}
\usepackage{amssymb}
\usepackage{pifont}
\usepackage{graphicx}
\usepackage{psfrag} 
\usepackage[usenames,dvipsnames]{color}
\usepackage[normalem]{ulem}
\usepackage{amsthm}
\usepackage{bbm}
\usepackage{enumerate}
\usepackage{array}
\usepackage{amsmath}

\usepackage{hyperref}
\hypersetup{colorlinks=true,linkcolor={Brown},citecolor={Brown},urlcolor={Brown}}

\usepackage{cleveref}

\numberwithin{section}{chapter}
\numberwithin{equation}{section}
\makeindex

\usepackage[all]{xy}
\CompileMatrices

\newdir{ >}{{}*!/-10pt/\dir{>}}

\swapnumbers 

\newtheorem{Thm}[equation]{Theorem}
\newtheorem{Prop}[equation]{Proposition}
\newtheorem{Lem}[equation]{Lemma}
\newtheorem{Cor}[equation]{Corollary}
\newtheorem{Sch}[equation]{Scholium}

\theoremstyle{remark}
\newtheorem{Rem}[equation]{Remark}
\newtheorem{Def}[equation]{Definition}
\newtheorem{Ter}[equation]{Terminology}
\newtheorem{Not}[equation]{Notation}
\newtheorem{Exa}[equation]{Example}
\newtheorem{Exas}[equation]{Examples}
\newtheorem{Cons}[equation]{Construction}
\newtheorem{Conv}[equation]{Convention}
\newtheorem{Hyp}[equation]{Hypotheses}

\newcommand{\nc}{\newcommand}
\nc{\dmo}{\DeclareMathOperator}

\dmo{\Ab}{Ab}
\dmo{\Aut}{Aut}
\dmo{\bicMack}{\biMack_{\mathsf{ic}}} 
\dmo{\biMack}{\mathsf{Mack}} 
\dmo{\Ch}{Ch}
\dmo{\CoInd}{CoInd}
\dmo{\Der}{D}
\dmo{\End}{End}
\dmo{\Fun}{\mathrm{Fun}} 
\dmo{\Hom}{Hom}
\dmo{\Ho}{Ho}
\dmo{\img}{im}
\dmo{\incl}{incl}
\dmo{\Ind}{Ind}
\dmo{\Inj}{Inj} 
\dmo{\Ker}{Ker}
\dmo{\Mackey}{Mack} 
\dmo{\Map}{Map}%
\dmo{\Mod}{Mod}
\dmo{\Mor}{Mor}%
\dmo{\Obj}{Obj}
\dmo{\Proj}{Proj} 
\dmo{\pr}{pr}
\dmo{\PsFunJJ}{\PsFun_{\JJ_!}^{\JJ^\prime\textrm{\!-}\mathsf{oplax}}}
\dmo{\PsFunJop}{\PsFun_{{{\JJ}_{{}_{*}}}}}
\dmo{\PsFunJ}{\PsFun_{\JJ_!}}
\dmo{\PsFunoplax}{\PsFun^{\mathsf{oplax}}}
\dmo{\PsFun}{\mathsf{PsFun}} 
\dmo{\Res}{Res}
\dmo{\SH}{SH}
\dmo{\Spanname}{{\sf Span}}
\dmo{\Stab}{Stab}
\dmo{\twoFun}{2\mathsf{Fun}}

\nc\noloc{\nobreak\mspace{6mu plus 1mu}{:}\nonscript\mkern-\thinmuskip\mathpunct{}\mspace{2mu}}
\nc{\ababs}{{\sl ab absurdo}}
\nc{\Add}{\mathsf{Add}}
\nc{\ADD}{\mathsf{ADD}}
\nc{\adhoc}{{\sl ad hoc}}
\nc{\adjto}{\rightleftarrows}
\nc{\adj}{\dashv\,}
\nc{\afortiori}{{\sl a fortiori}}
\nc{\aka}{{a.\,k.\,a.}\ }
\nc{\all}{\mathsf{all}}
\nc{\apriori}{{\sl a priori}}
\nc{\ass}{\mathrm{ass}} 
\nc{\bbA}{\mathbb{A}}
\nc{\bbB}{\mathbb{B}}
\nc{\bbC}{\mathbb{C}}
\nc{\bbD}{\mathbb{D}}
\nc{\bbF}{\mathbb{F}}
\nc{\bbI}{\mathbb{I}}
\nc{\bbM}{\mathbb{M}}
\nc{\bbN}{\mathbb{N}}
\nc{\bbP}{\mathbb{P}}
\nc{\bbQ}{\mathbb{Q}}
\nc{\bbR}{\mathbb{R}}
\nc{\bbZ}{\mathbb{Z}}
\nc{\bs}{\backslash}
\nc{\BurnG}{\cat{A}(G)}
\nc{\cat}[1]{\mathcal{#1}}
\nc{\Cat}{\mathsf{Cat}}
\nc{\CAT}{\mathsf{CAT}}
\nc{\cf}{{\sl cf.}\ }
\nc{\Cf}{{\sl Cf.}\ }
\nc{\colim}{\mathop{\mathrm{colim}}}
\nc{\costar}{**}
\nc{\co}{{\mathrm{co}}}
\nc{\DD}{\cat{D}}
\nc{\Displ}{\displaystyle}
\nc{\doublequot}[3]{#1\backslash #2/#3}
\nc{\Ecell}{\rotatebox[origin=c]{90}{$\Downarrow$}} 
\nc{\eg}{{\sl e.g.}\ } 
\nc{\Eg}{{\sl E.g.}\ } 
\nc{\eps}{\varepsilon}
\nc{\equalby}[1]{\overset{\textrm{#1}}{=}}
\nc{\exact}{\mathsf{ex}}
\nc{\faithful}{\mathsf{faithful}}
\nc{\faith}{\mathsf{faithf}}
\nc{\final}{\textrm{\scriptsize{\ding{93}}}} 
\nc{\Funadd}{\Fun_{\amalg}}
\nc{\Funplus}{\Fun_{+}}
\nc{\fun}{\mathrm{fun}} 
\nc{\GG}{\mathbb{G}}
\nc{\gpdG}{{\groupoidf_{\!\smallslash\!G}}} 
\nc{\gpd}{\groupoid}%
\nc{\gps}{\mathsf{groups}} 
\nc{\groconn}{\groupoid_{\mathsf{conn}}}
\nc{\groupoidf}{\groupoid{}^{\smallfaithful}}
\nc{\groupoid}{\mathsf{gpd}}
\nc{\group}{\mathsf{group}} 
\nc{\Gsets}{G\sset}
\nc{\HGfK}{\doublequot{H}{G}{f(K)}}%
\nc{\HGK}{\doublequot HGK}
\nc{\Homcat}[1]{\Hom_{\cat #1}}
\nc{\hooklongleftarrow}{\longleftarrow\joinrel\rhook}
\nc{\hooklongrightarrow}{\lhook\joinrel\longrightarrow}
\nc{\hook}{\hookrightarrow}
\nc{\Hsets}{H\mathsf{-sets}}
\nc{\ICAdd}{\Add_{\mathsf{ic}}}%
\nc{\ICADD}{\ADD_{\mathsf{ic}}}%
\nc{\Idcat}[1]{\Id_{\cat{#1}}}
\nc{\id}{\mathrm{id}}
\nc{\Id}{\mathrm{Id}}
\nc{\ie}{{\sl i.e.}\ }
\nc{\into}{\mathop{\rightarrowtail}}
\nc{\inv}{^{-1}}
\nc{\Iout}[1]{\Ivo{\sout{#1}}}
\nc{\isocell}[1]{\undersett{ #1}{\overset{\sim}{\Ecell}}} 
\nc{\Isocell}[1]{\undersett{ #1}{\overset{\sim}{\Longrightarrow}}}
\nc{\isoEcell}{\overset{\sim}{\Rightarrow}} 
\nc{\isotoo}{\stackrel{\sim}\longrightarrow}
\nc{\isoto}{\buildrel \sim\over\to}
\nc{\Ivo}[1]{{\color{OliveGreen}#1}}
\nc{\JJ}{\mathbb{J}}
\nc{\kk}{\Bbbk}
\nc{\KK}{\mathrm{KK}}
\nc{\leps}{{}^{\ell}\eps}
\nc{\leta}{{}^{\ell}\eta}
\nc{\loccit}{{\sl loc.\ cit.}}
\nc{\lotoo}[1]{\overset{#1}{\,\longleftarrow\,}}
\nc{\loto}[1]{\overset{#1}{\leftarrow}}
\nc{\lto}{\leftarrow}
\nc{\lun}{\mathrm{lun}} 
\nc{\Mackintro}[1]{(Mack\,\ref{Mack-#1-intro})}
\nc{\Mack}[1]{(Mack\,\ref{Mack-#1})}
\nc{\Mid}{\,\big|\,}
\nc{\MMod}{\,\text{-}\Mod}%
\nc{\MM}{\cat{M}}
\nc{\Muniv}{\cat{M}_{\mathsf{univ}}}
\nc{\Ncell}{\rotatebox[origin=c]{0}{$\Uparrow$}} 
\nc{\NEcell}{\rotatebox[origin=c]{135}{$\Downarrow$}} 
\nc{\NN}{\cat{N}}
\nc{\NWcell}{\rotatebox[origin=c]{-135}{$\Downarrow$}} 
\nc{\oEcell}[1]{\overset{\scriptstyle #1}{\Ecell}} 
\nc{\oWcell}[1]{\overset{\scriptstyle #1}{\Wcell}} 
\nc{\ointo}[1]{\overset{#1}{\rightarrowtail}}
\nc{\olto}[1]{\overset{#1}\lto}
\nc{\onto}{\mathop{\twoheadrightarrow}}
\nc{\op}{{\mathrm{op}}}
\nc{\otoo}[1]{\overset{#1}{\,\longrightarrow\,}}
\nc{\oto}[1]{\overset{#1}\to}
\nc{\Paul}[1]{{\color{Red}#1}}
\nc{\pih}[1]{\tau_{1}#1}%
\nc{\Pout}[1]{\Paul{\sout{#1}}}
\nc{\PsFunJindex}{\PsFun_{{\JJ_!}} \ \ {{\JJ}_{!}}\textrm{-strong pseudo-functors}}
\nc{\qquadtext}[1]{\qquad\textrm{#1}\qquad}
\nc{\quadtext}[1]{\quad\textrm{#1}\quad}
\nc{\ra}{\rightarrow}
\nc{\reps}{{}^{r\!}\eps}
\nc{\restr}[1]{{|_{\scriptstyle #1}}}
\nc{\reta}{{}^{r\!}\eta}
\nc{\run}{\mathrm{run}} 
\nc{\Sad}{\mathsf{Sad}}
\nc{\SAD}{\mathsf{SAD}}
\nc{\sbull}{{\scriptscriptstyle\bullet}}
\nc{\Scell}{\rotatebox[origin=c]{0}{$\Downarrow$}} 
\nc{\SEcell}{\rotatebox[origin=c]{45}{$\Downarrow$}} 
\nc{\SET}[2]{\big\{\,#1\Mid#2\,\big\}}
\nc{\set}{\mathsf{set}} 
\nc{\Set}{\mathsf{Set}}
\nc{\smallfaithful}{\mathsf{f}}
\nc{\smallslash}{{}^{\scriptscriptstyle/}}
\nc{\smat}[1]{\left(\begin{smallmatrix} #1 \end{smallmatrix}\right)}
\nc{\spanG}{{\widehat{\mathsf{gp}\,\,}\!\!\mathsf{d}}{}^\smallfaithful_{\!{}^{\scriptscriptstyle/}\!G}}
\nc{\Spanhat}{\textrm{\sf S}\widehat{\textrm{\sf pan}}} %
\nc{\Span}{\Spanname}
\nc{\sset}{\textrm{-}\set}
\nc{\str}{\mathsf{str}}
\nc{\SWcell}{\rotatebox[origin=c]{-45}{$\Downarrow$}} 
\nc{\too}{\mathop{\longrightarrow}\limits}
\nc{\tristars}{\begin{center} $ *\ *\ * $ \end{center}}
\nc{\tSpan}{\pih{\Spanname}}
\nc{\undersett}[1]{\underset{\scriptstyle #1}}
\nc{\un}{\mathrm{un}} 
\nc{\vcorrect}[1]{{\vphantom{\vbox to #1em{}}}}
\nc{\Wcell}{\rotatebox[origin=c]{90}{$\Uparrow$}} 
\nc{\what}[1]{\widehat{\cat{#1}}}
\nc{\xra}{\xrightarrow}

\nc{\xBur}{\mathrm{B^c}} 
\nc{\xBurk}{ \mathrm{B}^{\mathrm{c}}_{\kk} } 
\nc{\Bur}{\mathrm{B}} 
\nc{\Burk}{\Bur_{\kk}} 

\begin{document}
\pagenumbering{roman}


\title{Mackey 2-functors and Mackey 2-motives}
\author{Paul Balmer}
\author{Ivo Dell'Ambrogio}
\date{\today}

\address{\ \vfill\vfill\vfill\vfill
\noindent PB: UCLA Mathematics Department, Los Angeles, CA 90095-1555, USA}
\email{balmer@math.ucla.edu}
\urladdr{http://www.math.ucla.edu/$\sim$balmer}

\address{\break
\noindent ID:  Universit\'e de Lille, CNRS, UMR 8524 - Laboratoire Paul Painlev\'e, F-59000 Lille, France}
\email{ivo.dell-ambrogio@univ-lille.fr}
\urladdr{http://math.univ-lille1.fr/$\sim$dellambr}

\begin{abstract} \normalsize
We study collections of additive categories $\MM(G)$, indexed by finite groups~$G$ and related by induction and restriction in a way that categorifies usual Mackey functors. We call them `Mackey 2-functors'. We provide a large collection of examples in particular thanks to additive derivators. We prove the first properties of Mackey 2-functors, including separable monadicity of restriction to subgroups. We then isolate the initial such structure, leading to what we call `Mackey 2-motives'. We also exhibit a convenient calculus of morphisms in Mackey 2-motives, by means of string diagrams. Finally, we show that the 2-endomorphism ring of the identity of~$G$ in this 2-category of Mackey 2-motives is isomorphic to the so-called crossed Burnside ring of~$G$.
\end{abstract}

\thanks{First-named author partially supported by NSF grant~DMS-1600032.}
\thanks{Second-named author partially supported by Project ANR ChroK (ANR-16-CE40-0003) and Labex CEMPI (ANR-11-LABX-0007-01).}


\maketitle

\clearpage
\thispagestyle{empty}

\begin{center}
To our wives and daughters

-- Aleksandra, Anne, Chloe, Jeanne, Laura and Sophie --

for their love and support, and for patiently indulging us during months of groupoid-juggling and string-untangling.
\end{center}


\tableofcontents

%
\chapter*{Introduction}
\label{ch:introduction}%
\bigbreak

In order to study a given group~$G$, it is natural to look for mathematical objects on which $G$ acts by automorphisms. For instance, in ordinary representation theory, one considers vector spaces on which $G$ acts linearly. In topology, one might prefer topological spaces and continuous $G$-actions. In functional analysis, it might be operator algebras on which $G$ is expected to act. And so on, and so forth. Those `$G$-equivariant objects' usually assemble into a category, that we shall denote~$\MM(G)$. Constructing such categories $\MM(G)$ of $G$-equivariant objects in order to study the group~$G$ is a simple but powerful idea. It is used in all corners of what we shall loosely call `equivariant mathematics'.

In this work, we focus on \emph{finite groups}~$G$ and \emph{additive categories}~$\MM(G)$, \ie categories in which one can add objects and add morphisms. Although topological or analytical examples may not seem very additive at first sight, they can be included in our discussion by passing to stable categories. Thus, to name a few explicit examples of such categories~$\MM(G)$, let us mention categories of $\kk G$-modules $\MM(G)=\Mod(\kk G)$ or their derived categories $\MM(G)=\Der(\kk G)$ in classical representation theory over a field~$\kk$, homotopy categories of $G$-spectra $\MM(G)=\SH(G)$ in equivariant homotopy theory, and Kasparov categories~$\MM(G)=\KK(G)$ of \mbox{$G$-$C^*$\!-algebras} in noncommutative geometry. As the reader surely realizes at this point, the list of such examples is virtually endless: just let $G$ act wherever it can! In fact, the entire \Cref{ch:Examples} of this book is devoted to a review of examples.

Let us try to isolate the properties that such categories $\MM(G)$ have in common. First of all, it is clear that in all situations we can easily construct a similar category~$\MM(H)$ for any other group~$H$, in particular for subgroups~$H\le G$. The variance of $\MM(G)$ in the group~$G$, through restriction, induction, conjugation, etc, is the bread and butter of equivariant mathematics. It is then a natural question to axiomatize what it means to have a reasonable collection of additive categories~$\MM(G)$ indexed by finite groups~$G$, with all these links between them. In view of the ubiquity of such structures, it is somewhat surprising that such an axiomatic treatment did not appear earlier.

In fact, a lot of attention has been devoted to a similar but simpler structure, involving abelian groups instead of additive categories. These are the so-called Mackey functors. Let us quickly remind the reader of this standard notion, going back to work of Green~\cite{Green71} and Dress~\cite{Dress73} almost half a century ago.

An ordinary \emph{Mackey functor}~$M$ involves the data of abelian groups~$M(G)$ indexed by finite groups~$G$. These $M(G)$ come with restriction homomorphisms $R^G_H\colon M(G)\to M(H)$, induction or transfer homomorphisms $I^G_H\colon M(H)\to M(G)$, and conjugation homomorphisms $c_x\colon M(H)\to M({}^{x\!}H)$, for $H\le G$ and $x\in G$. This data is subject to a certain number of rules, most of them rather intuitive. Among them, the critical rule is the \emph{Mackey double-coset formula}, which says that for all $H,K\le G$ the following two homomorphisms $M(H)\to M(K)$ are equal:
\begin{equation}
\label{eq:old-Mackey-intro}%
\index{double-coset formula}%
R^G_K\circ I_H^G = \sum_{[x]\in K\bs G/H}I_{K\cap \,{}^{x\!}H}^K \circ c_x \circ R^H_{K^{x\,}\cap H}\,.
\end{equation}
These Mackey functors are quite useful in representation theory and equivariant homotopy theory. See Webb's survey~\cite{Webb00} or \Cref{app:old-Mackey}.

Let us return to our categories~$\MM(G)$ of `objects with $G$-actions'. In most examples, these $\MM(G)$ behave very much like ordinary Mackey functors, with the obvious difference that they involve additive categories~$\MM(G)$ instead of abelian groups~$M(G)$, and additive functors between them instead of $\bbZ$-linear homomorphisms. Actually, truth be told, the homomorphisms appearing in ordinary Mackey functors are often mere shadows of additive functors with the same name (restriction, induction, etc) existing at the level of underlying categories.

In other words, to axiomatize our categories~$\MM(G)$ and their variance in~$G$, we are going to \emph{categorify} the notion of ordinary Mackey functor. Our first, very modest, contribution is to propose a name for these categorified Mackey functors~$\MM$. We call them
\begin{center}
\emph{Mackey 2-functors}.
\end{center}
We emphasize that we do not pretend to `invent' Mackey 2-functors out of the blue. Examples of such structures have been around for a long time and are as ubiquitous as equivariant mathematics itself. So far, the only novelty is the snazzy name.

Our first serious task will consist in pinning down the precise definition of Mackey 2-functor. But without confronting the devil in the detail quite yet, the heuristic idea should hopefully be clear from the above discussion. In first approximation, a Mackey 2-functor~$\MM$ consists of the data of an additive category~$\MM(G)$ for each finite group~$G$, together with further structure like restriction and induction functors, and subject to a Mackey formula at the categorical level. An important aspect of our definition is that we shall want $\MM$ to satisfy
\begin{center}
\emph{ambidexterity}.
\index{ambidexterity}
\end{center}
This means that induction is both left and right adjoint to restriction: For each subgroup $H\le G$, the restriction functor $\MM(G)\to \MM(H)$ admits a two-sided adjoint. In pedantic parlance, induction and `co-induction' coincide in~$\MM$.

Once we start considering adjunctions, we inherently enter a 2-categorical world. We not only have categories $\MM(G)$ and functors to take into account (0-layer and 1-layer) but we also have to handle natural transformations of functors (2-layer), at the very least for the units and counits of adjunctions. Similarly, our version of the Mackey formula will not involve an \emph{equality} between homomorphisms as in~\eqref{eq:old-Mackey-intro} but an \emph{isomorphism} between functors. This 2-categorical information is essential, and it distinguishes our Mackey 2-functors from a more naive notion of `Mackey functor with values in the category of additive categories' (which would miss the adjunction between $R^G_H$ and $I_H^G$ for instance). This important 2-layer in the structure of a Mackey 2-functor also explains our choice of the name. Still, the reader who is not versed in the refinements of 2-category theory should not throw the towel in despair. Most of this book can be understood by keeping in mind the usual 2-category $\CAT$ of categories, functors and natural transformations.
\tristars

In a nutshell, the purpose of this work is to
\begin{itemize}
\item
lay the foundations of the theory of Mackey 2-functors
\item
justify this notion by a large catalogue of examples
\item
provide some first applications, and
\item
construct a `motivic' approach.
\end{itemize}

\goodbreak

Let us now say a few words of these four aspects, while simultaneously outlining the structure of the book. After the present gentle introduction, \Cref{ch:expanded-introduction} will provide an expanded introduction with more technical details.

\tristars

The first serious issue is to give a solid definition of Mackey 2-functor that simultaneously can be checked in examples and yet provides enough structure to prove theorems. This balancing act relies here on three components:
\begin{enumerate}[(1)]
\item
A `light' definition of Mackey 2-functor, to be found in \Cref{Def:Mackey-2-functor-intro}. It involves four axioms \Mackintro{1}--\Mackintro{4} that the data $G\mapsto \MM(G)$ should satisfy. These four axioms are reasonably easy to verify in examples. Arguably the most important one, \Mackintro{4}, states that $\MM$ satisfies ambidexterity.
\smallbreak
\item
A `heavier' notion of \emph{rectified} Mackey 2-functor, involving another six axioms \Mack{5}--\Mack{10}. Taken together, those ten axioms make it possible to reliably prove theorems about (rectified) Mackey 2-functors. However, some of these six extra axioms can be unpleasant to verify in examples.
\smallbreak
\item
A Rectification \Cref{Thm:rectification-intro}, which roughly says that there is always a way to modify the 2-layer of any Mackey 2-functor $G\mapsto \MM(G)$ satisfying~\Mackintro{1}--\Mackintro{4} so that the additional axioms \Mack{5}--\Mack{10} are satisfied as well. In particular, one does not have to verify \Mack{5}--\Mack{10} in examples.
\end{enumerate}

An
introduction to the precise definition of Mackey 2-functor is to be found in \Cref{sec:Mackey-2-functors}. The full treatment appears in \Cref{ch:2-Mackey}. The motivation for the idea of rectification is given in \Cref{sec:rectification-intro}, with details in \Cref{ch:Theta}.
%

A first application follows immediately from the Rectification Theorem, namely we prove that for any subgroup $H\le G$, the category $\MM(H)$ is a \emph{separable extension} of~$\MM(G)$. This result provides a unification and a generalization of a string of results brought to light in~\cite{Balmer15} and~\cite{BalmerDellAmbrogioSanders15}, where we proved separability by an \adhoc\ argument in each special case. In the very short \Cref{sec:monadicity}, we give a uniform proof that all (rectified) Mackey 2-functors~$\MM$ automatically satisfy this separability property. Conceptually, the problem is the following. How can we `carve out' the category $\MM(H)$ of $H$-equivariant objects over a subgroup from the category $\MM(G)$ of $G$-equivariant objects over the larger group? The most naive guess would be to do the `carving out' via localization. This basically never works, $\MM(H)$ is almost never a localization of~$\MM(G)$, but separable extensions are the next best thing. Considering separable extensions instead of localizations is formally analogous to considering the \'etale topology instead of the Zariski topology in algebraic geometry. See further commentary on the meaning and relevance of separability in \Cref{sec:monadicity-intro}.

We return to the topic of applications below, when we comment on motives. For now, let us address the related question of examples. We discuss this point at some length because we consider the plethora of examples to be a great positive feature of the theory. Also, the motivic approach that we discuss next is truly justified by this very fact that Mackey 2-functors come in all shapes and forms.

It should already be intuitively clear from our opening paragraphs that Mackey 2-functors pullulate throughout equivariant mathematics. In any case, beyond this gut feeling that they should exist in many settings, a reliable source of rigorous examples of Mackey 2-functors can be found in the theory of \emph{Grothendieck derivators} (see Groth~\cite{Groth13}). Our Ambidexterity \Cref{Thm:ambidex-der} says that the restriction of an \emph{additive} derivator to finite groups automatically satisfies the ambidexterity property making it a Mackey 2-functor. This result explains why it is so common in practice that induction and co-induction coincide in additive settings. It also provides a wealth of examples of Mackey 2-functors in different subjects. Let us emphasize this point: The theory of derivators itself covers a broad variety of backgrounds, in algebra, topology, geometry, etc. Furthermore, derivators can always be stabilized (see \cite{Heller97} and~\cite{Coley19}) and stable derivators are always additive. In other words, via derivators, that is, via general homotopy theory, we gain a massive collection of readily available examples of Mackey 2-functors from algebra, topology, geometry, etc. In particular, any stable Quillen model category~$\mathcal{Q}$ provides a Mackey 2-functor $G\mapsto \MM(G):=\Ho(\mathcal{Q}^G)$, via diagram categories. In the five decades since~\cite{Quillen67}, examples of Quillen model categories have been discovered in all corners of mathematics, see for instance Hovey~\cite{Hovey99} or~\cite{HoveyPalmieriStrickland97}. An expanded introduction to these ideas can be found in \Cref{sec:Mackey-vs-derivators}.

But there is even more! In Sections~\ref{sec:sub-quotients}-\ref{sec:equivobj}, we provide further methods to handle trickier examples of Mackey 2-functors which cannot be obtained directly from a derivator. For instance, stable module categories (\Cref{Prop:Mackeycat_quot}) in modular representation theory or genuine $G$-equivariant stable homotopy categories (\Cref{Exa:SH(G)}) can be shown not to come from the restriction of a derivator to finite groups. Yet they are central examples of Mackey 2-functors and we explain how to prove this in \Cref{ch:Examples}.
\tristars

Let us now say a word of the motivic approach, which is our most ambitious goal. It will occupy the lion's share of this work, namely \Cref{ch:bicat-spans,ch:2-motives,ch:additive-motives}. We now discuss these ideas for readers with limited previous exposure to motives. A more technical introduction can be found in \Cref{sec:Mackey-2-motives-intro}.

In algebraic geometry, Grothendieck's \emph{motives} encapsulate the common themes recurring throughout a broad range of `Weil' cohomology theories. These cohomology theories are defined on algebraic varieties (\eg~on smooth projective varieties), take values in all sorts of different abelian categories, and are described axiomatically.
Instead of algebraic varieties, we consider here finite groups. Instead of Weil cohomology theories, we consider of course Mackey 2-functors.

The motivic program seeks to construct an initial structure through which all other instances of the same sort of structure will factor. These ideas led Grothendieck to the plain 1-category of (pure) motives in algebraic geometry. Because of our added 2-categorical layer, the same philosophy naturally leads us to a
\begin{center}
\emph{2-category of Mackey 2-motives}.
\end{center}

The key feature of this 2-category is that every single Mackey 2-functor out there factors uniquely via Mackey 2-motives. The proof of this non-trivial fact is another application of the Rectification Theorem, together with some new constructions. Since we hammered the point that Mackey 2-functors are not mere figments of our imagination but very common structures, this factorization result applies broadly to many situations pre-dating our theory.
%

Perhaps this is a good place to further comment in non-specialized terms on the virtues of the motivic approach, beginning with algebraic geometry. The fundamental idea is of course the following. Since every Weil cohomology theory factors canonically via the category of motives, each result that can be established motivically will have a realization, an avatar, in every single example. Among the most successful such results are the so-called `motivic decompositions'. In the motivic category, some varieties~$X$ decompose as a direct sum of other simpler motives. As a corollary, every single Weil cohomology theory evaluated at~$X$ will decompose into simpler pieces accordingly. The motivic decomposition happens entirely within the `abstract' motivic world but the application happens wherever the Weil cohomology takes its values. And since Weil cohomology theories come in all shapes and forms, this type of result is truly powerful.

Let us see how this transposes to Mackey 2-motives. The overall pattern is the same. Whenever we find a motivic decomposition of the 2-motive of a given finite group~$G$, we know in advance that \emph{every single} Mackey 2-functor $\MM(G)$ evaluated at that group will decompose into smaller pieces accordingly. Because of the additional 2-layer, things happen `one level up', namely we decompose the identity 1-cell of~$G$, which really amounts to decomposing the 2-motive~$G$ up to an equivalence (see the `block decompositions' of~\ref{sec:additive-bicats}). Again, the range of applications is as broad as the list of examples of Mackey 2-functors.

In order to obtain concrete motivic decompositions, one needs to compute some endomorphism rings in the 2-category of Mackey 2-motives, more precisely the ring of 2-endomorphisms of the identity 1-cell~$\Id_G$ of the Mackey 2-motive of~$G$. Every decomposition of those rings, \ie any splitting of the unit into sum of idempotents, will produce decompositions of the categories $\MM(G)$ into `blocks' corresponding to those idempotents.

In this direction, we prove in \Cref{ch:additive-motives} that the above 2-endomorphism ring of~$\Id_G$ is isomorphic to a ring already known to representation theorists, namely the so-called \emph{crossed Burnside ring} of~$G$ introduced by Yoshida~\cite{Yoshida97}. See also Oda-Yoshida~\cite{OdaYoshida01} or Bouc~\cite{Bouc03}. The blas\'e reader should pause and appreciate the little miracle: A ring that we define through an a priori very abstract motivic construction turns out to be a ring with a relatively simple description, already known to representation theorists. It follows from this computation that every decomposition of the crossed Burnside ring yields a block decomposition of the Mackey 2-motive of~$G$ and therefore of \emph{every} Mackey 2-functor evaluated at~$G$, in every single example known today or to be discovered in the future.

\tristars
\smallbreak

This concludes the informal outline of this book. In addition to the seven main chapters mentioned above, we include two appendices. \Cref{app:categorical-reminders} collects all categorical prerequisites whereas \Cref{app:old-Mackey} is dedicated to ordinary Mackey functors. We also draw the reader's attention to the extensive index at the very end, that will hopefully show useful in navigating the text.

\smallbreak

A comparison with existing literature can be found in \Cref{sec:literature}, after we introduce some relevant terminology in \Cref{ch:expanded-introduction}.

\medbreak
\subsection*{Acknowledgements:}

We thank Serge Bouc, Yonatan Harpaz, Ioannis Lagkas, Akhil Mathew, Hiroyuki Nakaoka, Beren Sanders, Stefan Schwede and Alexis Virelizier, for many motivating discussions and for technical assistance.

For the final version of this work, we are especially thankful to Serge Bouc, who recognized in an earlier draft a ring that was known to specialists as the crossed Burnside ring. Our revised \Cref{ch:additive-motives} owes a lot to Serge's insight and to his generosity.

We are also grateful to an anonymous referee for their careful reading and helpful suggestions.

%
\chapter{Survey of results}
\label{ch:expanded-introduction}%
\pagenumbering{arabic}
\bigbreak
\begin{chapter-one}

This chapter is a more precise introduction to the ideas contained in this book.

\bigbreak
\section{The definition of Mackey 2-functors}
\label{sec:Mackey-2-functors}
\medskip

Our first task is to clarify the notion of \emph{Mackey 2-functor}. As discussed in the Introduction, Mackey 2-functors are supposed to axiomatize the assignment $G\mapsto \MM(G)$ of additive categories to finite groups, in a way that categorifies ordinary Mackey functors and captures the examples arising in Nature.

In fact, not only are 2-categories the natural framework for the output of a Mackey 2-functor $G\mapsto \MM(G)$, the input of~$\MM$ is truly 2-categorical as well. Indeed, the class of finite groups is advantageously replaced by the class of finite \emph{groupoids}. This apparently modest generalization not only harmonizes the input and output of our Mackey 2-functor~$G\mapsto\MM(G)$ but also distinguishes the two roles played by conjugation with respect to an element $x$ of a group~$G$, either as a plain group homomorphism ${}^x(-)\colon H\isoto {}^{x\!} H$ (at the 1-level) or as a relation ${}^{x\!} f_1=f_2$ between parallel group homomorphisms $f_1,f_2\colon H\to G$ (at the 2-level). Furthermore, the 2-categorical approach allows for much cleaner Mackey formulas, in the form of Beck-Chevalley base-change formulas. The classical Mackey formula in the form of the `double-coset formula'~\eqref{eq:old-Mackey-intro} involves non-canonical choices of representatives in double-cosets. Such a non-natural concept cannot hold up very long in 2-categories, where equalities are replaced by isomorphisms which would then also depend on these choices. Just from the authors' personal experience, the reader may want to consult~\cite{DellAmbrogio14} and~\cite{Balmer15} for a glimpse of the difficulties that quickly arise when trying to keep track of such choices. This further motivates us to use the cleaner approach via groupoids.

\begin{Not}
\label{Not:groupoid}%
\index{$gpd$@$\gpd$ \, 2-category of finite groupoids}
\index{groupoid}%
\index{gpd@$\groupoid$}
Of central use in this work is the 2-category
\[
\groupoid=\{\textrm{finite groupoids, functors, natural transformations}\}
\]
of finite groupoids, \ie categories with finitely many objects and morphisms, in which all morphisms are invertible. The 2-category $\groupoid$ is a 1-full and 2-full 2-subcategory of the 2-category of small categories~$\Cat$. Note that every 2-morphism in~$\groupoid$ is invertible, that is, $\groupoid$ is a (2,1)-category (\Cref{Def:2-1-category}).

We denote the objects of~$\groupoid$ by the same letters we typically use for groups, namely~$G$, $H$, etc. The role played by subgroups $H\le G$ in groups is now taken over by \emph{faithful} functors $H\into G$ in groupoids. In view of its importance for our discussion, we fix a notation ($\into$) to indicate faithfulness.
\index{$^^^$@$\into$ \, faithful functor or 1-cell}%
\end{Not}

\begin{Rem}
\label{Rem:group(oid)}%
\index{connected groupoid} \index{groupoid!connected --}%
There is essentially no difference between a group and a groupoid with one object. Therefore \emph{we identify each finite group~$G$ with the associated one-object groupoid with morphism group~$G$} and still denote it by~$G$ in~$\groupoid$. Accordingly, there is no difference between group homomorphisms $f\colon G\to G'$ and the associated 1-morphisms of one-object groupoids, and we denote them by the same symbol $f\colon G\to G'$. In that case, the functor $f\colon G\into G'$ is faithful if and only if the group homomorphism~$f$ is injective. The 2-morphisms $f_1\Rightarrow f_2$ between such functors $f_1,f_2\colon G\to G'$ in~$\groupoid$ are given at the group level by elements $x\in G'$ of the target group which conjugate one homomorphism into the other, ${}^{x\!}f_1=f_2$, that is, $x\,f_1(g)x\inv=f_2(g)$ for all~$g\in G$.

A groupoid is equivalent to a group if and only if it is \emph{connected}, meaning that every two of its objects are isomorphic.
\end{Rem}

\begin{Rem}
\label{Rem:iso-comma-intro}%
\index{iso-comma of groupoids}%
Given two morphisms of groupoids $i\colon H\to G$ and $u\colon K\to G$, with same target, we have the \emph{iso-comma} groupoid~$(i/u)$ whose objects are
\[
\Obj(i/u)=\SET{(x,y,g)}{x\in \Obj(H),\,y\in \Obj(K),\ g\colon i(x)\isoto u(y)\textrm{ in }G}
\]
with component-wise morphisms on the~$x$ and~$y$ parts (in~$H$ and~$K$) compatible with the isomorphisms~$g$ (in~$G$). See \Cref{sec:comma}. This groupoid~$(i/u)$ fits in a 2-cell
\begin{equation}
\label{eq:iso-comma-intro}%
\vcenter{\xymatrix@C=14pt@R=14pt{
& (i/u) \ar[dl]_-{p} \ar[dr]^-{q}
 \ar@{}[dd]|(.5){\isocell{\gamma}}
\\
H \ar[dr]_-{i}
&& K \ar[dl]^-{u}
\\
&G
}}
\end{equation}
where $p\colon (i/u)\to H$ and $q\colon (i/u)\to K$ are the obvious projections and $\gamma\colon i\,p\isoEcell u\,q$ is the isomorphism given at each object~$(x,y,g)$ of~$(i/u)$ by the third component~$g$. It is easy to check that if $i\colon H\into G$ is faithful then so is~$q\colon (i/u)\into K$.

Iso-comma squares like~\eqref{eq:iso-comma-intro}, and those equivalent to them, provide a refined version of pullbacks in the world of groupoids and will play a critical role throughout the work. \Cref{sec:comma} is dedicated to their study. For instance, when $G$ is a group and $H$ and $K$ are subgroups then the groupoid $(i/u)$ is equivalent to a coproduct of groups $K\cap {}^{x\!}H$, as in the double-coset formula. See details in \Cref{Rem:old-Mackey}.
\end{Rem}

We are going to consider 2-functors $\MM\colon \groupoid^{\op}\to \ADD$, contravariant on 1-cells (hence the `op'), defined on finite groupoids and taking values in the 2-category $\ADD$ of \emph{additive} categories and additive functors. Details about additivity are provided in \Cref{sec:additive-sedative} and~\ref{sec:additive-bicats}. For simplicity we apply the following customary rule:
\begin{Conv}
\label{Conv:ADD}%
Unless explicitly stated, every functor between additive categories is assumed to be additive (\Cref{Def:additive_fun}).
\end{Conv}

\begin{Rem}
\label{Rem:pre-2-Mackey}%
A \emph{2-functor} $\MM\colon \groupoid^{\op}\to \ADD$ is here always understood in the strict sense (see \Cref{Ter:pseudofun}) although we will occasionally repeat `\emph{strict}' 2-functor as a reminder to the reader and in contrast to pseudo-functors. So, such an $\MM$ consists of the following data:
\index{functor@2-functor $\groupoid^{\op}\to \ADD$}
\begin{enumerate}[\rm(a)]
\item
for every finite groupoid~$G$, an additive category~$\MM(G)$,
\smallbreak
\item
for every functor $u\colon H\to G$ in~$\groupoid$, a `restriction' functor $u^*\colon\MM(G)\to \MM(H)$,
\smallbreak
\item
for every natural transformation $\alpha\colon u\Rightarrow u'$ between two parallel functors $u,u'\colon H\to G$, a natural transformation\,(\footnote{\,Like in the theory of (pre)derivators, the variance of $\MM$ on 2-morphisms is a matter of convention and could be chosen opposite since $\groupoid^\co\cong \groupoid$ via $G\mapsto G^{\op}$. (The superscript `$\co$' on a 2-category denotes the formal reversal of \emph{2}-cells.)}) $\alpha^*\colon u^*\Rightarrow (u')^*$,
\end{enumerate}
subject to the obvious compatibilities with identities and compositions on the nose (hence the word `strict'). In particular $(uv)^*=v^*u^*$ and $(\alpha\beta)^*=\alpha^*\beta^*$.
\end{Rem}

\smallbreak
With this preparation, we can give our central definition.
\begin{Def}
\label{Def:Mackey-2-functor-intro}%
\index{Mackey 2-functor}%
A \emph{(global) Mackey 2-functor} is a strict 2-functor (\Cref{Rem:pre-2-Mackey})
\[
\MM\colon \groupoid^{\op}\to \ADD
\]
from finite groupoids to additive categories which satisfies the following axioms:
\begin{enumerate}[\rm({Mack}\,1)]
\item
\label{Mack-1-intro}%
\emph{Additivity}: For every finite family $\{G_c\}_{c\in C}$ in~$\groupoid$, the natural functor
\[
\big(\incl_c^*)_{{}_{c\in C}}\colon\MM\big(\,\coprod_{d\in C} G_d\,\big) \ \too\ \prod_{c\in C}\MM(G_c)
\]
is an equivalence, where $\incl_c\colon G_c\into \coprod_{d}G_d$ is the inclusion for all~$c\in C$.
\smallbreak
\item
\label{Mack-2-intro}%
\emph{Induction and coinduction}: For every faithful functor $i\colon H\into G$, the restriction functor $i^*\colon \MM(G)\to \MM(H)$ admits a left adjoint~$i_!$ and a right adjoint~$i_*$:
\[
\vcenter{\xymatrix@R=2em{
\MM(G) \ar[d]|(.48){\,i^*}
\\
\MM(H) \ar@/^1em/@<1em>[u]^-{i_!}_-{\ \adj} \ar@/_1em/@<-1em>[u]_-{i_*}^-{\adj\ }
}}
\]
\smallbreak
\item
\label{Mack-3-intro}%
\index{Mackey formula} \index{base-change} \index{Beck-Chevalley formula}%
\emph{Base-change formulas}: For every iso-comma square of finite groupoids as in~\eqref{eq:iso-comma-intro} in which~$i$ and (therefore) $q$ are faithful
\[
\vcenter{\xymatrix@C=14pt@R=14pt{
& (i/u) \ar[dl]_-{p} \ar@{ >->}[dr]^-{q}
 \ar@{}[dd]|(.5){\isocell{\gamma}}
\\
H \ar@{ >->}[dr]_-{i}
&& K \ar[dl]^-{u}
\\
&G
}}
\]
we have two isomorphisms
\[
q_!\circ p^* \Isocell{\gamma_!} u^*\circ i_!
\qquadtext{and}
u^*\circ i_* \Isocell{(\gamma\inv)_*} q_*\circ p^*
\]
given by the left mate $\gamma_!$ of $\gamma^*\colon p^*i^*\Rightarrow q^*u^*$ and the right mate $(\gamma\inv)_*$ of $(\gamma\inv)^*\colon q^*u^*\Rightarrow p^*i^*$. See \Cref{sec:mates} for details about mates.
\smallbreak
\item
\label{Mack-4-intro}%
\index{ambidexterity}%
\emph{Ambidexterity}: For every faithful~$i$, there exists an isomorphism
\[
i_!\simeq i_*
\]
between some (hence any) left and right adjoints of $i^*$ given in~\Mackintro{2}.
\end{enumerate}
\end{Def}

\begin{Rem} \label{Rem:dual-Mackey}
The axioms are self-dual in the sense that, if $\MM$ is a Mackey 2-functor, then there is a Mackey 2-functor $\MM^\op$ defined by $\MM^\op(G):=\MM(G)^\op$. This has the effect of exchanging the roles of the left and right adjunctions $i_!\dashv i^*$ and $i^*\dashv i_*$.
\end{Rem}

\begin{Exas}
\label{Exa:Mackey-2-functors}%
Mackey 2-functors abound in Nature. They include:
\begin{enumerate}[(a)]
\item
Usual $\kk$-linear representations $\MM(G)=\Mod(\kk G)$. See \Cref{Exa:lin_reps}.
\smallbreak
\item
Their derived categories $\MM(G)=\Der(\kk G)$. See \Cref{Exa:derived-cat}.
\smallbreak
\item
\label{it:Exa-Mackey-Stab}%
Stable module categories $\MM(G)=\Stab(\kk G)$. See \Cref{Exa:stable-module-cat}. (In this case, we shall restrict attention to a sub-2-category of groupoids by allowing only faithful functors as 1-cells.)
\smallbreak
\item
Equivariant stable homotopy categories $\MM(G)=\SH(G)$. See \Cref{Exa:SH(G)}.
\smallbreak
\item
Equivariant Kasparov categories $\MM(G)=\KK(G)$. See \Cref{Exa:KKetc}.
\smallbreak
\item
Abelian categories of ordinary Mackey functors $\MM(G)=\Mackey_\kk(G)$. See \Cref{Cor:Mackey-functors-abelian-example}.
\item
Abelian and derived categories of equivariant sheaves over a locally ringed space with $G$-action. See Examples~\ref{Exa:LRS}-\ref{Exa:Der}. (In this case, we shall restrict attention to a suitable comma 2-category of groupoids faithfully embedded in~$G$.)
\end{enumerate}
\end{Exas}

\begin{Rem}
\label{Rem:Mackey-def}%
Let us comment on \Cref{Def:Mackey-2-functor-intro}.
\begin{enumerate}[(a)]
\item
\label{it:add}%
In \Cref{Def:Mackey-2-functor}, we shall generalize the above definition by allowing the input of~$\MM$ to consist only of a specified 2-subcategory of groupoids. The necessity for this flexibility already appears in \Cref{Exa:Mackey-2-functors}\,\eqref{it:Exa-Mackey-Stab} above. Later we will even consider more abstract 2-categories as input for~$\MM$ (\Cref{Hyp:G_and_I_for_Span}). The above \Cref{Def:Mackey-2-functor-intro} is the `global' version of Mackey 2-functor~$\MM$ where $\MM(G)$ is defined for all groupoids~$G$, and $u^*$ for all functors~$u$.
\smallbreak
\item
\label{it:groups}
The first axiom \Mackintro{1} is straightforward. Every finite groupoid is equivalent to the finite coproduct of its connected components, themselves equivalent to one-object groupoids (\ie groups). Thus \Mackintro{1} allows us to think of Mackey 2-functors~$\MM$ as essentially defined on finite groups. More on this in \Cref{sec:more-examples}.
\smallbreak
\item
\label{it:civilization}%
Just like the other axioms, the second and fourth ones are \emph{properties} of the 2-functor~$\MM$. The adjoints $i_!$ and $i_*$, and later the isomorphism $i_!\simeq i_*$, are \emph{not part of the structure} of a Mackey 2-functor. In particular all the units and counits involved in the adjunctions $i_!\adj i^* \adj i_*$ could be rather wild, at least in the above primeval formulation. All four axioms are stated in a way that is independent of the actual choices of left and right adjoints and associated units and counits: If the axioms hold for one such choice, they will hold for all choices. We shall spend some energy on making better choices than others, in order to establish civilized formulas. This is the topic of `rectification' discussed in \Cref{sec:rectification-intro}.
\smallbreak
\item
\index{BC-property}%
The third axiom is a standard Base-Change condition of Beck-Chevalley type (referred to as `BC-property' in any case). In the iso-comma square~\eqref{eq:iso-comma-intro}
\[
\vcenter{\xymatrix@C=10pt@R=10pt{
& (i/u) \ar[dl]_-{p} \ar@{ >->}[dr]^-{q}
 \ar@{}[dd]|(.5){\isocell{\gamma}}
\\
H \ar@{ >->}[dr]_-{i}
&& K \ar[dl]^-{u}
\\
&G
}}
\]
induction along~$i$ followed by restriction along~$u$ can equivalently be computed as first doing restriction along $p$ followed by induction along~$q$. The latter composition passes via the groupoid~$(i/u)$ which is typically a disjoint union of `smaller' groupoids, as in the double-coset formula (see \Cref{Rem:old-Mackey}). Of course the dual axiom $u^*\, i_*\stackrel{\sim}{\Rightarrow} q_*\,p^*$ should more naturally involve the dual iso-comma $(i\bs u)$. However, in groupoids we have a canonical \emph{isomorphism} $(i\bs u)\cong (i/u)$ when the latter is equipped with the 2-cell~$\gamma\inv$. This explains our simplified formulation with $(\gamma\inv)_*\colon u^*\, i_*\stackrel{\sim}{\Rightarrow} q_*\,p^*$ and no mention of~$(i\bs u)$.
\smallbreak
\item
\label{it:i_!=i^*}%
The fourth axiom is a standard property of many 2-functors from groups to additive categories: induction and co-induction coincide. Any ambidexterity isomorphism $i_!\simeq i_*$ can be used to equip $i_!$ with the units and counits of $i^*\adj i_*$, thus making the left adjoint $i_!$ a \emph{right} adjoint as well. So we can equivalently assume the existence of a single two-sided adjoint $i_!=i_*$ of~$i^*$. This simplification will be useful eventually but at first it can also be confusing. In most examples, Nature provides us with canonical left adjoints~$i_!$ and canonical right adjoints~$i_*$, for instance by means of (derived) Kan extensions. Such adjoints are built differently on the two sides and happen to be isomorphic in the equivariant setting. We shall give in \Cref{ch:Theta} a mathematical explanation of why this phenomenon is so common.
\end{enumerate}
\end{Rem}

\bigbreak
\section{Rectification}
\label{sec:rectification-intro}
\medskip

Following up on Remark~\ref{Rem:Mackey-def}\,\eqref{it:civilization}, we emphasize the slightly naive nature of the ambidexterity axiom~\Mackintro{4}. As stated, this axiom is easy to verify in examples as it only requires some completely \adhoc\ isomorphism $i_!\simeq i_*$ for each faithful~$i\colon H\into G$, with no reference to the fact that $i\mapsto i_!$ and $i\mapsto i_*$ are canonically pseudo-functorial. Standard adjunction theory (\Cref{Rem:pseudo-func-of-adjoints}) tells us that every 2-cell $\alpha\colon i\Rightarrow i'$ will yield $\alpha_!\colon i'_!\Rightarrow i_!$ and $\alpha_*\colon i'_*\Rightarrow i_*$. Furthermore, every composable $j\colon K\into H$, $i\colon H\into G$  will yield isomorphisms $(ij)_!\cong i_!j_!$ and $(ij)_*\cong i_*j_*$. It is then legitimate to ask whether the isomorphism $i_!\simeq i_*$ can be `rectified' so as to be compatible with all of the above.

Similarly, following up on Remark~\ref{Rem:Mackey-def}\,\eqref{it:i_!=i^*}, let us say we choose a single two-sided adjoint $i_!=i_*$ for all faithful~$i\colon H\into G$.
In particular, in an iso-comma~\eqref{eq:iso-comma-intro}
\[
\vcenter{\xymatrix@C=10pt@R=10pt{
& (i/u) \ar[dl]_-{p} \ar@{ >->}[dr]^-{q}
 \ar@{}[dd]|(.5){\isocell{\gamma}}
\\
H \ar@{ >->}[dr]_-{i}
&& K \ar[dl]^-{u}
\\
&G
}}
\]
we not only have $i_!=i_*$ but also $q_!=q_*$. So we can write $u^*i_!$ as $u^*i_*$ and similarly $q_!p^*$ as~$q_*p^*$.
Then the BC-formulas \Mackintro{3} provide \emph{two} ways of comparing the `bottom' composition $u^*i_!=u^*i_*$ with the `top' composition $q_!p^*=q_*p^*$, one via~$\gamma_!$ and one via~$(\gamma\inv)_*$. It is again legitimate to wonder whether they agree.

The solution to these questions appears in \Cref{ch:Theta}, where we reach two goals. First, we show how to prove ambidexterity by induction on the order of the finite groupoids (\Cref{Prop:induction-on-G}); this will be an essential part of the Ambidexterity \Cref{Thm:ambidex-der}. Second, assuming that ambidexterity holds even only in the weak sense of \Cref{Def:Mackey-2-functor-intro}, we show that it must then hold \emph{for a good reason}: There exists a \emph{canonical} isomorphism between induction and coinduction satisfying several extra properties (\eg it is a pseudo-natural transformation as in \Cref{Ter:Hom_bicats}). This Rectification \Cref{Thm:rectification} yields several improvements to the notion of Mackey 2-functor, like a `strict' Mackey formula~\Mack{7}, the agreement of the pseudo-functorialities of induction and coinduction as discussed above, a `special Frobenius' property, etc. We also provide, \textsl{en passant}, some less important but convenient normalization of the values of the units and counits of the adjunctions $i_!\adj i^* \adj i_*$ in connection with additivity, and in `trivial' cases. Here is the full statement:

\begin{Thm}[Rectification Theorem; see \Cref{Thm:rectification}]
\label{Thm:rectification-intro}%
Consider a Mackey 2-functor $\MM\colon \groupoid^\op\to \ADD$ as in \Cref{Def:Mackey-2-functor-intro}. Then there is for each faithful $i\colon H\into G$ in~$\groupoid$ a unique choice (up to unique isomorphism) of a two-sided adjoint
\[
i_!=i_*\colon \MM(H)\to \MM(G)
\]
of restriction~$i^*\colon \MM(G)\to \MM(H)$ and units and counits
\[
\leta\colon \Id \Rightarrow i^*i_!
\qquad
\leps\colon i_! i^* \Rightarrow \Id
\qquadtext{and}
\reta\colon \Id \Rightarrow i_* i^*
\qquad
\reps\colon i^*i_* \Rightarrow \Id
\]
for $i_!\adj i^*$ and $i^* \adj i_*$ respectively, such that all the following properties hold:
\begin{enumerate}[\rm({Mack}\,1)]
\setcounter{enumi}{4}
\item
\label{Mack-5}%
Additivity of adjoints: Whenever $i=i_1\sqcup i_2\colon H_1\sqcup H_2\into G$, under the identification $\MM(H_1\sqcup H_2)\cong \MM(H_1)\oplus \MM(H_2)$ of~\Mackintro{1} we have
\[
(i_1\sqcup i_1)_!=\big((i_1)_! \ (i_2)_!\big)
\qquadtext{and}
(i_1\sqcup i_1)_*=\big((i_1)_* \ (i_2)_*\big)
\]
with the obvious `diagonal' units and counits. (See \Cref{Rem:add-adjoint}.)
\smallbreak
\item
\label{Mack-6}%
Two-sided adjoint equivalences: Whenever $i^*$ is an equivalence, the units and counits are isomorphisms and $(\leta)\inv=\reps$ and $(\leps)\inv=\reta$. Furthermore when $i=\Id$ we have $i_!=i_*=\Id$ with identity units and counits.
\smallbreak
\item
\label{Mack-7}%
\index{strict Mackey formula} \index{Mackey formula!strict --}%
Strict Mackey Formula: For every iso-comma as in~\eqref{eq:iso-comma-intro}, the two isomorphisms $\gamma_!\colon q_! p^* \stackrel{\sim}{\Rightarrow} u^* i_!$ and $(\gamma\inv)_*\colon u^* i_* \stackrel{\sim}{\Rightarrow} q_* p^*$ of~\Mackintro{3} are moreover inverse to one another
\[
\gamma_!\circ (\gamma\inv)_* =\id
\qquadtext{and}
(\gamma\inv)_* \circ \gamma_!=\id
\]
under the equality $u^*i_!=u^*i_*$ and $q_!p^*=q_*p^*$ of their sources and targets.
\smallbreak
\item
\label{Mack-8}%
Agreement of pseudo-functors: The pseudo-functors $i\mapsto i_!$ and $i\mapsto i_*$ coincide, namely: For every 2-cell $\alpha\colon i\Rightarrow i'$ between faithful $i,i'\colon H\into G$ we have $\alpha_!=\alpha_*$ as morphisms between the functors $i_!=i_*$ and $i'_!=i'_*$; and for every composable faithful morphisms $j\colon K\into H$ and $i\colon H\into G$ the isomorphisms $(ij)_!\cong i_! j_!$ and $(ij)_*\cong i_* j_*$ coincide.
\smallbreak
\item
\label{Mack-9}%
Special Frobenius Property: For every faithful $i\colon H\into G$, the composite
\[
\Id_{\MM(H)} \stackrel{\leta\ }{\Longrightarrow} i^*i_!=i^*i_* \stackrel{\reps\ }{\Longrightarrow} \Id_{\MM(H)}
\]
of the left unit and the right counit is the identity.
\smallbreak
\item
\label{Mack-10}%
Off-diagonal vanishing: For every faithful $i\colon H\into G$, if $\incl_C\colon C\hookrightarrow (i/i)$ denotes the inclusion of the complement $C:=(i/i)\smallsetminus \Delta_i(H)$ of the `diagonal component'~$\Delta_i(H)$ in the iso-comma square
\begin{equation*}
\vcenter{
\xymatrix@C=14pt@R=14pt{
& (i/i) \ar[dl]_-{p_1} \ar[dr]^-{p_2}
\\
H \ar[dr]_-{i} \ar@{}[rr]|-{\isocell{\lambda}}
&& H \ar[dl]^-{i}
\\
&G
}}
\end{equation*}
(see \Cref{sec:self-iso} for the definition of the diagonal~$\Delta_i\colon H\into (i/i)$ and first properties) then the whiskered natural transformation
\[
\incl_C^* \left(
p_1^* \stackrel{p_1^*\;\leta\;}{\Longrightarrow} p_1^*i^*i_! \stackrel{\lambda^* \;\id\;}{\Longrightarrow} p_2^* i^*i_* \stackrel{p_2^*\;\reps\;}{\Longrightarrow} p_2^*
\right)
\]
is zero.
\end{enumerate}
\end{Thm}

\begin{Rem} \label{Rem:expl_inv_BC}
Most notable in the list of properties of \Cref{Thm:rectification-intro} is perhaps the `Strict Mackey Formula'~(Mack\,\ref{Mack-7}). It can be understood as saying that the base-change formula that we give in~\Mackintro{3} is substantially nicer than ordinary BC-formulas encountered in the literature, which usually just say that $\gamma_!$ is an isomorphism without providing an actual inverse. Here an explicit inverse appears as part of the rectified structure:
\begin{equation}
\label{eq:magic-compact}%
(\gamma_!)\inv=(\gamma\inv)_*\,.
\end{equation}
This formula is a purely 2-categorical property which has no counterpart in the world of ordinary Mackey functors. The authors did not anticipate its existence when first embarking on this project.
\end{Rem}

\bigbreak
\section{Separable monadicity}
\label{sec:monadicity-intro}
\medskip

In \Cref{sec:monadicity}, we immediately put the Rectification \Cref{Thm:rectification-intro} to use, and more specifically the Special Frobenius Property~\Mack{9}. Indeed, we prove in \Cref{Thm:monadicity} that for every Mackey 2-functor~$\MM$ and for every subgroup~$H\le G$, restriction and (co)\,induction functors along $i\colon H\into G$
\[
\xymatrix{
\MM(G) \ar@<-.6em>[d]_-{i^*=\Res^G_H}
\\
\MM(H) \ar@<-.6em>[u]_-{i_*=\Ind_H^G}^-{\adj}
}
\]
automatically satisfy \emph{separable monadicity}. Formally, monadicity means that this adjunction induces an equivalence between the bottom category~$\MM(H)$ and the Eilenberg-Moore category of modules (\aka algebras) over the associated monad $\bbA:=\Ind_H^G\Res^G_H$ on the top category~$\MM(G)$.

In simpler terms it means that one can construct $\MM(H)$ out of~$\MM(G)$, as a category of modules with respect to a generalized ring (the monad), in such a way that restriction $\Res^G_H\colon \MM(G)\to \MM(H)$ becomes an extension-of-scalars functor. This realizes the intuition that the category of $H$-equivariant objects should be in some sense `carved out' of the bigger category of $G$-equivariant objects. This intuition can almost never be realized via a more naive construction, like a categorical localization for instance. However, it can be realized via an extension-of-scalar as above. Moreover, this extension is very nice: it is \emph{separable}.

Recall that a monad $\bbA$ is separable \index{separable monad} if its multiplication $\mu\colon \bbA\circ \bbA\to \bbA$ admits an $\bbA,\bbA$-bilinear section $\sigma\colon \bbA\to \bbA\circ \bbA$. For rings (think of monads of the form $\bbA=A\otimes_R-$ for an algebra~$A$ over a commutative ring~$R$), this notion of separability is classical and goes back to Auslander-Goldman~\cite{AuslanderGoldman60}. Over fields, it covers the notion of finite separable extension. The simplest form of separable monad are the idempotent monads, \ie those whose multiplication $\mu\colon\bbA \circ \bbA \isoto \bbA$ is an isomorphism (think of the ring $A=S\inv R$). Idempotent monads are exactly Bousfield localizations. In fact, as we explain for instance in~\cite{BalmerDellAmbrogioSanders15,Balmer16}, separable monads are to idempotent monads what separable extensions are to localizations, or what the \'etale topology is to the Zariski topology in algebraic geometry.

In other words, knowing that the category $\MM(G)$ is part of a Mackey 2-functor~$\MM$ automatically tells us that the collection of restrictions $\MM(G)\to \MM(H)$ for all subgroups~$H\le G$ provides us with a collection of `abstract \'etale extensions' of~$\MM(G)$. These extensions can then be used with an intuition coming from algebraic geometry, for instance in combination with the theory of descent. We refer the interested reader to~\cite{Balmer15,Balmer16} for earlier developments along these lines in special cases, for instance in modular representation theory.

After the first such separable monadicity result was isolated in~\cite{Balmer15} for ordinary representation theory, we undertook in~\cite{BalmerDellAmbrogioSanders15}, together with Sanders, to transpose the idea to other equivariant settings, beyond algebra. Although we gave only a few examples, they came from sufficiently different backgrounds that the existence of a deeper truth was already apparent. Yet, we could not formulate the result axiomatically. With Mackey 2-functors, we now can.

Applying the ideas of~\cite{Balmer16} to Mackey 2-functors taking values in tensor-triangulated categories is then a natural follow-up project of the present book.

\bigbreak
\section{Mackey 2-functors and Grothendieck derivators}
\label{sec:Mackey-vs-derivators}
\medskip

One of our main goals is to support our definition of Mackey 2-functor with a robust catalogue of examples, in common use in `equivariant mathematics'. In order to do so, we prove theorems showing that some standard structures can be used to construct Mackey 2-functors. In particular, we prove that every additive Grothendieck derivator~\cite{Groth13} provides a Mackey 2-functor when its domain is restricted to finite groupoids (\Cref{Thm:ambidex-der}). This result specializes to say that a Mackey 2-functor can be associated to any `stable homotopy theory', in the broad modern sense of `stable homotopy' that includes usual derived categories for instance. Further sources on derivators include~\cite{Franke96pp} and~\cite{Heller88}.

We recall the precise axioms (Der\,\ref{Der-1})--(Der\,\ref{Der-4}) of derivators in \Cref{sec:add-der-Mackey}. The prototype of a derivator is the strict 2-functor defined on all small categories
\[
\DD\colon \Cat^{\op}\to \CAT
\]
by $J\mapsto\Ho(\cat{Q}^J)$, the homotopy category of diagrams associated to a Quillen model category~$\cat{Q}$. In other words, every homotopy theory provides a derivator which encapsulates its 2-categorical information in terms of homotopy categories and homotopy limits and colimits (homotopy Kan extensions). In this way, every \emph{stable} model category~$\cat{Q}$ gives an \emph{additive} derivator, \ie one taking values in the 2-category of additive categories.

The analogies between Grothendieck's notion of derivators and our Mackey 2-functors are apparent. First of all, the axioms \Mackintro{1}, \Mackintro{2} and \Mackintro{3} are strongly inspired by the derivators' axioms (Der\,\ref{Der-1}), (Der\,\ref{Der-3}) and (Der\,\ref{Der-4}): We start from a strict 2-functor and require existence of adjoints and Beck-Chevalley properties for base-change along comma squares. For this very reason, additive derivators are a great source of Mackey 2-functors once we prove the Ambidexterity Theorem~\ref{Thm:ambidex-der}, which gives us the remaining~\Mackintro{4} for free in this case.

On the other hand, there are also important differences between the theory of derivators and that of Mackey 2-functors, beyond the obvious fact that Mackey 2-functors are only defined on finite groupoids and are required to take values in additive categories. Let us say a word about those differences.

The critical point is the lack of (Der\,\ref{Der-2}) for Mackey 2-functors. Indeed, for a derivator~$\DD$, the various values $\DD(J)$ at small categories~$J$ (\eg finite groupoids) are to be thought of as `coherent' versions of diagrams with shape~$J$ in the base~$\DD(1)$ over the final category~$1$ (which is often denoted $e$ in derivator theory, or~$[0]$). In the prototype of $\DD(J)=\Ho(\cat{Q}^J)$, this is the well-known distinction $\Ho(\cat{Q}^J) \neq \Ho(\cat{Q})^J$ between the homotopy category of diagrams and diagrams in the homotopy category. Axiom (Der\,\ref{Der-2}) then says that the canonical functor $\DD(J)\to \DD(1)^J$ is conservative, \ie isomorphisms in~$\DD(J)$ can be detected pointwise, by restricting along the functors $x\colon 1\to J$ for all objects $x$ of~$J$. We do not have such an axiom for Mackey 2-functors~$\MM$. Morphisms in~$\MM(G)$ which are pointwise isomorphisms are \emph{not} necessarily isomorphisms. This is already illustrated with $\MM(G)=\SH(G)$, the stable equivariant homotopy category, in which $\Res^G_1\colon \SH(G)\to \SH$ is not conservative. Hence proving that $\SH(G)$ forms a Mackey 2-functor requires a little more care; see \Cref{Exa:SH(G)}. Yet, the most striking example of a Mackey 2-functor which is not the restriction of a derivator because it fails~(Der\,\ref{Der-2}) is certainly $\MM(G)=\Stab(\kk G)$, the stable module category of $\kk G$-modules modulo projectives. Indeed, in this extreme case the base (or non-equivariant) category $\MM(1)=0$ is trivial and thus cannot detect much of anything.

Another difference between Mackey functors and derivators comes from ambidexterity~\Mackintro{4}, which is clearly a feature specific to Mackey 2-functors. Ambidexterity is also essential to our construction of the bicategory of Mackey 2-motives, discussed in the second part of this work (see \Cref{sec:Mackey-2-motives-intro} below). Of course, it is conceivable that one could construct an analogous 2-motivic version of derivators, resembling what we do here with Mackey 2-motives. In broad strokes, these derivator 2-motives could consist of a span-flavored construction in which two \emph{separate} forward functors $u_!$ and $u_*$ have to be formally introduced, one left adjoint and one right adjoint to the given~$u^*$. Composition of such $u_!$ and $u_*$ is however rather mysterious, and inverses to the BC-maps would have to be introduced artificially (\cf \Cref{Rem:expl_inv_BC}). If feasible, such a construction seems messy. It is a major simplification of the Mackey setting that we only need \emph{one} covariant functor $i_!=i_*$ and thus obtain a relatively simple motivic construction, as we explain next.

\bigbreak
\section{Mackey 2-motives}
\label{sec:Mackey-2-motives-intro}%
\medskip

\Cref{ch:bicat-spans,ch:2-motives,ch:additive-motives} are dedicated to the motivic approach to Mackey 2-functors. They culminate with \Cref{Thm:UP-Spanhat} in which we prove the universal property of (semi-additive) Mackey 2-motives. This part requires a little more of the theory of 2-categories and \emph{bicategories}, the generalizations of 2-categories in which horizontal composition of 1-morphisms works only up to coherent isomorphisms.

The basic tool for our constructions is the concept of `span', \ie short zig-zags of morphisms ${\scriptstyle\bullet}\lto{\scriptstyle\bullet}\to {\scriptstyle\bullet}$ and the unfamiliar reader can review ordinary categories of spans in \Cref{sec:ordinary-spans}.

Instead of producing a possibly mysterious universal construction via generators and relations, we follow a more down-to-earth approach. Our construction of Mackey 2-motives involves two layers of spans, first for 1-cells and then for 2-cells. The price for this explicit construction is paid when proving the universal property.

Although very explicit, there is no denying that these constructions and the proof of their universal properties are computation-heavy. As a counterweight, we establish a calculus of string diagrams in Mackey 2-motives which removes a great deal of the technicalities of this double-span construction and gives to some pages of this book an almost artistic quality. Arguably, we in fact provide \emph{two} explicit descriptions, one by means of spans of spans and one by means of string diagrams. Our 2-smart readers will identify the former as a bicategory and the latter as a biequivalent 2-category, \ie a `strictification'; see \Cref{sec:string-presentation}.

The voluminous \Cref{ch:bicat-spans} is mostly a preparation for the central \Cref{ch:2-motives}, whereas \Cref{ch:additive-motives} provides $\bbZ$-linearizations of the semi-additive results obtained in \Cref{ch:2-motives}. In more details, we construct the bicategory of additive Mackey 2-motives $\bbZ\Spanhat(\groupoid)$ through two layers of `span constructions' and one layer of `block-completion':
\[
\xymatrix@C=4.5em@L=1ex{
\groupoid^{\op} \ar[r]_-{\textrm{\Cref{ch:bicat-spans}}}
& \Span(\groupoid) \ar[r]_-{\textrm{\Cref{ch:2-motives}}}
& \Spanhat(\groupoid) \ar[r]_-{\textrm{\Cref{ch:additive-motives}}}
& \bbZ\Spanhat(\groupoid).}\kern-1em
\]
The first step (\Cref{ch:bicat-spans}) happens mainly at the level of 1-cells and creates left adjoints to every faithful~$i\colon H\into G$ but does not necessarily create right adjoints. The second span construction (\Cref{ch:2-motives}) takes place at the level of 2-cells and creates ambidexterity. The last step (\Cref{ch:additive-motives}) in the pursuit of the universal Mackey 2-functor out of~$\groupoid$ appears for a minor reason: With $\Spanhat(\groupoid)$, we have only achieved \emph{semi}-additivity of the target, not plain additivity. Explicitly, the 2-cells in the bicategory~$\Spanhat(\groupoid)$ can be added but they do not admit opposites. We solve this issue in \Cref{ch:additive-motives} by formally group-completing the 2-cells. While at it, we also locally idempotent-complete our bicategory in order to be able to split 1-cells according to idempotent 2-cells, and we do the same to 0-cells one level down, which is the meaning of `\emph{block-completion}'. The latter construction works as expected but might not be entirely familiar, so it is discussed in some detail in \Cref{sec:additive-bicats}. Such idempotent-completions are hallmarks of every theory of motives and they make sense in our 2-categorical setting at two different levels. The ultimate bicategory~$\bbZ\Spanhat(\groupoid)$ of truly \emph{additive} Mackey 2-motives satisfies a universal property (\Cref{sec:additive-motives}), which is easily deduced from the significantly harder universal properties of~$\Span(\groupoid)$ and~$\Spanhat(\groupoid)$ that we establish first (in \Cref{sec:UP-Span,sec:Mackey-UP} respectively).

We put the additive enrichment of~$\bbZ\Spanhat(\groupoid)$ to task in \Cref{sec:Yoneda-2-motives}, showing that the represented 2-functor $\bbZ\Spanhat(\groupoid)(G_0,-)$ is a Mackey 2-functor (in the variable~``$-$") for every fixed groupoid~$G_0$. For instance even the trivial group~$G_0=1$ produces an interesting Mackey 2-functor in this way (\Cref{Thm:1-Mack-is-2-Mack}). In the very short \Cref{sec:presheaves-2Mack}, we use another Yonedian technique (\Cref{Prop:extend-2Mack-via-Yoneda}) to show that the abelian category of ordinary Mackey functors on~$G$ is the value at~$G$ of some Mackey 2-functor: The Mackey 2-functor of Mackey functors (\Cref{Cor:Mackey-functors-abelian-example}).

We conclude the text with a critical aspect of the motivic construction, namely motivic decompositions. As explained in the Introduction, there are two components to this. First, we need to compute the endomorphism ring of the identity 1-cell $\Id_G$ of the 2-motive of~$G$ in the 2-category of Mackey 2-motives $\bbZ\Spanhat(\gpd)$. Secondly, we need to see how a decomposition of this ring yields block decompositions of~$\MM(G)$ for every Mackey 2-functor~$\MM$. We do the former in the important \Cref{sec:B(G)} and we explain the latter in the more formal \Cref{sec:decomposition}.

As we shall see, this 2-endomorphism ring of~$\Id_G$ turns out, rather miraculously, to be a known commutative ring in representation theory going by the name of \emph{crossed Burnside ring}. The usual Burnside ring~$\Bur(G)$ is perhaps better known, and can be described as the Grothendieck group of the category of finite $G$-sets. It admits a basis consisting of isomorphism classes of $G$-orbits~$G/H$, \ie indexed by conjugacy classes of subgroups~$H\le G$. The crossed Burnside ring~$\xBur(G)$ is similarly defined as a Grothendieck group but is bigger than $\Bur(G)$, which it admits as a retract. There is a basis of~$\xBur(G)$ consisting of conjugacy classes of pairs~$(H,a)$, where $H\le G$ is a subgroup together with a centralizer~$a\in C_G(H)$ of~$H$ in~$G$. In fact, we can identify the ordinary Burnside ring as another 2-endomorphism ring in $\bbZ\Spanhat(\gpd)$, namely that of the particular 1-cell given by the span $1\gets G=G$, see~\eqref{eq:first-B(G)}, whereas the identity 1-cell, $\Id_G$ is given by the span~$G=G=G$. As a consequence of these connections, every ring decomposition of the crossed Burnside ring~$\xBur(G)$, and in particular every ring decomposition of the ordinary Burnside ring~$\Bur(G)$, induces a block-decomposition of the Mackey 2-motive of~$G$ and consequently, by universality, of the category~$\MM(G)$ for every Mackey 2-functor~$\MM$. As said, the latter is explained in the final \Cref{sec:decomposition}, where we show that each additive category~$\MM(G)$ is enriched over $\xBur(G)$-modules.

\bigbreak
\section{Pointers to related works}
\label{sec:literature}%
\medskip

Let us say a word of existing literature.

Bicategories of spans have been considered by many authors in several variants and settings, starting already with~\cite{Benabou67}. We shall in particular rely on~\cite{Hoffnung11pp} to avoid tedious verifications. The interested reader can also consult~\cite{Miller17} and~\cite{BaezHoffnungWalker10} for the relevance of spans of groupoids to topology and physics, respectively. There is no shortage of Mackey-related publications and the use of spans in this context is well-known and widespread.
Some versions of the universal property of spans have been known to category theorists for a long time and have appeared in print, \eg in \cite[Thm.\,A.2]{Hermida00} and \cite{DawsonParePronk04}.

An approach via $(\infty,1)$-categories can be found in the interesting work of Barwick~\cite{Barwick17}. In this context, Harpaz \cite{Harpaz17pp} has proved that the $(\infty,1)$-category of spans of finite $n$-truncated spaces is the universal way of turning $n$-truncated spaces into an \emph{$n$-(semi-)additive} $\infty$-category, in the sense of Hopkins-Lurie \cite{HopkinsLurie13}. Our theory can be seen as an extension or refinement of the $n=1$ case (groupoids being 1-truncated spaces) of his result. Indeed, although our 2-level approach obviously fails to capture higher equivalences, it does allow for \emph{non-invertible} 2-cells and therefore provides a direct grip on adjunctions and their properties, without any need to climb further up the higher-categorical ladder.
Formally, a simultaneous common generalization of the Barwick-Harpaz-Hopkins-Lurie theory and ours would require the framework of $(\infty,2)$-categories, for which we refer to the book of Gaitsgory and Rozenblyum~\cite[App.]{GaitsgoryRozenblyum17}. It was pointed out to us by Harpaz that Hopkins and Lurie do hint at something resembling our construction of Mackey 2-motives in terms of an $(\infty,2)$-category of spans of spans; see \cite[Remark 4.2.5]{HopkinsLurie13}.

\end{chapter-one}
%
\chapter{Mackey 2-functors}
\label{ch:2-Mackey}%
\bigbreak
\begin{chapter-two}

We discuss Mackey 2-functors beyond the survey of \Cref{sec:Mackey-2-functors}, beginning with details on iso-commas and Mackey squares (\Cref{sec:comma,sec:mackey-squares}). In \Cref{sec:Mackey-2-functors-general}, we clarify what a class~$\GG$ of groupoids `of interest' should consist of (\Cref{Hyp:GG}) and we define Mackey 2-functors in that generality. We conclude the chapter by discussing the separability of restriction~$\MM(G)\to \MM(H)$ to subgroupoids (\Cref{sec:monadicity}) and the decategorification of Mackey 2-functors down to ordinary Mackey functors (\Cref{sec:decategorification}).

\bigbreak
\section{Comma and iso-comma squares}
\label{sec:comma}%
\medskip

In any 2-category (or even bicategory), one can define a strict notion of pullback square, which will usually not be invariant under equivalence. The correct notion, at least in the case of groupoids, will consist of those squares \emph{equivalent to iso-comma squares}. We call these \emph{Mackey squares} and discuss them in \Cref{sec:mackey-squares}. We first recall the general notion of \emph{comma square}, which plays a role in the theory of derivators, and we then specialize to the case of groupoids.

\begin{Def} \label{Def:comma}
\index{comma square} \index{iso-comma square}%
Let $\cat{B}$ be a 2-category. A \emph{comma square} over a given cospan $A \stackrel{a}{\to} C \stackrel{b}{\leftarrow} B$ of 1-cells of $\cat{B}$ is a 2-cell
\begin{align} \label{eq:comma-square}
\vcenter{\xymatrix@C=14pt@R=14pt{
& a/b \ar[ld]_-{p} \ar[dr]^-{q}
 \ar@{}[dd]|{\oEcell{\gamma}}
\\
A \ar[dr]_a && B \ar[dl]^b \\
&C &
}}
\end{align}
having the following two properties, jointly expressing the fact that the 2-cell $\gamma$ is \emph{2-universal} among those sitting over the given cospan:
\begin{enumerate}[\rm(a)]
\item
\label{it:comma-a}%
For every pair of 1-cells $f\colon T\to A$ and $g\colon T\to B$ and for every 2-cell $\delta\colon af \Rightarrow bg$, there is a unique 1-cell $h\colon T\to a/b$ such that $p h=f$, $q h=g$ and $\gamma h = \delta$.
\begin{align*}
\xymatrix@C=14pt@R=14pt{
& T \ar[d]^h \ar@/_3ex/[ddl]_f \ar@/^3ex/[ddr]^g &
 &&& T \ar@/_3ex/[ddl]_f \ar@/^3ex/[ddr]^g
 \ar@{}[ddd]|{\oEcell{\delta}}
& \\
& a/b \ar[ld]_-{p} \ar[dr]^-{q}
 \ar@{}[dd]|{\oEcell{\gamma}}
&
 & = &&& \\
A \ar[dr]_a && B \ar[dl]^b
 && A \ar[dr]_a && B \ar[dl]^b \\
&C & &&& C &
}
\end{align*}
\index{$((<$@$\langle\ldots\rangle$ \, induced 1-cell into comma}%
We will write $\langle f, g, \delta \rangle$ for the unique 1-cell~$h$ as above
\[
\langle f, g, \delta \rangle\colon T\to a/b
\]
determined by these three components.
\smallbreak
\item
\label{it:comma-b}%
For every pair of 1-cells $h,h'\colon T\to a/b$ and every pair of 2-cells $\tau_A\colon ph\Rightarrow ph'$ and $\tau_B\colon qh\Rightarrow qh'$ such that $(\gamma h') (a\tau_A)= (b\tau_B)(\gamma h)$
\begin{align*}
\xymatrix@C=14pt@R=14pt{
& T \ar@/^1ex/[d]^-{h'} \ar@/_6ex/[ldd]_-{ph} \ar@{}[dl]|{\oEcell{\tau_A}} &
 &&& T \ar@/_1ex/[d]_h \ar@/^6ex/[rdd]^-{qh'} \ar@{}[dr]|{\oEcell{\tau_B}} & \\
& a/b \ar[ld]_-{p} \ar[dr]^-{q}
 \ar@{}[dd]|{\oEcell{\gamma}}
&& = && a/b \ar[dl]_p \ar[dr]^q
 \ar@{}[dd]|{\oEcell{\gamma}}
& \\
A \ar[dr]_a && B \ar[dl]^b
 && A \ar[dr]_a && B \ar[dl]^b \\
&C & &&& C &
}
\end{align*}
there exists a unique $\tau\colon h\Rightarrow h'$ such that $p \tau =\tau_A$ and $q \tau= \tau_B$.
\end{enumerate}
If the 2-cell $\gamma$ is moreover invertible, and if~\eqref{it:comma-a} holds (only) for those $\delta$ which are invertible, then the comma square is called an \emph{iso-comma} square. Note that if a comma square is such that $\gamma$ is invertible, then it is also an iso-comma square.

It is sometimes convenient to denote the comma object by $A\diagup_{\!\!\!\scriptscriptstyle C}\, B$ rather than~$a/b$.
\end{Def}

\begin{Exa} \label{Exa:comma-in-1cats}
In any (2,1)-category, comma and iso-comma squares coincide. If $\cat{B}$ is furthermore locally discrete, \ie is just a 1-category, then comma squares and iso-comma squares are precisely the same as ordinary pullback squares.
\end{Exa}

\begin{Exa} \label{Exa:inv-comma-square}
Given an \emph{iso}-comma square~\eqref{eq:comma-square}, we can invert its 2-cell to obtain a new square:
\begin{align*}
\vcenter{\xymatrix@C=14pt@R=14pt{
& a/b \ar[ld]_-{q} \ar[dr]^-{p}
 \ar@{}[dd]|{\oEcell{\gamma^{-1}}}
\\
B \ar[dr]_b && A \ar[dl]^a \\
&C &
}}
\end{align*}
It is easy to see that this is an iso-comma square for $b\colon B\to C\gets A \,:\!a$.
\end{Exa}

The following example is the essential prototype:
\begin{Exa} \label{Exa:comma-in-Cat}
If $\cat{B}=\Cat$ is the 2-category of (small) categories, then the comma square over $a\!: A \to C \leftarrow B:\!b$ has a well-known and transparent construction, where the objects of $a/b$ are triples $(x,y,\gamma)$ with $x$ an object of~$ A$, with $y$ an object of~$B$ and with $\gamma\colon a(x)\to b(y)$ an arrow of~$ C$, and where a morphism $(x,y,\gamma)\to (x',y',\gamma')$ is a pair $(\alpha,\beta)$ of an arrow $\alpha\colon x\to x'$ of $A$ and an arrow $\beta\colon y\to y'$ of $B$ such that the evident square commutes in~$C$, namely $\gamma' a(\alpha) = b(\beta) \gamma$. Then $p\colon a/b\to A$ and $q \colon a/b\to B$ are the obvious projections $(x,y,\gamma)\mapsto x$ and $(x,y,\gamma)\mapsto y$, and the two properties of \Cref{Def:comma} are immediately verified. Iso-comma squares are constructed similarly, by only considering triples $(x,y,\gamma)$ with $\gamma$ invertible.

The construction of iso-comma squares in the 2-category of categories provides the iso-comma squares in the sub-2-category of groupoids, see \Cref{Rem:iso-comma-intro}.
\end{Exa}

Example~\ref{Exa:comma-in-Cat} allows us to characterize comma squares in general 2-categories.

\begin{Rem} \label{Rem:commas_translated}
In a 2-category~$\cat{B}$, a 2-cell as in~\eqref{eq:comma-square} is a comma square if and only if composition with~$p$, $q$ and $\gamma$ induces an \emph{isomorphism} of categories
\begin{equation}
\label{eq:compare-to-comma}%
\begin{array}{ccc}
\cat{B}(T, a/b) & \stackrel{\cong}{\longrightarrow} & \cat{B}(T,a) / \cat{B}(T,b)
\\[.5em]
h & \mapsto & (p h, q h, \gamma h)
\end{array}
\end{equation}
for every $T\in \cat{B}_0$, where the category on the right-hand side is the comma category over
$\cat{B}(T,A)\otoo{\cat{B}(T,a)} \cat{B}(T,C) \lotoo{\cat{B}(T,b)} \cat{B}(T,B)$ in~$\Cat$, as described in Example~\ref{Exa:comma-in-Cat}. Indeed, parts~\eqref{it:comma-a} and~\eqref{it:comma-b} of \Cref{Def:comma} are equivalent to this functor inducing a bijection on objects and on arrows, respectively. In particular, it follows that the defining property of a comma square is a universal property, characterizing it up to a unique canonical isomorphism in~$\cat{B}$. For the same reasons, the output comma object $a/b$ is natural in the input cospan $\oto{a}\loto{b}$.
\end{Rem}

\begin{Rem} \label{Rem:assoc}
By the usual arguments, the universal property of comma objects yields unique associativity isomorphisms compatible with the structure 2-cells:
\[
\xymatrix@C=14pt@R=14pt{
&A\diagup_{\!\!\!\scriptscriptstyle U}\, (B\diagup_{\!\!\!\scriptscriptstyle V}\, C) \ar@/_2ex/[dddl] \ar[ddrr] \ar@{}[dd]|{\Ecell} \ar@{-->}[rr] &&
 (A\diagup_{\!\!\!\scriptscriptstyle U}\, B)\diagup_{\!\!\!\scriptscriptstyle V}\, C \ar[ddll]|{\phantom{m}} \ar@/^2ex/[dddr] \ar@<-1ex>@{-->}[ll]_-{\cong} \ar@{}[dd]|{\Ecell} & \\
 &&&& \\
& A\diagup_{\!\!\!\scriptscriptstyle U}\, B \ar[dl] \ar[dr] && B\diagup_{\!\!\!\scriptscriptstyle V}\, C \ar[dl] \ar[dr] & \\
A \ar[dr] \ar@{}[rr]|{\Ecell} && B \ar[dl] \ar[dr] \ar@{}[rr]|{\Ecell} && C \ar[dl] \\
& U && V &
}
\]
Here ``compatible'' means that the above diagram of 2-cells commutes, with the two slanted triangles being identities. All details of this construction will be spelled out in the course of a proof, see~\eqref{Eq:precise-assoc}.
\end{Rem}

\begin{Rem} \label{Rem:units}
Building the comma square on a cospan of the form $A \stackrel{\Id}{\to} A \stackrel{b}{\leftarrow} B$, we obtain the diagram
\begin{equation}
\label{eq:units}%
\vcenter{
\xymatrix@C=14pt@R=14pt{
&B \ar@/_3ex/[ddl]_b \ar@/^3ex/[ddr]^\Id \ar[d]^-{i_b} & \\
& (\Id/b) \ar[dl] \ar[dr]^-{q_b} & \\
A \ar[dr]_\Id \ar@{}[rr]|{\oEcell{\gamma}} && B \ar[dl]^b \\
&A &
}}
\end{equation}
where $i_b\colon B\to \Id/b$ is the 1-cell $\langle b, \Id_B , \id_b\colon b\Rightarrow b \rangle$ in the notation of \Cref{Def:comma}\,\eqref{it:comma-a}.
In particular $i_b$ is a canonical right inverse of $q_b:\Id / b\to B$, the comma base-change of~$\Id_A$ along~$b$. Typically, $q_b$ and $i_b$ are not strictly invertible. But if the comma square is an iso-comma square (\eg if we are working in a (2,1)-category), they will always be mutually quasi-inverse equivalences. Indeed, by the universal property on arrows there is a unique invertible 2-cell $\tau\colon i_bq_b \Rightarrow \Id_{(\Id/b)}$ with components $\tau_A:=\gamma\inv$ and $\tau_B:=\id_{q_b}$.

A similar remark holds for cospans of the form $A \oto{a} B \loto{\Id}B$.
\end{Rem}

In view of Remarks~\ref{Rem:commas_translated} and~\ref{Rem:units}, it is natural to relax the evil property that the functor~\eqref{eq:compare-to-comma} be an isomorphism into the more convenient property that it be an equivalence. Similarly, we would like to accept squares like the outside one in~\eqref{eq:units} when $i_b$ is an equivalence. This relaxing of the definition will yield the correct class of squares for our treatment of Mackey 2-functors.

\begin{Prop}
\label{Prop:weak-comma}%
Consider a 2-cell in a 2-category~$\cat{B}$
\begin{equation}
\label{eq:2-cell-D}%
\vcenter{
\xymatrix@C=14pt@R=12pt{
& D \ar[ld]_-{\tilde p} \ar[dr]^-{\tilde q}
\ar@{}[dd]|{\oEcell{\tilde\gamma}}
\\
A \ar[dr]_a
&& B \ar[dl]^b \\
&C &
}}
\end{equation}
and assume the comma square on $\oto{a}\loto{b}$ exists in~$\cat B$. The following are equivalent:
\begin{enumerate}[\rm(i)]
\item
The induced 1-cell $\langle \tilde p,\tilde q,\tilde\gamma\rangle\colon D\to (a/b)$ of \Cref{Def:comma} is an equivalence.
\smallbreak
\item
\label{it:htpy-comma}%
For every~$T\in\cat{B}_0$, the following induced functor is an equivalence in~$\Cat$:
\[
\begin{array}{ccc}
\cat{B}(T, D) & \too & \cat{B}(T,a) / \cat{B}(T,b)
\\[.5em]
h & \mapsto & (\tilde p h, \tilde q h, \tilde\gamma h)\,.
\end{array}
\]
Compare with~\eqref{eq:compare-to-comma}.
\end{enumerate}
\end{Prop}

\begin{proof}
Let $k:=\langle \tilde p,\tilde q,\tilde\gamma\rangle\colon D\to (a/b)$ denote the 1-cell in~(i). For every object $T\in\cat{B}_0$, we have a commutative diagram in~$\Cat$ as follows:
\[
\xymatrix{
\cat{B}(T,D) \ar[rr]^-{\eqref{it:htpy-comma}} \ar[rd]_-{\cat{B}(T,k)}
&& \cat{B}(T,a)/\cat{B}(T,b)\,.
\\
& \cat{B}(T,a/b) \ar[ru]^-{\cong}_-{\eqref{eq:compare-to-comma}}
}
\]
As the right-hand functor~\eqref{eq:compare-to-comma} is an isomorphism by \Cref{Rem:commas_translated}, the other two functors are simultaneously equivalences. Finally, $\cat{B}(T,k)$ is an equivalence in~$\Cat$ for all~$T\in\cat{B}_0$ if and only if~$k$ is an equivalence in~$\cat{B}$, by \Cref{Cor:2cat_Yoneda_equivs}.
\end{proof}

\begin{Rem}
\label{Rem:weak-comma}%
One can unpack property~\eqref{it:htpy-comma} in the above statement in the spirit of \Cref{Def:comma}\,\eqref{it:comma-a} and~\eqref{it:comma-b}. Indeed, part~\eqref{it:comma-b} remains unchanged (it expresses full-faithfulness) but~\eqref{it:comma-a}, which expressed surjectivity on the nose, is replaced by essential surjectivity, \ie the existence for every 2-cell $\delta\colon a f\Rightarrow b g$ of a (non necessarily unique) 1-cell~$h\colon T\to D$ together with isomorphisms $\varphi\colon f\stackrel{\sim}{\Rightarrow} \tilde p h$ and $\psi\colon \tilde q h\stackrel{\sim}{\Rightarrow} g$ such that $\delta=(b\psi)(\tilde\gamma h)(a\varphi)$. Such an $h$ is then unique up to isomorphism.
\end{Rem}

\bigbreak
\section{Mackey squares}
\label{sec:mackey-squares}%
\medskip

We now specialize our discussion of comma squares to the case of (2,1)-cat\-e\-go\-ries, keeping in mind our main example~$\groupoid$ of finite groupoids. As already mentioned in \Cref{Exa:comma-in-1cats}, since every 2-cell is invertible there is no distinction between comma and iso-comma squares. As is often the case, both the strict and the pseudo-version will be useful: To define spans of groupoids and their composition, we shall rely on the explicit nature of the iso-comma squares. On the other hand, to define our Mackey 2-functors $\MM\colon \groupoid^{\op}\to \ADD$ any square which is equivalent to an iso-comma square can be considered. We give the latter a simple name:

\begin{Def}
\label{Def:Mackey-square}%
\index{Mackey square}%
A 2-cell in a (2,1)-category~$\cat{B}$ (\eg in the 2-category $\groupoid$)
\begin{equation}
\label{eq:Mackey-square}%
\vcenter{
\xymatrix@C=14pt@R=14pt{
& L \ar[ld]_-{v} \ar[dr]^-{j}
 \ar@{}[dd]|-{\isocell{\gamma}}
\\
H \ar[dr]_-{i}
&& K \ar[dl]^-{u} \\
&G &
}}
\end{equation}
is called a \emph{Mackey square} if the induced functor $\langle v,j,\gamma\rangle\colon L\to (i/u)$ is an equivalence. See \Cref{Prop:weak-comma} and \Cref{Rem:weak-comma} for equivalent formulations. (The latter could be used to define Mackey squares directly, bypassing comma squares.)
\end{Def}

\begin{Rem}
Any square equivalent to a Mackey square is a Mackey square.
Here two squares $\tau$ and $\sigma$ are `equivalent' if there exists an equivalence $f$ between their top objects and two invertible 2-cells $\varphi, \psi$ identifying their 2-cells as follows:
\begin{equation*}
\vcenter{
\xymatrix@C=18pt@R=18pt{
& \ar@/_3ex/[ddl] \ar@/^3ex/[ddr] \ar[d]_f^\sim \ar@{}[ddl]|{\oEcell{\varphi}} \ar@{}[ddr]|{\oEcell{\psi}} & \\
&  \ar[dl] \ar[dr] & \\
 \ar[dr] \ar@{}[rr]|{\oEcell{\tau}} &&  \ar[dl] \\
& &
}}
\quad = \quad
\vcenter{
\xymatrix@C=18pt@R=18pt{
& \ar@/_3ex/[ddl] \ar@/^3ex/[ddr]  \ar@{}[ddd]|{\oEcell{\sigma}} & \\
&  
\\
 \ar[dr] 
  &&  \ar[dl] \\
& &
}}
\end{equation*}
\end{Rem}

\begin{Rem} \label{Rem:Mackey_sq_vs_biequivs}
Using \Cref{Rem:weak-comma}, it is straightforward to check that any bi\-equivalence $\cat B\stackrel{\sim}{\to} \cat B'$ (see \ref{Ter:Hom_bicats}) preserves Mackey squares. However there is no reason in general for it to preserve iso-comma squares, even when it happens to be a 2-functor between 2-categories.
\end{Rem}

\begin{Exa} \label{Exa:iso_Mackey_square}
It is an easy exercise to verify that, if in~\eqref{eq:Mackey-square} we choose $i$ and $j$ to be identity 1-cells and $\gamma\colon v\Rightarrow u$ to be any (invertible) 2-cell, then the resulting square is a Mackey square. The special case $\gamma = \id$ yields the (outer) Mackey squares discussed in Remark~\ref{Rem:units}.
\end{Exa}

\begin{Rem}
\label{Rem:htpy-pullback}%
\index{Mackey square!as homotopy cartesian square}%
Alternatively, Mackey squares could be called `homotopy cartesian', following for instance Strickland~\cite[Def.\,6.9]{Strickland00}. Indeed, consider the 1-category $\cat{G}$ of all groupoids, with functors as morphisms, ignoring 2-morphisms. Then $\cat{G}$ admits the structure of a Quillen model category in which weak equivalences are the (categorical) equivalences and in which our Mackey squares coincide with homotopy cartesian ones. See details in~\cite[\S\,6]{Strickland00}. We avoid this terminology for several reasons. First, calling equivalences of groupoids `weak-equivalences' could be judged pedantic in our setting. Second, there are many other forms of homotopy at play in the theory of Mackey 2-functors, typically in connections to the derivators appearing in examples (as the homotopy categories of Quillen model categories). Finally, it is conceptually useful to understand iso-comma squares of groupoids as a special case of comma squares of small categories, for the latter are the ones which yield base-change formulas for derivators; and (non-iso) comma squares are not homotopy pullbacks in any obvious way.
\end{Rem}

Of course, the ancestral example of Mackey square is the one which motivates the whole discussion:
\begin{Rem}
\label{Rem:old-Mackey}%
\index{Mackey formula! as Mackey square}%
Consider a finite group~$G$, two subgroups $H,K\le G$ and the corresponding inclusions $i\colon H\to G$ and $u\colon K\to G$ of one-object groupoids. Even in this case, the groupoid $(i/u)$ usually has more than one connected component. In fact, it has one connected component for each double-coset $[x]=KxH$ in~$K\bs G/H$ and the canonical groupoid~$(i/u)$ becomes non-canonically equivalent to a coproduct of one-object groupoids:
\begin{equation}
\label{eq:conn-comp}%
\coprod_{[x]\in K\bs G/H} K\cap{}^{x\!}H \stackrel{\sim}{\too} (i/u) \,.
\end{equation}
This decomposition depends on the choice of the representatives $x$ in the double-coset $[x]\in K\bs G/H$. When such choices become overwhelming, as they eventually always do, the canonical construction $(i/u)$ is preferable. For instance, compare associativity as in \Cref{Rem:assoc} to the homologous mess with double-cosets. And things only get worse with more involved diagrams.

In any case, replacing $(i/u)$ by $\coprod_{[x]\in K\bs G/H} K\cap{}^{x\!}H$ via the equivalence~\eqref{eq:conn-comp} shows that the following square is a Mackey square:
\begin{equation}
\label{eq:old-Mackey-square}%
\vcenter{\xymatrix@C=14pt@R=14pt{
& {\coprod\limits_{[x]\in K\bs G/H}} K\cap{}^{x\!}H
 \ar[dl]_-{v} \ar[dr]^-{j}
 \ar@{}[dd]|(.5){\isocell{\gamma}}
\\
H \ar[dr]_-{i}
&& K \ar[dl]^-{u}
\\
&G
}}
\end{equation}
where the $x$-component of~$v$ is a conjugation-inclusion $v_x=(-)^x\colon K\cap{}^{x\!}H\into H$ whereas each component of~$j$ is the mere inclusion $j_x=\incl\colon K\cap{}^{x\!}H\into K$; the $x$-component of~$\gamma$ is the \emph{2-cell} $\gamma_x={}^{x}(-)\colon i\,v_x\Rightarrow u\,j_x\colon K\cap{}^{x\!}H \into G$. We see here conjugation playing its two roles, at the 1-cell and at the 2-cell levels.
\end{Rem}

\begin{Rem} \label{Rem:comma_vs_Mackey}
Requiring a strict 2-functor~$\MM\colon \cat{B}\to \cat{B}'$ to satisfy base-change with respect to every iso-comma square is equivalent to the (\emph{a~priori} stronger) condition that $\MM$ satisfies base-change with respect to every Mackey square. Indeed, this comes from a more general fact about mates: Suppose given two 2-cells which are equivalent, then the mate of the first one is an isomorphism if and only if the mate of the other is. See \Cref{Prop:mates-under-top-functor} and \Cref{Rem:mate-natural} if necessary.
\end{Rem}

\bigbreak
\section{General Mackey 2-functors}
\label{sec:Mackey-2-functors-general}%
\medskip

The Mackey 2-functors $G\mapsto \MM(G)$ discussed in \Cref{sec:Mackey-2-functors} were the `global' type, \ie those defined on all finite group(oid)s~$G$ and all functors $u\colon H\to G$. However, we already saw in \Cref{Exa:Mackey-2-functors}\,\eqref{it:Exa-Mackey-Stab} that it is sometimes necessary to restrict to some class of groupoids, or some class of morphisms. We isolate below the conditions such a choice must satisfy.

\begin{Hyp}
\label{Hyp:GG}%
\index{$g$@$\GG$}%
We consider a 2-category
\[
\GG \subseteq \groupoid
\]
of finite groupoids of interest. We assume that $\GG$ is a \emph{2-full} 2-subcategory of the 2-category $\groupoid$ of all finite groupoids, which is closed under finite coproducts, faithful inclusions and iso-commas along faithful morphisms. More precisely, this means that for every $G$ in~$\GG$ and every faithful functor $i\colon H\into G$, the groupoid $H$ and the functor~$i$ belong to the 2-category~$\GG$ and furthermore for any~$u\colon K\to G$ in~$\GG$ the iso-comma~$(i/u)$ and the two functors $p\colon (i/u)\to H$ and $q\colon (i/u)\into K$ as in~\eqref{eq:iso-comma-intro} belong to~$\GG$ as well. (Note that only $p$ comes in question here, given the stability-under-faithful-inclusion assumption; and even this is only a question if $\GG$ is not 1-full in~$\groupoid$.) In other words, if the bottom cospan $H\into G \leftarrow K$ of a Mackey square~\eqref{eq:Mackey-square} belongs to~$\GG$ then so does the top span $H\lto L\into K$.

Finally, in this setting we denote by
\index{$j$@$\JJ$}%
\[
\JJ=\JJ(\GG):=\SET{i\in\GG_1(H,G)}{i\textrm{ is faithful}}
\]
the class of faithful morphisms in~$\GG$.
\end{Hyp}

\begin{Exa}
In a first reading, the reader can safely assume $\GG=\groupoid$ everywhere, unless specifically mentioned otherwise.
\end{Exa}

\begin{Exa}
There is a gain in allowing more general $\GG$ than the main example $\GG=\groupoid$. For instance, the general formalism covers Mackey 2-functors like the stable module category, $G\mapsto \Stab(\kk G)$ in \Cref{Exa:stable-module-cat}, which are only defined on the (2,1)-category $\groupoidf$ of finite groupoids with \emph{faithful} morphisms.
\end{Exa}

\begin{Rem}
It is legitimate to wonder whether~$\GG$ and~$\JJ$ need to consist of groupoids or whether more general 2-categories~$\GG$ and classes~$\JJ$ can be considered. Such an extended formalism is used in \Cref{ch:bicat-spans,ch:2-motives}; see \Cref{Hyp:G_and_I_for_Span}.
\end{Rem}

\begin{Def}
\label{Def:Mackey-2-functor}%
\index{Mackey 2-functor} \index{Mackey 2-functor!-- on a general (2,1)-category}%
Let $\GG$ be a (2,1)-category of finite groupoids of interest (and the class~$\JJ$ of faithful morphisms) as in \Cref{Hyp:GG}. Alternatively, let $(\GG,\JJ)$ be an admissible pair, as in \Cref{Hyp:G_and_I_for_Span}.

A \emph{Mackey 2-functor on~$\GG$} (or in full, an \emph{additive\,\footnote{\,Here in the sense of `$\ADD$-valued'.} Mackey 2-functor on the (2,1)-category~$\GG$, with respect to the class~$\JJ$}) is a strict 2-functor $\MM\colon \GG^{\op}\to \ADD$ satisfying the four axioms \Mack{1}--\Mack{4} below. See details as to what such a 2-functor $\MM\colon \GG^{\op}\too \ADD$ amounts to in \Cref{Rem:pre-2-Mackey}\,(a)-(c).
\begin{enumerate}[(\rm{Mack}~1)]
\item
\label{Mack-1}%
\emph{Additivity}: For every finite family $\{G_\alpha\}_{\alpha\in\aleph}$ in~$\GG$, the natural functor
\[
\MM\big(\,\coprod_{\alpha\in\aleph} G_\alpha\,\big) \too \prod_{\alpha\in\aleph}\MM(G_\alpha)=\bigoplus_{\alpha\in\aleph} \MM(G_\alpha)
\]
is an equivalence (see \Cref{Exa:ADD-dir-sums} for the right-hand side rewriting).
\smallbreak
\item
\label{Mack-2}%
\emph{Induction-coinduction}: For every $i\colon H\into G$ in the class~$\JJ$, restriction $i^*\colon\MM(G)\to \MM(H)$ admits a left adjoint $i_!$ and a right adjoint~$i_*$.
\smallbreak
\item
\label{Mack-3}%
\emph{BC-formulas}: For every Mackey square as in diagram~\eqref{eq:Mackey-square}, the following two mates are isomorphisms:
\[
\gamma_!: j_!\circ v^* \isoEcell u^*\circ i_!
\quadtext{and}
(\gamma\inv)_*:u^*\circ i_* \isoEcell j_*\circ v^*\,.
\]
(\cf \Cref{Rem:comma_vs_Mackey}).
\smallbreak
\item
\label{Mack-4}%
\emph{Ambidexterity}: We have isomorphisms $i_!\simeq i_*$ for every faithful $i\colon H\into G$ in~$\JJ$.
\end{enumerate}
\end{Def}

\begin{Def}
\label{Def:Mackey-2-functor-rectified}%
\index{rectified Mackey 2-functor} \index{Mackey 2-functor!rectified --}%
A \emph{rectified} Mackey 2-functor~$\MM\colon \GG^\op\to \ADD$ is a Mackey 2-functor together with a specified choice of functors $i_!$ for all $i\in \JJ$ and adjunctions $i_!\dashv i^* \dashv i_*:=i_!$ ($i\in \JJ$) which in addition to \Mack{1}--\Mack{4} of \Cref{Def:Mackey-2-functor}, further satisfy (Mack\,\ref{Mack-5})--(Mack\,\ref{Mack-10}) of \Cref{Thm:rectification-intro}.
\end{Def}

\begin{Rem}
The Rectification \Cref{Thm:rectification} guarantees that as soon as we have verified \Mack{1}--\Mack{4}, the units and counits of~$i_!\adj i^*\adj i_*=i_!$ can be arranged to satisfy all the extra properties in (Mack\,\ref{Mack-5})--(Mack\,\ref{Mack-10}): Every Mackey 2-functor can be rectified. Of course, not having to prove the latter six properties greatly simplifies the verification that a specific example of~$\MM$ is indeed a Mackey 2-functor. On the other hand, these additional properties will be extremely precious when \emph{carefully} proving results about Mackey 2-functors.
\end{Rem}

\begin{Rem}
\label{Rem:M()=0}%
It is easy to deduce from Additivity~\Mack{1} that $\MM(\varnothing)\cong 0$ is the zero additive category, by inspecting the image of the equivalence $\nabla=(\Id\ \Id)\colon \varnothing\sqcup \varnothing \isoto\varnothing$ under~$\MM$.
\end{Rem}

\begin{Rem}
\label{Rem:SAD}%
\index{semi-additive} \index{$sad$@$\SAD$ \, 2-category of semi-additive categories}%
Virtually everything we say here about additive categories $\MM(G)$ will make perfect sense with semi-additive categories instead, \ie categories in which we can add objects and morphisms, without requesting additive opposites of morphisms. See \Cref{Ter:additive_cat_etc} or~\cite[VIII.2]{MacLane98}. Furthermore, if $\MM\colon \GG^{\op}\to \CAT$ is any 2-functor satisfying Additivity \Mack{1} and Ambidexterity \Mack{4}, then each category $\MM(G)$ is automatically \emph{semi-additive}. In other words, the only thing $\MM(G)$ is missing to be additive are the additive inverses of maps.

To see why $\MM(G)$ is semi-additive, let $\{G_\alpha\}_{\alpha\in\aleph}$ be a finite set of copies of~$G$. The folding functor $\nabla\colon \coprod_\alpha G_\alpha \to G$, which is the identity on each component, induces the diagonal functor $\mathrm{diag}\colon \MM(G)\to \prod_\alpha \MM(G)$ after an application of Additivity:
\[
\xymatrix@R=1.5em{
{ \prod_{\alpha} \MM(G_\alpha) }
 \ar@/_11ex/[ddd]_-{\coprod_\aleph}
 \ar@/^11ex/[ddd]^-{\prod_\aleph}
\\
{ \MM(\coprod_\alpha G_\alpha) }
 \ar[u]^-{\simeq}
 \ar@/_5ex/[dd]_-{\nabla_!}
 \ar@/^5ex/[dd]^-{\nabla_*}
\\
\\
{ \MM(G) }
 \ar[uu]|-{\nabla^* \vcorrect{.8}_{\vcorrect{.4}}}
}
\]
Now recall that the left and right adjoint of $\mathrm{diag}$ are precisely the functors $\coprod_{\aleph}$ and $\prod_\aleph$ assigning to a family $\{X_\alpha\}_{\alpha\in\aleph}$ its coproduct and product in $\MM(G)$, respectively. These adjoints exist and are isomorphic by the Ambidexterity of~$\MM$.

The interested reader can therefore replace accordingly
\[
\ADD\rightsquigarrow\SAD\,.
\]
However, we are human and so are most of our readers. Discussing at length semi-additive Mackey 2-functors would be somewhat misleading given that almost all examples we use are additive. We therefore require enrichment over abelian groups (not just abelian monoids) out of habit, convenience and social awareness.

\end{Rem}

It follows from \Cref{Rem:SAD} that, given a strict 2-functor $\MM\colon \GG^\op\too\CAT$ taking values in arbitrary categories and satisfying \Mack{1}-\Mack{4} as in \Cref{Def:Mackey-2-functor}, each category $\MM(G)$ must be semi-additive. Similarly, restriction~$i^*$ and (co)\,induction $i_!\cong i_*$ are additive functors for every~$i\in \JJ$ but the same is not necessarily true of~$u^*$ for 1-cells~$u$ not in~$\JJ$. Including the latter gives:

\begin{Def}
\label{Def:semi-additive-Mackey-2-functor}%
A \emph{semi-additive Mackey 2-functor} on~$\GG$ is a strict 2-functor $\MM\colon \GG^\op\too\SAD$, taking values in semi-additive categories and additive functors, and satisfying \Mack{1}-\Mack{4} as in \Cref{Def:Mackey-2-functor}.
\end{Def}

\bigbreak
\section{Separable monadicity}
\label{sec:monadicity}%
\medskip

To give an early simple application of Mackey 2-functors, we illustrate how the knowledge that $\MM(H)$ and $\MM(G)$ are part of the same Mackey 2-functor~$\MM$ for a group(oid)~$G$ and a subgroup(oid)~$H$ bear some consequence on so-called `separable monadicity'. Let us remind the reader.

In \cite{BalmerDellAmbrogioSanders15}, we proved with Sanders that many examples of `equivariant' categories $\MM(G)$ had the property that the category $\MM(H)$ associated to a subgroup~$H\le G$ could be described in terms of modules over a monad defined over the category~$\MM(G)$. In technical terms, this means that the adjunction $\Res^G_H\colon \MM(G)\adjto \MM(H)\noloc \CoInd_H^G$ satisfies monadicity. This property allows us to use descent techniques to analyze the extension of objects of~$\MM(H)$ to~$\MM(G)$, \ie extension of objects from the subgroup to the big group. In the case of tensor-triangulated categories, monadicity also yields a better understanding of the connections between the triangular spectra of~$\MM(G)$ and~$\MM(H)$. See~\cite{Balmer16}.

\begin{Thm}
\label{Thm:monadicity}%
Let $\MM\colon \groupoid^\op\to \ADD$ be any (rectified) Mackey 2-functor (\Cref{Def:Mackey-2-functor}) and $i\colon H\into G$ be a faithful functor in~$\groupoid$. Then the adjunction $i^*\adj i_*$ is \emph{monadic}, \ie the Eilenberg-Moore comparison functor
\[
E\colon\MM(H)\too A^G_H\MMod_{\MM(G)}
\]
between $\MM(H)$ and the Eilenberg-Moore category of modules in~$\MM(G)$ over the monad $A^G_H=i_*i^*\colon \MM(G)\to \MM(G)$ induces an equivalence on idempotent-com\-ple\-tions
(\Cref{Rem:completions})
\[
\MM(H)^\natural\isotoo\big(A^G_H\MMod_{\MM(G)}\big)^\natural=A^G_H\MMod_{\MM(G)^\natural}.
\]
In particular, $E$ is an equivalence if $\MM\colon \groupoid^\op\to \ADD$ takes values in idempotent-complete categories. Moreover, the monad $A^G_H\colon \MM(G)\to \MM(G)$ is separable.
\end{Thm}

\begin{proof}
This is a standard consequence of the existence of a natural section of the counit $\reps\colon i^*i_*\Rightarrow \Id$, which follows from~(Mack\,\ref{Mack-9}). Indeed, the multiplication $A^G_H\circ A^G_H\Rightarrow A^G_H$ is induced by~$\reps$ and the section of the latter tells us that $A^G_H$ is separable. It follows that every $A^G_H$-module is a direct summand of a free one, and therefore both $\MM(H)$ and $A^G_H\MMod_{\MM(G)}$ receive the Kleisli category of free $A^G_H$-modules as a `dense' subcategory, in a compatible way. (Here a subcategory of an additive category is called `dense' if every object of the big category is a direct summand of an object of the subcategory.) It follows that both categories $\MM(H)$ and $A^G_H\MMod_{\MM(G)}$ have the same idempotent-completion as the Kleisli category. See details in \cite[Lemma\,2.2]{BalmerDellAmbrogioSanders15}.
\end{proof}

\bigbreak
\section{Decategorification}
\label{sec:decategorification}%
\medskip

It is natural to discuss `decategorification' from Mackey 2-functors down to ordinary Mackey (1-)\,functors at this stage of the exposition, in order to facilitate understanding of our new definition. However, the treatment we present here will become clearer after the reader becomes familiar with the bicategories of (double) spans that will only appear in \Cref{ch:bicat-spans,ch:2-motives}. In particular, we are going to use the universal property of $\Spanhat(\GG;\JJ)$ as a black box (whose proof does not rely on the present section, of course).

In \Cref{app:old-Mackey}, we describe ordinary Mackey 1-functors on a fixed finite group~$G$ as additive functors on a suitable 1-category of spans, $\spanG$, built out of the 2-category~$\gpdG$ of groupoids faithful over~$G$ (\Cref{Def:gpdG}). Inspired by \Cref{Thm:Dress_Mackey_via_gpd}, we replace the 2-category~$\gpdG$ by other (2,1)-categories $\GG$ and consider proper classes~$\JJ$ of 1-cells in~$\GG$. For simplicity, the reader can assume that $(\GG,\JJ)$ satisfies \Cref{Hyp:GG} but this section makes sense in the greater generality of~\Cref{Hyp:G_and_I_for_Span}.

\begin{Def}
\index{$tspan$@${\tSpan}(\GG;\JJ)$}%
The \emph{category of spans over the (2,1)-category~$\GG$ (with respect to the class~$\JJ\subseteq\GG_1$)} is the 1-category
\[
\tSpan(\GG;\JJ)
\]
whose objects are the same as those of~$\GG$ and whose morphisms are equivalence classes of spans $G\loto{a} P\oto{b} H$ with $b\in\JJ$, where two such pairs are declared equivalent if there exists an equivalence between the two middle objects making the triangles commute up to isomorphism:
\[
\xymatrix@R=.5em@C=6em{
& P \ar[ld]_-{a} \ar[dd]_-{\simeq}^-{f} \ar[rd]^-{b}
\\
G \ar@{}[r]|-{\simeq} && H \ar@{}[l]|-{\simeq}
\\
& P' \ar[lu]^-{a'} \ar[ru]_-{b'}
&
}
\]
Composition is done in the usual way: Choose representatives for the fractions, construct the comma squares (compare \Cref{Def:Span-bicat}), and then retake equivalence classes. When $\JJ$ is not mentioned, we mean $\JJ=\all$ as always: $\tSpan(\GG):=\pih{\Span(\GG;\all)}$.
\end{Def}

\begin{Rem}
\label{Rem:alt_descr_Span}
The advanced readers who are already familiar with \Cref{ch:bicat-spans,ch:2-motives} will observe that $\tSpan(\GG;\JJ)$ is precisely the 1-truncation (\Cref{Not:htpy_cat}) of the bicategory of spans $\Span(\GG;\JJ)$ as in \Cref{Def:Span-bicat}, hence the notation. It is also the 1-truncation of the bicategory of spans of spans $\Spanhat(\GG;\JJ)$ as in \Cref{Def:Spanhat-bicat}
\[
\pih{\Span(\GG;\JJ)} = \pih{\big(\Span(\GG;\JJ)\big)} = \pih{\big(\Spanhat(\GG;\JJ)\big)}\,.
\]
This holds simply because $\Span(\GG;\JJ)$ and $\Spanhat(\GG;\JJ)$ have the same 0-cells and the same 1-cells and because the only invertible 2-cells of $\Spanhat(\GG;\JJ)$ are already in~$\Span(\GG;\JJ)$ by \Cref{Lem:isos-in-spans}. As a consequence the notion of Mackey functors for~$(\GG,\JJ)$ that we are about to consider will not see the difference between the two bicategories of spans studied in \Cref{ch:bicat-spans,ch:2-motives}.
\end{Rem}

\begin{Rem}
If the category~$\pih{\GG}$ (\Cref{Not:htpy_cat}) has enough pullbacks, the 1-category of spans $\tSpan(\GG;\JJ)$ has an alternative description where one \emph{first} takes the 1-truncation $\pih{\GG}$ and \emph{then} considers spans in this 1-category (in the spirit of \Cref{Def:ordinary-spans}, except that only morphisms in $\pih{\JJ}$ are allowed on the right). In particular, when $\JJ=\all$, the category $\tSpan(\GG)$ is nothing but $\widehat{\pih{\GG}}$.

If $\pih{\GG}$ does not have pullbacks, it is slightly abusive to view $\pih{\Span(\GG)}$ as the `category $\widehat{\pih{\GG}}$ of spans in $\pih{\GG}$' for composition still requires to chose representatives of spans in~$\GG$ and to work with iso-comma squares in~$\GG$. In other words, the composition of spans in $\pih{\Span(\GG)}$ still really depends on the underlying 2-category~$\GG$. (This issue did not appear with $\GG=\gpdG$ in~\Cref{app:old-Mackey}, since $\pih{(\gpdG)}$ admits pullbacks by \Cref{Cor:Gset_vs_gpdG}.) Alternatively, one should remember the relevant class of squares in $\pih{\GG}$, which might have an intrinsical characterization in~$\pih{\GG}$. These are sometimes called \emph{weak pullbacks}.
\end{Rem}

\begin{Def}
\label{Def:Mackey-functor-for-GI}%
\index{Mackey functor! -- over a (2,1)-category~$\GG$} \index{Mackey functor! -- for $(\GG,\JJ)$}%
A \emph{(generalized) Mackey functor over the (2,1)-category~$\GG$, with respect to the class~$\JJ\subseteq\GG_1$}, is an additive (\ie coproduct-preserving) functor $M\colon \tSpan(\GG;\JJ) \to \Ab$. As always when we do not specify~$\JJ$, a \emph{Mackey functor over~$\GG$} means that we have taken~$\JJ=\all$.

Explicitly, a Mackey functor for $(\GG,\JJ)$ consists of an abelian group $M(G)$ for every object $G\in \GG_0$, a homomorphism $a^*\colon M(G)\to M(H)$ for every 1-cell $a\in\GG_1(H,G)$ and a homomorphism $a_*\colon M(H)\to M(G)$ if furthermore $a\in\JJ_1(H,G)$. This data is subject to a few rules:
\begin{enumerate}[(1)]
\item
Additivity: The canonical morphism $M(G_1\sqcup G_2)\isoto M(G_1)\oplus M(G_2)$ is an isomorphism, for all $G_1,G_2\in\GG$.
\smallbreak
\item If two 1-cells $a\simeq b$ are isomorphic in the category $\GG(H,G)$ then $a^*=b^*$, and furthermore $a_*=b_*$ if $a,b$ belong to~$\JJ$.
\smallbreak
\item For every (iso)comma square (or Mackey square) in~$\GG$ with $i\in \JJ$
\[\vcenter{\xymatrix@C=14pt@R=14pt{
& (i/u) \ar[dl]_-{v} \ar[dr]^-{j}
 \ar@{}[dd]|(.5){\isocell{\gamma}}
\\
H \ar[dr]_-{i}
&& K \ar[dl]^-{u}
\\
&G
}}
\]
we have $u^*i_*=j_*v^*\colon M(H)\to M(K)$.
\end{enumerate}
\end{Def}

Once we establish the universal property of $\Spanhat(\GG;\JJ)$, there are obvious ways to recover ordinary Mackey 1-functors through decategorification of Mackey 2-functors. For instance, one can use the Grothendieck group~$K_0$ as follows. We shall expand these ideas in forthcoming work.

\begin{Prop} \label{Prop:decat_K0}
Let $\MM \colon \GG^{\op} \to \ADD$ be a Mackey 2-functor on $(\GG,\JJ)$ in the sense of \Cref{Def:Mackey-2-functor}. Then the composite $K_0 \circ \MM$ factors uniquely as
\[
\xymatrix{
\GG^{\op} \ar[r]^-{\MM} \ar[d] &
 \ADD \ar[d]^-{K_0} \\
\pih\Span(\GG;\JJ) \ar@{-->}[r]^-{M} & \Ab
}
\]
where $\GG^{\op}\to \pih\Span(\GG;\JJ)$ is the functor sending $u\colon H\to G$ to the equivalence class of the span $G \stackrel{u}{\leftarrow} H \oto{\Id} H$. This functor $M\colon \pih\Span(\GG;\JJ)\to \Ab$ is a generalized Mackey functor over~$\GG$, in the sense of \Cref{Def:Mackey-functor-for-GI}.
\end{Prop}

\begin{proof}
By the universal property of \Cref{Thm:UP-Spanhat}, the Mackey 2-functor $\MM$ factors through $\GG^{\op} \to \Spanhat$. As $\Ab$ is a 1-category, the composite $K_0 \circ \MM$ must factor through the quotient $\Spanhat \to \pih\Spanhat = \pih\Span$, as claimed. The resulting functor $M\colon \pih\Span \to \Ab$ is additive by the Additivity axiom for~$\MM$, hence is a Mackey functor.
\end{proof}

\begin{Rem} It is clear from the proof of Proposition~\ref{Prop:decat_K0} and the equality $\pih\Span =\pih\Spanhat$ that one does not really need $\MM$ to be a Mackey 2-functor: The left adjoints satisfying the (left) BC-formula would suffice for~$\MM$ to factor through $\Span$, and Additivity for it to yield a Mackey functor~$M$. This was already observed in~\cite{Nakaoka16a}.
\end{Rem}

\begin{Rem}
We shall discuss other `decategorifications' in subsequent work, for instance by considering Mackey 2-functors~$\MM$ whose values $\MM(G)$ are not mere additive categories but richer objects, like exact or triangulated categories, in which case the Grothendieck group~$K_0$ has a finer definition.
\end{Rem}

\end{chapter-two}
%
\chapter{Rectification and ambidexterity}
\label{ch:Theta}%
\bigbreak
\begin{chapter-three}

We want to justify our definition of `Mackey 2-functors' (\Cref{Def:Mackey-2-functor-intro}). There are many examples of categories depending on finite groups for which induction and co-induction coincide. Our main goal is to show that this happens systematically in various additive settings, including `stable homotopy'. This general ambidexterity result (Theorems~\ref{Thm:Theta-properties} and~\ref{Thm:ambidex-der}) explains why we do not treat induction and coinduction separately in the sequel. \Cref{Thm:ambidex-der} will also be the source of a large class of examples, as we shall discuss more extensively in \Cref{sec:add-der-Mackey}.

\bigbreak
\section{Self iso-commas}
\label{sec:self-iso}%
\medskip

The present section prepares for the rectification process of \Cref{sec:rectification}. We assemble here the ingredients which do not depend on a 2-functor $\MM\colon \groupoid^\op\to \ADD$ but only on constructions with iso-commas of groupoids. Specifically, we study the `self-iso-commas' obtained by considering~\eqref{eq:iso-comma-intro} in the case~$u=i$: We identify certain diagonal components of such self-iso-commas and record their functorial behavior with respect to~$i$.

\begin{Not}
\label{Not:self-ic}%
For every faithful~$i\colon H\into G$ in~$\groupoid$, we consider the `self-iso-comma' over the cospan~$\oto{i}\loto{i}$, that is, for~$i$ versus itself:
\begin{equation}
\label{eq:self-ic}%
\vcenter{
\xymatrix@C=14pt@R=14pt{
& (i/i) \ar[dl]_-{p_1} \ar[dr]^-{p_2}
\\
H \ar[dr]_-{i} \ar@{}[rr]|-{\isocell{\lambda}}
&& H \ar[dl]^-{i}
\\
&G
}}
\end{equation}
We start by isolating some special (`diagonal') connected components of~$(i/i)$.
\end{Not}

\begin{Prop}
\label{Prop:Delta}%
There exists a functor $\Delta_i\colon H\to (i/i)$
\[
\xymatrix@C=14pt@R=14pt{
& H \ar[d]^-{\Delta_i} \ar@/_4ex/[ldd]_-{\Id} \ar@/^4ex/[rdd]^-{\Id} & \\
& (i/i) \ar[ld]_-{p_1} \ar[dr]^-{p_2} & \\
H \ar[dr]_i \ar@{}[rr]|{\oEcell{\lambda}} && H \ar[dl]^i \\
&G &
}
\]
characterized by $p_1 \Delta_i=p_2 \Delta_i=\Id_H$ and $\lambda \Delta_i=\id_{i}$, namely $\Delta_i=\langle \Id_H, \Id_H, \id_{i}\rangle$ in the notation of \Cref{Def:comma}. This functor is fully faithful, that is, $\Delta_i$ induces an equivalence between $H$ and its essential image in~$(i/i)$.
\end{Prop}

\begin{proof}
Consider $x,y\in H$ and a morphism $(f,g)\colon (x,x,\id_{i(x)})\to (y,y,\id_{i(y)})$ between their images in~$(i/i)$ under~$\Delta_i$. The pair $f\colon x\to y$ and $g\colon x\to y$ in~$H$ must satisfy $\id_{i(y)}\circ i(f)=i(g)\circ\id_{i(x)}$ in~$G$, hence $f=g$ by the faithfulness of~$i$.
\end{proof}

\begin{Rem}
\label{Rem:roadmap-self-ic}%
We need to understand the behavior of the self-iso-comma~$(i/i)$ and the embedding~$\Delta_i$ in two situations. First we want to relate $(i/i)$ and $(j/j)$ in the presence of a 2-cell $\gamma\colon i v\Rightarrow u j$, typically but not necessarily a Mackey square, as in~\eqref{eq:Mackey-square}. Second, we want to relate $(i/i)$, $(j/j)$ and $(ij/ij)$ when~$i$ and $j$ are composable morphisms $\ointo{j}\ointo{i}$.
\end{Rem}

Here is the first setting:
\begin{Not}
\label{Not:i-j}%
We consider a 2-cell $\gamma\colon i v\isoEcell u j$
\begin{equation}
\label{eq:2-cell-for-i-j}%
\vcenter{\xymatrix@C=14pt@R=14pt{
& L \ar[dl]_-{v} \ar@{ >->}[dr]^-{j}
\\
H \ar@{ >->}[dr]_-{i} \ar@{}[rr]|-{\oEcell{\gamma}}
&& K \ar[dl]^-{u}
\\
&G
}}
\end{equation}
in which both~$i$ and $j$ are assumed faithful. (This square may be an iso-comma but we do not assume this here.) Consider the self-iso-comma~\eqref{eq:self-ic} associated to~$j$
\begin{equation}
\label{eq:self-ic-j}%
\vcenter{\xymatrix@C=14pt@R=14pt{
& (j/j) \ar[dl]_-{q_1} \ar[dr]^-{q_2}
\\
L \ar[dr]_-{j} \ar@{}[rr]|-{\oEcell{\rho}}
&& L \ar[dl]^-{j}
\\
&K
}}
\end{equation}
We compare the two iso-commas~\eqref{eq:self-ic} and~\eqref{eq:self-ic-j}, for~$i$ and for $j$ respectively:
\end{Not}

\begin{Prop}
\label{Prop:w}%
With notation as in~\ref{Not:i-j}, there exists a (unique) functor $w=w_{v,u,\gamma}\colon (j/j)\to (i/i)$ such that $p_1 w=vq_1$, $p_2 w=vq_2$ and such that the following two morphisms $ivq_1\Rightarrow ivq_2$ are equal (see \Cref{Def:comma}):
\begin{equation}
\label{eq:ic-f}%
\vcenter{\xymatrix@C=14pt@R=14pt{
& (j/j) \ar@{->}[d]_-{\exists!}^-{w} \ar@/_3ex/[ddl]_-{v q_1} \ar@/^3ex/[ddr]^-{v q_2}
&&&&& (j/j) \ar[ld]_-{q_1} \ar[rd]^-{q_2} \ar@{}[dd]|-{\oEcell{\rho}}
\\
& (i/i) \ar[dl]_-{p_1} \ar[dr]^-{p_2} \ar@{}[dd]|-{\oEcell{\lambda}}
&& =
&& L \ar[ld]_-{v} \ar[rd]^-{j} \ar@{}[dd]|-{\oEcell{\gamma}}
&& L \ar[rd]^-{v} \ar[ld]_-{j} \ar@{}[dd]|-{\oEcell{\gamma}\inv}
\\
H \ar[dr]_-{i} && H \ar[dl]^-{i}
&& H \ar@/_.5em/[rrd]_-{i}
&& K \ar[d]^-{u}
&& H \ar@/^.5em/[lld]^-{i}
\\
& G
&&&&& G &
}}
\end{equation}
In other words $w=\big\langle vq_1\,,\, vq_2\,,\, (\gamma\inv q_2 )(u\rho)(\gamma q_1 )\big\rangle$ and the following diagram of natural transformations commutes
\begin{equation}
\label{eq:lambda-gamma-rho}%
\vcenter{
\xymatrix@C=2em@R=1em@L=1ex{
& ip_1w \ar@{=>}[r]^-{\lambda}
& ip_2w \ar@{=}[rd]
\\
ivq_1 \ar@{=}[ru] \ar@{=>}[rd]_-{\gamma}
&&& ivq_2 \ar@{=>}[ld]^-{\gamma}
\\
& ujq_1 \ar@{=>}[r]^-{\rho}
& ujq_2
}}
\end{equation}
Moreover, $w\colon (j/j) \to (i/i)$ is compatible with the morphism $\Delta_i$ of \Cref{Prop:Delta}, namely we have a commutative diagram
\begin{equation}
\label{eq:Delta-nat}%
\vcenter{
\xymatrix{
L \ar[d]_-{v} \ar[r]^-{\Delta_j}
& (j/j) \ar[d]^-{w}
\\
H \ar[r]^-{\Delta_i}
& (i/i)\,.\!\!
}}
\end{equation}
\end{Prop}

\begin{proof}
The universal property of~$(i/i)$ as in \Cref{Def:comma}\,(a) gives us the wanted~$w$. We only need to verify~\eqref{eq:Delta-nat}. Let us describe the two functors
\[
w\Delta_j\,,\,\Delta_iv\colon L \to (i/i) \,.
\]
Since $\Delta_i=\langle \Id_H,\Id_H,\id_i\rangle$, we have $\Delta_i\circ v=\langle v, v, \id_{iv}\rangle$. On the other hand, since $w=\langle vq_1\,,\, vq_2\,,\, (\gamma\inv q_2 )(u\rho)(\gamma q_1 )\rangle$, we have that $w\circ \Delta_j$ is equal to the morphism
\[
\langle vq_1\Delta_j\,,\, vq_2\Delta_j\,,\, (\gamma\inv q_2 \Delta_j)(u\rho\Delta_j)(\gamma q_1 \Delta_j)\rangle=\langle v,v, \id_{iv}\rangle\,.
\]
The latter uses $q_1 \Delta_j=\Id$ and $q_2 \Delta_j=\Id$ and the fact that $\rho\Delta_j=\id_{j}$ by construction of~$\Delta_j=\langle j,j,\id_{j}\rangle$, thus allowing $\gamma$ and $\gamma\inv$ to cancel out.
\end{proof}

To state the next result, we need to isolate some class of 2-cells as in~\eqref{eq:2-cell-for-i-j}.
\begin{Def}
\label{Def:partial-Mackey}%
We say that a 2-cell in~$\groupoid$
\[
\vcenter{\xymatrix@C=14pt@R=14pt{
& L \ar[dl]_-{v} \ar[dr]^-{j}
\\
H \ar[dr]_-{i} \ar@{}[rr]|-{\isocell{\gamma}}
&& K \ar[dl]^-{u}
\\
&G
}}
\]
is a \emph{partial Mackey square} if the induced functor $\langle v,j,\gamma\rangle\colon L\to (i/u)$ is fully faithful. In other words, $L$ is equivalent to a union of connected components of the iso-comma~$(i/u)$, although perhaps not all of them. When we assume $j$ faithful (as we did in \Cref{Not:i-j}) then the condition simply means that $\langle v,j,\gamma\rangle\colon L\to (i/u)$ is full. The latter can be stated explicitly by saying that for every objects $x,y\in L$, every $h\colon v(x)\isoto v(y)$ and $k\colon j(x)\isoto j(y)$ such that $u(k) \gamma_x =\gamma_y i(h)$, there exists a morphism $\ell\colon x\isoto y$ such that $v(\ell)=h$ and $j(\ell)=k$.
\end{Def}

\begin{Exa}
\label{Exa:partial-alpha}%
Of course, every Mackey square (\Cref{Def:Mackey-square}) is a partial Mackey square. In particular, for every 2-cell $\alpha\colon i\isoEcell i'$ between faithful $i,i'\colon H\into G$, the following square is a (partial) Mackey square:
\[
\xymatrix@C=14pt@R=14pt{
& H \ar@{=}[ld] \ar@{ >->}[rd]^-{i'} \ar@{}[dd]|-{\isocell{\alpha}}
\\
H \ar@{ >->}[rd]_-{i}
&& G \ar@{=}[ld]
\\
& G}
\]
\end{Exa}

\begin{Exa}
\label{Exa:partial-diag}%
For every faithful $i\colon H\into G$ the following is a partial Mackey square by \Cref{Prop:Delta}:
\[
\xymatrix@C=14pt@R=14pt{
& H \ar@{=}[ld] \ar@{=}[rd] \ar@{}[dd]|-{\isocell{\id}}
\\
H \ar@{ >->}[rd]_-{i}
&& H \ar@{->}[ld]^-{i}
\\
& G}
\]
\end{Exa}

One can easily generalize the above example as follows.
\begin{Exa}
\label{Exa:partial-ij}%
Consider two composable faithful functors $K\ointo{j} H \ointo{i} G$. Then the following is a partial Mackey square
\[
\xymatrix@C=14pt@R=14pt{
& K \ar@{=}[ld] \ar@{ >->}[rd]^-{j} \ar@{}[dd]|-{\isocell{\id}}
\\
K \ar@{ >->}[rd]_-{ij}
&& H \ar@{->}[ld]^-{i}
\\
& G}
\]
as one can directly verify that $\langle \Id_K,j,\id\rangle\colon K\to (ij/i)$ is fully faithful.
\end{Exa}

In \Cref{Prop:w} we saw that $w=w_{v,u,\gamma}\colon (j/j)\to (i/i)$ sends the diagonal component $\Delta_j(H)$ of~$(j/j)$ to the diagonal component of~$(i/i)$ whenever we have a 2-cell $\gamma\colon i v \Rightarrow u j$ as in~\eqref{eq:2-cell-for-i-j}. We need a condition for the non-diagonal components to be sent by~$w$ to non-diagonal components as well, which is not automatic. The partial Mackey squares do the job:

\begin{Prop}
\label{Prop:w-partial-Mackey}%
Consider a 2-cell in~$\groupoid$ as in~\eqref{eq:2-cell-for-i-j}
\[
\vcenter{\xymatrix@C=14pt@R=14pt{
& L \ar[dl]_-{v} \ar@{ >->}[dr]^-{j}
\\
H \ar@{ >->}[dr]_-{i} \ar@{}[rr]|-{\isocell{\gamma}}
&& K \ar[dl]^-{u}
\\
&G
}}
\]
and recall the induced 1-cell $w=w_{v,u,\gamma}\colon (j/j)\to (i/i)$ of \Cref{Prop:w}. Suppose that the above is a partial Mackey square (\Cref{Def:partial-Mackey}). Then for each connected component $D\subset (j/j)$ which does not meet the image of~$\Delta_j$ the connected component of~$w(D)$ in~$(i/i)$ does not meet the image of~$\Delta_i$.
\end{Prop}

\begin{proof}
This is a lengthy but straightforward exercise. Here is an outline of the contrapositive. Let $(x,y,k)\in (j/j)$ be an object such that $w(x,y,k)$ is isomorphic in~$(i/i)$ to~$\Delta_i(z)$ for some $z\in H$. We want to show that $(x,y,k)$ is isomorphic to~$\Delta_j(x)$ in~$(j/j)$. Let $(h_1,h_2)\colon w(x,y,k)\isoto \Delta_i(z)$ be an isomorphism in~$(i/i)$. Explicitly, $w(x,y,k)=(v(x),v(y),\gamma_y\inv u(k)\gamma_x)$. Consider $h:=h_2\inv h_1\colon v(x)\isoto v(y)$ and the given $k\colon j(x)\isoto j(y)$. They satisfy $u(k) \gamma_x =\gamma_y i(h)$ because $(h_1,h_2)$ is a morphism in~$(i/i)$. By the explicit formulation of the square~\eqref{eq:2-cell-for-i-j} being a partial Mackey square (\Cref{Def:partial-Mackey}) there exists $\ell\colon x\isoto y$ such that $v(\ell)=h$ and $j(\ell)=k$. We can then verify that $(\id_x,\ell)\colon \Delta_j(x)\isoto (x,y,k)$ is an isomorphism in~$(j/j)$, as wanted.
\end{proof}

We now turn to the second behavior of~$(i/i)$ announced in \Cref{Rem:roadmap-self-ic}. Here is the precise setting:

\begin{Not}
\label{Not:ij-comp}%
Let $K \stackrel{j}{\into} H \stackrel{i}{\into} G$ be two composable faithful functors, and consider the three associated iso-comma squares:
\[
\vcenter{
\xymatrix@C=14pt@R=14pt{
& (i/i) \ar[dl]_-{p_1} \ar[dr]^-{p_2}
\\
H \ar[dr]_-{i} \ar@{}[rr]|-{\isocell{\lambda}}
&& H \ar[dl]^-{i}
\\
&G
}}
\qquad
\vcenter{\xymatrix@C=14pt@R=14pt{
& (ij/ij) \ar[dl]_{r_1} \ar[dr]^-{r_2}
\\
K \ar[dr]_-{ij} \ar@{}[rr]|-{\isocell{\sigma}}
&& K \ar[dl]^-{ij}
\\
&G
}}
\qquad
\vcenter{
\xymatrix@C=14pt@R=14pt{
& (j/j) \ar[dl]_-{q_1} \ar[dr]^-{q_2}
\\
K \ar[dr]_-{j} \ar@{}[rr]|-{\isocell{\rho}}
&& K \ar[dl]^-{j}
\\
&H
}}
\]
Consider also the two functors
\[
(i/i) \stackrel{\,w'}{\longleftarrow} (ij/ij) \stackrel{\,\,w''}{\longleftarrow} (j/j)
\]
induced by the universal property of iso-comma squares and defined by their components $w''=\langle q_1, q_2, i\rho \rangle$ and $w'=\langle j r_1, j r_2, \sigma \rangle$, see \Cref{Def:comma}\,\eqref{it:comma-a}. One can readily verify that $w''=w_{\Id_K,i,\id}$ and $w'=w_{j,\Id_G,\id}$ in the notation of \Cref{Prop:w} applied to the following two squares, respectively:
\begin{equation} \label{eq:ij-squares}
\vcenter{\xymatrix@C=14pt@R=14pt{
& K \ar@{=}[dl] \ar[dr]^-{j}
\\
K \ar[dr]_-{ij} \ar@{}[rr]|-{\isocell{\id}}
&& H \ar[dl]^-{i}
\\
&G
}}
\qquadtext{and}
\vcenter{\xymatrix@C=14pt@R=14pt{
& K \ar[dl]_j \ar[dr]^-{ij}
\\
H \ar[dr]_-{i} \ar@{}[rr]|-{\isocell{\id}}
&& G \ar@{=}[dl]
\\
&G
}}
\end{equation}
(The `conjugation' by $\gamma$ in \Cref{Prop:w} disappears here since $\gamma=\id$.)
\end{Not}

Then the connected components of~$(ij/ij)$ are organized as follows:

\begin{Prop}
\label{Prop:ij-components}%
Retain \Cref{Not:ij-comp}. Each connected component $C \subseteq (ij/ij)$ falls in \emph{exactly one} of the following three cases:
\begin{enumerate}[\rm(1)]
\item
\label{it:ij-comp-1}%
$C$ belongs to the (essential) image of the diagonal $\Delta_{ij}(K)$ and the inclusion $C\hookrightarrow (ij/ij)$ lifts uniquely along~$w''$ to an equivalence between $C$ and a connected component $D$ of the diagonal $\Delta_j(K)\subseteq (j/j)$.
\smallbreak
\item
\label{it:ij-comp-2}%
$C$ is disjoint from $\Delta_{ij}(K)$ and $w'(C)$ is disjoint from $\Delta_{i}(H)$ in~$(i/i)$.
\smallbreak
\item
\label{it:ij-comp-3}%
$C$ is disjoint from $\Delta_{ij}(K)$ and the inclusion $C\hookrightarrow (ij/ij)$ lifts uniquely along $w''$ to an equivalence between $C$ and a connected component $D$ of $(j/j)$ which is disjoint from the diagonal $\Delta_j(K)$.
\end{enumerate}
\end{Prop}

\begin{proof}
By~\eqref{eq:Delta-nat} applied to the two squares in~\eqref{eq:ij-squares}, we have a commutative diagram
\begin{equation} \label{Eq:diagonal-comparisons}
\vcenter{\xymatrix@R=1.5em{
K \ar@{=}[d] \ar[r]^-{\Delta_j} & (j/j) \ar[d]^{w''} \\
K \ar[d]_j \ar[r]^-{\Delta_{ij}} & (ij/ij) \ar[d]^{w'} \\
H \ar[r]^{\Delta_i} & (i/i)\,.\!\!
}}
\end{equation}
We may apply \Cref{Prop:w-partial-Mackey} to the left square in~\eqref{eq:ij-squares}, which is a partial Mackey square by \Cref{Exa:partial-ij}, and conclude the following: If a connected component $D$ of $(j/j)$ does not meet the diagonal $\Delta_j$ then $w''D$ does not meet $\Delta_{ij}$.

Now the right square in~\eqref{eq:ij-squares} is not partial Mackey in general, so we cannot apply \Cref{Prop:w-partial-Mackey} to it. Instead, we can prove:
\begin{Lem} \label{Lem:w'-components}
If a component $C\subseteq (ij/ij)$ is such that $w'C$ meets~$\Delta_{i}$ in~$(i/i)$, then the inclusion $C\hookrightarrow (ij/ij)$ lifts uniquely along $w''$ to an equivalence between~$C$ and a component of $(j/j)$.
\end{Lem}
\begin{proof}
The facts that the lifting, if it exists, is unique and induces an equivalence of components, are both easy consequences of $w''$ being fully faithful and injective on objects. The latter are seen by direct inspection: On objects, $w''$ sends $(x,y,h)$ to $(x,y,i(h))$ hence is injective by the faithfulness of~$i$; on maps, $w''$ sends $(k_1,k_2)$ to $(k_1, k_2)$, and the commutativity conditions for a pair $(k_1, k_2)\in (\Mor K)^2$ to define a morphism in $(j/j)$ or $(ij/ij)$ are identical, again by the faithfulness of~$i$.

To prove the existence of the lifting, let $(x,y,g\colon ij(x)\to ij(y))\in D$ be such that $w'(x,y,g)$ is isomorphic to an object in the image of~$\Delta_{i}$, say $(h_1,h_2)\colon w'(x,y,g)\stackrel{\sim}{\to}\Delta_{i}(z)$. Thus we have isomorphisms $h_1\colon j(x)\to z$ and $h_2\colon j(y)\to z$ such that $i(h_2)g=i(h_1)$. As $w''$ is so nice, to construct the lifting of $D\hookrightarrow (ij/ij)$ it suffices to lift the object $(x,y,g)$. Setting $h:=h_2^{-1}h_1\colon j(x)\to j(y)$, we see that $(x,y,h)$ is an object of $(j/j)$ such that $w''(x,y,h)=(x,y,g)$.
\end{proof}
Now the trichotomy claimed in the proposition follows immediately from the commutativity of~\eqref{Eq:diagonal-comparisons} and the above properties of $w''$ and $w'$ with respect to components.
\end{proof}

\bigbreak
\section{Comparing the legs of a self iso-comma}
\label{sec:self-iso-comma-legs}%
\medskip

We proceed with our gentle build up towards rectification (\Cref{sec:rectification}) by assuming the existence of a strict 2-functor $\MM\colon \groupoid^\op\to \ADD$ which satisfies additivity, without requesting any of the further properties of Mackey 2-functors. Using only this and the self-iso-commas~$(i/i)$ of \Cref{sec:self-iso} we can already compare some of the functors. For the sake of generality we replace the 2-category $\groupoid$ by any 2-subcategory~$\GG$ as in \Cref{Hyp:GG}. Note that all constructions performed in \Cref{sec:self-iso} remain inside~$\GG$ since the latter is assume closed under iso-commas (etc). We begin with a very simple natural transformation:

\begin{Prop}
\label{Prop:delta}%
Let $i\colon H\into G$ in~$\GG$ be faithful and let $\MM\colon \GG^\op\to \ADD$ be a strict 2-functor satisfying additivity~\Mack{1}. Consider the self-iso-comma of \Cref{Not:self-ic}
\[
\vcenter{
\xymatrix@C=14pt@R=14pt{
& (i/i) \ar[dl]_-{p_1} \ar[dr]^-{p_2}
\\
H \ar[dr]_-{i} \ar@{}[rr]|-{\isocell{\lambda}}
&& H \ar[dl]^-{i}
\\
&G
}}
\]
and the `diagonal' embedding $\Delta_i\colon H\into (i/i)$ of \Cref{Prop:Delta}. Then there exists a unique natural transformation $\delta_i\colon p_1^*\Rightarrow p_2^*$ of functors $\MM(H)\to \MM(i/i)$ with the following properties:
\begin{enumerate}[\rm(1)]
\item
The natural transformation $\Delta_i^*(\delta_i)$ from $\Delta_i^*p_1^*=\Id_{\MM(H)}$ to $\Delta_i^*p_2^*=\Id_{\MM(H)}$ equals the identity (of the functor $\Id_{\MM(H)}$).
\smallbreak
\item
For every connected component $\incl_C\colon C\into(i/i)$ of~$(i/i)$ disjoint from $\Delta_i(H)$, the natural transformation $\incl_C^*(\delta_i)$ from $\incl_C^*p_1^*$ to $\incl_C^*p_2^*$ equals zero.
\end{enumerate}
\end{Prop}

\begin{proof}
By additivity and \Cref{Prop:Delta}, we have an equivalence
\[
\xymatrix@C=7em{
\MM(i/i) \ar[r]^-{\big(\Delta_i^*\quad (\incl_C^*)_{C}\big)}
& \MM(H) \oplus \bigoplus_{C\cap \Delta_i(H)\,=\,\varnothing}\MM(C)
}
\]
where $C$ runs among the connected components of~$(i/i)$ disjoint from~$\Delta_i(H)$, since the union of the other components is precisely the essential image of~$\Delta_i$. As in the statement, we denote by $\incl_C\colon C\into (i/i)$ the inclusions. Therefore, we can uniquely describe the morphism $\delta_i\colon p_1^*\Rightarrow p_2^*$ by its image under this equivalence.
\end{proof}

\begin{Rem}
\label{Rem:roadmap-delta}%
By analogy with \Cref{Rem:roadmap-self-ic}, we now want to discuss the behavior of~$\delta_i\colon p_1^* \overset{\sim}{\Rightarrow} p_2^*$ when~$i$ moves as in~\Cref{Not:i-j}.
\end{Rem}

\begin{Prop}
\label{Prop:delta-nat}%
Let $\MM\colon \GG^\op\to \ADD$ be a strict 2-functor satisfying additivity~\Mack{1}. Consider a 2-cell in~$\GG$ with~$i$ and $j$ faithful as in \Cref{Not:i-j}
\begin{equation}
\label{eq:2-cell-for-delta-nat}%
\vcenter{\xymatrix@C=14pt@R=14pt{
& L \ar[dl]_-{v} \ar@{ >->}[dr]^-{j}
\\
H \ar@{ >->}[dr]_-{i} \ar@{}[rr]|-{\isocell{\gamma}}
&& K \ar[dl]^-{u}
\\
&G
}}
\end{equation}
and the self-iso-commas~\eqref{eq:self-ic} and~\eqref{eq:self-ic-j} for~$i$ and~$j$ respectively:
\[
\vcenter{
\xymatrix@C=14pt@R=14pt{
& (i/i) \ar[dl]_-{p_1} \ar[dr]^-{p_2}
\\
H \ar[dr]_-{i} \ar@{}[rr]|-{\isocell{\lambda}}
&& H \ar[dl]^-{i}
\\
&G
}}
\qquadtext{and}
\vcenter{
\xymatrix@C=14pt@R=14pt{
& (j/j) \ar[dl]_-{q_1} \ar[dr]^-{q_2}
\\
L \ar[dr]_-{j} \ar@{}[rr]|-{\isocell{\rho}}
&& L \ar[dl]^-{j}
\\
&K\,.\!\!
}}
\]
The two natural transformations $\delta_i\colon p_1^*\Rightarrow p_2^*$ and $\delta_j\colon q_1 ^*\Rightarrow q_2 ^*$ of \Cref{Prop:delta} and the functor $w=w_{v,u,\gamma}\colon (j/j)\to (i/i)$ of \Cref{Prop:w} satisfy the following relations. The functors $w^*p_1^*=q_1 ^*v^*$ and $w^*p_2^*=q_2 ^*v^*$ relate the following categories: $\MM(H)\to \MM(j/j)$ and we can compare the morphisms $w^*\delta_i$ and $\delta_jv^*$
\[
\xymatrix@R=1em@L=1ex{
w^* p_1^* \ar@{=>}[r]^-{w^*\delta_i} \ar@{=}[d] \ar@{}[rd]|-{?}
& w^* p_2^* \ar@{=}[d]
\\
q_1 ^* v^* \ar@{=>}[r]_-{\delta_j v^*}
& q_2 ^* v^*
}
\]
one connected component of~$(j/j)$ at a time, by additivity.
\begin{enumerate}[\rm(a)]
\item
\label{it:delta-nat-a}%
Let $\Delta_j\colon L\into (j/j)$ be the `diagonal' embedding of \Cref{Prop:Delta}. Then $\Delta_j^*(w^*\delta_i)=\Delta_j^*(\delta_jv^*)$.
\smallbreak
\item
\label{it:delta-nat-b}%
Let $\incl_D\colon D\into (j/j)$ be a connected component of~$(j/j)$ such that $w(D)$ is disconnected from $\Delta_i(H)$, that is, no connected component of~$(i/i)$ meets both $w(D)$ and $\Delta_i(H)$. Then $\incl_D^*(w^*\delta_i)=\incl_D^*(\delta_jv^*)$ as well.
\smallbreak
\item
\label{it:delta-nat-c}%
If the square~\eqref{eq:2-cell-for-delta-nat} is partial Mackey (\Cref{Def:partial-Mackey}) then $w^*\delta_i=\delta_jv^*$.
\end{enumerate}
\end{Prop}

\begin{proof}
Recall from \Cref{Prop:w} that $p_1 w=vq_1$ and $p_2 w=vq_2$ and therefore the above statement makes sense. The target category, $\MM(j/j)$, decomposes as a product over the connected components of~$(j/j)$. The latter components are of two sorts, those meeting the image of~$\Delta_j\colon L\to (j/j)$, discussed in~\eqref{it:delta-nat-a}, and those $D\subset (j/j)$ which do not meet the image of~$\Delta_j$. For~\eqref{it:delta-nat-a}, applying $\Delta_j^*\colon \MM(j/j)\to \MM(L)$, our two morphisms become the two morphisms in the third square from the top in the following diagram, which we want to show is commutative
\[
\xymatrix@R=1em@L=1ex{
v^* \ar@{=>}[r]^-{\id}
& v^*
\\
v^* \Delta_i^* p_1^* \ar@{=>}[r]^-{v^* \Delta_i^* \delta_i} \ar@{=}[u]
& v^* \Delta_i^* p_2^* \ar@{=}[u]
\\
\Delta_j^*w^* p_1^* \ar@{=>}[r]^-{\Delta_j^*w^*\delta_i} \ar@{=}[d] \ar@{=}[u]
& \Delta_j^*w^* p_2^* \ar@{=}[d] \ar@{=}[u]
\\
\Delta_j^*q_1 ^* v^* \ar@{=>}[r]^-{\Delta_j^*\delta_j v^*} \ar@{=}[d]
& \Delta_j^*q_2 ^* v^* \ar@{=}[d]
\\
v^* \ar@{=>}[r]^-{\id}
& v^*
}%
\]
The `outer' square clearly commutes. The bottom square commutes by $q_1 \Delta_j=\Id=q_2 \Delta_j$ and the characterization of~$\delta_j$ on the components corresponding to~$\Delta_j$ given in \Cref{Prop:delta}\,(1) for the faithful functor~$j$ (instead of~$i$). The second square from the top commutes by $w\Delta_j=\Delta_i v$, see~\eqref{eq:Delta-nat}. The top square commutes by $p_1 \Delta_i=\Id=p_2 \Delta_i$ and the characterization of~$\delta_i$ on the components corresponding to~$\Delta_i$ given in \Cref{Prop:delta}\,(1) for~$i$. So the third square must commute as well, hence~\eqref{it:delta-nat-a}.

A connected components $D\subset (j/j)$ as in~\eqref{it:delta-nat-b} is necessarily disjoint from~$\Delta_j(L)$ for $w\Delta_j=\Delta_i v$ and we assume $w(D)$ disjoint from~$\Delta_i(H)$. So we can apply \Cref{Prop:delta}\,(2) for $\delta_i$ and for $\delta_j$ to see that the restriction of both $w^*\delta_i$ and $\delta_j v^*$ are zero. Hence they agree there as well.

For~\eqref{it:delta-nat-c} the assumption that we start with a partial Mackey square and \Cref{Prop:w-partial-Mackey} tell us that every connected component $D\subset(j/j)$ which is disjoint from~$\Delta_j(L)$ satisfies the hypothesis of~\eqref{it:delta-nat-b}. So combining~\eqref{it:delta-nat-a} and~\eqref{it:delta-nat-b} we see that all components of $w^*\delta_i$ and $\delta_j v^*$ coincide, hence the two morphisms are equal.
\end{proof}

\begin{Rem}
The 2-functors $\MM\colon \GG^\op\to \ADD$ which appeared in this section were only assumed to satisfy additivity. We did not use induction or coinduction, and \afortiori\ we did not invoke base-change formulas. This will change in the next section.
\end{Rem}

\bigbreak
\section{The canonical morphism $\Theta$ from left to right adjoint}
\label{sec:Theta}%
\medskip

In this section, we discuss the fundamental connection between the left and the right adjoint of restriction in the presence of additivity and of BC-formulas on both sides, without assuming ambidexterity. The main result, \Cref{Thm:Theta-properties}, sets the stage for our Rectification \Cref{Thm:rectification} and for \Cref{Thm:ambidex-der} where we show that every additive derivator yields a Mackey 2-functor.

As in the previous section, $\GG\subseteq \groupoid$ is a sub-2-category satisfying \Cref{Hyp:GG} and $\JJ$ denotes the faithful functors in~$\GG$.

Let us begin with some easy restrictions.

\begin{Rem}
\label{Rem:Mack-5-6}%
Let $\MM\colon \GG^\op\to \ADD$ be a strict 2-functor which satisfies additivity~\Mack{1} and existence of induction and coinduction~\Mack{2}.

Additivity allows us to reduce choices of units-counits to the case of connected groupoids. Indeed, suppose we have chosen the units and counits for the adjunctions $i_!\adj i^*\adj i_*$ whenever $i\colon H\into G$ is faithful and $H$ and $G$ are connected, then for $i=i_1\sqcup i_2\colon H_1\sqcup H_2\into G$, under the identification $\MM(H_1\sqcup H_2)\cong \MM(H_1)\oplus \MM(H_2)$, we can set
\[
(i_1\sqcup i_1)_!=\big((i_1)_! \ (i_2)_!\big)
\qquadtext{and}
(i_1\sqcup i_1)_*=\big((i_1)_* \ (i_2)_*\big)
\]
with the obvious `diagonal' units and counits as explained in \Cref{Rem:add-adjoint}. In other words, we can easily arrange Additivity of adjoints as in~(Mack\,\ref{Mack-5}) of \Cref{Thm:rectification-intro} (without assuming necessarily $i_!\simeq i_*$, just one side at a time).

Similarly, for every~$i$ (between connected groupoids) such that $i^*$ is an equivalence, we can choose $i_!=i_*$ to be any inverse equivalence $(i^*)\inv$ and the units and counits to be the chosen isomorphisms $(i^*)\inv i^*\cong \Id$, $\Id\cong (i^*)\inv i^*$ and their inverses. When~$i$ is an identity, we can even pick $i_!=i_*=\Id$ and all units identitities. Combined with~(Mack\,\ref{Mack-5}) for non-connected~$i$, we can assume without loss of generality that (Mack\,\ref{Mack-6}) holds true.

So the first two rectifications of \Cref{Thm:rectification-intro} are easy to obtain.
\end{Rem}

\begin{Conv}
\label{Conv:Mack-5-6}%
When using explicit units and counits for $i_!\adj i^*\adj i_*$ we tacitly assume that they satisfy (Mack\,\ref{Mack-5}) and (Mack\,\ref{Mack-6}).
\end{Conv}

So, let us gather all hypotheses in one place:
\begin{Hyp}
\label{Hyp:Theta}%
Let $\MM\colon \GG^\op\to \ADD$ be a strict 2-functor satisfying additivity~\Mack{1}, existence of adjoints on both sides~\Mack{2} for restriction along faithful morphisms and the BC-formulas on both sides~\Mack{3}. We apply \Cref{Conv:Mack-5-6}.
\end{Hyp}

\begin{Thm}
\label{Thm:Theta}%
Let $\MM\colon \GG^\op\to \ADD$ satisfy all the properties of a Mackey 2-functor, except perhaps ambidexterity, as in \Cref{Hyp:Theta}. Let $i\colon H\into G$ be faithful and consider the following self-iso-comma as in \Cref{Not:self-ic}:
\begin{equation}
\label{eq:Theta}%
\vcenter{\xymatrix@C=14pt@R=14pt{
& (i/i) \ar[ld]_-{p_1} \ar[dr]^-{p_2} & \\
H \ar[dr]_i \ar@{}[rr]|{\isocell{\lambda}} && H \ar[dl]^i \\
&G &
}}
\end{equation}
Then, there exists a unique morphism
\[
\Theta_i\colon i_!\Rightarrow i_*
\]
such that the following morphism $p_1^*\Rightarrow p_2^*$
\begin{equation}
\label{eq:Theta-delta}%
\vcenter{\xymatrix@R=2em@L=1ex{
&&& p_1^* i^* i_* \ar@{=>}[rd]^-{\lambda^* i_*}
\\
\delta_i\colon
& p_1^* \ar@{=>}[r]^-{\Displ\eta}_-{i_!\adj i^*}
& p_1^* i^* i_! \ar@{=>}[rr]^-{\Displ\lambda^* \, \Theta_i} \ar@{=>}[ru]^-{p_1^*i^*\Theta_i} \ar@{=>}[rd]_-{\lambda^* i_!}
&& p_2^* i^* i_* \ar@{=>}[r]^-{\Displ\eps}_-{i^*\adj i_*}
& p_2^*
\\
&&& p_2^* i^* i_! \ar@{=>}[ru]_-{p_2^*i^*\Theta_i}
}}
\end{equation}
is the canonical morphism $\delta_i\colon p_1^*\Rightarrow p_2^*$ of \Cref{Prop:delta}.
\end{Thm}

The idea of the proof is to show that the construction $\theta\mapsto \eps\circ (\lambda^* \theta) \circ \eta$ of the statement is an isomorphism between suitable sets of natural transformations, so that we can define~$\Theta_i$ by deciding its image under this construction to be whatever we wish, for instance our dear~$\delta_i$. We need some preparation.

\begin{Not}
For the sake of completeness, we recall the formulas for the mates. Given a 2-cell
\begin{equation}
\label{eq:2-cell-for-magic}%
\vcenter{\xymatrix@C=14pt@R=14pt{
& L \ar[dl]_-{v} \ar@{ >->}[dr]^-{j}
\\
H \ar@{ >->}[dr]_-{i} \ar@{}[rr]|-{\isocell{\gamma}}
&& K \ar[dl]^-{u}
\\
&G
}}
\end{equation}
iso-comma or not, but with~$i$ and $j$ faithful, we have
\begin{align*}
\gamma_! & =
\xymatrix@L=1ex{
{\Big(} j_! v^* \ar@{=>}[r]^-{\eta}_-{i_! \adj i^*}
& j_! v^*i^*i_! \ar@{=>}[r]^-{\gamma^*}
& j_! j^*u^*i_! \ar@{=>}[r]^-{\eps}_-{j_! \adj j^*}
& u^*i_! {\Big)}}
\\
(\gamma\inv)_* & = \xymatrix@L=1ex{
{\Big(} u^* i_* \ar@{=>}[r]^-{\eta}_-{j^* \adj j_*}
& j_* j^* u^* i_* \ar@{=>}[r]^-{(\gamma\inv)^*}
& j_* v^* i^* i_* \ar@{=>}[r]^-{\eps}_-{i^* \adj i_*}
& j_* v^* {\Big)}\,.
}
\end{align*}
Base-change \Mack{3} tells us that these are isomorphisms when~\eqref{eq:2-cell-for-magic} is an iso-comma. If $u$ and $v$ are also faithful, as will happen for self-iso-commas, we have by symmetry (\Cref{Exa:inv-comma-square}):
\begin{equation}
\label{eq:gamma!*-special}%
(\gamma\inv)_!\colon v_!j^*\Rightarrow i^*u_!
\qquadtext{and}
\gamma_*\colon i^* u_*\Rightarrow v_*j^*\,.
\end{equation}
In the case of the self-iso-commas~\eqref{eq:self-ic} and~\eqref{eq:self-ic-j}, these base-change isomorphisms specialize to the following isomorphisms:
\begin{eqnarray}
\label{eq:lambda!*}%
\lambda_!\colon {p_2}_!p_1^*\stackrel{\sim}{\Rightarrow} i^*i_!
& \quadtext{and}
& \lambda_*\colon i^* i_*\stackrel{\sim}{\Rightarrow} {p_1}_*p_2^*
\\
\label{eq:rho!*}%
\rho_!\colon {q_2}_!q_1 ^*\stackrel{\sim}{\Rightarrow} j^*j_!
& \quadtext{and}
& \rho_*\colon j^* j_*\stackrel{\sim}{\Rightarrow} {q_1}_*q_2 ^*\,.\!\!
\end{eqnarray}
\end{Not}

The critical point is the following result:
\begin{Prop}
\label{Prop:nabla}%
Consider a faithful $i\colon H\into G$ and the corresponding self-iso-comma square as in~\eqref{eq:self-ic}. Then for any category~$\cat{C}$ and any pair of parallel functors $F,F'\colon \cat{C}\to \MM(H)$, we have a natural isomorphism
\[
\nabla_{i}\colon [\,i_!\,F\,,\,i_*\,F'\,]\overset{\sim}{\longrightarrow} [\,p_1^*\,F\,,\,p_2^*\,F'\,]
\]
where we denote by~$[-,-]$ the sets of natural transformations (in $\MM(G)^{\cat{C}}$ and $\MM(i/i)^{\cat{C}}$ respectively). More precisely, using the isomorphisms $\lambda_!\colon {p_2}_!p_1^*\stackrel{\sim}{\Rightarrow} i^* i_!$ and $\lambda_*\colon i^* i_*\stackrel{\sim}{\Rightarrow} {p_1}_* p_2^*$ of~\eqref{eq:lambda!*}, the following diagram commutes and defines~$\nabla_{i}$:
\begin{equation}
\label{eq:nabla-FF'}%
\vcenter{
\xymatrix{
& [\,i_!F,i_*F'] \ar[rd]^-{\adj}_-{\cong} \ar[ld]^-{\cong}_-{\adj} \ar[ddd]^-{\Displ\nabla_{i}}_{\cong}
\\
[F,i^* i_*F'] \ar[d]_-{[\id,\lambda_*]}^-{\cong}
&& [i^*i_!F,F'] \ar[d]^-{[\lambda_!,\id]}_-{\cong}
\\
[F, {p_1}_*p_2^*F'] \ar[rd]^-{\cong}_-{\adj}
&& [{p_2}_!p_1^*F,F'] \ar[ld]_-{\cong}^-{\adj}
\\
& [\,p_1^*F,p_2^*F']\,.
}}
\end{equation}
Explicitly, the isomorphism $\nabla_{i}$ maps a morphism $\big(i_!F \stackrel{\Displ\theta}{\Rightarrow} i_*F' \big)$ to
\begin{equation}
\label{eq:nabla-expl}%
\vcenter{
\xymatrix{
{\Big(} p_1^*F \ar@{->}[r]^-{\Displ\eta}_-{i_!\adj i^*}
& (p_1^* i^*)(i_!F) \ar@{->}[r]^-{\Displ\lambda^*\theta}
& (p_2^* i^*)(i_*F') \ar@{->}[r]^-{\Displ\eps}_-{i^*\adj i_*}
& p_2^*F' {\Big).}
}}
\end{equation}
\end{Prop}

\begin{proof}
It suffices to unpack~$\lambda_!$ and $\lambda_*$. The left (respectively right) composite in~\eqref{eq:nabla-FF'} reduces to~\eqref{eq:nabla-expl} thanks to naturality and the unit-counit relation for $p_1^*\adj {p_1}_*$ (respectively for ${p_2}_!\adj p_2^*$). For instance, given a natural transformation $\theta\colon i_!F\Rightarrow i_*F'$, the left composite yields
\[
\vcenter { \hbox{
\xymatrix@L=1pt@C=15pt@R=15pt{
& \ar[ld]_-{F} \ar[rd]^-{F'}
 \ar@{}[ddd]|{\oEcell{\theta}}
\\
\ar@/_12pt/[rdd]_-{i_!}
&& \ar@/^12pt/[ldd]^-{i_*}
&
\\
&&
\\
&&
}}}
\quad \overset{\underset{\phantom{m}}{i_!\dashv \;i^*}}{\longmapsto} \quad\quad
\vcenter { \hbox{
\xymatrix@L=1pt@C=15pt@R=15pt{
& \ar[ld]_-{F} \ar[rd]^-{F'}
 \ar@{}[ddd]|(.4){\oEcell{\theta}}
\\
 \ar[rdd]_(.4){i_!}
 \ar@{=}@/_4ex/[ddd] \ar@{}[ddd]|{\oEcell{\eta}}
&& \ar@/^5pt/[ldd]^-{i_*}
\\
&&
\\
& \ar[ld]_-{i^*}
&
\\
&}
}}
\quad \overset{\underset{\phantom{m}}{[\id, \lambda_*]}}{\longmapsto}
\]
\[
\vcenter { \hbox{
\xymatrix@L=1pt@C=15pt@R=15pt{
& \ar[ld]_-{F} \ar[rd]^-{F'}
 \ar@{}[ddd]|(.4){\oEcell{\theta}}
\\
 \ar[rdd]_(.4){i_!}
 \ar@{=}@/_4ex/[ddd]
&& \ar[ldd]^(.4){i_*}
 \ar@{=}@/^4ex/[ddd]
&
\\
&&
\\
& \ar[ld]_-{i^*}
 \ar[dr]^-{i^*} &
\\
\ar[rd]|-{p_1^*}
 \ar@{=}@/_3ex/[dd]
 \ar@{}[dd]|{\oEcell{\eta}}
 \ar@{}[rr]|{\oEcell{\lambda^*}}
 \ar@{}[uuu]|{\oEcell{\eta}}
&&
 \ar[ld]^-{p_2^*}
 \ar@{}[uuu]|{\oEcell{\varepsilon}}
\\
& \ar[ld]^-{p_{1*}}
 & \\
&&
}
}}
\quad \overset{\underset{\phantom{m}}{p_1^*\dashv\; p_{1*}}}{\longmapsto} \quad\quad
\vcenter { \hbox{
\xymatrix@L=1pt@C=15pt@R=15pt{
& \ar[ld]_-{F} \ar[rd]^-{F'}
 \ar@{}[ddd]|(.4){\oEcell{\theta}}
\\
 \ar[rdd]_(.4){i_!}
 \ar@{=}@/_4ex/[ddd]
&& \ar[ldd]^(.4){i_*}
 \ar@{=}@/^4ex/[ddd]
&
\\
&&
\\
& \ar[ld]_-{i^*}
 \ar[dr]^-{i^*} & \\
\ar[rd]|-{p_1^*}
 \ar@{=}@/_3ex/[dd]
 \ar@{}[dd]|{\oEcell{\eta}}
 \ar@{}[rr]|{\oEcell{\lambda^*}}
 \ar@{}[uuu]|{\oEcell{\eta}}
 &&
 \ar[ld]^-{p_2^*}
 \ar@{}[uuu]|{\oEcell{\varepsilon}}
 \\
& \ar[ld]|-{\;p_{1*}}
 \ar@{=}@/^3ex/[dd]
 \ar@{}[dd]|{\oEcell{\varepsilon}}
 & \\
\ar[rd]_-{p_1^*} && \\
&&
}
}}
\quad = \quad \quad
\vcenter { \hbox{
\xymatrix@L=1pt@C=15pt@R=15pt{
& \ar[ld]_-{F} \ar[rd]^-{F'}
 \ar@{}[ddd]|(.4){\oEcell{\theta}}
\\
 \ar[rdd]_(.4){i_!}
 \ar@{=}@/_4ex/[ddd]
&& \ar[ldd]^(.4){i_*}
 \ar@{=}@/^4ex/[ddd]
&
\\
&&
\\
& \ar[ld]_-{i^*}
 \ar[dr]^-{i^*} &
\\
\ar[rd]_{p_1^*}
 \ar@{}[rr]|{\oEcell{\lambda^*}}
 \ar@{}[uuu]|{\oEcell{\eta}}
 &&
 \ar[ld]^-{p_2^*}
 \ar@{}[uuu]|{\oEcell{\varepsilon}}
 \\
&&
}
}}
\]
which is the claimed formula~\eqref{eq:nabla-expl}. The right composite works similarly.
\end{proof}

\begin{Rem}
\label{Rem:nabla-on-F}%
Technically speaking, we should include $F$ and $F'$ in the notation for the isomorphism $\nabla_i$ but it is clear that neither $F$ nor $F'$ play any serious role in it. We shall nevertheless write ``$\nabla_i$ for $F$ and~$F'$ as in~\eqref{eq:nabla-FF'}'' when we need to make clear what form of $\nabla_i$ we use.
\end{Rem}

\begin{Cor}
\label{Cor:nabla}%
With notation as above, we have a natural isomorphism
\begin{equation}
\label{eq:nabla}%
\left\{\vcenter{
\xymatrix@C=2em@R=1em@L=1ex{
[i_!,i_*] \ar[r]^-{\nabla_i}_-{\cong}
& [p_1^*,p_2^*]
\\
{\big(}i_! \overset{\Displ\theta}\Rightarrow
i_* { \big) }  \ar@{|->}[r]^-{\nabla_i}
& {\Big(} p_1^* \ar@{=>}[r]^-{\Displ\eta}_-{i_!\adj i^*}
& (p_1^* i^*)(i_!) \ar@{=>}[r]^-{\Displ\lambda^* \theta}
& (p_2^* i^*)(i_*) \ar@{=>}[r]^-{\Displ\eps}_-{i^*\adj i_*}
& p_2^* {\Big).}
}}\right.
\end{equation}
\end{Cor}

\begin{proof}
Plug $F=F'=\Id_{\MM(H)}$ in \Cref{Prop:nabla}.
\end{proof}

\begin{proof}[Proof of \Cref{Thm:Theta}]
By \Cref{Cor:nabla} we can uniquely define a morphism $\Theta_i\colon i_! \Rightarrow i_*$ by deciding the corresponding $\nabla_i(\Theta_i)=\eps\circ(\lambda^*\,\Theta_i)\circ\eta$ from $p_1^*$ to $p_2^*$. We define $\Theta_i$ to be the one corresponding to~$\delta_i\colon p_1^*\Rightarrow p_2^*$ as in \Cref{Prop:delta}.
\end{proof}

\begin{Rem}
\label{Rem:Theta}%
Unpacking the above construction, $\Theta_i$ is the unique morphism $i_!\Rightarrow i_*$ such that the following morphism
\[
\xymatrix@C=4em{
{\Big(} p_1^* \ar@{->}[r]^-{\Displ\eta}_-{i_!\adj i^*}
& (p_1^* i^*)(i_!) \ar@{->}[r]^-{\Displ\lambda^* \Theta_i}
& (p_2^* i^*)(i_*) \ar@{->}[r]^-{\Displ\eps}_-{i^*\adj i_*}
& p_2^* {\Big)}
}
\]
of functors $\MM(H)\to \MM(i/i)\cong \MM(H)\times \prod_{C\cap\Delta_i(H)\,=\,\varnothing}\MM(C)$ is the identity on the component $\MM(H)\to \MM(H)$ and zero on the other components $\MM(H)\to \MM(C)$ for~$C\subset(i/i)$ disjoint from the image of~$\Delta_i\colon H\to (i/i)$.

Or, to express this in a slightly different way, let $C:= (i/i)\smallsetminus \Delta_i(H)$ be the complement of the diagonal component (\ie considering $\Delta_i(H)$ as the essential image of~$\Delta_i\colon H\into (i/i)$) and write $\incl_C\colon C\hookrightarrow (i/i)$ for the inclusion functor. Then $\Theta_i$ is the unique natural transformation $i_!\Rightarrow i_*$ such that the following two equations hold:
\begin{equation} \label{eq:proto(f)}
 \vcenter { \hbox{
 \xymatrix@C=15pt@R=15pt{
& \MM(H) & \\
& \MM(i/i)
 \ar[u]^<<<{\Delta^*_i} & \\
\MM(H)
 \ar[ur]_{p_1^*}
  \ar@{}[rr]|{\oEcell{\lambda^*}}
   \ar@/^1ex/@{=}[uur]
     \ar@/_2ex/@{=}[dd]
      \ar@{}[dd]^{\,\;\oEcell{\eta}} &&
\MM(H)
 \ar[ul]^{p_2^*}
  \ar@/_1ex/@{=}[uul]
    \ar@/^2ex/@{=}[dd]
     \ar@{}[dd]_{\oEcell{\varepsilon}\,\;} \\
& \MM(G)
 \ar[ul]^{i^*}
  \ar[ur]_{i^*} & \\
\MM(H)
 \ar[ur]_{i_!}
  \ar@{}[rr]|{\oEcell{\Theta_i}}
   \ar@/_6ex/@{=}[rr] &&
 \MM(H)
  \ar[ul]^{i_*} \\
 &&
 }
 }}
\qquad = \quad \id_{\Id_{\MM(H)}}
\end{equation}
and
\begin{equation} \label{eq:proto(g)}
\kern3em\vcenter { \hbox{
 \xymatrix@C=15pt@R=15pt{
& \MM(C) & \\
& \MM(i/i)
 \ar[u]|(.45){\incl_C^*} & \\
\MM(H)
 \ar@/^1ex/[uur]^-{(p_1\restr{C})^*}
 \ar[ur]_{p_1^*}
  \ar@{}[rr]|{\oEcell{\lambda^*}}
     \ar@/_2ex/@{=}[dd]
      \ar@{}[dd]^{\,\;\oEcell{\eta}} &&
\MM(H)
 \ar@/_1ex/[uul]_-{(p_2\restr{C})^*}
 \ar[ul]^{p_2^*}
    \ar@/^2ex/@{=}[dd]
     \ar@{}[dd]_{\oEcell{\varepsilon}\,\;} \\
& \MM(G)
 \ar[ul]^{i^*}
  \ar[ur]_{i^*} & \\
\MM(H)
 \ar[ur]_{i_!}
  \ar@{}[rr]|{\oEcell{\Theta_i}}
   \ar@/_6ex/@{=}[rr] &&
 \MM(H)
  \ar[ul]^{i_*} \\
 &&
 }
 }}
\qquad = \quad 0 \,\colon (p_1\restr{C})^* \Rightarrow (p_2\restr{C})^*.
\end{equation}
\end{Rem}

We are now ready for the main result of this section, which describes the fundamental properties of the comparison morphism $\Theta_i\colon i_!\Rightarrow i_*$.

\begin{Thm}
\label{Thm:Theta-properties}%
Let $\MM\colon \GG^\op\to \ADD$ satisfy all properties of a Mackey 2-functor, except perhaps ambidexterity, as in \Cref{Hyp:Theta}. Then the morphisms $\Theta_i\colon i_!\Rightarrow i_*$ of \Cref{Thm:Theta} satisfy the following:
\begin{enumerate}[\rm(a)]
\item
\label{it:Theta-add}%
If $H=H_1\sqcup H_2$ is disconnected and $i=(i_1\ i_2)\colon H_1\sqcup H_2\into G$ then, under the canonical identifications $i_!\cong {i_1}_!\oplus {i_2}_!$ and $i_*\cong {i_1}_*\oplus {i_2}_*$ (\Cref{Conv:Mack-5-6}), we have $\Theta_i=\Theta_{i_1}\oplus \Theta_{i_2}$ diagonally.
\smallbreak
\item
\label{it:Theta-eq}%
If $i^*\colon\MM(G)\to \MM(H)$ is an equivalence (for instance if~$i\colon H\isoto G$ is an equivalence in~$\GG$) then $\Theta_i$ is an isomorphism. In the special case of $i=\Id$ then $\Theta_i$ is the identity.
\smallbreak
\item
\label{it:Theta-magic}%
For every partial Mackey square (\Cref{Def:partial-Mackey})
\begin{equation}
\label{eq:partial-Mackey}
\vcenter{\xymatrix@C=14pt@R=14pt{
& L \ar[dl]_-{v} \ar@{ >->}[dr]^-{j}
\\
H \ar@{ >->}[dr]_-{i} \ar@{}[rr]|-{\isocell{\gamma}}
&& K \ar[dl]^-{u}
\\
&G
}}
\end{equation}
with~$i$ and $j$ faithful, we have a commutative diagram
\begin{equation}
\label{eq:magic-partial}%
\vcenter{
\xymatrix@C=4em@L=1ex{
u^* i_! \ar@{=>}[r]^-{u^*{\Theta_i}}
& u^* i_* \ar@{=>}[d]^-{(\gamma\inv)_*}
\\
j_! v^* \ar@{=>}[r]^-{{\Theta_j}\ v^*} \ar@{=>}[u]^-{\gamma_!}
& j_*v^*.\!\!
}}
\end{equation}
In particular, the above holds if~\eqref{eq:partial-Mackey} is a Mackey square (\Cref{Def:Mackey-square}) or, even stronger, an iso-comma square. In those cases, the vertical morphisms $\gamma_!$ and $(\gamma\inv)_*$ in~\eqref{eq:magic-partial} are isomorphisms by base-change~\Mack{3}.
\smallbreak
\item
\label{it:Theta-2-functorial}%
For every 2-cell $\alpha\colon i\isoEcell i'$ between faithful functors $i,i'\colon H\into G$, the following diagram commutes:
\begin{equation}
\label{eq:Theta-2-functorial}%
\vcenter{
\xymatrix@C=4em@L=1ex{
{i}_! \ar@{=>}[r]^-{\Theta_{i}}
& {i}_* \ar@{<=}[d]^-{\alpha_*}_-{\cong}
\\
{i'}_! \ar@{=>}[r]^-{\Theta_{i'}} \ar@{=>}[u]^-{\alpha_!}_-{\cong}
& {i'}_*
}}
\end{equation}
\smallbreak
\item
\label{it:Theta-ij}%
For every composable faithful $K\ointo{j} H \ointo{i} G$ the following diagram commutes
\begin{equation}
\label{eq:Theta-ij}%
\vcenter{\xymatrix@R=1em@L=1ex{
(ij)_! \ar@{=>}[r]^-{\Theta_{ij}} \ar@{=}[d]
& (ij)_* \ar@{=}[d]
\\
i_! j_! \ar@{=>}[r]^-{\Theta_{i}\Theta_{j}}
& i_* j_*
}}
\end{equation}
where the vertical identifications $(ij)_!\cong i_! j_!$ and $(ij)_*\cong i_* j_*$ are the canonical isomorphisms between compositions of adjoints.
\smallbreak
\item
\label{it:Theta-special}%
For every faithful $i\colon H\into G$ the following composite is the identity:
\[
\vcenter{
\xymatrix@C=4em@L=1ex{
\Id \ar@{=>}[r]^-{\eta}_-{i_!\adj i^*}
& i^*{i}_! \ar@{=>}[r]^-{i^* \Theta_{i}}
& i^* {i}_* \ar@{=>}[r]^-{\eps}_-{i^*\adj i_*}
& \Id \,.\!\!
}}
\]
\smallbreak
\item
\label{it:Theta-vanishing}%
For every faithful $i\colon H\into G$ the following natural transformation is zero:
\[
\incl_C^* \left(
\vcenter{
\xymatrix@C=4em@L=1ex{
p_1^* \ar@{=>}[r]^-{p_1^* \eta}_-{i_!\adj i^*}
& p_1^* i^*{i}_! \ar@{=>}[r]^-{\lambda^* \Theta_{i}}
& p_2^* i^* {i}_* \ar@{=>}[r]^-{p_2^* \eps}_-{i^*\adj i_*}
& p_2^*
}} \right)\,,
\]
where $\incl_C\colon C\hookrightarrow (i/i)$ is the inclusion of the complement $C:=(i/i)\smallsetminus \Delta_i(H)$ of the essential image of the diagonal embedding~$\Delta_i\colon H\into (i/i)$.
\end{enumerate}
\end{Thm}

We need again some preparation. Specifically we want to show that the isomorphism $\nabla_i\colon [i_!,i_*]\isoto [p_1^*,p_2^*]$ of \Cref{Cor:nabla} behaves nicely in~$i$, very much in the spirit of what we did in \Cref{sec:self-iso}.

\begin{Not}
\label{Not:nabla-nat}%
Consider a 2-cell as usual (not necessarily partial Mackey)
\[
\vcenter{\xymatrix@C=14pt@R=14pt@L=1ex{
& L \ar[dl]_-{v} \ar@{ >->}[dr]^-{j}
\\
H \ar@{ >->}[dr]_-{i} \ar@{}[rr]|-{\isocell{\gamma}}
&& K \ar[dl]^-{u}
\\
&G
}}
\]
and the self-iso-commas~\eqref{eq:self-ic} and~\eqref{eq:self-ic-j} for~$i$ and~$j$, together with the comparison functor~$w=w_{v,u,\gamma}\colon (j/j)\to (i/i)$ of \Cref{Prop:w}:
\[
\xymatrix@L=2pt@C=20pt@R=20pt{
&(j/j)
 \ar@{}[dd]|<<<<<<<<<{\isocell{\rho}}
 \ar[drrr]^-w
 \ar[dr]^>>>>{q_2}
 \ar[dl]_-{q_1} &&&& \\
L
\ar[drrr]^<<<<<<{v}
 \ar[dr]_{j} &&
 L
 \ar@{}[dr]|{\stackrel{\sim}{\Ecell}\,\gamma}
 \ar[dl]|(.53){\phantom{M}}^<<<{j}
 \ar[drrr]|<<<<<<<<<<<<<<<<<<{\phantom{M}}^<<<<<<<<{v} &&
 (i/i)
 \ar@{}[dd]|>>>>>>>>{\isocell{\lambda}}
 \ar[dr]^{p_2}
 \ar[dl]_<<<{p_1} & \\
& K
 \ar[drrr]_-{u} &&
  H
  \ar[dr]_<<<{i}
 \ar@{}[ll]|{\stackrel{\sim}{\Ecell}\,\gamma} &&
H
  \ar[dl]^{i}  \\
&&&& G &
}
\]
\end{Not}

\begin{Prop}
\label{Prop:nabla-nat}%
Under \Cref{Hyp:Theta} and the above \Cref{Not:nabla-nat}, the two isomorphisms~$\nabla_i$ and $\nabla_j$ of~\Cref{Cor:nabla} for~$i$ and~$j$ are compatible; namely for every parallel functors $F,F'\colon \cat{C}\to \MM(H)$ (\eg $F=F'=\Id$) the following diagram of abelian groups commutes:
\begin{equation}
\label{eq:nabla-nat}%
{\vcenter{\xymatrix@C=12em{
[i_!F,i_*F'] \ar[d]_-{u^*-} \ar[r]^-{\nabla_i}_-{\cong\ \textrm{\rm\ (\ref{eq:nabla-FF'}) for $F$ and $F'$}}
& [p_1^*F,p_2^*F'] \ar[d]^-{w^* -}
\\
[u^*i_!F\,,\,u^*i_*F'] \ar[d]_-{[\gamma_!\,,\,(\gamma\inv)_*]}
& [w^*p_1^*F,w^*p_2^*F'] \ar@{=}[d]_-{(p_{1}w\,=\,vq_{1})}^-{(p_{2}w\,=\,vq_{2})}
\\
[j_!v^*F,j_*v^*F'] \ar[r]^-{\nabla_j}_{\cong\ \textrm{\rm\ (\ref{eq:nabla-FF'}) for $v^*F$ and $v^*F'$}}
& [q_1 ^*v^*F,q_2 ^*v^*F']
}}}
\end{equation}
where $[-,-]$ stands everywhere for the suitable set of natural transformations.
\end{Prop}

\begin{proof}
For completeness, let us make the morphism $\nabla_j\colon [j_! v^*F,j_* v^*F']\to [q_1 ^* v^*F,q_2 ^* v^*F']$ of~\eqref{eq:nabla-FF'} more explicit. It uses~\eqref{eq:nabla-expl} with $v^*F$ and $v^*F'$ (instead of~$F$ and~$F'$) and with the faithful functor $j$ (instead of~$i$), and therefore with the self-iso-comma~\eqref{eq:self-ic-j} instead of~\eqref{eq:self-ic}. Unpacking~\eqref{eq:nabla-expl} gives us
\[
\xymatrix@C=1.4em@R=1em@L=1ex{
{\Big(}j_! v^*F
\ar@{=>}[r]^-{\Displ\omega} &
j_* v^*F { \Big) } \ar@{|->}[r]^-{\nabla_j}
& {\Big(} q_1 ^* v^*F \ar@{=>}[r]^-{\Displ\eta}_-{j_!\adj j^*}
& (q_1 ^* j^*)(j_! v^*F') \ar@{=>}[r]^-{\Displ\rho^* \omega}
& (q_2 ^* j^*)(j_* v^*F') \ar@{=>}[r]^-{\Displ\eps}_-{j^*\adj j_*}
& q_2 ^* v^*F' {\Big).}}
\]

Now, by starting with a $\theta\in [i_!F,i_*F']$ and successively pasting on 2-cells according to the definitions, we see that the down-then-right route in~\eqref{eq:nabla-nat} produces the cell on the left-hand side below, while the right-then-down route in~\eqref{eq:nabla-nat} yields the one on the right-hand side:
\[
\vcenter { \hbox{
\xymatrix@L=1pt@C=15pt@R=15pt{
&& \ar[ld]_-{F} \ar[rd]^-{F'}
 \ar@{}[dd]|(.4){\oEcell{\theta}}
\\
& \ar[rd]_(.4){i_!}
 \ar@{=}@/_2ex/[ldd]
 \ar@{}[ldd]|(.4){\ \SEcell\eta}
&& \ar[ld]^(.4){i_*}
 \ar@{=}@/^2ex/[rdd]
 \ar@{}[rdd]|(.4){\NEcell\eps\ }
&
\\
&&
 \ar[dd]^-{u^*}
 \ar[lld]_-{i^*}
 \ar[rrd]^-{\;i^*} && \\
\ar[dd]_-{v^*}
 \ar@{}[rrd]|{\SEcell\,\gamma^*} &&&& \ar@{}[dll]|{\NEcell\,\gamma^{*-1}}
 \ar[dd]^-{v^*} \\
&& \ar@{=}[dd]
 \ar[lld]_-{j^*}
 \ar[rrd]^-{\;\,j^*} &&
\\
\ar[rrd]|{j_!}
\ar@{=}[dd]
&&
 \ar@{}[ll]|{\Ecell\,\varepsilon}
 \ar@{}[rr]|{\Ecell\,\eta}
&&
 \ar[lld]|{j_*}
 \ar@{=}[dd]
\\
&& \ar[lld]^-{j^*}
 \ar[rrd]_-{j^*}
 \ar@{}[dd]|{\Ecell\,\rho^*}
 \ar@{}[ll]|>>>{\;\;\;\;\Ecell\,\eta}
 \ar@{}[rr]|>>>{\Ecell\,\varepsilon\;\;\;\;} && \\
\ar[rrd]_-{q_1 ^*}
&&&& \ar[lld]^-{q_2 ^*} \\
&&&&
}
}}
\quadtext{=}
\vcenter { \hbox{
\xymatrix@L=1pt@C=15pt@R=15pt{
&& \ar[ld]_-{F} \ar[rd]^-{F'}
 \ar@{}[dd]|(.4){\oEcell{\theta}}
\\
& \ar[rd]_(.4){i_!}
 \ar@{=}@/_2ex/[ldd]
 \ar@{}[ldd]|(.4){\ \SEcell\eta}
&& \ar[ld]^(.4){i_*}
 \ar@{=}@/^2ex/[rdd]
 \ar@{}[rdd]|(.4){\NEcell\eps\ }
&
\\
&&
 \ar[dd]^-{u^*}
 \ar[lld]_-{i^*}
 \ar[rrd]^-{\;i^*} && \\
\ar[dd]_-{v^*}
 \ar@{}[rrd]|{\SEcell\,\gamma^*} &&&& \ar@{}[dll]|{\NEcell\,\gamma^{*-1}}
 \ar[dd]^-{v^*}
\\
&&
 \ar[lld]_-{j^*}
 \ar[rrd]^-{\;\,j^*} && \\
 \ar[rrd]_-{q_1 ^*}
 \ar@{}[rrrr]|{\Ecell\,\rho^*}
&&&&
 \ar[lld]^-{q_2 ^*}
\\
&&
}
}}
\quadtext{=}
\vcenter { \hbox{
\xymatrix@L=1pt@C=20pt@R=15pt{
& \ar[ld]_-{F} \ar[rd]^-{F'}
 \ar@{}[dd]|(.4){\oEcell{\theta}}
\\
 \ar[rd]^(.4){i_!}
 \ar@{=}@/_2ex/[dd]
 \ar@{}[dd]|(.4){\ \SEcell\eta}
&& \ar[ld]_(.4){i_*}
 \ar@{=}@/^2ex/[dd]
 \ar@{}[dd]|(.4){\NEcell\eps\ }
&&
\\
& \ar[ld]_-{i^*}
 \ar[dr]^-{i^*} &&
\\
\ar[rd]_-{p_1^*}
 \ar@{}[rr]|{\Ecell\,\lambda^*}
&&
 \ar[ld]^-{p_2^*}
&
\\
& \ar[dd]_-{w^*} &
\\
&&
\\
&&
}
}}
\]
We see that they are both equal to the middle one. Indeed, the left-hand equality holds by the unit-counit relations for $j_!\adj j^*$ and for $j^*\adj j_*$. The second equality is direct from~\eqref{eq:lambda-gamma-rho} by applying~$(-)^*=\MM(-)$, which is \emph{co}variant on 2-cells.
\end{proof}

\begin{proof}[Proof of \Cref{Thm:Theta-properties}]
We leave the isomorphism and additivity statements~\eqref{it:Theta-add} and~\eqref{it:Theta-eq} to the reader, using \Cref{Conv:Mack-5-6}. We focus on the critical property~\eqref{it:Theta-magic}, namely we want to show that the following diagram commutes:
\begin{equation}
\label{eq:Theta-gamma}%
\vcenter{
\xymatrix@C=4em@L=1ex{
u^* i_! \ar@{=>}[r]^-{u^*{\Displ\Theta_i}}
& u^* i_* \ar@{=>}[d]^-{(\gamma\inv)_*}
\\
j_! v^* \ar@{=>}[r]^-{{\Displ\Theta_j}\ v^*} \ar@{=>}[u]^-{\gamma_!}
& j_*v^*
}}
\end{equation}
when the initial 2-cell $\gamma\colon i v \Rightarrow u j$ is partial Mackey. For this, let us follow $\Theta_i\in [i_!,i_*]$ and $\Theta_j\in[j_!,j_*]$ in the following commutative diagram:
\[
\xymatrix@C=4.5em{
\Theta_i \ar@{}[r]|-{\in} \ar@{|->}[d]
& [i_!,i_*] \ar[d]_-{u^*-} \ar[rr]^-{\nabla_i}_-{\cong~\eqref{eq:nabla}}
&& [p_1^*,p_2^*] \ar[d]^-{w^* -}
& \delta_i \ar@{}[l]|-{\ni} \ar@{|->}[d]
\\
u^* \Theta_i \ar@{}[r]|-{\in}
& [u^*i_!\,,\,u^*i_*] \ar@{=>}[d]_-{[\gamma_!\,,\,(\gamma\inv)_*]}
&& [w^*p_1^*,w^*p_2^*] \ar@{=}[d]_-{(p_{1}w\,=\,vq_{1})}^-{(p_{2}w\,=\,vq_{2})}
& w^*\,\delta_i \ar@{}[l]|-{\ni}
\\
\Theta_j\,v^* \ar@{}[r]|-{\in}
& [j_!v^*,j_*v^*] \ar[rr]^-{\nabla_j}_{\cong~\eqref{eq:nabla-FF'}\textrm{\rm\ for $F=F'=v^*$}}
&& [q_1 ^*v^*,q_2 ^*v^*]
& \delta_j\,v^* \ar@{}[l]|-{\ni}
\\
\Theta_j \ar@{}[r]|-{\in} \ar@{|->}[u]
& [j_!,j_*] \ar[rr]^-{\nabla_j}_-{\cong~\eqref{eq:nabla}} \ar[u]^-{-\ v^*}
&& [q_1 ^*,q_2 ^*] \ar[u]_-{-\ v^*}
& \delta_j \ar@{}[l]|-{\ni} \ar@{|->}[u]
}
\]
Commutativity of the top square comes from \Cref{Prop:nabla-nat}; that of the bottom is straightforward. The claim of the statement is that $\Theta_i$ and $\Theta_j$ map to one another under the left-hand vertical morphisms, \ie that they have the same image in the middle, say, in~$[j_!v^*,j_*v^*]$, which is still isomorphic under $\nabla_j$ to the right-hand~$[q_1 ^*v^*,q_2 ^*v^*]$. By construction in \Cref{Thm:Theta}, under $\nabla_i$ and $\nabla_j$, our morphisms $\Theta_i$ and $\Theta_j$ correspond, on the right-hand side of the above diagram, to~$\delta_i$ and $\delta_j$ respectively. Therefore, it suffices to observe that $\delta_j\,v^*=w^*\delta_i$ as we proved in \Cref{Prop:delta-nat}\,\eqref{it:delta-nat-c}. This is the place where we use that~\eqref{eq:partial-Mackey} is partial Mackey, not just any square. This finishes the proof of the crucial part~\eqref{it:Theta-magic} of Theorem~\ref{Thm:Theta-properties}.

Part~\eqref{it:Theta-2-functorial} follows from~\eqref{it:Theta-magic} applied to the (partial) Mackey square of \Cref{Exa:partial-alpha}, which gives the commutativity of
\[
\vcenter{
\xymatrix@C=4em@L=1ex{
{i}_! \ar@{=>}[r]^-{\Theta_{i}}
& {i}_* \ar@{=>}[d]^-{(\alpha\inv)_*}
\\
{i'}_! \ar@{=>}[r]^-{\Theta_{i'}} \ar@{=>}[u]^-{\alpha_!}
& {i'}_*.\!\!
}}
\]
In this case, $\alpha_!$ and $(\alpha\inv)_*$ are isomorphisms by base change (since the square in~\Cref{Exa:partial-alpha} is actually a Mackey square). By compatibility of mates with pasting, we have $(\alpha\inv)_*=(\alpha_*)\inv$ which gives the commutative square~\eqref{eq:Theta-2-functorial} of the statement. (See \Cref{Rem:_!_*-for-id} if necessary.)

Similarly, part~\eqref{it:Theta-special} follows from~\eqref{it:Theta-magic} applied to the partial Mackey square of \Cref{Exa:partial-diag} and the fact that $\id_!\colon \Id\Rightarrow i^*i_!$ and $\id_*\colon i^*i_*\Rightarrow \Id$ are nothing but the unit of~$i_!\adj i^*$ and the counit of~$i^*\adj i_*$ respectively. See \Cref{Exa:units-as-mates}.

Alternatively, after remembering that $\lambda \Delta_i =\id_{\Id_{H}}$ (\Cref{Prop:Delta}) we see that~\eqref{it:Theta-special} is nothing but the equality \eqref{eq:proto(f)}, which holds by the construction of~$\Theta_i$. Likewise, and even more directly, part~\eqref{it:Theta-vanishing} is precisely~\eqref{eq:proto(g)}.

Finally, for part~\eqref{it:Theta-ij}, consider two composable morphisms $K\ointo{j} H \ointo{i} G$ in~$\GG$ and recall the discussion in \Cref{sec:self-iso}. In particular, consider as in \Cref{Not:ij-comp} the self-iso-commas
\[
\vcenter{
\xymatrix@C=14pt@R=14pt{
& (i/i) \ar[dl]_-{p_1} \ar[dr]^-{p_2}
\\
H \ar[dr]_-{i} \ar@{}[rr]|-{\isocell{\lambda}}
&& H \ar[dl]^-{i}
\\
&G
}}
\qquad
\vcenter{\xymatrix@C=14pt@R=14pt{
& (ij/ij) \ar[dl]_{r_1} \ar[dr]^-{r_2}
\\
K \ar[dr]_-{ij} \ar@{}[rr]|-{\isocell{\sigma}}
&& K \ar[dl]^-{ij}
\\
&G
}}
\qquad
\vcenter{
\xymatrix@C=14pt@R=14pt{
& (j/j) \ar[dl]_-{q_1} \ar[dr]^-{q_2}
\\
K \ar[dr]_-{j} \ar@{}[rr]|-{\isocell{\rho}}
&& K \ar[dl]^-{j}
\\
&H
}}
\]
and the two comparison functors $w''=w_{\Id_K,i,\id}$ and $w'=w_{j,\Id_G,\id}$
\[
(i/i) \stackrel{\,w'}{\longleftarrow} (ij/ij) \stackrel{\,\,w''}{\longleftarrow} (j/j)
\]
induced by \Cref{Prop:w} applied to the obvious squares, see~\eqref{eq:ij-squares}. To prove $\Theta_{ij}=\Theta_i\Theta_j$ we use the definition of~$\Theta_{ij}$ from \Cref{Thm:Theta}, namely we trace $\Theta_i\Theta_j$ along the morphisms
\[
\xymatrix@C=1.5em{
\Theta_i\Theta_j\in [i_!j_!\,,i_*j_*] \ar@{=}[r]^-{\sim}
& [(ij)_!\,,(ij)_*] \ar[rr]_-{\simeq}^-{\nabla_{ij}}
&& [r_1^* ,r_2^*] \ar[rr]^-{\incl_C^*}
&& [\incl_C^*r_1^*, \incl_C^*r_2^*]
}
\]
associated to each connected component $\incl_C\colon C\hook (ij/ij)$ of~$(ij/ij)$. Recall from \Cref{Prop:ij-components} that these connected components come in three distinct sorts; the proof will be slightly different in each case. We need to show that the image of~$\Theta_i\Theta_j$ under the above map is zero when $C\subset (ij/ij)$ is disjoint from the diagonal~$\Delta_{ij}(K)$, which are cases~\eqref{it:ij-comp-2} and~\eqref{it:ij-comp-3} of \Cref{Prop:ij-components}; and we need to show that the image of~$\Theta_i\Theta_j$ on the diagonal components is the identity, which by \Cref{Prop:Delta} amounts to show that its image under
\begin{equation}
\label{eq:aux-nabla-delta}%
\xymatrix@C=1em{
\Theta_i\Theta_j\in [i_!j_!\,,i_*j_*] \ar@{=}[r]^-{\sim}
& [(ij)_!\,,(ij)_*] \ar[rr]_-{\simeq}^-{\nabla_{ij}}
&& [r_1^* ,r_2^*] \ar[rr]^-{\Delta_{ij}^*}
&& [\Id_{\MM(K)}, \Id_{\MM(K)}]
}\kern-2em
\end{equation}
is the identity (of $\Id_{\MM(K)}$). These are the components of type~\eqref{it:ij-comp-1} in \Cref{Prop:ij-components}.

Let us apply \Cref{Prop:nabla-nat} to the first square of~\eqref{eq:ij-squares}, namely
\begin{equation}
\label{eq:aux-ij-1}%
\vcenter{\xymatrix@C=14pt@R=14pt{
& K \ar@{=}[ld] \ar@{ >->}[rd]^-{j} \ar@{}[dd]|-{\isocell{\id}}
\\
K \ar@{ >->}[rd]_-{ij}
&& H \ar@{->}[ld]^-{i}
\\
& G}}
\end{equation}
whose associated $w_{\Id,i,\id}\colon (j/j)\to (ij/ij)$ is our~$w''$ above. The commutative diagram~\eqref{eq:nabla-nat} for~$F=F'=\Id$ becomes here
\begin{equation}
\label{eq:aux-nabla-nat-ij}%
{\vcenter{\xymatrix@C=12em@R=2em{
[(ij)_!,(ij)_*] \ar[d]_-{i^*} \ar[r]^-{\nabla_{ij}}_-{\cong\ \eqref{eq:nabla}}
& [r_1^*,r_2^*] \ar[d]^-{w''^*}
\\
[i^*(ij)_!\,,\,i^*(ij)_*] \ar[d]_-{[\id_!\,,\,(\id\inv)_*]}
& [w''^*r_1^*,w''^*r_2^*] \ar@{=}[d]_-{(r_{1}w''\,=\,q_{1})}^-{(r_{2}w''\,=\,q_{2})}
\\
[j_!,j_*] \ar[r]^-{\nabla_j}_{\cong\ \eqref{eq:nabla}}
& [q_1 ^*,q_2 ^*]\,.\!\!
}}}
\end{equation}
The mates $\id_!$ and $(\id\inv)_*$ on the left-hand side are the ones associated to~\eqref{eq:aux-ij-1}. We now verify that the image of $\Theta_i\Theta_j$ under the above left-hand vertical composite (after the identifications $i_!j_!\cong (ij)_!$ and $i_*j_*\cong(ij)_*$) is precisely~$\Theta_j$:
\[
{\vcenter{\xymatrix@C=1.5em@R=2em{
\Theta_i\Theta_j \ar@{}[r]|-{\in} \ar@{|->}[rdd]
& [i_!j_!,i_*j_*] \ar@{=}[r]^-{\sim}
& [(ij)_!,(ij)_*] \ar[d]^-{i^*}
\\
&& [i^*(ij)_!\,,\,i^*(ij)_*] \ar[d]^-{[\id_!\,,\,(\id\inv)_*]}
\\
& \Theta_j \ar@{}[r]|-{\in}
& [j_!,j_*]
}}}
\]
Indeed, the image of $\Theta_i \Theta_j$ is defined by the following pasting:
\[
\vcenter { \hbox{
\xymatrix@L=1pt@C=15pt@R=15pt{
&& \ar@/_2ex/[dd]_{j_!}
 \ar@/^2ex/[dd]^{j_*}
 \ar@{}[dd]|{\oEcell{\;\Theta_j}}
 \ar@/_6ex/[dddd]|>>>>>>>>>>{(ij)_!}
 \ar@/^6ex/[dddd]|>>>>>>>>>>{\;\;(ij)_*}
 \ar@{=}@/_6ex/[dddddll]
 \ar@{=}@/^6ex/[dddddrr] && \\
&& && \\
&& \ar@{-}@/_2ex/[dd]_{i_!} \ar@{-}@/^2ex/[dd]^{i_*}
 \ar@{}[dd]|{\oEcell{\;\Theta_i}}
 \ar@{}[l]|>>{\cong}
 \ar@{}[r]|>>{\cong} && \\
&&&& \\
&& \ar[dll]_{(ij)^*}
 \ar[drr]^{\;(ij)^*}
 \ar[dd]_{i^*}
 \ar@{}[llu]^{\eta\;\Ecell}
 \ar@{}[rru]_{\Ecell\;\eps} && \\
\ar@{=}[dd]
 \ar@{}[drr]|{\id \SEcell} &&&&
 \ar@{=}[dd]
 \ar@{}[dll]|{\NEcell \id^{-1}} \\
&& \ar@{=}[dd]
 \ar[dll]_{j^*}
 \ar[drr]^{\;j^*} && \\
\ar[drr]_{j_!}
 \ar@{}[rr]|{\;\;\eps\;\Ecell} &&&&
 \ar[dll]^{j_*}
 \ar@{}[ll]|{\Ecell\;\eta\;\;} \\
&&&&
}
}}
\quad=\quad
\vcenter { \hbox{
\xymatrix@L=1pt@C=15pt@R=15pt{
&& \ar@/_2ex/[dd]_{j_!}
 \ar@/^2ex/[dd]^{j_*}
 \ar@{}[dd]|{\oEcell{\;\Theta_j}}
 \ar@{=}@/_6ex/[dddddddll]
 \ar@{=}@/^6ex/[dddddddrr] && \\
&& && \\
&& \ar@/_2ex/[dd]_{i_!} \ar@/^2ex/[dd]^{i_*}
 \ar@{}[dd]|{\oEcell{\;\Theta_i}}
 \ar@{=}@/_8ex/[dddd]
 \ar@{=}@/^8ex/[dddd] && \\
&&&& \\
&& \ar[dd]_{i^*}
 \ar@{}[rd]|{\;\;\;\Ecell\;\eps}
 \ar@{}[ld]|{\eta\;\Ecell\;\;\;}
 && \\
 &&&& \\
&& \ar@{=}[dd]
 \ar[dll]_{j^*}
 \ar[drr]^{\;j^*}
 \ar@{}[ll]_{\eta\;\Ecell\;\;\;\;}
 \ar@{}[rr]^{\;\;\;\;\Ecell\;\eps} && \\
\ar[drr]_{j_!}
 \ar@{}[rr]|{\;\;\eps\;\Ecell} &&&&
 \ar[dll]^{j_*}
 \ar@{}[ll]|{\Ecell\;\eta\;\;} \\
&&&&
}
}}
\quad=\quad
\vcenter { \hbox{
\xymatrix@L=1pt@C=15pt@R=15pt{
& \ar@/_2ex/[dd]_{j_!}
 \ar@/^2ex/[dd]^{j_*}
 \ar@{}[dd]|{\oEcell{\;\Theta_j}}
 & \\
& & \\
& \ar@/_2ex/[dd]_{i_!} \ar@/^2ex/[dd]^{i_*}
 \ar@{}[dd]|{\oEcell{\;\Theta_i}}
 \ar@{=}@/_8ex/[dddd]
 \ar@{=}@/^8ex/[dddd] & \\
&& \\
& \ar[dd]_{i^*}
 \ar@{}[rd]|{\;\;\;\Ecell\;\eps}
 \ar@{}[ld]|{\eta\;\Ecell\;\;\;}
 & \\
&& \\
&&
}
}}
\quad \stackrel{\textrm{\eqref{it:Theta-special}}}{=} \; \Theta_j
\]
To see why this reduces to $\Theta_j$, we first use that the canonical isomorphisms $(ij)_!\cong i_!j_!$ and $(ij)_*\cong i_*j_*$ identify the units and counits of the adjunctions. Then we apply the triangle identities and conclude with the already proved property~\eqref{it:Theta-special} for~$i$.

The commutative diagram~\eqref{eq:aux-nabla-nat-ij} and the above verification allow us to compute the projections of $\nabla_{ij}(\Theta_i\Theta_j)$ along $\incl_C^*$ for two-thirds of the connected components $C\subset (ij/ij)$, namely those of type~\eqref{it:ij-comp-1} and~\eqref{it:ij-comp-3}. First for the diagonal ones, following~\eqref{eq:aux-nabla-delta}, we obtain the commutative diagram
\[
{\vcenter{\xymatrix@C=3em{
\Theta_i\Theta_j \ar@{}[r]|-{\in} \ar@{|->}[rd]
& [i_!j_!,i_*j_*] \ar@{=}[r]^-{\sim}
& [(ij)_!,(ij)_*] \ar[d]_-{[\id_!\,,\,(\id\inv)_*]\,\circ\, i^*} \ar[r]^-{\nabla_{ij}}_-{\cong}
& [r_1^*,r_2^*] \ar[d]^-{w''^*} \ar[r]^-{\Delta_{ij}^*}
& [\Id_{\MM(K)},\Id_{\MM(K)}]
\\
& \Theta_j \ar@{}[r]|-{\in} & [j_!,j_*] \ar[r]^-{\nabla_j}_{\cong}
& [q_1 ^*,q_2 ^*] \ar[ru]_-{\Delta_j^*}
}}}
\]
The triangle to the right commutes because $w''\Delta_j=\Delta_{ij}$, see~\eqref{Eq:diagonal-comparisons}. We then use that $\Delta_j^*\nabla_j(\Theta_j)$ is the identity by definition (apply \Cref{Cor:nabla}, \Cref{Thm:Theta} and \Cref{Prop:delta} with $j$ instead of~$i$) to conclude in this case.

Consider now a connected component $C\subset (ij/ij)$ of type~\eqref{it:ij-comp-3} in \Cref{Prop:ij-components}, namely such that $C$ is disjoint from~$\Delta_{ij}(K)$ and such that there exists a commutative diagram
\[
\xymatrix@C=5em{
D \ar[d]_-{w''\restr{D}}^-{\simeq} \ar@{ >->}[r]^-{\incl_D}
& (j/j) \ar[d]^-{w''}
\\
C \ar@{ >->}[r]^-{\incl_C}
& (ij/ij)
}
\]
for a connected component~$D\subset (j/j)$ disjoint from~$\Delta_j(K)$. We compute similarly as above
\[
{\vcenter{\xymatrix@C=3em{
\Theta_i\Theta_j \ar@{}[r]|-{\in} \ar@{|->}[rd]
& [i_!j_!,i_*j_*] \ar@{=}[r]^-{\sim}
& [(ij)_!,(ij)_*] \ar[d]_-{[\id_!\,,\,(\id\inv)_*]\,\circ\, i^*-} \ar[r]^-{\nabla_{ij}}_-{\cong}
& [r_1^*,r_2^*] \ar[d]^-{w''^*} \ar[r]^-{\incl_C^*}
& [\ldots,\ldots] \ar[d]_-{\simeq}^-{(w''\restr{D})^*}
\\
& \Theta_j \ar@{}[r]|-{\in} & [j_!,j_*] \ar[r]^-{\nabla_j}_-{\cong}
& [q_1 ^*,q_2 ^*] \ar[r]_-{\incl_D^*}
& [\ldots,\ldots]
}}}
\]
Since $D$ is disjoint from~$\Delta_j(K)$, we have by definition of $\Theta_j$ (\Cref{Thm:Theta} etc.) that $\incl_D^*\nabla_j(\Theta_j)=0$ and therefore $\incl_C^*\nabla_{ij}(\Theta_i\Theta_j)=0$; this finishes case~\eqref{it:ij-comp-3}.

We are left to prove the same relation $\incl_C^*\nabla_{ij}(\Theta_i\Theta_j)=0$ but now for the connected components $C\subset (ij/ij)$ of type~\eqref{it:ij-comp-2}, \ie such that $w'(C)\subset (i/i)$ is disjoint from $\Delta_i(H)$. In that case, we want to use that $\Theta_i$ vanishes on a suitable component of~$(i/i)$. This requires another little preparation. For this, consider the second square of \Cref{eq:ij-squares}, namely
\begin{equation}
\label{eq:aux-ij-2}%
\vcenter{\xymatrix@C=14pt@R=14pt{
& K \ar[dl]_j \ar[dr]^-{ij}
\\
H \ar[dr]_-{i} \ar@{}[rr]|-{\isocell{\id}}
&& G \ar@{=}[dl]
\\
&G
}}
\end{equation}
to which we also apply \Cref{Prop:nabla-nat}. In this case, the commutative diagram~\eqref{eq:nabla-nat} for $F=j_!$ and $F'=j_*$ provides the following commutative diagram:
\[
{\vcenter{\xymatrix@C=12em{
[i_!j_!,i_*j_*] \ar@{=}[d] \ar[r]^-{\nabla_i}_-{\cong\ \eqref{eq:nabla-FF'}}
& [p_1^*j_!,p_2^*j_*] \ar[d]^-{{w'}^*}
\\
[i_!j_!\,,\,i_*j_*] \ar[d]_-{[\id_!\,,\,(\id\inv)_*]}
& [{w'}^*p_1^*j_!,{w'}^*p_2^*j_*] \ar@{=}[d]_-{(p_{1}{w'}\,=\,j r_{1})}^-{(p_{2}{w'}\,=\,j r_{2})}
\\
[(ij)_!j^*j_!,(ij)_*j^*j_*] \ar[r]^-{\nabla_{ij}}_{\cong\ \eqref{eq:nabla-FF'}}
& [r_1^*j^*j_!,r_2^*j^*j_*]
}}}
\]
The mates $\id_!$ and $(\id\inv)_*$ on the left-hand column are the ones associated to~$\id_{ij}$ in~\eqref{eq:aux-ij-2}. Note that $j_!$ and $j_*$ are not touched by the left bottom morphism, since they appear as $F=j_!$ and $F'=j_*$ in~\eqref{eq:nabla-nat}. Similarly, the bottom isomorphism is ``$\nabla_{ij}$ on~$F=j^*j_!$ and $F'=j^*j_*$". Finally, the functor $w_{j,\Id,\id}\colon (ij/ij)\to (i/i)$ is here our~$w'$, hence the $w'$ on the right-hand column. We paste to the above diagram the obvious diagram obtained by collapsing $j^*j_!$ and $j^*j_*$ thanks to the unit $\eta\colon\Id\Rightarrow j^*j_!$ and counit $\eps\colon j^*j_*\Rightarrow \Id$ of $i_!\adj i^*$ and $i^*\adj i_*$ respectively. This provides the following commutative diagram:
\[
{\vcenter{\xymatrix@C=12em{
[i_!j_!,i_*j_*] \ar[d]_-{[\id_!\,,\,(\id\inv)_*]} \ar[r]^-{\nabla_i}_-{\cong\ \eqref{eq:nabla-FF'}}
& [p_1^*j_!,p_2^*j_*] \ar[d]^-{{w'}^*}
\\
[(ij)_!j^*j_!,(ij)_*j^*j_*] \ar[r]^-{\nabla_{ij}}_{\cong\ \eqref{eq:nabla-FF'}}
 \ar[d]_-{[\eta,\eps]}
& [r_1^*j^*j_!,r_2^*j^*j_*]
 \ar[d]^-{[\eta,\eps]}
\\
[(ij)_!,(ij)_*] \ar[r]^-{\nabla_{ij}}_{\cong\ \eqref{eq:nabla}}
& [r_1^*,r_2^*]
}}}
\]
Let us verify that the left vertical composite is equal to the map induced by the canonical isomorphisms $i_!j_!\cong (ij)_!$ and $i_*j_*\cong (ij)_*$. Equivalently, we can check that by \emph{precomposing} the vertical composite with the inverse of said induced map we get the identity; indeed, this sends every $\alpha\in [(ij)_!,(ij)_*]$ to
\[
\vcenter { \hbox{
\xymatrix@L=1pt@C=18pt@R=18pt{
&
\ar@{=}@/_6ex/[dddd]^{\;\;\oEcell{\eta\;}} &&
 \ar[ddl]|{(ij)_!}
 \ar[ddr]|{(ij)_*}
 \ar[dll]_{j_!}
 \ar[drr]^{\;j_*}
 \ar@{=}[ll]
 \ar@{=}[rr] &&
\ar@{=}@/^6ex/[dddd]_{\oEcell{\eps\;}\;\;} & \\
&
 \ar[dr]_{i_!\!\!}
 \ar@{=}@/_2ex/[dd]
 \ar@{}[dd]|{\oEcell{\eta\;}}
 \ar@{}[r]_>{\cong} &&&&
\ar[dl]^{i_*}
 \ar@{=}@/^2ex/[dd]
 \ar@{}[dd]|{\oEcell{\eps\;}}
 \ar@{}[l]^>{\cong} & \\
&&
\ar[dl]_{i^*\!\!}
 \ar@{=}[dr]
 \ar@{}[rr]^{\oEcell{\alpha}} &&
\ar[dr]^{i^*}
 \ar@{=}[dl] && \\
& \ar[dr]_<<<{j^*\!\!} &&
\ar[dl]|{(ij)^*}
 \ar[dr]|{(ij)^*}
 \ar@{=}[dd]
 \ar@{}[ll]|{\oEcell{\id}}
 \ar@{}[rr]|{\oEcell{\id^{-1}}} &&
\ar[dl]^<<<{\!\!j^*} & \\
&&
\ar@{=}[l]
 \ar[dr]_{(ij)_!}
 \ar@{}[r]|{\;\oEcell{\eps\;}} &&
 \ar@{=}[r]
 \ar[dl]^{(ij)_*}
 \ar@{}[l]|{\oEcell{\eta\;}\;} && \\
&&&&&&
}
}}
\quad = \quad
\vcenter { \hbox{
\xymatrix@L=1pt@C=18pt@R=18pt{
&&
\ar@{=}@/_10ex/[ddddl]
 \ar@{=}@/^10ex/[ddddr]
 \ar[ddl]|{(ij)_!}
 \ar[ddr]|{(ij)_*}
 && \\
&&&& \\
&
 \ar@{=}[dr]
 \ar@{}[rr]^{\oEcell{\alpha}}
 \ar@{}[l]_{\oEcell{\eta\;}} &&
 \ar@{=}[dl]
 \ar@{}[r]^{\oEcell{\eps\;}}
 & \\
&&
\ar[dl]_{(ij)^*\!\!}
 \ar[dr]^{(ij)^*}
 \ar@{=}[dd]
 && \\
&
 \ar[dr]_{(ij)_!}
 \ar@{}[r]|{\;\oEcell{\eps\;}} &&
 \ar[dl]^{(ij)_*}
 \ar@{}[l]|{\oEcell{\eta\;}\;} & \\
&&&&
}
}}
\quad = \quad
\alpha \,.
\]
Now we can compute $\incl_C^*\nabla_{ij}(\Theta_i\Theta_j)$ by following $\Theta_i\Theta_j$ from the upper-left corner, down to the lower-right one and then applying $\incl_C^*$. Let us do that. We obtain the following commutative diagram
\[
{\vcenter{\xymatrix@C=5em{
\Theta_i\Theta_j \ar@{}[r]|-{\in}
& [i_!j_!,i_*j_*] \ar@{=}[dd]^-{\cong} \ar[r]^-{\nabla_i}_-{\cong}
& [p_1^*j_!,p_2^*j_*] \ar[d]_-{{w'}^*} \ar[r]^-{\incl_{D'}^*}
& [\ldots,\ldots] \ar[d]_-{(w'\restr{C})^*}
\\
&& [r_1^*j^*j_!,r_2^*j^*j_*]
 \ar[d]_-{[\eta,\eps]} \ar[r]^-{\incl_C^*}
& [\ldots,\ldots] \ar[d]_-{[\eta,\eps]}
\\
& [(ij)_!,(ij)_*] \ar[r]^-{\nabla_{ij}}_-{\cong}
& [r_1^*,r_2^*] \ar[r]^-{\incl_C^*}
& [\ldots,\ldots]
}}}
\]
in which the left-hand half is the fruit of the above discussion whereas the right-hand part is simply obtained from the commutative diagram
\[
\xymatrix@C=5em{
C \ar[d]_-{w'\restr{C}} \ar@{ >->}[r]^-{\incl_C}
& (ij/ij) \ar[d]^-{w'}
\\
D' \ar@{ >->}[r]^-{\incl_{D'}}
& (i/i)
}
\]
coming from the fact that $C\subset (ij/ij)$ is of type~\eqref{it:ij-comp-2}, that is, so that the component $D'$ of the image~$w'(C)$ is disjoint from~$\Delta_i(H)$. This property guarantees that $\incl_{D'}^*\nabla_i(\Theta_i)=0$ by definition of~$\Theta_i$ (\Cref{Thm:Theta}). Following $\Theta_i\Theta_j$ in the above diagram, along the top horizontal arrow, we see that its image in the upper-right corner is simply $(\incl_{D'}^*\nabla_i(\Theta_i)\big)\Theta_j$ which is zero by what we just discussed.

To summarize, the image of $\Theta_i\Theta_j\in[i_!j_!,i_*j_*]$ in the canonically isomorphic $[(ij)_!,(ij)_*]$ has the property that its image under~$\nabla_{ij}$ projects to the identity on the diagonal components of~$(ij/ij)$ and to zero on all other components (if for two separate sets of reasons). So $\nabla_{ij}(\Theta_i\Theta_j)=\delta_{ij}=\nabla_{ij}(\Theta_{ij})$ and therefore $\Theta_i\Theta_j=\Theta_{ij}$.
This concludes the proof of \Cref{Thm:Theta-properties}.
\end{proof}

\begin{Rem}
\label{Rem:non-choosando}%
It is perhaps surprising to some readers that we could prove commutativity of~\eqref{eq:magic-partial} independently of the choice of the units and counits for $i_!\adj i^* \adj i_*$ and $j_!\adj j^* \adj j_*$. However, the \emph{same} units and counits appear in the construction of $\Theta_i$ and $\Theta_j$ and in the construction of $\gamma_!$ and $(\gamma\inv)_*$. Changing one of these units or counits (up to an automorphism of the corresponding adjoint) would have no impact on the commutativity of~\eqref{eq:magic-partial}, as two occurrences of the automorphism would cancel out.
\end{Rem}

Expanding on \Cref{Rem:non-choosando}, we now record the good behavior of $\Theta_i$ under a change of adjoints:

\begin{Prop} \label{Prop:Theta-proto-uniqueness}
For $i\colon H\into G$ faithful, suppose that we have two choices of left adjoints for the restriction functor $\MM(i)=i^*$, denoted $i_!$ and~$\overline{i_!}$; and suppose similarly that we have two right adjoints, written $i_*$ and~$\overline{i_*}$.
Then there is a commutative square
\begin{equation*}
\vcenter{
\xymatrix@C=4em@L=1ex{
i_! \ar@{=>}[r]^-{\Theta_i}
& i_* \ar@{=>}[d]^-{\psi}_-{\cong}
\\
{\overline{i_!}} \ar@{=>}[r]^-{\overline{\Theta_i}} \ar@{<=}[u]^-{\varphi}_-{\cong}
& {\overline{i_*}}
}}
\end{equation*}
where $\Theta_i\colon i_!\Rightarrow i_*$ and $\overline{\Theta_i}\colon \overline{i_!}\Rightarrow \overline{i_*}$ are the comparison morphisms of \Cref{Thm:Theta-properties} constructed, respectively, from the adjunctions $i_!\adj i^* \adj i_*$ and $\overline{i_!}\adj i^* \adj \overline{i_*}$; and where $\varphi \colon i_!\overset{\sim}{\Rightarrow} \overline{i_!}$ denotes the unique isomorphism identifying the two left adjunctions, and $\psi \colon i_*\overset{\sim}{\Rightarrow} \overline{i_*}$ the unique one identifying the two right adjunctions.
\end{Prop}
\begin{proof}
By construction, both $\Theta_i$ and $\Theta_i':=\psi^{-1}\circ \overline{\Theta_i} \circ \phi$ are natural transformations $i_! \Rightarrow i_*$ satisfying the equations \eqref{eq:proto(f)} and \eqref{eq:proto(g)} of \Cref{Rem:Theta}. Hence they must coincide by the uniqueness of this characterization.
\end{proof}

We conclude this section by explaining how $\Theta_i\colon i_!\Rightarrow i_*$ is natural in~$\MM$.

\begin{Prop}
\label{Prop:Theta-nat}%
Let $\MM,\NN\colon \GG^\op\to \ADD$ be two strict 2-functors satisfying all properties of a Mackey 2-functor, except perhaps ambidexterity, as in \Cref{Hyp:Theta}. Let $t\colon \MM\to \NN$ be a pseudo-natural transformation, with components $t_G\colon \MM(G)\to \NN(G)$ for all~$G\in \GG_0$ and $t_u\colon u^* t_G\isoEcell t_H u^*$ for all $u\colon H\to G$ in~$\GG_1$ (see \Cref{Ter:Hom_bicats}).
Let $i\colon H\into G$ be in~$\JJ$ and write $\Theta_i^{\scriptscriptstyle \MM}\colon i_!^{\scriptscriptstyle \MM}\Rightarrow i_*^{\scriptscriptstyle \MM}$ and $\Theta_i^{\scriptscriptstyle \NN}\colon i_!^{\scriptscriptstyle \NN}\Rightarrow i_*^{\scriptscriptstyle \NN}$ for the two natural transformations constructed in \Cref{Thm:Theta}, for~$\MM$ and $\NN$ respectively. Then the following square of natural transformations between functors $\MM(H)\to \NN(G)$ commutes:
\begin{equation}
\label{eq:Theta-nat}%
\vcenter{\xymatrix@C=4em@L=1ex{
t_G i_!^{\scriptscriptstyle \MM} \ar@{=>}[r]^-{\Displ t_G \Theta_i^{\MM}}
& t_G i_*^{\scriptscriptstyle \MM} \ar@{=>}[d]^-{\Displ(t_i)_*}
\\
i_!^{\scriptscriptstyle \NN} t_H \ar@{=>}[r]^-{\Displ\Theta_i^{\NN} t_H} \ar@{=>}[u]^-{\Displ (t_i\inv)_!}
& i_*^{\scriptscriptstyle \NN} t_H
}}
\end{equation}
\end{Prop}

\begin{proof}
By construction of~$\Theta_i^{\scriptscriptstyle \NN}$, it suffices to verify that the composition around the top $(t_i)_*\circ (t_G\Theta_i^{\scriptscriptstyle \MM})\circ (t_i\inv)_!$ has the defining property of~$\Theta_i^{\scriptscriptstyle \NN}$, namely that under the isomorphism $\nabla_i\colon [i_! t_H,i_*t_H]\isoto [p_1^* t_H, p_2^*t_H]$ of \Cref{Prop:nabla} (for $F=F'=t_H$), our around-the-top composite should map to~$\delta_i^{\scriptscriptstyle \NN} t_H$. We repeat the notation of \Cref{Thm:Theta}, for the reader's convenience: $p_1$ and $p_2$ are the two projections $(i/i)\to H$ in the self-iso-comma for~$i$, as in~\eqref{eq:self-ic}, and $\delta_i^{\scriptscriptstyle \NN}\colon p_1^*=\NN(p_1)\Rightarrow p_2^*=\NN(p_2)$ is the distinguished morphism of \Cref{Prop:delta} for~$\NN$. Unsurprisingly, we must first establish the compatibility of~$\delta_i$ for~$\MM$ and for~$\NN$, which is expressed as follows:

\begin{Lem} \label{Lem:comm-square-comp}
There is a commutative square
\begin{equation}
\label{eq:delta_i-M-N}%
\vcenter{\xymatrix@C=6em@L=1ex{
p_1^* t_H \ar@{=>}[r]^-{\Displ \delta_i^{\scriptscriptstyle \NN}\,t_H} \ar@{=>}[d]^-{\simeq}_-{\Displ t_{p_1}}
& p_2^* t_H \ar@{=>}[d]_-{\simeq}^-{\Displ t_{p_2}}
\\
t_{(i/i)}p_1^* \ar@{=>}[r]^-{\Displ t_{(i/i)}\delta_i^{\scriptscriptstyle \MM}}
& t_{(i/i)}p_2^*
}}
\end{equation}
of natural transformations between functors $\MM(H)\to \NN(i/i)$.
\end{Lem}

\begin{proof}
Let $C$ be either of the subgroupoids $\Delta_i(H)$ or $C= (i/i)\smallsetminus \Delta_i(H)$ of the comma groupoid~$(i/i)$, and let $\incl_C\colon C\hookrightarrow (i/i)$ be the inclusion functor. Recall that $\delta_i^\MM$ and $\delta_i^\cat N$ are characterized in \Cref{Prop:delta} by the property that if we apply $\incl^*_C\colon \MM(i/i)\to \MM(C)$ (respectively $\incl^*_C\colon \cat N(i/i)\to \cat N(C)$) to it, we obtain the identity or zero according as to whether $C=\Delta_i(H)$ or $C= (i/i)\smallsetminus \Delta_i(H)$.

Let us apply $\incl^*_C$ to~\eqref{eq:delta_i-M-N}. The key to showing that it commutes in both cases is to `conjugate' with the components of $t$ at all the relevant 1-cells:
\[
\xymatrix{
&& \MM (i/i)
 \ar[dd]^{\incl_C^*} \ar[rr]^-{t_{(i/i)}} &&
 \cat N(i/i)
 \ar[dd]^{\incl_C^*}
 \ar@{}[ddll]|{\quad \underset{t_{\incl_C}}{\overset{\simeq}{\SWcell}} } \\
\MM(H)
 \ar@/^3ex/[urr]^-{p_1^*}
 \ar@/_2.6ex/[urr]_-{p_2^*}
 \ar@{}[urr]|{\SEcell\;\delta^\MM_i}
 \ar@/^2.6ex/[drr]
 \ar@/_3ex/[drr]
 \ar@{}[drr]|{\SWcell\, \id \textrm{ or }0} && && \\
&& \MM(C) \ar[rr]_-{t_H} && \cat N(C)
}
\]
This also makes use of other 2-cell components of~$t$ (not depicted above) as well as of its functoriality property (\Cref{Ter:Hom_bicats}). In fact, such an argument is most easily written up in terms of string diagrams, hence we postpone a fully detailed proof until the latter have been introduced in \Cref{sec:string_diagrams}; see \Cref{Lem:full-details-comp}.
\end{proof}

Let us now turn to the announced property of the `around-the-top' composition $(t_i)_*\circ (t_G\Theta_i^{\scriptscriptstyle \MM})\circ (t_i\inv)_!$ in~\eqref{eq:Theta-nat}. Unpacking the definitions of the mates $(t_i\inv)_!$ and $(t_i)_*$, applying $p_1^*$ and $p_2^*$ and using $\lambda^*\colon p_1^* i^*\Rightarrow p_2^* i^*$ for $\lambda$ as in~\eqref{eq:self-ic}, we need to show the commutativity of the central region (marked (?)) of the following diagram:
\[
\xymatrix@C=1.2em@L=1ex{
&& t_{i/i} p_1^*i^*i_! \ar@{=>}[rr]^-{t_{i/i}\lambda^*\Theta^{\scriptscriptstyle \MM}_i} \ar@{=}@/_2em/[llddd] \ar@{<=}[d]_-{t_{ip_1}}^-{\simeq}
&& t_{i/i} p_2^*i^*i_* \ar@{=}@/^2em/[rrddd] \ar@{<=}[d]^-{t_{ip_2}}_-{\simeq}
\\
&& p_1^*i^* t_{G} i_! \ar@{=>}[rr]^-{\lambda^* t_{G}\Theta^{\scriptscriptstyle \MM}_i}
&& p_2^*i^*t_{G} i_*
\\
& p_1^*i^* t_{G} i_! \ar@{=}[ru] \ar@{=>}[d]_-{t_{i}}^-{\simeq} \ar@{=>}[r]_-{\eta}
& p_1^*i^*i_!i^* t_{G} i_! \ar@{<=}[d]_-{t_{i}\inv}^-{\simeq} \ar@{=>}[u]_-{\eps}
&& p_2^*i^*i_*i^* t_{G} i_* \ar@{=>}[d]^-{t_{i}}_-{\simeq} \ar@{<=}[u]_-{\eta} \ar@{=>}[r]_-{\eps}
& p_2^*i^*t_{G} i_* \ar@{=}[lu] \ar@{=>}[d]^-{t_{i}}_-{\simeq}
&
\\
t_{i/i} p_1^*i^*i_!
& p_1^* t_{H} i^* i_! \ar@{=>}[l]_-{t_{p_1}}^-{\simeq} \ar@{=>}[r]_-{\eta}
& p_1^*i^*i_! t_{H} i^*i_!
&& p_2^*i^*i_* t_{H} i^*i_* \ar@{=>}[r]_-{\eps}
& p_2^* t_{H} i^* i_* \ar@{=>}[r]^-{t_{p_2}}_-{\simeq}
& t_{i/i} p_2^*i^*i_*
\\
t_{i/i} p_1^* \ar@{=>}[u]_-{\eta}
 \ar@{=>}@/_4em/[rrrrrr]_-{t_{i/i}\delta_i^{\scriptscriptstyle \MM}}
& p_1^* t_{H} \ar@{=>}[l]_-{t_{p_1}}^-{\simeq} \ar@{=>}[u]_-{\eta} \ar@{=>}[r]^-{\eta}
 \ar@{=>}@/_2em/[rrrr]^-{\delta_i^{\scriptscriptstyle \NN} t_H}_-{\textrm{(\ref{eq:delta_i-M-N})}}
& p_1^*i^*i_! t_{H} \ar@{=>}[u]_-{\eta} \ar@{}[rru]|-{\textrm{(?)}}
&& p_2^*i^*i_* t_{H} \ar@{=>}[r]^-{\eps} \ar@{<=}[u]_-{\eps}
& p_2^* t_{H} \ar@{=>}[r]^-{t_{p_2}}_-{\simeq} \ar@{<=}[u]_-{\eps}
& t_{i/i} p_2^* \ar@{<=}[u]_-{\eps}
}
\]
All seven squares commute by naturality (the top one by `naturality' of~$t$, as in \Cref{Ter:Hom_bicats}, applied to $\lambda\colon i p_1\Rightarrow i p_2\colon (i/i)\to G$). The two triangles commute by the unit-counit relation of the adjunctions~$i_!\adj i^*$ and $i^*\adj i_*$. The bottom curvy area commutes by the compatibility of the~$\delta_i$ as indicated by the reference to~\eqref{eq:delta_i-M-N}. The two `shoulders' (pentagons) commute by the `functoriality' of~$t$ (\Cref{Ter:Hom_bicats}). Finally, the outer diagram (hexagon) commutes by the construction of~$\Theta_i^{\scriptscriptstyle \MM}$ as in \Cref{Thm:Theta}, postwhiskered by~$t_{i/i}$. This proves that the central area marked~(?) does indeed commute, as announced.
\end{proof}

\bigbreak
\section{Rectification of Mackey 2-functors}
\label{sec:rectification}%
\medskip

In this section, we want to show that ambidexterity $i_!\simeq i_*$ can only happen if the preferred $\Theta_i\colon i_!\Rightarrow i_*$ of \Cref{Thm:Theta} is an isomorphism. This then allows us to prove the Rectification \Cref{Thm:rectification}.
We begin by isolating an `induction on the order of~$G$' that reduces the proof of $\Theta_i\colon i_!\Rightarrow i_*$ being an isomorphism to the following test: Does $i^* \Theta_i$ being an isomorphism imply that $\Theta_i$ is an isomorphism?

As in the last two sections, $\JJ\subseteq \GG \subseteq \groupoid$ are as in \Cref{Hyp:GG}.

\begin{Prop}
\label{Prop:induction-on-G}%
Let $\MM\colon \GG^\op\to \ADD$ satisfy all properties of a Mackey 2-functor except perhaps ambidexterity, as in \Cref{Hyp:Theta}. Consider $\Theta_i\colon i_!\Rightarrow i_*$ as in \Cref{Thm:Theta}. Suppose $\MM$ further satisfies the following property:
\begin{enumerate}[\noindent\rm(A)]
\item
\label{it:A}%
If $i\colon H\into G$ is faithful, with $H$ and $G$ connected and non-empty, and if $i^*\Theta_i$ is an isomorphism, then $\Theta_i$ is an isomorphism.
\end{enumerate}
Then $\Theta_i$ is an isomorphism for all faithful~$i\colon H\into G$.
\end{Prop}

\begin{proof}
As said, we proceed by induction on the `order' of~$G$, that is, the maximum of the orders of the finite groups~$\Aut_G(x)$ over all~$x\in G$. Note that $i_!=0=i_*$ when $H$ is empty, in which case the result is trivial. By \Cref{Thm:Theta-properties}\,\eqref{it:Theta-eq} we know that $\Theta_i$ is an isomorphism whenever~$i$ is an equivalence. In particular, $\Theta_i$ is an isomorphism when $G$ is trivial, \ie equivalent to the trivial group. So we can assume the result known for every $i'\colon H'\into G'$ with $G'$ of order less than that of~$G$. By additivity, \Cref{Thm:Theta-properties}\,\eqref{it:Theta-add}, we can also assume that $H$ and~$G$ are connected, and we can still assume $H$ non-empty. Then if $i\colon H\into G$ is full it is an equivalence and we are done. So we are reduced to the situation where $i\colon H\into G$ is not full and $H$ and $G$ are connected, with $H$ non-empty. (Thinking `groups', this is $H$ being a proper subgroup of~$G$.) In particular, $H$ has order strictly less than $G$ and we can apply the induction hypothesis to the faithful $p_2\colon (i/i)\into H$ which appears in the self-iso-comma of~\eqref{eq:self-ic}:
\[
\xymatrix@C=10pt@R=10pt{
& (i/i) \ar[dl]_-{p_1} \ar@{ >->}[dr]^-{p_2}
\\
H \ar@{ >->}[dr]_-{i} \ar@{}[rr]|-{\isocell{\lambda}}
&& H \ar[dl]^-{i}
\\
&G
}
\]
So $\Theta_{p_2}$ is an isomorphism by induction hypothesis. We now use \Cref{Thm:Theta-properties}\,\eqref{it:Theta-magic} applied to the above Mackey square to obtain the commutative diagram
\[
\vcenter{
\xymatrix@C=4em@L=1ex{
i^* i_! \ar@{=>}[r]^-{i^*{\Displ\Theta_i}}
& i^* i_* \ar@{=>}[d]^-{(\lambda\inv)_*}_-{\cong}
\\
{p_2}_! {p_1}^* \ar@{=>}[r]^-{{\Displ\Theta_{p_2}}\ {p_1}^*}_-{\cong} \ar@{=>}[u]^-{\lambda_!}_-{\cong}
& {p_2}_*{p_1}^*
}}
\]
in which the vertical morphisms are isomorphism by the BC-property~\Mack{3}. This proves that $i^*\Theta_i$ is an isomorphism and we are reduced to property~\eqref{it:A}.
\end{proof}

Let us now show that if ambidexterity holds then the isomorphism~$i_!\simeq i_*$ can be chosen to be the canonical $\Theta_i\colon i_!\Rightarrow i_*$ of Theorems~\ref{Thm:Theta} and~\ref{Thm:Theta-properties}.

\begin{Thm}
\label{Thm:Mackey-and-Theta}%
Let $\MM\colon \GG^\op\to \ADD$ be a strict 2-functor satisfying all properties of a Mackey 2-functor, except perhaps ambidexterity, as in \Cref{Hyp:Theta}. Then the following are equivalent:
\begin{enumerate}[\rm(i)]
\item
The 2-functor~$\MM$ is a Mackey 2-functor, \ie \Mack{4} holds (\Cref{Def:Mackey-2-functor}).
\smallbreak
\item
For every faithful $i\colon H\into G$ the natural transformation $\Theta_i$ of \Cref{Thm:Theta} is an isomorphism.
\end{enumerate}
\end{Thm}

\begin{proof}
Of course, if $\Theta_i$ is an isomorphism, then $i_!\simeq i_*$ and we have ambidexterity~\Mack{4} and \Cref{Def:Mackey-2-functor} is complete. So the interesting direction is the converse. Suppose that $i_!\simeq i_*$ and let us prove that the morphism $\Theta_i\colon i_!\Rightarrow i_*$ of \Cref{Thm:Theta} is indeed an isomorphism, by only using the properties of~$\Theta_i$ listed in \Cref{Thm:Theta-properties}. We use \Cref{Prop:induction-on-G} (which itself relies on induction on the `order' of~$G$). We have to verify that $\MM$ satisfies property~\eqref{it:A}. So suppose that $i\colon H\into G$ is such that $i^*\Theta_i$ is an isomorphism. Then $\Theta_i$ is also an isomorphism as well thanks to the general \Cref{Cor:detect-iso}. Indeed, $i_!$ is also a \emph{right} (sic) adjoint of~$i^*$ by ambidexterity.
\end{proof}

\begin{Thm}[Rectification Theorem]
\label{Thm:rectification}%
Let $\MM\colon \GG^\op\to \ADD$ be a Mackey 2-functor (\Cref{Def:Mackey-2-functor}). Then we can choose for each faithful $i\colon H\into G$ in~$\GG$ a single two-sided adjoint
\[
i_!=i_*\colon \MM(H)\to \MM(G)
\]
of restriction~$i^*\colon \MM(G)\to \MM(H)$ and units and counits
\[
\leta\colon \Id \Rightarrow i^*i_!
\qquad
\leps\colon i_! i^* \Rightarrow \Id
\qquadtext{and}
\reta\colon \Id \Rightarrow i_* i^*
\qquad
\reps\colon i^*i_* \Rightarrow \Id
\]
for $i_!\adj i^*$ and $i^* \adj i_*$ respectively, such that the additional properties (Mack\,\ref{it:Mack-5})--(Mack\,\ref{it:Mack-10}) below hold true. Moreover, the choice of such ambidextrous adjunctions is unique in the following strong sense: If $(\overline{i_!}, \overline{\leta}, \overline{\leps} ,\overline{\reta},\overline{\reps})_{i\in \JJ}$ is another such choice, there exist unique isomorphisms $i_! \cong \overline{i_!}$ matching both sets of units and counits.
\begin{enumerate}[\rm({Mack}\,1)]
\setcounter{enumi}{4}
\item
\label{it:Mack-5}%
\emph{Additivity of adjoints:} Whenever $i=i_1\sqcup i_2\colon H_1\sqcup H_2\into G$, under the identification $\MM(H_1\sqcup H_2)\cong \MM(H_1)\oplus \MM(H_2)$, we have
\[
(i_1\sqcup i_1)_!=\big((i_1)_! \ (i_2)_!\big)
\qquadtext{and}
(i_1\sqcup i_1)_*=\big((i_1)_* \ (i_2)_*\big)
\]
with the obvious `diagonal' units and counits. (See \Cref{Rem:add-adjoint}.)
\smallbreak
\item
\label{it:Mack-6}%
\emph{Two-sided adjoint equivalences:} Whenever $i^*$ is an equivalence, the units and counits are isomorphisms and $(\leta)\inv=\reps$ and $(\leps)\inv=\reta$. Furthermore when $i=\Id$ we have $i_!=i_*=\Id$ with identity units and counits.
\smallbreak
\item
\label{it:Mack-7}%
\index{strict Mackey formula}\index{Mackey formula!strict --}%
\emph{Strict Mackey Formula:} For every Mackey square (\Cref{Def:Mackey-square}) with~$i$ and~$j$ faithful
\[
\vcenter{
\xymatrix@C=14pt@R=14pt{
& L \ar[ld]_-{v} \ar@{ >->}[dr]^-{j}
 \ar@{}[dd]|-{\isocell{\gamma}}
\\
H \ar@{ >->}[dr]_-{i}
&& K \ar[dl]^-{u} \\
&G &
}}
\]
the two isomorphisms $\gamma_!\colon j_! v^* \stackrel{\sim}{\Rightarrow} u^* i_!$ and $(\gamma\inv)_*\colon u^* i_* \stackrel{\sim}{\Rightarrow} j_* v^*$ of~\Mack{3} are in fact inverse to one another
\begin{equation}
\label{eq:strict-Mackey}%
\gamma_!\circ (\gamma\inv)_* =\id
\qquadtext{and}
(\gamma\inv)_* \circ \gamma_!=\id
\end{equation}
via the identification of their source and target $u^*i_!=u^*i_*$ and $j_!v^*=j_*v^*$.
(If the above square is only assumed to be a \emph{partial} Mackey square (\Cref{Def:partial-Mackey}) then $(\gamma\inv)_* \circ \gamma_!=\id$ still holds true.)
\smallbreak
\item
\label{it:Mack-8}%
\emph{Agreement of pseudo-functors:} The two pseudo-functors $i\mapsto i_!$ and $i\mapsto i_*$ (\Cref{Rem:pseudo-func-of-adjoints}) coincide: For every 2-cell $\alpha\colon i\Rightarrow i'$ between faithful $i,i'\colon H\into G$ we have $\alpha_!=\alpha_*$ as morphisms between the functors $i_!=i_*$ and $i'_!=i'_*$. For every composable faithful morphisms $j\colon K\into H$ and $i\colon H\into G$ the isomorphisms $(ij)_!\cong i_! j_!$ and $(ij)_*\cong i_* j_*$ coincide.
\smallbreak
\item
\label{it:Mack-9}%
\emph{Special Frobenius:} For every faithful $i\colon H\into G$, the composite of the left unit and right counit $\Id \stackrel{\leta\;\;\,}{\Rightarrow} i^*i_!=i^*i_* \stackrel{\reps\;\;}{\Rightarrow} \Id$ is the identity of~$\Id_H$.
\smallbreak
\item
\label{it:Mack-10}%
\emph{Off-diagonal vanishing:} For every faithful $i\colon H\into G$, the natural transformation
$\incl_C^* \left(
p_1^* \stackrel{p_1^*\;\leta\;}{\Longrightarrow} p_1^*i^*i_! \stackrel{\lambda^* \;\id\;}{\Longrightarrow} p_2^* i^*i_* \stackrel{p_2^*\;\reps\;}{\Longrightarrow} p_2^*
\right)$
is zero, where we denote by $\incl_C\colon C\hookrightarrow (i/i)$ the inclusion functor of the complement $C:=(i/i)\smallsetminus \Delta_i(H)$ of the diagonal component (see \Cref{Prop:Delta}) in the self-isocomma~$(i/i)$.
\end{enumerate}
\end{Thm}

\begin{proof}
Choose left and right adjoints $i_!\adj i^*\adj i_*$ for every faithful~$i\colon H\into G$ following \Cref{Conv:Mack-5-6} as usual. Then \Cref{Thm:Theta} gives us natural transformations $\Theta_i\colon i_!\Rightarrow i_*$ for all faithful~$i$, which satisfy all the nice properties listed in \Cref{Thm:Theta-properties}. By \Cref{Thm:Mackey-and-Theta}, ambidexterity forces this $\Theta_i$ to be an isomorphism and therefore we can carry the units and counits of the right adjunction $i^*\adj i_*$ over to $i_!$ via $\Theta_i$ to make $i_!$ a two-sided adjoint. With this \emph{new} set of adjoints $i_!\adj i^* \adj i_!$ one can verify directly from the construction in \Cref{Thm:Theta} that the \emph{new} $\Theta_i^{\textrm{new}}\colon i_!\Rightarrow i_!$ is nothing but the identity, see~\eqref{eq:Theta-delta} in which $\Theta_i$ is absorbed into $\reps^{\textrm{new}}$. Now the seven properties \eqref{it:Theta-add}--\eqref{it:Theta-vanishing} of \Cref{Thm:Theta-properties} give us the announced properties (Mack\,\ref{it:Mack-5})--(Mack\,\ref{it:Mack-10}). Finally, the claimed uniqueness of the data $(i_!, \leta, \leps, \reta, \reps)_{i\in \JJ}$ is the direct translation of \Cref{Prop:Theta-proto-uniqueness}.
\end{proof}

\begin{Cor}
\label{Cor:pseudo-func-_!_*}%
Let $\MM\colon \GG^\op\to\ADD$ be a Mackey 2-functor. Then there exists a pseudo-functor $(-)_!\colon \JJ^{\co}\to\ADD$ on faithful morphisms, which agrees with~$\MM$ on 0-cells, such that each $i_!$ is a two-sided adjoint of~$i^*$ and such that for every Mackey square~\eqref{eq:Mackey-square} in which~$i$ and~$j$ are faithful, we have
\[
(\gamma_!)\inv=(\gamma\inv)_*
\]
as in~\eqref{eq:magic-compact}.
\qed
\end{Cor}

\end{chapter-three}
%
\chapter{Examples}
\label{ch:Examples}%
\bigbreak
\begin{chapter-four}

We assemble here a number of examples, and some non-examples, of Mackey 2-functors $ G\mapsto \MM(G) $ as in \Cref{Def:Mackey-2-functor}, defined on all finite groupoids, or sometimes only defined on a 2-subcategory $\GG$ of $\groupoid$ as in \Cref{Hyp:GG}. These examples are essentially all well-known.

\bigbreak
\section{Examples from additive derivators}
\label{sec:add-der-Mackey}%
\medskip

As announced, the most straightforward source of Mackey 2-functors comes via stable homotopy theories. We now explain this in some detail. We refer to Groth~\cite{Groth13} for further details on derivators. We consider an additive derivator, that is, a (strict) 2-functor
\[
\DD\colon \Cat^{\op}\too \ADD
\]
from the 2-category $\Cat$ of small categories (or any suitable diagram category~$\mathsf{Dia}\subseteq\Cat$ containing finite groupoids) to the 2-category $\ADD$ of additive categories and additive functors (\Cref{sec:additive-sedative}), satisfying the axioms (Der\,\ref{Der-1})-(Der\,\ref{Der-4}) of derivators; see~\cite[Def.\,1.5]{Groth13}. In telegraphic style, they are:
\begin{enumerate}[\rm(Der\,1)]
\smallbreak
\item
\label{Der-1}%
Additivity: $\DD(\coprod_{\alpha\in\aleph} I_\alpha)\cong \prod_{\alpha\in\aleph}\DD(I_\alpha)$.
\smallbreak
\item
\label{Der-2}%
Pointwise detection of isomorphisms: For each small category~$I$, the functor $\prod_{x\in I_0}x^*\colon \DD(I)\too\prod_{x\in I_0}\DD(1)$ is conservative ($1$ is the final category).
\smallbreak
\item
\label{Der-3}%
For every functor $u\colon I\to J$, the functor $u^*\colon \DD(J)\to \DD(I)$ has a left adjoint $u_!$ and a right adjoint~$u_*$.
\smallbreak
\item
\label{Der-4}%
$\DD$ satisfies base-change (or Beck-Chevalley) with respect to any comma square (\Cref{Def:comma}) of small categories. See~\cite[Prop.\,1.26]{Groth13}.
(The left and right BC are here equivalent because every $u^*$ has adjoints.)
\[
\vcenter{\xymatrix@C=14pt@R=14pt{
& (p/q) \ar[dl]_-{r} \ar[dr]^-{s}
\\
I \ar[dr]_-{p} \ar@{}[rr]|-{\oEcell{\gamma}}
&& J \ar[dl]^-{q}
\\
& K
}}
\qquadtext{$\leadsto$}
\gamma_! \colon s_! r^* \overset{\sim}{\Rightarrow} q^* p_!
\quadtext{and}
\gamma_* \colon p^* q_* \overset{\sim}{\Rightarrow} r_* s^*.\kern-3em
\]
\end{enumerate}

A large class of examples comes from stable derivators. Recall that a derivator $\DD\colon \Cat^\op\to \CAT$ is \emph{stable} if it is pointed (the base category $\DD(1)$ has a zero object) and homotopy pullback squares and homotopy pushout squares coincide, or equivalently its suspension functor is an equivalence. For instance the derivator associated to any stable homotopy theory is stable. From~\cite[Cor.\,4.14]{Groth13}, a stable derivator $\DD$ is automatically additive.

Given a derivator~$\DD$, one can restrict the input diagrams to finite groupoids. For every faithful morphism $i\colon H\into G$, restriction $i^*\colon\DD(G)\to \DD(H)$ therefore admits adjoints on both sides. We claim that doing so for an additive derivator forces those adjoints, induction~$i_!$ and coinduction~$i_*$, to coincide. This is the content of ambidexterity~\Mack{4}. We have prepared the stage for this result in \Cref{sec:Theta}.

\begin{Thm}[Ambidexterity]
\label{Thm:ambidex-der}%
Let $\DD$ be an additive derivator. Then there exists for every faithful morphism $i\colon H\into G$ of finite groupoids an isomorphism
\[
\Theta_i:i_! \overset{\sim}{\Rightarrow} i_*
\]
of functors $\DD(H)\to \DD(G)$. In other words, the restriction of~$\DD$ to finite groupoids $\MM=\DD\restr{\groupoid}:\groupoid^{\op}\too \ADD$ is a Mackey 2-functor.
\end{Thm}

\begin{proof}
Consider $\MM=\DD\restr{\groupoid}$, the restriction of the 2-functor $\DD$ to finite groupoids, and apply Theorem~\ref{Thm:Theta}. Its hypotheses are satisfied, namely $\MM$ has all properties of a Mackey 2-functor except perhaps ambidexterity; indeed~\Mack{1}, \Mack{2} and~\Mack{3} are special cases of (Der\,\ref{Der-1}), (Der\,\ref{Der-3}) and~(Der\,\ref{Der-4}), respectively. So Theorem~\ref{Thm:Theta} gives us~$\Theta_i\colon i_!\Rightarrow i_*$ and it only remains to prove it is invertible. We already know that~$\Theta_i$ satisfies the properties of \Cref{Thm:Theta-properties}, most notably~\eqref{it:Theta-magic}. We use \Cref{Prop:induction-on-G}, or `induction on the order of~$G$', which reduces the problem to proving that~$\MM$ has property~\eqref{it:A}. So, let $i\colon H\into G$ be faithful, with $H$ and~$G$ connected and non-empty, such that $i^*\Theta_i$ is an isomorphism and let us show that $\Theta_i$ is an isomorphism in this case. (Here we cannot use \Cref{Cor:detect-iso} anymore, as we did in the proof of \Cref{Thm:Mackey-and-Theta}, since we do not know that $i_!$ is a \emph{right} adjoint of~$i^*$ -- we are actually proving exactly that.) We want to use~(Der\,\ref{Der-2}). Let $x\colon 1\to G$ be an object of~$G$. Since $G$ is connected and $H$ is non-empty, $i\colon H\into G$ is essentially surjective and thus there exists $y\colon 1\to H$ and an isomorphism $i\circ y\simeq x$. Therefore $x^*\Theta_i\simeq y^*i^*\Theta_i$ is an isomorphism since we assume $i^*\Theta_i$ is. This holds for every object~$x$ of~$G$. In short, $\Theta_i\colon i_!\Rightarrow i_*\colon \MM(H)\to \MM(G)$ is pointwise an isomorphism and therefore an isomorphism by~(Der\,\ref{Der-2}).
\end{proof}

\begin{Rem}
The educated practitioner of derivators will easily extend the above result to any additive derivator $\DD\colon \mathsf{Dia}^\op\to \ADD$ defined on any smaller class of diagrams $\mathsf{Dia}\subset\Cat$, as long as $\mathsf{Dia}$ contains all finite groupoids. This extension can be relevant when dealing with `small' derivators, like the one of compact objects~$\DD^c$ in a `big' derivator~$\DD$.
\end{Rem}

\begin{Exa}[Model categories with additive homotopy category] \label{Exa:model-cats}
Let $\cat{C}$ be a (sufficiently complete and cocomplete) model category with class of weak equivalences~$W$, then by a powerful theorem of Cisinski~\cite{Cisinski03} it gives rise by pointwise localization to a derivator $\DD\colon J\mapsto \cat{C}^J[W_J\inv]$. Here of course we denote by
\[ W_J:=\{f\in \Mor (\cat{C}^J) \mid x^*(f) \textrm{ belongs to } W\subseteq \cat{C}, \; \forall x\in \Obj(J) \} \]
the class of pointwise weak equivalences in the diagram category $\cat{C}^J$.
By~\cite[Prop.\,5.2]{Groth13}, the derivator is additive as soon as the homotopy category $\cat{C}[W\inv]$ of $\cat{C}$ happens to be additive (\eg $\cat{C}$ could be a stable model category, or an exact model category as in~\cite{Stovicek13}).
By \Cref{Thm:ambidex-der}, we obtain in this situation a Mackey 2-functor
\[
G\mapsto \MM(G)=\MM_{(\cat{C},W)}(G):=\cat{C}^G [W_G\inv] \,,
\]
which really only depends on the relative category $(\cat{C},W)$, that is, on the underlying category and the weak equivalences of the model category~$\cat{C}$.
\end{Exa}

Of the many examples arising this way, let us mention three notable ones, starting with the minimal way (with weak equivalences being isomorphisms):

\begin{Exa}[Linear representations] \label{Exa:lin_reps}
\index{modules as Mackey functor}%
Let $\kk$ be any commutative ring. Then the \emph{representable} additive derivator $J\mapsto (\Mod \kk)^J$ restricts to a Mackey 2-functor $\MM\colon G \mapsto (\Mod \kk)^G$; its value at a finite group $G$ is the category $\Mod(\kk G)$ of $\kk$-linear representations of~$G$. As any Grothendieck category, $\cat{A}=\Mod\kk$ represents an additive derivator $J\mapsto \cat{A}^J$, in which the adjoints $u_*$ and $u_!$ are Kan extensions (without need to derive them).
\end{Exa}

\begin{Exa}[Derived categories]
\label{Exa:derived-cat}%
\index{derived categories as Mackey functor}%
Let $\kk$ be any commutative ring. Then the category $\Ch(\kk)$ of chain complexes of $\kk$-modules admits model structures where the weak equivalences $W$ are the quasi-isomorphisms, so it fits Example~\ref{Exa:model-cats}. The value of the resulting Mackey 2-functor $\MM$ at a finite group $G$ is the derived category $\Der(\kk G)$ of the group algebra of $G$ with coefficients in~$\kk$.
\end{Exa}
\begin{Exa}[Spectra] \label{Exa:supernaive-spectra}
\index{spectra as Mackey functor}%
Let $\mathrm{Sp}$ be any of the nice model categories of spectra where the weak equivalences $W$ are the (stable) weak homotopy equivalences. Then Example~\ref{Exa:model-cats} specializes to yield a Mackey 2-functor~$\MM$. Its value $\MM(G)$ at a group~$G$ is the homotopy category of `$G$-shaped diagrams' of spectra. Beware: these are neither the genuine nor the naive $G$-spectra of topologists; see \Cref{Exa:SH(G)}.
\end{Exa}

\begin{Rem}
As already mentioned, there exists an abundance of examples of model categories, besides the above three staples. They all give rise to Mackey 2-functors by \Cref{Thm:ambidex-der}. We shall refrain from making the detailed list of obvious variations on this theme. The eager reader may want to consult, for instance, the families of stable model categories listed in \cite[Examples~2.3.(i)-(vii)]{SchwedeShipley03}. Among them, we highlight the category of modules over a symmetric ring spectrum, which alone is already a huge source of examples. Motivic theory in algebraic geometry (not to be confused with our Mackey 2-motives) is also a major purveyor of model categories, that can be fed into our machine. Schwede-Shipley provide a second list of examples in~\cite[Examples~2.4.(i)-(vi)]{SchwedeShipley03}, of more algebraic nature, in the sense that the underlying category of the model structure is abelian. Furthermore, numerous examples of exact but not necessarily stable models can be found in~\cite{Stovicek13}, such as those built out of categories of quasi-coherent modules on schemes or diagrams of rings.
\end{Rem}

Finally, let us mention a basic non-example.
\begin{Exa}
The category $\Set$ of sets gives rise to the representable derivator $J\mapsto \Set^J$, whose restriction $\MM\colon G\mapsto \Set^G$ to finite groupoids does not define a Mackey 2-functor for the trivial reason that $\Set^G$ is not additive. We may be tempted to drop the additivity requirement in order to save this example (which can be made pointed too) but there is a less trivial failure showing that this will not work: The left and right adjoints to a restriction functor $i^*$ are (almost) never isomorphic, even when working with pointed sets instead. This follows \eg from elementary size considerations on finite sets: Already for a subgroup~$H\le G$ and an $H$-set~$X$, the natural map of $G$-sets $G\times_H X\too \Map_H(G,X)$  is not a bijection in general. The left adjoint $i_!=G\times_H-$ is often referred to as \emph{ordinary} or \emph{additive} induction, the right one $i_*=\Map_H(G,-)$ as \emph{tensor} or \emph{multiplicative} induction.
\end{Exa}
%

\bigbreak
\section{Mackey sub-2-functors and quotients}
 \label{sec:sub-quotients}%
\medskip

%
\begin{Def}
\label{Def:Mackey-sub-2-functor}%
\index{Mackey sub-2-functor}%
Let $\MM\colon \GG^{\op}\too\ADD$ be a Mackey 2-functor.
A \emph{Mackey sub-2-functor} $\NN\subseteq\MM$ consists of a collection of (full and replete) additive subcategories $\NN(G)\subseteq \MM(G)$ for $G\in \GG$ closed under restriction and (co)induction:
\[ u^*(\NN(G))\subseteq \NN(H)
\quad \textrm{ and } \quad
i_*( \NN(H)) \subseteq \NN(G)
\]
for all functors $u\colon H\to G$ and all faithful functors $i\colon H\into G$ in~$\GG$.
\end{Def}

More generally:

\begin{Def}
\label{Def:morphism-of-Mackey-2-functors}%
\index{pre-morphism of Mackey 2-functors}%
\index{morphism of Mackey 2-functors}%
A \emph{pre-morphism} of Mackey 2-functors $F\colon \MM\to \cat N$ is simply a (strong pseudo-natural) transformation, as in \Cref{Ter:Hom_bicats}.
A \emph{morphism} $F\colon \MM\to \cat N$ is a pre-morphism which interacts nicely with the two adjunctions $i_!\dashv i^* \dashv i_*$, in the sense that
the mates $(F_i)_*\colon F i_*\Rightarrow i_* F$ and $ (F_i\inv)_!\colon i_!F\Rightarrow F i_!$ of $F_i\colon i^*F \isoEcell F i^*$ are isomorphisms for all~$i\in \JJ$. Actually it suffices that all the $(F_i)_*$ are invertible, or equivalently that all the $(F_i\inv)_!$ are invertible, because of \Cref{Prop:Theta-nat} and because $\Theta$ is invertible for Mackey 2-functors.
\end{Def}

We will say more on morphisms in \Cref{sec:bicat_2Mack}.

\begin{Exa}
\label{Exa:Ker-sub-2-functor}%
Let $F\colon \MM\to \MM'$ be a morphism of Mackey 2-functors (\Cref{Def:morphism-of-Mackey-2-functors}). Then we have a Mackey sub-2-functor $\Ker F \subseteq\MM$ defined at each~$G$ to be the full subcategory~$(\Ker F)(G)=\SET{X\in \MM(G)}{F_G(X)\cong 0\textrm{ in }\MM'(G)}$.
\end{Exa}

Amusing examples occur when $\MM$ takes values in abelian categories, like for instance the one of \Cref{Exa:lin_reps}.
\begin{Prop}
\label{Prop:abelian}%
Let $\MM\colon\GG^\op\to \ADD$ be a (rectified) Mackey 2-functor such that every $\MM(G)$ is an abelian category. Then for every faithful $i\colon H\into G$ the functors $i^*$ and $i_!=i_*$ are exact and they preserve injective and projective objects. Consequently, the assignments
\[
G\mapsto \Inj(\MM(G))
\qquadtext{and}
G\mapsto \Proj(\MM(G))
\]
define Mackey 2-functors on the sub-2-category $\GG^\faith$ of~$\GG$ with only faithful 1-cells; these are Mackey sub-2-functors of~$\MM$ restricted to~$\GG^\faith$.
\end{Prop}

\begin{proof}
Since $i^*$ is a left and a right adjoint it is both right and left exact. Similarly for $i_!=i_*$. Then the last statements follow because left (resp.\ right) adjoints of exact functors preserve projectives (resp.\ injectives).
\end{proof}

Another way to produce examples is by taking additive quotients.
\begin{Prop} \label{Prop:Mackeycat_quot}
\index{additive quotient of Mackey 2-functors}%
Let $\MM\colon \GG^{\op}\too\ADD$ be any Mackey 2-functor and $\NN\subseteq\MM$ a Mackey sub-2-functor. Then the additive quotient categories $\MM (G)/\NN(G)$ inherit a canonical structure of a Mackey 2-functor $\MM/\NN\colon \GG^{\op}\to \ADD$ such that the quotient $\MM\to \MM/\NN$ is a (strict) morphism of Mackey 2-functors.
\end{Prop}
\begin{proof}
Recall that the additive quotient $\MM(G)/\NN(G)$ is, by definition, the category with the same objects as $\MM(G)$ and where morphisms are taken modulo the additive ideal of maps that factor through some object of $\NN(G)$. Thus its Hom groups are given by the following quotient of abelian groups, for all objects $X,Y$:
\[
\big(\MM(G)/\NN(G)\big) (X,Y) := \MM(G)(X,Y) \big/
\{
 f\colon X\to Y \mid \exists Z\in \NN(G) \textrm{ and }
 \vcenter { \hbox{ \xymatrix@C=3pt@R=7pt{X \ar[rr]^-f \ar[dr] && Y \\ & Z \ar[ur] &} }}
\}
\]
There is an evident quotient functor $\MM(G)\to \MM(G)/\NN(G)$ which is the identity on objects and the projection on Hom groups. As a 2-functor on $\GG^{\op}$
\[
\vcenter { \hbox{
\xymatrix{
G \\
H \ar@/^4ex/[u]^-{u} \ar@/_4ex/[u]_-{v} \ar@{}[u]|{\oEcell{\alpha}}
}
}}
\quad\mapsto \quad
\vcenter { \hbox{
\xymatrix{
\MM(G)/\NN(G)
 \ar@/_4ex/[d]_-{u^*}
 \ar@/^4ex/[d]^-{v^*}
 \ar@{}[d]|{\oEcell{\alpha^*}}
 && \MM(G)
 \ar@/_4ex/[d]_-{u^*}
 \ar@/^4ex/[d]^-{v^*}
 \ar@{}[d]|{\oEcell{\alpha^*}}
 \ar[ll] \\
\MM(H)/\NN(H) &&
 \NN(H)
 \ar[ll]
}
}}
\]
$\MM/\NN$ consists simply of the 1-cells and 2-cells induced between the quotient categories, which exist by the hypothesis that $u^*(\NN(G))\subseteq \NN(H)$ for all~$u$. Similarly, we obtain induction functors $i_*$ because $i_*( \NN(H)) \subseteq \NN(G)$. We also verify immediately that the units and counits of the adjunctions $i_*\dashv i^* \dashv i_*$, as well as all the existing relations between 2-cells, also descend to the quotients, so that indeed $\MM/\NN \colon G\mapsto \MM(G)/\NN(G)$ inherits the structure of a Mackey 2-functor.
\end{proof}

We can now consider the case of an important example which is \emph{not} globally defined in the sense that not all morphisms of groupoids can be allowed.
\begin{Exa}[Stable module categories]
\label{Exa:stable-module-cat}%
\index{stable module category}%
Fix a base field~$\kk$, and consider the Mackey 2-functor $\MM\colon G\mapsto (\Mod \kk)^G$ of $\kk$-linear representations as in Example~\ref{Exa:lin_reps}. We define the \emph{stable module category} of a groupoid $G$ to be the additive quotient
\[ \Stab(\kk G) := (\Mod \kk)^G / \NN(G)\]
where $\NN(G):=\Proj ((\Mod \kk)^G)$ is the full subcategory of projective objects.
We would like to show that $\MM$ induces on $G\mapsto \Stab(G)$ the structure of a Mackey 2-functor by applying Proposition~\ref{Prop:Mackeycat_quot}. However, in order for the subcategories~$\NN(G)$ to form a Mackey sub-2-functor we cannot allow all functors $u\colon H\to G$ between groupoids. For instance, if $u\colon G\to 1$ is the projection from a non-trivial group to the trivial one, then $u^*\colon \Mod \kk \to \Mod (\kk G)$ sends $\kk$ to the trivial $G$-representation $\kk^{\textrm{triv}}$. So if $u^*$ descends to an additive functor $u^*\colon 0\cong \Stab(\kk 1) \to \Stab(\kk G)$ we conclude that $\kk^{\textrm{triv}}$ is projective, but this is only possible in the semi-simple case, \ie if $\Stab(\kk G)\cong \{0\}$. Still, if $i\colon H\into G$ is faithful then $i^*(\NN(G))\subseteq \NN(H)$ and $i_*(\NN(H))\subseteq \NN(G)$ by \Cref{Prop:abelian}. Hence we may apply \Cref{Prop:Mackeycat_quot} and conclude that $G\mapsto \Stab(G)$ is a Mackey 2-functor defined on~$\GG=\groupoidf$, the 2-full 2-subcategory of finite groupoids and faithful functors, rather than the whole~$\groupoid$.
\end{Exa}

\begin{Rem}
In the previous example, projective and injective objects coincide in the abelian category $(\Mod \kk)^G$, hence the stable module category is triangulated (see \cite{Happel88} \cite{Heller60}). However, $G\mapsto \Stab(G)$ cannot come from restricting a (stable or otherwise) derivator, since such a derivator would have trivial base category $\Stab(\kk 1)\cong\{0\}$ and thus it would have to be the zero derivator $J\mapsto \{0\}$ by~(Der\,\ref{Der-2}). Therefore we could not have used \Cref{Thm:ambidex-der} directly.
\end{Rem}

\bigbreak
\section{Extending examples from groups to groupoids}
 \label{sec:more-examples}%
\medskip

Following up on \Cref{Rem:Mackey-def}\,\eqref{it:groups}, let us see how one can reduce a Mackey 2-functor $\MM\colon \groupoid^\op\to \ADD$ to the values $\MM(G)$ for $G$ actual groups, \ie one-object groupoids (\Cref{Rem:group(oid)}). For simplicity, we only treat $\GG=\groupoid$ but the reader can adapt to general 2-categories~$\GG$ as in \Cref{Hyp:GG}.
\begin{Not}
\label{Not:Fun}%
\index{$2fun$@$\twoFun$ \, 2-category of 2-functors} \index{$2category$@2-category!-- of 2-functors}%
\index{fun@$\twoFun$}%
\index{$2fun$@$\twoFun_\amalg$ \, additive 2-functors} \index{$2category$@2-category!-- of additive 2-functors}%
\index{Fun@$\twoFun_\amalg$}%
\index{$group$@$\group$ \, 2-category of groups} \index{group@$\group$}%
Let $\cat{B}$ be a full sub-2-category of~$\groupoid$, for instance $\cat{B}=\groupoid$ itself, or $\cat{B}=\groconn\subset \groupoid$ the sub-2-category of connected groupoids, or $\cat{B}=\group\subset \groconn$ the full sub-2-category of one-objects groupoids (\Cref{Rem:group(oid)}). (In this section `full' means `1-full and 2-full'.) We emphasize that here $\group$ does not stand for the ordinary 1-category of finite groups but for its enhancement obtained by adding 2-cells $\gamma_g\colon f_1\Rightarrow f_2$ between homomorphisms $f_1,f_2\colon H\to G$ for every $g\in G$ such that ${}^{g\!}f_1=f_2$. We write
\[
\twoFun(\cat{B}^\op,\ADD)
\]
for the 2-category of 2-functors from $\cat{B}^\op$ to~$\ADD$, with pseudo-natural transformations as 1-cells and modifications as 2-cells (\Cref{Ter:Hom_bicats}). When $\cat{B}$ is closed under coproducts in~$\gpd$, we write
\[
\twoFun_{\amalg}(\cat{B}^\op,\ADD)\;\subseteq\; \twoFun(\cat{B}^\op,\ADD)
\]
for the full 2-category of \emph{additive} 2-functors in the sense of \Mack{1}, \ie sending the coproducts of $\gpd$ to products in $\ADD$.
\end{Not}

\begin{Lem}
\label{Lem:pre-Mackey-on-groups}%
With \Cref{Not:Fun}, the restrictions
\[
\twoFun_{\amalg}(\groupoid^\op,\ADD)\too \twoFun (\groconn^\op,\ADD)\too \twoFun (\group^\op,\ADD)
\]
are biequivalences  (see \Cref{Ter:Hom_bicats}).
\end{Lem}

\begin{proof}
Additivity gives the first biequivalence and the biequivalence $\group\subset \groconn$ gives the second.
\end{proof}

\begin{Rem} \label{Rem:not_groups}
Consider a 2-functor $\MM\colon \group^\op\to \ADD$. Explicitly, we have:
\begin{enumerate}[\rm(1)]
\item
\label{it:M-gps-1}%
For every finite group~$G$, an additive category~$\MM(G)\in\ADD$.
\smallbreak
\item
\label{it:M-gps-2}%
For every homomorphism~$f\colon H\to G$, a functor~$f^*\colon\MM(G)\to \MM(H)$, with the obvious (strict) functoriality $\Id_G^*=\Id_{\MM(G)}$ and $(f_1f_2)^*=f_2^*f_1^*$.
\smallbreak
\item
\label{it:M-gps-3}%
For every two homomorphisms $f_1,f_2\colon H\to G$ and every $g\in G$ such that ${}^g f_1=f_2$, an (invertible) natural transformation $\gamma_g^*\colon f_1^*\isoEcell f_2^*\colon\MM(G)\to \MM(H)$, with the obvious functoriality $\gamma_{g_1g_2}^*=\gamma_{g_1}^*\gamma_{g_2}^*$.
\end{enumerate}

As particular cases of~\eqref{it:M-gps-2} we have the functors associated to inclusions $H\le G$:
\begin{equation}
\label{eq:M-gps-4}%
\Res^G_H := (\incl_H^G)^*\colon\MM(G)\to \MM(H)\,.
\end{equation}
As another case, for every subgroup~$L\le G$ and every element $g\in G$ the homomorphism $c_g\colon L\isoto {}^{g\!}L$, $x\mapsto g x g\inv$ induces a functor
\[ c_g^*\colon \MM({}^{g\!}L)\isoto \MM(L)\,.
\]
(This should not be confused with the 2-cell~$\gamma_g^*$ of~\eqref{it:M-gps-3}.) Note that $c_{g_1g_2}=c_{g_1}c_{g_2}$, hence $c_{g_1g_2}^*=c_{g_2}^*c_{g_1}^*$. Composing with the above, we get the `twisted restriction'
\begin{equation}
\label{eq:M-gps-5}%
{}^g\Res^H_L := c_g^*\circ\Res^H_{{}^{g\!}L}\colon\MM(H)\to \MM({}^{g\!}L)\isoto \MM(L)
\end{equation}
for all subgroups~$H,L\le G$ and element $g\in G$ such that ${}^{g\!}L\le H$.

As an instance of~\eqref{it:M-gps-3}, for $f_1=\Id_G$ and $f_2=c_g\colon G\isoto G$, we have an isomorphism of functors
\begin{equation}
\label{it:M-gps-6}%
\gamma_g^*\colon \Id_{\MM(G)}\isoEcell c_g^*\colon\MM(G)\to \MM(G)\,.
\end{equation}
\end{Rem}

We can now explicitly translate along the biequivalence of \Cref{Lem:pre-Mackey-on-groups} what it means  for some additive $\MM\colon \groupoid^{\op}\to \ADD$ to be a Mackey 2-functor,  in terms of its restriction to groups.

\begin{Prop}
\label{Prop:test-on-gps}%
Let $\MM\colon \groupoid^\op\to \ADD$ be a 2-functor satisfying \Mack{1} and consider its restriction $\MM\colon \group^\op\to \ADD$ to the 2-category of groups with conjugations as 2-cells. Retaining \Cref{Not:Fun} and \Cref{Rem:not_groups}, we have:
\begin{enumerate}[\rm(a)]
\item
\label{it:test-on-gps-2}%
The 2-functor~$\MM$ satisfies \Mack{2} if and only if for all subgroup $H\le G$ the restriction $\Res^G_H\colon \MM(G)\to \MM(H)$ has both a left and a right adjoint, say $\Ind_H^G$ and $\CoInd_H^G$ respectively.
\end{enumerate}
For the remaining statements, we assume \Mack{2} for~$\MM$.
\begin{enumerate}[\rm(a)]
\setcounter{enumi}{1}
\item
\label{it:test-on-gps-3}%
The 2-functor~$\MM$ satisfies the left BC-formula \Mack{3} if and only if for every subgroup~$H\le G$, every homomorphism $f\colon K\to G$ and every choice of representatives $g\in [g]\in f(K)\backslash G/H$ the following natural transformation is an isomorphism
\[
\xymatrix{
{\bigoplus_{[g]\in f(K)\backslash G/H}} \Ind_{f\inv({}^{g\!}H)}^K\circ (f^g)^* \ar[r]^-{(\beta_g)_g}
& f^* \circ \Ind_H^G
}
\]
where $f^g\colon f\inv({}^{g\!}H)\to H$ is $c_g^{-1}\circ f$ and $\beta_g\colon \Ind_{f\inv({}^{g\!}H)}^K\circ (f^g)^* \to f^* \circ \Ind_H^G$ is the composite
\[
\xymatrix{
\Ind_{f\inv({}^{g\!}H)}^K(f^g)^* \ar[r]^-{\beta_g} \ar[d]^-{\eta}
& f^* \Ind_H^G
\\
\Ind_{f\inv({}^{g\!}H)}^K(f^g)^*\Res^G_H\Ind_H^G \ar[r]^-{\gamma_g^*}
& \Ind_{f\inv({}^{g\!}H)}^K\Res^{K}_{f\inv({}^{g\!}H)}f^* \Ind_H^G \ar[u]^-{\eps}
}
\]
where $\gamma_g^*\colon (f^g)^*\Res^G_H\stackrel{\sim}{\Rightarrow} \Res^{K}_{f\inv(H^{g})}
f^*\colon \MM(G) \to \MM(f\inv({}^{g\!}H))$ is the natural transformation associated to conjugation~$\gamma_g\colon f^g\stackrel{\sim}{\Rightarrow} f$ between the two 1-cells~$f^g,f\colon f\inv({}^{g\!}H)\to G$ in~$\group$ (\aka homomorphisms) obtained as the composites $f\inv({}^{g\!}H)\oto{f^g}\;H\into G$ and $f\inv({}^{g\!}H)\into K\oto{f}G$. Note that ${}^{g\!}(f^g)=f$.

Dually, the right BC-formula reduces to a dual formula for~$f^*\circ\CoInd_H^G$.
\smallbreak
\item
\label{it:test-on-gps-4}%
The 2-functor~$\MM$ satisfies \Mack{4} if and only if for every subgroup~$H\le G$ there exists some isomorphism $\Ind_H^G\simeq\CoInd_H^G$.
\end{enumerate}
\end{Prop}

\begin{proof}
This is a lengthy exercise. By additivity, we reduce all problems to the case of connected finite groupoids. In that case, everything like existence of adjoints or ambidexterity reduces up to equivalence to the one-object situation. This pattern can be followed to prove~\eqref{it:test-on-gps-2} and~\eqref{it:test-on-gps-4}. Let us say a few more words about~\eqref{it:test-on-gps-3}. Again, we can reduce everything to one-object groupoids but in that case we also need to understand the resulting iso-comma. Similarly to what happened in~\Cref{Rem:old-Mackey}, if $i\colon H\into G$ is the inclusion, there is an equivalence
\[
\coprod_{[g]\in f(K)\backslash G/H} f\inv({}^{g\!}H)\overset{\sim}{\too} (i/f)
\]
mapping the unique object $\bullet$ of~$f\inv({}^{g\!}H)$ to the object~$(\bullet,\bullet,g)$ of~$(i/f)$, and a morphism~$k\in f\inv({}^{g\!}H)$ to~$(f(k)^g,k)$. In more 2-categorical terms, this equivalence is~$w:=\langle \coprod_{g}f^g\,,\,\coprod_{g}\incl_{f\inv({}^{g\!}H)}^K\,,\,\coprod_{g}\gamma_g\rangle$
\[
\xymatrix{
& {\coprod\limits_{[g]\in f(K)\backslash G/H}} f\inv({}^{g\!}H) \ar[ld]_-{\coprod_g f^g} \ar@{ >->}[rd]^-{\coprod_g \incl} \ar@{}[dd]|-{\oEcell{\coprod_g\gamma_g}}
\\
H \ar@{ >->}[rd]_-{i}
&& K \ar[ld]^-{f}
\\
& G}
\]
with notations as in \Cref{Def:comma} and the statement. It remains to explicitly trace the construction of the mate~$\gamma_!$ which is exactly the one announced in~\eqref{it:test-on-gps-3}.
\end{proof}

\begin{Exa}[Genuine equivariant spectra] \label{Exa:SH(G)}
Let $\MM(G)=\SH(G)$ be the stable homotopy category of genuine $G$-equivariant spectra, in the sense of topology, for every finite group~$G$. By the well-known functoriality properties of equivariant stable homotopy, we obtain a Mackey 2-functor on $\GG=\groupoid$. This can be verified via \Cref{Prop:test-on-gps}, by inspecting one of the available constructions of a model for $\SH(G)$, such as can be found in \cite{LewisMaySteinbergerMcClure86}, \cite{MandellMay02} or~\cite{Schwede16pp}. We leave the details to the interested topologists, limiting ourselves to the observation that the axiom \Mack{4} holds by the so-called \emph{Wirthm\"uller isomorphism} (see \eg \cite[Cor.\,5.25]{Schwede16pp}), which also exists, with an extra twist, when $G$ and $H$ are compact Lie groups.
\end{Exa}

\begin{Exa}[Equivariant Kasparov theory] \label{Exa:KKetc}
Let $\MM(G)=\KK(G)$ be the $G$-equivariant Kasparov category, in the sense of operator algebraists, for every finite group~$G$. By the well-known properties of equivariant Kasparov theory, we obtain a Mackey 2-functor on $\GG=\groupoid$. Again, this can be verified via \Cref{Prop:test-on-gps} by inspecting the available explicit constructions of $\KK(G)$ and its functoriality at the level of algebras, see \eg~\cite{MeyerNest06} or~\cite{Meyer08}.
We leave the details to the interested operator algebraists, and simply note that in this example there is a concrete construction of an `induction' functor $\KK(H)\to \KK(G)$ for closed subgroups of quite general topological groups~$G$, which is known to be \emph{right} adjoint to restriction when $G/H$ is compact and \emph{left} adjoint to it when $G/H$ is discrete -- whence \Mack{4} when the groups are finite.
\end{Exa}

\begin{Rem}
Unlike our ambidexterity results, \eg \Cref{Thm:ambidex-der}, the above \Cref{Prop:test-on-gps} does not really provide any deep reason \emph{why} some data $G\mapsto \MM(G)$ is a Mackey 2-functor. All it does is reduce the verification to the more ordinary data of $\MM(G)$ for actual finite groups~$G$, as the examples are typically presented in the literature in such terms. This is mildly unsatisfying but, at least in cases such as \Cref{Exa:SH(G)} and \Cref{Exa:KKetc}, perhaps unavoidable at this point in time because of the lack of something like a theory of `equivariant derivators', and of results on how to produce them from model-level constructions. Indeed, unlike all the examples of Mackey 2-functors in \Cref{sec:add-der-Mackey}, in the above two cases the variance $G\mapsto \MM(G)$ \emph{cannot} be obtained by a simple-minded restriction of an additive derivator to finite groupoids. This is because of the failure of (Der~\ref{Der-2}) in the variable~$G$: A map $f\colon X\to X'$ between `$G$-objects' $X,X'\in\MM(G)$ is not an isomorphism merely because the underlying (non-equivariant) morphism in the base category is an isomorphism.
\end{Rem}

\begin{Rem}
Most of the examples we have encountered so far carry an additional structure on the additive categories of values~$\MM(G)$, typically that of a triangulated category, of a symmetric monoidal category, or both. We will elaborate on this interesting topic in a sequel.
\end{Rem}

\begin{Rem} \label{Rem:eq_sheaves}
One can use the techniques of this chapter to discuss further examples, for instance equivariant objects in suitable categories with $G$-actions, like equivariant sheaves over a $G$-scheme. Note however that such examples of Mackey 2-functors are typically defined on more sophisticated 2-categories of groupoids~$\GG$, not merely sub-2-categories of~$\groupoid$ as we considered so far (based on \Cref{Hyp:GG}). Those more general 2-categories~$\GG$ will be considered in \Cref{ch:bicat-spans,ch:2-motives}, starting with \Cref{Hyp:G_and_I_for_Span}. For instance, for the Mackey 2-functor of equivariant sheaves over a given `base' $G$-scheme~$X$, one can take as $\GG$ the 2-category of groupoids with a \emph{chosen} embedding into the fixed `ambient' group~$G$ which acts on~$X$. In other words, the natural 2-category of definition~$\GG$ appears to be a comma 2-category in~$\groupoid$ (see \Cref{Def:gpdG}).
\end{Rem}

\bigbreak
\section{Mackey 2-functors of equivariant objects}
\label{sec:equivobj}%
\medskip

Following up on \Cref{Rem:eq_sheaves}, the goal of this section is to prove \Cref{Thm:Mackey2fun-equi-objects} and derive examples of Mackey 2-functors for a fixed group~$G{}$ associated with various categories of equivariant sheaves.

\begin{Not} \label{Not:pseudo-functor}
Let $\cat G$ be a small category and let $\cat S \colon \cat G^\op \to \CAT$ be a pseudo-functor to the 2-category of all categories. This amounts to giving the following data, satisfying the coherence conditions of~\Cref{Ter:pseudofun}:
\begin{enumerate} [\rm(a)]
\item categories $\cat S_p$ for every object $p\in \cat G$,
\item `pullback' functors $g^*\colon \cat S_{p'}\to \cat S_p$ for every morphism $g\colon p\to p'$ of $\cat G$, and
\item natural isomorphisms $\varpi_p\colon \Id_{\cat S_p}\overset{\sim}{\to}(\id_p)^*$ for every object $p\in \cat G$ as well as $\varpi_{g_1,g_2}\colon g_1^*g_2^*\overset{\sim}{\to}(g_2g_1)^*$ for every pair of composable morphisms $g_1,g_2$ in~$\cat G$.
\end{enumerate}
\end{Not}

\begin{Exa} \label{Exa:single-cat}
Our main example occurs when $\cat G=G$ is a group, with one object~$\bullet$. The data of a pseudo-functor~$\cat{S}\colon \cat G^\op\to \CAT$ as above then reduces to a category~$\cat{S}_{\sbull}$ together with a (right, pseudo-) `action' of~$G$.
\end{Exa}

\begin{Rem}
Grothendieck associates to a pseudo-functor $\cat S$ as in \Cref{Not:pseudo-functor} a category $\int \cat S$ fibered over~$\cat G$. The objects of $\int \cat S$ are pairs $(p,s)$ where $p\in \Obj \cat G$ and $s\in \Obj \cat S_p$, and a morphism $(p,s)\to (p',s')$ is a pair $(g,\sigma)$ of morphisms $g\colon p\to p'$ in $\cat G$ and $\sigma \colon s\to g^*(s')$ in~$\cat S_p$. Its composition is simply given by $(g',\sigma')\circ (g,\sigma)= (g'g, \varpi_{g,g'} g^*(\sigma')\sigma)$. This category $\int \cat S$ comes equipped with the obvious functor $\pi\colon \int \cat S\to \cat G$ projecting on the first component, which is a \emph{Grothendieck fibration}.
This construction defines a biequivalence $\int \colon \PsFun(\cat G^\op,\CAT)\xrightarrow{\sim} \mathsf{Fib}(\cat G)$ between the 2-category of pseudo-functors $\cat G^\op\to \CAT$ and Grothendieck fibrations over~$\cat G$; see \eg~\cite{Vistoli05}. In the other direction, to any Grothendieck fibration $\pi\colon \overline {\cat S}\to \cat G$, one can associate a pseudo-functor $\cat S\colon \cat G^\op\to \CAT$ where $\cat S_p = \pi^{-1}(p)$ is the fiber over~$p$, \ie the (typically \emph{not} full) subcategory of $\overline {\cat S}$ of all objects mapping to $p$ and morphisms mapping to $\id_p$.
The hypothesis that $\pi$ is a fibration allows one (after choosing a `cleavage') to construct the functors~$g^*$ and the coherence isomorphisms~$\varpi$.
\end{Rem}

\begin{Def}
\label{Def:Gamma}
Let $\cat S\colon \cat G^\op\to \CAT$ be a pseudo-functor as above. We define its associated \emph{category of global sections}, or \emph{category of $\cat G$-equivariant objects}
\[
\Gamma(\cat G,\cat S)
\]
to be the functor category of sections $\cat G\to \int \cat S$ of the corresponding Grothendieck fibration $\pi\colon \int S\to \cat G$.
Unfolding the definition, objects of $\Gamma(\cat G,\cat S)$ are pairs $(X,\varphi)$ where $X=\{X_p\}_{p\in \Obj \cat G}$ is a family of objects $X_p\in \cat S_p$ and $\varphi=\{\varphi_g\}_{g\in \Mor\cat G}$ a family of morphisms $\varphi_g\colon X_p\to g^*(X_{p'})$ for every arrow $g\colon p\to p'$ of~$\cat G$.
This data must be such that
\begin{equation} \label{eq:equiv-obj}
\varphi_{\id_{p}} =\varpi_{p}
\quad \textrm{ and } \quad
\vcenter{
\xymatrix@R=12pt{
X_{p_1}
 \ar[r]^-{\varphi_{g_1}}
  \ar@/_1ex/[drr]_{\varphi_{g_2g_1}} & g^*_1 X_{p_2} \ar[r]^-{g^*_1(\varphi_{g_2})} & g^*_1g^*_2 X_{p_3} \ar[d]^{\varpi_{g_1,g_2}}_\simeq \\
&& (g_2g_1)^*X_{p_3}
}}
\end{equation}
for all $p\in \Obj\cat G$ and all composable $g_1,g_2\in \Mor\cat G$.
A morphism $\xi\colon (X,\varphi)\to (X',\varphi')$ in $\Gamma(\cat G,\cat S)$ is a family $\{\xi_p\}_{p\in \Obj \cat G}$ of morphisms $\xi_p\colon X_p\to X_p'$ which is equivariant, \ie such that the square
\begin{equation} \label{eq:equiv-map}
\vcenter{\xymatrix{
X_{p_1} \ar[r]^-{\xi_{p_1}} \ar[d]_{\varphi_g} & {X_{p_1}' } \ar[d]^{\varphi'_g} \\
g^* X_{p_2} \ar[r]^-{g^* (\xi_{p_2})} & g^* {X_{p_2}'}
}}
\end{equation}
is commutative (in $\cat S_{p_1}$) for all $g\colon p_1\to p_2$ in~$\cat G$.

\end{Def}

\begin{Not}
In order to unburden our notations, in the following we will suppress the structure isomorphisms~$\varpi$, \ie we will pretend that any $\cat S$ under consideration is a \emph{(strict) 2-functor}. This is justified by the fact that every pseudo-functor as in \Cref{Not:pseudo-functor} can be strictified, see~\cite[\S\,4.2]{Power89} or \Cref{Rem:coh_pseudofun}.
\end{Not}

\begin{Rem} \label{Rem:groupoid-case}
We will soon restrict attention to the case when $\cat G = G{}$ is a finite groupoid, for instance a finite group viewed as a category with a unique object.
In this case the pullback functors $g^*$ are all equivalences and the structure maps $\varphi_g$ of an equivariant object $(X,\varphi)$ are all invertible by~\eqref{eq:equiv-obj}.
\end{Rem}

We now let the category~$\cat G$ vary.

\begin{Cons} \label{Cons:restriction}
Let $f\colon \cat H\to \cat G$ be a functor between small categories.
Given a pseudo-functor $\cat S \colon \cat G^\op\to \CAT$, we can consider its precomposition
\[
\cat S\circ f^\op\colon \cat H^\op\to \CAT
\]
mapping each $q\in \cat H$ to $\cat S_{f(q)}$ and each $h\in \Mor \cat H$ to $h^*:= f(h)^*$.
We thus have two categories of sections $\Gamma(\cat G,\cat S)$ and $\Gamma(\cat H,\cat S):=\Gamma(\cat H,\cat S\circ f^\op)$.
Precomposition with $f$ also defines an obvious \emph{restriction functor} $f^*\colon \Gamma(\cat G,\cat S)\to \Gamma(\cat H,\cat S)$, sending a $\cat G$-equivariant object $(X,\varphi)$ to the $\cat H$-equivariant object $f^*(X,\varphi)=(f^*X,f^*\varphi)$ with $(f^*X)_q:= X_{f(q)}$ and $(f^*\varphi)_h := \varphi_{f(h)}$ (for all $q\in \Obj \cat H$ and $h\in \Mor\cat H$), and similarly on equivariant morphisms: $(f^*\xi)_q := \xi_{f(q)}$.
\end{Cons}

\begin{Rem} \label{Rem:Gamma-refined}
Note that if $\cat S$ takes values in additive categories and additive functors, then $\Gamma(\cat G, \cat S)$ is also an additive category, in the evident `componentwise' way, and restriction along any $f\colon \cat H\to \cat G$ is an additive functor.
\end{Rem}

\begin{Lem} \label{Lem:adjoints}
Assume that $\cat G$ and $\cat H$ are finite groupoids and that the categories $\cat S_p$ admit all finite limits (resp.\ finite colimits). Then the restriction functor $f^*$ of \Cref{Cons:restriction}
admits a right adjoint $f_*$ (resp.\ a left adjont~$f_!$)
\[
\xymatrix{
\Gamma(\cat G,\cat S) \ar[d]|{f^*} \\
\Gamma(\cat H,\cat S) \ar@/^3ex/[u]^{f_!} \ar@/_3ex/[u]_{f_*}
}
\]
which moreover is given by a generalized Kan-extension formula.
\end{Lem}

\begin{proof}
Let us begin by constructing the right adjoint~$f_*$. For each $p\in \Obj \cat G$, let $p\backslash f$ denote the slice category of $f$ under~$p$, whose objects are pairs $(x,\gamma)$ with $x\in \Obj \cat H$ and $\gamma \colon p\to f(x)$ in~$ \cat G$ and a morphism $(x,\gamma)\to (x',\gamma')$ is a morphism $\delta\colon x\to x'$ such that $f(\delta)\gamma = \gamma'$.
Note that $p\backslash f$ is a finite category as $\cat G$ and $\cat H$ are assumed finite.
We must define the $\cat G$-equivariant object $f_*(Y,\psi)= (f_* Y, f_* \psi)$ for every $\cat H$-equivariant object $(Y,\psi)$.
Given the latter and a $p\in \Obj \cat G$, define a functor $\tilde Y_p \colon p\backslash f \too \cat S_p$ by sending $(x,\gamma)$ to $\tilde Y_p(x,\gamma):=\gamma^*Y_x\in \cat S_p$ and $\delta\colon (x,\gamma)\to (x',\gamma')$ to the map
$
\tilde Y_p(\delta) := \gamma^* (\psi_\delta) \colon \gamma^* Y_x \to \gamma^*f(\delta)^* Y_{x'} = {\gamma'}^* Y_{x'}
$.
Define now
\begin{equation}
\label{eq:*obj}%
\vcenter{\xymatrix@C=10pt{
{ (f_*Y)_p := \lim \tilde Y_p } \ar@{=}[r] & {\displaystyle{\lim_{(x,\gamma)\in (p\backslash f)} \gamma^* Y_x} } \ar[rrd]^(.55){\pr_{(x',\gamma')}} \ar[d]_(.6){\pr_{(x,\gamma)}} && \\
& \gamma^* Y_x \ar[r]^-{\gamma^* \psi_\delta} &  \gamma^* (f(\delta)^* Y_{x'}) \ar@{=}[r] & (f(\delta) \gamma)^* Y_{x'}
}}
\end{equation}
as a limit in~$\cat S_p$ (which exists because $p\backslash f$ is finite and $\cat S_p$ finitely complete); by construction, it comes equipped with canonical projection maps $\pr_{(x,\gamma)}$ making the above triangles commute for every map $\delta\colon (x,\gamma)\to (x',\gamma')$ in~$p\backslash f$. We must still define the structure maps $f_*\psi =\{(f_*\psi)_g\}_g$.
Using the universal property of the limit, for each $g\colon p\to p'$ in~$\cat G$ we let $(f_*\psi)_g$ be the unique map in $\cat S_p$ making the following diagram commute:
\begin{equation*} 
\xymatrix@C=10pt{
& && g^*(f_* Y_{p'}) \ar@{=}[r] &
 {\displaystyle{ g^* \left( \lim_{(x',\gamma')\in p'\backslash f} {\gamma'}^* Y_{x'} \right) }}
  \ar[d]_{\cong}
   \ar@/^15ex/[dd]^{g^* \pr_{(x',\gamma')}} \\
(f_*Y)_p \ar@{=}[r]&
 {\displaystyle{\lim_{(x,\gamma) \in p\backslash f} \gamma^*Y_x}}
  \ar[drr]_{\pr_{(x',\gamma' g)}}
  \ar@{..>}[urr]^-{(f_* \psi)_g}
   \ar@{..>}[rrr]_-{\exists \textrm{ by lim}}
   && & {\displaystyle{ \lim_{(x',\gamma')\in p'\backslash f} g^*  {\gamma'}^* Y_{x'}  }}
    \ar[d]_(.6){\pr_{(x',\gamma')}} \\
& && (\gamma' g)^* Y_{x'} \ar@{=}[r] & g^*({\gamma'}^* Y_{x'})
}
\end{equation*}
Note that the functor $g^*$ commutes with the limit because it is an equivalence, and the morphisms $\pr_{(x',\gamma' g)}$ are compatible with the morphisms of $p'\backslash f$ thanks to~\eqref{eq:*obj} and~\eqref{eq:equiv-obj}, whence the map into the limit.
We leave to the reader the easy verifications that the above is a well-defined $G$-equivariant object $f_*(Y,\psi)$, and that $(Y,\psi)\mapsto f_*(Y,\psi)$ extends to a functor $f_*\colon \Gamma(\cat H,\cat S\circ f^\op)\to \Gamma(\cat G,\cat S)$ as claimed.

The unit $\reta\colon \Id_{\Gamma(\cat G,\cat S)} \Rightarrow f_*f^* $ of the adjunction at an object $(X,\varphi)$ is the morphism $\reta_{(X,\varphi)}$ in $\Gamma(\cat G,\cat S)$ with component at $p\in \cat G$ given by the universal property of the limit and the commutativity of the following triangles
\begin{equation*} 
\xymatrix{
X_p
 \ar@{..>}[rr]^-{\reta_{(X,\varphi), p}}
  \ar[drrr]_{\varphi_\gamma} &&
 (f_*f^* X)_p \ar@{=}[r] &
 { \displaystyle{ \lim_{(x,\gamma) \in p\backslash f} \gamma^* X_{f(x)}} }
  \ar[d]^-{\pr_{(x,\gamma)}} \\
&&& \gamma^* X_{f(x)}
}
\end{equation*}
for all $(x,  p\xrightarrow{\gamma} f(x)) \in \Obj( p\backslash f )$.
The naturality of $\reta$ in $(X,\varphi)$ follows by~\eqref{eq:equiv-map}.
The counit $\reps \colon f^*f_* \Rightarrow \Id_{\Gamma(\cat H,\cat S)}$ at an object $(Y,\psi)$ is the morphims $\reps_{(Y,\psi)}$ with components given by the projection maps
\begin{equation*} 
\reps_{(Y,\psi),q} \colon (f^*f_* (Y,\psi))_q = \lim_{(x,\gamma) \in f(q)\backslash f} \gamma^* Y_x \xrightarrow{ \pr_{(q,\id_{f(q)})}}  (\id_{f(q)})^* Y_q = Y_q
\end{equation*}
for every $q\in \cat H$.
Again by using~\eqref{eq:*obj} and~\eqref{eq:equiv-obj} one verifies that $\reps_{(X,\varphi)}$ is $\cat H$-equivariant and natural in~$(Y,\psi)$.
The unit-counit relations are similarly straightforward verifications from the definitions, which we leave to the reader.

The left adjunction $(f_! \dashv f^*, \leta, \leps)$ can be constructed dually, although some care should be taken as inverses appear; we therefore give the explicit formulas for ease of reference. For an object $p\in \cat G$, let $f/p$ denote the comma category of $f$ \emph{over}~$p$, with objects pairs $(x, \gamma\colon f(x) \to p)$ and morphisms $(x,\gamma)\to (x',\gamma')$ given by morphisms $\delta\colon x\to x'$ in $\cat H$ such that $\gamma' f(\delta) =  \gamma$. For every $(Y,\psi)\in \Gamma(\cat H,\cat S\circ f^\op)$, the $\cat G$-equivariant object $f_!(Y,\psi)=(f_!Y,f_!\psi)$ is defined as follows. For $p\in \Obj (\cat G)$, the object $(f_!Y)_p\in \cat S_{p}$ is given by the colimit (in~$\cat S_p$)
\begin{equation*} 
(f_! Y)_p := \colim \widehat Y_p = \colim_{(x,\gamma) \in f/p} (\gamma^{-1})^* (Y_x)
\end{equation*}
where now we use the functor $\widehat Y_p\colon f/p \to \cat S_p$ sending $(x,\gamma)$ to $(\gamma^{-1})^*(Y_x)$ and $\delta\colon (x,\gamma)\to (x',\gamma')$ to $(\gamma^{-1})^*(\psi_{\delta})\colon (\gamma^{-1})^*(Y_x) \to (\gamma^{-1})^*f(\delta)^*(Y_{x'}) = (\gamma'^{-1})^* (Y_{x'})$.
For $(g\colon p\to p')\in \Mor( \cat G)$, the map $(f_!\psi)_g$ is given by the universal property of the colimit and the following commutative diagram:
\begin{equation*} 
\xymatrix@C=10pt{
{\displaystyle{ (g^{-1})^*\left( \colim_{(x,\gamma)\in f/p} (\gamma^{-1})^* Y_x \right) }}
 \ar@{=}[r] &
(g^{-1})^*(f_!Y)_p
 \ar@{..>}[drr]^-{\,\,\,(g^{-1})^*(f_!\psi)_p} && & \\
{\displaystyle{ \colim_{(x,\gamma)\in f/p} (g^{-1})^*(\gamma^{-1})^* Y_x }}
 \ar@{..>}[rrr]_-{\exists \textrm{ by colim}}
  \ar[u]^{\simeq} & &&
 {\displaystyle{\colim_{(x,\gamma)\in f/p'} (\gamma^{-1})^*Y_{x} }}
   \ar@{=}[r] &
  (f_!Y)_{p'}\\
(g^{-1})^*(\gamma^{-1})^* Y_x
 \ar[u]^{\incl_{(x,\gamma)}}
  \ar@{=}[r]
     &
((g\gamma)^{-1})^* Y_x
    \ar[urr]_-{\;\; \incl_{(x, g \gamma)}}
     && &
}
\end{equation*}
The component at $(Y,\psi)$ of the unit $\leta\colon \Id_{\Gamma(\cat H,\cat S)}\Rightarrow f^*f_!$  is given for every $q\in \Obj(\cat H)$  by the canonical injection
\begin{equation*}
\leta_{(Y,\psi), q} \colon Y_q = (\id^{-1}_{f(q)})^* Y_q \xrightarrow{ \incl_{(q, \id_{f(q)})} } \colim_{(x,\gamma) \in f/f(q)} (\gamma^{-1})^*(Y_x) = (f^*f_!(Y,\psi))_q
\end{equation*}
and the component at $(X,\varphi)$ of the counit $\leps\colon f_!f^*\Rightarrow \Id_{\Gamma(\cat G,\cat S)}$ is induced at every $p\in \Obj(\cat G)$ by the colimit:
\begin{equation*}
\xymatrix{
{\displaystyle{ \colim_{(x,\gamma) \in f/p} (\gamma^{-1} )^* X_{f(x)} }} \ar@{=}[r] & (f_!f^* X)_p \ar@{..>}[rr]^-{\leps_{(X,\varphi),p}} &&
 X_p \\
(\gamma^{-1})^* X_{f(x)}
 \ar[u]^{\incl_{(x,\gamma)}}
  \ar[urrr]_{\quad\quad\quad (\varphi_{\gamma^{-1}})^{-1} \,=\, (\gamma^{-1})^*(\varphi_\gamma) } & &&
}
\end{equation*}
We leave the analogous verifications to the reader.
\end{proof}

\begin{Rem} \label{Rem:add-co-lims}
If in \Cref{Lem:adjoints} the categories $\cat S_p$ are additive and the functor $f$ is faithful, then the finite (co)compleness hypothesis is not necessary for the adjoints $f_!$ and $f_*$ to exist, because finite direct sums suffice in this case. Indeed, for each $p\in \cat G$ both comma categories $f/p$ and $p\backslash f$ are finite groupoids as so are $\cat G$ and~$\cat H$, and  by the faithfulness of~$f$ the latter are moreover \emph{thin} groupoids, meaning that they are equivalent to finite discrete sets.

\end{Rem}

\begin{Lem} \label{Lem:can-theta}
Suppose that $\cat S\colon \cat G^\op\to \ADD$ takes values in additive categories and additive functors, and that $f\colon \cat H\to \cat G$ is a faithful functor between finite groupoids. Then there exists a canonical natural isomorphism
\[
\Theta_f \colon f_! \overset{\sim}{\Longrightarrow} f_*
\]
between the left and right adjoints of the restriction functor~$f^*$ of \Cref{Cons:restriction}.
\end{Lem}

\begin{proof}
We define the component $\Theta_{f,(Y,\psi),p}$ of $\Theta_f$ at any objects $(Y,\psi)\in \Gamma(\cat H, \cat S)$ and $p\in \cat G$ by the universal properties of the colimit and limit defining $f_!$ and~$f_*$, as the unique morphism making the following square commute in~$\cat S_p$
\[
\xymatrix@R=8pt{
(f_!Y)_p  \ar@{=}[d] \ar@{..>}[rrrrrr]^-{\Theta_{f,(Y,\psi),p}} &&&&&& (f_*Y)_p \ar@{=}[d] \\
{\displaystyle{ \colim_{(x,\gamma\colon fx\to p)} (\gamma^{-1})^*Y_x }} &&&&&& {\displaystyle{ \lim_{(y,\delta\colon p\to fy)} \delta^*Y_y}} \ar[dd]^(.6){\pr_{(y,\delta)}}  \\
&&&&&& \\
 (\gamma^{-1})^* Y_x \ar[uu]^(.4){\incl_{(x,\gamma)}}
   \ar[rrrrrr]^-{  \displaystyle{\chi_{\delta\gamma} \,:=\, } \small{ \left\{\begin{array}{ll} (\gamma^{-1})^*(\psi_{f^{-1}(\delta\gamma)}) & \textrm{if } \delta\gamma \in f(\cat H), \\ 0 & \textrm{otherwise} \end{array} \right.  } } &&&&&&
  \delta^* Y_y
}
\]
 for all $(x,\gamma)\in f/p$ and $(y,\delta)\in p\backslash f$
(note that the zero map~$0$ makes sense as $\cat S_p$ is an additive category).
Here $f^{-1}(\delta \gamma)$ denotes the unique (by the faithfulness of~$f$) antecedent of $\delta \gamma \in f(\cat H)$.
The necessary compatibilities for $\Theta_{f,(Y,\psi),p}$ to exist follow immediately from~\eqref{eq:equiv-obj}.
It is straightforward but somewhat long to verify that, as $p$ varies, these assemble into a $\cat G$-equivariant map $\Theta_{f,(Y,\psi)}$, and that the latter is natural in~$(Y,\psi)$.

It remains to see that each $\Theta_{f,(Y,\psi),p}$ as above is invertible.
To this end, we rewrite the limit and colimit for $(f_!Y)_p$ and $(f_*Y)_p$ as direct sums as in \Cref{Rem:add-co-lims}, exploiting the equivalences $f/p\simeq \pi_0(f/p)$ and $p\backslash f\simeq \pi_0(p\backslash f)$.
This explains the two vertical isomorphisms in the following commutative diagram of~$\cat S_p$:
\[
\xymatrix{
(\gamma'^{-1})^*Y_{x'}
 \ar[dr]^{\incl_{(x',\gamma')}}
  \ar@/_6ex/[ddr]|(.4){ \underset{\textrm{at } [x,\gamma]}{(\gamma'^{-1})^* \left(\psi_{f^{-1} (\gamma^{-1}\gamma')}\right) } }
   \ar[rrrr]^-{\chi_{\delta'\gamma'}} & && & \delta'^* Y_{y'} \\
& {\displaystyle{\colim_{(x,\gamma)} (\gamma^{-1})^*Y_x} } \ar[rr]^-{\Theta_{f,(Y,\psi),p}} &&
 {\displaystyle{ \lim_{(y,\delta)} \delta^*Y_y} } \ar[ur]^{\pr_{(y',\delta')}} & \\
& {\displaystyle{\bigoplus_{[x,\gamma]\in \pi_0(f/p)} (\gamma^{-1})^*Y_x } }
 \ar[u]_\simeq^{ \left( \incl_{(x,\gamma)} \right) }
  \ar[rr]^-{\Delta}_-{\simeq} &&
  {\displaystyle{\bigoplus_{[y,\delta]\in \pi_0(p\backslash f)} \delta^*Y_y } } \ar@{<-}[u]^\simeq_{\left( \pr_{(y,\delta)} \right)}
   \ar@/_6ex/@{<-}[uur]|(.6){ \underset{\textrm{at } [y,\delta]}{\delta'^* \left( \psi_{ f^{-1}(\delta \delta'^{-1})} \right)} } &
}
\]
Note that the two direct sums involve choices of representative objects $(x,\gamma)$ for the equivalence classes $[x,\gamma] \in \pi_0(f/p)$, and similarly for $p\backslash f$.
Moreover, the isomorphism $f/p \overset{\sim}{\to} p\backslash f$, $(x,\gamma)\mapsto (x,\gamma^{-1})$, tells us that the two sums are isomorphic.
Indeed, for any choices of representative objects the matrix $\Delta:= ( \chi_{\delta\gamma} )_{[x,\gamma],[y,\delta]}$ yields such an isomorphism, where the component $\chi_{\delta \gamma}$ is the same map as above and thus is an isomorphism on the `diagonal' and zero off it.

Now it suffices to verify that the middle square in the last diagram commutes. For this we pre- and post-compose $\Theta$ with the canonical maps for arbitrary $(x',\gamma')\in f/p$ and $(y',\delta')\in p\backslash f$. We must have $(x',\gamma')\in [x,\gamma]$ and $(y',\delta') \in [y,\delta]$ for two of the chosen representatives $(x,\gamma)$ and $(y,\delta)$, that is $\gamma^{-1}\gamma'\colon f(x')\overset{\sim}{\to} f(x)$ and $\delta\delta'^{-1}\colon f(y')\to f(y)$ are in the image of~$f$; hence we may write the curved maps in the diagram, where ``at~$[x,\gamma]$'' indicates that the displayed map is the component into $(\gamma^{-1})^*Y_x$ with all others being zero, and similarly for the right one.
Now note that the upper square commutes by the definition of $\Theta$ and the two triangles by~\eqref{eq:*obj} (and its analogue for the colimit), hence it remains to verify the commutativity of the outermost square.
Because of the commutative diagram of isomorphisms in~$\cat G$
\begin{equation*} 
\xymatrix@R=10pt{
f(x') \ar[rr]^-{\delta'\gamma'} \ar[dr]_{\gamma'} \ar[dd]_{f(\cat H) \, \ni \; \gamma^{-1}\gamma'} & & f(y') \ar[dd]^{\delta\delta'^{-1} \; \in \,f(\cat H)} \\
& p \ar[ur]_{\delta'} \ar[dr]^{\delta} & \\
f(x) \ar[ur]^\gamma \ar[rr]^-{\delta\gamma} & & f(y)
}
\end{equation*}
we see that $\delta\gamma\in f(\cat H)$ iff $\delta'\gamma'\in f(\cat H)$, hence it suffices the check the square commutes when the latter holds (note for later that in this case $\delta\gamma' \in f(\cat H)$ too).
After applying~$\gamma'^*$, we are reduced to checking the commutativity of the square
\[
\xymatrix{
Y_{x'}
 \ar[d]_{\psi_{f^{-1}(\gamma^{-1}\gamma')}}
  \ar[rrr]^-{\psi_{f^{-1} (\delta'\gamma')}}
   \ar@{..>}[drrr]|{\psi_{f^{-1}(\delta\gamma')}} &&&
(\delta'\gamma')^*Y_{y'}
  \ar[d]^{(\delta'\gamma')^*(\psi_{f^{-1}(\delta\delta'^{-1})})} \\
(\gamma^{-1}\gamma')^* Y_x
  \ar[rrr]_-{(\gamma^{-1}\gamma')^*(\psi_{f^{-1}(\delta\gamma)})} &&&
(\delta\gamma')^*Y_y
}
\]
which immediately follows from~\eqref{eq:equiv-obj}.
\end{proof}

\begin{Cons} \label{Cons:restriction-comma}
Fix a finite groupoid $G{}$ and a pseudo-functor $\cat S\colon G{}^\op\to \CAT$. Recall from~\Cref{Def:gpdG} the 2-category $\gpdG$ of finite groupoids faithfully embedded in~$G{}$. We now explain how to extend the previous constructions to define a 2-functor $\Gamma(-,\cat S)\colon (\gpdG)^\op\to \CAT$; this is all straightforward, though notationally heavy when done precisely.
For an object $(H,i_H)$ of $\gpdG$, that is a finite groupoid $H$ equipped with a faithful functor $i_H\colon H\rightarrowtail G{}$, we set
\[
\Gamma((H,i_H), \cat S) := \Gamma (H, \cat S\circ i_H^\op)
\]
to be the category of $H$-equivariant objects as in \Cref{Def:Gamma}.
For every 1-morphism $(i,\theta_i)\colon (K,i_K)\to (H,i_H)$ in~$\gpdG$, that is a (necessarily faithful) functor $i\colon K\to H$ equipped with a natural isomorphism~$\theta_i\colon i_H\circ i\overset{\sim}{\Rightarrow} i_K$
\[
\xymatrix@R=8pt@C=8pt{
H \ar[drr]^{i_H} && && \\
\ar@{}[rr]|(.35){\SEcell  \theta_i} && G{} \ar[rr]^-{\cat S}_-{\mathrm{op}} && \CAT \\
K \ar[urr]_{i_K} \ar[uu]^i && &&
}
\]
we need a `restriction' functor $(i,\theta_i)^*\colon \Gamma((H,i_H), \cat S)\to \Gamma((K,i_K), \cat S)$.
We define it as in \Cref{Cons:restriction}, except that we must `correct' it by~$\theta_i$, as follows: For an object $(X,\varphi)$ and a morphism $\xi\colon (X,\varphi)\to (X',\varphi')$ in $\Gamma(H, \cat S\circ i_H^\op)$, we define
$(i,\theta_i)^*(X,\varphi)$ and $(i,\theta_i)^*(\xi)\colon (i,\theta_i)^*(X,\varphi)\to (i,\theta_i)^*(X',\varphi')$ in $\Gamma(K,\cat S\circ i_K^\op)$ by
\[
((i,\theta_i)^*(X,\varphi))_q := (\theta_{i,q}^{-1})^* (X_{i(q)}) \;\; \in \;\; \cat S_{i_K(q)}
\]
and
\[
((i,\theta_i)^* \xi )_q := (\theta_{i,q}^{-1})^* (\xi_{i(q)})
\]
for every $q\in \Obj(K)$, where $(\theta_{i,q}^{-1})^*\colon \cat S_{i_Hi(q)} \to \cat S_{i_K(q)}$ is the pull-back functor from~$\cat S$ associated with the morphism $\theta_{i,q}^{-1}\colon i_K(q) \to i_Hi(q)$ of~$G{}$.
For every 2-morphism $\alpha \colon (i,\theta_i)\Rightarrow (j,\theta_j)$ of~$\gpdG$, that is a natural isomorphism $\alpha \colon i\Rightarrow j$ such that $\theta_i = \theta_j(i_H \alpha)$, we need a natural isomorphism $\alpha^* \colon (i,\theta_i)^*\Rightarrow (j,\theta_j)^*$; we define $\alpha^*$ at an object $(X,\varphi) \in \Gamma(H,\cat S)$ to have component in~$\cat S_{i_K (q)}$ given by
\[
\xymatrix@R=10pt{
( (i,\theta_i)^*(X,\varphi))_q \ar@{=}[d]_{\textrm{def.}} \ar[rr]^-{\alpha^*_{(X,\varphi),q}} && ((j,\theta_j)^*(X,\varphi))_q \ar@{=}[d]^{\textrm{def.}} \\
(\theta_{i,q}^{-1})^* (X_{i(q)} ) \ar[rr]^-{(\theta_{i,q}^{-1})^* (\varphi_{i_H(\alpha_q)} ) } && (\theta_{j,q}^{-1})^* (X_{j(q)} )
}
\]
for every $q\in \Obj(K)$. (This makes sense because it uses the structural isomorphism $\varphi_{i_H(\alpha_q)}\colon X_{i(q)} \overset{\sim}{\to} (i_H\alpha_q)^* X_{j(q)}$ of $(X,\varphi)$ and $(\theta_{i,q}^{-1})^* (i_H\alpha_q)^* = ((i_H\alpha_q) \theta_{i,q}^{-1})^* = (\theta_{j,q}^{-1})^*$ since $(i_H\alpha_q) \theta_{i,q}^{-1}=\theta_{j,q}^{-1}$ by hypothesis on~$\alpha$.) The naturality of $\alpha^*_{(X,\varphi)}$ for morphisms $\xi \colon (X,\varphi)\to (X',\varphi')$ is easily verified using~\eqref{eq:equiv-map}.
We leave to the reader the similarly straightforward verification that the above data defines a 2-functor $(\gpdG)^\op\to \CAT$ as claimed.
\end{Cons}

\begin{Thm} \label{Thm:Mackey2fun-equi-objects}
Fix a finite groupoid $G{}$ and a pseudo-functor $\cat S\colon G{}^\op\to \ADD$ taking values in additive categories and additive functors.
Then the 2-functor
\[ \MM:= \Gamma(- ,\cat S)\colon (\gpdG)^\op\to \ADD \]
of \Cref{Cons:restriction-comma} is a Mackey 2-functor.
\end{Thm}

\begin{proof}
First note that $\Gamma(-,\cat S)$ takes values in $\ADD$ because $\cat S$ does, by \Cref{Rem:Gamma-refined}.
We must verify the axioms \Mack{1}-\Mack{4} as in \Cref{Def:Mackey-2-functor}. Additivity \Mack{1} is immediate from \Cref{Def:Gamma}.
The existence of the right and left adjoints $(i,\theta_i)_!$ and $(i,\theta_i)_*$ to every restriction $(i,\theta_i)^*$ is easily deduced from~\Cref{Lem:adjoints}; indeed, as we did when defining~$(i,\theta_i)^*$, one can simply `adjust' the explicit Kan formulas for the adjoints and their units and counits, as given in the lemma, by pulling them back via~$\theta_i$ as appropriate so they land in the right category. (More precisely, \eg for $(i,\theta_i)_*$: Given $(Y,\psi) \in \Gamma(K,\cat S\circ i_K^\op)$ and $p\in H$, define $((i,\theta_i)_*Y)_p:= \lim_{(x,\gamma)\in (p\backslash i)} \gamma^* \theta_{i,x}^* Y_x$ in $\cat S_{i_H(p)}$.)
Similarly, ambidexterity \Mack{4} is obtained by pulling back the isomorphism $\Theta_i$ of \Cref{Lem:can-theta}.
Only the base-change formulas \Mack{3} for the adjunctions $(i,\theta_i)_!\dashv (i,\theta_i)^*\dashv (i,\theta_i)_*$ remain to be proved.
These can be verified directly using the explicit (adjusted) adjunctions, \eg by rewriting the (co)limits over the slice categories as direct sums and computing the resulting matrix, as in the proof of \Cref{Lem:can-theta}; we leave details to the reader.
\end{proof}
\tristars

We now mention some examples where, for the sake of familiarity, we describe what happens only in terms of groups. Fix a finite group~$G{}$.

\begin{Exa} \label{Exa:LRS}
Let $X=(X,\mathcal O_X)$ be a (locally) ringed space (\eg a scheme) on which $G{}$ acts by morphisms of locally ringed spaces, and consider the abelian category $\cat{S}_{\sbull}=\Mod(X)$ of sheaves of $\mathcal O_X$-modules. Then $G{}$ acts on~$\cat{S}_{\sbull}$, in the pseudo-sense of~\Cref{Exa:single-cat}, by the pullback functors $g^*\colon \Mod(X) \overset{\sim}{\to} \Mod(X)$ ($g\in G{}$).
Associated to this pseudo-functor $\cat S\colon G^{\op}\to \End(\cat{S}_{\sbull})$, we obtain by \Cref{Thm:Mackey2fun-equi-objects} a Mackey 2-functor $\MM\colon (\gpdG)^\op\to \ADD$ whose value $\cat M(H)$ at a subgroup $H\leq G{}$ is the category $\Mod (X/\!\!/H) := \Gamma(H,\cat S)$ of $H$-equivariant sheaves of $\mathcal O_X$-modules. As $\cat{S}_{\sbull}$ is an abelian category, so is~$\Mod (X/\!\!/H)$.
\end{Exa}

\begin{Exa} \label{Exa:Qcoh}
If the locally ringed space $X$ of \Cref{Exa:LRS} is a scheme, we can also consider instead of $\Mod(X)$ its abelian full subcategory $\mathrm{Qcoh}(X)$ of quasi-coherent sheaves of $\mathcal O_X$-modules.
We obtain this way a Mackey 2-functor whose value at $H\leq G{}$ is the (Grothendieck) abelian category $\mathrm{Qcoh}(X/\!\!/H) \subset \Mod(X/\!\!/H) $ of $H$-equivariant quasi-coherent sheaves on~$X$.
\end{Exa}

\begin{Exa} \label{Exa:Der}
The previous examples can also be derived, by replacing the abelian category $\cat A \in \{\Mod(X), \mathrm{Qcoh}(X) \}$ with the category $\Ch(\cat A)$ of chain complexes in it and by levelwise extending to it the $G{}$-action. \Cref{Thm:Mackey2fun-equi-objects} yields a Mackey 2-functor $\cat M$ with values $\MM(H) = \Ch(\Gamma(H,\cat A)) = \Gamma(H, \Ch(\cat A))$ ($H\leq G{}$). By ambidexterity \Mack{4}, the restriction and induction functors between categories of complexes are exact, hence localization on the nose (\ie without deriving the adjoints) gives us Mackey 2-functors whose values are the derived categories $\Der(\Gamma(H,\cat A))$. Further variations on the theme, \eg with bounded complexes or with finiteness conditions, are left to the interested reader.
\end{Exa}


\end{chapter-four}
%
\chapter{Bicategories of spans}
\label{ch:bicat-spans}%
\bigbreak
\begin{chapter-five}

Our next goal is to construct the universal Mackey 2-functor: a Mackey 2-functor through which all the others factor. By definition, its target will be the 2-category (or rather, at first, the bicategory) of \emph{Mackey 2-motives}. This goal will only be achieved in \Cref{ch:2-motives}.
In the present chapter we provide a detailed study of an auxiliary construction, the bicategory of spans (\Cref{Def:Span-bicat}), which in some sense only captures `half' the properties of a universal Mackey 2-functor (see \Cref{Thm:UP-Span}).

Of course spans -- also known as `correspondences' -- have been studied for a long time; they are after all one of the basic categorical tools of symmetrization. More specifically, the bicategory of spans $\Span(\cat E)$ on a given category with pullbacks $\cat E$ already appeared in B\'enabou~\cite{Benabou67}, alongside the axioms of a bicategory. What we need to do here, however, is to apply the span construction to our 2-subcategory $\GG$ of finite groupoids of interest with distinguished class of faithful functors~$\JJ$, so we must generalize B\'enabou's construction in several directions. First, we must allow nontrivial (yet invertible) 2-cells in the input $\cat E$. Second, we must replace strict pull-backs by iso-comma squares. And finally, we must allow some asymmetry in the picture, by only allowing wrong-way 1-cells to go in one of the two directions (only the faithful $i\in \JJ$ give rise to induction functors).
Each of these generalizations has been explored before in some contexts, but we could not find any reference for all the basic properties that we will need, hence the present chapter. For future reference, we prove all results in somewhat greater generality by allowing $\GG$ to be any suitable (2,1)-category, not necessarily of groupoids, as we explain next.

\bigbreak
\section{Spans in a (2,1)-category}
\label{sec:Span}%
\medskip

%
\begin{Hyp} \label{Hyp:G_and_I_for_Span}
\index{$gandj$@$\GG$ and $\JJ$}%
\index{admissible class~$\JJ$ of 1-cells}%
We fix an essentially small (2,1)-category~$\GG$, that is, a (strict) 2-category where every 2-cell is invertible and where the Hom categories are small and the equivalence classes of objects form a set. We also fix a class $\JJ \subseteq \GG_1$ of 1-cells of~$\GG$. We say that the pair $(\GG,\JJ)$ is \emph{admissible} if it has the following rather mild properties:
\begin{enumerate}[\rm(a)]
\smallbreak
\item
\label{Hyp:G-I-a}%
The class $\JJ$ contains all equivalences and is closed under horizontal composition and isomorphism of 1-cells. Moreover, if $ij \in \JJ$ then $j\in \JJ$.
\smallbreak
\item
\label{Hyp:G-I-b}%
For every cospan $H \stackrel{i}{\to} G \stackrel{v}{\leftarrow} K$ of 1-cells with $i\in \JJ$, the comma square
\begin{align*}
\xymatrix@C14pt@R14pt{
& (i/v) \ar[ld]_-{\tilde v} \ar[dr]^-{\tilde i} & \\
H \ar[dr]_i \ar@{}[rr]|{\oEcell{\gamma}} && K \ar[dl]^v \\
&G &
}
\end{align*}
exists in~$\GG$ (\Cref{Def:comma}) and moreover $\tilde i$ still belongs to~$\JJ$.
\smallbreak
\item
\label{Hyp:G-I-c}%
Every 1-cell $i\colon H\to G$ in~$\JJ$ is \emph{faithful}, see \Cref{Ter:internal}\eqref{it:faithful}.
\smallbreak
\item
\label{Hyp:G-I-d}%
The 2-category~$\GG$ admits (strict) coproducts and $\JJ$ is closed under coproducts.
\end{enumerate}
We will occasionally use $\JJ$ to denote also the corresponding 2-full subcategory of~$\GG$.
\end{Hyp}

\begin{Rem} \label{Rem:less-hyps}
The implication ``$ij \in \JJ \Rightarrow j\in \JJ$'' in~\eqref{Hyp:G-I-a} and the faithfulness assumption~\eqref{Hyp:G-I-c} are only needed from \Cref{sec:pullback_2cells} onwards. In particular, they are not required in order to construct the span bicategory $\Span(\GG;\JJ)$ nor to prove its universal property in \Cref{sec:UP-Span}. It is possible that even the later construction of the bicategory $\Spanhat(\GG;\JJ)$ could also be carried out without those hypotheses, but because they are so convenient and still nicely cover all our examples, we renounce generalizing in this direction.
\end{Rem}

\begin{Rem}
Similarly, Hypothesis~\eqref{Hyp:G-I-d} is not needed for the construction of $\Span(\GG;\JJ)$ or its universal property. However when we return to Mackey 2-functors, additivity is a basic property and, to express it, we need Hypothesis~\eqref{Hyp:G-I-d}. See \Cref{sec:bicat_2Mack}.
\end{Rem}

\begin{Exa}
The example to keep in mind is of course $\GG=\groupoid$, the (2,1)-category of all finite groupoids, functors between them and natural transformations. In this case we may let $\JJ$ be the collection of all faithful functors. More generally, any 2-category~$\GG$ of finite groupoids as we used in previous chapters will do (see \Cref{Hyp:GG}).
\end{Exa}

\begin{Rem}
We are relaxing some of the assumptions made about~$\GG$ in \Cref{ch:2-Mackey}. Indeed, we do not assume anymore that $\GG$ is a sub-2-category of~$\groupoid$.
Furthermore, even for~$\GG$ a sub-2-category of~$\groupoid$, we do not require that it be closed under faithful $H\into G$ (\ie $G\in \GG$ and $H\into G$ faithful does not necessarily imply $H\in\GG$).
There is a gain in using a general (2,1)-category~$\GG$ instead of our main example $\groupoid$ and its sub-2-categories. First, it will be convenient at one point below to use~$\GG^\co$, the (2,1)-category with the same 0-cells and 1-cells but with reversed 2-cells. Secondly, and more importantly, we can use comma 2-categories of groupoids.
Although related to groupoids, the objects of such 2-categories contain additional information. For instance, we would like to consider for a fixed~$G$ the (2,1)-category $\gpdG$ of groupoids~$H$ together with a \emph{chosen} faithful functor $H\into G$ and such information is lost when one only focuses on~$H$.
\end{Rem}

Given a pair $(\GG,\JJ)$ as in \Cref{Hyp:G_and_I_for_Span}, we now give the first fundamental construction of this chapter in a formal way. For a more gentle approach, the reader is invited to consult the heuristic \Cref{sec:Span-co}.
\begin{Def}
\label{Def:Span-bicat}%
\index{bicategory!-- of spans $\Span(\GG;\JJ)$}%
\index{SpanG@$\Span(\GG;\JJ)$}%
\index{$spang$@$\Span$ \, bicategory of spans} %
The \emph{bicategory of spans} in the (2,1)-category~$\GG$ with respect to the class of 1-cells~$\JJ$, denoted
\[
\Span := \Span(\GG;\JJ)
\]
is given by the following data:
\begin{enumerate}[{$\bullet$}]
\item The objects are the same as those of $\GG$, that is, $\Span_0= \GG_0$.
\smallbreak
\item A 1-cell from $G$ to $H$ is a span
\[
\xymatrix{
G & P \ar[l]_-{u} \ar[r]^-{i} & H
}
\]
of 1-cells of $\GG$, with the second leg~$i$ belonging to~$\JJ$.
\item A 2-cell with domain $G \stackrel{u}{\leftarrow} P \stackrel{i}{\to} H$ and codomain $G \stackrel{v}{\leftarrow} Q \stackrel{j}{\to} H$ is an isomorphism class $[a,\alpha_1,\alpha_2]$ of diagrams
\begin{equation} \label{eq:2-cell-of-Span}
\vcenter{\xymatrix{
G \ar@{=}[d] \ar@{}[rrd]|{\SEcell\,\alpha_1}
&& P \ar[ll]_-{u} \ar[rr]^-{i} \ar[d]^a \ar@{}[rrd]|{\NEcell\,\alpha_2}
&& H \ar@{=}[d]
\\
G
&& Q \ar[ll]^-{v} \ar[rr]_-{j}
&& H\,.
}}
\end{equation}
\index{components of morphism in~$\Span(\GG;\JJ)$}%
in~$\GG$ with $a$ in~$\JJ$; we call $a$ the \emph{1-cell component} and $\alpha_1$ and $\alpha_2$ the \emph{2-cell components} (in~$\GG$) of the 2-cell~$[a,\alpha_1,\alpha_2]$ (in~$\Span$); two such diagrams $(a,\alpha_1,\alpha_2)$ and $(b,\beta_1,\beta_2)$ (with the same domain and codomain) are \emph{isomorphic}, \ie define the same 2-cell $[a,\alpha_1,\alpha_2]=[b,\beta_1,\beta_2]$, if there exists a 2-cell $\varphi\colon a \isoEcell b$ of~$\GG$
such that $(v \varphi)\alpha_1=\beta_1$ and $\alpha_2 = \beta_2(j \varphi)$:
\[
\vcenter{\xymatrix@C=14pt{
G \ar@{=}[d] \ar@{}[drr]|{\SEcell\alpha_1\;\;\;\;\;\;}  &&
 P \ar@/_2ex/[d]_a \ar@{}[d]|{\oEcell{\varphi}} \ar@/^2ex/[d]^b \ar[ll]_-{u} \\
 G && Q \ar[ll]^-{v}
}}
\;=\;
\vcenter{\xymatrix@C=14pt{
G \ar@{=}[d] \ar@{}[drr]|{\SEcell\beta_1}  &&
 P \ar[d]^b \ar[ll]_-u \\
 G && Q \ar[ll]^-{v}
}}
\quad\quad\quad
\vcenter{\xymatrix@C=14pt{
P \ar[rr]^-i \ar[d]_a \ar@{}[drr]|{\NEcell\alpha_2} && H \ar@{=}[d] \\
Q \ar[rr]_-j && H
}}
\;=\;
\vcenter{\xymatrix@C=14pt{
P \ar[rr]^-i \ar@{}[drr]|{\;\;\;\;\;\; \NEcell\beta_2} \ar@/_2ex/[d]_a \ar@{}[d]|{\oEcell{\varphi}} \ar@/^2ex/[d]^b &&
 H \ar@{=}[d] \\
Q \ar[rr]_-j && H
}}
\]
 Vertical composition of 2-cells is induced by pasting in~$\GG$ in the evident way. The identity 2-cell of $G \stackrel{u}{\leftarrow} P \stackrel{i}{\to} H$ is $[\Id_P,\id_u,\id_i]$.
\smallbreak
\item
The horizontal composition functors
\[ \circ = \circ_{G,H,K} \colon \Span(H,K) \times \Span(G,H) \longrightarrow \Span(G,K)
\]
are induced by iso-comma squares in~$\GG$. Explicitly, on objects this sends the pair
\[
(\xymatrix{
H & Q \ar[l]_-v \ar[r]^-j & K\,, \,
G & P \ar[l]_-{u} \ar[r]^-{i} & H
} )
\]
to the span $G \stackrel{u\tilde v}{\longleftarrow} (i/v) \stackrel{j\tilde i}{\longrightarrow} K$ obtained from the following comma square:
\begin{equation}
\label{eq:horiz-comp-Span}%
\vcenter{\xymatrix@R=10pt{
&& (i/v) \ar[dl]_-{\tilde v} \ar[dr]^-{\tilde i} \ar@/_4ex/[ddll]_-{u\tilde v} \ar@/^4ex/[ddrr]^-{j\tilde i} && \\
& P \ar[dl]^-u \ar[dr]_-{i} \ar@{}[rr]|{\stackrel{\sim}{\Ecell}} && Q \ar[dr]_-{j} \ar[dl]^-{v} & \\
G && H && K\,.
}}
\end{equation}
This is a well-defined 1-cell by the closure properties of~$\JJ$ detailed in \Cref{Hyp:G_and_I_for_Span}\,\eqref{Hyp:G-I-a} and~\eqref{Hyp:G-I-b}.
On arrows, the functor $\circ_{G,H,K}$ maps $([b,\beta_1,\beta_2],[a,\alpha_1,\alpha_2])$ to the 2-cell $[c, \alpha_1 \tilde v , \beta_2 \tilde i]$ defined by the following diagram:
\begin{align*}
\xymatrix@C-10pt{
&& && &
 (i/v) \ar[dd]^>>>>>>>>c \ar@/_2ex/[dlll]_-{\tilde v} \ar@/^1ex/[dr]^-{\tilde i} \ar@{}[dl]|{\Ecell \, \gamma\;\;\;} & && \\
 G \ar@{=}[dd] \ar@{}[ddrr]|{\SEcell \alpha_1} &&
 P \ar[ll]_-{u} \ar[rr]^-{i} \ar[dd]^a \ar@{}[ddrr]|{\NEcell \alpha_2} &&
 H \ar@{=}[dd]|>>>>>>>>>{{\phantom{\stackrel{m}{M}}}} \ar@{}[dr]|{\SEcell\beta_1} &&
 Q \ar[dd]^b \ar[ll]_-<<<<<<v|{{\phantom{M}}} \ar[rr]^-j \ar@{}[ddrr]|{\NEcell \beta_2} &&
 K \ar@{=}[dd] \\
&& && & (i'/v') \ar@/_2ex/[dlll]^-<<<{\tilde v'} \ar@/^1ex/[dr]^-{\tilde i'} \ar@{}[d]|{\Ecell \, \gamma'} & && \\
G && P' \ar[ll]^-{u'} \ar[rr]_-{i'} && H && Q' \ar[ll]^-{v'} \ar[rr]_-{j'} && K
}
\end{align*}
Here $c$ is the unique 1-cell $(i/v) \to (i'/v')$ such that $\tilde v' c = a \tilde v$, $\tilde i' c =b \tilde i$ and $\gamma' c= (\beta_1 \tilde i)\gamma (\alpha_2 \tilde v)$, given by the universal property of~$(i'/v')$. (Note that the two slanted squares are strictly commutative). In the notation of \Cref{Def:comma}, we have $c= \langle a\tilde v , b \tilde i , (\beta_1 \tilde i)\gamma (\alpha_2 \tilde v) \rangle$.
\smallbreak
\item
The associators are induced by the associativity isomorphisms of Remark~\ref{Rem:assoc}.
\smallbreak
\item
The left and right unitors are provided by the equivalences of Remark~\ref{Rem:units}. These will be invertible 2-cells, as required, because of Lemma~\ref{Lem:equiv-are-iso} below.
\end{enumerate}
\end{Def}

\begin{Conv}
When drawing diagrams involving 2-cells $[a,\alpha_1,\alpha_2]$ in~$\Span$ (and later in $\Spanhat$), we shall not display the 2-cell components $\alpha_1$ or $\alpha_2$ when they are equal to the identity, \ie when they appear in a square which commutes strictly. Conversely, when not displayed, we do mean that they are equal to the relevant identity.
\end{Conv}

\begin{Rem} \label{Rem:KEKSEKSA}
Apart from the variation involving the class~$\JJ$, the construction in \Cref{Def:Span-bicat} has been studied in details by~\cite{Hoffnung11pp}. In particular, the verification that the above data forms a bicategory can be deduced from~\cite[Theorem~3.0.3]{Hoffnung11pp}. To be precise, here we are only considering the bicategorical truncation of a natural tricategory structure, in the sense that what we call a 2-cell is actually an isomorphism class of a 2-cell with respect to certain naturally occurring 3-cells. In the present work, we do not need to consider this higher information.
\end{Rem}

\begin{Rem} \label{Rem:2-cats-no-go}
It is legitimate to wonder whether one could also define a bicategory of spans in a 2-category with non-invertible 2-cell (other than by forgetting them to get down to a (2,1)-category). Indeed, in that setting the Beck-Chevalley condition should involve comma squares rather than iso-comma squares, like in the (Der\,\ref{Der-4}) axiom for derivators. One would therefore want to compose spans via comma squares. This creates a new problem: we would not have any identity 1-cells in such a bicategory of spans, for the analogue of Remark~\ref{Rem:units} fails for comma squares. See \cite[Rem.\,3.1.5]{Hoffnung11pp}.

An alternative approach is suggested in the proof of \cite[Prop.\,4.9]{Hoermann18} (under some `multicategorical' layers). One can define a bicategory of spans using only spans in which the two legs are already Grothendieck (op)fibrations. In that case, composition can be defined via iso-comma squares for one expects them to satisfy Beck-Chevalley anyway -- this is where having (op)fibrations helps. The price to pay is that the `motivic' embedding, say contravariantly from $\GG$ to spans, would not send a 1-cell~$u$ to the obvious span $\loto{u}=$ anymore. Instead, $u$ would be sent to the span $\loto{p}\oto{q}$ where $p$ and $q$ are the Grothendieck (op)fibrations which appear in the following \emph{comma} square, associated to~$u$:
\[
\vcenter{\xymatrix@C=14pt@R=14pt{
& (\Id/u) \ar[dl]_-{p} \ar[dr]^-{q}
 \ar@{}[dd]|(.5){\Ecell}
\\
\ar@{=}[dr]_-{\Id}
&& \ar[dl]^-{u}
\\
&
}}
\]
We leave such generalizations to the interested reader.
\end{Rem}

\begin{Lem} \label{Lem:equiv-are-iso}
A 2-cell $[a,\alpha_1,\alpha_2]$ of $\Span$ as in~\eqref{eq:2-cell-of-Span} is invertible if and only if the 1-cell $a\colon P\to Q$ of~$\GG$ is an equivalence.
\end{Lem}
\begin{proof}
Clearly if $[a,\alpha_1,\alpha_2]$ is an invertible arrow of the category $\Span (G,H)$ then the 1-cell $a$ is invertible in~$\GG$ up to isomorphism, \ie is an equivalence. Conversely, assume that $a$ is an equivalence in~$\GG$, so that we find a 1-cell $b\colon Q\to P$ and isomorphisms $\varepsilon\colon ab \isoEcell \Id_{Q}$ and $\eta\colon\Id_{P} \isoEcell ba$.
Define a 2-cell as follows:
\begin{equation} \label{eq:inverse-2cell}
\vcenter{\xymatrix{
G\ar@{=}[dd] \ar@{}[ddr]|{\alpha_1\inv\Scell} &
 Q \ar[l]_-v \ar@{}[ddr]|{\varepsilon\inv\SEcell} &
 Q \ar[l]_-\Id \ar[r]^-\Id \ar[dd]^<<<<<b &
 Q \ar[r]^-j \ar@{}[ddl]|{\NEcell \,\varepsilon} &
 H \ar@{=}[dd] \ar@{}[ddl]|{\Ncell\alpha_2^{-1}} \\
&& && \\
G && P \ar@/^4ex/[uul]^-{a} \ar@/_4ex/[uur]_-{a} \ar[ll]^-u \ar[rr]_-{i} &&
 H
}}
\end{equation}
Composing vertically, we get
\begin{align*}
\xymatrix{
G\ar@{=}[dd] \ar@{}[ddr]|{\alpha_1\inv\Scell} &
 Q \ar[l]_-v \ar@{}[ddr]|{\varepsilon\inv\SEcell} &
 Q \ar[l]_-\Id \ar[r]^-\Id \ar[dd]^<<<<<b &
 Q \ar[r]^-j \ar@{}[ddl]|{\NEcell \,\varepsilon} &
 H \ar@{=}[dd] \ar@{}[ddl]|{\Ncell\alpha_2^{-1}} \\
&& && \\
G \ar@{=}[d] \ar@{}[rrd]|{\SEcell\,\alpha_1} &&
 P \ar[d]^a \ar@/^4ex/[uul]^-{a} \ar@/_4ex/[uur]_-{a} \ar[ll]^-u \ar[rr]_-{i} \ar@{}[rrd]|{\NEcell\,\alpha_2} &&
 H \ar@{=}[d] \\
G &&
 Q \ar[ll]^-{v} \ar[rr]_-{j} &&
 H\,.
}
\end{align*}
which is isomorphic to $(\Id_Q,\id_v,\id_j)$ via $\varepsilon\colon ab \isoEcell \Id_Q$. Therefore~\eqref{eq:inverse-2cell} is a right inverse of $[a,\alpha_1,\alpha_2]$. But~\eqref{eq:inverse-2cell} is also a 2-cell whose middle 1-cell component is an equivalence, therefore by the same argument it admits itself a right inverse. It follows formally that $[a,\alpha_1,\alpha_1]$ and~\eqref{eq:inverse-2cell} are mutually inverse 2-cells.
\end{proof}
\begin{Not} \label{Not:Bicat-1-cells}
Given 1-cells $u,i\in \GG_1$ with $i\in \JJ$, we will use the notation
\[ i_!u^* := (
\xymatrix@C=2em{
 G & P \ar[l]_-{u} \ar[r]^-{i} & H
} )
\quad \in \quad
\Span(G,H) \]
for the associated 1-cell of~$\Span$, as well as the short-hand
\[
i_! := i_!\Id^* = (
\xymatrix@C=2em{
 P & P \ar[l]_-\Id \ar[r]^-{i} & H
} )
\quadtext{and}
u^* := \Id_!u^* = (
\xymatrix@C=2em{
 G & P \ar[l]_-{u} \ar[r]^-\Id & P
} ) \,. \]
\end{Not}
\begin{Rem} \label{Rem:notation_assoc}
With this notation, for an arbitrary 1-cell $i_!u^*$ of $\Span$ there is a canonical isomorphism $i_!u^* \isoEcell i_! \circ u^*$ given by the following diagram
\begin{equation*}
\xymatrix{
G \ar@{=}[d] \ar@{-->}@/^4ex/[rrrr]^-{i_!u^*} &&
 P \ar[d]^-{\Delta_P} \ar[ll]_-{u} \ar[rr]^-{i} &&
 H\ar@{=}[d] \\
G \ar@/_2ex/@{-->}[rrd]_-{u^*} &
 P \ar[l]_-{u} \ar[dr]_-\Id &
 \Id/ \Id \ar[l]_-{\pr_1} \ar[r]^-{\pr_2} \ar@{}[d]|{\oEcell{\gamma}} &
 P \ar[r]^-{i} \ar[dl]^-{\Id} &
 H \\
&& P \ar@/_2ex/@{-->}[rru]_-{i_!} &&
}
\end{equation*}
where $\Delta_P=\langle\Id_P, \Id_P, \id_{\Id_P}\rangle$. Indeed $\Delta_P$ is an equivalence by Remark~\ref{Rem:units}, hence by Lemma~\ref{Lem:equiv-are-iso} the above 2-cell $[\Delta_P,\id,\id]$ is invertible. Explicitly, both
\begin{equation}
\label{eq:inverse-DeltaP}%
\vcenter{
\xymatrix@C=12pt{
G \ar@{=}[d] &
 \ar[l]_-{u} P
& (\Id/\Id) \ar[d]_-{\pr_1} \ar[l]_-{\pr_1} \ar[r]^-{\pr_2}
 \ar@{}[dll]|{\id\;\SEcell} \ar@{}[rrd]|{\NEcell\;i\gamma}
& P \ar[r]^-{i}
& H \ar@{=}[d] \ar@{}[rrd]|*{\textrm{and}} &&
G \ar@{=}[d]
& P \ar[l]_-{u}
& (\Id/\Id) \ar[d]^-{\pr_2} \ar[l]_-{\pr_1} \ar[r]^-{\pr_2}
 \ar@{}[dll]|{\SEcell\; u \gamma} \ar@{}[rrd]|{\NEcell\;\id}
& P \ar[r]^-{i} &
 H \ar@{=}[d] \\
G&& P \ar[ll]^-u \ar[rr]_-{i} &&H &&
G && P \ar[ll]^-u \ar[rr]_-{i} && H
}}\kern-2em
\end{equation}
provide diagrams representing its inverse.
\end{Rem}

\begin{Rem}
\label{Rem:Span-comp-up-to-iso}%
We have defined horizontal composition in~$\Span(\GG;\JJ)$ by means of iso-commas; see~\eqref{eq:horiz-comp-Span}. In view of our discussion of \emph{Mackey squares} in~\Cref{sec:mackey-squares}, the reader might wonder whether one can get away with filling-in any Mackey square instead of the iso-comma in~\eqref{eq:horiz-comp-Span}. As long as we want a well-defined horizontal composition, it is necessary to choose an actual composite.\,(\footnote{\,This is of course a place where a post-modernist reader would be equally satisfied with a contractible space of choices. But we follow a more patriarchal definition of bicategory.}) However, in practice, a Mackey square often comes to mind more easily than the precise iso-comma. A good example is of course the one discussed in~\Cref{Rem:units}. Thanks to the above \Cref{Lem:equiv-are-iso}, we know that horizontal composition with the iso-comma is isomorphic to the result of `composing' by means of a Mackey square. We shall return to this idea in the next chapter, see specifically~\Cref{Rem:vertical-comp-Spanhat}.
\end{Rem}

\begin{Cons}
\label{Cons:first-embeddings}%
Following \Cref{Not:Bicat-1-cells}, we define two pseudo-functors
\[
(-)^*\colon \GG^{\op}\hook \Span(\GG;\JJ)
\qquadtext{and}
(-)_!\colon \JJ^{\co} \hook \Span(\GG;\JJ)
\]
where $\JJ$ is seen as a 2-full subcategory of~$\GG$ for the latter. Both constructions are the identity on 0-cells. The first pseudo-functor $(-)^*$ sends a 1-cell $u\colon H\to G$ to~$u^*\colon G\to H$ and sends a 2-cell $\alpha\colon u\Rightarrow v$ to the 2-cell $\alpha^*\colon u^*\Rightarrow v^*$ represented by
\begin{align*}
\xymatrix{
G \ar@{=}[d]
&& H \ar[ll]_-{u} \ar[rr]^-\Id \ar[d]^\Id
 \ar@{}[dll]|{\SEcell\,\alpha} \ar@{}[rrd]|{\NEcell\,\id}
&& H \ar@{=}[d]
\\
G
&& H \ar[ll]^-{v} \ar[rr]_-\Id
&& H\,.\!
}
\end{align*}
The second pseudo-functor $(-)_!$ sends a 1-cell $i\colon H\into G$ (in~$\JJ$) to $i_!\colon H\to G$ and sends a 2-cell $\alpha\colon i\Rightarrow j$ to the 2-cell $\alpha_!\colon j_!\Rightarrow i_!$ represented by
\begin{align*}
\xymatrix{
G \ar@{=}[d]
&& G \ar[ll]_-{\Id} \ar[rr]^-{j} \ar[d]^\Id
 \ar@{}[dll]|{\SEcell\,\id} \ar@{}[rrd]|{\NEcell\,\alpha}
&& H \ar@{=}[d]
\\
G
&& G \ar[ll]^-{\Id} \ar[rr]_-{i}
&& H\,.\!
}
\end{align*}
\end{Cons}

\begin{Rem}
\label{Rem:pseudo-func-embeddings}%
Both $(-)^*$ and $(-)_!$ in \Cref{Cons:first-embeddings} are pseudo-functors, not strict 2-functors. Indeed, for composable 1-cells $K\oto{v}H\oto{u}G$ the canonical equivalence $\langle v, \Id_K, \id_v\rangle\colon K\isoto (\Id_H/v)$
\begin{equation*}
\xymatrix{
G \ar@{=}[d] \ar@{-->}@/^4ex/[rrrr]^-{(uv)^*} &&
 K \ar[d]^-{\simeq} \ar[ll]_-{uv} \ar[rr]^-{\Id} &&
 K\ar@{=}[d] \\
G \ar@/_2ex/@{-->}[rrd]_-{u^*} &
 H \ar[l]_-{u} \ar[dr]_-\Id
& (\Id/v) \ar[l]_-{\pr_1} \ar[r]^-{\pr_2} \ar@{}[d]|{\Ecell } &
 K \ar[r]^-{\Id} \ar[dl]^-{v} &
 K \\
&& H \ar@/_2ex/@{-->}[rru]_-{v^*} &&
}
\end{equation*}
constructed as in Remark~\ref{Rem:units} only yields canonical isomorphisms $(u v)^*\isoto v^* \circ u^*$ rather than equalities. Similarly, for $K\oto{j}H\oto{i}G$ in~$\JJ$, the canonical equivalence $\langle \Id_{K},j,\id_j\rangle\colon K\isoto (j/\Id_H)$ yields the canonical isomorphisms $(i j)_!\isoto i_!\circ j_!$.
One verifies immediately that these pseudo-functors are 2-fully faithful, \ie are bijective on each set of parallel 2-cells. We will thus consider them as embeddings of our 2-categories~$\GG$ and~$\JJ$ into the bicategory $\Span(\GG;\JJ)$.
\end{Rem}

\begin{Rem}
\label{Rem:co-on-alpha_!}%
The 2-contravariance of~$(-)_!$ in the above \Cref{Cons:first-embeddings} is opposite to the 2-covariance of~$(-)^*$. We are going to prove in \Cref{Prop:first-adjoints} that $i_!$ is left adjoint to $i^*$ in the bicategory $\Span(\GG;\JJ)$. Therefore, given a 2-cell $\alpha\colon i\Rightarrow j$ in~$\JJ$, the 2-cell $\alpha^*\colon i^*\Rightarrow j^*$ of $\Span$ admits a mate $\alpha_!\colon j_!\Rightarrow i_!$ (compare \Cref{Rem:pseudo-func-of-adjoints}) which will be seen to agree with the homonymous 2-cell $\alpha_!$ constructed above (\Cref{Rem:special-mate}).
\end{Rem}

Recall the internal definition of adjunctions in a bicategory from \Cref{Ter:internal}\,\eqref{it:adjunction}. The essential advantage that $\Span(\GG;\JJ)$ has over $\GG$ is that 1-cells in~$\JJ$ get a left adjoint in~$\Span(\GG;\JJ)$ under the embedding discussed above.
\begin{Prop} \label{Prop:first-adjoints}
Let $i\colon H\to G$ be any element of the distinguished class~$\JJ\subset \GG_1$.
Then there exists in~$\Span(\GG;\JJ)$ a canonical adjunction $i_! \dashv i^*$, with explicit units and counits respectively given in~\eqref{eq:can_unit} and~\eqref{eq:can_counit} below.
\end{Prop}
\begin{proof}
The unit $\eta\colon \Id_H\Rightarrow i^*\circ i_!$ of the adjunction is given by the 2-cell $[\Delta_i,\id,\id]$ where $\Delta_i\colon H\to (i/i)$ is given by $\Delta_i=\langle \Id_H,\Id_H,\id_i\rangle$:
\begin{equation} \label{eq:can_unit}
\vcenter{\xymatrix{
H \ar@{=}[d] &&
 H \ar[d]^-{\Delta_i} \ar[ll]_-{\Id} \ar[rr]^-{\Id} &&
 H\ar@{=}[d] \\
H \ar@/_2ex/@{-->}[rrd]_-{i_!} &
 H \ar[l]_-\Id \ar[dr]_-{i} &
 (i/i) \ar[l]_-{\pr_1} \ar[r]^-{\pr_2} \ar@{}[d]|{\oEcell{\gamma_i}} &
 H \ar[r]^-\Id \ar[dl]^-{i} &
 H \\
&& G \ar@/_2ex/@{-->}[rru]_-{i^*} &&
}}
\end{equation}
The counit $\varepsilon\colon i_!\circ i^* \Rightarrow \Id_G$ is the following composite $[i,\id,\id]\circ [\Delta_H,\id,\id]\inv$, where $\Delta_H=\langle\Id_H, \Id_H, \id_{\Id_H}\rangle\colon H\isoto (\Id_H/\Id_H)$:
\begin{equation} \label{eq:can_counit}
\vcenter{\xymatrix{
&& H \ar@/^2ex/@{-->}[rrd]^-{i_!} && \\
G\ar@{=}[d] \ar@/^2ex/@{-->}[rru]^-{i^*} &
 H \ar[l]_-{i} \ar[ur]^-{\Id} &
 (\Id/\Id) \ar@{}[u]|{\oEcell{\gamma_H}} \ar[l]^-{\pr_1} \ar[r]_-{\pr_2} &
 H \ar[r]^-{i} \ar[ul]_-{\Id} &
 G \ar@{=}[d] \\
G \ar@{=}[d] && H \ar[ll]_-{i} \ar[rr]^-{i} \ar[u]_-{\Delta_H}^\simeq \ar[d]^i &&
 G \ar@{=}[d] \\
G && G \ar[ll]_-{\Id} \ar[rr]^-{\Id} && G
}}
\end{equation}
(In the above two diagrams, we sloppily denote by the same symbols $\pr_1$ and $\pr_2$ the two projections of each comma square, although we take care to distinguish their 2-cells.) Note that $[\Delta_H,\id,\id]$ is an instance of the isomorphism of Remark~\ref{Rem:notation_assoc}, namely $i_!\circ i^*\cong i_!i^*$. Thus there are two canonical diagrams representing $[\Delta_{\Id},\id,\id]\inv$, and therefore $\varepsilon = [i,\id,\id]\circ [\Delta_H,\id,\id]\inv = [i \pr_1 , \id , i \gamma_H] = [ i \pr_2 , i \gamma_H , \id]$.

We must verify the two triangle identities
$ (\varepsilon i_! )(i_! \eta) = \id_{i_!}$
and
$(i^*\varepsilon)(\eta i^*)=\id_{i^*}$, which involve the associators and unitors of the bicategory~$\Span$.
For the first identity, we compute the whiskerings $i_! \eta$ and $\varepsilon i_!$ as follows (recall the notation $(A\diagup_{\!\!\!\scriptscriptstyle C}\, B)$ from the end of \Cref{Def:comma}):
\begin{align*}
\xymatrix{
&& && & (H\diagup_{\!\!\!\scriptscriptstyle H}\, H) \ar[dd]|(.7){\langle\Delta_i \pr_1, \pr_2,\gamma_H \rangle}
 \ar@/_1em/[dlll]_-{\pr_1} \ar@/^1em/[dr]^-{\pr_2} \ar@{}[dl]|{\oEcell{\gamma_H} \;\;\;}
& &&
\\
H \ar@{=}[dd]
&& H \ar[ll]_-{\Id} \ar[rr]^(.6){\Id} \ar[dd]^-{\Delta_i}
&& H \ar@{=}[dd]|(.59){\hole}
&& H \ar@{=}[dd] \ar[ll]|(.5){\hole}_(.3){\Id} \ar[rr]^-{i}
&& G \ar@{=}[dd]
\\
&& && & ((H\diagup_{\!\!\!\scriptscriptstyle G}\, H)\diagup_{\!\!\!\scriptscriptstyle H}\, H) \ar@/_1em/[dlll]_(.6){\pr_1} \ar@/^1em/[dr]^(.55){\pr_2} \ar@{}[d]|{\Ecell }
& &&
\\
H
&& (H\diagup_{\!\!\!\scriptscriptstyle G}\, H) \ar[ll]^-{\pr_1} \ar[rr]_-{\pr_2}
&& H
&& H \ar[ll]^-{\Id} \ar[rr]_-{i}
&& G
}
\end{align*}
\begin{align*}
\xymatrix{
&& && & (H\diagup_{\!\!\!\scriptscriptstyle G}\, (H\diagup_{\!\!\!\scriptscriptstyle H}\, H))
 \ar[dd]|(.7){\langle\pr_1, i \pr_2\pr_2,(i\gamma_H \pr_2)\delta\rangle \;\;} \ar@/_1em/[dlll]_-{\pr_1} \ar@/^1em/[dr]^-{\pr_2} \ar@{}[dl]|{\oEcell{\delta} \;\;\;} & &&
\\
H \ar@{=}[dd]
&& H \ar[ll]_-{\Id} \ar[rr]^(.6){i} \ar@{=}[dd]
&& G \ar@{=}[dd]|(.62){\hole}
&& (H\diagup_{\!\!\!\scriptscriptstyle H}\, H) \ar[dd]^-{i \pr_2} \ar[ll]|(.55){\hole}_(.35){i \pr_1} \ar[rr]^-{i \pr_2}
 \ar@{}[ld]^(.3){\SEcell\,i\gamma_H}
&& G \ar@{=}[dd]
\\
&& && & (H\diagup_{\!\!\!\scriptscriptstyle G}\, G) \ar@/_1em/[dlll]_(.6){\pr_1} \ar@/^1em/[dr]^(.5){\pr_2} \ar@{}[d]|{\oEcell{\lambda}\quad } & &&
\\
H
&& H \ar[ll]^-{\Id} \ar[rr]_-{i} && G && G \ar[ll]^-{\Id} \ar[rr]_-{\Id} && G
}
\end{align*}
Now we combine them with the appropriate associator and unitors, and we claim that the resulting composite 2-cell
\[
\xymatrix@L=1ex{
i_! \ar@{=>}[d]_-{\run\inv}
&& H \ar@{=}[d]
&& H \ar[d]^-{\Delta_H}_{\simeq} \ar[ll]_-{\Id} \ar[rr]^-{i}
&& G \ar@{=}[d]
\\
i_! \circ \Id_H \ar@{=>}[d]_-{i_! \eta}
&& H \ar@{=}[d] & H \ar[l]_-{\Id} \ar[d]_-{\Delta_i}
& (H\diagup_{\!\!\!\scriptscriptstyle H}\, H) \ar[d]|{\langle\Delta_i \pr_1,\pr_2,\gamma_H\rangle} \ar[r]^-{\pr_2} \ar[l]_-{\pr_1} & H \ar[r]^-{i} \ar@{=}[d] & G \ar@{=}[d]
\\
i_! \circ (i^* \circ i_!) \ar@{=>}[d]_-{\ass\inv}
&& H \ar@{=}[d]
& (H\diagup_{\!\!\!\scriptscriptstyle G}\, H) \ar[l]_-{\pr_1}
& ((H\diagup_{\!\!\!\scriptscriptstyle G}\, H)\diagup_{\!\!\!\scriptscriptstyle H}\, H) \ar[d]^-{\cong} \ar[l]_-{\pr_1} \ar[r]^-{\pr_2} & H \ar[r]^-{i} & G \ar@{=}[d]
\\
(i_! \circ i^*) \circ i_! \ar@{=>}[d]_-{\varepsilon i_!}
&& H \ar@{=}[d]
&& (H\diagup_{\!\!\!\scriptscriptstyle G}\, (H\diagup_{\!\!\!\scriptscriptstyle H}\, H)) \ar[d]|{\langle\pr_1, i\pr_2\pr_2,(i \gamma_H \pr_2)\delta\rangle\;\;} \ar[ll]_-{\pr_1} \ar[rr]^-{i \pr_2 \pr_2}
&& G \ar@{=}[d]
\\
\Id_G \circ i_! \ar@{=>}[d]_-{\lun}
&& H \ar@{=}[d]
&& (H\diagup_{\!\!\!\scriptscriptstyle G}\, G) \ar[d]^-{\pr_1} \ar[ll]_-{\pr_1} \ar[rr]^-{\pr_2}
&& G \ar@{=}[d] \ar@{}[dll]|{\NEcell\; \lambda}
\\
i_! &&
H && H \ar[ll]_-{\Id} \ar[rr]_-{i} && G
}
\]
is the identity of~$i_!$. To prove this, it suffices to verify that by composing down the middle column except for the last step~$\pr_1$, we obtain the 1-cell $\langle \Id_H, i, \id_i\rangle\colon H\to (H\diagup_{\!\!\!\scriptscriptstyle G}\, G)$, because this would imply that the corresponding 2-cell of $\Span$ is the inverse of the left unitor~$\lun\colon \Id_G\circ i_!\Rightarrow i_!$.

The only non-evident part of this is to correctly identify the 2-cell components. 
One way to compute it efficiently is to precompose the above composite $H\to (H\diagup_{\!\!\!\scriptscriptstyle G}\, G)$ with an arbitrary 1-cell $t\colon T\to H$ and then exploit the element-wise description of comma squares in the 2-category of categories as detailed in Remark~\ref{Rem:commas_translated}. In this way it is easy to see that by composing any $t\colon T\to H$ down to $(H\diagup_{\!\!\!\scriptscriptstyle G}\, G)$ we obtain the 1-cell $\langle t,it,\id_{it}\rangle$, which is the required result, and then we conclude with Corollary~\ref{Cor:2cat_Yoneda_equivs}.

The verification of the second triangle identity is completely similar and is left to the reader.
\end{proof}

\begin{Rem}
\label{Rem:special-mate}%
As announced in \Cref{Rem:co-on-alpha_!}, for a 2-cell $\alpha\colon i\Rightarrow j$ in~$\JJ$, the corresponding 2-cell $\alpha_!$ in~$\Span(\GG;\JJ)$ given in \Cref{Cons:first-embeddings} is in fact the mate of the 2-cell $\alpha^*\colon i^*\Rightarrow j^*$ under the adjunctions $i_!\adj i^*$ and $j_!\adj j^*$ of \Cref{Lem:BC-for-Span}. This verification is left to the reader but it justifies the chosen 2-\emph{contra}variance of the pseudo-functor $(-)_!\colon \JJ^{\co}\to \Span(\GG;\JJ)$.
\end{Rem}

By \Cref{Cons:first-embeddings} and Proposition~\ref{Prop:first-adjoints}, we see that $\Span(\GG;\JJ)$ is an enlargement of $\GG$ which accommodates a left adjoint $i_!\adj i^*$ for every 1-cell $i\in \JJ$. These adjoints are not freely added to~$\GG$ though, but instead automatically satisfy the following Beck-Chevalley or base-change condition:

\begin{Lem} \label{Lem:BC-for-Span}
The canonical pseudo-functor $\GG^{\op}\to \Span(\GG; \JJ)$ of \Cref{Cons:first-embeddings} has the following property: For any comma square in~$\GG$
\begin{align*}
\xymatrix@C=14pt@R=14pt{
& (i/v) \ar[ld]_-{\tilde v} \ar[dr]^-{\tilde i} & \\
P \ar[dr]_i \ar@{}[rr]|{\oEcell{\gamma}} && Q \ar[dl]^v \\
&H &
}
\end{align*}
with $i\in \JJ$, the mate $\gamma_!$ of the image $\gamma^*$ of $\gamma$ in $\Span(\GG; \JJ)$, with respect to the adjunctions $i_!\dashv i^*$ and $\tilde{i}_!\dashv \tilde{i}^*$ of Proposition~\ref{Prop:first-adjoints}
\begin{align*}
\xymatrix@C=14pt@R=14pt{
& (i/v) \ar[dr]^-{\tilde i_!} & \\
P \ar[ur]^-{\tilde v^*} \ar[rd]_-{i_!} \ar@{}[rr]|{\Scell\, \gamma_!} && Q \\
&H \ar[ru]_-{v^*} &
}
\end{align*}
is invertible in $\Span(\GG;\JJ)$.
\textup(\footnote{\,The meaning of the notation $\gamma_!$ appearing here is \emph{not} the meaning of $(-)_!$ in \Cref{Cons:first-embeddings}. Note that the latter does not necessarily make sense on the 2-cell~$\gamma\colon i\tilde{v}\Rightarrow v\tilde{i}$ since neither $i\tilde{v}$ nor $v\tilde{i}$ are assumed to belong to~$\JJ$. Compare \Cref{Rem:special-mate}.}\textup)
\end{Lem}

\begin{proof}
By definition, the mate $\gamma_!$ in question is the 2-cell of $\Span(\GG;\JJ)$ defined by the following pasting, where $\eta$ and $\varepsilon$ are the unit and counit of the adjunction:
\begin{equation}\label{eq:gamma_!}
\vcenter{\xymatrix@C=14pt@R=14pt{
&&& Q \ar@{}[dd]|(.4){\oEcell{\varepsilon}} & \\
&& i/v \ar[ur]^-{\tilde i_!} && \\
& P \ar[ur]^-{\tilde v^*} \ar@{}[rr]|{\oEcell{\gamma^*}} && Q \ar[ul]_-{\tilde i^*} \ar@/_4ex/[uu]_-{\Id} & \\
 && H \ar[ru]_-{v^*} \ar[ul]^-{i^*} && \\
& P \ar@/^4ex/[uu]^-{\Id} \ar[ru]_-{i_!} \ar@{}[uu]|(.4){\oEcell{\eta}} &&&
}}
\end{equation}
Making the unitors and associators explicit, $\gamma_!$ is the following composite:
\[
\xymatrix@C=14pt@R=8pt@L=1ex{
\tilde i_! \circ \tilde v^* \ar@{=>}[r]^-{\sim} &
(\tilde i_! \circ \tilde v^*) \circ \Id_P \ar@{=>}[r]^-{\eta} &
(\tilde i_! \circ \tilde v^*) \circ (i^* \circ i_!) \ar@{=>}[r]^-{\sim} &
(\tilde i_! \circ (\tilde v^* \circ i^*)) \circ i_! \ar@{=}[r]^-{\gamma^*} & \cdots \\
\cdots \ar@{=>}[r]^-{\gamma^*} &
(\tilde i_! \circ (\tilde i^* \circ v^*)) \circ i_! \ar@{=>}[r]^-{\sim} &
(\tilde i_! \circ \tilde i^* ) \circ ( v^* \circ i_! ) \ar@{=>}[r]^-{\varepsilon} &
\Id_Q \circ ( v^* \circ i_! ) \ar@{=>}[r]^-{\sim} &
v^* \circ i_!
}
\]
Note that by the definition of horizontal composition the target span is $v^* \circ i_! = \tilde i_! \tilde v^*$.

We must show that $\gamma_!$ is invertible. We claim that $\gamma_!$ is in fact the inverse of the canonical isomorphism $\tilde i_!\tilde v^* \isoEcell \tilde i_! \circ \tilde v^*$ of Remark~\ref{Rem:notation_assoc}. In order to prove our claim, we explicitly compute the composite morphism of spans
\[
\xymatrix@L1ex{
\tilde i_!\tilde v^*
 \ar@{=>}[r]^-{\sim} &
\tilde i_! \circ \tilde v^*
 \ar@{=>}[r]^-{\gamma_!} &
\tilde i_!\tilde v^*
}
\]
and show it is equal to~$\id_{\tilde i_! \tilde v^*}$. This can be done as follows.

As a first step, we compute the composite
$\tilde i_! \tilde v^* \Rightarrow (\tilde i_! \circ (\tilde v^* \circ i^*)) \circ i_!$
of the first four maps. This is straightforward, as the 2-cell components of these maps are all identities, so this amounts to determining the composite 1-cell component. The target span is
\[
\xymatrix@!C=16pt@R=12pt{
&&&& (i/ (P/\tilde v))/(i/v) \ar[dl] \ar[dddrrr] &&&& \\
&&& i/(P/\tilde v) \ar[ddll] \ar[dr] \ar@{}[ddrrrr]|{\oEcell{\rho}} &&&&& \\
&&&& P/\tilde v \ar[dl] \ar[dr] &&&& \\
& P \ar[dl] \ar[dr]^i \ar@{}[urrr]|{\Ecell} &&
 P \ar[dl]_i \ar[dr] \ar@{}[rr]|{\oEcell{\delta}} &&
 i/v \ar[ld]_-{\tilde v} \ar[dr] &&
 i/v \ar[dl] \ar[dr]^-{\tilde i} & \\
P \ar@{..>}[rr]_-{i_!}&&
 H \ar@{..>}[rr]_-{i^*} &&
 P \ar@{..>}[rr]_-{\tilde v^*} &&
 i/v \ar@{..>}[rr]_-{\tilde i_!} &&
 Q
}
\]
(where we have labeled $\delta$ and $\rho$ for later use) and we easily compute that the resulting 1-cell is
\begin{align} \label{eq:first_four_maps}
\big\langle \;
\big\langle \tilde v , \langle \tilde v , \Id_{i/v} , \id_{\tilde v} \rangle , \id_{i \tilde v} \big\rangle \;,\; \Id_{i/v} \;,\;
\id_{\Id_{i/v}}
\; \big\rangle
\colon (i/v) \longrightarrow (i/ (P/\tilde v))/(i/v) \;.
\end{align}

For the second step we take a closer look at the next map,~$\gamma^*$. Note that this $\gamma^*$ is actually defined by the commutative square
\[
\xymatrix@C=14pt@L=1ex{
(\tilde i_! \circ ( \tilde v^* \circ i^*) ) \circ i_!
 \ar@{==>}[rr]^-{\gamma^*}
 \ar@{=>}[d]_-{\simeq} &&
(\tilde i_! \circ ( \tilde i^* \circ v^*) ) \circ i_!
 \ar@{=>}[d]^-{\simeq} \\
(\tilde i_! \circ ( i \tilde v)^* ) \circ i_! \ar@{=>}[rr]^-{\gamma^*} &&
(\tilde i_! \circ (v \tilde i)^* ) \circ i_!
}
\]
where the bottom $\gamma^*$ is, properly speaking, the image of $\gamma$ under the pseudo-functor $\GG^{\op}\to \Span$ and the vertical maps are induced by the structure isomorphisms of the latter (see \Cref{Cons:first-embeddings}). By computing this composite, the top $\gamma^*$ turns out to be the morphism of spans with identity 2-cell components and the 1-cell component
\begin{align} \label{eq:top_gamma*}
& \big\langle \;\;
 \big\langle
 \pr_1 , \langle \tilde i \pr_2 , \pr_2 , \id \rangle \pr_2 , ((\gamma \pr_2) (i \delta)) \pr_2
 \big\rangle \pr_1 \;\;,\;\;
 \pr_2 \;\;,\;\;
 \rho
\;\; \big\rangle \\ \nonumber
& \colon (i/ (P/\tilde v))/(i/v) \longrightarrow ( i / (Q/\tilde i)) / (i/v) \;,
\end{align}
where $\pr_1$ and $\pr_2$ denote, respectively, the left and right projections of the uniquely relevant comma squares, and the target span looks as follows:
\[
{
\xymatrix@!C=16pt@R=12pt{
&&&& ( i / (Q/\tilde i)) / (i/v) \ar[dl] \ar[dddrrr] \ar@{}[rrdddd]|{\Ecell}
&&&& \\
&&& i/ (Q/ \tilde i) \ar[ddll] \ar[dr] \ar@{}[lddd]|{\Ecell}
 &&&&& \\
&&&& Q/\tilde i \ar[dl] \ar[dr] &&&& \\
& P \ar[dl] \ar[dr]^i &&
 Q \ar[dl]_v \ar[dr] \ar@{}[rr]|{\Ecell} &&
 i/v \ar[ld]_-{\tilde i} \ar[dr] &&
 i/v \ar[dl] \ar[dr]^-{\tilde i} & \\
P \ar@{..>}[rr]_-{i_!} &&
 H \ar@{..>}[rr]_-{v^*} &&
 Q \ar@{..>}[rr]_-{\tilde i^*} &&
 i/v \ar@{..>}[rr]_-{\tilde i_!} && Q
}}
\]
It is now straightforward to compose~\eqref{eq:first_four_maps} and~\eqref{eq:top_gamma*} and to see that the resulting 1-cell component is
\begin{align*}
\big\langle \;
 \big\langle \tilde v , \langle \tilde i , \Id_{i/v} , \id_{\tilde i} \rangle , \gamma \big\rangle \;,\;
 \Id_{i/v} \;,\;
 \id_{\Id_{i/v}}
\; \big\rangle
\colon (i/v) \longrightarrow ( i / (Q/\tilde i)) / (i/v)
\end{align*}
(the 2-cell components are still identities, of course).

Next, we further compose with the associator
$(\tilde i_! \circ (\tilde i^* \circ v^*)) \circ i_! \Rightarrow
(\tilde i_! \circ \tilde i^* ) \circ ( v^* \circ i_! )$.
The target span is
\[
\xymatrix@!C=16pt@R=12pt{
&&&& (i/v)/((i/v)/(i/v)) \ar[ddll] \ar[ddrr] &&&& \\
&&&&&&&& \\
&& i/v \ar[ld]_-{\tilde v} \ar[dr]^-{\tilde i} \ar@{}[rrrr]|{\oEcell{\xi}} &&&&
 (i/v)/(i/v) \ar[dl] \ar[dr] && \\
& P \ar[dl] \ar[dr]^i \ar@{}[rr]|{\oEcell{\gamma}} &&
 Q \ar[dl]_v \ar[dr] &&
 i/v \ar[ld]_-{\tilde i} \ar[dr] \ar@{}[rr]|{\oEcell{\theta}} &&
 i/v \ar[dl] \ar[dr]^-{\tilde i} & \\
P \ar@{..>}[rr]_-{i_!} &&
H \ar@{..>}[rr]_-{v^*} &&
Q \ar@{..>}[rr]_-{\tilde i^*} &&
i/v \ar@{..>}[rr]_-{\tilde i_!} && Q
}
\]
(where we have named $\xi$ and $\theta$ for later) and the resulting total composite into it is easily seen to have 1-cell component
\begin{align} \label{eq:comp_first_six}
\big\langle \;
\Id_{i/v} \;,\;
\langle \Id_{i/v} , \Id_{i/v} , \id_{\Id_{i/v}} \rangle \;,\;
\id_{\tilde i}
\;\big\rangle
\colon (i/v) \longrightarrow (i/v) / ((i/v)/(i/v))
\end{align}
and to still have trivial 2-cell components. (Note that~$\gamma$ has now vanished. This is because $\langle \tilde v , \tilde i , \gamma \rangle = \Id_{i/v}$ by the definition of the original comma square.)

For the final step, consider the last two maps.
We have a whiskered counit followed by a left unitor:
\begin{align} \label{eq:last_two_maps}
\xymatrix@L=1ex{
(\tilde i_! \circ \tilde i^* ) \circ ( v^* \circ i_! ) \ar@{=>}[r]^-{\varepsilon (v^* \circ i_!)} &
\Id_Q \circ ( v^* \circ i_! ) \ar@{=>}[r]^-{\lun} &
v^* \circ i_!
}
\end{align}
If we choose the appropriate representative for~$\varepsilon$ of the two available ones, both are given by morphisms of spans with trivial left 2-cell component but \emph{non}trivial right 2-cell component. To wit, their composite~\eqref{eq:last_two_maps} computes as follows (writing $X:= (i/v)/((i/v)/(i/v))$ to save space):
\begin{align*}
\xymatrix{
P
\ar@{=}[dd] && i/v
\ar[ll]_-{\tilde v} &&
 X
 \ar[ll]_-{\pr_1}
 \ar[rr]^-{\pr_2}
 \ar[dd]_-{f \,:=\, \langle \pr_1, \tilde i \pr_1 \pr_2 , \xi \rangle} &&
 (i/v)/(i/v) \ar[r]^-{\pr_2} \ar[d]_-{\pr_1} &
 i/v \ar[r]^-{\tilde i} &
 Q \ar@{=}[dd] \ar@{}[ddll]|{\NEcell \; \tilde i \,\theta} \\
&&&&&& i/v \ar[d]_-{\tilde i} && \\
P \ar@{=}[d] &&
i/v \ar[ll]_-{\tilde v} &&
 (i/v)/Q
 \ar[rr]^-{\pr_2}
 \ar[ll]_-{\pr_1}
 \ar[d]_-{\pr_1} &&
 Q \ar[rr] &&
 Q \ar@{=}[d] \ar@{}[dllll]|{\NEcell \; \tau} \\
P &&&& i/v \ar[llll]^-{\tilde v} \ar[rrrr]_-{\tilde i} &&&& Q
}
\end{align*}
Here $\tau$ denotes the following comma square:
\begin{align*}
\xymatrix@C=14pt@R=14pt{
& (i/v)/Q \ar[ld]_-{\pr_1} \ar[dr]^-{\pr_2} & \\
i/v \ar[rd]_-{\tilde i} \ar@{}[rr]|{\oEcell{\tau}} && Q \ar[dl] \\
&Q &
}
\end{align*}

Now we hit $X$ with~\eqref{eq:comp_first_six}, which we denote~$g\colon i/v\to X$ for short. First, it is clear that the composite 1-cell
$\pr_1 f g $
down the middle is equal to $\Id_{i/v}$. The left 2-cell of the resulting morphism of spans is obviously trivial. It remains to show that the right 2-cell is also trivial. It is as follows:
\begin{align*}
\xymatrix{
i/v
\ar[rrrr]^-{\tilde i}
\ar[d]_-{g \,=\, \langle \Id, \langle \Id, \Id, \id_{\Id} \rangle , \id_{\tilde i} \rangle }
\ar[rrd]^-{\quad\; \langle \Id, \Id, \id_{\Id} \rangle \,=:\, \Delta_{i/v}} &&&&
 Q \ar@{=}[d] \\
 X
 \ar[rr]^-{\pr_2}
 \ar[dd]_-{f \,=\, \langle \pr_1, \tilde i \pr_1 \pr_2 , \xi \rangle} &&
 (i/v)/(i/v) \ar[r]^-{\pr_2} \ar[d]_-{\pr_1} &
 i/v \ar[r]^-{\tilde i} \ar@/^5ex/[dl]
 \ar@{}[dl]|{\NEcell\;\theta} &
 Q \ar@{=}[dd] \\
&& i/v \ar[d]_-{\tilde i} && \\
 (i/v)/Q
 \ar[rr]^-{\pr_2}
 \ar[d]_-{\pr_1} &&
 Q \ar[rr] &&
 Q \ar@{=}[d] \ar@{}[dllll]|{\NEcell \; \tau} \\
i/v \ar[rrrr]_-{\tilde i} &&&& Q
}
\end{align*}
By the very definitions of~$\Delta_{i/v}$, $f$ and~$g$, we have in~$\GG$ whiskerings
\[
\theta \Delta_{i/v} = \id_{\Id_{i/v}}
\quad \quad \quad
\tau f = \xi
\quad \quad \quad
\xi g = \id_{\tilde i}
\]
from which it immediately follows that the above 2-cell is indeed the identity of~${\tilde i}$, as claimed.
\end{proof}

After having checked as in \Cref{Rem:special-mate} that, for 2-cells $\alpha\colon i\Rightarrow j$ with $i,j\in \JJ$, the two meanings of the notation $\alpha_!$ coincide, \Cref{Lem:BC-for-Span} has the following immediate corollary:

\begin{Cor} \label{Cor:update-can-pseudofun}
The pseudo-functors $(-)^*\colon \GG^{\op} \to \Span(\GG;\JJ) \leftarrow \JJ^{\co}:\!\! (-)_!$ of \Cref{Cons:first-embeddings} have the property that for every $i\in \JJ$ there is a canonical adjunction $i_!\dashv i^*$ in $\Span(\GG;\JJ)$.
\qed
\end{Cor}

The following result explains how the 2-cells of~$\Span$ are generated by the 2-cells coming from~$\GG\hook \Span$, by the counits of the adjunctions $a_!\adj a^*$ for all $a\in\JJ$ and by the structural isomorphisms.
\begin{Prop}
\label{Prop:2-cells-of-Span}%
Let $\alpha=[a,\alpha_1,\alpha_2]\colon i_! u^* \Rightarrow j_! v^*$ be an arbitrary 2-cell in~$\Span$ as displayed in~\eqref{eq:2-cell-of-Span}. Then, modulo the canonical isomorphisms $i_!u^*\stackrel{\sim}{\to}i_!\circ u^*$ and $j_!v^*\stackrel{\sim}{\to}j_!\circ v^*$ of \Cref{Rem:notation_assoc}, the following pasting in~$\Span$ is equal to~$\alpha$:
\[
\xymatrix{
G \ar@{=}[d] \ar[rr]^-{u^*} \ar@{}[rrd]|{\simeq \; \Scell\; \alpha_1^*}
&& P \ar@{=}[d] \ar[rr]^-{i_!} \ar@{}[rrd]|{\simeq \; \Scell\; (\alpha_2)_!}
&& H \ar@{=}[d]
\\
G \ar@{}[rrd]|{\Scell\;\cong} \ar@{=}[d] \ar[rr]^-{(v a)^*}
&& P \ar@{}[rrd]|{\Scell\;\cong} \ar@{=}[d] \ar[rr]^-{(j a)_!}
&& H \ar@{=}[d]
\\
G \ar[r]^-{v^*} \ar@{=}[d]
& Q \ar@{}[rrd]|{\Scell\; \varepsilon} \ar@{=}[d] \ar[r]^-{a^*}
& P \ar[r]^-{a_!}
& Q \ar@{=}[d] \ar[r]^-{j_!}
& H \ar@{=}[d]
\\
G \ar[r]_-{v^*}
& Q \ar[rr]_-{\Id}
&& Q \ar[r]_-{j_!}
& H
}
\]
in which we use the 2-functoriality of $(-)^*$ and of $(-)_!$ as in \Cref{Cons:first-embeddings}.
\end{Prop}

\begin{proof}
Exercise.
\end{proof}

\bigbreak
\section{The universal property of spans}
\label{sec:UP-Span}%
\medskip

We are now going to prove a universal property of the canonical 2-functor $(-)^*\colon \GG^{\op}\to \Span(\GG;\JJ)$.
Our result is a significant generalization and strengthening of~\cite[Prop.\,1.10]{DawsonParePronk04}. A similar result had appeared without proof in~\cite[Theorem~A2]{Hermida00}.
The first version of our universal property (\Cref{Thm:UP-Span} below) only considers pseudo-functors, and is formulated as `strictly' as possible. The second version (see \Cref{Thm:UP-PsFun-Span}) will also consider transformations between them.

\begin{Thm}
\label{Thm:UP-Span}%
Let $\GG$ and $\JJ$ be as in Hypotheses~\ref{Hyp:G_and_I_for_Span} (see also \Cref{Rem:less-hyps}). Let $\cat{C}$ be any 2-category, and let $\cat{F}\colon \GG^{\op}\to \cat{C}$ be a pseudo-functor such that
\begin{enumerate}[\rm(a)]
\item
\label{it:UP-Span-a}%
for every $i\in \JJ$, there exists in $\cat{C}$ a left adjoint $(\cat{F}i)_!$ to~$\cat{F}i$, and
\smallbreak
\item
\label{it:UP-Span-b}%
the adjunctions $(\cat{F}i)_! \dashv \cat{F}i$ satisfy base-change with respect to all Mackey squares with two parallel sides in~$\JJ$.
\end{enumerate}
Then there exists a pseudo-functor $\cat{G}\colon \Span(\GG; \JJ)\to \cat{C} $ such that the diagram
\[
\xymatrix{
 \GG^{\op} \ar[d]_-{(-)^*} \ar[r]^-{\cat{F}} & \cat{C} \\
 \Span(\GG;\JJ) \ar@{-->}[ru]_-{\cat{G}} &
 }
\]
is strictly commutative. This extension $\cat{G}$ is unique up to a unique isomorphism restricting to the identity of $\cat F$, and is entirely determined by the choice of the left adjoints, with units and counits, for all~$\cat{F}i$ (see details in \Cref{Rem:rectif-UP-Span} below). Conversely, any pseudo-functor $\cat{F}\colon \GG^{\op}\to \cat{C}$ factoring as above must enjoy the above two properties \eqref{it:UP-Span-a} and~\eqref{it:UP-Span-b}.
\end{Thm}

The extension $\cat{G}$ is given in \Cref{Cons:UP-Span} below.

\begin{Rem} \label{Rem:rectif-UP-Span}
In order to obtain \Cref{Thm:UP-Span} precisely as formulated, we must make certain choices of adjoints, in the spirit of \Mack{6}. Whenever $i\in \JJ$ is a \emph{strictly} invertible 1-cell, we assume that the left adjoint of $\cat Fi$ is $\cat F(i^{-1})$, with counit given by the composite $\cat F(i^{-1})\circ \cat F(i)\cong \cat F(ii^{-1})=\cat F(\Id) \cong \Id$ of structure morphisms of~$\cat F$, and similarly for the unit. Similarly, whenever $i=\Id_H$ is an identity we take care to choose $(\cat F\Id_H)_!:=\Id_{\cat FH}$ as the left adjoint of $\cat F(\Id_H)$, with the coherent isomorphism $\un_H\colon \Id_{\cat FH} \Rightarrow \cat F(\Id_H)$ and its inverse as our choice for the unit and counit of adjunction. This is so as to obtain \emph{strict} commutativity $\cat G \circ (-)^* = \cat F$ in the theorem, rather than just up to an isomorphism $\cat G \circ (-)^* \simeq \cat F$, and in order to simplify formulas a little. (Compare the Rectification \Cref{Thm:rectification-intro}.)
\end{Rem}

\begin{proof}[Proof of \Cref{Thm:UP-Span}.]
The last claim is clear: Any pseudo-functor of the form $\cat{F}= \cat{G}\circ (-)^*\colon \cat{G}^{\op}\to \cat{C}$ must satisfy properties \eqref{it:UP-Span-a} and~\eqref{it:UP-Span-b}, since the required adjunctions and base-change properties take place in $\Span(\GG;\JJ)$, by Proposition~\ref{Prop:first-adjoints} and Lemma~\ref{Lem:BC-for-Span}, and are preserved by the pseudo-functor~$\cat{G}$.

Uniqueness of~$\cat{G}$ is forced by its agreement with~$\cat{F}$ on 0-cells, by the property that every 1-cell in~$\Span$ is of the form $i_!\circ u^*$ (\Cref{Rem:notation_assoc}) and that every 2-cell in~$\Span$ is also controlled by data coming via the embeddings of $\GG$ and $\JJ$ into~$\Span$ (\Cref{Prop:2-cells-of-Span}).
To be more precise, assume we have two pseudo-functors $\cat G,\cat G'\colon \Span \to \cat C$ such that $\cat G\circ (-)^* = \cat F = \cat G'\circ (-)^*$. They obviously agree on objects. Moreover, we can construct a strictly invertible pseudo-natural transformation $t\colon \cat G\stackrel{\sim}{\to}\cat G'$ (see \Cref{Ter:Hom_bicats}) as follows. The component $t_G\colon \cat GG\to \cat G'G$ at an object $G$ is just the identity $\Id_{\cat FG}$.
For $i\in \JJ$, we obtain two adjunctions $\cat G (i_!)\dashv \cat G(i^*)=\cat Fi$ and $\cat G' (i_!)\dashv \cat G'(i^*)=\cat Fi$ and therefore a unique invertible 2-cell $t_i \colon \cat G' (i_!) \Rightarrow \cat G (i_!)$ identifying their units and counits.
Then we define the component $t_{i_!u^*}$ of $t$ at each 1-cell
$i_!u^* = (G \overset{u}{\leftarrow}P \overset{i}{\rightarrow} H)$ 
of $\Span$ by the pasting
\[
\vcenter { \hbox{
\xymatrix{
{\cat GG} \ar[r]^-{t_G} \ar[d]_{\cat G(i_!u^*)} &
 {\cat G' G} \ar[d]^{\cat G'(i_!u^*)} \\
{\cat G H} \ar[r]_-{t_H} \ar@{}[ur]|{\SWcell \; t_{i_!u^*}} &
 {\cat G' H}
}
}}
\quad:=\quad
\vcenter { \hbox{
\xymatrix{
& {\cat FG} \ar@{=}[r]
 \ar@/_8ex/[dd]_{\cat G(i_!u^*)}
 \ar[d]_{\cat Fu} &
 {\cat F G}
 \ar@/^8ex/[dd]^{\cat G'(i_!u^*)}
 \ar[d]^{\cat F u} & \\
\ar@{}[r]|{\cong} &
 {\cat F P}
 \ar@{=}[r]
 \ar[d]_{\cat G i_!} &
 {\cat F P}
 \ar[d]^{\cat G' i_!} &
 \ar@{}[l]|{\cong} \\
& {\cat G H}
 \ar@{=}[r]
 \ar@{}[ur]|{\SWcell \; t_i} &
 {\cat G' H} &
}
}}
\]
which makes use of the coherent structural isomorphisms of the pseudo-functors $\cat G$ and~$\cat G'$. (Here we have omitted the canonical isomorphisms $i_!u^*\cong i_!\circ u^*$.) The verification that $t=(t_G,t_{i_!u^*})$ is a pseudo-natural transformation is left as an exercise; it follows from \Cref{Prop:2-cells-of-Span} together with the uniqueness of the~$t_i$ and their compatibility with units and counits. The uniqueness of the $t_i$ also implies the claimed uniqueness of such an isomorphism $t\colon\cat G\stackrel{\sim}{\to} \cat G'$. (More details on such extended transformations will be given in the proof of \Cref{Thm:UP-PsFun-Span}.)

We now turn to the existence of~$\cat G$, given a pseudo-functor $\cat{F}\colon \GG^{\op}\to \cat{C}$ satisfying \eqref{it:UP-Span-a} and~\eqref{it:UP-Span-b}. As observed in Remark~\ref{Rem:pseudo-func-of-adjoints}, the adjunctions $(\cat{F}i)_! \dashv \cat{F}i$ provided in~\eqref{it:UP-Span-a} assemble into a pseudo-functor $\cat{F}_!\colon \JJ^\co\to \cat{C}$ which has the same values as $\cat{F}$ on 0-cells, is given by the chosen adjoints $(\cat{F}i)_!$ on 1-cells~$i$ and by the mates $\alpha_!$ on 2-cells~$\alpha$ (which involve the chosen units and counits).

It remains to show that the base-change property~\eqref{it:UP-Span-b} allows us to `glue' $\cat{F}$ and $\cat{F}_!$ into a pseudo-functor $\cat{G}\colon \Span(\GG;\JJ)\to \cat{C}$ defined as follows.

\begin{Cons}
\label{Cons:UP-Span}%
On objects, $\cat{G}$ must be the same as~$\cat{F}$. On 1-cells, we set
\[
\cat{G}(i_!u^*):= \cat{F}_!(i) \circ \cat{F}(u) = (\cat{F} i)_! \circ (\cat{F}u)\,,
\]
with the choices of \Cref{Rem:rectif-UP-Span} when~$i$ happens to be invertible.

For a 2-cell $[a,\alpha_1,\alpha_2]$ of $\Span(\GG;\JJ)$ represented by a diagram
\begin{align*}
\xymatrix{
G \ar@{=}[d] \ar@{}[rrd]|{\SEcell\,\alpha_1} &&
 P \ar[ll]_-{u} \ar[rr]^-{i} \ar[d]^a \ar@{}[rrd]|{\NEcell\,\alpha_2} &&
 H \ar@{=}[d] \\
G &&
 Q \ar[ll]^-{v} \ar[rr]_-{j} &&
 H
}
\end{align*}
we define its image to be the following pasting in~$\cat{C}$ (compare \Cref{Prop:2-cells-of-Span}):
\begin{equation}
\label{eq:pasting-in-UP-Span}%
\quad \cat{G}([a,\alpha_1,\alpha_2]) \; :=\quad
\vcenter{\xymatrix{
\cat{F} G \ar@{=}[d] \ar[rr]^-{\cat{F} u} \ar@{}[rrd]|{\Scell\; \cat{F} \alpha_1}
&& \cat{F} P \ar@{=}[d] \ar[rr]^-{(\cat{F} i)_!} \ar@{}[rrd]|{\Scell\; (\cat{F} \alpha_2)_!}
&& \cat{F} H \ar@{=}[d]
\\
\cat{F} G \ar@{}[rrd]|{\Scell\;\simeq} \ar@{=}[d] \ar[rr]^-{\cat{F} (v a)}
&& \cat{F} P \ar@{}[rrd]|{\Scell\;\simeq} \ar@{=}[d] \ar[rr]^-{(\cat{F} j a)_!}
&& \cat{F} H \ar@{=}[d]
\\
\cat{F} G \ar[r]^-{\cat{F} v} \ar@{=}[d]
& \cat{F} Q \ar@{}[rrd]|{\Scell\; \varepsilon} \ar@{=}[d] \ar[r]^-{\cat{F} a}
& \cat{F} P \ar[r]^-{(\cat{F} a)_!}
& \cat{F} Q \ar@{=}[d] \ar[r]^-{(\cat{F} j)_!}
& \cat{F} H \ar@{=}[d]
\\
\cat{F} G \ar[r]_-{\cat{F} v}
& \cat{F} Q \ar[rr]_-{\Id}
&& \cat{F} Q \ar[r]_-{(\cat{F} j)_!}
& \cat{F} H
}}
\end{equation}
The latter uses the counit~$\eps$ of the adjunction $(\cat{F} a)_!\dashv \cat{F} a$ as well as the coherent isomorphisms $\fun\inv\colon \cat{F}(v a)\Rightarrow (\cat{F}a)(\cat{F}v)$ and $\fun\inv\colon\cat{F} (j a)_!\Rightarrow \cat{F}(j)_! \cat{F}(a)_!$ --- the latter obtained from those of $\cat{F}$ by taking mates.
Given a composite of two spans
\begin{align} \label{Eq:hor_composite}
\vcenter{\xymatrix{
G
\ar@/^6ex/@{..>}[rrrr]^-{j_!v^*\,\circ\,i_!u^* \,=\, (j\tilde i)_!(u\tilde v)^*}
\ar@/_6ex/@{..>}[ddrr]_-{i_!u^*} &&
 i/v
 \ar[dl]_-{\tilde v}
 \ar[dr]^-{\tilde i}
 \ar[ll]_-{u\tilde v}
 \ar[rr]^-{j\tilde i} && K
\\
& P
\ar[ul]^-u \ar[dr]_-{i}
\ar@{}[rr]|{\oEcell{\gamma}} &&
 Q
 \ar[ur]_-{j}
 \ar[dl]^-{v} &
\\
 && H
 \ar@/_6ex/@{..>}[rruu]_-{\ j_!v^*} &&
}}
\end{align}
we define the structure isomorphism $\fun_\cat G\colon \cat{G}(j_!v^*)\circ \cat{G} (i_!u^*) \stackrel{\sim}{\to} \cat{G}(j_!v^* \circ i_!u^*)$ by
\begin{align} \label{Eq:fun_for_G}
\vcenter{\xymatrix{
&&&& \\
\cat{F} G
\ar@{..>}@/_7ex/[rrdd]_-{\cat{G}(i_!u^*)}
\ar[rd]_-{\cat{F} u}
\ar[rr]^-{\cat{F} u\tilde v}
\ar@{..>}@/^7ex/[rrrr]^-{\cat{G}( (j\tilde i)_!(u\tilde v)^*)} &&
 \cat{F} i/v
 \ar[rr]^-{(\cat{F} j\tilde i)_! }
 \ar[dr]^-{(\cat{F} \tilde i)_!} &&
 \cat{F} K \\
& \cat{F} P
\ar@{}[rr]|{\Ncell\;((\cat{F}\gamma)_!)\inv}
\ar@{}[u]|{\Ncell\,\simeq}
\ar[ur]^-{\cat{F} \tilde v}
 \ar[rd]_-{(\cat{F} i)_!} &&
 \cat{F} Q
 \ar@{}[u]|{\simeq \; \Ncell}
 \ar[ru]_(.4){(\cat{F} j)_!} & \\
 && \cat{F} H
 \ar[ru]_-{\cat{F} v}
 \ar@{..>}@/_7ex/[uurr]_-{\cat{G}(j_!v^*)} &&
}}
\end{align}
which is well-defined because $\cat{F}$ satisfies base-change with respect to the square~$\gamma$.
As for the unitors $\un_\cat G\colon \Id_{\cat GG}\stackrel{\sim}{\to}\cat G(\Id_G)$, since $\cat G(\Id_G) = \cat F(\Id_G)_! \cat F (\Id_G)= \cat F (\Id_G)$ we may simply take them to be identical to the unitors of~$\cat F$.
\end{Cons}

In order to prove the theorem it remains to verify that $\cat{G}$ is a well-defined pseudo-functor and that it extends~$\cat F$.
The former is precisely Lemma~\ref{Lem:Gwelldef} below.
For the latter, it is immediate from the construction that $\cat G\circ (-)^*$ agrees with $\cat F$ on 0-cells and 1-cells, and that for a 2-cell $\alpha \colon u\Rightarrow v\colon H\to G$ the image $\cat G(\alpha^*)= \cat G([\Id_H,\alpha,\id_{\Id_H}])$ is equal to the composite
\[
\xymatrix@L=1ex{
{\cat Fu} \ar@{=>}[r]^-{\cat F \alpha} &
 {\cat Fv = \cat F(v \circ \Id)} \ar@{=>}[r]^-{\fun^{-1}} &
 {\cat F (\Id_H) \circ \cat F v} \ar@{=>}[rr]^-{\un^{-1} \,\circ\, \cat Fv} &&
 {\Id_{\cat FH} \circ \cat Fv = \cat F v}
}
\]
(just read it off~\eqref{eq:pasting-in-UP-Span}, using our careful choice of adjunctions for the identity 1-cells).
But this composite is equal to $\cat F\alpha$, as one sees by a direct application of one of the three pseudo-functor axioms. Hence $\cat G\circ (-)^*=\cat F$.
\end{proof}

\begin{Lem} \label{Lem:Gwelldef}
\Cref{Cons:UP-Span} yields a pseudo-functor $\cat G\colon \Span(\GG;\JJ) \to \cat C$.
\end{Lem}

\begin{proof}
The proof is rather straightforward but lengthy and not so easy to write down explicitly, hence it is included here for completeness. It will also be a good occasion for introducing string diagrams. See \Cref{sec:string_diagrams} and~\ref{sec:strings-here} for a quick tutorial, and let us recall here that in our string diagrams 1-cells are always oriented left-to-right whereas 2-cells are always oriented top-to-bottom:
\[
\xymatrix{
\ar@{-->}[r] \ar@{==>}[d] & \\
&
}
\]

To begin, we need to show that $\cat G$ is well-defined on 2-cells. Suppose we have two diagrams in the 2-category~$\cat G$
\begin{equation*}
\vcenter{\xymatrix{
G \ar@{=}[d] \ar@{}[rrd]|{\SEcell\,\alpha_1}
&& P \ar[ll]_-{u} \ar[rr]^-{i} \ar[d]^a \ar@{}[rrd]|{\NEcell\,\alpha_2}
&& H \ar@{=}[d]
\\
G
&& Q \ar[ll]^-{v} \ar[rr]_-{j}
&& H
}}
\quad \overset{\underset{}{\varphi}}{\Rightarrow} \quad
\vcenter{\xymatrix{
G \ar@{=}[d] \ar@{}[rrd]|{\SEcell\,\beta_1}
&& P \ar[ll]_-{u} \ar[rr]^-{i} \ar[d]^b \ar@{}[rrd]|{\NEcell\,\beta_2}
&& H \ar@{=}[d]
\\
G
&& Q \ar[ll]^-{v} \ar[rr]_-{j}
&& H
}}
\end{equation*}
representing the same 2-cell of $\Span$, as testified by the existence of an isomorphism $\varphi\colon a \Rightarrow b$ satisfying $(v \varphi) \alpha_1 = \beta_1$ and $\alpha_2 (j \varphi^{-1})=\beta_2$. In string diagrams, the latter yield the following equations in~$\cat C$ after applying the (1-contravariant and 2-covariant!)~$\cat F$:
\begin{equation} \label{eq:pseudofun-setup}
\vcenter { \hbox{
\psfrag{A}[Bc][Bc]{\scalebox{1}{\scriptsize{$\cat Fv$}}}
\psfrag{B}[Bc][Bc]{\scalebox{1}{\scriptsize{$\cat Fb$}}}
\psfrag{C}[Bc][Bc]{\scalebox{1}{\scriptsize{$\;\;\;\;\;\cat F (vb)$}}}
\psfrag{D}[Bc][Bc]{\scalebox{1}{\scriptsize{$\cat F u$}}}
\psfrag{F}[Bc][Bc]{\scalebox{1}{\scriptsize{$\fun^{-1}$\;\;\;\;}}}
\psfrag{G}[Bc][Bc]{\scalebox{1}{\scriptsize{$\cat F \beta_1$}}}
\includegraphics[scale=.4]{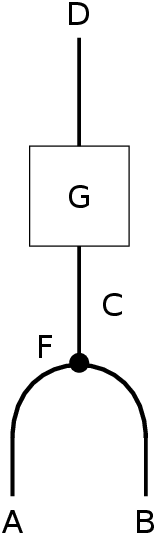}
}}
\;\;\;=\;\;\;
\vcenter { \hbox{
\psfrag{A}[Bc][Bc]{\scalebox{1}{\scriptsize{$\cat Fv$}}}
\psfrag{B}[Bc][Bc]{\scalebox{1}{\scriptsize{$\cat Fb$}}}
\psfrag{C}[Bc][Bc]{\scalebox{1}{\scriptsize{$\;\;\cat Fa$}}}
\psfrag{D}[Bc][Bc]{\scalebox{1}{\scriptsize{$\cat F u$}}}
\psfrag{F}[Bc][Bc]{\scalebox{1}{\scriptsize{$\cat F \varphi$}}}
\psfrag{G}[Bc][Bc]{\scalebox{1}{\scriptsize{$\cat F \alpha_1$}}}
\psfrag{H}[Bc][Bc]{\scalebox{1}{\scriptsize{$\fun^{-1}\;\;\;\;$}}}
\includegraphics[scale=.4]{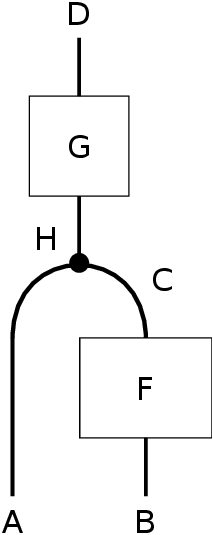}
}}
\quad\quad\quad\quad
\vcenter { \hbox{
\psfrag{A}[Bc][Bc]{\scalebox{1}{\scriptsize{$\cat Fj$}}}
\psfrag{B}[Bc][Bc]{\scalebox{1}{\scriptsize{$\cat Fb$}}}
\psfrag{C}[Bc][Bc]{\scalebox{1}{\scriptsize{$\;\;\;\;\;\cat F (jb)$}}}
\psfrag{D}[Bc][Bc]{\scalebox{1}{\scriptsize{$\cat F i$}}}
\psfrag{F}[Bc][Bc]{\scalebox{1}{\scriptsize{$\fun$\;\;}}}
\psfrag{G}[Bc][Bc]{\scalebox{1}{\scriptsize{$\cat F \beta_2$}}}
\includegraphics[scale=.4]{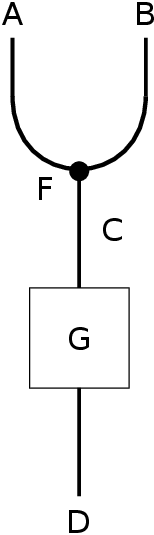}
}}
\;\;\;=\;
\vcenter { \hbox{
\psfrag{A}[Bc][Bc]{\scalebox{1}{\scriptsize{$\cat Fj$}}}
\psfrag{B}[Bc][Bc]{\scalebox{1}{\scriptsize{$\cat Fb$}}}
\psfrag{C}[Bc][Bc]{\scalebox{1}{\scriptsize{$\cat F a$}}}
\psfrag{D}[Bc][Bc]{\scalebox{1}{\scriptsize{$\cat F i$}}}
\psfrag{F}[Bc][Bc]{\scalebox{1}{\scriptsize{$\cat F \varphi^{-1}$}}}
\psfrag{G}[Bc][Bc]{\scalebox{1}{\scriptsize{$\cat F \alpha_2$}}}
\psfrag{H}[Bc][Bc]{\scalebox{1}{\scriptsize{$\fun\;\;$}}}
\includegraphics[scale=.4]{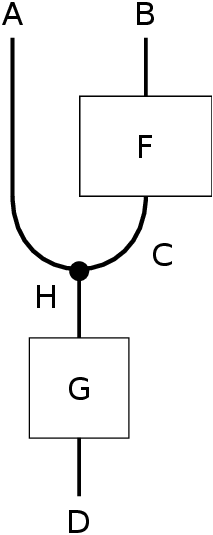}
}}
\end{equation}
Here we have also used the naturality of~$\fun$, see \Cref{Exa:strings-for-natural-fun}.

We must show that $\cat G([a,\alpha_1,\alpha_2])$ and $\cat G([b,\beta_1,\beta_2])$ are equal. In strings, the construction~\eqref{eq:pasting-in-UP-Span} of these 2-cells of $\cat C$ takes the following form:
\[
\cat G([a,\alpha_1,\alpha_2])
\;=\quad
\vcenter { \hbox{
\psfrag{A}[Bc][Bc]{\scalebox{1}{\scriptsize{$\cat Fu$}}}
\psfrag{B}[Bc][Bc]{\scalebox{1}{\scriptsize{$(\cat F i)_!$}}}
\psfrag{C}[Bc][Bc]{\scalebox{1}{\scriptsize{$\cat F (va)$\;\;}}}
\psfrag{D}[Bc][Bc]{\scalebox{1}{\scriptsize{\;\;\;$\cat F (ja)_!$}}}
\psfrag{U}[Bc][Bc]{\scalebox{1}{\scriptsize{$\cat F v$}}}
\psfrag{V}[Bc][Bc]{\scalebox{1}{\scriptsize{$(\cat F j)_!$}}}
\psfrag{R}[Bc][Bc]{\scalebox{1}{\scriptsize{$\cat F a$}}}
\psfrag{T}[Bc][Bc]{\scalebox{1}{\scriptsize{$\;\;(\!\cat F a)_!$}}}
\psfrag{F}[Bc][Bc]{\scalebox{1}{\scriptsize{$\cat F \alpha_1$}}}
\psfrag{G}[Bc][Bc]{\scalebox{1}{\scriptsize{$(\cat F \alpha_2)_!$}}}
\psfrag{H}[Bc][Bc]{\scalebox{1}{\scriptsize{$\;\;\fun^{\!\!\!-\!1}$}}}
\psfrag{L}[Bc][Bc]{\scalebox{1}{\scriptsize{$\;\fun^{\!\!\!-\!1}$}}}
\psfrag{E}[Bc][Bc]{\scalebox{1}{\scriptsize{$\varepsilon$}}}
\includegraphics[scale=.4]{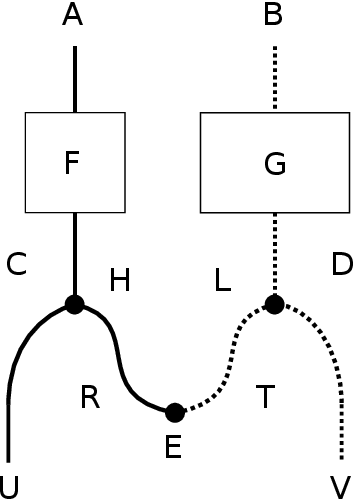}
}}
\quad\; \stackrel{\textrm{?}}{=} \quad\;
\vcenter { \hbox{
\psfrag{A}[Bc][Bc]{\scalebox{1}{\scriptsize{$\cat Fu$}}}
\psfrag{B}[Bc][Bc]{\scalebox{1}{\scriptsize{$(\cat F i)_!$}}}
\psfrag{C}[Bc][Bc]{\scalebox{1}{\scriptsize{$\cat F (vb)$\;\;}}}
\psfrag{D}[Bc][Bc]{\scalebox{1}{\scriptsize{\;\;$\cat F (jb)$}}}
\psfrag{U}[Bc][Bc]{\scalebox{1}{\scriptsize{$\cat F v$}}}
\psfrag{V}[Bc][Bc]{\scalebox{1}{\scriptsize{$(\cat F j)_!$}}}
\psfrag{R}[Bc][Bc]{\scalebox{1}{\scriptsize{$\cat F b$}}}
\psfrag{T}[Bc][Bc]{\scalebox{1}{\scriptsize{$\;\;(\!\cat F b)_!$}}}
\psfrag{F}[Bc][Bc]{\scalebox{1}{\scriptsize{$\cat F \beta_1$}}}
\psfrag{G}[Bc][Bc]{\scalebox{1}{\scriptsize{$(\cat F \beta_2)_!$}}}
\psfrag{H}[Bc][Bc]{\scalebox{1}{\scriptsize{$\;\;\fun^{\!\!\!-\!1}$}}}
\psfrag{L}[Bc][Bc]{\scalebox{1}{\scriptsize{$\;\fun^{\!\!\!-\!1}$}}}
\psfrag{E}[Bc][Bc]{\scalebox{1}{\scriptsize{$\varepsilon$}}}
\includegraphics[scale=.4]{anc/halfpseudofun-firstrepr.eps}
}}
\quad=\;
\cat G([b,\beta_1,\beta_2])
\]
For clarity, in this proof the strings corresponding to the chosen left adjoints $(\cat F i)_!$ of all $i\in \JJ$ appear as dotted. Here $\varepsilon$ denotes the counit of one of these adjunctions. Zooming in, we see the other units and counits hiding within the mates $(\cat F\alpha_2)_!$ and $(\fun_{\cat F_!})^{-1} = (\fun_{\cat F})_!$, which are defined as
\[
\vcenter { \hbox{
\psfrag{A}[Bc][Bc]{\scalebox{1}{\scriptsize{$(\cat F i)_!$}}}
\psfrag{B}[Bc][Bc]{\scalebox{1}{\scriptsize{$(\cat F ja)_!$}}}
\psfrag{F}[Bc][Bc]{\scalebox{1}{\scriptsize{$(\cat F \alpha_2)_!$}}}
\includegraphics[scale=.4]{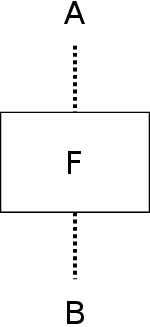}
}}
\;=\;
\vcenter { \hbox{
\psfrag{A}[Bc][Bc]{\scalebox{1}{\scriptsize{$\cat F i$}}}
\psfrag{B}[Bc][Bc]{\scalebox{1}{\scriptsize{\;\;\;$\cat F ja$}}}
\psfrag{F}[Bc][Bc]{\scalebox{1}{\scriptsize{$\cat F \alpha_2$}}}
\psfrag{N}[Bc][Bc]{\scalebox{1}{\scriptsize{$\eta$}}}
\psfrag{E}[Bc][Bc]{\scalebox{1}{\scriptsize{$\varepsilon$}}}
\includegraphics[scale=.4]{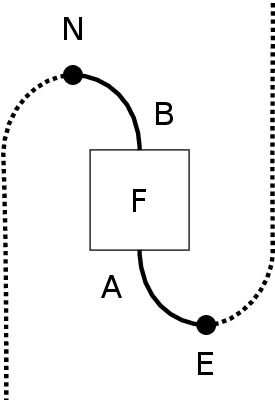}
}}
\quad\quad\quad\quad\quad\quad
\vcenter { \hbox{
\psfrag{A}[Bc][Bc]{\scalebox{1}{\scriptsize{$(\cat F a)_!\;\;\;$}}}
\psfrag{B}[Bc][Bc]{\scalebox{1}{\scriptsize{$\;\;\;(\cat F j)_!$}}}
\psfrag{C}[Bc][Bc]{\scalebox{1}{\scriptsize{$(\cat F ja)_!$}}}
\psfrag{F}[Bc][Bc]{\scalebox{1}{\scriptsize{$\fun_{\cat F_!}^{-1}$\;\;\;\;\;\;\;}}}
\includegraphics[scale=.4]{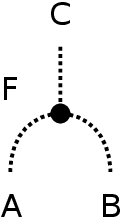}
}}
\;=\;
\vcenter { \hbox{
\psfrag{A}[Bc][Bc]{\scalebox{1}{\scriptsize{$\;\;\cat F j$}}}
\psfrag{B}[Bc][Bc]{\scalebox{1}{\scriptsize{$\;\;\cat F a$}}}
\psfrag{C}[Bc][Bc]{\scalebox{1}{\scriptsize{\;\;\;\;$\cat F ja$}}}
\psfrag{F}[Bc][Bc]{\scalebox{1}{\scriptsize{$\fun_{\cat F}$\;\;}}}
\psfrag{N}[Bc][Bc]{\scalebox{1}{\scriptsize{$\eta$}}}
\psfrag{E}[Bc][Bc]{\scalebox{1}{\scriptsize{$\varepsilon$}}}
\includegraphics[scale=.4]{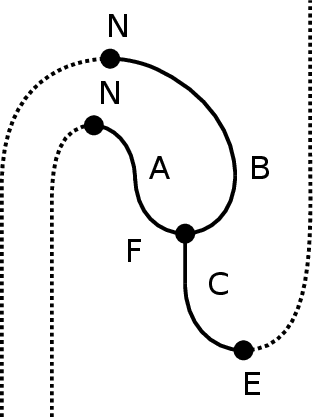}
}}
\]
and similarly for the other ones.

We can now calculate as follows, using the definitions and the adjunctions (and, tacitly, the exchange law to slide blocks up and down -- see~\eqref{Exa:strings-for-exchange}):
\begin{align} \label{eq:def-G-2-cell-strings}
\cat G([a,\alpha_1,\alpha_2])
\;=\;\quad
\vcenter { \hbox{
\psfrag{A}[Bc][Bc]{\scalebox{1}{\scriptsize{$\cat F u$}}}
\psfrag{B}[Bc][Bc]{\scalebox{1}{\scriptsize{$(\cat F i)_!$}}}
\psfrag{C}[Bc][Bc]{\scalebox{1}{\scriptsize{\;\;\;$\cat F ja$}}}
\psfrag{C'}[Bc][Bc]{\scalebox{1}{\scriptsize{$(\cat F ja)_!$}}}
\psfrag{D}[Bc][Bc]{\scalebox{1}{\scriptsize{$\cat F j$}}}
\psfrag{E}[Bc][Bc]{\scalebox{1}{\scriptsize{$\cat F a$}}}
\psfrag{R}[Bc][Bc]{\scalebox{1}{\scriptsize{$\cat F va$\;\;}}}
\psfrag{S}[Bc][Bc]{\scalebox{1}{\scriptsize{$\cat F a$\;}}}
\psfrag{S'}[Bc][Bc]{\scalebox{1}{\scriptsize{$(\cat F a)_!$\;}}}
\psfrag{T}[Bc][Bc]{\scalebox{1}{\scriptsize{$\cat F v$}}}
\psfrag{U}[Bc][Bc]{\scalebox{1}{\scriptsize{\;\;\;\;$(\cat F j)_!$}}}
\psfrag{F}[Bc][Bc]{\scalebox{1}{\scriptsize{$\cat F \alpha_1$}}}
\psfrag{G}[Bc][Bc]{\scalebox{1}{\scriptsize{$\cat F \alpha_2$}}}
\includegraphics[scale=.4]{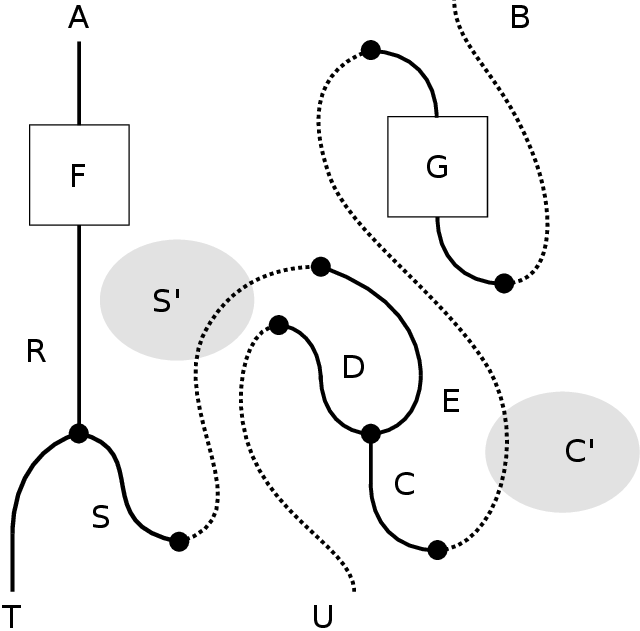}
}}
\;\stackrel{2\times\textrm{(\ref{Exa:strings-for-adjoints})}}{=}\;
\vcenter { \hbox{
\psfrag{A}[Bc][Bc]{\scalebox{1}{\scriptsize{$\cat F u$}}}
\psfrag{B}[Bc][Bc]{\scalebox{1}{\scriptsize{$(\cat F i)_!$}}}
\psfrag{C}[Bc][Bc]{\scalebox{1}{\scriptsize{\;\;\;$\cat F ja$}}}
\psfrag{D}[Bc][Bc]{\scalebox{1}{\scriptsize{$\cat F j$}}}
\psfrag{E}[Bc][Bc]{\scalebox{1}{\scriptsize{\;\;$\cat F a$}}}
\psfrag{R}[Bc][Bc]{\scalebox{1}{\scriptsize{$\cat F va$\;\;}}}
\psfrag{S}[Bc][Bc]{\scalebox{1}{\scriptsize{$\cat F a$\;}}}
\psfrag{T}[Bc][Bc]{\scalebox{1}{\scriptsize{$\cat F v$}}}
\psfrag{U}[Bc][Bc]{\scalebox{1}{\scriptsize{\;\;\;\;$(\cat F j)_!$}}}
\psfrag{F}[Bc][Bc]{\scalebox{1}{\scriptsize{$\cat F \alpha_1$}}}
\psfrag{G}[Bc][Bc]{\scalebox{1}{\scriptsize{$\cat F \alpha_2$}}}
\includegraphics[scale=.4]{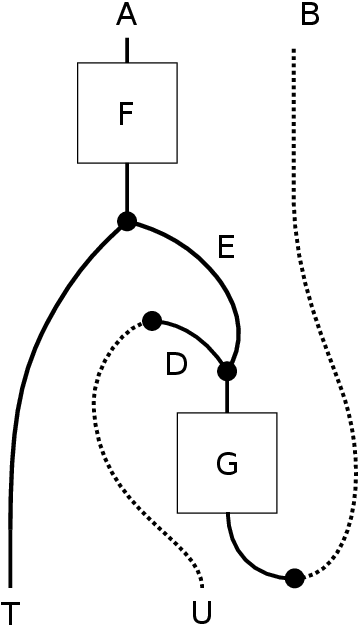}
}}
\end{align}
\[
\;=\;
\vcenter { \hbox{
\psfrag{A}[Bc][Bc]{\scalebox{1}{\scriptsize{$\cat F u$}}}
\psfrag{B}[Bc][Bc]{\scalebox{1}{\scriptsize{$(\cat F i)_!$}}}
\psfrag{C}[Bc][Bc]{\scalebox{1}{\scriptsize{$\cat F b$}}}
\psfrag{E}[Bc][Bc]{\scalebox{1}{\scriptsize{\;\;$\cat F a$}}}
\psfrag{T}[Bc][Bc]{\scalebox{1}{\scriptsize{$\cat F v$}}}
\psfrag{U}[Bc][Bc]{\scalebox{1}{\scriptsize{\;\;\;\;$(\cat F j)_!$}}}
\psfrag{F}[Bc][Bc]{\scalebox{1}{\scriptsize{$\cat F \alpha_1$}}}
\psfrag{F1}[Bc][Bc]{\scalebox{1}{\scriptsize{$\cat F \varphi$}}}
\psfrag{F2}[Bc][Bc]{\scalebox{1}{\scriptsize{$\cat F \varphi^{-1}$}}}
\psfrag{G}[Bc][Bc]{\scalebox{1}{\scriptsize{$\cat F \alpha_2$}}}
\includegraphics[scale=.4]{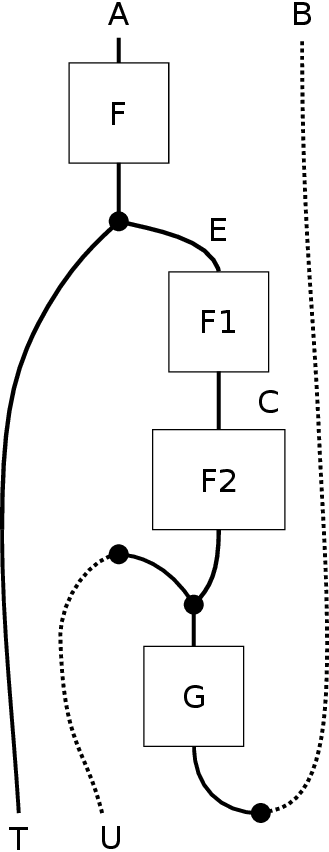}
}}
\stackrel{\textrm{\eqref{eq:pseudofun-setup}}}{=}
\vcenter { \hbox{
\psfrag{A}[Bc][Bc]{\scalebox{1}{\scriptsize{$\cat F u$}}}
\psfrag{B}[Bc][Bc]{\scalebox{1}{\scriptsize{$(\cat F i)_!$}}}
\psfrag{C}[Bc][Bc]{\scalebox{1}{\scriptsize{\;\;\;$\cat F jb$}}}
\psfrag{D}[Bc][Bc]{\scalebox{1}{\scriptsize{$\cat F j$}}}
\psfrag{E}[Bc][Bc]{\scalebox{1}{\scriptsize{\;\;$\cat F b$}}}
\psfrag{R}[Bc][Bc]{\scalebox{1}{\scriptsize{$\cat F vb$\;\;}}}
\psfrag{S}[Bc][Bc]{\scalebox{1}{\scriptsize{$\cat F b$\;}}}
\psfrag{T}[Bc][Bc]{\scalebox{1}{\scriptsize{$\cat F v$}}}
\psfrag{U}[Bc][Bc]{\scalebox{1}{\scriptsize{\;\;\;\;$(\cat F j)_!$}}}
\psfrag{F}[Bc][Bc]{\scalebox{1}{\scriptsize{$\cat F \beta_1$}}}
\psfrag{G}[Bc][Bc]{\scalebox{1}{\scriptsize{$\cat F \beta_2$}}}
\includegraphics[scale=.4]{anc/halfpseudofun-welldef-2.eps}
}}
\stackrel{2\times\textrm{\eqref{Exa:strings-for-adjoints}}}{=}
\vcenter { \hbox{
\psfrag{A}[Bc][Bc]{\scalebox{1}{\scriptsize{$\cat F u$}}}
\psfrag{B}[Bc][Bc]{\scalebox{1}{\scriptsize{$(\cat F i)_!$}}}
\psfrag{C}[Bc][Bc]{\scalebox{1}{\scriptsize{\;\;\;$\cat F jb$}}}
\psfrag{C'}[Bc][Bc]{\scalebox{1}{\scriptsize{$(\cat F jb)_!$}}}
\psfrag{D}[Bc][Bc]{\scalebox{1}{\scriptsize{$\cat F j$}}}
\psfrag{E}[Bc][Bc]{\scalebox{1}{\scriptsize{$\cat F b$}}}
\psfrag{R}[Bc][Bc]{\scalebox{1}{\scriptsize{$\cat F vb$\;\;}}}
\psfrag{S}[Bc][Bc]{\scalebox{1}{\scriptsize{$\cat F b$\;}}}
\psfrag{S'}[Bc][Bc]{\scalebox{1}{\scriptsize{$(\cat F b)_!$\;}}}
\psfrag{T}[Bc][Bc]{\scalebox{1}{\scriptsize{$\cat F v$}}}
\psfrag{U}[Bc][Bc]{\scalebox{1}{\scriptsize{\;\;\;\;$(\cat F j)_!$}}}
\psfrag{F}[Bc][Bc]{\scalebox{1}{\scriptsize{$\cat F \beta_1$}}}
\psfrag{G}[Bc][Bc]{\scalebox{1}{\scriptsize{$\cat F \beta_2$}}}
\includegraphics[scale=.4]{anc/halfpseudofun-welldef-1.eps}
}}
\]
\[
=\; \cat G([b,\beta_1,\beta_2])\,.
\]
Thus $\cat G$ is well-defined on 2-cells. Here, we began highlighting parts of the strings to help the reader focus on the parts that are being changed at the given step in the series of equalities. This visual aide has no further mathematical meaning.

Now it remains to verify that the structure maps $\un_\cat G$ and $\fun_\cat G$ satisfy the axioms of a pseudo-functor, whose string form is recalled in \Cref{Exa:strings-for-fun}. Recall that while the target $\cat C$ is a strict 2-category the source $\Span$ is not, hence to be precise we may want to write its unitors and associators explicitly.

Given two composable spans $G\loto{u}P\oto{i}H$ and $H\loto{v}Q\oto{j}K$ and their composite $G\loto{u\tilde{v}}(i/v)\oto{j\tilde{i}}K$ as in~\eqref{Eq:hor_composite}, the map $\fun_\cat G$ of~\eqref{Eq:fun_for_G} takes the following stringy form. Actually, it is more convenient to write $\fun_{\cat G}^{-1}$:
\begin{align} \label{Eq:strings_fun_for_G}
\vcenter { \hbox{
\psfrag{A}[Bc][Bc]{\scalebox{1}{\scriptsize{$\cat G (i_!u^*)\;\;\;\;\;\;$}}}
\psfrag{B}[Bc][Bc]{\scalebox{1}{\scriptsize{$\;\;\;\;\;\;\cat G( j_!v^*)$}}}
\psfrag{C}[Bc][Bc]{\scalebox{1}{\scriptsize{$\cat G((j_!v^*)\,(i_!u^*))$}}}
\psfrag{F}[Bc][Bc]{\scalebox{1}{\scriptsize{$\fun^{-1}_{\cat G}$\;\;\;\;}}}
\includegraphics[scale=.4]{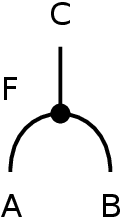}
}}
\quad=\quad
\vcenter { \hbox{
\psfrag{A}[Bc][Bc]{\scalebox{1}{\scriptsize{$\cat Fu$}}}
\psfrag{B}[Bc][Bc]{\scalebox{1}{\scriptsize{$\cat Fv$}}}
\psfrag{C}[Bc][Bc]{\scalebox{1}{\scriptsize{$\cat F(u\tilde v)$}}}
\psfrag{A1}[Bc][Bc]{\scalebox{1}{\scriptsize{$(\cat Fi)_!$}}}
\psfrag{B1}[Bc][Bc]{\scalebox{1}{\scriptsize{$(\cat F j)_!$}}}
\psfrag{C1}[Bc][Bc]{\scalebox{1}{\scriptsize{$(\cat Fj \tilde i)_!$}}}
\psfrag{F}[Bc][Bc]{\scalebox{1}{\scriptsize{$\fun^{-1}_{\cat F}$}}}
\psfrag{F1}[Bc][Bc]{\scalebox{1}{\scriptsize{$\fun^{-1}_{\cat F_!}$}}}
\psfrag{G}[Bc][Bc]{\scalebox{1}{\scriptsize{$(\cat F\gamma)_!$}}}
\includegraphics[scale=.4]{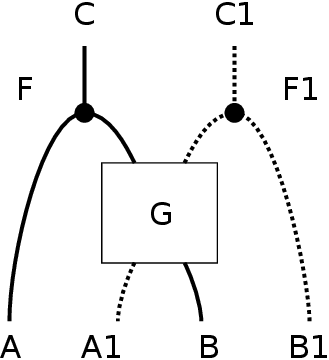}
}}
\;=\;
\vcenter { \hbox{
\psfrag{A}[Bc][Bc]{\scalebox{1}{\scriptsize{$\cat Fu$}}}
\psfrag{B}[Bc][Bc]{\scalebox{1}{\scriptsize{$\cat Fv$}}}
\psfrag{C}[Bc][Bc]{\scalebox{1}{\scriptsize{$\cat F(u\tilde v)$}}}
\psfrag{A1}[Bc][Bc]{\scalebox{1}{\scriptsize{$(\cat Fi)_!$}}}
\psfrag{B1}[Bc][Bc]{\scalebox{1}{\scriptsize{$(\cat F j)_!$}}}
\psfrag{C1}[Bc][Bc]{\scalebox{1}{\scriptsize{$(\cat Fj \tilde i)_!$}}}
\psfrag{U}[Bc][Bc]{\scalebox{1}{\scriptsize{$\;\;\cat Fj$}}}
\psfrag{V}[Bc][Bc]{\scalebox{1}{\scriptsize{$\cat F\tilde i$}}}
\psfrag{W}[Bc][Bc]{\scalebox{1}{\scriptsize{$\cat Fj\tilde i$}}}
\psfrag{R}[Bc][Bc]{\scalebox{1}{\scriptsize{$\cat F i$\;\;}}}
\psfrag{S}[Bc][Bc]{\scalebox{1}{\scriptsize{$\cat F\tilde v$}}}
\psfrag{T}[Bc][Bc]{\scalebox{1}{\scriptsize{$\cat F\tilde i$}}}
\psfrag{G}[Bc][Bc]{\scalebox{1}{\scriptsize{$\cat F\gamma$}}}
\includegraphics[scale=.4]{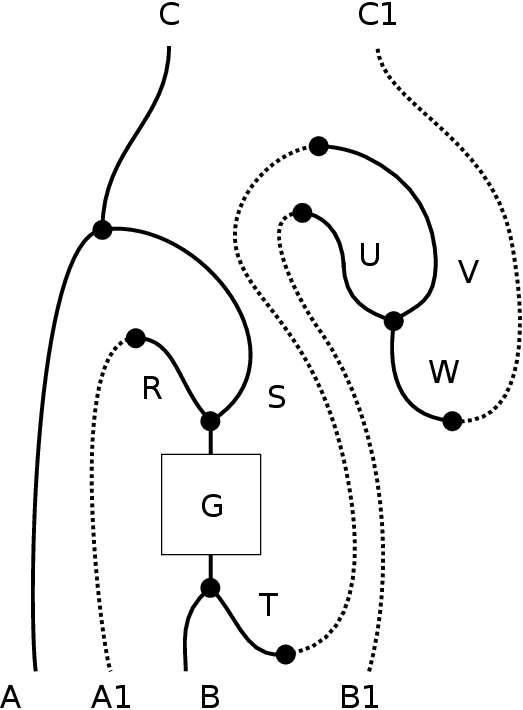}
}}
\end{align}
Let us begin with checking the left unit identity, namely:
\begin{equation}
\label{Eq:left-unit-for-G}%
\vcenter {\hbox{
\psfrag{F}[Bc][Bc]{\scalebox{1}{\scriptsize{$\;\;\un^{-1}$}}}
\psfrag{G}[Bc][Bc]{\scalebox{1}{\scriptsize{$\;\;\;\;\;\fun_{\cat G}\inv$}}}
\psfrag{H}[Bc][Bc]{\scalebox{1}{\scriptsize{$\;\;\;\;\;\;\;\;\;\cat G(\lun^{-1})$}}}
\psfrag{A}[Bc][Bc]{\scalebox{1}{\scriptsize{$\cat G i_!u^*$}}}
\psfrag{B}[Bc][Bc]{\scalebox{1}{\scriptsize{$\Id_{H}$}}}
\psfrag{C}[Bc][Bc]{\scalebox{1}{\scriptsize{$\;\;\;\;\;\;\;\cat G (\Id_H)$}}}
\psfrag{D}[Bc][Bc]{\scalebox{1}{\scriptsize{$\cat G (\Id_H \!\circ\, i_!u^*)$\;\;\;\;\;\;\;\;\;\quad}}}
\psfrag{E}[Bc][Bc]{\scalebox{1}{\scriptsize{$\cat G i_!u^*$}}}
\includegraphics[scale=.4]{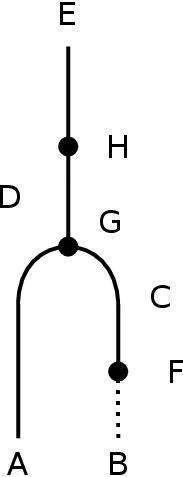}
}}
\quad\qquad \overset{?}{=}\quad\quad
\vcenter { \hbox{
\psfrag{A}[Bc][Bc]{\scalebox{1}{\scriptsize{$\cat G(i_!u^*)$}}}
\psfrag{C}[Bc][Bc]{\scalebox{1}{\scriptsize{$\cat G(i_!u^*)$}}}
\includegraphics[scale=.4]{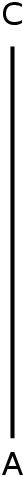}
}}
\end{equation}
To compute this, we need to specialize~\eqref{Eq:strings_fun_for_G} to the case $v=j=\Id_H$ in order to describe $\fun_{\cat G}\inv\colon \cat{G}(\Id_H \circ i_!u^*)\Rightarrow \cat{G}(\Id_H)\circ\cat{G}(i_!u^*)$. We also need the (inverted) left unitor $\lun^{-1} \colon i_!u^* \Rightarrow \Id_H \circ (i_!u^*)=\tilde i_!(u\,\tilde \Id)^*$ in $\Span$, which is constructed as follows, where $\ell := \langle \Id_P, i, \id_i \rangle$:
\[
\xymatrix{
G \ar@{=}[d] && P
 \ar[d]^\ell_{\simeq}
 \ar[ll]_u
 \ar[rr]^i
 \ar[dl]_{\Id}
 && H \ar@{=}[d] \\
G
\ar@/_6ex/@{..>}[ddrr]_-{i_!u^*} &
 P \ar[l]^-{u} &
 (i/\Id)
 \ar[dl]_-{\widetilde{\Id}}
 \ar[dr]^-{\tilde i}
 \ar[l]_-{\widetilde{\Id}}
 \ar[rr]^-{\tilde i} && H
\\
& P
\ar[ul]^-u \ar[dr]_-{i}
\ar@{}[rr]|{\oEcell{\gamma\;}} &&
 H
 \ar[ur]_-{\Id}
 \ar[dl]^-{\Id} &
\\
 && H
 \ar@/_6ex/@{..>}[rruu]_-{\Id_H} &&
}
\]
Note for later use that, by definition of~$\ell=\langle\cdots,\cdots,\id_i\rangle$, we have $\gamma\ell=\id_i$. Applying~$\cat G$ to this $\lun\inv$ we obtain
\[
\cat G(\lun^{-1})
\; = \;
\cat G([\ell, \id_u, \id_i])
\; = \;
\vcenter { \hbox{
\psfrag{A}[Bc][Bc]{\scalebox{1}{\scriptsize{$\cat F u$}}}
\psfrag{B}[Bc][Bc]{\scalebox{1}{\scriptsize{$(\cat F i)_!$}}}
\psfrag{C}[Bc][Bc]{\scalebox{1}{\scriptsize{\;\;\;$\cat F i$}}}
\psfrag{D}[Bc][Bc]{\scalebox{1}{\scriptsize{$\cat F \tilde i$}}}
\psfrag{E}[Bc][Bc]{\scalebox{1}{\scriptsize{$\cat F \ell$}}}
\psfrag{S}[Bc][Bc]{\scalebox{1}{\scriptsize{$\cat F \ell$\;}}}
\psfrag{T}[Bc][Bc]{\scalebox{1}{\scriptsize{$\cat F (u \,\widetilde{\Id})$}}}
\psfrag{U}[Bc][Bc]{\scalebox{1}{\scriptsize{\;\;\;\;$(\cat F \tilde i)_!$}}}
\includegraphics[scale=.4]{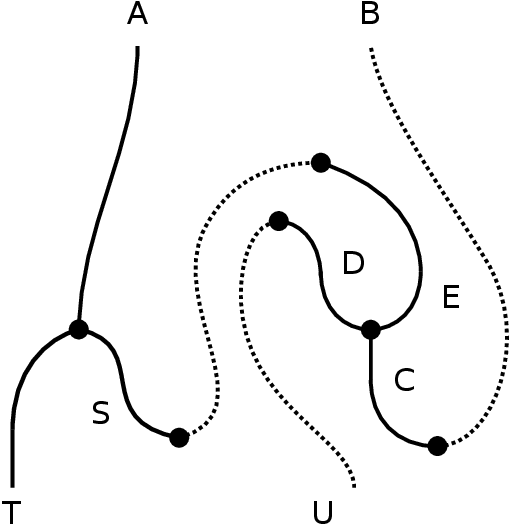}
}}
\]
We can now verify the left unit identity~\eqref{Eq:left-unit-for-G} as follows, first plugging the above partial computations:
\[
\vcenter {\hbox{
\psfrag{F}[Bc][Bc]{\scalebox{1}{\scriptsize{$\;\;\un^{-1}$}}}
\psfrag{G}[Bc][Bc]{\scalebox{1}{\scriptsize{$\;\;\;\fun$}}}
\psfrag{H}[Bc][Bc]{\scalebox{1}{\scriptsize{$\;\;\;\;\;\;\;\;\;\cat G(\lun^{-1})$}}}
\psfrag{A}[Bc][Bc]{\scalebox{1}{\scriptsize{$\cat G i_!u^*$}}}
\psfrag{B}[Bc][Bc]{\scalebox{1}{\scriptsize{$\Id_{H}$}}}
\psfrag{C}[Bc][Bc]{\scalebox{1}{\scriptsize{$\;\;\;\;\;\;\;\cat G (\Id_H)$}}}
\psfrag{D}[Bc][Bc]{\scalebox{1}{\scriptsize{$\cat G (\Id_H \!\circ\, i_!u^*)$\;\;\;\;\;\;\;\;\;\quad}}}
\psfrag{E}[Bc][Bc]{\scalebox{1}{\scriptsize{$\cat G i_!u^*$}}}
\includegraphics[scale=.4]{anc/unit-fun-lun-dot-inverted.eps}
}}
\quad =\quad
\vcenter { \hbox{
\psfrag{A}[Bc][Bc]{\scalebox{1}{\scriptsize{$\cat Fu$}}}
\psfrag{B}[Bc][Bc]{\scalebox{1}{\scriptsize{$\Id_H$}}}
\psfrag{C}[Bc][Bc]{\scalebox{1}{\scriptsize{$\cat Fu$}}}
\psfrag{A1}[Bc][Bc]{\scalebox{1}{\scriptsize{$(\cat Fi)_!$}}}
\psfrag{B1}[Bc][Bc]{\scalebox{1}{\scriptsize{$\Id_H$}}}
\psfrag{C1}[Bc][Bc]{\scalebox{1}{\scriptsize{$(\cat Fi )_!$}}}
\psfrag{U}[Bc][Bc]{\scalebox{1}{\scriptsize{$\;\cat F\Id$}}}
\psfrag{V}[Bc][Bc]{\scalebox{1}{\scriptsize{$\cat F\tilde i$}}}
\psfrag{W}[Bc][Bc]{\scalebox{1}{\scriptsize{$\cat F \tilde i$}}}
\psfrag{R}[Bc][Bc]{\scalebox{1}{\scriptsize{$\cat F \widetilde{\Id}$\;\;}}}
\psfrag{T}[Bc][Bc]{\scalebox{1}{\scriptsize{$\cat Fu \, \widetilde{\Id}$\;\;\;\;}}}
\psfrag{I}[Bc][Bc]{\scalebox{1}{\scriptsize{$\;\;\cat F \tilde i$}}}
\psfrag{J}[Bc][Bc]{\scalebox{1}{\scriptsize{$\;\cat F\ell$}}}
\psfrag{K}[Bc][Bc]{\scalebox{1}{\scriptsize{$\cat F i$}}}
\psfrag{S}[Bc][Bc]{\scalebox{1}{\scriptsize{$\cat F\Id$}}}
\psfrag{G}[Bc][Bc]{\scalebox{1}{\scriptsize{$\cat F\gamma$}}}
\psfrag{H}[Bc][Bc]{\scalebox{1}{\scriptsize{$\;\;\;\un^{-1}_\cat F$}}}
\includegraphics[scale=.4]{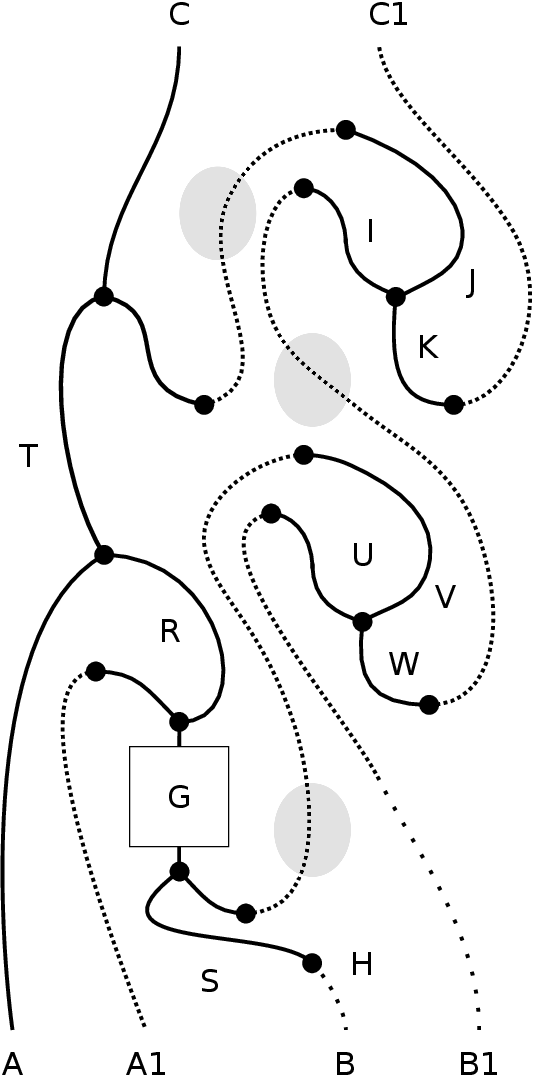}
}}
\stackrel{3 \times \textrm{\eqref{Exa:strings-for-adjoints}}}{=}
\vcenter { \hbox{
\psfrag{A}[Bc][Bc]{\scalebox{1}{\scriptsize{$\cat Fu$}}}
\psfrag{B}[Bc][Bc]{\scalebox{1}{\scriptsize{$\Id_H$}}}
\psfrag{C}[Bc][Bc]{\scalebox{1}{\scriptsize{$\cat Fu$}}}
\psfrag{A1}[Bc][Bc]{\scalebox{1}{\scriptsize{$(\cat Fi)_!$}}}
\psfrag{B1}[Bc][Bc]{\scalebox{1}{\scriptsize{$\Id_H$}}}
\psfrag{C1}[Bc][Bc]{\scalebox{1}{\scriptsize{$(\cat Fi )_!$}}}
\psfrag{U}[Bc][Bc]{\scalebox{1}{\scriptsize{$\cat F\Id$}}}
\psfrag{V}[Bc][Bc]{\scalebox{1}{\scriptsize{$\cat F\tilde i$}}}
\psfrag{W}[Bc][Bc]{\scalebox{1}{\scriptsize{$\cat F\tilde i$}}}
\psfrag{R}[Bc][Bc]{\scalebox{1}{\scriptsize{$\cat F \widetilde{\Id}$\;\;}}}
\psfrag{S}[Bc][Bc]{\scalebox{1}{\scriptsize{$\cat F \Id$}}}
\psfrag{T}[Bc][Bc]{\scalebox{1}{\scriptsize{$\cat Fu \, \widetilde{\Id}$\;\;\;\;}}}
\psfrag{I}[Bc][Bc]{\scalebox{1}{\scriptsize{$\;\;\cat F\tilde i$}}}
\psfrag{J}[Bc][Bc]{\scalebox{1}{\scriptsize{$\;\cat F\ell$}}}
\psfrag{K}[Bc][Bc]{\scalebox{1}{\scriptsize{$\cat F i$}}}
\psfrag{G}[Bc][Bc]{\scalebox{1}{\scriptsize{$\cat F\gamma$}}}
\psfrag{H}[Bc][Bc]{\scalebox{1}{\scriptsize{$\;\;\;\un^{-1}_\cat F$}}}
\includegraphics[scale=.4]{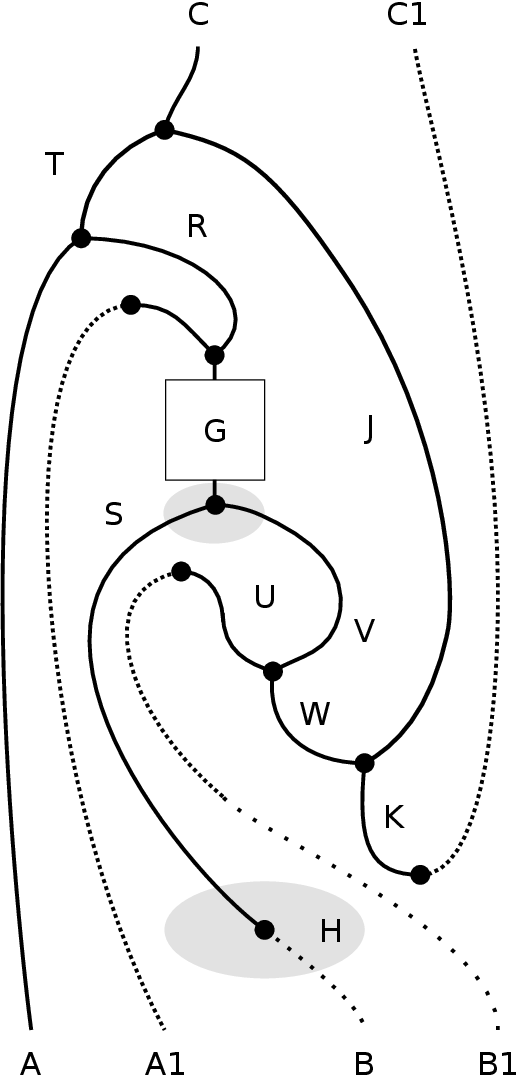}
}}
\]
\[
\stackrel{\textrm{\eqref{eq:fun-unit-simpler}}}{=}
\vcenter { \hbox{
\psfrag{A}[Bc][Bc]{\scalebox{1}{\scriptsize{$\cat Fu$}}}
\psfrag{B}[Bc][Bc]{\scalebox{1}{\scriptsize{$\Id_H$}}}
\psfrag{C}[Bc][Bc]{\scalebox{1}{\scriptsize{$\cat Fu$}}}
\psfrag{A1}[Bc][Bc]{\scalebox{1}{\scriptsize{$(\cat Fi)_!$}}}
\psfrag{B1}[Bc][Bc]{\scalebox{1}{\scriptsize{$\Id_H$}}}
\psfrag{C1}[Bc][Bc]{\scalebox{1}{\scriptsize{$(\cat Fi )_!$}}}
\psfrag{U}[Bc][Bc]{\scalebox{1}{\scriptsize{$\cat F\Id$}}}
\psfrag{V}[Bc][Bc]{\scalebox{1}{\scriptsize{$\cat F\tilde i$}}}
\psfrag{W}[Bc][Bc]{\scalebox{1}{\scriptsize{$\cat F\tilde i$}}}
\psfrag{R}[Bc][Bc]{\scalebox{1}{\scriptsize{$\cat F \widetilde{\Id}$\;\;}}}
\psfrag{S}[Bc][Bc]{\scalebox{1}{\scriptsize{$\cat F \Id$}}}
\psfrag{T}[Bc][Bc]{\scalebox{1}{\scriptsize{$\cat Fu \, \widetilde{\Id}$\;\;\;\;}}}
\psfrag{I}[Bc][Bc]{\scalebox{1}{\scriptsize{$\;\;\cat F\tilde i$}}}
\psfrag{J}[Bc][Bc]{\scalebox{1}{\scriptsize{$\;\cat F\ell$}}}
\psfrag{K}[Bc][Bc]{\scalebox{1}{\scriptsize{$\cat F i$}}}
\psfrag{G}[Bc][Bc]{\scalebox{1}{\scriptsize{$\cat F\gamma$}}}
\includegraphics[scale=.4]{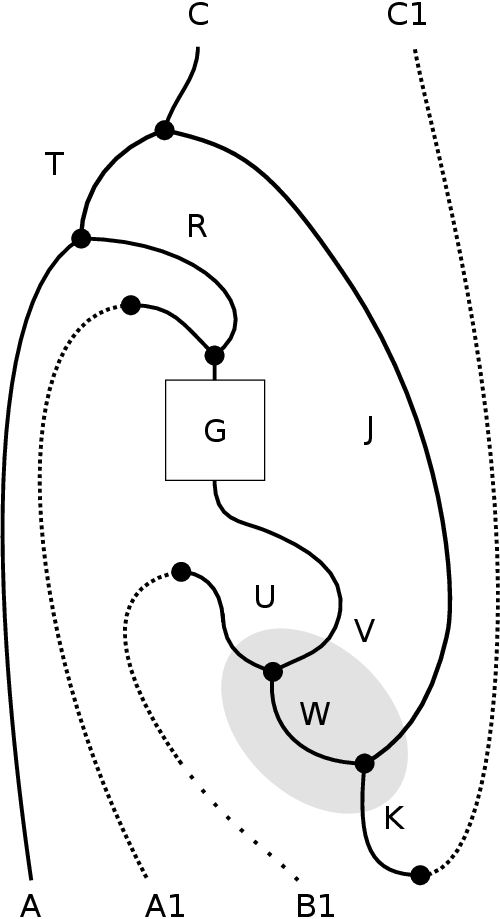}
}}
\stackrel{\textrm{\eqref{eq:fun-ass-simpler}}}{=}
\vcenter { \hbox{
\psfrag{A}[Bc][Bc]{\scalebox{1}{\scriptsize{$\cat Fu$}}}
\psfrag{B}[Bc][Bc]{\scalebox{1}{\scriptsize{$\Id_H$}}}
\psfrag{C}[Bc][Bc]{\scalebox{1}{\scriptsize{$\cat Fu$}}}
\psfrag{A1}[Bc][Bc]{\scalebox{1}{\scriptsize{$(\cat Fi)_!$}}}
\psfrag{B1}[Bc][Bc]{\scalebox{1}{\scriptsize{$\Id_H$}}}
\psfrag{C1}[Bc][Bc]{\scalebox{1}{\scriptsize{$(\cat Fi )_!$}}}
\psfrag{U}[Bc][Bc]{\scalebox{1}{\scriptsize{$\cat F\Id$}}}
\psfrag{V}[Bc][Bc]{\scalebox{1}{\scriptsize{$\cat F\tilde i$}}}
\psfrag{R}[Bc][Bc]{\scalebox{1}{\scriptsize{$\cat F \widetilde{\Id}$\;\;}}}
\psfrag{S}[Bc][Bc]{\scalebox{1}{\scriptsize{$\cat F \Id$}}}
\psfrag{T}[Bc][Bc]{\scalebox{1}{\scriptsize{$\cat Fu \, \widetilde{\Id}$\;\;\;\;}}}
\psfrag{I}[Bc][Bc]{\scalebox{1}{\scriptsize{$\;\;\cat F\tilde i$}}}
\psfrag{J}[Bc][Bc]{\scalebox{1}{\scriptsize{$\;\cat F\ell$}}}
\psfrag{K}[Bc][Bc]{\scalebox{1}{\scriptsize{$\cat F i$}}}
\psfrag{G}[Bc][Bc]{\scalebox{1}{\scriptsize{$\cat F\gamma$}}}
\includegraphics[scale=.4]{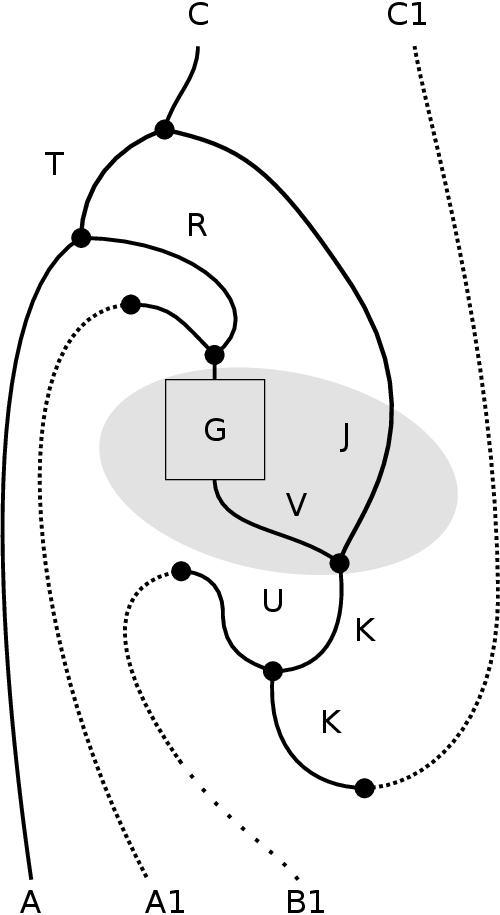}
}}
\stackrel{\textrm{\eqref{Exa:strings-for-natural-fun}}}{=}
\vcenter { \hbox{
\psfrag{A}[Bc][Bc]{\scalebox{1}{\scriptsize{$\cat Fu$}}}
\psfrag{B}[Bc][Bc]{\scalebox{1}{\scriptsize{$\Id_H$}}}
\psfrag{C}[Bc][Bc]{\scalebox{1}{\scriptsize{$\cat Fu$}}}
\psfrag{A1}[Bc][Bc]{\scalebox{1}{\scriptsize{$(\cat Fi)_!$}}}
\psfrag{B1}[Bc][Bc]{\scalebox{1}{\scriptsize{$\Id_H$}}}
\psfrag{C1}[Bc][Bc]{\scalebox{1}{\scriptsize{$(\cat Fi )_!$}}}
\psfrag{U}[Bc][Bc]{\scalebox{1}{\scriptsize{$\cat F\Id$}}}
\psfrag{V}[Bc][Bc]{\scalebox{1}{\scriptsize{$\cat F\tilde i$}}}
\psfrag{W}[Bc][Bc]{\scalebox{1}{\scriptsize{$\cat F\tilde i$}}}
\psfrag{R}[Bc][Bc]{\scalebox{1}{\scriptsize{$\cat F \widetilde{\Id}$\;\;}}}
\psfrag{S}[Bc][Bc]{\scalebox{1}{\scriptsize{$\cat F \Id$}}}
\psfrag{T}[Bc][Bc]{\scalebox{1}{\scriptsize{$\cat Fu \, \widetilde{\Id}$\;\;\;\;}}}
\psfrag{I}[Bc][Bc]{\scalebox{1}{\scriptsize{$\;\;\cat F\tilde i$}}}
\psfrag{J}[Bc][Bc]{\scalebox{1}{\scriptsize{$\;\cat F\ell$}}}
\psfrag{K}[Bc][Bc]{\scalebox{1}{\scriptsize{$\cat F i$}}}
\psfrag{L}[Bc][Bc]{\scalebox{1}{\scriptsize{$\cat F i\widetilde{\Id}\;\;$}}}
\psfrag{G}[Bc][Bc]{\scalebox{1}{\scriptsize{$\cat F(\gamma \, \ell)$}}}
\includegraphics[scale=.4]{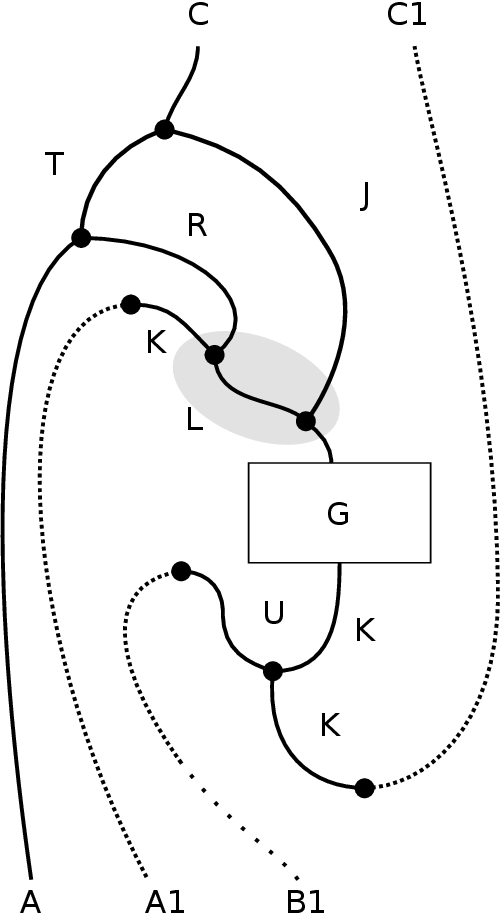}
}}
\]
The next passage uses the equation $\gamma \ell=\id_i$ as well as the associativity of $\fun_{\cat F}$:
\[
\stackrel{\eqref{eq:fun-ass-simpler}}{=}
\vcenter { \hbox{
\psfrag{A}[Bc][Bc]{\scalebox{1}{\scriptsize{$\cat Fu$}}}
\psfrag{B}[Bc][Bc]{\scalebox{1}{\scriptsize{$\Id_H$}}}
\psfrag{C}[Bc][Bc]{\scalebox{1}{\scriptsize{$\cat Fu$}}}
\psfrag{A1}[Bc][Bc]{\scalebox{1}{\scriptsize{$(\cat Fi)_!$}}}
\psfrag{B1}[Bc][Bc]{\scalebox{1}{\scriptsize{$\Id_H$}}}
\psfrag{C1}[Bc][Bc]{\scalebox{1}{\scriptsize{$(\cat Fi )_!$}}}
\psfrag{U}[Bc][Bc]{\scalebox{1}{\scriptsize{$\cat F\Id$}}}
\psfrag{V}[Bc][Bc]{\scalebox{1}{\scriptsize{$\cat F\tilde i$}}}
\psfrag{W}[Bc][Bc]{\scalebox{1}{\scriptsize{$\cat F\tilde i$}}}
\psfrag{R}[Bc][Bc]{\scalebox{1}{\scriptsize{$\cat F \widetilde{\Id}$\;\;}}}
\psfrag{S}[Bc][Bc]{\scalebox{1}{\scriptsize{$\cat F \Id$}}}
\psfrag{T}[Bc][Bc]{\scalebox{1}{\scriptsize{$\cat Fu \, \widetilde{\Id}$\;\;\;\;}}}
\psfrag{I}[Bc][Bc]{\scalebox{1}{\scriptsize{$\;\;\cat F\tilde i$}}}
\psfrag{J}[Bc][Bc]{\scalebox{1}{\scriptsize{$\;\cat F\ell$}}}
\psfrag{K}[Bc][Bc]{\scalebox{1}{\scriptsize{$\cat F i$}}}
\psfrag{L}[Bc][Bc]{\scalebox{1}{\scriptsize{$\cat F \Id$}}}
\psfrag{G}[Bc][Bc]{\scalebox{1}{\scriptsize{$\cat F(\gamma \circ \ell)$}}}
\includegraphics[scale=.4]{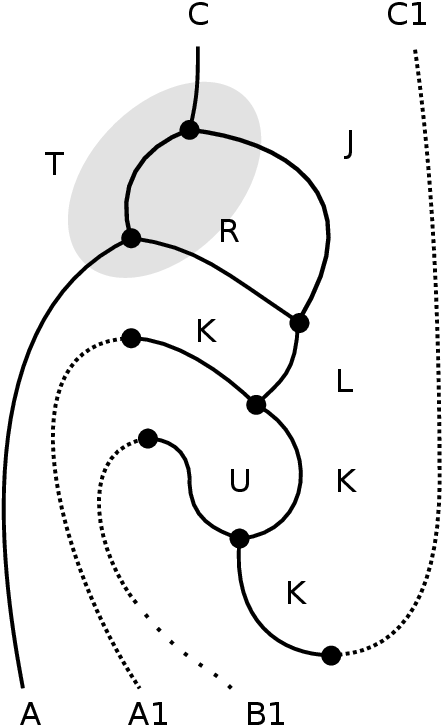}
}}
\stackrel{\eqref{eq:fun-ass-simpler}}{=}
\vcenter { \hbox{
\psfrag{A}[Bc][Bc]{\scalebox{1}{\scriptsize{$\cat Fu$}}}
\psfrag{B}[Bc][Bc]{\scalebox{1}{\scriptsize{$\Id_H$}}}
\psfrag{C}[Bc][Bc]{\scalebox{1}{\scriptsize{$\cat Fu$}}}
\psfrag{A1}[Bc][Bc]{\scalebox{1}{\scriptsize{$(\cat Fi)_!$}}}
\psfrag{B1}[Bc][Bc]{\scalebox{1}{\scriptsize{$\Id_H$}}}
\psfrag{C1}[Bc][Bc]{\scalebox{1}{\scriptsize{$(\cat Fi )_!$}}}
\psfrag{U}[Bc][Bc]{\scalebox{1}{\scriptsize{$\cat F\Id$}}}
\psfrag{V}[Bc][Bc]{\scalebox{1}{\scriptsize{$\cat F\tilde i$}}}
\psfrag{W}[Bc][Bc]{\scalebox{1}{\scriptsize{$\cat F\tilde i$}}}
\psfrag{R}[Bc][Bc]{\scalebox{1}{\scriptsize{$\cat F \widetilde{\Id}$\;\;}}}
\psfrag{S}[Bc][Bc]{\scalebox{1}{\scriptsize{$\cat F \Id$}}}
\psfrag{T}[Bc][Bc]{\scalebox{1}{\scriptsize{$\cat Fu \, \widetilde{\Id}$\;\;\;\;}}}
\psfrag{I}[Bc][Bc]{\scalebox{1}{\scriptsize{$\;\;\cat F\tilde i$}}}
\psfrag{J}[Bc][Bc]{\scalebox{1}{\scriptsize{$\;\cat F\ell$}}}
\psfrag{K}[Bc][Bc]{\scalebox{1}{\scriptsize{$\cat F i$}}}
\psfrag{L}[Bc][Bc]{\scalebox{1}{\scriptsize{$\cat F \Id$}}}
\psfrag{G}[Bc][Bc]{\scalebox{1}{\scriptsize{$\cat F(\gamma \circ \ell)$}}}
\includegraphics[scale=.4]{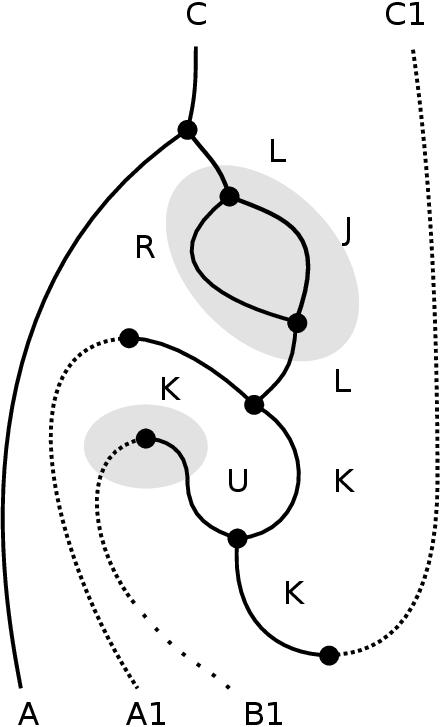}
}}
\stackrel{(\cat F \Id)_!=\Id}{=}
\vcenter { \hbox{
\psfrag{A}[Bc][Bc]{\scalebox{1}{\scriptsize{$\cat Fu$}}}
\psfrag{A1}[Bc][Bc]{\scalebox{1}{\scriptsize{$(\cat Fi)_!$}}}
\psfrag{C}[Bc][Bc]{\scalebox{1}{\scriptsize{$\cat Fu$}}}
\psfrag{C1}[Bc][Bc]{\scalebox{1}{\scriptsize{$(\cat Fi )_!$}}}
\psfrag{V}[Bc][Bc]{\scalebox{1}{\scriptsize{$\cat F i$}}}
\psfrag{R}[Bc][Bc]{\scalebox{1}{\scriptsize{$\cat F \Id$}}}
\psfrag{F}[Bc][Bc]{\scalebox{1}{\scriptsize{$\un_\cat F$}}}
\includegraphics[scale=.4]{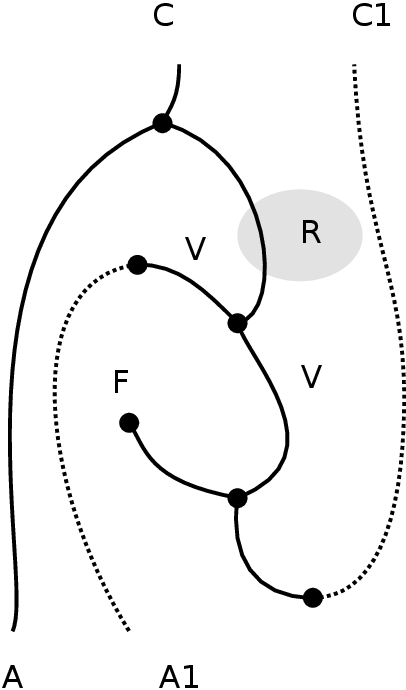}
}}
\]
\[
=
\vcenter { \hbox{
\psfrag{A}[Bc][Bc]{\scalebox{1}{\scriptsize{$\cat Fu$}}}
\psfrag{A1}[Bc][Bc]{\scalebox{1}{\scriptsize{$(\cat Fi)_!$}}}
\psfrag{C}[Bc][Bc]{\scalebox{1}{\scriptsize{$\cat Fu$}}}
\psfrag{C1}[Bc][Bc]{\scalebox{1}{\scriptsize{$(\cat Fi )_!$}}}
\psfrag{V}[Bc][Bc]{\scalebox{1}{\scriptsize{$\cat F i$}}}
\psfrag{F}[Bc][Bc]{\scalebox{1}{\scriptsize{$\un_\cat F$}}}
\psfrag{G}[Bc][Bc]{\scalebox{1}{\scriptsize{$\fun_{\cat F}$}}}
\psfrag{G1}[Bc][Bc]{\scalebox{1}{\scriptsize{$\fun_{\cat F}^{-1}$}}}
\psfrag{U1}[Bc][Bc]{\scalebox{1}{\scriptsize{$\un_{\cat F}^{-1}$}}}
\psfrag{U2}[Bc][Bc]{\scalebox{1}{\scriptsize{$\un_{\cat F}$}}}
\psfrag{R}[Bc][Bc]{\scalebox{1}{\scriptsize{$\cat F \Id$}}}
\psfrag{S}[Bc][Bc]{\scalebox{1}{\scriptsize{$\Id$}}}
\includegraphics[scale=.4]{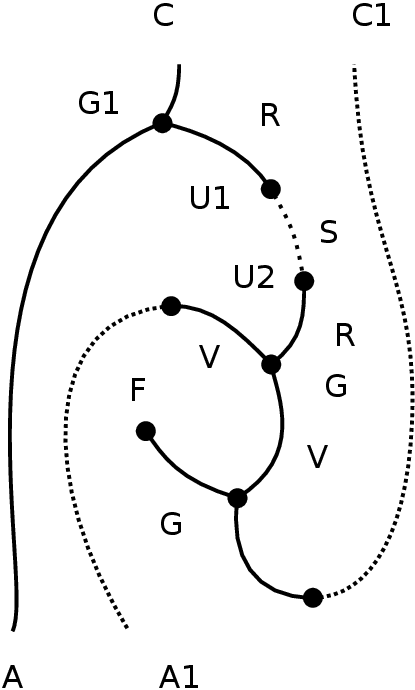}
}}
\stackrel{3\times\textrm{\eqref{eq:fun-unit-simpler}}}{=}
\vcenter { \hbox{
\psfrag{A}[Bc][Bc]{\scalebox{1}{\scriptsize{$\cat Fu$}}}
\psfrag{A1}[Bc][Bc]{\scalebox{1}{\scriptsize{$(\cat Fi)_!$}}}
\psfrag{C}[Bc][Bc]{\scalebox{1}{\scriptsize{$\cat Fu$}}}
\psfrag{C1}[Bc][Bc]{\scalebox{1}{\scriptsize{$(\cat Fi )_!$}}}
\psfrag{U}[Bc][Bc]{\scalebox{1}{\scriptsize{$\eta$}}}
\psfrag{V}[Bc][Bc]{\scalebox{1}{\scriptsize{$\cat F i$}}}
\psfrag{E}[Bc][Bc]{\scalebox{1}{\scriptsize{$\varepsilon$}}}
\includegraphics[scale=.4]{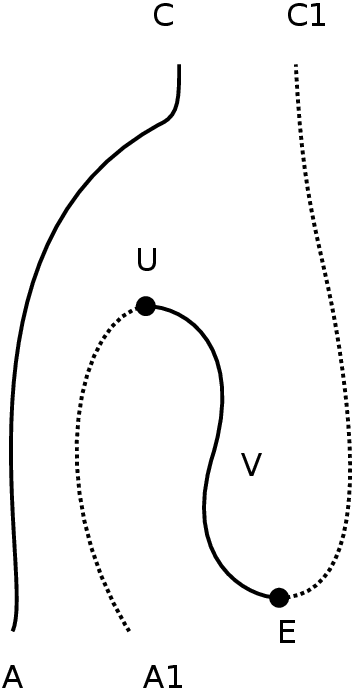}
}}
\stackrel{\textrm{\eqref{Exa:strings-for-adjoints}}}{=}
\vcenter { \hbox{
\psfrag{A}[Bc][Bc]{\scalebox{1}{\scriptsize{$\cat Fu\;$}}}
\psfrag{A1}[Bc][Bc]{\scalebox{1}{\scriptsize{$\;(\cat Fi)_!$}}}
\psfrag{C}[Bc][Bc]{\scalebox{1}{\scriptsize{$\cat Fu\;$}}}
\psfrag{C1}[Bc][Bc]{\scalebox{1}{\scriptsize{$\;(\cat Fi )_!$}}}
\includegraphics[scale=.4]{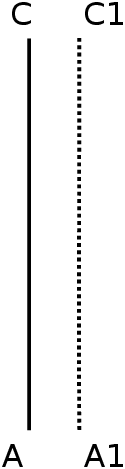}
}}
= \quad
\vcenter { \hbox{
\psfrag{A}[Bc][Bc]{\scalebox{1}{\scriptsize{$\cat G(i_!u^*)$}}}
\psfrag{C}[Bc][Bc]{\scalebox{1}{\scriptsize{$\cat G(i_!u^*)$}}}
\includegraphics[scale=.4]{anc/halfpseudofun-unital12.eps}
}}
\]
The proof of the other unit identity is similar and left to the reader.

In order to check the associativity axiom, we must first make explicit the associator in $\Span$. Given three composable spans, one first constructs the following four iso-commas:
\begin{equation} \label{Eq:precise-assoc}
\vcenter{\xymatrix@C=8pt@R=10pt{
&&
 &&
 &&
 (i/v)/w
 \ar[dll]_x
 \ar[ddrrrr]^s
 \ar@{}[dddrr]|{\oEcell{\delta}} &&
 &&
 &&
 \\
&&
 &&
 (i/v)
 \ar[dll]_{\tilde v}
 \ar[drr]^{\tilde i}
 \ar@{}[dd]|{\oEcell{\gamma}} && && && && \\
&&
 P
 \ar[dll]_{u}
 \ar[drr]^{i}
 &&
 &&
 Q
 \ar[dll]_{v}
 \ar[drr]^{j}
 &&
 &&
 R
 \ar[dll]_{w}
 \ar[drr]^{k} && \\
G
 \ar@{..>}[rrrr]|{\;i_!u^*\;} &&
 &&
 H
 \ar@{..>}[rrrr]|{\;j_!v^*\;}
 &&
 &&
 K
 \ar@{..>}[rrrr]|{\;k_!w^*\;}
 &&
 &&
 L \\
&&
 P
 \ar[ull]^{u}
 \ar[urr]_i
 &&
 &&
 Q
 \ar[ull]^{v}
 \ar[urr]_{j}
 &&
 &&
 R
 \ar[ull]^{w}
 \ar[urr]_{k}
 && \\
&& && && &&
 (j/w)
 \ar[ull]^{\tilde w}
 \ar[urr]_{\tilde j}
 \ar@{}[uu]|{\oEcell{\alpha}}
 &&
 && \\
&& && &&
 i/(j/w)
 \ar[uullll]^{y}
 \ar[urr]_t
 \ar@{}[lluuu]|{\oEcell{\beta}}
 && && &&
}}
\end{equation}
Then the universal property of iso-comma squares yields strictly invertible 1-cells $f\colon i/(j/w)\stackrel{\sim}{\leftrightarrow}(i/v)/w \,:\! f^{-1}$ of $\GG$, uniquely determined by the following relations:
\begin{align} \label{Eq:def-f}
f &:
\left\{\begin{array}{l}
xf = \langle y, \tilde wt, \beta\rangle, \quad\textrm{\ie:} \quad \tilde vxf=y, \;\; \tilde i xf =\tilde wt, \;\; \gamma x f=\beta \\
sf = \tilde j t \\
\delta f = \alpha t
\end{array}\right.
\\
\label{Eq:def-f-1}
f^{-1} &:
\left\{\begin{array}{l}
y f^{-1} = \tilde v x \\
tf^{-1} = \langle \tilde ix, s, \delta \rangle, \quad\textrm{\ie:} \quad \tilde w t f^{-1}=\tilde ix, \;\; \tilde jt f^{-1}=s, \;\; \alpha tf^{-1} = \delta \\
\beta f^{-1} = \gamma x
\end{array}\right.
\end{align}
The associator and its inverse are then simply the 2-cells $[f^{-1},\id,\id]$ and $[f,\id,\id]$.

Now, the associativity axiom for $\cat G$ looks as follows (where once again we write the inverse 2-cells out of convenience):
\begin{equation} \label{Eq:assoc-for-G}
\vcenter { \hbox{
\psfrag{A}[Bc][Bc]{\scalebox{1}{\scriptsize{$\cat F u$}}}
\psfrag{A1}[Bc][Bc]{\scalebox{1}{\scriptsize{$(\cat Fi)_!$}}}
\psfrag{B}[Bc][Bc]{\scalebox{1}{\scriptsize{$\cat Fv$}}}
\psfrag{B1}[Bc][Bc]{\scalebox{1}{\scriptsize{$(\cat Fj)_!$}}}
\psfrag{C}[Bc][Bc]{\scalebox{1}{\scriptsize{$\cat Fw$}}}
\psfrag{C1}[Bc][Bc]{\scalebox{1}{\scriptsize{$(\cat Fk)_!$}}}
\psfrag{D}[Bc][Bc]{\scalebox{1}{\scriptsize{$\cat Fuy$}}}
\psfrag{D1}[Bc][Bc]{\scalebox{1}{\scriptsize{$(\cat Fk\tilde j t)_!$}}}
\psfrag{E}[Bc][Bc]{\scalebox{1}{\scriptsize{$\cat Fy$}}}
\psfrag{H}[Bc][Bc]{\scalebox{1}{\scriptsize{$(\cat Ft)_!\;\;$}}}
\psfrag{I}[Bc][Bc]{\scalebox{1}{\scriptsize{$(\cat Fk \tilde j)_!$}}}
\psfrag{J}[Bc][Bc]{\scalebox{1}{\scriptsize{\;$\cat Fv\tilde w$}}}
\psfrag{K}[Bc][Bc]{\scalebox{1}{\scriptsize{$\cat F\tilde w$}}}
\psfrag{L}[Bc][Bc]{\scalebox{1}{\scriptsize{$(\cat F\tilde j)_!$}}}
\psfrag{F}[Bc][Bc]{\scalebox{1}{\scriptsize{$(\cat F\beta)_!$}}}
\psfrag{G}[Bc][Bc]{\scalebox{1}{\scriptsize{$(\cat F\alpha)_!$}}}
\includegraphics[scale=.4]{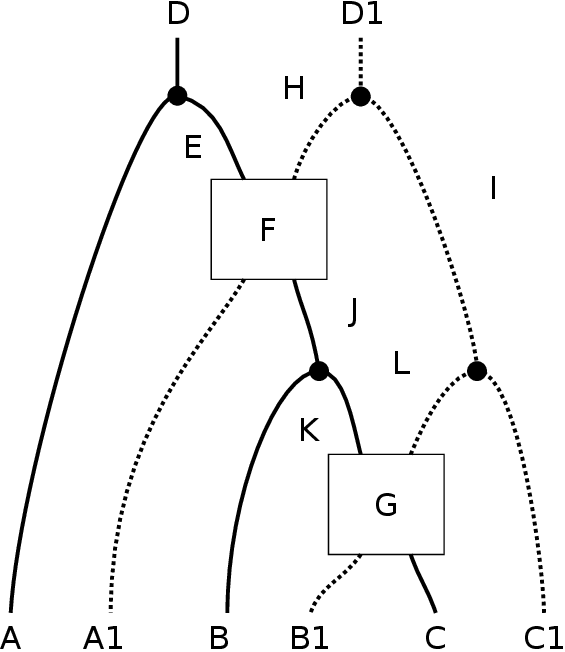}
}}
\quad \stackrel{\textrm{?}}{=} \quad
\vcenter { \hbox{
\psfrag{A}[Bc][Bc]{\scalebox{1}{\scriptsize{$\cat F u$}}}
\psfrag{A1}[Bc][Bc]{\scalebox{1}{\scriptsize{$(\cat Fi)_!$}}}
\psfrag{B}[Bc][Bc]{\scalebox{1}{\scriptsize{$\cat Fv$}}}
\psfrag{B1}[Bc][Bc]{\scalebox{1}{\scriptsize{$(\cat Fj)_!$}}}
\psfrag{C}[Bc][Bc]{\scalebox{1}{\scriptsize{$\cat Fw$}}}
\psfrag{C1}[Bc][Bc]{\scalebox{1}{\scriptsize{$(\cat Fk)_!$}}}
\psfrag{D}[Bc][Bc]{\scalebox{1}{\scriptsize{$\cat Fu\tilde vx$}}}
\psfrag{D1}[Bc][Bc]{\scalebox{1}{\scriptsize{$(\cat Fks)_!$}}}
\psfrag{E}[Bc][Bc]{\scalebox{1}{\scriptsize{$\cat Fu\tilde vxf$}}}
\psfrag{E1}[Bc][Bc]{\scalebox{1}{\scriptsize{$(\cat Fksf)_!$}}}
\psfrag{M}[Bc][Bc]{\scalebox{1}{\scriptsize{$(\cat Fs)_!$}}}
\psfrag{H}[Bc][Bc]{\scalebox{1}{\scriptsize{$\cat Fx\;\;$}}}
\psfrag{I}[Bc][Bc]{\scalebox{1}{\scriptsize{$\cat Fu \tilde v\;$}}}
\psfrag{J}[Bc][Bc]{\scalebox{1}{\scriptsize{$(\cat Fj \tilde i)_!$}}}
\psfrag{K}[Bc][Bc]{\scalebox{1}{\scriptsize{$(\cat F\tilde i)_!$}}}
\psfrag{L}[Bc][Bc]{\scalebox{1}{\scriptsize{$\cat F\tilde v$\;\;}}}
\psfrag{R}[Bc][Bc]{\scalebox{1}{\scriptsize{$\cat Ff$}}}
\psfrag{S}[Bc][Bc]{\scalebox{1}{\scriptsize{$(\cat Ff)_!$}}}
\psfrag{F}[Bc][Bc]{\scalebox{1}{\scriptsize{$(\cat F\delta)_!$}}}
\psfrag{G}[Bc][Bc]{\scalebox{1}{\scriptsize{$(\cat F\gamma)_!$}}}
\includegraphics[scale=.4]{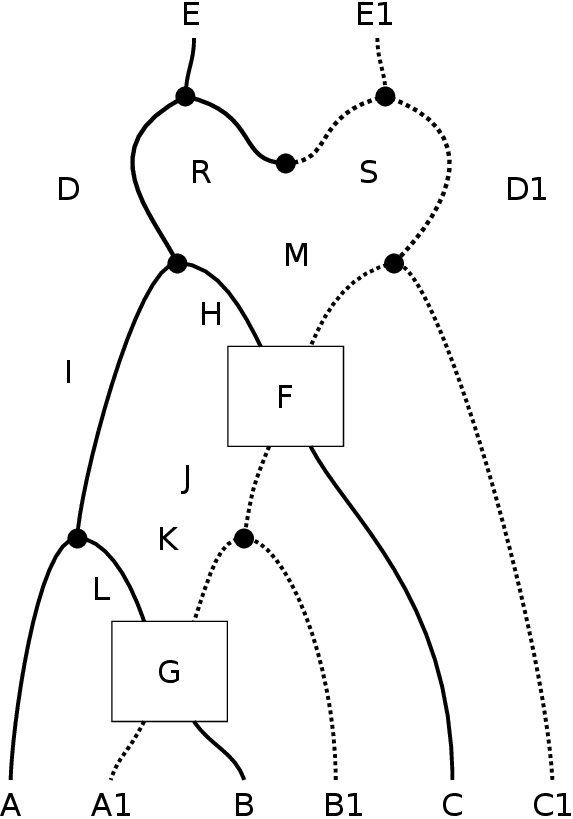}
}}
\end{equation}
Unfurling all mates and introducing $\cat F(ff^{-1})=\cat F(\Id)\cong \Id$, the left-hand-side expands as follows:
\[
\vcenter { \hbox{
\psfrag{A}[Bc][Bc]{\scalebox{1}{\scriptsize{$\cat Fu$}}}
\psfrag{A1}[Bc][Bc]{\scalebox{1}{\scriptsize{$(\cat Fi)_!$}}}
\psfrag{B}[Bc][Bc]{\scalebox{1}{\scriptsize{$\cat Fv$}}}
\psfrag{B1}[Bc][Bc]{\scalebox{1}{\scriptsize{$(\cat Fj)_!$}}}
\psfrag{C}[Bc][Bc]{\scalebox{1}{\scriptsize{$\cat Fw$}}}
\psfrag{C1}[Bc][Bc]{\scalebox{1}{\scriptsize{$(\cat Fk)_!$}}}
\psfrag{D}[Bc][Bc]{\scalebox{1}{\scriptsize{$\cat Fuy$}}}
\psfrag{D1}[Bc][Bc]{\scalebox{1}{\scriptsize{$(\cat Fk \tilde j t)_!$}}}
\psfrag{U}[Bc][Bc]{\scalebox{1}{\scriptsize{$\;\;\;\cat Fk\tilde j$}}}
\psfrag{V}[Bc][Bc]{\scalebox{1}{\scriptsize{$\cat Ft$}}}
\psfrag{W}[Bc][Bc]{\scalebox{1}{\scriptsize{$\cat F\!k\tilde jt$}}}
\psfrag{U'}[Bc][Bc]{\scalebox{1}{\scriptsize{$\;\cat Fk$}}}
\psfrag{V'}[Bc][Bc]{\scalebox{1}{\scriptsize{$\cat F\tilde j$}}}
\psfrag{W'}[Bc][Bc]{\scalebox{1}{\scriptsize{$\cat F\!k\tilde j$}}}
\psfrag{U1}[Bc][Bc]{\scalebox{1}{\scriptsize{$(\cat Fk\tilde j)_!$}}}
\psfrag{V1}[Bc][Bc]{\scalebox{1}{\scriptsize{$(\cat Ft)_!$}}}
\psfrag{V'1}[Bc][Bc]{\scalebox{1}{\scriptsize{$(\cat F\tilde j)_!$}}}
\psfrag{R}[Bc][Bc]{\scalebox{1}{\scriptsize{$\cat F y$}}}
\psfrag{S}[Bc][Bc]{\scalebox{1}{\scriptsize{$\cat Fi$}}}
\psfrag{T}[Bc][Bc]{\scalebox{1}{\scriptsize{$\cat Fv\tilde w$}}}
\psfrag{X}[Bc][Bc]{\scalebox{1}{\scriptsize{$\cat Fj$}}}
\psfrag{Y}[Bc][Bc]{\scalebox{1}{\scriptsize{$\cat F\tilde w$}}}
\psfrag{F}[Bc][Bc]{\scalebox{1}{\scriptsize{$\cat F\beta$}}}
\psfrag{G}[Bc][Bc]{\scalebox{1}{\scriptsize{$\cat F\alpha$}}}
\includegraphics[scale=.4]{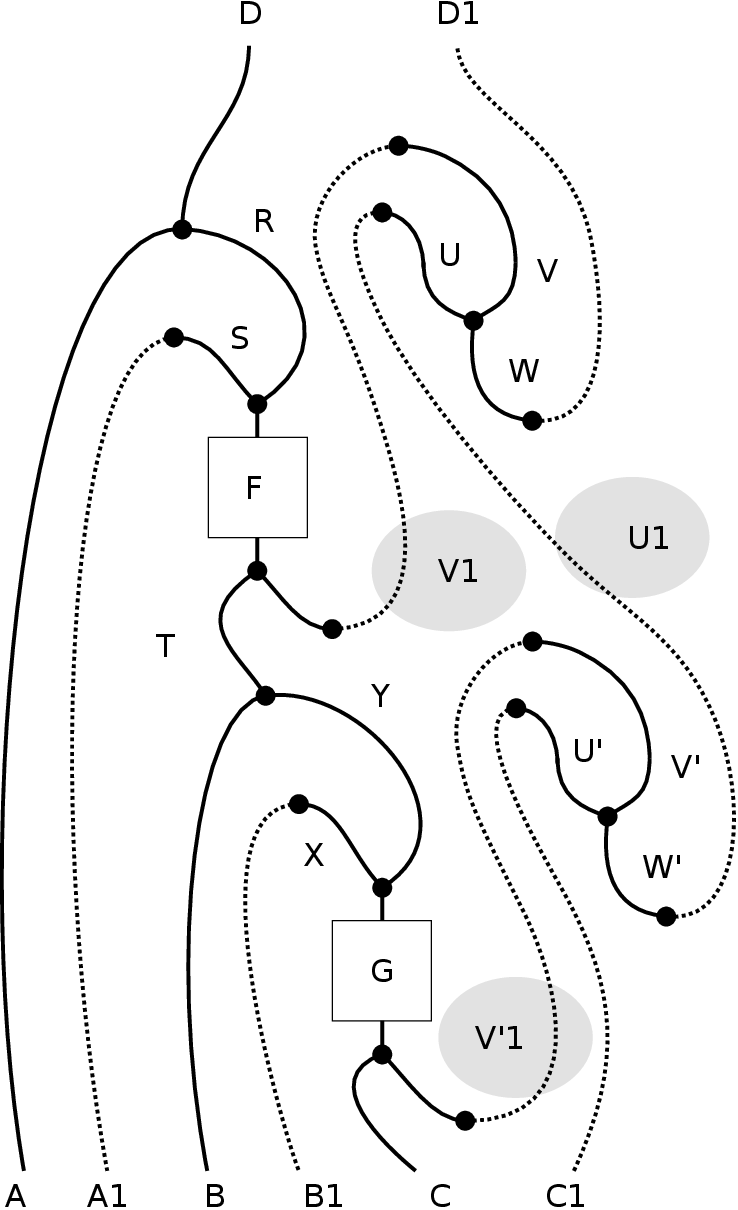}
}}
\stackrel{3 \times \textrm{\eqref{Exa:strings-for-adjoints}}}{=}
\vcenter { \hbox{
\psfrag{A}[Bc][Bc]{\scalebox{1}{\scriptsize{$\cat Fu$}}}
\psfrag{A1}[Bc][Bc]{\scalebox{1}{\scriptsize{$(\cat Fi)_!$}}}
\psfrag{B}[Bc][Bc]{\scalebox{1}{\scriptsize{$\cat Fv$}}}
\psfrag{B1}[Bc][Bc]{\scalebox{1}{\scriptsize{$(\cat Fj)_!$}}}
\psfrag{C}[Bc][Bc]{\scalebox{1}{\scriptsize{$\cat Fw$}}}
\psfrag{C1}[Bc][Bc]{\scalebox{1}{\scriptsize{$(\cat Fk)_!$}}}
\psfrag{D}[Bc][Bc]{\scalebox{1}{\scriptsize{$\cat Fuy$}}}
\psfrag{D1}[Bc][Bc]{\scalebox{1}{\scriptsize{$(\cat Fk\tilde jt)_!$}}}
\psfrag{U}[Bc][Bc]{\scalebox{1}{\scriptsize{$\;\;\;\cat Fk\tilde j$}}}
\psfrag{V}[Bc][Bc]{\scalebox{1}{\scriptsize{$\cat Ft$}}}
\psfrag{W}[Bc][Bc]{\scalebox{1}{\scriptsize{$\cat F\!k\tilde jt$}}}
\psfrag{U'}[Bc][Bc]{\scalebox{1}{\scriptsize{$\;\cat Fk$}}}
\psfrag{V'}[Bc][Bc]{\scalebox{1}{\scriptsize{$\cat F\tilde j$}}}
\psfrag{W'}[Bc][Bc]{\scalebox{1}{\scriptsize{$\cat F\!k\tilde j$}}}
\psfrag{U1}[Bc][Bc]{\scalebox{1}{\scriptsize{$(\cat Fk\tilde j)_!$}}}
\psfrag{V1}[Bc][Bc]{\scalebox{1}{\scriptsize{$(\cat Ft)_!$}}}
\psfrag{V'1}[Bc][Bc]{\scalebox{1}{\scriptsize{$(\cat F\tilde j)_!$}}}
\psfrag{R}[Bc][Bc]{\scalebox{1}{\scriptsize{$\cat F y$}}}
\psfrag{S}[Bc][Bc]{\scalebox{1}{\scriptsize{$\cat F i$}}}
\psfrag{T}[Bc][Bc]{\scalebox{1}{\scriptsize{$\cat Fv \tilde w$}}}
\psfrag{X}[Bc][Bc]{\scalebox{1}{\scriptsize{$\cat Fj$}}}
\psfrag{Y}[Bc][Bc]{\scalebox{1}{\scriptsize{$\;\;\cat F\tilde w$}}}
\psfrag{Z}[Bc][Bc]{\scalebox{1}{\scriptsize{$\cat Ff$}}}
\psfrag{Z'}[Bc][Bc]{\scalebox{1}{\scriptsize{$\cat F\!f^{-\!1}$}}}
\psfrag{F}[Bc][Bc]{\scalebox{1}{\scriptsize{$\cat F\beta$}}}
\psfrag{G}[Bc][Bc]{\scalebox{1}{\scriptsize{$\cat F\alpha$}}}
\psfrag{N}[Bc][Bc]{\scalebox{1}{\scriptsize{$\un$}}}
\psfrag{N'}[Bc][Bc]{\scalebox{1}{\scriptsize{$\un^{\!-\!1}$}}}
\includegraphics[scale=.4]{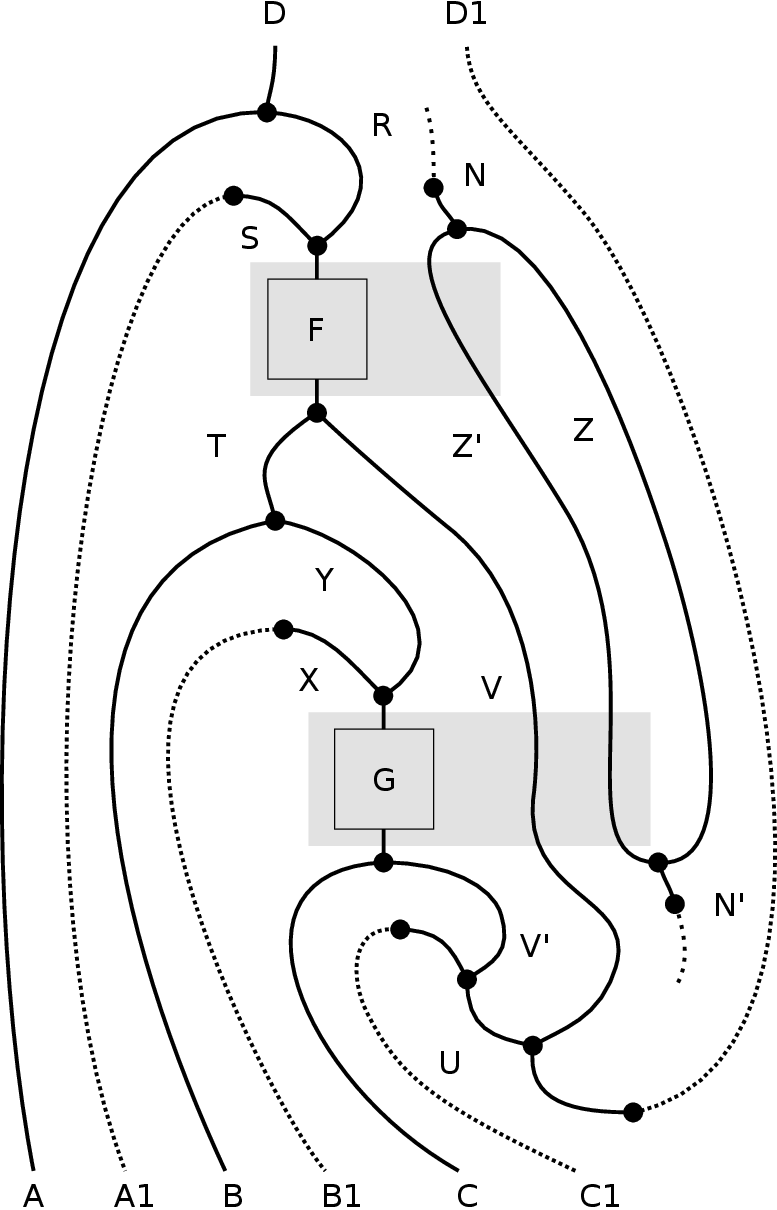}
}}
\]
For the next step we use the equations~\eqref{Eq:def-f} and~\eqref{Eq:def-f-1} to replace the 2-cells in the two shaded areas. We repeatedly use the associativity and unitality of $\fun_\cat F$, leaving some straightforward details to the reader:
\[
=
\vcenter { \hbox{
\psfrag{A}[Bc][Bc]{\scalebox{1}{\scriptsize{$\cat Fu$}}}
\psfrag{A1}[Bc][Bc]{\scalebox{1}{\scriptsize{$(\cat Fi)_!$}}}
\psfrag{B}[Bc][Bc]{\scalebox{1}{\scriptsize{$\cat Fv$}}}
\psfrag{B1}[Bc][Bc]{\scalebox{1}{\scriptsize{$(\cat Fj)_!$}}}
\psfrag{C}[Bc][Bc]{\scalebox{1}{\scriptsize{$\cat Fw$}}}
\psfrag{C1}[Bc][Bc]{\scalebox{1}{\scriptsize{$(\cat Fk)_!$}}}
\psfrag{D}[Bc][Bc]{\scalebox{1}{\scriptsize{$\cat Fuy$}}}
\psfrag{D1}[Bc][Bc]{\scalebox{1}{\scriptsize{$(\cat Fk\tilde jt)_!$}}}
\psfrag{U}[Bc][Bc]{\scalebox{1}{\scriptsize{$\;\;\;\cat Fk\tilde j$}}}
\psfrag{V}[Bc][Bc]{\scalebox{1}{\scriptsize{$\cat Ft$}}}
\psfrag{W}[Bc][Bc]{\scalebox{1}{\scriptsize{$\cat F\!k\tilde jt$}}}
\psfrag{U'}[Bc][Bc]{\scalebox{1}{\scriptsize{$\;\cat Fk$}}}
\psfrag{V'}[Bc][Bc]{\scalebox{1}{\scriptsize{$\cat F\tilde j$}}}
\psfrag{W'}[Bc][Bc]{\scalebox{1}{\scriptsize{$\cat F\!k\tilde j$}}}
\psfrag{U1}[Bc][Bc]{\scalebox{1}{\scriptsize{$(\cat Fk\tilde j)_!$}}}
\psfrag{V1}[Bc][Bc]{\scalebox{1}{\scriptsize{$(\cat Ft)_!$}}}
\psfrag{V'1}[Bc][Bc]{\scalebox{1}{\scriptsize{$(\cat F\tilde j)_!$}}}
\psfrag{R}[Bc][Bc]{\scalebox{1}{\scriptsize{$\cat F y$}}}
\psfrag{S}[Bc][Bc]{\scalebox{1}{\scriptsize{$\cat Fi$}}}
\psfrag{T}[Bc][Bc]{\scalebox{1}{\scriptsize{$\cat Fv\tilde w$}}}
\psfrag{X}[Bc][Bc]{\scalebox{1}{\scriptsize{$\cat Fj$}}}
\psfrag{Y}[Bc][Bc]{\scalebox{1}{\scriptsize{$\;\;\cat F\tilde w$}}}
\psfrag{Z}[Bc][Bc]{\scalebox{1}{\scriptsize{$\cat Ff$}}}
\psfrag{Z'}[Bc][Bc]{\scalebox{1}{\scriptsize{$\cat F\!f^{-\!1}$}}}
\psfrag{F}[Bc][Bc]{\scalebox{1}{\scriptsize{$\cat F(\gamma x)$}}}
\psfrag{G}[Bc][Bc]{\scalebox{1}{\scriptsize{$\cat F\delta$}}}
\includegraphics[scale=.4]{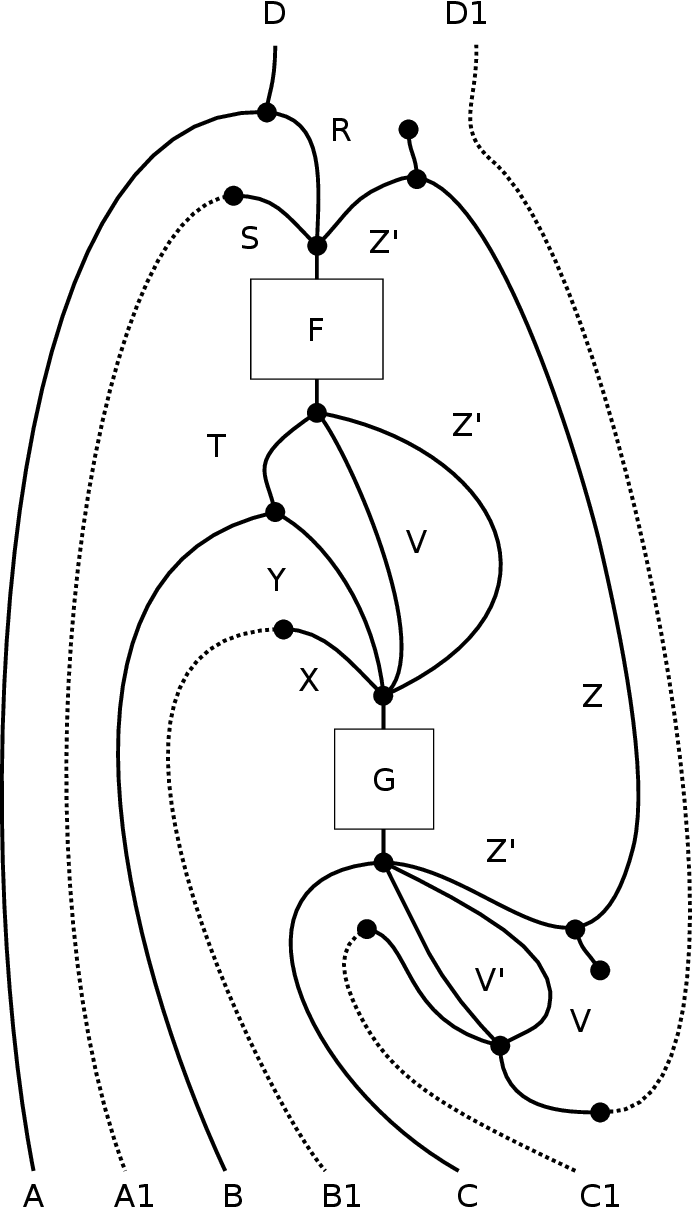}
}}
\quad = \quad
\vcenter { \hbox{
\psfrag{A}[Bc][Bc]{\scalebox{1}{\scriptsize{$\cat Fu$}}}
\psfrag{A1}[Bc][Bc]{\scalebox{1}{\scriptsize{$(\cat Fi)_!$}}}
\psfrag{B}[Bc][Bc]{\scalebox{1}{\scriptsize{$\cat Fv$}}}
\psfrag{B1}[Bc][Bc]{\scalebox{1}{\scriptsize{$(\cat Fj)_!$}}}
\psfrag{C}[Bc][Bc]{\scalebox{1}{\scriptsize{$\cat Fw$}}}
\psfrag{C1}[Bc][Bc]{\scalebox{1}{\scriptsize{$(\cat Fk)_!$}}}
\psfrag{D}[Bc][Bc]{\scalebox{1}{\scriptsize{$\cat Fu\tilde vxf$}}}
\psfrag{D1}[Bc][Bc]{\scalebox{1}{\scriptsize{$(\cat Fk sf )_!$}}}
\psfrag{U}[Bc][Bc]{\scalebox{1}{\scriptsize{$\;\;\cat Fks$}}}
\psfrag{V}[Bc][Bc]{\scalebox{1}{\scriptsize{$\cat Fs$}}}
\psfrag{W}[Bc][Bc]{\scalebox{1}{\scriptsize{$\;\;\cat Fks\!f$}}}
\psfrag{W'}[Bc][Bc]{\scalebox{1}{\scriptsize{$\cat Fks$}}}
\psfrag{R}[Bc][Bc]{\scalebox{1}{\scriptsize{$\cat F x$}}}
\psfrag{S}[Bc][Bc]{\scalebox{1}{\scriptsize{$\cat F\tilde v$}}}
\psfrag{T}[Bc][Bc]{\scalebox{1}{\scriptsize{$\cat F\tilde i$}}}
\psfrag{X}[Bc][Bc]{\scalebox{1}{\scriptsize{$\cat Fj$}}}
\psfrag{Z}[Bc][Bc]{\scalebox{1}{\scriptsize{$\cat Ff$}}}
\psfrag{F}[Bc][Bc]{\scalebox{1}{\scriptsize{$\cat F\gamma$}}}
\psfrag{G}[Bc][Bc]{\scalebox{1}{\scriptsize{$\cat F\delta$}}}
\includegraphics[scale=.4]{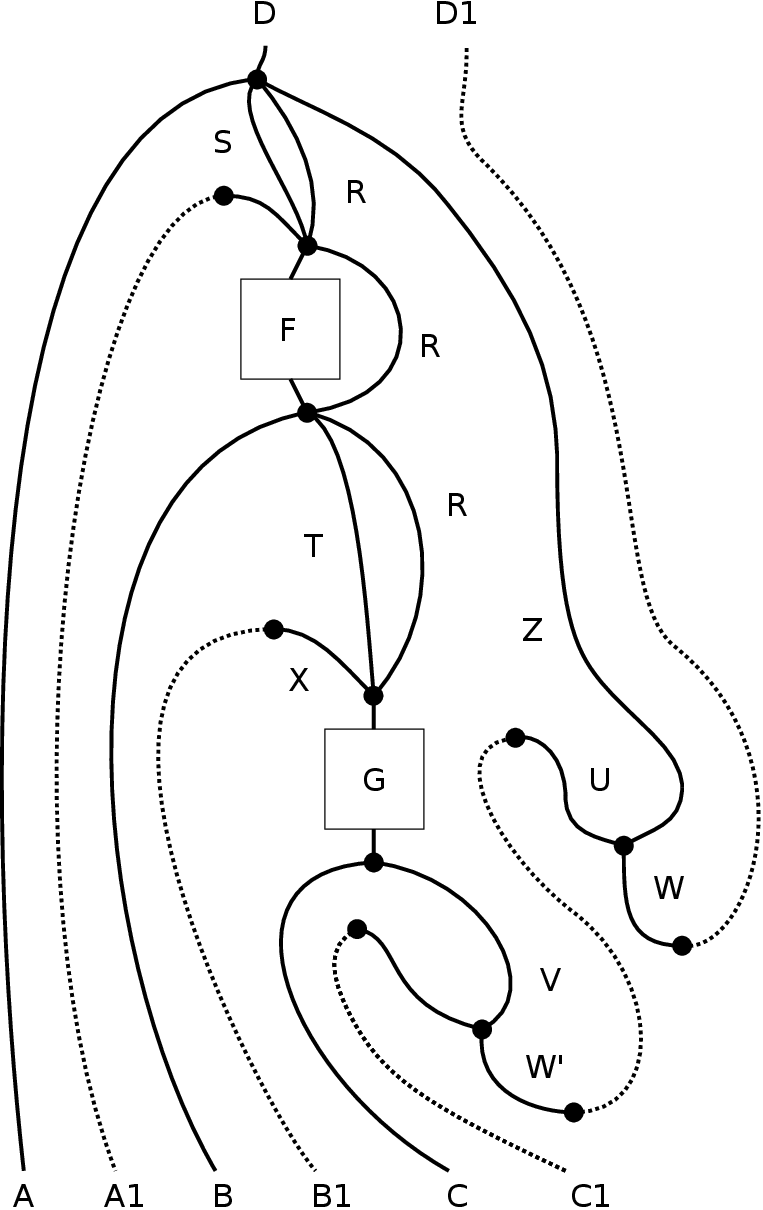}
}}
\]
For the last passage, we used~\eqref{Eq:def-f} and~\eqref{Eq:def-f-1} again to rewrite several 1-cells, and we used our choice of adjunction $\cat F(f^{-1})\dashv \cat F(f)$ for the strictly invertible~$f\in \JJ$. The next step uses the naturality of $\fun_\cat F$ to merge the three strands of~$\cat F x$:
\[
\quad = \quad
\vcenter { \hbox{
\psfrag{A}[Bc][Bc]{\scalebox{1}{\scriptsize{$\cat Fu$}}}
\psfrag{A1}[Bc][Bc]{\scalebox{1}{\scriptsize{$(\cat Fi)_!$}}}
\psfrag{B}[Bc][Bc]{\scalebox{1}{\scriptsize{$\cat Fv$}}}
\psfrag{B1}[Bc][Bc]{\scalebox{1}{\scriptsize{$(\cat Fj)_!$}}}
\psfrag{C}[Bc][Bc]{\scalebox{1}{\scriptsize{$\cat Fw$}}}
\psfrag{C1}[Bc][Bc]{\scalebox{1}{\scriptsize{$(\cat Fk)_!$}}}
\psfrag{D}[Bc][Bc]{\scalebox{1}{\scriptsize{$\cat Fu\tilde vxf$}}}
\psfrag{D1}[Bc][Bc]{\scalebox{1}{\scriptsize{$(\cat Fk sf )_!$}}}
\psfrag{V}[Bc][Bc]{\scalebox{1}{\scriptsize{$\cat Fs$}}}
\psfrag{R}[Bc][Bc]{\scalebox{1}{\scriptsize{$\cat F x$}}}
\psfrag{X}[Bc][Bc]{\scalebox{1}{\scriptsize{$\cat F\tilde i$}}}
\psfrag{Y}[Bc][Bc]{\scalebox{1}{\scriptsize{$\;\;\cat F\tilde v$}}}
\psfrag{Z}[Bc][Bc]{\scalebox{1}{\scriptsize{$\cat Ff$}}}
\psfrag{F}[Bc][Bc]{\scalebox{1}{\scriptsize{$\cat F\gamma$}}}
\psfrag{G}[Bc][Bc]{\scalebox{1}{\scriptsize{$\cat F\delta$}}}
\includegraphics[scale=.4]{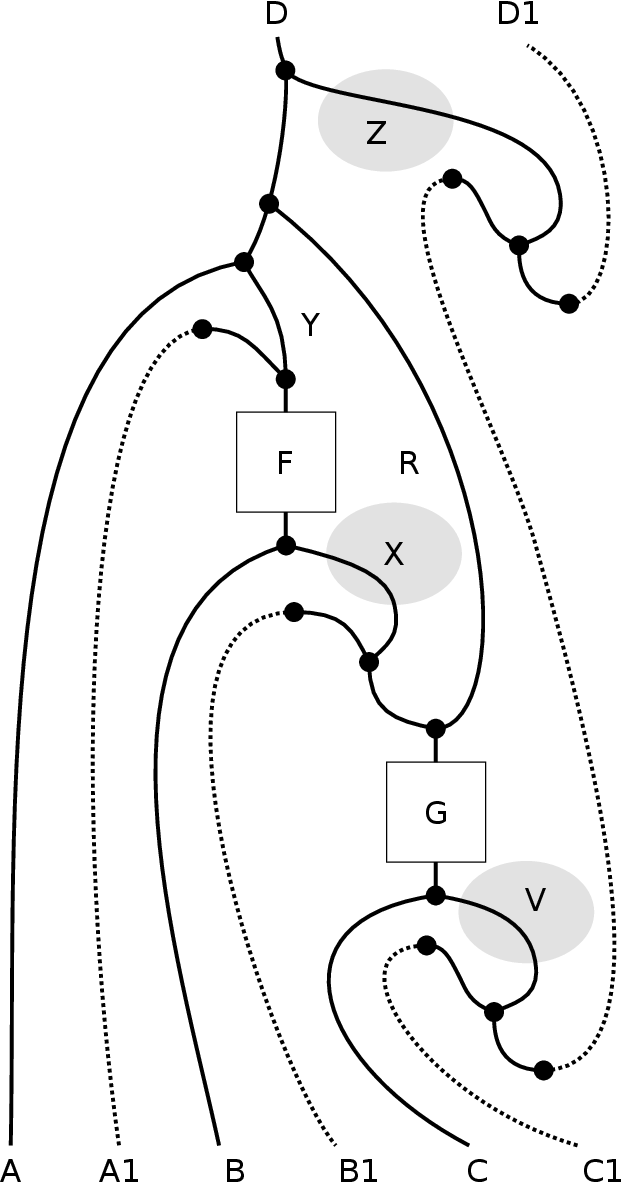}
}}
\quad = \quad
\vcenter { \hbox{
\psfrag{A}[Bc][Bc]{\scalebox{1}{\scriptsize{$\cat Fu$}}}
\psfrag{A1}[Bc][Bc]{\scalebox{1}{\scriptsize{$(\cat Fi)_!$}}}
\psfrag{B}[Bc][Bc]{\scalebox{1}{\scriptsize{$\cat Fv$}}}
\psfrag{B1}[Bc][Bc]{\scalebox{1}{\scriptsize{$(\cat Fj)_!$}}}
\psfrag{C}[Bc][Bc]{\scalebox{1}{\scriptsize{$\cat Fw$}}}
\psfrag{C1}[Bc][Bc]{\scalebox{1}{\scriptsize{$(\cat Fk)_!$}}}
\psfrag{D}[Bc][Bc]{\scalebox{1}{\scriptsize{$\cat Fu\tilde vxf$}}}
\psfrag{D1}[Bc][Bc]{\scalebox{1}{\scriptsize{$(\cat Fk sf )_!$}}}
\psfrag{X}[Bc][Bc]{\scalebox{1}{\scriptsize{$\cat Fu\tilde vx$}}}
\psfrag{R}[Bc][Bc]{\scalebox{1}{\scriptsize{$\cat F x$}}}
\psfrag{S}[Bc][Bc]{\scalebox{1}{\scriptsize{$\cat Fu\tilde v$}}}
\psfrag{Y}[Bc][Bc]{\scalebox{1}{\scriptsize{$\;\;\cat F\tilde v$}}}
\psfrag{Z}[Bc][Bc]{\scalebox{1}{\scriptsize{$\cat F\!f^{-\!1}\;\;$}}}
\psfrag{V}[Bc][Bc]{\scalebox{1}{\scriptsize{$(\cat Fj\tilde i)_!$}}}
\psfrag{U}[Bc][Bc]{\scalebox{1}{\scriptsize{$\;\;(\!\cat F\!s)_!$}}}
\psfrag{W}[Bc][Bc]{\scalebox{1}{\scriptsize{$\;\;(\cat Fks)_!$}}}
\psfrag{K}[Bc][Bc]{\scalebox{1}{\scriptsize{$\;\;(\cat F\tilde i)_!$}}}
\psfrag{F}[Bc][Bc]{\scalebox{1}{\scriptsize{$\cat F\gamma$}}}
\psfrag{G}[Bc][Bc]{\scalebox{1}{\scriptsize{$\cat F\delta$}}}
\includegraphics[scale=.4]{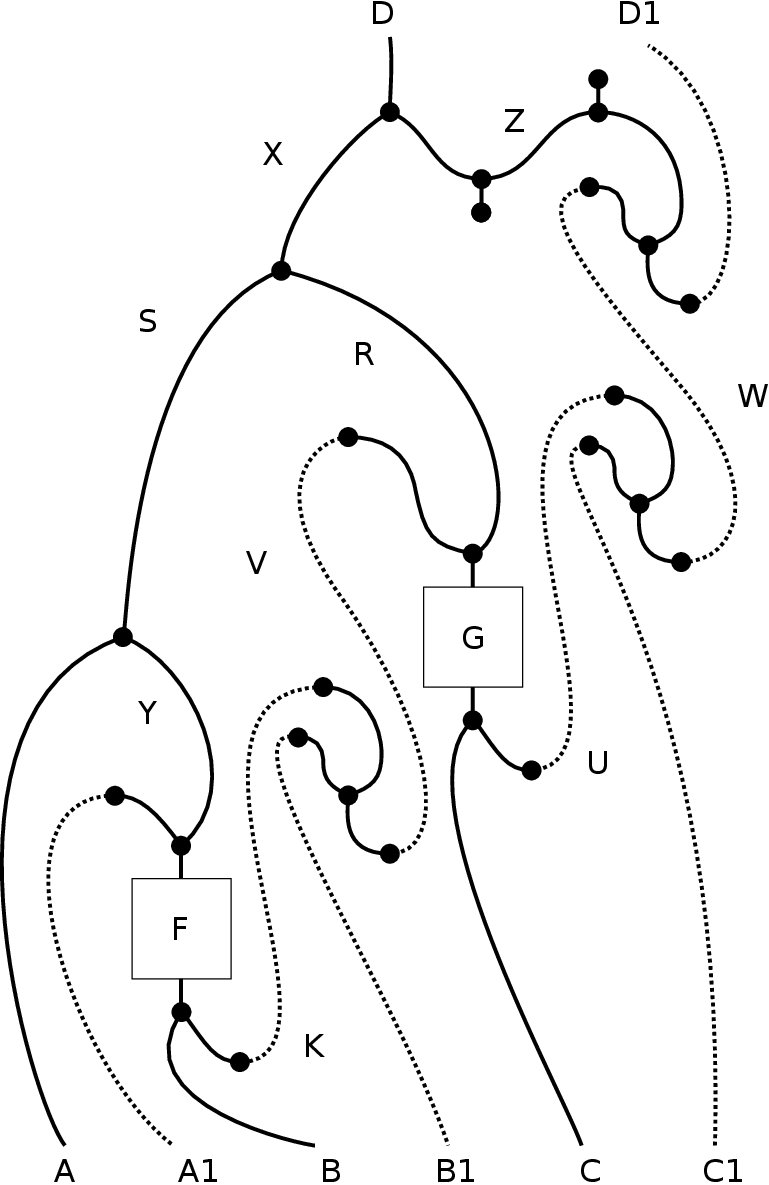}
}}
\]
Finally, we conclude by introducing three adjunction zig-zags; the result is precisely the right-hand side of~\eqref{Eq:assoc-for-G}; for this, we use at the top layer that $(\cat Ff)_!=\cat Ff^{-1}$ by construction.

This completes the verification that $\cat G$ is a pseudo-functor as claimed.
\end{proof}

\bigbreak
\section{Bicategorical upgrades of the universal property}
\label{sec:UP-Span-bicat}%
\medskip

Let $\GG$ and~$\JJ$ be as in \Cref{Hyp:G_and_I_for_Span}  (see also \Cref{Rem:less-hyps}) and let~$\cat{B}$ be any bicategory. The universal property given in \Cref{Thm:UP-Span} (together with the strictification theorem for bicategories; see \Cref{Rem:strictification}) determines when a pseudo-functor $\cat{F}\colon \GG^\op\to \cat{B}$ factors via the bicategory~$\Span=\Span(\GG;\JJ)$. We now want to understand what happens with 1-cells (pseudo-natural transformations) and 2-cells (modifications); see \Cref{Ter:Hom_bicats}. To facilitate the discussion we baptize the pseudo-functors in question.

\begin{Def}
\label{Def:J_!-pseudo-functor}%
A pseudo-functor $\cat{F}\colon \GG^\op\too \cat{B}$ is called a \emph{$\JJ_!$-pseudo-functor} if it satisfies conditions~\eqref{it:UP-Span-a} and~\eqref{it:UP-Span-b} of \Cref{Thm:UP-Span}, namely:
\begin{enumerate}[\rm(a)]
\smallbreak
\item
For every 1-cell $i\in \JJ$, the 1-cell $\cat{F}i$ admits a left adjoint $(\cat{F}i)_!$ in~$\cat{B}$.\,(\footnote{\,Adapting~\cite{DawsonParePronk04}, we could say that such~$\cat{F}$ are \emph{$\JJ$-dexter}.})
\smallbreak
\item
For any comma square $\gamma$ along an $i\in \JJ$, its mate $\gamma_!$ is invertible:
\begin{align*}
\xymatrix@C=14pt@R=14pt{
& i/v \ar[ld]_-{\tilde v} \ar[dr]^-{\tilde i} & \\
X \ar[dr]_i \ar@{}[rr]|{\oEcell{\gamma}} && Y \ar[dl]^v \\
&Z &
}
\quad \quad \quad \quad
\xymatrix@C=14pt@R=14pt{
& \cat{F} (i/v)
\ar[dr]^-{(\cat{F}\tilde i)_!} & \\
\cat{F} X
\ar[rd]_-{(\cat{F} i)_!}
 \ar[ur]^-{\cat{F}\tilde v}
 \ar@{}[rr]|{\Scell\, \gamma_!} &&
 \cat{F} Y. \\
& \cat{F} Z
\ar[ru]_-{\cat{F} v} &
}
\end{align*}
\end{enumerate}
The short name ``$\JJ_!$-pseudo-functor" certainly conveys the idea of~(a). We stress that the BC-property~(b) is an integral part of the definition.
\end{Def}

Consider two pseudo-functors $\cat G_1,\cat G_2\colon \Span\to \cat{B}$, or equivalently by \Cref{Thm:UP-Span} the associated $\JJ_!$-pseudo-functors $\cat G_1\circ (-)^*$ and $\cat G_2 \circ (-)^*\colon \GG^\op\to \cat{B}$. Starting with a pseudo-natural transformation $\cat G_1\Rightarrow \cat G_2$, we can wonder what kind of transformation $t\colon \cat G_1\circ (-)^*\Rightarrow \cat G_2 \circ (-)^*$ we get by restriction. As we shall prove, the answer is the following class of transformations.

\begin{Def}
\label{Def:J_!-strong}%
Let $t\colon \cat{F}_1 \Rightarrow \cat{F}_2$ be a (strong) pseudo-natural transformation between two $\JJ_!$-pseudo-functors $\GG^{\op}\to \cat{B}$ (\Cref{Def:J_!-pseudo-functor}). We say that $t$ is \emph{$\JJ_!$-strong} if for every $i\in \JJ$ the following mate~$(t_i^{-1})_!$
\begin{equation}
\label{eq:t_i_!}%
\vcenter { \hbox{
\xymatrix{
\cat{F}_1 X
 \ar[r]^-{t_X}
 \ar[d]_-{(\cat{F}_1 i)_!} &
 \cat{F}_2 X
 \ar[d]^-{(\cat{F}_2 i)_!} \\
\cat{F}_1 Y
\ar@{}[ur]|{\SWcell \; (t_i^{-1})_!}
 \ar[r]_-{t_Y} &
 \cat{F}_2 Y
}
}}
\quad \stackrel{\textrm{def.}}{=}
\vcenter { \hbox{
\xymatrix{
\ar@{}[dr]|{\SEcell} &
 \cat{F}_1 X
\ar@{}[dr]|{\SEcell\; t_i^{-1}}
 \ar[r]^-{t_X} &
 \cat{F}_2 X
 \ar@{}[dr]|{\SEcell}
 \ar[r]^-{(\cat{F}_2 i)_!} &
 \cat{F}_2 Y
\\
\cat{F}_1X
 \ar[r]_-{(\cat{F}_1 i)_!}
\ar@/^3ex/@{=}[ur] &
 \cat{F}_1 Y
 \ar[r]_-{t_Y}
 \ar[u]^-{\cat{F}_1 i} &
 \cat{F}_2 Y
 \ar[u]_-{\cat{F}_2 i}
 \ar@/_3ex/@{=}[ru] &
}
}}
\end{equation}
is invertible: $(t_i\inv)_!\colon (\cat{F}_2 i)_!\,t_X\isoEcell t_Y\,(\cat{F}_1 i)_!$.
\end{Def}

\begin{Rem}
\label{Rem:pseudo-nat-on-adjoints}%
In the spirit of \Cref{Rem:pseudo-func-of-adjoints}, suppose that $\cat{F}_1\,,\,\cat{F}_2\colon \GG^\op\to \cat{B}$ are two $\JJ_!$-pseudo-functors and assemble the left adjoints $(\cat{F}_1 i)_!$ and $(\cat{F}_2 i)_!$ into pseudo-functors $(\cat{F}_1)_!\colon \JJ^{\co}\to \cat{B}$ and $(\cat{F}_2)_!\colon \JJ^{\co}\to \cat{B}$. Consider now a transformation $t\colon \cat{F}_1\Rightarrow \cat{F}_2$. One hopes that $t$ induces a transformation $t_!\colon (\cat{F}_1)_!\Rightarrow (\cat{F}_2)_!$. Since $(\cat{F}_1)_!$ and $(\cat{F}_2)_!$ coincide with $\cat{F}_1$ and $\cat{F}_2$ on 0-cells, we can use $t_X$ to define $(t_!)_X$ on 0-cells. Now for every $i\in\JJ$, we define the 1-cell component $(t_!)_i:=(t_i\inv)_!$ of the transformation~$t_!$ via the above recipe~\eqref{eq:t_i_!}. Note that \emph{a~priori} there is no possibility of forming a mate with~$t_i$ itself. It is therefore important that $t$ is an actual \emph{strong} pseudo-natural transformation not merely an (op)lax one, at least on the 1-cells in~$\JJ$.

Now the punchline is that the above definition of $t_!\colon (\cat{F}_1)_!\Rightarrow(\cat{F}_2)_!$ does not yield a \emph{strong} pseudo-natural transformation but only an \emph{oplax} one (\Cref{Ter:Hom_bicats}\,\eqref{it:lax-transformation}), \ie its value $(t_!)_i$ on 1-cells~$i$ is not invertible \apriori. The transformation~$t\colon \cat{F}_1\Rightarrow \cat{F}_2$ is $\JJ_!$-strong in the sense of \Cref{Def:J_!-strong} precisely when $t_!$ is a strong transformation $(\cat{F}_1)_!\Rightarrow (\cat{F}_2)_!\colon \JJ^\co\to \cat{B}$. This characterization should clarify the terminology ``$\JJ_!$-strong''. We emphasize that all $\JJ_!$-strong transformations are assumed strong in the usual sense to begin with (for $t$ itself on~$\GG$), and that the $\JJ_!$-strength pertains to~$t_!$ on~$\JJ$.
\end{Rem}

\begin{Rem}
We leave to the reader the verification that the vertical and horizontal compositions of $\JJ_!$-strong transformations remain $\JJ_!$-strong.
\end{Rem}

\begin{Not}
\label{Not:I_!-PsFun}%
\index{PsFunJ@$\PsFunJ$}%
\index{$psfunj$@$\PsFunJindex$}%
We denote by
\[
\PsFunJ(\GG^\op,\cat{B})
\]
the sub-bicategory of $\PsFun(\GG^\op,\cat{B})$ with $\JJ_!$-pseudo-functors (\Cref{Def:J_!-pseudo-functor}) as 0-cells, with $\JJ_!$-strong (\Cref{Def:J_!-strong}) pseudo-natural transformations as 1-cells and all modifications between them as 2-cells (\Cref{Ter:Hom_bicats}).
\end{Not}

\begin{Thm}
\label{Thm:UP-PsFun-Span}%
Precomposition with the pseudo-functor $(-)^*\colon \GG^{\op}\to \Span(\GG;\JJ)$ induces a biequivalence of bicategories
\begin{equation}
\label{eq:UP-PsFun-Span}%
\PsFun\big(\Span(\GG;\JJ), \cat{B}\,\big) \stackrel{\sim}{\longrightarrow} \PsFunJ^{}( \GG^{\op}, \cat{B})\,.
\end{equation}
If $\cat{B}$ happens to be a 2-category, then the above is a strict 2-functor between two 2-categories, and said 2-functor is furthermore \emph{locally strict}, \ie it gives isomorphisms (rather than just equivalences) between all Hom categories.
\end{Thm}

\begin{proof}
Let us begin by assuming that the target $\cat{C} :=\cat B$ is a 2-category, so both ends of~\eqref{eq:UP-PsFun-Span} are 2-categories. Let us also forget modifications for the moment, and consider these two pseudo-functor 2-categories simply as 1-categories.

By \Cref{Thm:UP-Span}, we know that precomposition with~$(-)^*\colon \GG^\op\to \Span(\GG;\JJ)$ induces a well-defined functor
\begin{equation} \label{eq:first_functor_welldef}
\PsFun(\Span(\GG;\JJ), \cat{C}) \longrightarrow \PsFunJ(\GG^{\op}, \cat{C})
\end{equation}
that is surjective on objects. Let us now discuss morphisms.

For clarity, to avoid entire discussions happening `in indices', we write $t(u)$ instead of $t_u$ everywhere in this proof for the 1-cell component of a transformation~$t$.

Let us explain how, for every $\cat{G}_1,\cat{G}_2\in\PsFun(\Span,\cat{C})$, each $\JJ_!$-strong transformation $s\colon \cat{G}_1\circ (-)^*\Rightarrow\cat{G}_2\circ (-)^*$ extends uniquely to a transformation $t\colon \cat{G}_1\Rightarrow \cat{G}_2$.
Clearly such a $t$ is determined on objects and on 1-cells of the form~$u^*$ by $s=t|_{\GG}$.

By Lemma~\ref{Lem:tr_inv_mate}, the component $t(i_!)$ has to be the left mate $(s(i)^{-1})_!$ of $t(i^*)\inv=s(i)^{-1}$ for each $i\in \JJ$.
Furthermore, the functoriality axiom of transformations (with respect to the composite $i_!\circ u^*$) and the naturality axiom (with respect to the isomorphism $\zeta\colon i_!u^*\stackrel{\sim}{\to}i_!\circ u^*$ of \Cref{Rem:notation_assoc}) uniquely determine the component $t(i_!u^*)$ at every 1-cell $i_!u^*$ of $\Span$ by the following pasting:
\begin{align} \label{eq:decomp_comp}
t(i_! u^*) \; :=
\vcenter{\xymatrix@L=.4ex{
& \ar[rrrrr]^-t
 \ar[dddd]|<<<<<<<<<<{\cat G_1 (i_! \circ u^*)}
 \ar[ddrr]^{\cat G_1 u^*} \ar@/_8ex/[dddd]_<<<<<<<<<<<<<{\cat G_1(i_!u^*)}
 \ar@{}[ddddl]|{ \underset{\scriptstyle \cat G_1 \zeta^{-1}}{\overset{\scriptstyle \simeq}{\SWcell}} }
 \ar@{}[ddddrr]|{ \underset{\scriptstyle \fun_{\cat G_1}}{\overset{\scriptstyle \simeq}{\SWcell}} }
 \ar@{}[rrrrrdd]|{ \underset{\scriptstyle s(u)}{\SWcell} } &&&&&
 \ar[ddll]_{\cat G_2 u^*}
 \ar[dddd]|<<<<<<<<<<{\cat G_2(i_!\circ u^*)}
 \ar@/^8ex/[dddd]^<<<<<<<<<<<<<{\cat G_2(i_!u^*)}
 \ar@{}[ddddll]|{ \underset{\scriptstyle \fun^{-1}_{\cat G_2}}{\overset{\scriptstyle \simeq}{\SWcell}} }
 \ar@{}[ddddr]|{ \underset{\scriptstyle \cat G_2 \zeta}{\overset{\scriptstyle \simeq}{\SWcell}} } & \\
&&&&&&& \\
&&&
 \ar[r]^-{t}
 \ar[ddll]^{\cat G_1 i_!}
 \ar@{}[ddr]|{ \overset{\scriptstyle t(i_!) \,=\, (s(i)^{-1})_!}{\SWcell} } &
 \ar[ddrr]_{\cat G_2 i_!} &&& \\
&&&&&&& \\
& \ar[rrrrr]_-{t} &&&&&&
}}
\end{align}
Hence the extension $t\colon \cat G_1\Rightarrow \cat G_2$ is uniquely determined by $\cat G_1$, $\cat G_2$ and~$s$. Straightforward but lengthy calculations (which we defer until Lemma~\ref{Lem:t-welldef}) show that $t$ is indeed a transformation, as claimed. Clearly $t$ is strong precisely because $s=t|_{\GG}$ is $\JJ_!$-strong. Direct verification shows that the above construction $s\mapsto t$ is the inverse of $t\mapsto t\circ(-)^*$. So~\eqref{eq:first_functor_welldef} is indeed an isomorphism of (1-)categories.

Let us now consider modifications. Every modification $M\colon t\Rrightarrow t'$ of transformations $t,t'\colon \cat{G}_1\Rightarrow\cat{G}_2$ of pseudo-functors $\cat{G}_1,\cat{G}_2\colon \Span(\GG;\JJ)\to \cat{C}$ restricts to a modification of the restricted transformations $t,t'\colon \cat{G}_1\circ (-)^* \Rightarrow\cat{G}_2\circ (-)^*$. In fact, the \emph{data} for $M$ or for its restriction is the same: a family
$\{M_G \colon t_G\Rightarrow t'_G\}_{G\in \GG_0}$ of 2-cells of~$\cat{C}$ indexed by $\GG_0=\Span(\GG;\JJ)_0$.
Suppose we are given such a family~$M$. We now claim that if $M$ satisfies the modification axiom
\begin{align} \label{eq:modif_axiom_special}
\vcenter {\hbox{
\xymatrix{
\cat{G}_1 G \ar[r]^-{t_G}
\ar[d]_-{\cat{G}_1 u^*} &
 \cat{G}_2 G
 \ar[d]^-{\cat{G}_2 u^*}
 \ar@{}[dl]|{\SWcell\;tu^*} \\
\cat{G}_1 P
 \ar[r]^-{t_P}
 \ar@/_8ex/[r]_-{t'_P} &
 \cat{G}_2 P
 \ar@{}[dl]|{\SWcell M_P} \\
 &
}
}}
\quad\quad = \quad\quad
\vcenter {\hbox{
\xymatrix{
 &
 \ar@{}[dl]|{\SWcell M_G} \\
 \cat{G}_1 G
 \ar@/^8ex/[r]^-{t_G}
 \ar[r]_-{t'_G}
 \ar[d]_-{\cat{G}_1 u^*} &
 \cat{G}_2 G
 \ar@{}[dl]|{\SWcell\;t'u^*}
 \ar[d]^-{\cat{G}_2 u^*} \\
\cat{G}_1 P \ar[r]_-{t'_P} & \cat{G}_2P
}
}}
\end{align}
for all spans of the form $u^*\colon G\to P$ (\ie if it defines a modification of the restricted transformations), then it also satisfies
\begin{align*}
\vcenter {\hbox{
\xymatrix{
\cat{G}_1 G \ar[r]^-{t_G}
\ar[d]_-{\cat{G}_1 (i_!u^*)} &
 \cat{G}_2 G
 \ar[d]^-{\cat{G}_2 (i_!u^*)}
 \ar@{}[dl]|{\SWcell \;t(i_!u^*)} \\
\cat{G}_1 H
 \ar[r]^-{t_H}
 \ar@/_8ex/[r]_-{t' H} &
 \cat{G}_2 H
 \ar@{}[dl]|{\SWcell M_H} \\
 &
}
}}
\quad\quad = \quad\quad
\vcenter {\hbox{
\xymatrix{
 &
 \ar@{}[dl]|{\SWcell M_G} \\
 \cat{G}_1 G
 \ar@/^8ex/[r]^-{t_G}
 \ar[r]_-{t'_G}
 \ar[d]_-{\cat{G}_1 (i_!u^*)} &
 \cat{G}_2 G
 \ar@{}[dl]|{\SWcell\;t' (i_!u^*)}
 \ar[d]^-{\cat{G}_2 (i_!u^*)} \\
\cat{G}_1 H \ar[r]_-{t' H} & \cat{G}_2H
}
}}
\end{align*}
for general spans $i_!u^*\colon G \stackrel{u}{\leftarrow} P \stackrel{i}{\to} H$ (\ie $M$ automatically defines a modification of the original transformations of pseudo-functors on~$\Span(\GG;\JJ)$). Indeed, by~\eqref{eq:decomp_comp} we may reduce the verification to spans of the form~$i_!$.
Now recall that the components $t(i_!), t'(i_!)$ at $i_!$ can be recovered as the left mates of~$t(i^*)\inv, t'(i^*)\inv$. The diagram
\[
\xymatrix{
\cat{G}_1 P
 \ar[r]^-{\cat{G}_1 i_!}
 \ar@/_3ex/[ddr]_-{\Id} &
\cat{G}_1 H
 \ar@/^2ex/[rr]^-{t_H}
 \ar@{}[rr]|{\SWcell\;M_H}
 \ar@/_2ex/[rr]_-{t'_H}
 \ar[dd]_-{\cat{G}_1 i^*} &&
\cat{G}_2 H
 \ar@/^3ex/[rdd]^-{\Id}
 \ar[dd]^-{\cat{G}_2 i^*} & \\
\ar@{}[ur]|{\NEcell} & &&& \\
 & \cat{G}_1 P
 \ar@{}[uur]|{ \, (t'i^*)\inv\!\NEcell}
 \ar@{}[uurr]_-{(ti^*)\inv\!\NEcell}
 \ar@/_2ex/[rr]_-{t'_P}
 \ar@{}[rr]|{\SWcell\;M_P}
 \ar@{..>}@/^2ex/[rr]^-{t_P} &&
\cat{G}_2 P
 \ar@{}[ur]|{\NEcell}
 \ar[r]_-{\cat{G}_2 i_!} &
\cat{G}_2 H
}
\]
computes both mates in parallel and shows that in order to derive the axiom for $i_!$ it suffices to apply~\eqref{eq:modif_axiom_special} with $u=i$.
This shows that~\eqref{eq:first_functor_welldef} upgrades to a 2-functor which, as we have seen, must be a locally strict biequivalence. This proves the theorem in the case the target is a strict 2-category.

Finally if $\cat B$ is any bicategory, choose a biequivalence $\cat B\stackrel{\sim}{\to} \cat{C}$ to some 2-category~$\cat{C}$ (\Cref{Rem:strictification}). It induces a commutative square of pseudo-functors
\[
\xymatrix{
\PsFun\big(\Span(\GG;\JJ), \cat{B}\big) \ar[r] \ar[d]_\simeq &
 \PsFunJ(\GG^{\op}, \cat{B}) \ar[d]^\simeq \\
\PsFun\big(\Span(\GG;\JJ), \cat{C}\big) \ar[r]^-\simeq & \PsFunJ(\GG^{\op}, \cat{C})
}
\]
where the bottom one is a biequivalence by the special case of the theorem we have already proved, and the vertical ones are biequivalences by Remark~\ref{Rem:PsFun_bieq}. We deduce that the top pseudo-functor is also a biequivalence (note however there is no reason for it to be locally strict).

The proof of the theorem is almost complete. We still owe the reader the following:

\begin{Lem} \label{Lem:t-welldef}
With notation as above, the 2-cells in~\eqref{eq:decomp_comp} define an oplax transformation $t\colon \cat G_1\Rightarrow\cat G_2$.
\end{Lem}

\begin{proof}
Let $\cat F_1:=\cat G_1\circ (-)^*$ and $\cat F_2:=\cat G_2\circ (-)^*\colon \GG^\op\to \cat C$ be the restrictions to $\GG^\op$ of our two pseudo-functors $\cat G_1,\cat G_2\colon \Span\to \cat C$. We are also given a transformation $t=(t_G,t(u))_{G\in \GG_0,u\in \GG_1}\colon \cat F_1\Rightarrow \cat F_2$ (or actually just an oplax transformation~$t$ with the property that $t(i)$ is invertible for all~$i\in \JJ$) and we must show that the 2-cells in~\eqref{eq:decomp_comp} extend $t$ to a transformation $t=(t_G,t(i_!u^*))\colon \cat G_1\Rightarrow \cat G_2$. For this, we may as well assume that $\cat G_1$ and $\cat G_2$ are \emph{equal} to the extensions on $\Span$ of $\cat F_1$ and $\cat F_2$ as in \Cref{Cons:UP-Span}, since they are isomorphic as pseudo-functors. We will also omit the canonical isomorphisms $i_!u^*\cong i_!\circ u^*$ from the computations, as they immediately cancel out anyway.

We compute with string diagrams, as before using dotted strings to denote the images in~$\cat C$ under $\cat G_1,\cat G_2$ of the 1-cells in $\Span(\GG;\JJ)$ of the form $i_!$ for $i\in \JJ$ (\ie the adjoint 1-cells $(\cat F_1 i)_!$ and $(\cat F_2 i)_!$).
The string forms of the transformation axioms can be found in \Cref{Exa:trans}.

Thus~\eqref{eq:decomp_comp} becomes
\[
t(i_! u^*)
\quad = \quad\quad
\vcenter { \hbox{
\psfrag{A}[Bc][Bc]{\scalebox{1}{\scriptsize{$t$}}}
\psfrag{B}[Bc][Bc]{\scalebox{1}{\scriptsize{$\cat G_2(i_! u^*)$}}}
\psfrag{C}[Bc][Bc]{\scalebox{1}{\scriptsize{$\cat G_1(i_! u^*)$}}}
\psfrag{D}[Bc][Bc]{\scalebox{1}{\scriptsize{$t$}}}
\psfrag{I}[Bc][Bc]{\scalebox{1}{\scriptsize{$t$}}}
\psfrag{B1}[Bc][Bc]{\scalebox{1}{\scriptsize{$\cat F_2u$\;\;}}}
\psfrag{B2}[Bc][Bc]{\scalebox{1}{\scriptsize{\;\;$\cat F_2i_!$}}}
\psfrag{C1}[Bc][Bc]{\scalebox{1}{\scriptsize{$\cat F_1u$\;\;}}}
\psfrag{C2}[Bc][Bc]{\scalebox{1}{\scriptsize{\;\;$\cat F_1 i_!$}}}
\psfrag{F}[Bc][Bc]{\scalebox{1}{\scriptsize{$t(u^*)$}}}
\psfrag{H}[Bc][Bc]{\scalebox{1}{\scriptsize{$t(i_!)$}}}
\psfrag{G1}[Bc][Bc]{\scalebox{1}{\scriptsize{\;\;\;\;\;\;\;$\fun^{-1}\!=\id$}}}
\psfrag{G2}[Bc][Bc]{\scalebox{1}{\scriptsize{$\id=\fun\;\;\;\;$}}}
\includegraphics[scale=.4]{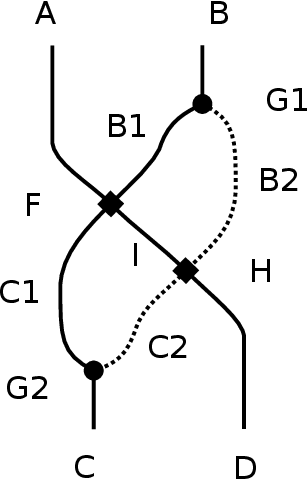}
}}
\quad\quad\quad = \quad\quad
\vcenter { \hbox{
\psfrag{A}[Bc][Bc]{\scalebox{1}{\scriptsize{$t$}}}
\psfrag{B}[Bc][Bc]{\scalebox{1}{\scriptsize{$\cat G_2(i_!\circ u^*)$}}}
\psfrag{C}[Bc][Bc]{\scalebox{1}{\scriptsize{$\cat G_1(i_!\circ u^*)$}}}
\psfrag{D}[Bc][Bc]{\scalebox{1}{\scriptsize{$t$}}}
\psfrag{I}[Bc][Bc]{\scalebox{1}{\scriptsize{$t$}}}
\psfrag{u}[Bc][Bc]{\scalebox{1}{\scriptsize{$\eta$}}}
\psfrag{n}[Bc][Bc]{\scalebox{1}{\scriptsize{$\varepsilon$}}}
\psfrag{B1}[Bc][Bc]{\scalebox{1}{\scriptsize{$\cat F_2u$\;}}}
\psfrag{B2}[Bc][Bc]{\scalebox{1}{\scriptsize{\;\;$\cat F_2i_!$}}}
\psfrag{C1}[Bc][Bc]{\scalebox{1}{\scriptsize{$\cat F_1u$\;\;}}}
\psfrag{C2}[Bc][Bc]{\scalebox{1}{\scriptsize{\;\;$\cat F_1 i_!$}}}
\psfrag{F}[Bc][Bc]{\scalebox{1}{\scriptsize{$t(u)$}}}
\psfrag{H}[Bc][Bc]{\scalebox{1}{\scriptsize{$\;\;\;t(i)^{-1}$}}}
\psfrag{G1}[Bc][Bc]{\scalebox{1}{\scriptsize{\;\;\;\;\;\;\;$\fun^{-1}\!=\id$}}}
\psfrag{G2}[Bc][Bc]{\scalebox{1}{\scriptsize{$\id=\fun\;\;\;\;$}}}
\includegraphics[scale=.4]{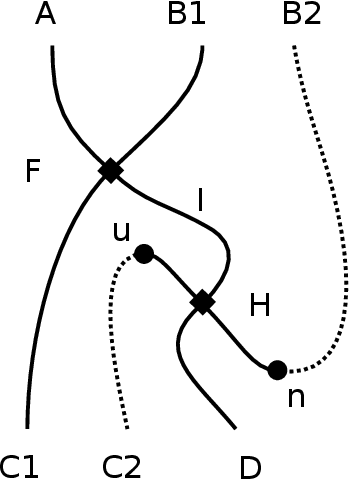}
}}
\]
where necessarily
$t(i_!)=(t(i^*)^{-1})_!=(t(i)^{-1})_!$ is the left mate of the inverse of~$t(i)$, as we have seen. (\Cf \Cref{Lem:tr_inv_mate}.)

Let $\alpha=[a,\alpha_1,\alpha_2]\colon i_!u^*\Rightarrow j_!v^*$ be any 2-cell of~$\Span$ as in~\eqref{eq:2-cell-of-Span}. The naturality axiom is verified by the following computation:
\[
\vcenter {\hbox{
\psfrag{F}[Bc][Bc]{\scalebox{1}{\scriptsize{$t(j_!v^*)$}}}
\psfrag{G}[Bc][Bc]{\scalebox{1}{\scriptsize{$\cat G_2 \alpha$}}}
\psfrag{A}[Bc][Bc]{\scalebox{1}{\scriptsize{$t$}}}
\psfrag{B}[Bc][Bc]{\scalebox{1}{\scriptsize{$\cat G_2 (i_!u^*)$}}}
\psfrag{C}[Bc][Bc]{\scalebox{1}{\scriptsize{$\cat G_1(j_!v^*)$}}}
\psfrag{D}[Bc][Bc]{\scalebox{1}{\scriptsize{$t$}}}
\psfrag{E}[Bc][Bc]{\scalebox{1}{\scriptsize{$t$}}}
\psfrag{H}[Bc][Bc]{\scalebox{1}{\scriptsize{$\cat G_2 (j_!v^*)$}}}
\includegraphics[scale=.4]{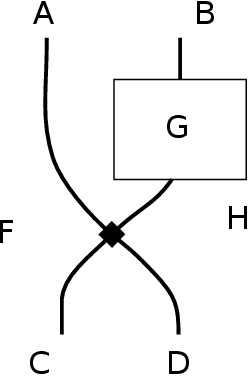}
}}
\quad\quad = \quad
\vcenter { \hbox{
\psfrag{I}[Bc][Bc]{\scalebox{1}{\scriptsize{$t$}}}
\psfrag{A1}[Bc][Bc]{\scalebox{1}{\scriptsize{$\cat F_2u$}}}
\psfrag{A2}[Bc][Bc]{\scalebox{1}{\scriptsize{$\cat F_2i_!$}}}
\psfrag{C1}[Bc][Bc]{\scalebox{1}{\scriptsize{$\cat F_1v$\;\;}}}
\psfrag{C2}[Bc][Bc]{\scalebox{1}{\scriptsize{\;\;$\cat F_1 j_!$}}}
\psfrag{F}[Bc][Bc]{\scalebox{1}{\scriptsize{$\cat F_2 \alpha_1$}}}
\psfrag{G}[Bc][Bc]{\scalebox{1}{\scriptsize{$\cat F_2 \alpha_2$}}}
\psfrag{L}[Bc][Bc]{\scalebox{1}{\scriptsize{$t(v)$}}}
\psfrag{H}[Bc][Bc]{\scalebox{1}{\scriptsize{$\;\;\;t(j)^{-1}$}}}
\psfrag{S'}[Bc][Bc]{\scalebox{1}{\scriptsize{$\cat F_2a_!$}}}
\psfrag{C'}[Bc][Bc]{\scalebox{1}{\scriptsize{$(\cat F_2ja)_!$}}}
\psfrag{U'}[Bc][Bc]{\scalebox{1}{\scriptsize{$\cat F_2j_!$}}}
\psfrag{R}[Bc][Bc]{\scalebox{1}{\scriptsize{$\cat F_2va\;\;\;$}}}
\psfrag{S}[Bc][Bc]{\scalebox{1}{\scriptsize{$\cat F_2a$}}}
\psfrag{T}[Bc][Bc]{\scalebox{1}{\scriptsize{$\cat F_2v$}}}
\psfrag{D}[Bc][Bc]{\scalebox{1}{\scriptsize{$\;\cat F_2j$}}}
\psfrag{C}[Bc][Bc]{\scalebox{1}{\scriptsize{$\;\;\;\cat F_2ja$}}}
\includegraphics[scale=.4]{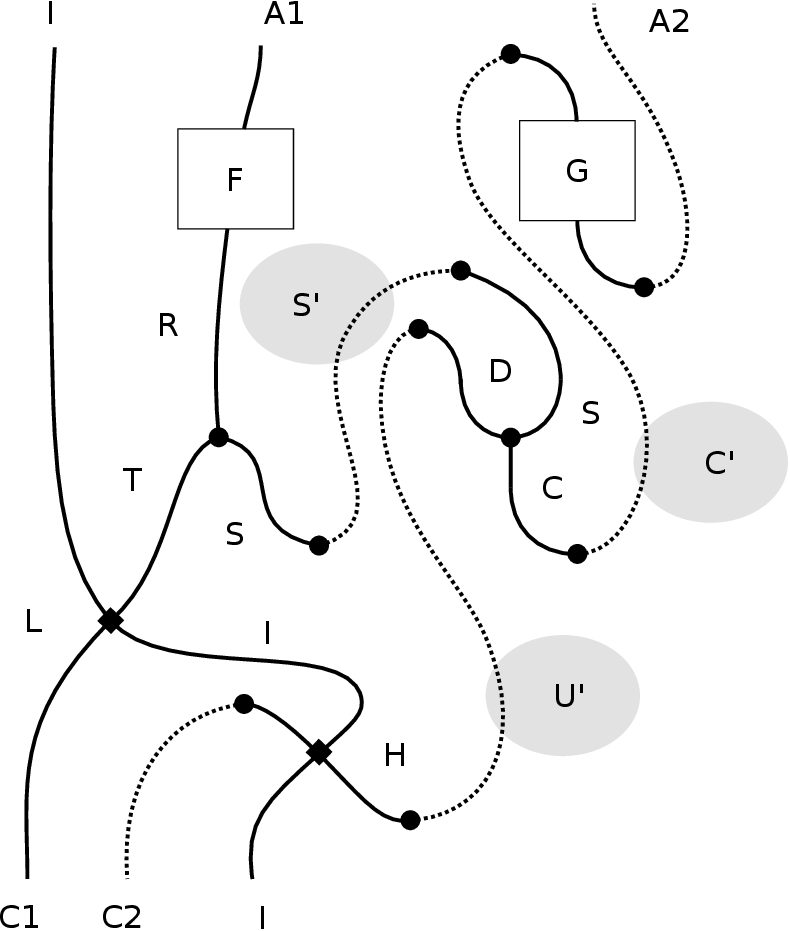}
}}
\quad\stackrel{3 \times \textrm{\eqref{Exa:strings-for-adjoints}}}{=}
\]
\[
\vcenter { \hbox{
\psfrag{I}[Bc][Bc]{\scalebox{1}{\scriptsize{$t$}}}
\psfrag{A1}[Bc][Bc]{\scalebox{1}{\scriptsize{$\cat F_2u$}}}
\psfrag{A2}[Bc][Bc]{\scalebox{1}{\scriptsize{$\cat F_2i_!$}}}
\psfrag{C1}[Bc][Bc]{\scalebox{1}{\scriptsize{$\cat F_1v$\;\;}}}
\psfrag{C2}[Bc][Bc]{\scalebox{1}{\scriptsize{\;\;$\cat F_1 j_!$}}}
\psfrag{F}[Bc][Bc]{\scalebox{1}{\scriptsize{$\cat F_2 \alpha_1$}}}
\psfrag{G}[Bc][Bc]{\scalebox{1}{\scriptsize{$\cat F_2 \alpha_2$}}}
\psfrag{S}[Bc][Bc]{\scalebox{1}{\scriptsize{$\fun_{\cat F_2}$}}}
\psfrag{S'}[Bc][Bc]{\scalebox{1}{\scriptsize{$\fun^{-1}_{\cat F_2}$}}}
\psfrag{L}[Bc][Bc]{\scalebox{1}{\scriptsize{$t(v)$}}}
\psfrag{H}[Bc][Bc]{\scalebox{1}{\scriptsize{$t(j)^{-1}$}}}
\psfrag{T}[Bc][Bc]{\scalebox{1}{\scriptsize{$\cat F_2a$}}}
\includegraphics[scale=.4]{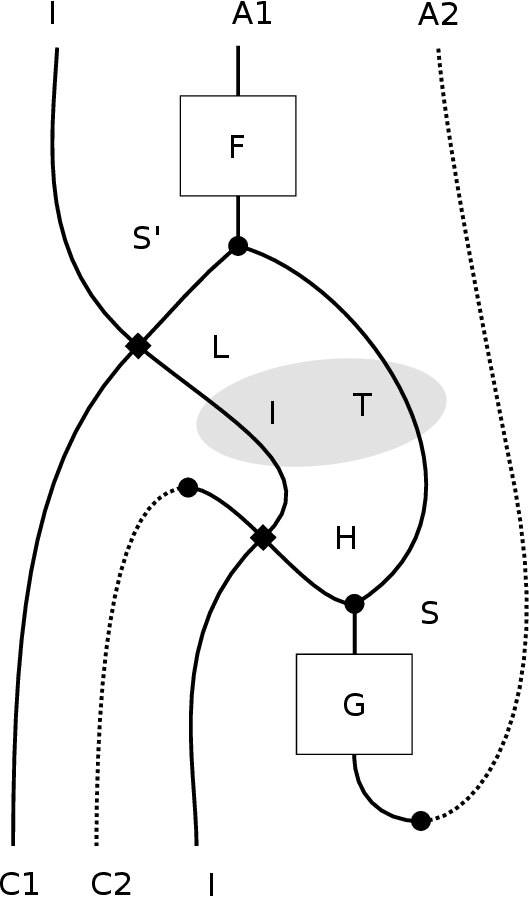}
}}
=
\vcenter { \hbox{
\psfrag{I}[Bc][Bc]{\scalebox{1}{\scriptsize{$t$}}}
\psfrag{A1}[Bc][Bc]{\scalebox{1}{\scriptsize{$\cat F_2u$}}}
\psfrag{A2}[Bc][Bc]{\scalebox{1}{\scriptsize{$\cat F_2i_!$}}}
\psfrag{C1}[Bc][Bc]{\scalebox{1}{\scriptsize{$\cat F_1v$\;\;}}}
\psfrag{C2}[Bc][Bc]{\scalebox{1}{\scriptsize{\;\;$\cat F_1 j_!$}}}
\psfrag{F}[Bc][Bc]{\scalebox{1}{\scriptsize{$\cat F_2 \alpha_1$}}}
\psfrag{G}[Bc][Bc]{\scalebox{1}{\scriptsize{$\cat F_2 \alpha_2$}}}
\psfrag{S}[Bc][Bc]{\scalebox{1}{\scriptsize{$\fun_{\cat F_2}$}}}
\psfrag{S'}[Bc][Bc]{\scalebox{1}{\scriptsize{$\fun^{-1}_{\cat F_2}$}}}
\psfrag{L}[Bc][Bc]{\scalebox{1}{\scriptsize{$t(v)$}}}
\psfrag{L1}[Bc][Bc]{\scalebox{1}{\scriptsize{$t(a)$}}}
\psfrag{L2}[Bc][Bc]{\scalebox{1}{\scriptsize{$t(a)^{-1}$}}}
\psfrag{H}[Bc][Bc]{\scalebox{1}{\scriptsize{$t(j)^{-1}$}}}
\psfrag{T}[Bc][Bc]{\scalebox{1}{\scriptsize{$\cat F_1a$}}}
\includegraphics[scale=.4]{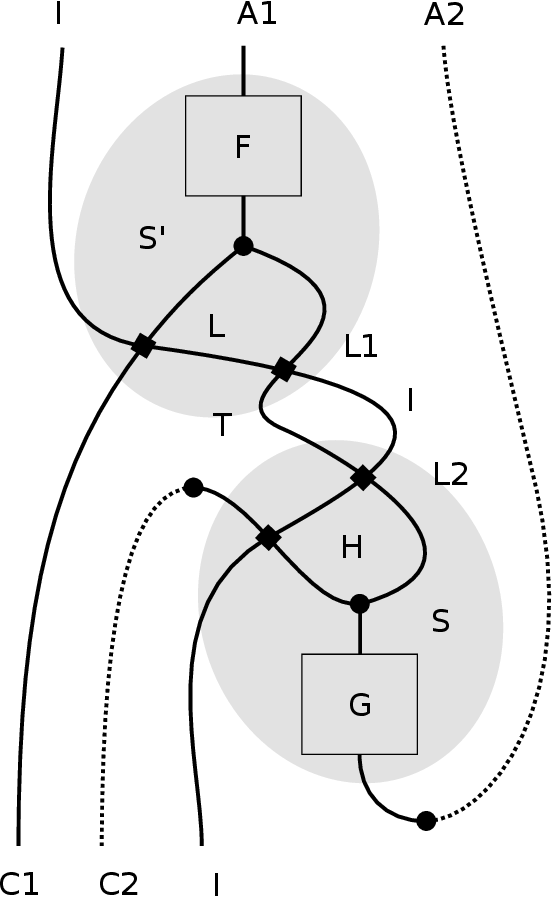}
}}
\underset{\textrm{for }t|_\GG}{\overset{\textrm{\eqref{Exa:trans}}}{=}}
\vcenter { \hbox{
\psfrag{I}[Bc][Bc]{\scalebox{1}{\scriptsize{$t$}}}
\psfrag{A1}[Bc][Bc]{\scalebox{1}{\scriptsize{$\cat F_2u$}}}
\psfrag{A2}[Bc][Bc]{\scalebox{1}{\scriptsize{$\cat F_2i_!$}}}
\psfrag{C1}[Bc][Bc]{\scalebox{1}{\scriptsize{$\cat F_1v$\;\;}}}
\psfrag{C2}[Bc][Bc]{\scalebox{1}{\scriptsize{\;\;$\cat F_1 j_!$}}}
\psfrag{F}[Bc][Bc]{\scalebox{1}{\scriptsize{$\cat F_1 \alpha_1$}}}
\psfrag{G}[Bc][Bc]{\scalebox{1}{\scriptsize{$\cat F_1 \alpha_2$}}}
\psfrag{S}[Bc][Bc]{\scalebox{1}{\scriptsize{$\fun_{\cat F_1}$}}}
\psfrag{S'}[Bc][Bc]{\scalebox{1}{\scriptsize{$\fun^{-1}_{\cat F_1}$}}}
\psfrag{L}[Bc][Bc]{\scalebox{1}{\scriptsize{$t(u)$}}}
\psfrag{H}[Bc][Bc]{\scalebox{1}{\scriptsize{$t(i)^{-1}$}}}
\psfrag{T}[Bc][Bc]{\scalebox{1}{\scriptsize{$\cat F_1 a$}}}
\psfrag{B}[Bc][Bc]{\scalebox{1}{\scriptsize{$\cat F_1 ja$}}}
\psfrag{D}[Bc][Bc]{\scalebox{1}{\scriptsize{$\cat F_1 i$}}}
\includegraphics[scale=.4]{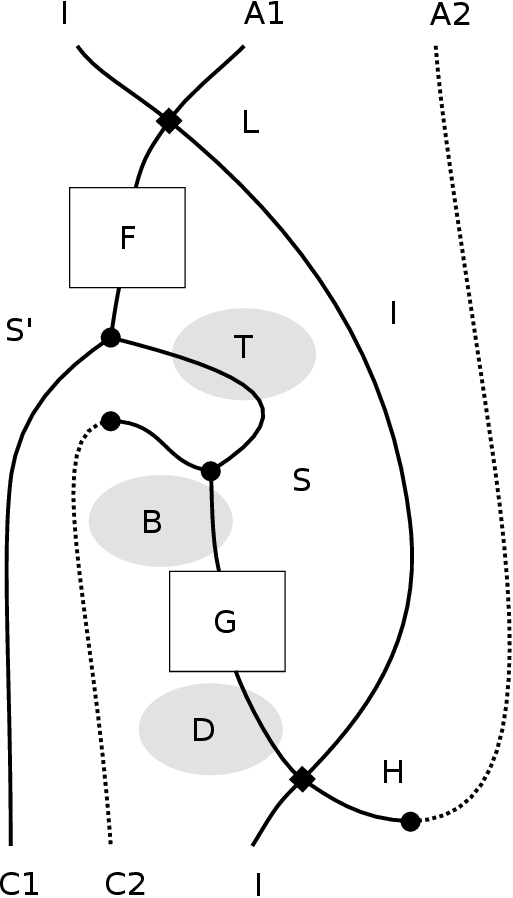}
}}
\]
\[
\stackrel{3 \times \textrm{\eqref{Exa:strings-for-adjoints}}}{=}\quad
\vcenter { \hbox{
\psfrag{I}[Bc][Bc]{\scalebox{1}{\scriptsize{$t$}}}
\psfrag{A1}[Bc][Bc]{\scalebox{1}{\scriptsize{$\cat F_2u$}}}
\psfrag{A2}[Bc][Bc]{\scalebox{1}{\scriptsize{$\cat F_2i_!$}}}
\psfrag{A3}[Bc][Bc]{\scalebox{1}{\scriptsize{$\cat F_1u$}}}
\psfrag{C1}[Bc][Bc]{\scalebox{1}{\scriptsize{$\cat F_1v$\;\;}}}
\psfrag{C2}[Bc][Bc]{\scalebox{1}{\scriptsize{\;\;$\cat F_1 j_!$}}}
\psfrag{F}[Bc][Bc]{\scalebox{1}{\scriptsize{$\cat F_1 \alpha_1$}}}
\psfrag{G}[Bc][Bc]{\scalebox{1}{\scriptsize{$\cat F_1 \alpha_2$}}}
\psfrag{L}[Bc][Bc]{\scalebox{1}{\scriptsize{$t(u)$}}}
\psfrag{H}[Bc][Bc]{\scalebox{1}{\scriptsize{$\;\;\;t(i)^{-1}$}}}
\psfrag{S'}[Bc][Bc]{\scalebox{1}{\scriptsize{$\cat F_1a_!$}}}
\psfrag{C'}[Bc][Bc]{\scalebox{1}{\scriptsize{$(\cat F_1ja)_!$}}}
\psfrag{U'}[Bc][Bc]{\scalebox{1}{\scriptsize{$\cat F_1i_!$}}}
\psfrag{R}[Bc][Bc]{\scalebox{1}{\scriptsize{$\cat F_1va\;\;$}}}
\psfrag{S}[Bc][Bc]{\scalebox{1}{\scriptsize{$\cat F_1a$}}}
\psfrag{T}[Bc][Bc]{\scalebox{1}{\scriptsize{$\cat F_1v\;$}}}
\psfrag{D}[Bc][Bc]{\scalebox{1}{\scriptsize{$\;\cat F_1j$}}}
\psfrag{C}[Bc][Bc]{\scalebox{1}{\scriptsize{$\;\;\;\cat F_1ja$}}}
\includegraphics[scale=.4]{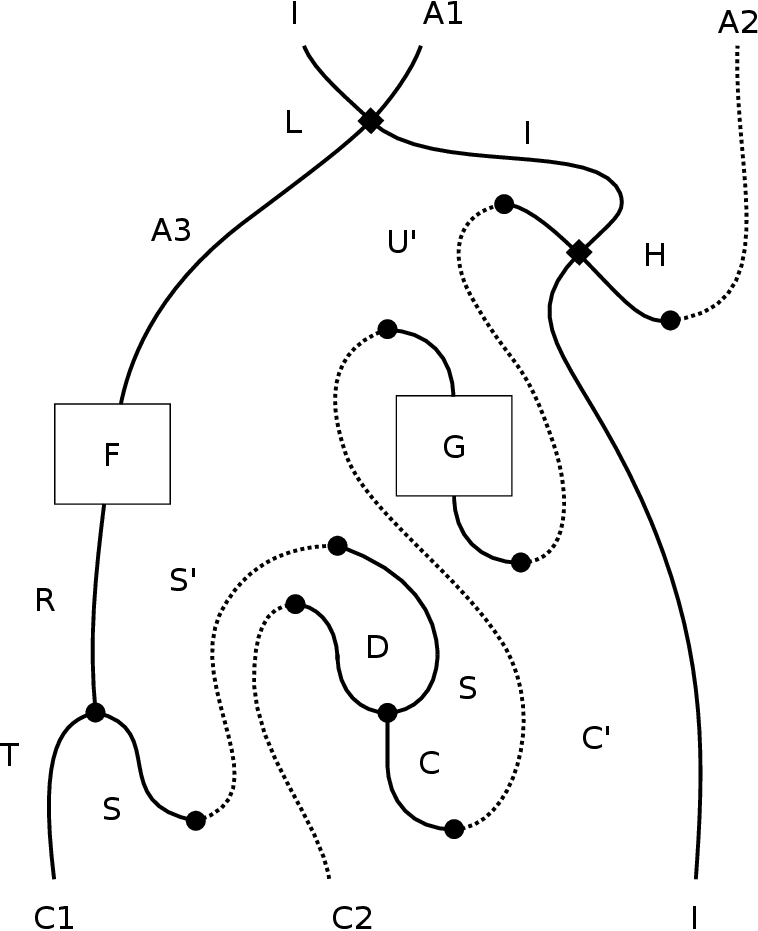}
}}
\quad\quad = \quad\quad
\vcenter {\hbox{
\psfrag{F}[Bc][Bc]{\scalebox{1}{\scriptsize{$t(i_!u^*)$}}}
\psfrag{G}[Bc][Bc]{\scalebox{1}{\scriptsize{$\cat G_1 \alpha$}}}
\psfrag{A}[Bc][Bc]{\scalebox{1}{\scriptsize{$t$}}}
\psfrag{B}[Bc][Bc]{\scalebox{1}{\scriptsize{$\cat G_2 (i_!u^*)$}}}
\psfrag{C}[Bc][Bc]{\scalebox{1}{\scriptsize{$\cat G_1(j_!v^*)$}}}
\psfrag{D}[Bc][Bc]{\scalebox{1}{\scriptsize{$t$}}}
\psfrag{E}[Bc][Bc]{\scalebox{1}{\scriptsize{$t$}}}
\psfrag{H}[Bc][Bc]{\scalebox{1}{\scriptsize{$\cat G_1 (j_!v^*)$}}}
\includegraphics[scale=.4]{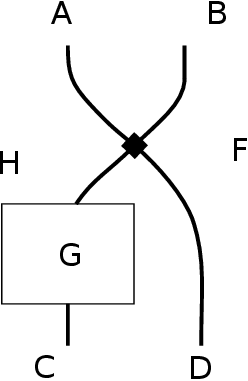}
}}
\]
The proofs of the compatibility with $\fun$ and $\un$ are quite similar and are left as an exercise for the reader.
\end{proof}

This completes the proof of \Cref{Thm:UP-PsFun-Span}.
\end{proof}

\begin{Rem}
The above proof yields a more general biequivalence on bicategories of pseudo-functors and \emph{oplax} transformations, which actually restricts to the one of the statement.

To explain this, let $\cat{G}_1\,,\,\cat{G}_2\colon \Span(\GG;\JJ)\to \cat{B}$ be two pseudo-functors. Let $\cat{F}_1=\cat{G}_1\circ(-)^*$ and $\cat{F}_2=\cat{G}_2\circ(-)^*$ be the corresponding $\JJ_!$-pseudo-functors $\GG^\op\to \cat{B}$. As we saw in \Cref{Rem:pseudo-nat-on-adjoints}, our focus on $\JJ_!$-strong transformations $t\colon \cat{F}_1\Rightarrow \cat{F}_2$ is related to the desire of $t_!\colon (\cat{F}_1)_!\Rightarrow (\cat{F}_2)_!$ to be an actual strong transformation, not only an op-lax one, and the proof of \Cref{Thm:UP-PsFun-Span} used this property to extend~$t$ to a strong transformation $t\colon \cat{G}_1\Rightarrow \cat{G}_2$; see~\eqref{eq:decomp_comp}.

Relaxing both sides, if we consider an oplax transformation $\cat{G}_1\Rightarrow \cat{G}_2$ and restrict it along $(-)^*$ we evidently obtain an oplax transformation $t\colon \cat{F}_1\Rightarrow \cat{F}_2$ but \emph{not any} oplax transformation! Because of \Cref{Lem:tr_inv_mate}, $t\colon \cat{F}_1\Rightarrow\cat{F}_2$ must have the property that the component $t_i$ is invertible for every~$i\in \JJ$, because $\cat{F}_1$ and $\cat{F}_2$ factor as $\GG^\op\oto{(-)^*} \Span\to \cat{B}$ and $i^*$ has a left adjoint in the intermediate bicategory~$\Span(\GG;\JJ)$. Conversely, if we start with an oplax transformation $t\colon \cat{F}_1\Rightarrow \cat{F}_2$ which has the property that its component $t_i$ is invertible, then we can consider its left mate to define $t(i_!)\colon \cat{G}_1(i_!)\Rightarrow\cat{G}_2(i_!)$ as in the proof of \Cref{Thm:UP-PsFun-Span}. The very same proof then gives the following oplax version.
\end{Rem}

\begin{Sch}
Precomposition with $(-)^*\colon \GG^\op\to \Span(\GG;\JJ)$ induces a biequivalence
\[
\PsFunoplax( \Span(\GG;\JJ), \cat{B}) \stackrel{\sim}{\longrightarrow} \PsFunJJ( \GG^{\op}, \cat{B})
\]
where the left-hand side is the bicategory of pseudo-functors, oplax transformations and modifications, and the right-hand side is the bicategory of $\JJ_!$-pseudo-functors (\Cref{Def:J_!-pseudo-functor}), oplax transformations~$t$ with the property that $t_i$ is invertible for all~$i\in \JJ$ (\,\footnote{\,One could call these `$\JJ$-strong', not to be confused with the $\JJ_!$-strong of \Cref{Def:J_!-strong}.}), and modifications between them.
\end{Sch}

\begin{Rem}
We can recover~\cite[Prop.\,1.10]{DawsonParePronk04} as a very special case of our \Cref{Thm:UP-PsFun-Span}. First note that if we set $\JJ=\GG_1$ and take $\GG$ to be a 1-category (\ie a (2,1)-category where the only 2-arrows are the identities), then the situation is completely symmetric with respect to span transposition. In this case, there is a dual statement, with a dual proof, for the \emph{co}variant embedding $\GG\to \Span(\GG;\JJ)$ (\Cref{Cons:first-embeddings}). By forgetting modifications in this dual statement, it becomes precisely the result of \emph{loc.\,cit}.
\end{Rem}

\bigbreak
\section{Pullback of 2-cells in the bicategory of spans}
\label{sec:pullback_2cells}%
\medskip

For the whole section, we abbreviate
\[
\Span:=\Span(\GG;\JJ)\,.
\]
So for instance $\Span(G,H)$ means $\Span(\GG;\JJ)(G,H)$ for every $G,H\in\GG_0$.
Recall~\Cref{Hyp:G_and_I_for_Span}, including part~\eqref{Hyp:G-I-c}, \ie the faithfulness of every $i\in\JJ$, and the implication $ij\in \JJ \Rightarrow j\in \JJ$ in part~\eqref{Hyp:G-I-a}. One reason for requiring these extra properties is the following very convenient consequence.
\begin{Prop} \label{Prop:pullbacks}
For all $G,H\in \GG_0$, the category $\Span(G,H)$ admits arbitrary pullbacks (in the usual strict sense). They are induced by the comma squares in~$\GG$.
\end{Prop}
\begin{proof}
It is well-known that iso-comma squares provide the homotopy pullbacks in the homotopy category of small groupoids, where the ordinary equivalences are inverted (\cf Remark~\ref{Rem:htpy-pullback}). Here we need something stronger, because we claim extra compatibility and uniqueness properties for the map induced by the usual weak pullback property of homotopy pullbacks.

Consider the following (left-hand) cospan of maps of $\Span(G,H)$, represented in~$\GG$ by the diagram on the right-hand side:
\begin{equation}
\vcenter{\label{eq:pullback-Bicat-data}%
\xymatrix@L=1ex{
i_!u^* \ar@{=>}[d]^-{[a,\alpha_1,\alpha_2]}_{[a] \, :=} \\
k_!w^* \\
j_!v^* \ar@{=>}[u]_-{[b,\beta_1,\beta_2]}^{[b] \, :=}
}}
\quad\quad\quad
\vcenter{\xymatrix{
G \ar@{=}[d] \ar@{}[rrd]|{\SEcell\,\alpha_1} &&
 P \ar[d]^a \ar[ll]_-{u} \ar[rr]^-{i} \ar@{}[rrd]|{\NEcell\,\alpha_2} &&
 H \ar@{=}[d] \\
G \ar@{}[rrd]|{\NEcell\,\beta_1} \ar@{=}[d] &&
 R \ar[ll]_-{w} \ar[rr]^-{k} \ar@{}[rrd]|{\SEcell\,\beta_2} &&
 H \ar@{=}[d] \\
G && Q \ar[u]_b \ar[ll]^-{v} \ar[rr]_-{j} &&
 H
}}
\end{equation}
Construct in~$\GG$ an iso-comma square over the cospan appearing in the middle column (this is possible because now $a$ and $b$ are in $\JJ$ by \Cref{Hyp:G_and_I_for_Span}\,\eqref{Hyp:G-I-a}):
\begin{align*}
\xymatrix@C=14pt@R=14pt{
& (a/b) \ar[ld]_-{\tilde b} \ar[dr]^-{\tilde a} & \\
P \ar[dr]_a \ar@{}[rr]|{\isocell{\gamma}} && Q \ar[dl]^b \\
&R &
}
\end{align*}
We can now complete the span $P \stackrel{\tilde b}{\leftarrow} (a/b) \stackrel{\tilde a}{\to}Q$ to represent a pair of 2-cells in~$\Span$:
\begin{align} \label{eq:pullback_in_Bicat}
\xymatrix@L=1ex{
i_!u^* &
G \ar@{=}[dd] \ar@{}[ddrr]|{\NEcell \alpha_1\inv } &&&
 P \ar[lll]_-{u} \ar[rrr]^-{i} &&&
 H \ar@{=}[dd] \ar@{}[ddll]|{\SEcell \alpha_2\inv} \\
& &&& &&& \\
(ka\tilde b)_!(wa \tilde b)^* \ar@{=>}[uu]|-{[\tilde b] \, := \, [\tilde b,\ldots,\ldots] \, :=\;} \ar@{=>}[dd]|-{[\tilde a] \, := \, [\tilde a,\ldots,\ldots] \, :=\;}
&
G \ar@{=}[dd] \ar@{}[ddrr]|{\SEcell \beta_1\inv} &
 R \ar[l]_-{w} &
 P \ar[l]_-{a} \ar[uur]^\Id \ar@{}[rd]_-{\SEcell \gamma} &
 (a/b) \ar[l]_-{\tilde b} \ar[uu]_-{\tilde b} \ar[dd]^-{\tilde a} \ar[r]^-{\tilde b} &
 P \ar[r]^-{a} \ar[uul]_\Id \ar@{}[ld]^-{\NEcell \gamma\inv\;} &
 R \ar[r]^-{k} & H \ar@{=}[dd] \ar@{}[ddll]|{\NEcell \beta_2\inv} \\
& &&& &&& \\
j_!v^* & G &&& Q \ar[lll]^-{v} \ar[rrr]_-{j} \ar[uull]^-{b} \ar[uurr]_b &&& H
}
\end{align}
(note that there is also another possible, isomorphic, choice for such a diagram, where the $\gamma$'s appear in the top half, \ie as part of the 2-cell~$[\tilde b]$).
We are going to test the universal property of a pullback in $\Span(G,H)$ for the resulting square
\[
\xymatrix@L=1ex{
& (ka\tilde b)_!(wa \tilde b)^* \ar@{=>}[ld]_-{[\tilde b]} \ar@{=>}[rd]^-{[\tilde a]}
\\
i_!u^* \ar@{=>}[rd]_-{[a]}
&& j_!v^* \ar@{=>}[ld]^-{[b]}
\\
& k_!w^*
}
\]
Suppose we are given a pair of 2-cells
\begin{align*}
\xymatrix{
G \ar@{=}[d] \ar@{}[rrd]|{\SEcell\,\rho_1} &&
 T \ar[ll]_-x \ar[rr]^-\ell \ar[d]^r \ar@{}[rrd]|{\NEcell\,\rho_2} &&
 H \ar@{=}[d] \\
G &&
 P \ar[ll]^-u \ar[rr]_-{i} &&
 H
}
\quad\quad
\xymatrix{
G \ar@{=}[d] \ar@{}[rrd]|{\SEcell\,\sigma_1} &&
 T \ar[ll]_-x \ar[rr]^-\ell \ar[d]^s \ar@{}[rrd]|{\NEcell\,\sigma_2} &&
 H \ar@{=}[d] \\
G &&
 Q \ar[ll]^-{v} \ar[rr]_-{j} &&
 H
}
\end{align*}
such that $[a,\alpha_1,\alpha_2][r,\rho_1,\rho_2] =[b,\beta_1,\beta_2][s,\sigma_1,\sigma_2]$ as 2-cells from $\ell_!x^*$ to~$k_!w^*$. The latter equation means that there exists an isomorphisms $\varphi\colon ar \isoEcell bs$ identifying its left and right 2-cell components as in~\Cref{Def:Span-bicat}:
\begin{equation} \label{eq:properties-of-phi}
(w \varphi) (\alpha_1 r) \rho_1 = (\beta_1 s) \sigma_1
\quad \textrm{ and } \quad
\rho_2 (\alpha_2 r) = \sigma_2 (\beta_2 s) (k \varphi)\,.
\end{equation}
In particular, by the universal property of the comma object~$a/b$ there is a unique 1-cell $t=\langle r,s,\varphi\rangle\colon T\to (a/b)$ of~$\GG$, \ie such that
\begin{equation} \label{eq:def-of-t}
\tilde b t =r , \quad \tilde a t =s \quad \textrm{ and } \quad \gamma t = \varphi \,.
\end{equation}
Complete $t$ to a 2-cell $[t]:=[t,\ldots,\ldots]\in \Span_2$ as follows:
\begin{equation}
\label{eq:def-of-[t]}%
\vcenter{\xymatrix{
G \ar@{=}[dd] \ar@{}[ddr]|{\SEcell \alpha_1} && \ar@{}[dd]|{\SEcell \rho_1} &
 T \ar[dd]^t \ar[lll]_-x \ar[rrr]^-\ell \ar[ddl]^r \ar[ddr]_r & \ar@{}[dd]|{\NEcell \rho_2} &&
 H \ar@{=}[dd]\\
&&& &&& \\
G & R \ar[l]^-w & P \ar[l]^-{a} \ar[uull]_u &
 a/b \ar[l]^-{\tilde b} \ar[r]_-{\tilde b} &
 P \ar[r]_-{a} \ar[uurr]^i &
 R \ar[r]_-k \ar@{}[uur]|{\NEcell \alpha_2} &
 H \,.\!\!
}}
\end{equation}
We must verify that $[\tilde b][t] = [r]$ and $[\tilde a][t] =[s]$ as 2-cells of~$\Span$.
In fact, these are even \emph{strict} equalities of diagrams in~$\GG$, because we evidently have
\[
\xymatrix{
G \ar@{=}[dd] \ar@{}[ddr]|{\SEcell \alpha_1} && \ar@{}[dd]|{\SEcell \rho_1} &
 T \ar[dd]^t \ar[lll]_-x \ar[rrr]^-{\ell} \ar[ddl]^r \ar[ddr]_-r & \ar@{}[dd]|{\NEcell \rho_2} &&
 H \ar@{=}[dd] &\\
&&& &&& & \\
G \ar@{=}[dd] \ar@{}[ddrr]|{\SEcell \alpha_1\inv} &
 R \ar[l]^-w &
 P \ar[l]^-{a} \ar[uull]_u &
 a/b \ar[l]^-{\tilde b} \ar[dd]^-{\tilde b} \ar[r]_-{\tilde b} &
 P \ar[r]_-{a} \ar[uurr]^i &
 R \ar[r]_-k \ar@{}[uur]|{\NEcell \alpha_2} &
 H \ar@{=}[dd] \ar@{}[ddll]|{\NEcell \alpha_2\inv} & = (r, \rho_1,\rho_2) \\
&&& &&& && \\
G &&& P \ar[lll]^-u \ar[rrr]_i \ar@{<-}[uul]^\Id \ar@{<-}[uur]_\Id &&& H &
}
\]
and also, slightly less obviously,
\[
\xymatrix{
G \ar@{=}[dd] \ar@{}[ddr]|{\SEcell \alpha_1} && \ar@{}[dd]|{\SEcell \rho_1} &
 T \ar[dd]^t \ar[lll]_-x \ar[rrr]^-{\ell} \ar[ddl]^r \ar[ddr]_-r & \ar@{}[dd]|{\NEcell \rho_2} &&
 H \ar@{=}[dd] &&\\
&&& &&& && \\
G \ar@{=}[dd] \ar@{}[ddrr]|{\SEcell \beta_1\inv} &
 R \ar[l]^-w &
 P \ar[l]^-{a} \ar[uull]_u \ar@{}[ddr] \ar@{}[rd]_-{\SEcell \gamma} &
 a/b \ar[l]^-{\tilde b} \ar[dd]^-{\tilde a} \ar[r]_-{\tilde b} &
 P \ar[r]_-{a} \ar[uurr]^i \ar@{}[ld]^-{\NEcell \gamma\inv} &
 R \ar[r]_-k \ar@{}[uur]|{\NEcell \alpha_2} &
 H \ar@{=}[dd] \ar@{}[ddll]|{\NEcell \beta_2\inv} & \stackrel{\textrm{\eqref{eq:def-of-t}}}{=} & \\
&&& &&& && \\
G &&& Q \ar[lll]^-{v} \ar[rrr]_j \ar[uull]^b \ar[uurr]_b &&& H &&
}
\]
\[
\xymatrix{
& G \ar@{=}[dd] \ar@{}[ddr]|{\SEcell \alpha_1} && \ar@{}[dd]|{\SEcell \rho_1} &
 T \ar[dddd]^s \ar[lll]_-x \ar[rrr]^-{\ell} \ar[ddl]^r \ar[ddr]_-r & \ar@{}[dd]|{\NEcell \rho_2} &&
 H \ar@{=}[dd] &&\\
& &&& &&& && \\
= & G \ar@{=}[dd] \ar@{}[ddrr]|{\SEcell \beta_1\inv} &
 R \ar[l]^-w &
 P \ar[l]^-{a} \ar[uull]_u \ar@{}[dr]|{\SEcell \varphi} &&
 P \ar[r]_-{a} \ar[uurr]^i \ar@{}[dl]|{\NEcell \varphi\inv} &
 R \ar[r]_-k \ar@{}[uur]|{\NEcell \alpha_2} &
 H \ar@{=}[dd] \ar@{}[ddll]|{\NEcell \beta_2\inv} & \stackrel{\textrm{\eqref{eq:properties-of-phi}}}{=} & \\
& &&& &&& && \\
& G &&& Q \ar[lll]^-{v} \ar[rrr]_j \ar[uull]^b \ar[uurr]_b &&& H &&
}
\]
\[
\xymatrix{
& G \ar@{=}[dd] \ar@{}[rdd]^-{\SEcell \sigma_1} && &
 T \ar[dddd]^s \ar[lll]_-x \ar[rrr]^-{\ell} \ar[dl]^s \ar[dr]_s & &&
 H \ar@{=}[dd] \ar@{}[ldd]_-{\NEcell \sigma_2} &\\
& && Q \ar[dll]^v \ar@{}[dd]|{\SEcell \, \beta_1\,s} && Q \ar[rrd]_j \ar@{}[dd]|{\NEcell\, \beta_2\,s} && & \\
= & G \ar@{=}[dd] \ar@{}[ddrr]|{\SEcell \beta_1\inv} &
 R \ar[l]^-w &
 &&&
 R \ar[r]_-k &
 H \ar@{=}[dd] \ar@{}[ddll]|{\NEcell \beta_2\inv} & = (s,\sigma_1,\sigma_2) \,. \\
& &&& &&& & \\
& G &&& Q \ar[lll]^-{v} \ar[rrr]_j \ar[uull]^b \ar[uurr]_b &&& H &
}
\]
Let us also note, for future reference, that the right-hand side of the latter identification amounts to the following identity of 2-cells in~$\GG$:
\begin{equation} \label{eq:blobby1}
\vcenter{\xymatrix@L=2pt{
\ar@/_4ex/[rdd]_\ell \ar@/^3ex/[rrd]^{\;\, s\,=\,\tilde a t}
&&&& \ar@/_7ex/[dddr]_\ell \ar@/^1ex/[r]^-t & \ar[d]^-{\tilde b} \ar@/^3ex/[rrd]^-{\tilde a}
& &
\\
\ar@{}[rr]|{\oWcell{\sigma_2}}
&& \ar@/^2ex/[dl]^j & = \ar@{}[rr]_-{\quad\quad\oWcell{\rho_2}}
&& \ar[dr]^-{a} \ar[dd]_i \ar@{}[rr]^-{\oWcell{\gamma\inv}}
&& \ar[ld]_b \ar@/^8ex/[ddll]^j
\\
&
&&&& \ar@{}[r]|{\oWcell{\alpha_2}\;\;} & \ar[dl]^k \ar@{}[dr]|{\oWcell{\beta_2\inv}\kern1em}
&
\\
& &&&& &&
}}
\end{equation}

Now it only remains for us to verify the uniqueness of the induced 2-cell~$[t]$ in~$\Span$. Suppose we are given another 2-cell~$[t']$ such that
\begin{align} \label{eq:up-test-for-t}
[\tilde b][t']=[r] \quad \textrm{ and } \quad [\tilde a][t']=[s] \,.
\end{align}
We must show that $[t']=[t]$, and for this we will use the \emph{morphism} part of the universal property of $a/b$ to construct a suitable isomorphism $\theta\colon t'\isoEcell t$.

The 2-cell $[t']$ is represented by a diagram
\begin{align*}
\xymatrix{
G \ar@{=}[d] \ar@{}[drrr]|{\SEcell\,\tau_1} &&&
 T \ar[lll]_-x \ar[rrr]^-{\ell} \ar[d]^-{t'} \ar@{}[drrr]|{\NEcell\,\tau_2} &&&
 H \ar@{=}[d] \\
G & R \ar[l]^-w & P \ar[l]^-{a} &
 a/b \ar[l]^-{\tilde b} \ar[r]_-{\tilde b} & P \ar[r]_-{a} & R \ar[r]_k &
 H
}
\end{align*}
and the two equations~\eqref{eq:up-test-for-t} are realized by isomorphisms $\zeta\colon \tilde bt' \isoEcell r$ and $\xi\colon \tilde a t' \isoEcell s$ as follows:
\[
\vcenter{\xymatrix@C=12pt{
G \ar@{=}[d] \ar@{}[drrr]|{\SEcell\,\tau_1} &&&
 T \ar[lll]_-x \ar[rrr]^-{\ell} \ar[d]^-{t'} \ar@{}[drrr]|{\NEcell\,\tau_2} &&&
 H \ar@{=}[d] &&
 G \ar@{=}[dd] \ar@{}[ddrr]|{\SEcell \rho_1} &&
 T \ar[dd]^r \ar[ll]_-x \ar[rr]^-\ell \ar@{}[ddrr]|{\NEcell \rho_2} &&
 H \ar@{=}[dd]
\\
G \ar@{=}[d] \ar@{}[rrd]|{\SEcell \alpha_1\inv} &
 R \ar[l]_-{w} &
 P \ar[l]_-{a} &
 (a/b) \ar[l]_-{\tilde b} \ar[d]^-{\tilde b} \ar[r]^-{\tilde b} &
 P \ar[r]^-{a} &
 R \ar[r]^{k} &
 H \ar@{=}[d] \ar@{}[dll]|{\NEcell \alpha_2\inv} &
 \stackrel{\zeta\,\,}{\Rightarrow} &&&&& \\
G &&&
 P \ar[lll]^-u \ar[rrr]_-{i} \ar@{<-}[ul]^\Id \ar@{<-}[ur]_\Id &&&
 H &&
 G &&
 P \ar[ll]^-u \ar[rr]_-{i} &&
 H
}}
\]
\[
\vcenter{\xymatrix@C=12pt{
G \ar@{=}[d] \ar@{}[drrr]|{\SEcell\,\tau_1} &&&
 T \ar[lll]_-x \ar[rrr]^-{\ell} \ar[d]^-{t'} \ar@{}[drrr]|{\NEcell\,\tau_2} &&&
 H \ar@{=}[d] &&
 G \ar@{=}[dd] \ar@{}[ddrr]|{\SEcell \sigma_1} &&
 T \ar[dd]^s \ar[ll]_-x \ar[rr]^-\ell \ar@{}[ddrr]|{\NEcell \sigma_2} &&
 H \ar@{=}[dd] \\
G \ar@{=}[d] \ar@{}[rrd]|{\SEcell \beta_1\inv} &
 R \ar[l]_-{w} &
 P \ar[l]_-{a}
& (a/b) \ar[l]_-{\tilde b} \ar[d]_-{\tilde a} \ar[r]^-{\tilde b} \ar@{}[ld]|(.3){\SEcell \gamma\kern2em} \ar@{}[rd]|(.3){\kern1.5em\NEcell \gamma\inv}
& P \ar[r]^-{a}
& R \ar[r]^-{k} &
 H \ar@{=}[d] \ar@{}[dll]|{\NEcell \beta_2\inv} &
 \stackrel{\xi\,\,}{\Rightarrow}
 &&&&& \\
G &&&
 Q \ar[lll]^-{v} \ar[rrr]_-{j} \ar[llu]^b \ar[rru]_b &&&
 H &&
 G &&
 Q \ar[ll]^-{v} \ar[rr]_-{j} &&
 H
}}
\]
which read
\begin{align}
\label{eq:zeta-identity-left}%
(u\,\zeta)(\alpha_1\inv\tilde{b}t')\tau_1 & =\rho_1
\\
\label{eq:zeta-identity-right}%
\tau_2(\alpha_2\inv\,\tilde{b}t') & =\rho_2(i\,\zeta)
\quadtext{hence}
\rho_2\inv \tau_2 =(i\,\zeta)(\alpha_2\,\tilde{b}t')
\\
\label{eq:xi-identity-left}%
(\beta_1\inv\tilde{a}t')(w\,\gamma\,t')\tau_1 &= (v \,\xi^{-1})\sigma_1
\\
\label{eq:xi-identity-right}%
\tau_2(k\,\gamma\inv\,t')(\beta_2\inv\,\tilde{a}t') & =\sigma_2(j\,\xi)
\quadtext{hence}
\sigma_2\inv \tau_2 =(j\,\xi)(\beta_2\,\tilde{a}t')(k\,\gamma\,t')
\end{align}

Thus we have in~$\GG$ two parallel 1-cells $t',t\colon T\to a/b$ and two 2-cells
$\zeta \colon \tilde b t' \Rightarrow r = \tilde b t$ and $\xi\colon \tilde a t' \Rightarrow s = \tilde a t$.
From the following computation
\begin{align*}
\vcenter{\xymatrix@L=2pt@C=16pt@R=16pt{
& \ar@/_1ex/[d]_-{t'} \ar@/^3ex/[rrd]^-{t}
\\
& \ar[ld]_-{\tilde b} \ar[dr]^-{\tilde a} \ar@{}[rr]^-{\oEcell{\xi}}
&& \ar[ld]^-{\tilde a}
\\
 \ar[dr]_a \ar@{}[rr]|{\oEcell{\gamma}}
&& \ar[dl]^b
\\
& \ar[d]^k
\\
&
}}
\qquad=\qquad
\vcenter{\xymatrix@L=2pt@C=16pt@R=16pt{
&& \ar@/_1ex/[ld]_-{t'} \ar@/^3ex/[rrd]^-{t} \ar[rdd]^-{s}
\\
& \ar[ld]_-{\tilde b} \ar[rrd]^-{\tilde a} \ar@{}[rr]^(.4){\oEcell{\xi}}
&&& \ar[ld]_-{\tilde a}
\\
 \ar[dr]_a \ar@{}[rr]|{\oEcell{\gamma}}
&&& \ar[lld]_-{b} \ar[ldd]_-{j} \ar@/^3ex/[rd]^-{b}
\\
& \ar[rd]_-{k} \ar@{}[r]|(.7){\;\oEcell{\;\;\beta_2}}
& \ar@{}[rr]|(.6){\oEcell{\beta_2\inv}}
&& \ar@/^3ex/[lld]^-{k}
\\
&&
}}
\qquad\stackrel{\textrm{\eqref{eq:xi-identity-right}}}{=}
\end{align*}
\begin{align*}
=\;\;
\vcenter{\xymatrix@L=2pt@C=16pt@R=16pt{
&&& \ar@/_1ex/[dll]_-{t'} \ar@/^3ex/[rrd]^t \ar[ddr]^s \ar@/_5ex/[dddd]_\ell && \\
& \ar[ld]_-{\tilde b} &&&& \ar@/^2ex/[ld]_-{\tilde a} \\
 \ar[dr]_a \ar@{}[rr]|{\oEcell{\tau_2}} && \ar@{}[rr]|{\oEcell{\sigma_2\inv}} && \ar@/^2ex/[dr]^-{b} \ar[ddl]_j & \\
& \ar@/_1ex/[rrd]_k &&&& \ar@/^3ex/[dll]^k \\
&&& \ar@{}[urr]|{\oEcell{\beta_2\inv}}
 &&
}}
\quad \stackrel{\textrm{\eqref{eq:blobby1}}}{=} \quad
\vcenter{\xymatrix@L=2pt@C=16pt@R=16pt{
&&& \ar@/_1ex/[lld]_-{t'} \ar[r]^t \ar@/_5ex/[dddd]_\ell &
 \ar[d]^-{\;\tilde b} \ar@/^2ex/[rrd]^-{\tilde a} && \\
& \ar[ld]_-{\tilde b} & \ar@{}[rr]_-{\oEcell{\rho_2\inv}} &&
 \ar[dr]_a \ar@/_2ex/[dddl]_i \ar@{}[rr]^-{\oEcell{\gamma}} &&
 \ar[ld]_-{b} \ar@/^7ex/[dddlll]^j \ar@/^5ex/[ddd]^b \\
 \ar[dr]_a \ar@{}[rr]|{\oEcell{\tau_2}}
 &&&
 \ar@{}[rr]_-{\oEcell{\alpha_2\inv}\;\; }
 &&
 \ar[ddll]^k \ar@{}[ld]^-{\oEcell{\beta_2}}
 & \ar@{}[dd]|{\;\;\oEcell{\;\;\;\;\beta_2\inv}}
 \\
& \ar@/_1ex/[rrd]_k &&&&& \\
&&& &&& \ar@/^2ex/[lll]^k
}}
\end{align*}
\begin{align*}
\stackrel{\textrm{\eqref{eq:zeta-identity-right}}}{=}
\qquad
\vcenter{\xymatrix@L=2pt@C=14pt@R=14pt{
&& \ar[dd]^r \ar@/_2ex/[dll]_-{t'} \ar@/^2ex/[rrd]^t &&& \\
 \ar@/_1ex/[rrd]_<<<{\tilde b} \ar@{}[rr]|{\;\;\;\oEcell{\zeta}}
 &&&&
 \ar[dll]_-{\tilde b} \ar[dr]^-{\tilde a} & \\
&& \ar@/_1ex/[dll]_>>>{a} \ar[rrd]^a \ar[dd]_i \ar@{}[rrr]|{\;\;\;\oEcell{\gamma}} &&&
 \ar[dl]^b \\
 \ar@/_2ex/[rrd]_-{k} \ar@{}[rr]|{\;\;\oEcell{\alpha_2}}
 && \ar@{}[rr]|{\oEcell{\alpha_2\inv\;\;}\;\;}
 &&
 \ar@/^2ex/[dll]^k & \\
&& &&&
}}
\qquad=\qquad
\vcenter{\xymatrix@L=2pt@C=16pt@R=16pt{
&& \ar@/^1ex/[d]^t \ar@/_3ex/[dll]_-{t'} & \\
 \ar[rd]_-{\tilde b} \ar@{}[rr]^-{\oEcell{\zeta}}
 &&
 \ar[ld]_-{\tilde b} \ar[dr]^-{\tilde a} & \\
& \ar[rd]_-{a} \ar@{}[rr]|{\oEcell{\gamma}} &&
 \ar[dl]^b \\
&& \ar[d]^k & \\
&& &
}}
\end{align*}
and the faithfulness of~$k\in\JJ$ of \Cref{Hyp:G_and_I_for_Span}~\eqref{Hyp:G-I-c}, we deduce:
\begin{align*}
\xymatrix@C=14pt@R=14pt{
& T \ar@/^1ex/[d]^-{t} \ar@/_6ex/[ldd]_-{\tilde b t'} \ar@{}[dl]|{\oEcell{\zeta}} &
 &&& T \ar@/_1ex/[d]_-{t'} \ar@/^6ex/[rdd]^-{\tilde a t} \ar@{}[dr]|{\oEcell{\xi}}
 &
 \\
& a/b \ar[ld]_-{\tilde b} \ar[dr]^-{\tilde a} &
 & = && a/b \ar[ld]_-{\tilde b} \ar[dr]^-{\tilde a} & \\
A \ar[dr]_a \ar@{}[rr]|{\oEcell{\gamma}} && B \ar[dl]^b
 && A \ar[dr]_a \ar@{}[rr]|{\oEcell{\gamma}} && B \ar[dl]^b \\
&C & &&& C \,.&
}
\end{align*}
Hence by the universal property of $a/b$ in \Cref{Def:comma}\,\eqref{it:comma-b} there exists a (unique) 2-cell $\theta\colon t'\isoEcell t$ in~$\GG$ such that
\begin{equation}
\label{eq:theta-def-properties}%
\tilde b \theta = \zeta
\quad \textrm{ and } \quad
\tilde a \theta = \xi \,.
\end{equation}
Finally, we claim that $\theta$ provides an isomorphism $[t']=[t]$ (see~\eqref{eq:def-of-[t]} for~$[t]$):
\[
\xymatrix@C=11pt{
G\ar@{=}[dd] \ar@{}[ddrrr]|{\SEcell\,\tau_1} &&&
 T \ar[dd]^-{t'} \ar[lll]_-x \ar[rrr]^-\ell &&&
 H \ar@{=}[dd] \ar@{}[ddlll]|{\NEcell\,\tau_2} &&
G \ar@{=}[dd] && \ar@{}[dd]|{\SEcell \rho_1} &
 T \ar[dd]^t \ar[lll]_-x \ar[rrr]^-\ell \ar[ddl]_<<<<<<r \ar[ddr]^<<<<<<r & \ar@{}[dd]|{\NEcell \rho_2} &&
 H \ar@{=}[dd] \\
 && &&& &&
 \stackrel{\theta}{\Rightarrow} & \ar@{}[dr]|{\;\;\;\;\SEcell \alpha_1}
&&& &&& \ar@{}[dl]|{\NEcell \alpha_2\;\;\;\;} \\
G & \bullet \ar[l]^-w & \bullet \ar[l]^-{a} & a/b \ar[l]^-{\tilde b} \ar[r]_-{\tilde b} & \bullet \ar[r]_-{a} & \bullet \ar[r]_-k & H &&
G & \bullet \ar[l]^-w & \bullet \ar[l]^-{a} \ar[uull]_>>>>>>>>>u &
 a/b \ar[l]^-{\tilde b} \ar[r]_-{\tilde b} &
 \bullet \ar[r]_-{a} \ar[uurr]^>>>>>>>>>i &
 \bullet \ar[r]_-k &
 H \,.
}
\]
Indeed, the two required identities hold as follows:
\[
\vcenter{
\xymatrix@L=2pt@R=2em{
 \ar@<-.5em>@/_3em/[dddd]_(.3){x} \ar@/_.7em/[d]_-{t'} \ar@/^.7em/[d]^-{t} \ar@{}[dddd]_-{\oEcell{\tau_1}\kern1em} \ar@{}[d]|(.4){\oEcell{\theta}}
\\
 \ar[d]^-{\tilde{b}}
\\
 \ar[d]^-{a}
\\
 \ar[d]^-{w}
\\
\textrm{\phantom{m}}
}}
\quad
\equalby{(\ref{eq:theta-def-properties})}
\quad
\vcenter{
\xymatrix@L=2pt@R=2em{
& \ar@<-1em>@/_6em/[ddd]_(.3){x} \ar@/_1em/[d]_-{\tilde{b}t'} \ar@/^.7em/[d]^-{\tilde{b}t\,=\,r} \ar@{}[lddd]_-{\oEcell{\tau_1}\kern2em} \ar@{}[d]|(.4){\oEcell{\zeta}}
\\
& \ar[ld]_-{a} \ar[rd]^-{a} \ar[dd]^(.3){u} \ar@{}[ldd]|-{\kern1em\oEcell{\alpha_1\inv}} \ar@{}[rdd]|-{\oEcell{\alpha_1}\kern1em}
\\
 \ar[rd]_(.3){w}
&& \ar[ld]^-{w}
\\
&
&
}}
\quad
\equalby{(\ref{eq:zeta-identity-left})}
\quad
\vcenter{
\xymatrix@L=2pt@R=2em{
& \ar@<-.5em>@/_3em/[ddd]_(.3){x} \ar[d]^-{r} \ar@{}[lddd]_-{\oEcell{\rho_1}\kern-.5em}
\\
& \ar[rd]^-{a} \ar[dd]_-{u} \ar@{}[rdd]|-{\oEcell{\alpha_1}\kern1em}
\\
&& \ar[ld]^-{w}
\\
&
&
}}
\]
for the left-hand components. The right-hand ones agree as follows:
\[
\vcenter{
\xymatrix@L=2pt@R=2em{
& \ar@/_.7em/[ld]_-{t'} \ar@/^.7em/[ld]^-{t} \ar@{}[ld]|-{\oEcell{\;\theta}\;} \ar[dd]^-{r} \ar@/^4em/[dddd]^-{\ell} \ar@{}[dddd]^-{\kern1em\oEcell{\rho_2}}
\\
 \ar[rd]_-{\tilde{b}}
\\
& \ar[ld]_-{a} \ar[dd]^-{i} \ar@{}[ldd]|-{\kern1em\oEcell{\alpha_2}}
\\
 \ar[rd]_-{k}
\\
&
}}
\quad
\equalby{(\ref{eq:theta-def-properties})}
\quad
\vcenter{
\xymatrix@L=2pt@R=2em{
& \ar[ld]_-{t'} \ar@{}[ldd]|-{\kern1em\oEcell{\zeta}} \ar[dd]^-{r\,=\,\tilde{b}t} \ar@/^4em/[dddd]^-{\ell} \ar@{}[dddd]^-{\kern1em\oEcell{\rho_2}}
\\
 \ar[rd]_-{\tilde{b}}
&
\\
& \ar[ld]_-{a} \ar[dd]^-{i} \ar@{}[ldd]|-{\kern1em\oEcell{\alpha_2}}
\\
 \ar[rd]_-{k}
\\
&
}}
\quad
\equalby{(\ref{eq:zeta-identity-right})}
\quad
\tau_2\,.
\]
This proves our last remaining claim, namely $[t']=[t]$ via the isomorphism~$\theta$.
\end{proof}
\begin{Lem} \label{Lem:pres_pullbacks}
The horizontal composition functors $- \circ -$ of $\Span$ preserve the pullbacks of Proposition~\ref{Prop:pullbacks} in both variables.
\end{Lem}

\begin{proof}
This is essentially just a commutativity property for iso-comma squares, but we still sketch an explicit proof. Let $\circ=\circ_{G,H,K}\colon \Span(H,K)\times \Span(G,H)\to \Span(G,K)$ be one of the composition functors, and let $\ell_!x^*=(G \stackrel{x}{\leftarrow}S \stackrel{\ell}{\to} H)\in \Span(G,H)$. It suffices to check that $- \circ (\ell_!x^*)$ preserves pullbacks, as the proof for the other variable is the same.

Consider a cospan $\bullet \stackrel{[a]}{\Rightarrow} \bullet \stackrel{[b]}{\Leftarrow} \bullet $ of maps in $\Span(H,K)$
\begin{align*}
\xymatrix{
H \ar@{=}[d] \ar@{}[rrd]|{\SEcell\,\alpha_1} &&
 P \ar[d]^a \ar[ll]_-{u} \ar[rr]^-{i} \ar@{}[rrd]|{\NEcell\,\alpha_2} &&
 K \ar@{=}[d] \\
H \ar@{}[rrd]|{\NEcell\,\beta_1} \ar@{=}[d] &&
 R \ar[ll]_-{w} \ar[rr]^-{k} \ar@{}[rrd]|{\SEcell\,\beta_2} &&
 K \ar@{=}[d] \\
H && Q \ar[u]_b \ar[ll]^-{v} \ar[rr]_-{j} &&
 K
}
\end{align*}
and construct its pullback by the following diagram, as in~\eqref{eq:pullback_in_Bicat}:
\begin{align*}
\xymatrix{
H \ar@{=}[dd] \ar@{}[ddrr]|{\NEcell \alpha_1\inv } &&&
 P \ar[lll]_-{u} \ar[rrr]^-{i} &&&
 K \ar@{=}[dd] \ar@{}[ddll]|{\SEcell \alpha_2\inv} \\
&&& &&& \\
H \ar@{=}[dd] \ar@{}[ddrr]|{\SEcell \beta_1\inv} &
 R \ar[l]_-{w} &
 P \ar[l]_-{a} \ar[uur]^\Id \ar@{}[rd]_-{\SEcell \gamma} &
 P\diagup_{\!\!\! R\,}Q \ar[l]_-{\pr_P} \ar[uu]_-{\pr_P} \ar[dd]^-{\pr_Q} \ar[r]^-{\pr_P} &
 P \ar[r]^-{a} \ar[uul]_\Id \ar@{}[ld]^-{\NEcell \gamma\inv} &
 R \ar[r]^-{k} &
 K \ar@{=}[dd] \ar@{}[ddll]|{\NEcell \beta_2\inv} \\
&&& &&& \\
H &&& Q \ar[lll]^-{v} \ar[rrr]_-{j} \ar[uull]^-{b} \ar[uurr]_b &&& K
}
\end{align*}
The verification that $-\circ (\ell_!x^*)$ preserves this pullback diagram boils down to checking that the induced 1-cell~$f$ in the diagram below is an equivalence:
\begin{align*}
\xymatrix@R=8pt{
(S\diagup_{\!\!\! H\,}P)\diagup_{\!\!\!{\scriptscriptstyle S\diagup_{\!\!\! H\,}R}\,}(S\diagup_{\!\!\! H\,}Q) \ar[dr] \ar[dddddr] &
 &&&& \\
& S\diagup_{\!\!\! H\,}P \ar[dr] \ar[dd] &&&& S\diagup_{\!\!\! H\,}(P\diagup_{\!\!\! R\,}Q) \ar[dd] \ar@{-->}[ulllll]_-f \\
&& S \ar@{=}[dd] \ar[r]^-\ell & H \ar@{=}[dd] & P \ar[dd]^<<<<<{a}|>>>>{\phantom{m}} \ar[l]_-{u} & \\
& S\diagup_{\!\!\! H\,}R \ar[dr] &&&& P\diagup_{\!\!\! R\,}Q \ar[dddl] \ar[ul] \ar@/_2ex/[dll]_->>>>{wa \pr_P} \\
&& S \ar@{=}[dd] \ar[r]_-\ell & H \ar@{=}[dd] & R \ar[l]^-w & \\
& S\diagup_{\!\!\! H\,}Q \ar[uu] \ar[dr] &&&& \\
&& S \ar[r]_-\ell & H & Q \ar[uu]_b \ar[l]^v &
}
\end{align*}
(the unlabeled arrows are the canonical projections). This can be checked in $\CAT$ after applying $\Span(T,-)$ and identifying
$\Span(T, \textrm{?}/\textrm{??}) $ with $ \Span(T,\textrm{?}) / \Span(T,\textrm{??})$, as in \Cref{Rem:commas_translated} (compare the proof of Proposition~\ref{Prop:first-adjoints}). In $\CAT$, the functor $\Phi:= \Span(T,f)$ sends
\[
\big( \; s,(p,q, \xymatrix@1@L=1ex{ ap \ar@{=>}[r]^-{\varphi} & bq } ),
 \xymatrix@1@L=1ex{ \ell s \ar@{=>}[r]^-{\delta} & wap } \;\big)
\quad \quad \in \quad \Span (T, S\diagup_{\!\!\! H\,}(P\diagup_{\!\!\! R\,}Q))
\]
to
\[
\left(
\;
(s,p, \xymatrix@1@L=1ex{ \ell s\ar@{=>}[r]^-{(\alpha_1\inv p)\delta} & up })
\;
,
\;
(s , q , \xymatrix@1@C=20pt@L=1ex{ \ell s \ar@{=>}[rr]^-{(\beta_1\inv q) (w\varphi) \delta} && vq } )
\;
,
\;
\left(
\vcenter{\vbox{ \xymatrix@R=12pt@C=4pt@L=1ex{ s \ar@{=>}[d]_-{\id} & a p \ar@{=>}[d]^{\varphi} \\ s & bq } }}
\right)
\right)
\]
in $\Span (T, (S\diagup_{\!\!\! H\,}P)\diagup_{\!\!\! {\scriptscriptstyle (S\diagup_{\!\!\! H\,}R)}\,}(S\diagup_{\!\!\! H\,}Q))$. We can define a functor in the opposite direction
\[
\Psi\colon \Span (T, (S\diagup_{\!\!\! H\,}P)\diagup_{\!\!\! {\scriptscriptstyle (S\diagup_{\!\!\! H\,}R)}\,}(S\diagup_{\!\!\! H\,}Q)) \too \Span(T,S\diagup_{\!\!\! H\,}(P\diagup_{\!\!\! R\,}Q))
\]
by mapping
\[
\left(
\;
(s_1,p, \xymatrix@1@C=16pt@L=1ex{ \ell s_1\ar@{=>}[r]^-{\eps_1} & up })
\;
,
\;
(s_2 , q , \xymatrix@1@C=16pt@L=1ex{ \ell s_2 \ar@{=>}[r]^-{\eps_2} & vq } )
\;
,
\;
\left(
\vcenter{\vbox{ \xymatrix@R=12pt@C=4pt@L=1ex{ s_1 \ar@{=>}[d]_-{\zeta} & a p \ar@{=>}[d]^\xi \\ s_2 & bq } }}
\colon
\vcenter{\vbox{
\xymatrix@R=12pt@C=14pt@L=1ex{
\ell s_1 \ar@{=>}[d]_-{\ell \zeta} \ar@{=>}[r]^-{\eps_1} & up \ar@{=>}[r]^-{\alpha_1} & wap \ar@{=>}[d]^-{w \xi} \\ \ell s_2 \ar@{=>}[r]^-{\eps_2} & vq \ar@{=>}[r]^-{\beta_1} & wbq
}
}}
\right)
\right)
\]
to
\[
\big( \;
s_1 , (p,q, \xymatrix@1@L=1ex{ ap \ar@{=>}[r]^-{\xi} & bq } ) , \xymatrix@1{ \ell s_1 \ar@{=>}[r]^-{(\alpha_1\inv p)\eps_1} & wap } )
\; \big)
\,,
\]
and one checks immediately that $\Psi \Phi = \Id$ and $\Phi \Psi \cong \Id$.
\end{proof}

\bigbreak
\section{Heuristic account and the 2-dual version}
\label{sec:Span-co}%
\medskip

Of course, instead of creating \emph{left} adjoints in the construction of $\Span(\GG;\JJ)$ in \Cref{sec:Span}, we could equally well create a universal bicategory with \emph{right} adjoints. In this short section, we explain the dual story by staying as close as possible to the `left' $\Span(\GG;\JJ)$ studied so far. This approach will help us understand the ambidextrous bicategory of 2-motives that comes in \Cref{sec:Mackey-UP}. But first, we review the original construction of $\Span(\GG;\JJ)$ from a slightly less formal perspective.

Suppose we know nothing and want to construct a universal left-adjoint-creating bicategory $\Span(\GG;\JJ)$, that receives $\GG$ contravariantly on 1-cells. This bicategory has to contain a 1-cell $v^*\colon G\to H$ for every $v\colon H\to G$ in~$\GG$ and a 1-cell $i_!\colon G\to H$ left adjoint to~$i^*$ for every $i\colon G\to H$ in~$\JJ$. Just from these basic 1-cells, and before even invoking 2-cells, we must also already expect zig-zags of 1-cells of the type $u^*$ and~$i_!$. The Mackey formula allows us to reduce such arbitrary zig-zags to a single short zig-zag, or span, of the form $i_! u^*$. The `horizontal' composition of two such 1-cells $G\loto{u} P \oto{i} H \loto{v} Q \oto{j} K$ is circumvented by flipping the middle cospan
\[
\vcenter{\xymatrix@C=14pt@R=14pt{
&& (i/v) \ar[dl]_-{v'} \ar[dr]^-{i'}
 \ar@{}[dd]|(.5){\isocell{\gamma}}
\\
& P \ar[dr]_-{i} \ar[ld]_-{u}
&& Q \ar[dl]^-{v} \ar[rd]^-{j}
\\
G
&& H
&& K
}}
\]
into $j_! \circ (v^* \circ i_! ) \circ u^*\cong j_! \circ (i'_! \circ v'^*) \circ u^*\cong (ji')_!(uv')^*$; the result is again a single span with the $(-)_!$-part in~$\JJ$, as wanted. Then, with $i_!=[\loto{\Id} \oto{i}]$ and $u^*=[\loto{u}\oto{\Id}]$, we proved (using suitable 2-cells) that $i_!\circ u^*\cong i_! u^*$ in the reassuring \Cref{Rem:notation_assoc}. Thus we have a fair understanding of what 1-cells in $\Span(\GG;\JJ)$ ought to be.

Of course, ignoring 2-cells is not possible, and not only because of the many canonical isomorphisms~$\cong$ of 1-cells involved in horizontal composition, as in the previous paragraph. More importantly, $i_!$ cannot be a left adjoint to~$i^*$ unless we have adequate 2-cells to serve as unit $\Id\Rightarrow i^* i_!$ and counit $i_! i^*\Rightarrow\Id$. Note that the counit 2-cell $i_! i^*\Rightarrow\Id$ only involves our chosen short zig-zags but the unit $\Id\Rightarrow i^* i_!$ involves a zig-zag $i^*\circ i_!=[\loto{\Id}\oto{i}\loto{i}\oto{\Id}]$ that needs to be circumvented via Mackey. Leaving the latter aside, let us focus on 2-cells which are reasonably straightforward to express in terms of short zig-zags $(-)_!(-)^*$. So here is a list of such 2-cells that we definitely need in our universal bicategory $\Span(\GG;\JJ)$:
\begin{enumerate}[(1)]
\smallbreak
\item
\label{it:wish-1}%
The 2-cells $\alpha^*\colon u^*\Rightarrow v^*$ associated to every 2-cell $\alpha\colon u\Rightarrow v$ in~$\GG$.
\smallbreak
\item
\label{it:wish-2}%
The mates of the above when left adjoints exist, \ie for each $\alpha\colon i\Rightarrow j$ in~$\JJ$ there should be a 2-cell $\alpha_!\colon j_!\Rightarrow i_!$ in our bicategory, mate of $\alpha^*\colon i^*\Rightarrow j^*$.
\smallbreak
\item
\label{it:wish-3}%
The counit $\eps\colon a_! a^* \Rightarrow \Id$ for every 1-cell $a\in \JJ$. These can also be `squeezed' in the middle of a 2-cell $j_!v^*$ (when it make sense) to give new 2-cells
\[
(j a)_!(v a)^* \cong j_! a_! a^* v^* \oEcell{\eps} j_! v^*.
\]
\end{enumerate}
Assembling those three building blocks, we can see how a 2-cell from $i_!u^*$ to $j_!v^*$ in~$\Span(\GG;\JJ)$ should indeed contain those three types of information, as in \Cref{Def:Span-bicat}, or pictorially:
\begin{equation}
\label{eq:2-cell-of-Span-repeated}%
[a,\alpha_1,\alpha_2]=
\vcenter{\xymatrix@C=5em{
G \ar@{=}[d] \ar@{}[rrd]|{\alpha_1\SEcell\textrm{ see~(\ref{it:wish-1})}}
&& P \ar[ll]_-{u} \ar[rr]^-{i} \ar[d]_-{a}^-{\textrm{see~(\ref{it:wish-3})}} \ar@{}[rrd]|{\quad\alpha_2\,\NEcell\textrm{ see~(\ref{it:wish-2})}}
&& H \ar@{=}[d]
\\
G
&& Q \ar[ll]^-{v} \ar[rr]_-{j}
&& H\,.
}}
\end{equation}
The $\alpha_1$ goes ``from~$u$ towards~$v$\," as in~\eqref{it:wish-1} but the $\alpha_2$ goes backwards, ``from~$j$ towards~$i$\," as in~\eqref{it:wish-2}. More precisely, using~$a=\Id$, we have $\alpha^*=[\Id,\alpha,\id]\colon u^*\Rightarrow v^*$ for any $\alpha\colon u\Rightarrow v$ in~$\GG$ and $\alpha_!=[\Id,\id,\alpha]\colon i_!\Rightarrow j_!$ for $\alpha\colon j\Rightarrow i$ in~$\JJ$. These are the definitions of $\alpha^*$ and $\alpha_!$ that appear in \Cref{Cons:first-embeddings}. Finally, the 1-cell $a\in\JJ$ goes ``from $i_!u^*\cong (j a)_!(v a)^*$ towards $j_!v^*$\," as in~\eqref{it:wish-3}. Specifically, $\eps=[a,\id,\id]\colon a_! a^* \Rightarrow \id$ is the counit of $a_!\adj a^*$ as we saw in~\eqref{eq:can_counit}.

Similarly to what happened with 1-cells (namely $i_!\circ u^*\cong i_!u^*$), we saw in \Cref{Prop:2-cells-of-Span} that all 2-cells of~$\Span(\GG;\JJ)$ are generated by these three basic 2-cells~\eqref{it:wish-1}-\eqref{it:wish-2}-\eqref{it:wish-3}, together with the structure isomorphisms for the pseudo-functors $(-)^*\colon \GG^{\op} \hook \Span(\GG;\JJ)$ and $(-)_!\colon \JJ^{\co}\hook \Span(\GG;\JJ)$ given in \Cref{Rem:pseudo-func-embeddings}.

At the level of 2-cells of $\Span(\GG;\JJ)$, let us return to one peculiarity of the \emph{left} construction. In the list~\eqref{it:wish-1}-\eqref{it:wish-2}-\eqref{it:wish-3}, the third one features the counit $a_!a^*\Rightarrow \Id$ but not the unit. (The emotive reader should not worry about the fate of this unit. It still survives via the mates of~\eqref{it:wish-2}.) Therefore we should expect the right-adjoint-creating construction to involve a `dual' version of~\eqref{it:wish-3}, with only the \emph{unit} of the $i^*\adj i_*$ adjunction playing a role, not the counit anymore.

The universal property of~$\Span(\GG;\JJ)$ is proven by verifying that the now `obvious' construction works. If $\cat{F}:\GG^{\op}\to \cat{C}$ satisfies the `left Mackey conditions'~\eqref{it:UP-Span-a} and~\eqref{it:UP-Span-b} of \Cref{Thm:UP-Span} then its extension $\cat{G}\colon \Span(\GG;\JJ)\to \cat{C}$ is $\cat{F}$ on 0-cells, maps a 1-cell $i_!u^*$ to $\cat{F}(i)_!\circ\cat{F}(u)$ and maps a 2-cells $[a,\alpha_1,\alpha_2]$ to the pasting~\eqref{eq:pasting-in-UP-Span} in~$\cat{C}$, corresponding under~$\cat{F}$ to the description of $\alpha$ as a pasting in $\Span(\GG;\JJ)$, in the already mentioned \Cref{Prop:2-cells-of-Span}.

This concludes our heuristic review of $\Span(\GG;\JJ)$; see details in \Cref{sec:Span}.

\tristars

In view of the above discussion, we can see how to modify our bicategory in order to have the universal construction of \emph{right} adjoints. We again keep the same 0-cells as~$\GG$. On top of the `old' 1-cells $u^*$ from~$\GG$, we introduce for each $i\in \JJ$ a new `forward' 1-cell $i_*$ meant to become right adjoint to~$i^*$. Using the Mackey formula carefully, we reduce every zig-zag of $u^*$'s and $i_*$'s to just one $i_* u^*=[\loto{u}\oto{i}]$, which strongly resembles what we had before. And at the 2-cell level, adapting~\eqref{it:wish-1}-\eqref{it:wish-2}-\eqref{it:wish-3} above, we expect a similar story, namely:
\begin{enumerate}[(1$^{\co}$)]
\item
\label{it:wish-co-1}%
A 2-cell $\beta^*\colon u^*\Rightarrow v^*$ for every 2-cell $\beta\colon u\Rightarrow v$ in~$\GG$.
\smallbreak
\item
\label{it:wish-co-2}%
A 2-cell $\beta_*\colon j_*\Rightarrow i_*$ for each $\beta\colon i\Rightarrow j$ in~$\JJ$.
\smallbreak
\item
\label{it:wish-co-3}%
A unit $\Id \Rightarrow b^* b_*$ for every 1-cell $b\in \JJ$, which can also be `squeezed' in the middle of suitable 2-cells $j_*v^*$, as follows:
$
j_* v^* \Rightarrow j_* b_* b^* v^* \cong (j b)_*(v b)^* .
$
\end{enumerate}
Assembling them pictorially, we get the following 2-cells from $i_* u^*$ to $j_*v^*$:
\begin{equation}
\label{eq:2-cell-of-Span-co}%
[b,\beta_1,\beta_2]=
\vcenter{\xymatrix@C=5em{
G \ar@{=}[d] \ar@{}[rrd]|{\beta_1\SWcell\textrm{ see~(\ref{it:wish-co-1}}^{\co})}
&& P \ar[ll]_-{u} \ar[rr]^-{i} \ar@{<-}[d]_-{b}^-{\textrm{see~(\ref{it:wish-co-3}}^{\co})} \ar@{}[rrd]|{\qquad\beta_2\,\NWcell\textrm{ see~(\ref{it:wish-co-2}}^{\co})}
&& H \ar@{=}[d]
\\
G
&& Q \ar[ll]^-{v} \ar[rr]_-{j}
&& H\,.
}}
\end{equation}
Note in particular the \emph{reversal} of~$b$ in the middle. This leads to the following auxiliary definition, 2-dual to \Cref{Def:Span-bicat}:
\begin{Not}
We denote by $\Span'(\GG;\JJ)$ the following bicategory:
\begin{enumerate}[{$\bullet$}]
\smallbreak
\item
The 0-cells of $\Span'(\GG;\JJ)$ are those of~$\GG$, \ie the same as those of $\Span(\GG;\JJ)$.
\item
The 1-cells of $\Span'(\GG;\JJ)$ are spans $i_*\, u^*:=[\loto{u}\oto{i}]$ with $i\in \JJ$, \ie again the same as those of $\Span(\GG;\JJ)$, up to renaming $i_!u^*$ into~$i_*u^*$.
\item
The 2-cells of $\Span'(\GG;\JJ)$, between $i_*u^*$ and $j_*v^*$ are isomorphism classes $[b,\beta_1,\beta_2]$ of double 2-cells as in~\eqref{eq:2-cell-of-Span-co}, under the essentially obvious notion of isomorphism, completely analogous to that of \Cref{Def:Span-bicat}.
\item
Horizontal composition uses iso-commas as in \Cref{Def:Span-bicat}.
\item
Vertical composition of 2-cells is the obvious pasting.
\end{enumerate}
This bicategory $\Span'(\GG;\JJ)$ receives $\GG$ and~$\JJ$ via
\begin{equation}
\label{eq:embedding-co-prime}%
(-)^*\colon \GG^{\op} \hook \Span'(\GG;\JJ)
\qquadtext{and}
(-)_*\colon \JJ^{\co} \hook \Span'(\GG;\JJ)
\end{equation}
defined naturally by the identity on 0-cells, by $u^*=[\loto{u}=]$ and $i_*=[=\oto{i}]$ on 1-cells and by $\beta^*=[\Id,\beta,\id]$ and $\beta_!=[\Id,\id,\beta]$ on 2-cells, as in \Cref{Cons:first-embeddings}.
\end{Not}

A hurried reader might think that $\Span'(\GG;\JJ)$ is nothing but $\Span(\GG;\JJ)^{\co}$. This is essentially true up to a small subtlety in~\eqref{eq:2-cell-of-Span-co}. The direction of~$b$ is indeed opposite to that of~$a$ in~\eqref{eq:2-cell-of-Span-repeated} but the $\beta$'s go in the \emph{same} direction as the $\alpha$'s. In other words, if one simply reverted the $b$'s to try to identify $\Span'(\GG;\JJ)$ with $\Span(\GG;\JJ)^{\co}$, that is, if we flip~\eqref{eq:2-cell-of-Span-co} upside down, we get
\[
\vcenter{\xymatrix@C=5em{
G \ar@{=}[d] \ar@{}[rrd]|-{\beta_1\NWcell}
&& Q \ar[ll]_-{v} \ar[rr]^-{j} \ar[d]_-{b} \ar@{}[rrd]|-{\qquad\beta_2\,\SWcell}
&& H \ar@{=}[d]
\\
G
&& P \ar[ll]^-{u} \ar[rr]_-{i}
&& H
}}
\]
with the $\beta$'s in the wrong direction. In the (2,1)-category~$\GG$, this can be fixed very easily by inverting the 2-cells of~$\GG$. In other words, the canonical isomorphism
\[
\Span'(\GG;\JJ) \overset{\cong}{\too} \Span(\GG;\JJ)^{\co}
\]
is the identity on 0-cells, is the renaming (identity) on 1-cells $i_*u^* \mapsto i_!u^*$ but on 2-cells it involves inverting the 2-cells of~$\GG$:
\begin{equation}
\label{eq:Span-co-Span}%
\begin{array}{ccc}
[b,\beta_1,\beta_2]\colon i_* u^*\Rightarrow j_* v^*
& \mapsto
& [b,\beta_1\inv,\beta_2\inv]\colon j_! v^* \Rightarrow i_! u^*
\\[.8em]
\textrm{in }\Span'(\GG;\JJ)
&& \textrm{in }\Span(\GG;\JJ)
\\[.2em]
\xymatrix@C=2em{
G \ar@{=}[d] \ar@{}[rrd]|{\beta_1\SWcell}
&& P \ar[ll]_-{u} \ar[rr]^-{i} \ar@{<-}[d]_-{b} \ar@{}[rrd]|{\beta_2\,\NWcell}
&& H \ar@{=}[d]
\\
G
&& Q \ar[ll]^-{v} \ar[rr]_-{j}
&& H
}
&&
\xymatrix@C=2em{
G \ar@{=}[d] \ar@{}[rrd]|{\beta_1\inv\SEcell}
&& Q \ar[ll]_-{v} \ar[rr]^-{j} \ar[d]_-{b} \ar@{}[rrd]|{\beta_2\inv\,\NEcell}
&& H \ar@{=}[d]
\\
G
&& P \ar[ll]^-{u} \ar[rr]_-{i}
&& H
}
\end{array}
\end{equation}
\begin{Cons}
\label{Cons:first-embeddings-co}%
Under the isomorphism of bicategories~\eqref{eq:Span-co-Span}, the canonical embeddings of~\eqref{eq:embedding-co-prime} become the following
\[
(-)^{\costar}\colon \GG^{\op}\to \Span(\GG;\JJ)^{\co}
\qquadtext{and}
(-)_{\costar}\colon \JJ^{\co}\to \Span(\GG;\JJ)^{\co}
\]
where the first one maps $u\colon H\to G$ to $u^*=[G\loto{u}H=H]$ and $\beta\colon u\Rightarrow v$ to $[\Id,\beta\inv,\id]$, and the second one maps $i\colon H\to G$ to $i_*=[H=H\oto{i}G]$ and $\beta\colon i\Rightarrow j$ to $[\Id,\id,\beta\inv]$. We draw the reader's attention to the inverses appearing on 2-cells.
\end{Cons}

A last devilish detail is hidden in the left-vs-right Mackey formulas. A priori, for a pseudo-functor $\cat{F}\colon \GG^{\op}\to\cat{C}$ to a bicategory in which each $\cat{F}i$ admits a \emph{right} adjoint~$(\cat{F}i)_*$, the Mackey condition~\eqref{it:UP-Span-b} in \Cref{Thm:UP-Span} should say the following: For every (co) iso-comma with $i\in \JJ$
\begin{equation}
\label{eq:iso-comma-dual}%
\vcenter{\xymatrix@C=14pt@R=14pt{
& (i\bs u) \ar[dl]_-{v} \ar[dr]^-{j}
 \ar@{}[dd]|(.5){\oWcell{\delta}}
\\
H \ar[dr]_-{i}
&& K \ar[dl]^-{u}
\\
&G
}}
\end{equation}
the mate $\delta_*\colon u^* i_* \Rightarrow j_* v^* $ of $\delta^*\colon j^* u^* \Rightarrow v^* i^*$ is an isomorphism. Note the co-comma $(i\bs u)$ instead of $(i/u)$ and the direction of the 2-cell, which is not anodyne since the position of~$i$ is imposed upon us in order to make sense in composition of spans (reduction of zig-zags). However, since we work with \emph{iso}-commas, the above trick of inverting 2-cells yields a 1-to-1 correspondence between iso-commas and co-iso-commas. Whenever we state the \emph{right}-adjoint versions by using `\emph{left}' iso-commas $(i/u)$ we should expect an inverse $\gamma\inv$ to enter the game.

\medbreak

In summary, we have the following dual to \Cref{Thm:UP-Span}:
\begin{Thm}
\label{Thm:UP-Span-co}%
Let $\GG$ and $\JJ$ be as in Hypotheses~\ref{Hyp:G_and_I_for_Span}. Let $\cat{C}$ be any 2-category, and let $\cat{F}\colon \GG^{\op}\to \cat{C}$ be a pseudo-functor such that
\begin{enumerate}[\rm(a)]
\item
\label{it:UP-Span-co-a}%
for every $i\in \JJ$, there exists in $\cat{C}$ a \emph{right} adjoint $(\cat{F}i)_*$ to~$\cat{F}i$;
\smallbreak
\item
\label{it:UP-Span-co-b}%
the adjunctions $\cat{F}i\adj (\cat{F}i)_*$ satisfy base-change with respect to all Mackey squares with two parallel sides in~$\JJ$, in the following sense: Given an iso-comma square in~$\GG$ with $i\in \JJ$
\[
\vcenter{\xymatrix@C=14pt@R=14pt{
& (i/u) \ar[dl]_-{v} \ar[dr]^-{j}
 \ar@{}[dd]|(.5){\isocell{\gamma}}
\\
H \ar[dr]_-{i}
&& K \ar[dl]^-{u}
\\
&G
}}
\]
the mate $(\gamma\inv)_*\colon \cat{F}u (\cat{F}i)_* \Rightarrow (\cat{F}j)_*\cat{F}v$
of $(\gamma\inv)^*=\cat F(\gamma\inv)\colon \cat{F}j \cat{F}u \Rightarrow \cat{F}v \cat{F}i$ with respect to the adjunctions $\cat{F}i\adj (\cat{F}i)_*$ and $\cat{F}j\adj (\cat{F}j)_*$ is an isomorphism in~$\cat{C}(\cat{F}H,\cat{F}K)$.
\end{enumerate}
Then there exists a unique pseudo-functor $\cat{G}\colon \Span(\GG; \JJ)^{\co}\to \cat{C} $ such that
\[
\xymatrix{
\GG^{\op} \ar[d]_-{(-)^{\costar}} \ar[r]^-{\cat{F}} & \cat{C} \\
\Span(\GG;\JJ)^{\co} \ar@{-->}[ru]_-{\cat{G}} &
}
\]
is commutative, where $(-)^{\costar}\colon \GG\to \Span(\GG;\JJ)^{\co}$ is the pseudo-functor of \Cref{Cons:first-embeddings-co}. This extension $\cat{G}$ is unique up to a unique isomorphism restricting to the identity of~$\cat{F}$, and is entirely determined by the choice of the right adjoints, with units and counits, for all~$\cat{F}i$. Conversely, any pseudo-functor $\cat{F}\colon \GG^{\op}\to \cat{C}$ factoring as above must enjoy the above two properties \eqref{it:UP-Span-co-a} and~\eqref{it:UP-Span-co-b}.
\end{Thm}

\begin{Cons}
\label{Cons:UP-Span-co}%
(Compare \Cref{Cons:UP-Span}.) Explicitly, $\cat{G}$ is defined as follows. Assemble as in \Cref{Rem:pseudo-func-of-adjoints} all $(\cat{F}i)_*$ into a pseudo-functor $\cat{F}_*\colon \JJ^{\co}\to \cat{C}$. On 0-cells, the pseudo-functor~$\cat{G}\colon \Span(\GG;\JJ)^{\co}\to \cat{C}$ agrees with $\cat{F}$. On 1-cells, set
\[
\cat{G}(i_*u^*):= \cat{F}_*(i) \circ \cat{F}(u) = (\cat{F} i)_* \circ (\cat{F}u).
\]
For a 2-cell $[b,\beta_1,\beta_2]\colon i_* u^* \Rightarrow j_* v^*$ of $\Span(\GG;\JJ)$ represented by a diagram
\begin{align*}
\xymatrix{
G \ar@{=}[d] \ar@{}[rrd]|{\SEcell\,\beta_1} &&
 P \ar[ll]_-{u} \ar[rr]^-{i} \ar[d]^b \ar@{}[rrd]|{\NEcell\,\beta_2} &&
 H \ar@{=}[d] \\
G &&
 Q \ar[ll]^-{v} \ar[rr]_-{j} &&
 H
}
\end{align*}
define its image $\cat{G}([b,\beta_1,\beta_2])\colon \cat{G}(j_* v^*)\Rightarrow \cat{G}(i_* u^*)$ to be the following pasting in~$\cat{C}$:
\begin{equation}
\label{eq:pasting-in-UP-Span-co}%
\cat{G}([b,\beta_1,\beta_2])=\qquad
\vcenter{\xymatrix{
\cat{F} G \ar[r]^-{\cat{F} v}
& \cat{F} Q \ar[rr]^-{\id}
&& \cat{F} Q \ar[r]^-{(\cat{F} j)_*}
& \cat{F} H
\\
\cat{F} G \ar[r]^-{\cat{F} v} \ar@{=}[u]
& \cat{F} Q \ar@{}[urr]|{\Scell\; \eta} \ar@{=}[u] \ar[r]^-{\cat{F} b}
& \cat{F} P \ar[r]^-{(\cat{F} b)_*}
& \cat{F} Q \ar@{=}[u] \ar[r]^-{(\cat{F} j)_*}
& \cat{F} H \ar@{=}[u]
\\
\cat{F} G \ar@{}[urr]|{\Scell\;\simeq} \ar@{=}[u] \ar[rr]^-{\cat{F} (v b)}
&& \cat{F} P \ar@{}[urr]|{\Scell\;\simeq} \ar@{=}[u] \ar[rr]^-{(\cat{F} j b)_*}
&& \cat{F} H \ar@{=}[u]
\\
\cat{F} G \ar@{=}[u] \ar[rr]_-{\cat{F} u} \ar@{}[urr]|{\Scell\; \cat{F} \beta_1\inv}
&& \cat{F} P \ar@{=}[u] \ar[rr]_-{(\cat{F} i)_*} \ar@{}[urr]|{\Scell\; (\cat{F} \beta_2\inv)_*}
&& \cat{F} H\,.\!\! \ar@{=}[u]
}}
\end{equation}
Note the presence of the inverses, which will compensate for the inverses involved in $(-)^{\costar}\colon \GG^{\op}\hook\Span(\GG;\JJ)^{\co}$ of \Cref{Cons:first-embeddings-co}, so that $\cat{G}\circ(-)^{\costar}=\cat{F}$.
\end{Cons}

Let us finish with the `bicategorical upgrade' (see \Cref{sec:UP-Span-bicat}) of the universal property of~$\Span(\GG;\JJ)^\co$, dual to the bicategorical upgrade for $\Span(\GG;\JJ)$ discussed in \Cref{Thm:UP-PsFun-Span}. In other words, let us add pseudo-natural transformations and modifications to \Cref{Thm:UP-Span-co}. As in \Cref{Def:J_!-pseudo-functor} and \Cref{Def:J_!-strong}, we give a name to the relevant properties. We fix a target bicategory~$\cat{B}$.
\begin{Def}
\label{Def:J_*-stuff}%
A pseudo-functor $\cat{F}\colon \GG^\op\too \cat{B}$ is called a \emph{$\JJ_*$-pseudo-functor} if it satisfies conditions~\eqref{it:UP-Span-co-a} and~\eqref{it:UP-Span-co-b} of \Cref{Thm:UP-Span-co}, namely:
\begin{enumerate}[\rm(a)]
\smallbreak
\item
For every 1-cell $i\in \JJ$, the 1-cell $\cat{F}i$ admits a right adjoint $(\cat{F}i)_*$ in~$\cat{B}$.\,
\smallbreak
\item
For every comma square $\gamma$ along an $i\in \JJ$, its mate $(\gamma\inv)_*$ is invertible:
\begin{align*}
\xymatrix@C=14pt@R=14pt{
& i/v \ar[ld]_-{\tilde v} \ar[dr]^-{\tilde i} & \\
X \ar[dr]_i \ar@{}[rr]|{\oEcell{\gamma}} && Y \ar[dl]^v \\
&Z &
}
\quad \quad \quad \quad
\xymatrix@C=14pt@R=14pt{
& \cat{F} (i/v)
\ar[dr]^-{(\cat{F}\tilde i)_*} & \\
\cat{F} X
\ar[rd]_-{(\cat{F} i)_*}
 \ar[ur]^-{\cat{F}\tilde v}
 \ar@{}[rr]|{\Ncell\, (\gamma\inv)_*} &&
 \cat{F} Y. \\
& \cat{F} Z
\ar[ru]_-{\cat{F} v} &
}
\end{align*}
\end{enumerate}
A (strong) pseudo-natural transformation $t\colon \cat{F}_1\to \cat{F}_2$ between two $\JJ_*$-strong pseudo-functors $\GG^{\op}\to \cat{B}$ is \emph{$\JJ_*$-strong} if for every $i\in \JJ$ the following mate~$(t_i)_*$
\begin{equation}
\label{eq:t_i_*}%
\vcenter { \hbox{
\xymatrix{
\cat{F}_1 X
 \ar[r]^-{t_X}
 \ar[d]_-{(\cat{F}_1 i)_*} &
 \cat{F}_2 X
 \ar[d]^-{(\cat{F}_2 i)_*} \\
\cat{F}_1 Y
\ar@{}[ur]|{\NEcell \; (t_i)_*}
 \ar[r]_-{t_Y} &
 \cat{F}_2 Y
}
}}
\quad \stackrel{\textrm{def.}}{=}\!
\vcenter { \hbox{
\xymatrix{
\ar@{}[dr]|{\NWcell} &
 \cat{F}_1 X
\ar@{}[dr]|{\NWcell\; t_i}
 \ar[r]^-{t_X} &
 \cat{F}_2 X
 \ar@{}[dr]|{\NWcell}
 \ar[r]^-{(\cat{F}_2 i)_*} &
 \cat{F}_2 Y
\\
\cat{F}_1X
 \ar[r]_-{(\cat{F}_1 i)_*}
\ar@/^3ex/@{=}[ur] &
 \cat{F}_1 Y
 \ar[r]_-{t_Y}
 \ar[u]^-{\cat{F}_1 i} &
 \cat{F}_2 Y
 \ar[u]_-{\cat{F}_2 i}
 \ar@/_3ex/@{=}[ru] &
}\kern-1em
}}
\end{equation}
is invertible: $(t_i)_*\colon (\cat{F}_2 i)_*\,t_X\isoEcell t_Y\,(\cat{F}_1 i)_*$. (Note that, this time, it is $(t_i\inv)_*$ which would not make sense!)
\end{Def}

\begin{Thm}
\label{Thm:UP-PsFun-Span-co}%
Let $\cat{B}$ be a bicategory. Precomposition with the pseudo-functor $(-)^{**}\colon \GG^{\op}\to \Span(\GG;\JJ)^\co$ of \Cref{Cons:first-embeddings-co} induces a biequivalence
\begin{equation}
\label{eq:UP-PsFun-Span-co}%
\PsFun\!\big(\Span(\GG;\JJ)^\co, \cat{B}\big) \stackrel{\sim}{\longrightarrow} \PsFunJop^{}( \GG^{\op}, \cat{B})
\end{equation}
where the right-hand bicategory has $\JJ_*$-pseudo-functors as 0-cells, $\JJ_*$-strong pseudo-natural transformations as 1-cells and all modifications as 2-cells.

If $\cat{B}$ happens to be a 2-category, then the above is a strict 2-functor between two 2-categories, and said 2-functor is furthermore locally strict.
\end{Thm}

\begin{proof}
The proof of this result is dual to the one of \Cref{Thm:UP-PsFun-Span}. Concretely, in terms of string diagrams the two proofs are upside-down mirrors of each other.

More precisely, and more generally for two bicategories $\cat A$ and~$\cat B$, there is an isomorphism of pseudo-functor bicategories
\begin{equation} \label{eq:co-formula}
(-)^\co\colon \PsFun(\cat{A},\cat{B})^\co \overset{\sim}{\too} \PsFun(\cat{A}^\co ,\cat{B}^\co)
\end{equation}
which is defined as follows (see also \Cref{Not:co-op}). A pseudo-functor $\cat F\colon \cat A\to \cat B$ is sent to the pseudo-functor $\cat F^\co\colon \cat{A}^\co \to\cat{B}^\co$ defined by $\cat F^\co=\cat F$ on 0- and 1-cells and by $\cat F^\co(\alpha^\co):=(\cat F\alpha)^\co$ on 2-cells; the structural isomorphisms of $\cat F^\co$ require taking inverses: $\fun_{\cat F^\co}:=(\fun_{\cat F}^{-1})^\co$ and $\un_{\cat F^\co}:=(\un_{\cat F}^{-1})^\co$.
A (strong, oplax-oriented) pseudo-natural transformation $t=\{t_X,t_u\}\colon \cat F_1\to \cat F_2$ maps to the (strong, oplax-oriented) transformation $t^\co$ with the same 0-cell components $t^\co_X=t_X$ and with 1-cell components $(t^{\co})_u:=(t_u^{-1})^\co$ defined, again, by taking inverses. Finally, a modification $M=\{M_X\colon t_X\Rightarrow s_X\}$ gives rise to a modification $M^\co$ with components $(M^\co)_X:=(M_X)^\co\colon s_X^\co\to t_X^\co$ going in the opposite direction.

Now we can combine \Cref{Thm:UP-PsFun-Span} and the above formula with $\Span(\GG;\JJ)^\co$ and $\GG^{\op,\co}$ instead of $\cat A$ and with $\cat B^\co$ instead of~$\cat B$. If we moreover pre-compose with the isomorphism $\alpha \mapsto \alpha^{-1}$ (so that pseudo-functors again start at $\GG^\op$ rather than~$\GG^{\op,\co}$), this has the total effect of replacing $(-)^*$ with~$(-)^{**}$, left mates $(-)_!$ in~$\cat B$ with right mates~$(-)_*$,
\Cref{Cons:UP-Span} with \Cref{Cons:UP-Span-co}, and $\JJ_!$-strength with $\JJ_*$-strength, so that we obtain the claims of \Cref{Thm:UP-PsFun-Span-co}.
\[
\xymatrix{
\PsFun(\Span, \cat B^\co)
 \ar[rr]^-{\overset{\textrm{Thm.\,\ref{Thm:UP-PsFun-Span}}}{\simeq}}_-{(-)^*} &&
 \PsFunJ(\GG^\op, \cat B^\co) \\
\PsFun(\Span^\co, \cat B)^\co
 \ar[rr]_-{(-)^*}
 \ar[u]^{\overset{\eqref{eq:co-formula}}{\cong}}_{(-)^\co}
 \ar@/_2ex/[drr]_{(-)^{**}} &&
 \PsFunJop(\GG^{\op,\co}, \cat B)^\co
 \ar[u]_{\overset{\eqref{eq:co-formula}}{\cong}}^{(-)^\co} \\
&&
 \PsFunJop(\GG^\op, \cat B)^\co
 \ar[u]_{\cong}^{(-)^{-1}}
}
\]
There is still an extra `$\co$' on the outside of the bottom $\PsFun$'s, meaning that modifications go in the other direction, but of course the result follows by 2-dualizing the biequivalence obtained at the bottom.
\end{proof}

\begin{Rem}
Note that in the definition of~\eqref{eq:co-formula} we could completely avoid taking inverses, thus obtaining a more natural-looking isomorphism~$(-)^\co$. But the new isomorphism would then map an \emph{oplax}-oriented transformation (the convention we have fixed throughout) to a \emph{lax}-oriented one, and similarly it would require reverting the direction of the structural isomorphisms $\fun$ and~$\un$ of pseudo-functors; \cf \ref{Ter:pseudofun} and \ref{Ter:Hom_bicats}. Thus the new formula would require the simultaneous use of both sets of conventions, both for pseudo-functors and for transformations.
\end{Rem}

\begin{Rem}
\label{Rem:inverse}%
The various $\co$'s and $\op$'s and inverses of the proof may be confusing, so we should spell out in detail how to perform the extension of transformations implicit in \Cref{Thm:UP-PsFun-Span-co}.
Suppose for simplicity that $\cat{B}=\cat{C}$ is an actual 2-category and let $\cat{G}_1\,,\,\cat{G}_2\colon \Span(\GG;\JJ)^\co\to \cat{C}$ be two pseudo-functors, with the corresponding $\JJ_*$-pseudo-functors $\cat{F}_1=\cat{G}_1\circ(-)^{**}$ and $\cat{F}_2=\cat{G}_2\circ(-)^{**}\,\colon \GG^\op\to \cat{C}$. Let $t\colon \cat{F}_1\Rightarrow \cat{F}_2$ be a transformation and let us analyze its unique extension $t\colon \cat{G}_1\Rightarrow \cat{G}_2$ given by the above theorem. In the proof of \Cref{Thm:UP-PsFun-Span}, we gave an explicit formula~\eqref{eq:decomp_comp} for the extension in the case of~$\Span$. The analogous formula for the present case of $\Span^\co$ differs slightly. Namely, we still keep $t_X$ on 0-cells $X\in \GG_0=\Span^\co_0$ and we still have $t(u^*)=t(u)$ for every $u\in\GG_1=\Span^\co_1$. However, for $i\in\JJ$, we now set
\[
t(i_*):=(t(i)_*)^{-1}
\]
where we take the inverse \emph{after} forming the mate rather than before. Note that there is no other choice. The $\JJ_*$-strength of~$t\colon \cat{F}_1\Rightarrow\cat{F}_2$ precisely guarantees that the above inverse makes sense and consequently that the extended transformation~$t\colon \cat{G}_1\Rightarrow \cat{G}_2$ remains strong. In other words, here the correct analogue of~\eqref{eq:decomp_comp} is given by the formula
\begin{equation}
\label{eq:decomp_comp-co}%
t(i_*u^*)\;:=\qquad\vcenter{\xymatrix{
&  \ar[rr]^-t \ar[d]_-{\cat{G}_1 u^*}
 \ar@/_9ex/[dd]_-{\cat{G}_1 (i_* \circ u^*)}
 \ar@{}[ddl]|{\SWcell\;\simeq} &&
  \ar[d]^-{\cat{G}_2 u^*}
 \ar@/^9ex/[dd]^-{\cat{G}_2(i_*\circ u^*)}
 \ar@{}[dll]|{\SWcell\, t(u)} &
\ar@{}[ddl]|{\SWcell\; \simeq} \\
&  \ar[rr]^-t \ar[d]_-{\cat{G}_1 i_*} &&
  \ar[d]^-{\cat{G}_2 i_*}
 \ar@{}[dll]|{\SWcell\, t (i)_*^{-1}} & \\
&  \ar[rr]_-t &&  &
}}
\end{equation}
(modulo the canonical isomorphisms $i_*u^*\cong i_*\circ u^*$, of course).
\end{Rem}

\end{chapter-five}
%
\chapter{Mackey 2-motives}
\label{ch:2-motives}%
\bigbreak
\begin{chapter-six}
\bigbreak
\section{Mackey 2-motives and their universal property}
\label{sec:Mackey-UP}%
\medskip

We are now ready to construct the universal Mackey 2-functor in its \emph{semi-additive} incarnation (see \Cref{sec:additive-sedative}).
As in \Cref{ch:bicat-spans}, we allow ourselves to work with a slightly more general setting.
Thus let $\GG$ be an essentially small strict (2,1)-category equipped with an admissible class $\JJ$ of faithful 1-cells, as in Hypotheses~\ref{Hyp:G_and_I_for_Span}, and let $\Span=\Span(\GG;\JJ)$ denote the bicategory of spans of \Cref{Def:Span-bicat}.

\begin{Def}
\label{Def:Spanhat-bicat}%
\index{bicategory!-- of Mackey 2-motives $\Spanhat(\GG;\JJ)$} \index{SpanhatG@$\Spanhat(\GG;\JJ)$}%
\index{Mackey 2-motives}%
\index{$spanhat$@$\Spanhat$ \, (semi-additive) Mackey 2-motives}%
The \emph{bicategory of (semi-additive) Mackey 2-motives}
\[
\Spanhat = \Spanhat(\GG;\JJ)
\]
is the bicategory obtained from $\Span$ by forming the ordinary category of spans on each Hom category $\Span(G,H)$. Thus, in more details, $\Spanhat$ consists of the following:
\begin{enumerate}[{$\bullet$}]
\item The same objects $\Spanhat_0=\Span_0 = \GG_0$.
\item The same 1-cells $\Spanhat_1=\Span_1$, \ie spans whose right leg belongs to~$\JJ$.
\item For two objects $G,H\in \Spanhat_0$, the Hom category $\Spanhat(G,H)$ is defined to be the 1-category of isomorphism classes of spans in $\Span(G,H)$.
In the notation of \Cref{Def:ordinary-spans}:
\[
\Spanhat(G,H) := \widehat{\Span(G,H)} \,.
\]
This makes sense, because by Proposition~\ref{Prop:pullbacks} the category $\Span(G,H)$ admits pullbacks. Thus now both the vertical and horizontal compositions are computed by taking iso-comma squares in~$\GG$.
\item The horizontal composition functors are the unique extensions
\[
\xymatrix{
\Span(H,K) \times \Span (G,H) \ar[d]_-{{(-)_\star} \times {(-)_\star}} \ar[rr]^-{\circ} && \Span (G,K) \ar[d]^-{(-)_\star} \\
\Spanhat(H,K) \times \Spanhat (G,H) \ar@{-->}[rr]^-{\hat\circ} && \Spanhat (G,K) \\
\Span(H,K)^{\op} \times \Span (G,H)^{\op} \ar[u]^-{(-)^\star\times (-)^\star} \ar[rr]^-{\circ^{\op}} && \Span (G,K)^{\op} \ar[u]_-{(-)^\star}
}
\]
of the composition functors of $\Span$ along the canonical covariant and contravariant embeddings of $\Span(A,B)$ in $\Spanhat (A,B)$. Since horizontal composition of $\Span$ preserves pullbacks (\Cref{Lem:pres_pullbacks}), these extensions exist by the functoriality and product-compatibility of ordinary spans (\Cref{Lem:functoriality_ordinary-spans}).
\item The unitor and associator 2-cells of $\Spanhat$ are simply the canonical covariant images under $(-)_\star$ of those of~$\Span$.
\end{enumerate}
It is immediate to verify that this data forms a well-defined bicategory.
\end{Def}

\begin{Rem}
\label{Rem:2-cells-in-Spanhat}%
We shall of course try to avoid expanding the multiple layers of construction as much as possible. Still, the reader might want to take a quick look down the barrel at least once. Objects of~$\Spanhat(\GG;\JJ)$ are very easy, just $G\in\GG_0$. The 1-cells $G\to H$ are given by spans $(G\loto{u}P\oto{i}H)$ with $i\in \JJ$. These 1-cells will be denoted $i_!u^*$ (by choice) but they are also $i_*u^*$. In particular, in $\Spanhat$ we have
\[
i_!=[\loto{\Id}\oto{i}]=i_*\,.
\]
This agreement is in line with the discussion of \Cref{sec:Span-co}. Indeed, in \Cref{Thm:UP-Span-co}, we saw that the right-adjoint-creating bicategory $\Span(\GG;\JJ)^{\co}$ only differs from the left one, $\Span(\GG;\JJ)$, at the level of 2-cells. Our ambidexter $\Spanhat$ will contain both sorts of 2-cells, those of $\Span$ and those of $\Span^{\co}$. Explicitly, a 2-cell from $i_!u^*$ to $j_!v^*$ consists of a diagram
\begin{equation} \label{eq:2-cell-of-Spanhat}
\vcenter { \hbox{
\xymatrix@L=1ex{
i_!u^* \\
k_!w^* \ar@{=>}[u]_-{[b,\beta_1,\beta_2]}
 \ar@{=>}[d]^-{[a,\alpha_1,\alpha_2]} \\
j_!v^*
}
}}
\quad\quad
\vcenter { \hbox{
\xymatrix@C=4em{
G \ar@{=}[d] \ar@{}[dr]|{\NEcell\,\beta_1}
& P \ar[l]_-{u} \ar[r]^-{i} \ar@{}[dr]|{\SEcell\,\beta_2}
& H \ar@{=}[d]
\\
G \ar@{}[dr]|{\SEcell\,\alpha_1} \ar@{=}[d]
& R \ar[u]_-{b} \ar[d]^-{a} \ar[l]_-{w} \ar[r]^-{k} \ar@{}[dr]|{\NEcell\,\alpha_2}
& H \ar@{=}[d]
\\
G
& Q \ar[l]^-{v} \ar[r]_-{j}
& H
}
}}
\end{equation}
up to isomorphism. A priori, there are two layers of `isomorphism' here, namely both $[a,\alpha_1,\alpha_2]$ and $[b,\beta_1,\beta_2]$ are already isomorphism classes but, moreover, the above (vertical) span $\loto{[b,...]}\oto{[a,...]}$ is only considered up to isomorphism in the category $\Span(G,H)$. In other words, an isomorphism between two such spans of spans
\[
\vcenter{
\xymatrix@C=4em{
\bullet \ar@{=}[d] \ar@{}[dr]|{\NEcell\,\beta_1} &
 \bullet \ar[l]_-{u} \ar[r]^-{i} \ar@{}[dr]|{\SEcell\,\beta_2} &
 \bullet \ar@{=}[d] \\
\bullet \ar@{}[dr]|{\SEcell\,\alpha_1} \ar@{=}[d] &
 \bullet \ar[u]_b \ar[d]^a \ar[l]_-{w} \ar[r]^-{k} \ar@{}[dr]|{\NEcell\,\alpha_2} &
 \bullet \ar@{=}[d] \\
\bullet & \bullet \ar[l]^-{v} \ar[r]_-{j} &
 \bullet
}}
\qquad
\simeq
\qquad
\vcenter{
\xymatrix@C=4em{
\bullet \ar@{=}[d] \ar@{}[dr]|{\NEcell\,\overline{\beta}_1} &
 \bullet \ar[l]_-{u} \ar[r]^-{i} \ar@{}[dr]|{\SEcell\,\overline{\beta}_2} &
 \bullet \ar@{=}[d] \\
\bullet \ar@{}[dr]|{\SEcell\,\overline{\alpha}_1} \ar@{=}[d] &
 \bullet \ar[u]_-{\overline{b}} \ar[d]^-{\overline{a}} \ar[l]_-{\overline{w}} \ar[r]^-{\overline{k}} \ar@{}[dr]|{\NEcell\,\overline{\alpha}_2} &
 \bullet \ar@{=}[d] \\
\bullet & \bullet \ar[l]^-{v} \ar[r]_-{j} &
 \bullet
}}
\]
consists of commutative diagrams of 2-cells in~$\Span$
\[
\vcenter { \hbox{
\xymatrix@C=10pt@L=1ex{
& i_!u^* & \\
k_!w^* \ar@{=>}[rr]^-{[f]} \ar@{=>}[ur]^-{[b]} \ar@{=>}[dr]_-{[a]} &&
 {\overline{k}_!\overline{w}^*} \ar@{=>}[ul]_-{[\overline b]} \ar@{=>}[dl]^-{[\overline a]} \\
& j_!v^* &
}
}}
\quad\quad
\textrm{ and }
\quad\quad
\vcenter { \hbox{
\xymatrix@C=10pt@L=1ex{
& i_!u^* & \\
k_!w^* \ar@{=>}[ur]^-{[b]} \ar@{=>}[dr]_-{[a]} &&
 {\overline{k}_!\overline{w}^*} \ar@{=>}[ll]_-{[g]} \ar@{=>}[ul]_-{[\overline b]} \ar@{=>}[dl]^-{[\overline a]} \\
& j_!v^* &
}
}}
\]
such that $[f]\circ [g] = \Id_{\overline{k}_! \overline{w}^*}$ and $[g]\circ [f]= \Id_{k_! w^*}$.
In particular, $f$ and $g$ are equivalences of~$\GG$ (Lemma~\ref{Lem:equiv-are-iso}). Expanding further and writing $[f]=[f,\varphi_1,\varphi_2]$, the commutativity of the two triangles on the left means that there are 2-cells $\tau\colon \overline b f \Rightarrow b$ and $\sigma\colon \overline a f \Rightarrow a$ of~$\GG$ providing two isomorphisms of maps of spans:
\[
\vcenter { \hbox{
\xymatrix@C=4em{
\bullet \ar@{=}[d] \ar@{}[dr]|{\NEcell\,\overline{\beta}_1} &
 \bullet \ar[l]_-{u} \ar[r]^-{i} \ar@{}[dr]|{\SEcell\,\overline{\beta}_2} &
 \bullet \ar@{=}[d] \\
\bullet \ar@{}[dr]|{\NEcell\,\varphi_1} \ar@{=}[d] &
 \bullet \ar[u]_-{\overline{b}} \ar@{<-}[d]^f \ar[l]_-{\overline w} \ar[r]^-{\overline k} \ar@{}[dr]|{\SEcell\,\varphi_2} &
 \bullet \ar@{=}[d] \\
\bullet & \bullet \ar[l]^-w \ar[r]_-k &
 \bullet
}
}}
\quad
\overset{\underset{}{\tau}}{\Rightarrow}
\quad
\vcenter { \hbox{
\xymatrix@C=4em{
\bullet \ar@{=}[dd] \ar@{}[ddr]|{\NEcell\,\beta_1} &
 \bullet \ar[l]_-{u} \ar[r]^-{i} \ar@{}[ddr]|{\SEcell\,\beta_2} &
 \bullet \ar@{=}[dd] \\
&& \\
\bullet & \bullet \ar[uu]_-<<<<<<{b} \ar[l]^-w \ar[r]_-k &
 \bullet
}
}}
\]
and
\[
\vcenter { \hbox{
\xymatrix@C=4em{
\bullet \ar@{=}[d] \ar@{}[dr]|{\SEcell\,\varphi_1} &
 \bullet \ar[l]_-{w} \ar[r]^-{k} \ar@{}[dr]|{\NEcell\,\varphi_2} &
 \bullet \ar@{=}[d] \\
\bullet \ar@{}[dr]|{\SEcell\,\overline{\alpha}_1} \ar@{=}[d] &
 \bullet \ar@{<-}[u]_f \ar[d]^-{\overline a} \ar[l]_-{\overline w} \ar[r]^-{\overline k} \ar@{}[dr]|{\NEcell\,\overline{\alpha}_2} &
 \bullet \ar@{=}[d] \\
\bullet & \bullet \ar[l]^-{v} \ar[r]_-{j} &
 \bullet
}
}}
\quad
\overset{\underset{}{\sigma}}{\Rightarrow}
\quad
\vcenter { \hbox{
\xymatrix@C=4em{
\bullet \ar@{=}[dd] \ar@{}[ddr]|{\SEcell\,\alpha_1} &
 \bullet \ar[l]_-{w} \ar[r]^-{k} \ar@{}[ddr]|{\NEcell\,\alpha_2} &
 \bullet \ar@{=}[dd] \\
&& \\
\bullet & \bullet \ar@{<-}[uu]_->>>>>>{a} \ar[l]^-{v} \ar[r]_-{j} &
 \bullet
}
}}
\]
There are similar diagrams for~$g$.
\end{Rem}

\begin{Rem}
\label{Rem:vertical-comp-Spanhat}%
We discussed in \Cref{Rem:Span-comp-up-to-iso} the effect of replacing iso-commas by arbitrary Mackey squares in the \emph{horizontal} composition in~$\Span(\GG;\JJ)$. Essentially the same remark holds for~$\Spanhat(\GG;\JJ)$. However there is now another direction where the same concern can be expressed. Indeed, the \emph{vertical} composition, which was a mere juxtaposition in~$\Span(\GG;\JJ)$ now involves taking pull-backs (\Cref{Prop:pullbacks}). Since the 2-cells of $\Spanhat$ are isomorphism classes of vertical spans, we can replace the middle object up to equivalence without changing a 2-cell. This flexibility allows us to compose vertically by taking any Mackey square.
\end{Rem}

Let us investigate a little further the notion of isomorphism involved in the 2-cells of~$\Spanhat$.
\begin{Prop}
\label{Prop:2-cells-of-Spanhat}%
Consider a 2-cell in~$\Spanhat$ between $i_!u^*$ and $j_!v^*$ as in~\eqref{eq:2-cell-of-Spanhat}
\[
\vcenter{
\xymatrix@L=1ex{
i_!u^* \\
k_!w^* \ar@{=>}[u]_-{[b,\beta_1,\beta_2]}
 \ar@{=>}[d]^-{[a,\alpha_1,\alpha_2]} \\
j_!v^*
}}
\quad\quad
\vcenter{
\xymatrix@C=4em{
G \ar@{=}[d] \ar@{}[dr]|{\NEcell\,\beta_1}
& P \ar[l]_-{u} \ar[r]^-{i} \ar@{}[dr]|{\SEcell\,\beta_2}
& H \ar@{=}[d]
\\
G \ar@{}[dr]|{\SEcell\,\alpha_1} \ar@{=}[d]
& R \ar[u]_-{b} \ar[d]^-{a} \ar[l]_-{w} \ar[r]^-{k} \ar@{}[dr]|{\NEcell\,\alpha_2}
& H \ar@{=}[d]
\\
G
& Q \ar[l]^-{v} \ar[r]_-{j}
& H.\!
}}
\]
Then we have the following basic isomorphisms of 2-cells:
\begin{enumerate}[\rm(1)]
\item
\label{it:iso-2-cells-1}%
\emph{Whiskering by an equivalence}: For every equivalence $g\colon \bar R\to R$ of the middle 0-cell in~$\GG$, the above 2-cell is equal to the 2-cell
\[
\vcenter{
\xymatrix@C=4em{
G \ar@{=}[d] \ar@{}[dr]|{\NEcell\,\beta_1\,g}
& P \ar[l]_-{u} \ar[r]^-{i} \ar@{}[dr]|{\ \SEcell\,\beta_2\,g}
& H \ar@{=}[d]
\\
G \ar@{}[dr]|{\SEcell\,\alpha_1\,g} \ar@{=}[d]
& \bar R \ar[u]_-{b g} \ar[d]^-{a g} \ar[l]_-{w g} \ar[r]^-{k g} \ar@{}[dr]|{\ \NEcell\,\alpha_2\,g}
& H \ar@{=}[d]
\\
G
& Q \ar[l]^-{v} \ar[r]_-{j}
& H.\!
}}
\]
\smallbreak
\item
\label{it:iso-2-cells-2}%
\emph{Isomorphism on branches}: For any invertible 2-cell $\sigma\colon a\isoEcell \bar a\colon R\to Q$, we can replace only the `$a$-branch' and the neighboring 2-cells of our representative to obtain another representative of the same 2-cell in~$\Spanhat$
\[
\vcenter {
\xymatrix@C=4em{
G \ar@{=}[d] \ar@{}[dr]|{\NEcell\,\beta_1}
& P \ar[l]_-{u} \ar[r]^-{i} \ar@{}[dr]|{\SEcell\,\beta_2}
& H \ar@{=}[d]
\\
G \ar@{}[dr]|{\SEcell\,\bar\alpha_1} \ar@{=}[d]
& R \ar[u]_-{b} \ar[d]^-{\bar a} \ar[l]_-{w} \ar[r]^-{k} \ar@{}[dr]|{\NEcell\,\bar\alpha_2}
& H \ar@{=}[d]
\\
G
& Q \ar[l]^-{v} \ar[r]_-{j}
& H
}}
\]
where $\bar\alpha_1=(v\sigma)\alpha_1$ and $\bar\alpha_2=\alpha_2(j\sigma)\inv$ are simply~$\alpha_1$ and $\alpha_2$ suitably transported by the isomorphism~$\sigma$. Similarly of course, if $\tau\colon b\isoEcell \bar b$, or $\psi_1\colon w\isoEcell \tilde w$, or $\psi_2\colon k\isoEcell \tilde k$, are invertible 2-cells in~$\GG$, then our 2-cell in $\Spanhat$ is equal to the class of any of the following representatives:
\[
\quad
\vcenter{
\xymatrix@C=2em@R=2em{
G \ar@{=}[d] \ar@{}[dr]|{\ \NEcell\bar\beta_1}
& P \ar[l]_-{u} \ar[r]^-{i} \ar@{}[dr]|{\ \SEcell\bar\beta_2}
& H \ar@{=}[d]
\\
G \ar@{}[dr]|{\ \SEcell\alpha_1} \ar@{=}[d]
& R \ar[u]_-{\bar b} \ar[d]^-{a} \ar[l]_-{w} \ar[r]^-{k} \ar@{}[dr]|{\ \NEcell\alpha_2}
& H \ar@{=}[d]
\\
G
& Q \ar[l]^-{v} \ar[r]_-{j}
& H
}}
\qquad\quad
\vcenter{
\xymatrix@C=2em@R=2em{
G \ar@{=}[d] \ar@{}[dr]|{\ \NEcell\tilde\beta_1}
& P \ar[l]_-{u} \ar[r]^-{i} \ar@{}[dr]|{\ \SEcell\beta_2}
& H \ar@{=}[d]
\\
G \ar@{}[dr]|{\ \SEcell\tilde\alpha_1} \ar@{=}[d]
& R \ar[u]_-{b} \ar[d]^-{a} \ar[l]_-{\tilde w} \ar[r]^-{k} \ar@{}[dr]|{\ \NEcell\alpha_2}
& H \ar@{=}[d]
\\
G
& Q \ar[l]^-{v} \ar[r]_-{j}
& H
}}
\qquad\quad
\vcenter{
\xymatrix@C=2em@R=2em{
G \ar@{=}[d] \ar@{}[dr]|{\ \NEcell\beta_1}
& P \ar[l]_-{u} \ar[r]^-{i} \ar@{}[dr]|{\ \SEcell\tilde\beta_2}
& H \ar@{=}[d]
\\
G \ar@{}[dr]|{\ \SEcell\alpha_1} \ar@{=}[d]
& R \ar[u]_-{b} \ar[d]^-{a} \ar[l]_-{w} \ar[r]^-{\tilde k} \ar@{}[dr]|{\ \NEcell\tilde\alpha_2}
& H \ar@{=}[d]
\\
G
& Q \ar[l]^-{v} \ar[r]_-{j}
& H
}}
\]
where $\bar \beta_1=(u\tau)\beta_1$ and $\bar\beta_2=\beta_2(i\tau)\inv$, where $\tilde \alpha_1=\alpha_1\psi_1\inv$ and $\tilde \beta_1=\beta_1\psi_1\inv$, and where $\tilde \alpha_2=\psi_2\alpha_2$ and $\tilde \beta_2=\psi_2\beta_2$.
\end{enumerate}
Conversely, every diagram representing the same 2-cell can be reached by a series of operations~\eqref{it:iso-2-cells-1} and~\eqref{it:iso-2-cells-2}.
\end{Prop}

\begin{proof}
It is clear that~\eqref{it:iso-2-cells-1} and~\eqref{it:iso-2-cells-2} define isomorphisms of spans of spans. Conversely, the general isomorphism of spans of spans expanded in \Cref{Rem:2-cells-in-Spanhat} is composed of~\eqref{it:iso-2-cells-1} and~\eqref{it:iso-2-cells-2}, as the reader can easily verify.
\end{proof}

\begin{Rem} \label{Rem:second-embedding}
There is an obvious (1-contravariant) pseudo-functor
\[
\xymatrix@C=4em{
\GG^{\op} \ \ar@{^(->}[r]^-{(-)^{*}}
& \Span(\GG;\JJ) \ \ar@{^(->}[r]^-{(-)_{\star}}
& \Spanhat(\GG;\JJ)
}
\]
composed of the embedding $(-)^*\colon \GG^{\op}\hook \Span(\GG;\JJ)$ of \Cref{Cons:first-embeddings} followed by the canonical embedding still denoted $(-)_\star\colon \Span\hook \Spanhat$. The latter consists of the identity on 0-cells (and 1-cells) and the standard functor $(-)_\star$ on Hom categories, as in \Cref{Def:ordinary-spans}. (We make it more explicit below.)

In view of \Cref{sec:Span-co}, there is \apriori\ another embedding of~$\GG$ into~$\Spanhat$ which composes the two 2-contravariant ones
\[
\xymatrix@C=4em{
\GG^{\op} \ \ar@{^(->}[r]^-{(-)^{\costar}}
& \Span(\GG;\JJ)^{\co} \ \ar@{^(->}[r]^-{(-)^{\star}}
& \Spanhat(\GG;\JJ)\,.
}
\]
Here the embedding $\alpha\mapsto [\Id,\alpha\inv,\id]$ of \Cref{Thm:UP-Span-co} is followed by the canonical embedding of $(-)^\star\colon \Span^{\co}\hook \Spanhat$ which is the identity on 0-cells (and 1-cells) and the contravariant functor $(-)^\star$ on Hom categories, as in \Cref{Def:ordinary-spans}.

Luckily the above two embeddings agree, \ie the following commutes:
\begin{equation}
\label{eq:two-embeddings}%
\vcenter{
\xymatrix@C=2em{
& \Span(\GG;\JJ) \ar[d]^-{(-)_\star}
\\
\GG^{\op} \ar@/^3ex/[ru]^-{(-)^*\ } \ar@/_3ex/[rd]_-{(-)^{\costar}}
& \Spanhat(\GG;\JJ)
\\
& \Span(\GG;\JJ)^{\co} \ar[u]_-{(-)^\star} &&
}}
\end{equation}
The reason is that, for every $u,v\colon H\to G$ and every $\alpha\colon u\Rightarrow v$ in~$\GG$, the classes of the following two (vertical) spans are equal in the category $\Spanhat(\GG;\JJ)(G,H)$, \ie the following spans are isomorphic
\begin{equation}
\label{Eq:two-alpha^*}%
(\alpha^*)_\star = \quad
\vcenter{\xymatrix{
\ar@{=}[d] \ar@{}[rd]|-{\NEcell\,\id}
& \ar[l]_-{u} \ar@{=}[d] \ar@{=}[r]
& \ar@{=}[d] \ar@{}[ld]|-{\SEcell\,\id}
\\
& \ar[l]_-{u} \ar@{=}[d] \ar@{=}[r]
&
\\
\ar@{=}[u] \ar@{}[ru]|-{\SEcell\,\alpha}
& \ar[l]_-{v} \ar@{=}[r]
& \ar@{}[lu]|-{\NEcell\,\id} \ar@{=}[u]
}}
\quad \cong \quad
\vcenter{\xymatrix{
\ar@{=}[d] \ar@{}[rd]|-{\NEcell\alpha\inv\,}
& \ar[l]_-{u} \ar@{=}[d] \ar@{=}[r]
& \ar@{=}[d] \ar@{}[ld]|-{\SEcell\,\id}
\\
& \ar[l]_-{v} \ar@{=}[d] \ar@{=}[r]
&
\\
\ar@{=}[u] \ar@{}[ru]|-{\SEcell\,\id}
& \ar[l]_-{v} \ar@{=}[r]
& \ar@{}[lu]|-{\NEcell\,\id} \ar@{=}[u]
}}
\quad = (\alpha^{\costar})^\star
\end{equation}
by \Cref{Prop:2-cells-of-Spanhat}\,\eqref{it:iso-2-cells-2}.
\end{Rem}

We have therefore reached the 2-motivic construction we were aiming for:
\begin{Not}
\label{Not:motive}%
We denote the 1-contravariant embedding~\eqref{eq:two-embeddings} by
\[
(-)^*\colon \GG^{\op}\hook \Spanhat(\GG;\JJ)\,.
\]
It is the identity on 0-cells, maps a 1-cell $u\colon H\to G$ to $u^*=[\loto{u}=]\colon G\to H$ and maps $\alpha\colon u\Rightarrow v$ to the 2-cell $\alpha^*\colon u^*\Rightarrow v^*$ described in~\eqref{Eq:two-alpha^*} above.
\end{Not}

\begin{Rem}
\label{Rem:second-embedding-II}%
Similarly on the 2-category~$\JJ$, the two ways to (2-contravariantly) embed $\JJ$ into $\Spanhat(\GG;\JJ)$ coincide. The resulting embedding
\[
(-)_!\colon\JJ^{\co}\hook \Spanhat(\GG;\JJ)
\]
is the identity on 0-cells, as usual. It maps a 1-cell $i\colon H\into G$ to $i_!=[=\oto{i}]\colon H\to G$ and a 2-cell $\beta\colon i\Rightarrow j$ to the 2-cell $\beta_!\colon j_!\Rightarrow i_!$ represented by any of the following isomorphic spans (by \Cref{Prop:2-cells-of-Spanhat}\,\eqref{it:iso-2-cells-2} again):
\[
\beta_! = \quad
\vcenter{\xymatrix{
\ar@{=}[d] \ar@{}[rd]|-{\NEcell\,\id}
& \ar@{=}[l]_-{} \ar@{=}[d] \ar[r]^-{j}
& \ar@{=}[d] \ar@{}[ld]|-{\SEcell\,\id}
\\
& \ar@{=}[l]_-{} \ar@{=}[d] \ar[r]^-{j}
&
\\
\ar@{=}[u] \ar@{}[ru]|-{\SEcell\,\id}
& \ar@{=}[l] \ar[r]^-{i}
& \ar@{}[lu]|-{\NEcell\,\beta} \ar@{=}[u]
}}
\quad \cong \quad
\vcenter{\xymatrix{
\ar@{=}[d] \ar@{}[rd]|-{\NEcell\,\id}
& \ar@{=}[l]_-{} \ar@{=}[d] \ar[r]^-{j}
& \ar@{=}[d] \ar@{}[ld]|-{\,\SEcell\beta\inv}
\\
& \ar@{=}[l]_-{} \ar@{=}[d] \ar[r]^-{i}
&
\\
\ar@{=}[u] \ar@{}[ru]|-{\SEcell\,\id}
& \ar@{=}[l] \ar[r]^-{i}
& \ar@{}[lu]|-{\NEcell\,\id} \ar@{=}[u]
}}
\]
\end{Rem}

Let us prove the first properties of the bicategory of Mackey 2-motives. The alert reader will surely recognize in~\eqref{eq:gluing-formula} the \emph{strict Mackey formula}
of the Rectification \Cref{Thm:rectification}\,\eqref{it:Mack-7}.

\begin{Prop}\label{Prop:Bhat-ambidextre}
In the bicategory $\Spanhat=\Spanhat(\GG;\JJ)$, the image $i^*$ of every 1-cell $i\in\JJ$ has a two-sided adjoint $i_!=i_*$. They satisfy the following properties:
\begin{enumerate}[\rm(a)]
\item
\label{it:ambidextrous}%
\index{transposition!-- in $\Spanhat$}%
Transposition of 2-cells (\ie ordinary transposition applied to the Hom categories $\cat C=\Span(G,H)$, as in \Cref{Rem:ordinary-spans}) defines an involution
\[
\Spanhat^{\co}\isoto \Spanhat
\]
which fixes 0-cells and 1-cells and maps the unit (resp.\ the counit) of the adjunction $i_!\adj i^*$, as described in Proposition~\ref{Prop:first-adjoints}, into the counit (resp.\ the unit) of the new adjunction $i^*\adj i_*$, and vice-versa.
\item
\label{it:special-Frob}
For every $i\colon H\to G$ in~$\JJ$, the ambidextrous adjunction $i_! \adj i^* \adj i_*=i_!$ is \emph{special Frobenius}, meaning that the composite $\Id_{H} \overset{\eta\;}\Rightarrow i^*i_! = i^* i_* \overset{\eps\;} \Rightarrow \Id_{H}$ of the counit and unit for the two adjunctions is the identity.
\item
\label{it:gluing-formula}
For any comma square in~$\GG$
\begin{align*}
\xymatrix@C=14pt@R=14pt{
& (i/v) \ar[ld]_-{\tilde v} \ar[dr]^-{\tilde i} & \\
P \ar[dr]_i \ar@{}[rr]|{\oEcell{\gamma}} && Q \ar[dl]^v \\
&H &
}
\end{align*}
with $i\in \JJ$, the mate $(\gamma\inv)_*\colon v^* i_*\Rightarrow \tilde i_* \tilde v^*$ of $(\gamma\inv)^*\colon \tilde i^* v^*\Rightarrow \tilde v^* i^*$ with respect to the new adjunction $i^* \adj i_*$ is the inverse of the mate $\gamma_!$ of Lemma~\ref{Lem:BC-for-Span}:
\begin{equation} \label{eq:gluing-formula}
(\gamma\inv)_*=(\gamma_!)\inv\,.
\end{equation}
\end{enumerate}
\end{Prop}

\begin{proof}
Part~\eqref{it:ambidextrous} is clear since transposition on each Hom category
\[
\Spanhat^{\co}(G,H)=\Spanhat(G,H)^{\op}=\widehat{(\Span(G,H))}{}^{\op}\otoo{(-)^t} \widehat{\Span(G,H)}=\Spanhat(G,H)
\]
is a contravariant isomorphism.
Part~\eqref{it:special-Frob} is a direct computation based on the following fact: If $i\colon H\to G$ is a faithful 1-cell then the iso-comma $\Delta_i/\Delta_i$ for the 1-morphism $\Delta_i\colon H\to (i/i)$ (against itself) is equivalent to~$H$ itself as follows:
\begin{align*}
\xymatrix@C=14pt@R=14pt{
& H \ar[d]^\simeq \ar@/_4ex/@{=}[ddl] \ar@/^4ex/@{=}[ddr] &
 &&& H \ar@/_3ex/@{=}[ddl] \ar@/^3ex/@{=}[ddr]
 \ar@{}[ddd]|{\undersett{\id_{\Delta_i}}\Ecell}
& \\
& \Delta_i/\Delta_i \ar[ld]_-{\pr_1} \ar[dr]^-{\pr_2}
 \ar@{}[dd]|{\undersett{\gamma}\Ecell}
&
 & = &&& \\
H \ar[rd]_-{\Delta_i} && H \ar[ld]^-{\Delta_i}
 && H \ar[rd]_-{\Delta_i} && H \ar[ld]^-{\Delta_i} \\
&(i/i) & &&& (i/i) &
}
\end{align*}
A more detailed proof will be provided in Example~\ref{Exa:special_Frobenius_strings} using string diagrams.
To see why the comparison 1-cell $\langle \Id_H,\Id_H, \id_{\Delta_i}\rangle\colon H\to \Delta_i/\Delta_i$ is an equivalence, as claimed, we may apply $\GG(T,-)$ to easily compute in $\Cat$ and conclude with Corollary~\ref{Cor:2cat_Yoneda_equivs}.

Part~\eqref{it:gluing-formula} is direct from part~\eqref{it:ambidextrous} and Lemma~\ref{Lem:isos-in-spans}.
Indeed, if we apply transposition~$(-)^t$ to the isomorphism~$\gamma_!$ as in~\eqref{eq:gamma_!} we obtain precisely~$(\gamma\inv)_*$.
\end{proof}

The next theorem says that the canonical embedding $\GG^{\op}\hook\Spanhat(\GG;\JJ)$ of \Cref{Rem:second-embedding} is the universal pseudo-functor defined on $\GG^{\op}$ which sends the 1-cells of~$\JJ$ to ambidextrous adjunctions satisfying base change and the strict Mackey formula, as in \Cref{Prop:Bhat-ambidextre}. More precisely:

\begin{Thm}[Universal property of Mackey 2-motives]
\label{Thm:UP-Spanhat}%
Let $\GG$ be our chosen (2,1)-category with a distinguished class~$\JJ$ of faithful 1-cells, as in Hypotheses~\ref{Hyp:G_and_I_for_Span}.
Let $\cat{C}$ be any 2-category, and consider a pseudo-functor $\cat{F}\colon \GG^{\op}\to \cat{C}$ such that
\begin{enumerate}[\rm(a)]
\item
\label{it:UP-Spanhat-a}%
for every $i\in \JJ$, there exists in $\cat{C}$ a right-and-left adjoint to~$i^*=\cat{F}i$;
\item
\label{it:UP-Spanhat-b}%
both the left and the right adjunctions satisfy base-change with respect to all comma squares with two parallel sides in~$\JJ$; and
\item
\label{it:UP-Spanhat-c}%
for every such comma square, the left and right adjunctions satisfy the strict Mackey formula~\eqref{eq:gluing-formula}.
\end{enumerate}
Then there exists a pseudo-functor $\widehat{\cat{F}}\colon \Span(\GG; \JJ)\to \cat{C} $ such that
\[
\xymatrix{
 \GG^{\op} \ar[d]_-{(-)^*} \ar[r]^-{\cat{F}} & \cat{C} \\
 \Spanhat(\GG;\JJ) \ar@{-->}[ru]_-{\widehat{\cat{F}}} &
 }
\]
is commutative. This extension $\widehat{\cat{F}}$ is unique up to a unique isomorphism whose restriction to $\cat F$ is the identity, and is entirely determined by the choice of the adjoints of all~$i^*$, together with units and counits. Conversely, any pseudo-functor $\cat{F}$ factoring as above must enjoy the above three properties \eqref{it:UP-Spanhat-a}, \eqref{it:UP-Spanhat-b} and~\eqref{it:UP-Spanhat-c}.
\end{Thm}

For identities and invertible 1-cells, we make here the same choices of adjoints as in \Cref{Rem:rectif-UP-Span}. This makes the triangle in the theorem strictly commute, and overall simplifies proofs a little.

We begin with the following observation.

\begin{Lem} \label{Lem:strict-Mackey-pseudo-funs}
Denote by
\[
\cat F_! \colon \JJ^{\co}\to \cat C
\quad \textrm{and} \quad
\cat F_*\colon \JJ^{\co}\to \cat C
\]
the pseudo-functors induced by $\cat F$, as explained in \Cref{Rem:pseudo-func-of-adjoints}, by taking mates with respect to a choice of adjunctions $(\cat Fi)_!\dashv \cat Fi $ and $\cat Fi \dashv (\cat Fi)_*=(\cat Fi)_!$ (including units and counits) satisfying all the hypotheses of \Cref{Thm:UP-Spanhat}.
Then $\cat F_! = \cat F_*$ as pseudo-functors.
\end{Lem}

\begin{proof}
By definition we have an equality $\cat F_!=\cat F_*$ on objects, and clearly on 1-cells too by our choice of a left-and-right adjoint $(\cat Fi)_!=(\cat Fi)_*$ for each~$i\in \JJ$, as in hypothesis~\eqref{it:UP-Spanhat-a}. The images of any 2-cell also coincide, and this follows from hypothesis~\eqref{it:UP-Spanhat-c}. Indeed, every 2-cell $\alpha\colon i\Rightarrow j$ in $\JJ$ gives rise to the following Mackey squares in~$\GG$:
\[
\vcenter { \hbox{
\xymatrix@C=14pt@R=14pt{
& \ar[ld]_-{\Id} \ar[dr]^-{j} & \\
 \ar[dr]_i \ar@{}[rr]|{\oEcell{\alpha}} && \ar[dl]^\Id \\
& &
}
}}
{\quad=:\;\gamma}
\quad\quad\quad
\vcenter { \hbox{
\xymatrix@C=14pt@R=14pt{
& \ar[ld]_-{i} \ar[dr]^-{\Id} & \\
 \ar[dr]_-{\Id} \ar@{}[rr]|{\oEcell{\alpha}} && \ar[dl]^j \\
& &
}
}}
\quad\quad\quad
\vcenter { \hbox{
\xymatrix@C=14pt@R=14pt{
& \ar[ld]_-{j} \ar[dr]^-{\Id} & \\
 \ar[dr]_{\Id} \ar@{}[rr]|{\oEcell{\alpha^{-1}}} && \ar[dl]^i \\
& &
}
}}
\]
By taking mates of their images under $\cat F$ (\ie apply $\cat F$ and insert the units and counits of one of the four adjunctions $(\cat Fi)_!\dashv \cat Fi \dashv (\cat Fi)_*$ and $(\cat Fj)_!\dashv \cat Fj \dashv (\cat Fj)_*$, in the only way that makes sense), we obtain respectively:
\begin{enumerate}[{$\bullet$}]
\item The image $\cat F_!(\alpha)$ under~$\cat F_!$, which is also the 2-cell $\gamma_!$ in the right-hand side of~\eqref{eq:gluing-formula};
\item The image $\cat F_*(\alpha)$ under~$\cat F_*$;
\item The map $(\gamma^{-1})_*=(\cat{F}\alpha^{-1})_*$ in the left-hand side of~\eqref{eq:gluing-formula}.
\end{enumerate}
Then we compute
\[
\cat F_* \alpha
= \cat F_*((\alpha\inv)\inv)
= (\cat F_* \alpha\inv)\inv
= ((\cat F \alpha\inv)_*)\inv
\stackrel{\textrm{\eqref{eq:gluing-formula}}}{=} (\cat F \alpha)_!
=\cat F_!\alpha
\]
as claimed. For this we also use that $\cat F_!$ and $\cat F_*$ preserve inverses of 2-cells, like any pseudo-functor (this also follows from hypothesis~\eqref{it:UP-Spanhat-b}, which however is not needed).

It remains to see that the coherent structure isomorphisms $\fun$ and $\un$ of $\cat F_!$ and $\cat F_*$ also coincide. As they arise as mates of those of $\cat F$, this follows in a similar straightforward way from the strict Mackey formula.
\end{proof}

\begin{proof}[Proof of Theorem~\ref{Thm:UP-Spanhat}]
By \Cref{Thm:UP-Span} and \Cref{Thm:UP-Span-co} we have two extensions of $\cat{F}$, one to $\Span(\GG;\JJ)$ and one to~$\Span(\GG;\JJ)^{\co}$. To distinguish them and to match the universal property of spans in 1-categories, we denote them by $\cat{F}_\star$ and $\cat{F}^\star$ respectively. These pseudo-functors make the left-hand diagram below commute:
\[
\vcenter{\xymatrix@C=2em{
& \Span(\GG;\JJ) \ar@/^3ex/[rd]^-{\cat{F}_\star}
\\
\GG^\op \ar[rr]^-{\cat{F}} \ar@/^3ex/[ru]^-{(-)^*} \ar@/_3ex/[rd]_-{(-)^{\costar}}
&& \cat{C}
\\
& \Span(\GG;\JJ)^{\co} \ar@/_3ex/[ru]_-{\cat{F}^\star}
}}
\quad
\leadsto
\qquad
\vcenter{
\xymatrix@C=3em{
& \Span(\GG;\JJ) \ar[d]^-{(-)_\star} \ar@/^3ex/[rd]^-{\cat{F}_\star}
\\
\GG^\op \ar@/^3ex/[ru]^(.4){(-)^*} \ar@/_3ex/[rd]_-{(-)^{\costar}}
 \ar@{}[r]|-{\eqref{eq:two-embeddings}}
& \Spanhat(\GG;\JJ) \ar@{-->}[r]^-{\widehat{\cat{F}}}
& \cat{C}
\\
& \Span(\GG;\JJ)^{\co} \ar[u]_-{(-)^\star} \ar@/_3ex/[ru]_-{\cat{F}^\star}
}}
\]
The theorem claims the existence of a pseudo-functor $\widehat{\cat{F}}$ as on the right-hand side above, whose composite with $\GG^{\op}\hook\Spanhat$ of~\eqref{eq:two-embeddings} is the original~$\cat{F}$. Since $\Spanhat$ is constructed by applying on each Hom category the span construction (for 1-categories) of \Cref{sec:ordinary-spans}, we need to show that the two functors given on Hom categories by $\cat{F}_\star$ and $\cat{F}^\star$ assemble to a functor on $\Spanhat$. In other words, for every pair of 0-cells $G$ and $H$, we consider the functors $\cat{F}_\star$ and $\cat{F}^\star$ as follows
\begin{equation} \label{eq:def-Fhat}
\vcenter{
\xymatrix{
\Span(\GG;\JJ) (G,H) \ar@/^3ex/[rrd]^-{\cat{F}_\star} \ar[d]_-{(-)_\star} & & \\
 \Spanhat(\GG;\JJ) (G,H) \ar@{-->}[rr]^-{\widehat{\cat{F}}} && \cat{C}(\cat{F} G, \cat{F} H) \\
 \Span(\GG;\JJ) (G,H)^{\op} \ar[u]^-{(-)^{\star}} \ar@/_3ex/[rru]_-{\cat{F}^\star} &&
}}
\end{equation}
and want to check they assemble into~$\widehat{\cat{F}}$. We need to check the hypotheses of \Cref{Prop:UP-ordinary-spans}, by recalling the definitions of $\cat{F}_\star$ and $\cat{F}^\star$ on Hom categories, which are the 1-cell and 2-cell parts of \Cref{Cons:UP-Span} and \Cref{Cons:UP-Span-co}, respectively. On objects of $\Span(\GG;\JJ)(G,H)$, \ie on 1-cells $i_!u^*=i_*u^*$ of~$\Spanhat$, both functors $\cat{F}_\star$ and $\cat{F}^\star$ agree since
\[
\cat{F}_\star(i_!u^*)=(\cat{F}i)_!\cat{F}u=(\cat{F}i)_*\cat{F}u=\cat{F}^\star(i_*u^*)\,.
\]
This gives the easy condition~\eqref{it:UP-ordinary-spans-a} of \Cref{Prop:UP-ordinary-spans}. Let us check~\eqref{it:UP-ordinary-spans-b}, which says that for any pullback square in the category $\Span(\GG;\JJ)(G,H)$
\[
\xymatrix@C=2em@R=2em{
& \bullet \ar@{=>}[ld]_-{[\tilde b]} \ar@{=>}[rd]^-{[\tilde a]}
\\
\bullet \ar@{=>}[rd]_-{[a]}
&& \bullet \ar@{=>}[ld]^-{[b]}
\\
& \bullet
}
\]
the following equality holds in the category~$\cat{C}(\cat{F}G,\cat{F}H)$:
\begin{equation}
\label{eq:aux-the-equality}%
\cat{F}^\star([b])\circ \cat{F}_\star([a])
\;\overset{\textrm{?}}{=}\;
\cat{F}_\star([\tilde{a}])\circ \cat{F}^\star([\tilde{b}])\,.
\end{equation}
Here we temporarily used the short form $[a]$ for $[a,\alpha_1,\alpha_2]$, etc, as we did in \Cref{sec:pullback_2cells}, where we explicitly constructed the above pullbacks in the category $\Span(\GG;\JJ)(G,H)$. Expanding the notation as in \Cref{Prop:pullbacks}, if we start with the cospan in the category~$\Span(\GG;\JJ)(G,H)$ as in~\eqref{eq:pullback-Bicat-data}
\[
\xymatrix@L=1ex{
i_!u^* \ar@{=>}[d]^-{[a,\alpha_1,\alpha_2]}_{[a]\,:=\,} \\
k_!w^* \\
j_!v^* \ar@{=>}[u]_-{[b,\beta_1,\beta_2]}^{[b]\,:=\,}
}
\quad\quad\quad
\xymatrix{
G \ar@{=}[d] \ar@{}[rrd]|{\SEcell\,\alpha_1} &&
 P \ar[d]^a \ar[ll]_-{u} \ar[rr]^-{i} \ar@{}[rrd]|{\NEcell\,\alpha_2} &&
 H \ar@{=}[d] \\
G \ar@{}[rrd]|{\NEcell\,\beta_1} \ar@{=}[d] &&
 R \ar[ll]_-{w} \ar[rr]^-{k} \ar@{}[rrd]|{\SEcell\,\beta_2} &&
 H \ar@{=}[d] \\
G && Q \ar[u]_b \ar[ll]^-{v} \ar[rr]_-{j} &&
 H
}
\]
the above pullback square is the following left-hand square
\[
\vcenter{
\xymatrix@L=1ex{
& (ka\tilde b)_!(wa \tilde b)^* \ar@{=>}[ld]_-{[\tilde b,\tilde \beta_1,\tilde \beta_2]\quad} \ar@{=>}[rd]^-{\quad[\tilde a,\tilde \alpha_1,\tilde \alpha_2]}
\\
i_!u^* \ar@{=>}[rd]_-{[a,\alpha_1,\alpha_2]\quad}
&& j_!v^* \ar@{=>}[ld]^-{\quad[b,\beta_1,\beta_2]}
\\
& k_!w^*
}}
\kern5em
\vcenter{\xymatrix@C=14pt@R=14pt{
& (a/b) \ar[ld]_-{\tilde b} \ar[dr]^-{\tilde a}
\\
P \ar[dr]_a \ar@{}[rr]|{\isocell{\gamma}}
&& Q \ar[dl]^b
\\
& R
}}
\]
where the 1-cells $\tilde{a}$ and $\tilde{b}$ of~$\GG$ come from the above right-hand iso-comma and the 2-cell components $\tilde{\alpha}_1,\tilde{\alpha}_2,\tilde{\beta}_1,\tilde{\beta}_2$ are given in~\eqref{eq:pullback_in_Bicat}. Armed with those explicit formulas for these 2-cells, we can compute the images of~$[a,...]$ and $[\tilde{a},...]$ under $\cat{F}_\star\colon \Span(\GG;\JJ)(G,H)\to \cat{C}(\cat{F}G,\cat{F}H)$ and those of $[b,...]$ and $[\tilde{b},...]$ under $\cat{F}^\star\colon \Span(\GG;\JJ)(G,H)\to \cat{C}(\cat{F}G,\cat{F}H)$, by following the recipes given in~\eqref{eq:pasting-in-UP-Span} and~\eqref{eq:pasting-in-UP-Span-co}, respectively.

For legibility, we will write with string diagrams (see \Cref{sec:string_diagrams} if necessary). As in \Cref{sec:UP-Span}, we depict the chosen left-and-right adjoints $(\cat Fi)_!$ in~$\cat C$ ($i\in \JJ$) with dotted lines. The required equality~\eqref{eq:aux-the-equality} now becomes as follows (ignore the shaded area on the left for now):
\[
\vcenter{\hbox{
\xymatrix@L=1ex{
{\phantom{m}}
\ar@{=>}[ddd]_{\cat F_\star([a])} \\
\\
\\
\ar@{=>}[ddd]_{\cat F^\star([b])} \\
\\
\\
{\phantom{m}}
}
}}
\quad
\vcenter {\hbox{
\psfrag{A}[Bc][Bc]{\scalebox{1}{\scriptsize{$\cat Fw$}}}
\psfrag{B}[Bc][Bc]{\scalebox{1}{\scriptsize{$(\cat Fk)_!$}}}
\psfrag{C}[Bc][Bc]{\scalebox{1}{\scriptsize{$\cat Fv$}}}
\psfrag{D}[Bc][Bc]{\scalebox{1}{\scriptsize{$(\cat Fj)_*$}}}
\psfrag{U1}[Bc][Bc]{\scalebox{1}{\scriptsize{$\cat Fa$}}}
\psfrag{U2}[Bc][Bc]{\scalebox{1}{\scriptsize{$\;\;(\cat Fa)_!$}}}
\psfrag{V1}[Bc][Bc]{\scalebox{1}{\scriptsize{$\cat F b$}}}
\psfrag{V2}[Bc][Bc]{\scalebox{1}{\scriptsize{$\;\;(\cat F b)_*$}}}
\psfrag{X}[Bc][Bc]{\scalebox{1}{\scriptsize{$\cat Fw$}}}
\psfrag{Y1}[Bc][Bc]{\scalebox{1}{\scriptsize{\;\;\;$(\cat Fk)_!$}}}
\psfrag{Y2}[Bc][Bc]{\scalebox{1}{\scriptsize{\;\;\;$(\cat Fk)_*$}}}
\psfrag{A1}[Bc][Bc]{\scalebox{1}{\scriptsize{$\cat F\alpha_1$}}}
\psfrag{A2}[Bc][Bc]{\scalebox{1}{\scriptsize{$(\cat F\alpha_2)_!$}}}
\psfrag{B1}[Bc][Bc]{\scalebox{1}{\scriptsize{$\cat F\beta_1^{-1}$}}}
\psfrag{B2}[Bc][Bc]{\scalebox{1}{\scriptsize{$(\cat F\beta_2^{-1})_*$}}}
\psfrag{F1}[Bc][Bc]{\scalebox{1}{\scriptsize{$\eta$}}}
\psfrag{F2}[Bc][Bc]{\scalebox{1}{\scriptsize{$\varepsilon$}}}
\includegraphics[scale=.4]{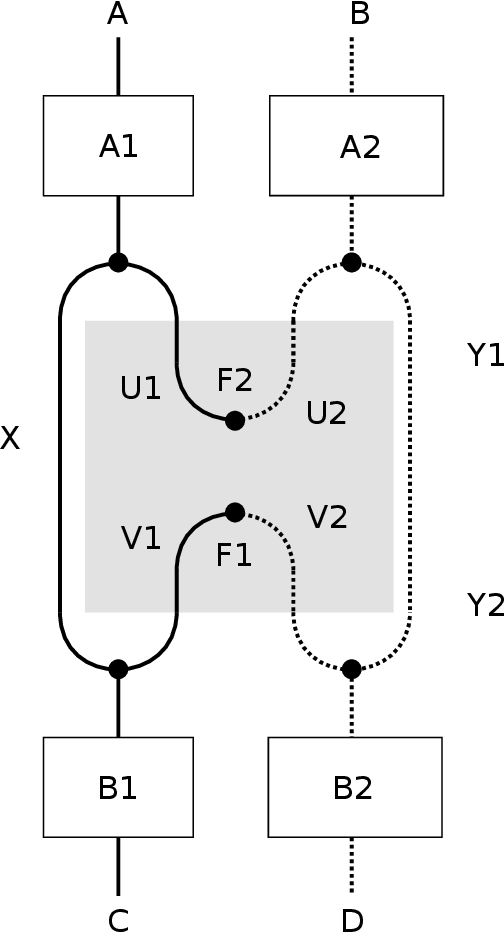}
}}
\quad\quad \stackrel{\textrm{?}}{=} \quad\quad
\vcenter {\hbox{
\psfrag{A}[Bc][Bc]{\scalebox{1}{\scriptsize{$\cat Fw$}}}
\psfrag{B}[Bc][Bc]{\scalebox{1}{\scriptsize{$(\cat Fk)_*$}}}
\psfrag{C}[Bc][Bc]{\scalebox{1}{\scriptsize{$\cat Fv$}}}
\psfrag{D}[Bc][Bc]{\scalebox{1}{\scriptsize{$(\cat Fj)_!$}}}
\psfrag{U1}[Bc][Bc]{\scalebox{1}{\scriptsize{$\cat F\tilde a$}}}
\psfrag{U2}[Bc][Bc]{\scalebox{1}{\scriptsize{$\;\;(\cat F\tilde a)_!$}}}
\psfrag{V1}[Bc][Bc]{\scalebox{1}{\scriptsize{$\cat F\tilde b$}}}
\psfrag{V2}[Bc][Bc]{\scalebox{1}{\scriptsize{$\;\;(\cat F\tilde b)_*$}}}
\psfrag{L}[Bc][Bc]{\scalebox{1}{\scriptsize{$\cat Fwa\tilde b$\;\;\;}}}
\psfrag{M}[Bc][Bc]{\scalebox{1}{\scriptsize{$(\cat Fka\tilde b)_* = (\cat Fka\tilde b)_!$\;\;\;\;\;\;\;\;}}}
\psfrag{X}[Bc][Bc]{\scalebox{1}{\scriptsize{$\cat Fw$}}}
\psfrag{Y}[Bc][Bc]{\scalebox{1}{\scriptsize{\;\;\;$(\cat Fk)_!$}}}
\psfrag{A1}[Bc][Bc]{\scalebox{1}{\scriptsize{$\cat F\alpha_1$}}}
\psfrag{A2}[Bc][Bc]{\scalebox{1}{\scriptsize{$(\cat F\alpha_2)_*$}}}
\psfrag{B1}[Bc][Bc]{\scalebox{1}{\scriptsize{$\cat F\beta_1^{-1}$}}}
\psfrag{B2}[Bc][Bc]{\scalebox{1}{\scriptsize{$(\cat F\beta_2^{-1})_!$}}}
\psfrag{F1}[Bc][Bc]{\scalebox{1}{\scriptsize{$\eta$}}}
\psfrag{F2}[Bc][Bc]{\scalebox{1}{\scriptsize{$\varepsilon$}}}
\psfrag{G1}[Bc][Bc]{\scalebox{1}{\scriptsize{$\cat F\gamma$}}}
\psfrag{G2}[Bc][Bc]{\scalebox{1}{\scriptsize{$(\cat F\gamma^{\!-1}\!)_!$}}}
\includegraphics[scale=.4]{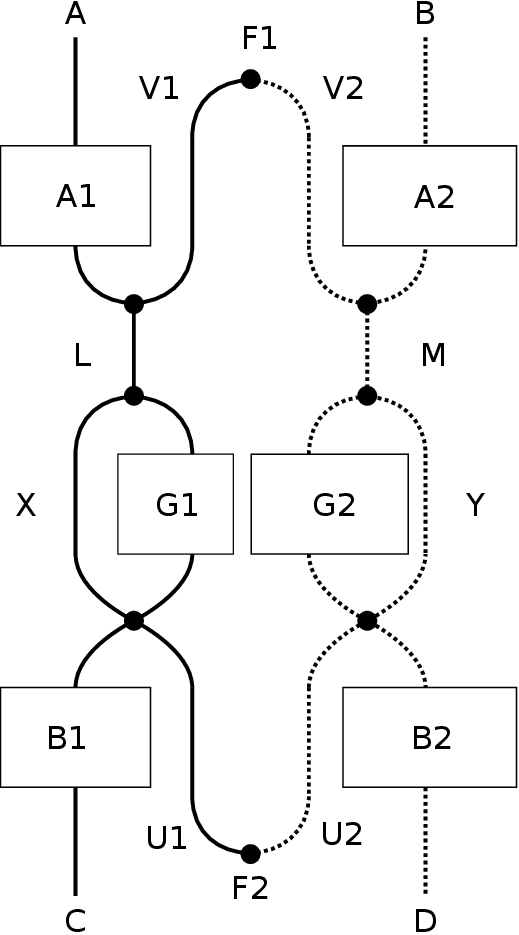}
}}
\quad
\vcenter{\hbox{
\xymatrix@L=1ex{
{\phantom{m}}
\ar@{=>}[dd]^{\cat F^\star([\tilde b])} \\
\\
\ar@{=>}[dddd]^{\cat F_\star([\tilde a])} \\
\\
\\
\\
{\phantom{m}}
}
}}
\]
Note the use, on the right-hand sides of each diagram, of the pseudo-functoriality of $\cat F_!$ and $\cat F_*$. We know from \Cref{Lem:strict-Mackey-pseudo-funs} that they actually agree.
All the unnamed 2-cells above denote some instance of the structure isomorphisms $\fun$ of $\cat F$ or $\cat F_!=\cat F_*$, or their composites and inverses (see \Cref{Exa:strings-for-fun}).

By the associativity of $\fun$ and the interchange law (\Cref{Exa:strings-for-exchange}), we can rewrite the right-hand side as follows:
\[
\vcenter {\hbox{
\psfrag{A}[Bc][Bc]{\scalebox{1}{\scriptsize{$\cat Fw$}}}
\psfrag{B}[Bc][Bc]{\scalebox{1}{\scriptsize{$(\cat Fk)_*$}}}
\psfrag{C}[Bc][Bc]{\scalebox{1}{\scriptsize{$\cat Fv$}}}
\psfrag{D}[Bc][Bc]{\scalebox{1}{\scriptsize{$(\cat Fj)_!$}}}
\psfrag{U1}[Bc][Bc]{\scalebox{1}{\scriptsize{$\cat F\tilde a$}}}
\psfrag{U2}[Bc][Bc]{\scalebox{1}{\scriptsize{$\;\;(\cat F\tilde a)_!$}}}
\psfrag{V1}[Bc][Bc]{\scalebox{1}{\scriptsize{$\cat F\tilde b$}}}
\psfrag{V2}[Bc][Bc]{\scalebox{1}{\scriptsize{$\;\;(\cat F\tilde b)_*$}}}
\psfrag{L}[Bc][Bc]{\scalebox{1}{\scriptsize{$\cat Fa$}}}
\psfrag{M}[Bc][Bc]{\scalebox{1}{\scriptsize{\;$(\!\cat Fa\!)_*$}}}
\psfrag{L1}[Bc][Bc]{\scalebox{1}{\scriptsize{$\cat Fb$}}}
\psfrag{M1}[Bc][Bc]{\scalebox{1}{\scriptsize{\;$(\!\cat Fb\!)_!$}}}
\psfrag{X}[Bc][Bc]{\scalebox{1}{\scriptsize{$\cat Fw$}}}
\psfrag{Y}[Bc][Bc]{\scalebox{1}{\scriptsize{\;\;\;$(\cat Fk)_*$}}}
\psfrag{A1}[Bc][Bc]{\scalebox{1}{\scriptsize{$\cat F\alpha_1$}}}
\psfrag{A2}[Bc][Bc]{\scalebox{1}{\scriptsize{$(\cat F\alpha_2)_*$}}}
\psfrag{B1}[Bc][Bc]{\scalebox{1}{\scriptsize{$\cat F\beta_1^{-1}$}}}
\psfrag{B2}[Bc][Bc]{\scalebox{1}{\scriptsize{$(\cat F\beta_2^{-1})_!$}}}
\psfrag{F1}[Bc][Bc]{\scalebox{1}{\scriptsize{$\eta$}}}
\psfrag{F2}[Bc][Bc]{\scalebox{1}{\scriptsize{$\varepsilon$}}}
\psfrag{G1}[Bc][Bc]{\scalebox{1}{\scriptsize{$\cat F\gamma$}}}
\psfrag{G2}[Bc][Bc]{\scalebox{1}{\scriptsize{$(\cat F\gamma^{\!-1}\!)_!$}}}
\includegraphics[scale=.4]{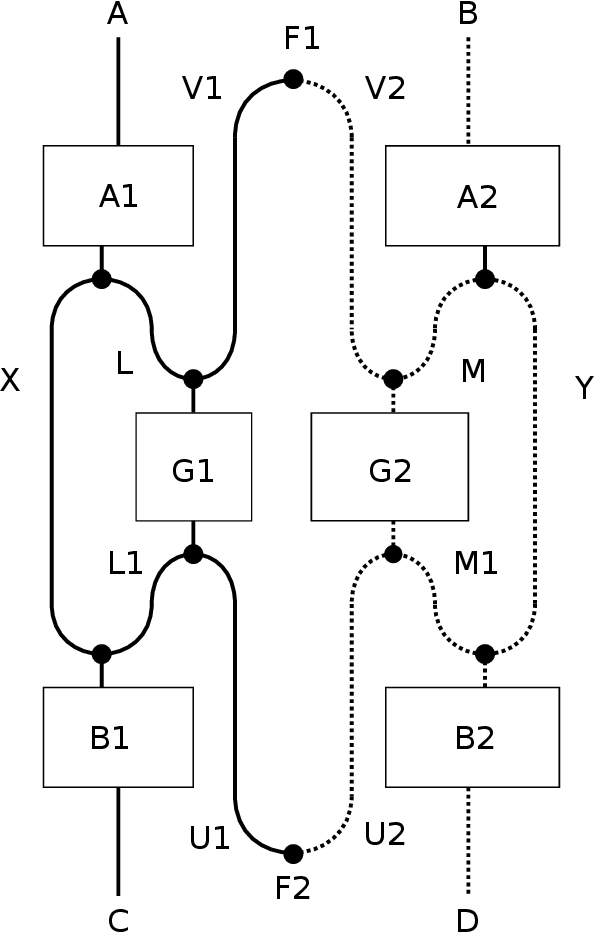}
}}
\quad =\quad
\vcenter {\hbox{
\psfrag{A}[Bc][Bc]{\scalebox{1}{\scriptsize{$\cat Fw$}}}
\psfrag{B}[Bc][Bc]{\scalebox{1}{\scriptsize{$(\cat Fk)_*$}}}
\psfrag{C}[Bc][Bc]{\scalebox{1}{\scriptsize{$\cat Fv$}}}
\psfrag{D}[Bc][Bc]{\scalebox{1}{\scriptsize{$(\cat Fj)_!$}}}
\psfrag{U1}[Bc][Bc]{\scalebox{1}{\scriptsize{$\cat F\tilde a$}}}
\psfrag{U2}[Bc][Bc]{\scalebox{1}{\scriptsize{$\;\;(\cat F\tilde a)_!$}}}
\psfrag{V1}[Bc][Bc]{\scalebox{1}{\scriptsize{$\cat F\tilde b$}}}
\psfrag{V2}[Bc][Bc]{\scalebox{1}{\scriptsize{$\;\;(\cat F\tilde b)_*$}}}
\psfrag{L}[Bc][Bc]{\scalebox{1}{\scriptsize{$\cat Fa$}}}
\psfrag{M}[Bc][Bc]{\scalebox{1}{\scriptsize{\;$(\!\cat Fa\!)_*$}}}
\psfrag{L1}[Bc][Bc]{\scalebox{1}{\scriptsize{$\cat Fb$}}}
\psfrag{M1}[Bc][Bc]{\scalebox{1}{\scriptsize{\;$(\!\cat Fb\!)_!$}}}
\psfrag{X}[Bc][Bc]{\scalebox{1}{\scriptsize{$\cat Fw$}}}
\psfrag{Y}[Bc][Bc]{\scalebox{1}{\scriptsize{\;\;\;$(\cat Fk)_*$}}}
\psfrag{A1}[Bc][Bc]{\scalebox{1}{\scriptsize{$\cat F\alpha_1$}}}
\psfrag{A2}[Bc][Bc]{\scalebox{1}{\scriptsize{$(\cat F\alpha_2)_*$}}}
\psfrag{B1}[Bc][Bc]{\scalebox{1}{\scriptsize{$\cat F\beta_1^{-1}$}}}
\psfrag{B2}[Bc][Bc]{\scalebox{1}{\scriptsize{$(\cat F\beta_2^{-1})_!$}}}
\psfrag{F1}[Bc][Bc]{\scalebox{1}{\scriptsize{$\eta$}}}
\psfrag{F2}[Bc][Bc]{\scalebox{1}{\scriptsize{$\varepsilon$}}}
\psfrag{G1}[Bc][Bc]{\scalebox{1}{\scriptsize{$\cat F\gamma$}}}
\psfrag{G2}[Bc][Bc]{\scalebox{1}{\scriptsize{$(\cat F\gamma^{\!-1}\!)_!$}}}
\includegraphics[scale=.4]{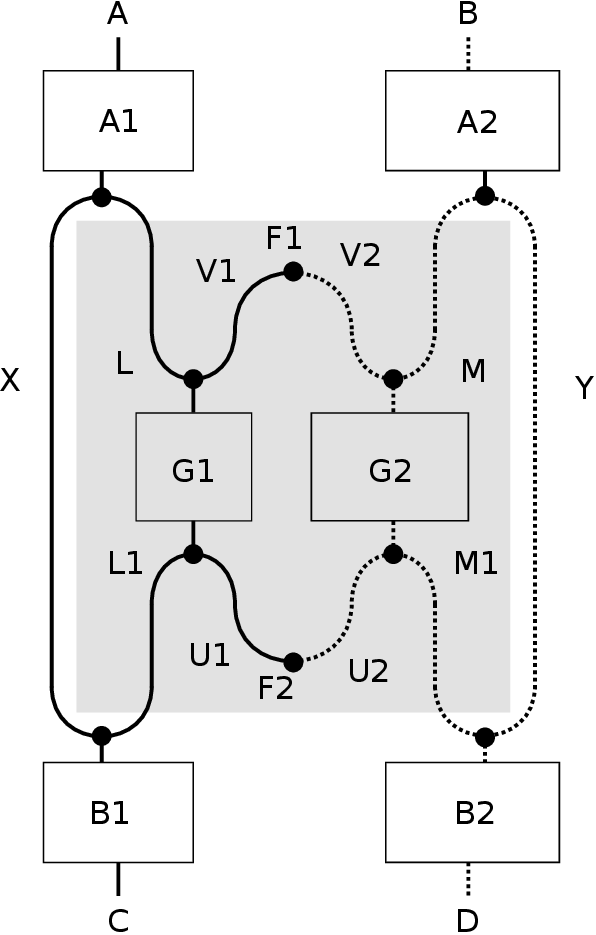}
}}
\]
Thus it suffices to show the equality of the two shaded areas above, since outside of them the two diagrams already agree.
Starting with the latter (right-hand side) shaded area, and making now all the mates explicit, we compute as follows using \Cref{Lem:strict-Mackey-pseudo-funs}, the interchange law, the triangular equalities of various adjunctions, and the strict Mackey formula:
\[
\vcenter {\hbox{
\psfrag{F}[Bc][Bc]{\scalebox{1}{\scriptsize{$\fun$}}}
\psfrag{G}[Bc][Bc]{\scalebox{1}{\scriptsize{$\fun^{\!-1}$\;\;}}}
\psfrag{R}[Bc][Bc]{\scalebox{1}{\scriptsize{$\cat Fa\tilde b$}}}
\psfrag{R1}[Bc][Bc]{\scalebox{1}{\scriptsize{$\cat Fb\tilde a$}}}
\psfrag{U1}[Bc][Bc]{\scalebox{1}{\scriptsize{$\cat F\tilde a$}}}
\psfrag{U2}[Bc][Bc]{\scalebox{1}{\scriptsize{$\;\;(\cat F\tilde a)_!$}}}
\psfrag{V1}[Bc][Bc]{\scalebox{1}{\scriptsize{$\cat F\tilde b$}}}
\psfrag{V2}[Bc][Bc]{\scalebox{1}{\scriptsize{$\;\;(\cat F\tilde b)_*$}}}
\psfrag{V3}[Bc][Bc]{\scalebox{1}{\scriptsize{$(\cat Fa\tilde b)_!$}}}
\psfrag{V4}[Bc][Bc]{\scalebox{1}{\scriptsize{$(\cat F b \tilde a)_!$}}}
\psfrag{L}[Bc][Bc]{\scalebox{1}{\scriptsize{$\cat Fa$}}}
\psfrag{M}[Bc][Bc]{\scalebox{1}{\scriptsize{\;$(\cat Fa)_*$}}}
\psfrag{L1}[Bc][Bc]{\scalebox{1}{\scriptsize{$\cat Fb$}}}
\psfrag{M1}[Bc][Bc]{\scalebox{1}{\scriptsize{\;$(\cat Fb)_!$}}}
\psfrag{F1}[Bc][Bc]{\scalebox{1}{\scriptsize{$\eta$}}}
\psfrag{F2}[Bc][Bc]{\scalebox{1}{\scriptsize{$\varepsilon$}}}
\psfrag{G1}[Bc][Bc]{\scalebox{1}{\scriptsize{$\cat F\gamma$}}}
\psfrag{G2}[Bc][Bc]{\scalebox{1}{\scriptsize{$\cat F\gamma^{-1}$}}}
\includegraphics[scale=.4]{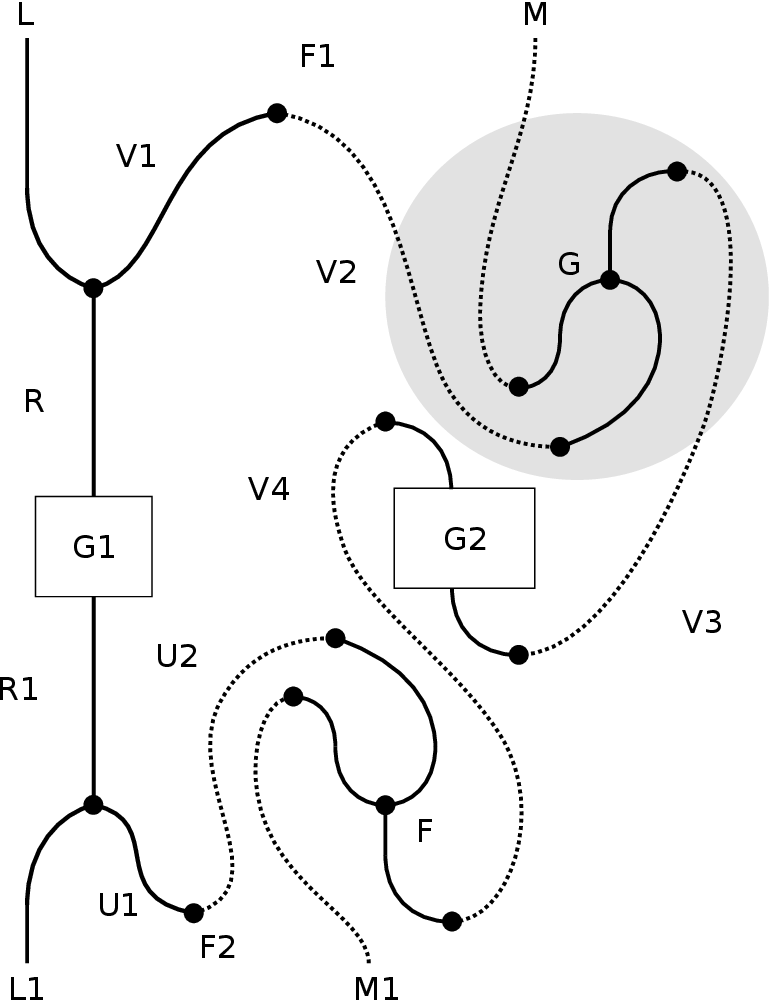}
}}
 \stackrel{\textrm{\eqref{Lem:strict-Mackey-pseudo-funs}}}{=}
\vcenter {\hbox{
\psfrag{F}[Bc][Bc]{\scalebox{1}{\scriptsize{$\fun$}}}
\psfrag{G}[Bc][Bc]{\scalebox{1}{\scriptsize{$\;\fun^{\!-1}$}}}
\psfrag{R}[Bc][Bc]{\scalebox{1}{\scriptsize{$\cat Fa\tilde b$}}}
\psfrag{R1}[Bc][Bc]{\scalebox{1}{\scriptsize{$\cat Fb\tilde a$}}}
\psfrag{U1}[Bc][Bc]{\scalebox{1}{\scriptsize{$\cat F\tilde a$}}}
\psfrag{U2}[Bc][Bc]{\scalebox{1}{\scriptsize{$\;\;(\cat F\tilde a)_!$}}}
\psfrag{V1}[Bc][Bc]{\scalebox{1}{\scriptsize{$\cat F\tilde b$}}}
\psfrag{V2}[Bc][Bc]{\scalebox{1}{\scriptsize{$(\cat F\tilde b)_*\!=\!(\cat F\tilde b)_!$}}}
\psfrag{V3}[Bc][Bc]{\scalebox{1}{\scriptsize{$(\cat Fa\tilde b)_!$}}}
\psfrag{V4}[Bc][Bc]{\scalebox{1}{\scriptsize{$(\cat F b \tilde a)_!$}}}
\psfrag{L}[Bc][Bc]{\scalebox{1}{\scriptsize{$\cat Fa$}}}
\psfrag{M}[Bc][Bc]{\scalebox{1}{\scriptsize{\;$(\cat Fa)_!$}}}
\psfrag{L1}[Bc][Bc]{\scalebox{1}{\scriptsize{$\cat Fb$}}}
\psfrag{M1}[Bc][Bc]{\scalebox{1}{\scriptsize{\;$(\cat Fb)_!$}}}
\psfrag{F1}[Bc][Bc]{\scalebox{1}{\scriptsize{$\eta$}}}
\psfrag{F2}[Bc][Bc]{\scalebox{1}{\scriptsize{$\varepsilon$}}}
\psfrag{G1}[Bc][Bc]{\scalebox{1}{\scriptsize{$\cat F\gamma$}}}
\psfrag{G2}[Bc][Bc]{\scalebox{1}{\scriptsize{$\cat F\gamma^{-1}$}}}
\includegraphics[scale=.4]{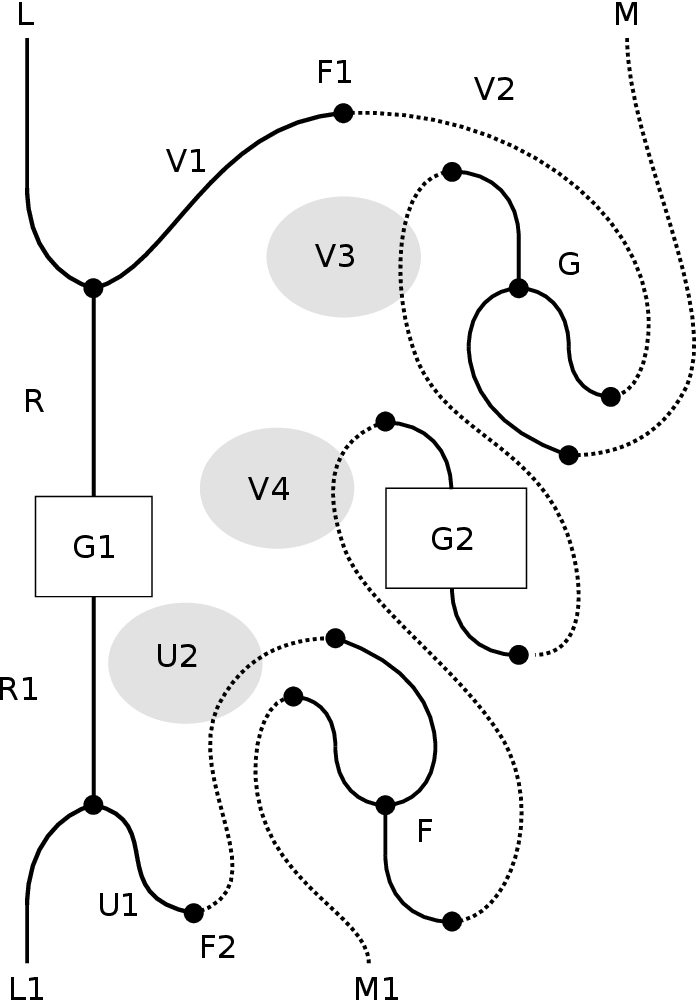}
}}
\stackrel{3 \times \textrm{\eqref{Exa:strings-for-adjoints}}}{=}
\]
\[
\vcenter {\hbox{
\psfrag{F}[Bc][Bc]{\scalebox{1}{\scriptsize{$\fun$}}}
\psfrag{G}[Bc][Bc]{\scalebox{1}{\scriptsize{$\;\fun^{\!-1}$}}}
\psfrag{R}[Bc][Bc]{\scalebox{1}{\scriptsize{$\cat Fa\tilde b$}}}
\psfrag{R1}[Bc][Bc]{\scalebox{1}{\scriptsize{$\cat Fb\tilde a$}}}
\psfrag{U1}[Bc][Bc]{\scalebox{1}{\scriptsize{$\cat F\tilde a$}}}
\psfrag{U2}[Bc][Bc]{\scalebox{1}{\scriptsize{$\;\;(\cat F\tilde a)_!$}}}
\psfrag{V1}[Bc][Bc]{\scalebox{1}{\scriptsize{$\cat F\tilde b$}}}
\psfrag{V2}[Bc][Bc]{\scalebox{1}{\scriptsize{$\;\;\;(\cat F\tilde b)_*$}}}
\psfrag{V3}[Bc][Bc]{\scalebox{1}{\scriptsize{$(\cat F\tilde b)_!$\;}}}
\psfrag{L}[Bc][Bc]{\scalebox{1}{\scriptsize{$\cat Fa$}}}
\psfrag{M}[Bc][Bc]{\scalebox{1}{\scriptsize{\;$(\cat Fa)_!$}}}
\psfrag{L1}[Bc][Bc]{\scalebox{1}{\scriptsize{$\cat Fb$}}}
\psfrag{M1}[Bc][Bc]{\scalebox{1}{\scriptsize{\;$(\cat Fb)_!$}}}
\psfrag{F1}[Bc][Bc]{\scalebox{1}{\scriptsize{$\eta$}}}
\psfrag{F2}[Bc][Bc]{\scalebox{1}{\scriptsize{$\varepsilon$}}}
\psfrag{G1}[Bc][Bc]{\scalebox{1}{\scriptsize{$\cat F\gamma$}}}
\psfrag{G2}[Bc][Bc]{\scalebox{1}{\scriptsize{$\cat F\gamma^{-1}$}}}
\includegraphics[scale=.4]{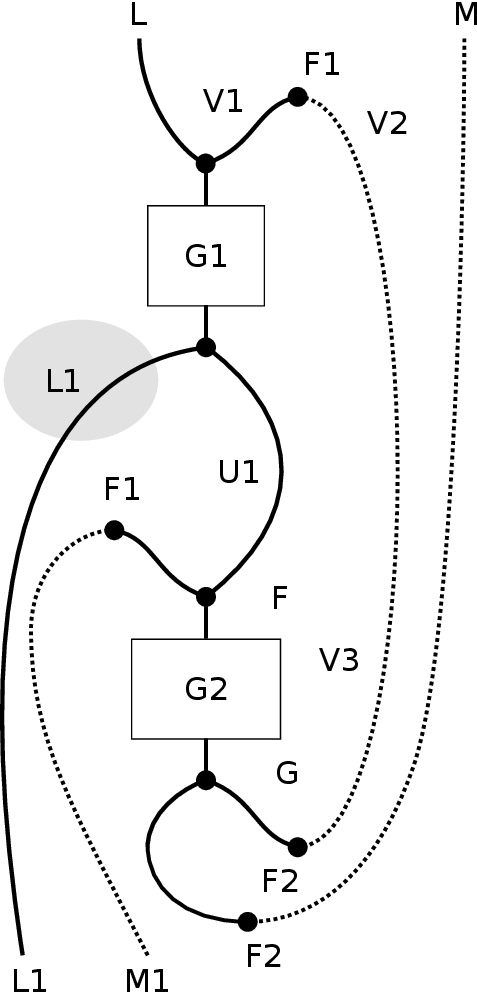}
}}
\stackrel{\textrm{\eqref{Exa:strings-for-adjoints}}}{=}
\vcenter {\hbox{
\psfrag{R}[Bc][Bc]{\scalebox{1}{\scriptsize{$\cat Fa\tilde b$}}}
\psfrag{R1}[Bc][Bc]{\scalebox{1}{\scriptsize{$\cat Fb\tilde a$}}}
\psfrag{U1}[Bc][Bc]{\scalebox{1}{\scriptsize{$\cat F\tilde a$}}}
\psfrag{U2}[Bc][Bc]{\scalebox{1}{\scriptsize{$\;\;(\cat F\tilde a)_!$}}}
\psfrag{V1}[Bc][Bc]{\scalebox{1}{\scriptsize{$\cat F\tilde b$}}}
\psfrag{V2}[Bc][Bc]{\scalebox{1}{\scriptsize{$\;\;(\cat F\tilde b)_*$}}}
\psfrag{V3}[Bc][Bc]{\scalebox{1}{\scriptsize{$(\cat F\tilde b)_!$}}}
\psfrag{L}[Bc][Bc]{\scalebox{1}{\scriptsize{$\cat Fa$}}}
\psfrag{M}[Bc][Bc]{\scalebox{1}{\scriptsize{\;$(\cat Fa)_!$}}}
\psfrag{L1}[Bc][Bc]{\scalebox{1}{\scriptsize{$\cat Fb$}}}
\psfrag{M1}[Bc][Bc]{\scalebox{1}{\scriptsize{\;$(\cat Fb)_!$}}}
\psfrag{M2}[Bc][Bc]{\scalebox{1}{\scriptsize{$(\cat Fb)_*$}}}
\psfrag{F1}[Bc][Bc]{\scalebox{1}{\scriptsize{$\eta$}}}
\psfrag{F2}[Bc][Bc]{\scalebox{1}{\scriptsize{$\varepsilon$}}}
\psfrag{G1}[Bc][Bc]{\scalebox{1}{\scriptsize{$\cat F\gamma$}}}
\psfrag{G2}[Bc][Bc]{\scalebox{1}{\scriptsize{$\cat F\gamma^{-1}$}}}
\includegraphics[scale=.4]{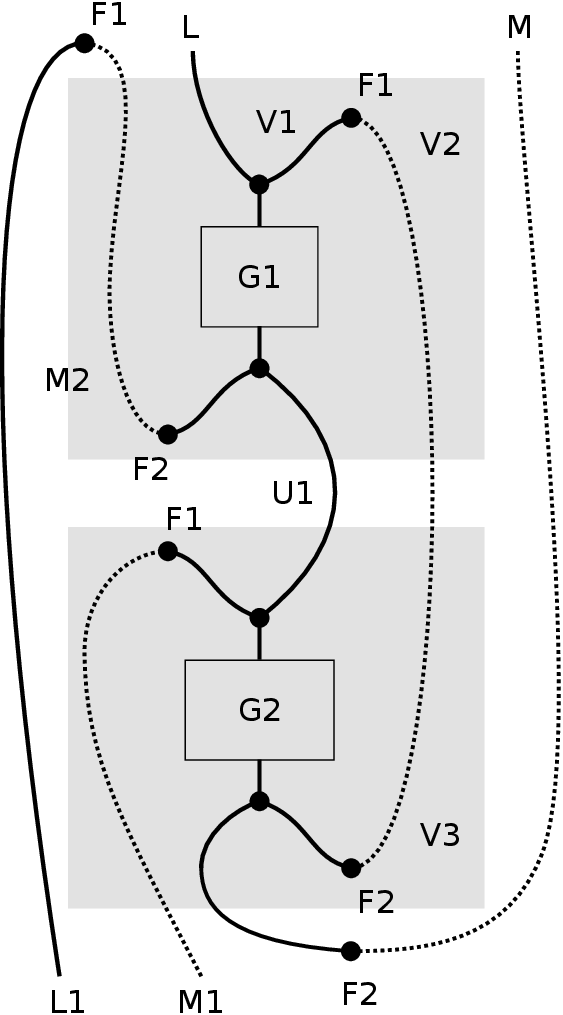}
}}
\stackrel{\textrm{\eqref{eq:gluing-formula}}}{=}
\vcenter {\hbox{
\psfrag{F}[Bc][Bc]{\scalebox{1}{\scriptsize{$\fun$}}}
\psfrag{G}[Bc][Bc]{\scalebox{1}{\scriptsize{$\;\fun^{\!-1}$}}}
\psfrag{R}[Bc][Bc]{\scalebox{1}{\scriptsize{$\cat Fa\tilde b$}}}
\psfrag{R1}[Bc][Bc]{\scalebox{1}{\scriptsize{$\cat Fb\tilde a$}}}
\psfrag{U1}[Bc][Bc]{\scalebox{1}{\scriptsize{$\cat F\tilde a$}}}
\psfrag{U2}[Bc][Bc]{\scalebox{1}{\scriptsize{$\;\;(\cat F\tilde a)_!$}}}
\psfrag{V1}[Bc][Bc]{\scalebox{1}{\scriptsize{$\cat F\tilde b$}}}
\psfrag{V2}[Bc][Bc]{\scalebox{1}{\scriptsize{$\;\;(\cat F\tilde b)_!$}}}
\psfrag{V3}[Bc][Bc]{\scalebox{1}{\scriptsize{$(\cat Fa\tilde b)_!$}}}
\psfrag{V4}[Bc][Bc]{\scalebox{1}{\scriptsize{$(\cat F b \tilde a)_!$}}}
\psfrag{L}[Bc][Bc]{\scalebox{1}{\scriptsize{$\cat Fa$}}}
\psfrag{M}[Bc][Bc]{\scalebox{1}{\scriptsize{\;$(\cat Fa)_!$}}}
\psfrag{L1}[Bc][Bc]{\scalebox{1}{\scriptsize{$\cat Fb$}}}
\psfrag{M1}[Bc][Bc]{\scalebox{1}{\scriptsize{\;$(\cat Fb)_*$}}}
\psfrag{F1}[Bc][Bc]{\scalebox{1}{\scriptsize{$\eta$}}}
\psfrag{F2}[Bc][Bc]{\scalebox{1}{\scriptsize{$\varepsilon$}}}
\psfrag{G1}[Bc][Bc]{\scalebox{1}{\scriptsize{$\cat F\gamma$}}}
\psfrag{G2}[Bc][Bc]{\scalebox{1}{\scriptsize{$\cat F\gamma^{-1}$}}}
\includegraphics[scale=.4]{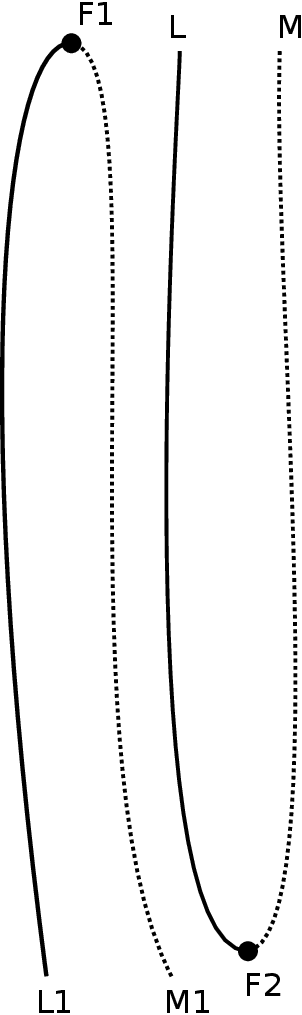}
}}
\]
For the last equality we use the identity $(\cat F\gamma^{\!-1})_!(\cat F\gamma)_*=\id_{(\cat Fa)(\cat Fb)_*}$, which is one half of the strict Mackey formula applied to the inverse Mackey square~$\gamma^{-1}$ (see \Cref{Exa:inv-comma-square}). By a routine application of the interchange law, the latter diagram is then nothing but
\[
\vcenter {\hbox{
\psfrag{U1}[Bc][Bc]{\scalebox{1}{\scriptsize{$\cat F a$}}}
\psfrag{U2}[Bc][Bc]{\scalebox{1}{\scriptsize{$(\cat Fa)_!$}}}
\psfrag{V1}[Bc][Bc]{\scalebox{1}{\scriptsize{$\cat F b$}}}
\psfrag{V2}[Bc][Bc]{\scalebox{1}{\scriptsize{$(\cat Fb)_*$}}}
\psfrag{F1}[Bc][Bc]{\scalebox{1}{\scriptsize{$\eta$}}}
\psfrag{F2}[Bc][Bc]{\scalebox{1}{\scriptsize{$\varepsilon$}}}
\includegraphics[scale=.4]{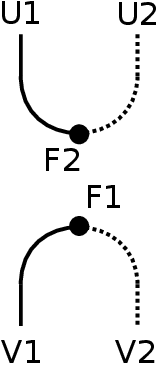}
}}
\]
as desired.

This concludes the proof of~\eqref{eq:aux-the-equality}, and thus that the functor
\[
\widehat{\cat{F}}\colon \Spanhat(\GG;\JJ)(G,H)\too \cat{C}(\cat{F}G,\cat{F}H)
\]
is well-defined. To show that these functors $\widehat{\cat{F}}$ on the Hom categories (for varying $G$ and~$H$) assemble into a pseudo-functor on $\Spanhat$ as required, we need only to find the coherent natural isomorphisms~$\un_{\widehat{\cat{F}}, G}$
\[
\xymatrix@R=10pt{
& \Spanhat(\GG;\JJ)(G,G) \ar[dd]^-{\widehat{\cat{F}}}\\
1 \ar@/^2ex/[ur] \ar@{}[r]|{\quad\NEcell} \ar@/_2ex/[dr] & \\
& \cat{C}(\cat{F} G , \cat{F} G)
}
\]
for each object $G$ and~$\fun_{\widehat{\cat{F}},G,H,K}$
\[
\xymatrix{
\Spanhat(\GG;\JJ) (H,K) \times \Spanhat(\GG;\JJ) (G,H)
 \ar[r]^-{\hat\circ}
 \ar[d]_-{\widehat{\cat{F}} \times \widehat{\cat{F}}}
 \ar@{}[dr]|{\;\;\NEcell} &
 \Spanhat(\GG;\JJ) (G,K)
 \ar[d]^-{\widehat{\cat{F}}} \\
\cat{C}(\cat{F} H, \cat{F} K) \times \cat{C}(\cat{F} G, \cat{F} H)
 \ar[r]^-{\circ} &
 \cat{C}(\cat{F} G, \cat{F} K)
}
\]
for each triple $G,H,K$. For this we take the same isomorphisms as for~$\cat{F}$, and we verify that they are natural also as transformations of functors extended to the span categories. Since they are \emph{invertible} natural transformations, this is automatic by \Cref{Lem:functoriality_ordinary-spans}\,\eqref{it:funct-ordinary-spans-b}.
\end{proof}

\bigbreak
\section{A strict presentation and a calculus of string diagrams}
\label{sec:string-presentation}%
\medskip

In view of the universal property of \Cref{Thm:UP-Spanhat}, it is now straightforward to give a presentation of 2-Mackey motives in terms of the given (2,1)-category~$\GG$, some extra 1-cells and 2-cells corresponding to the `wrong way' adjoint 1-cells associated to the distinguished 1-cells $i\in \JJ$, and the necessary relations. The result is a \emph{strict} 2-category $\Spanhat(\GG;\JJ)^\str$ biequivalent to the double-span $\Spanhat(\GG;\JJ)$. We also describe a string diagram version of this strict presentation, in \Cref{Cons:string_model} below, in order to infuse it with some helpful visual intuition.

Note that, in principle, a similar presentation could also be constructed for the intermediary bicategory $\Span(\GG;\JJ)$ since it enjoys a similar universal property. This is left to the interested reader, keeping \Cref{Rem:expl_inv_BC} in mind.

\begin{Cor}[The strict 2-category of Mackey 2-motives] \label{Cor:strict-Spanhat}
\index{strictification!-- of $\Spanhat(\GG;\JJ)$}%
The bicategory of Mack\-ey 2-motives $\Spanhat (\GG;\JJ)$ is canonically biequivalent to the strict 2-category $\Spanhat (\GG;\JJ)^\str$ generated by:
\begin{enumerate}[\rm(a)]
\item the same 0-cells as those of~$\GG$;
\item the two families of 1-cells
\[
u^*\colon G\to P
\quad \textrm{ and } \quad
i_*\colon P\to H
\]
where $(u\colon P\to G)\in \GG_1$ is an arbitrary 1-cell and $(i\colon P\to H)\in \JJ$ a distinguished faithful 1-cell;
\item the five families of 2-cells
\[
\alpha^*\colon u^*\Rightarrow v^*
\,,\quad
\leta\colon \Id \Rightarrow i^*i_*
\,,\quad
\leps\colon i_* i^* \Rightarrow \Id
\,,\quad
\reta\colon \Id \Rightarrow i_* i^*
\,,\quad
\reps\colon i^*i_* \Rightarrow \Id
\]
where $(\alpha\colon u\Rightarrow v)\in \GG_2$ is an arbitrary 2-cell and $(i\colon P\to H)\in \JJ$;
\end{enumerate}
subject to the following relations:
\begin{enumerate}[\rm(1)]
\item
the images under $(-)^*$ of all the relations in~$\GG$;
\item
the triangle equalities which turn the two families of quadruplets
\[
(i_*,i^*,\leta,\leps)
\quad \textrm{ and } \quad
(i^*,i_*,\reta,\reps)
\qquad (i\in \JJ)
\]
into adjunctions; and
\item
the strict Mackey formula, as in~\eqref{eq:gluing-formula}, namely:
\[
(\gamma_!)^{-1}
:= \big( (\leps\,u^*i_*)(j_* \gamma^* i_*)(j_*v^*\,\leta) \big)^{-1}
= (j_*v^*\,\reps)(j_*(\gamma^{-1})^* i_*)(\reta\,u^*i_*)
 =: (\gamma^{-1})_*
\]
for every Mackey square
\begin{equation}
\label{eq:Mackey-square-string-calculus}%
\vcenter{
\xymatrix@C=14pt@R=14pt{
& L \ar[ld]_-{v} \ar[dr]^-{j}
 \ar@{}[dd]|-{\isocell{\gamma}}
\\
H \ar[dr]_-{i}
&& K \ar[dl]^-{u} \\
&G &
}}
\end{equation}
in~$\GG$ with $i$, and thus~$j$, in~$\JJ$ (see \Cref{Def:Mackey-square});

\end{enumerate}
in addition, of course, to the necessary relations of a 2-category.
\end{Cor}

\begin{proof}
Note that, by the first family of relations, there is a well-defined 2-functor $(-)^*\colon \GG^\op\to \Spanhat(\GG;\JJ)^\str$ sending $u$ and $\alpha$ to the generators $u^*$ and~$\alpha^*$. By \Cref{Thm:UP-Spanhat}, the two remaining families of relations allow us to extend this 2-functor through the (homonymous) canonical pseudo-functor $(-)^*\colon \GG^\op\to \Spanhat(\GG;\JJ)$:
\[
\xymatrix{
 \GG^{\op} \ar[d]_{(-)^*} \ar[r]^-{\cat F\,:=\,(-)^*} & \Spanhat(\GG;\JJ)^{\str} \\
 \Spanhat(\GG;\JJ) \ar@{-->}[ru]_{\widehat{\cat{F}}} &
 }
\]
Indeed, we choose here the unique extension $\widehat{\cat{F}}$ which is determined by the adjunctions $(i_*,i^*,\leta,\leps)$ and $(i^*,i_*,\reta,\reps)$.
Let us define a pseudo-functor $\cat F'$ in the opposite direction: It send the generators $u^*,i_*,\alpha^*$ of $\Spanhat(\GG;\JJ)^\str$ to the homonymous 1- or 2-cells of $\Spanhat(\GG;\JJ)$, and it sends the remaining generating 2-cells $\leta,\leps,\reta,\reps$ to the corresponding canonical units and counits of adjunction as previously constructed in $\Spanhat(\GG;\JJ)$.

It is now a straightforward observation that $\cat F'$ can be endowed with the structure of a pseudo-functor, quasi-inverse to~$\widehat{\cat F}$: use \Cref{Prop:Bhat-ambidextre} and the evident candidates for $\fun_{\cat F'}$ and~$\un_{\cat F'}$. Alternatively, one can verify that the above 2-functor $(-)^*\colon \GG^\op\to \Spanhat(\GG;\JJ)^\str$ satisfies the \emph{same} universal property as $\Spanhat(\GG;\JJ)$, but limited to (strict) 2-functors to (strict) 2-categories~$\cat C$. (This `strict' verification is much easier!) As a consequence, $\widehat{\cat F}$ must be isomorphic to the strictification of $\Spanhat(\GG;\JJ)$, as in \Cref{Rem:strictification}, and in particular $\widehat{\cat F}$ is a biequivalence. We leave the remaining details to the reader.
\end{proof}

The above presentation of $\Spanhat (\GG;\JJ)^\str$ may look a little unwieldy. We will now rephrase it using string diagrams (see \Cref{sec:string_diagrams}), which will provide a visually suggestive string calculus for computing with Mackey 2-motives. This is strongly reminiscent of the Penrose and Reidemeister graphical calculus for (braided) pivotal monoidal categories; see \cite[\S\S\,2.2-2.3]{TuraevVirelizier17}. Contrary to \emph{loc.\ cit.\ }though, we will not attempt to rigorously prove a topological invariance theorem for our diagrams, because of the extra labelling and orientability issues which would make the result somewhat complicated.

\begin{Cons} \label{Cons:string_model}
 We can (re)define $\Spanhat (\GG;\JJ)^\str$ as the 2-category generated by the following basic string diagrams:
\begin{enumerate}[\rm(a)]
\item The plane regions (0-cells), which are labelled by the 0-cells of~$\GG$:
\[
\vcenter { \hbox{
\psfrag{G}[Bc][Bc]{\scalebox{1}{\scriptsize{$G$}}}
\includegraphics[scale=.4]{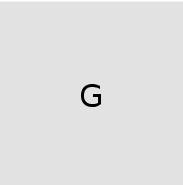}
}}
\]
\item Two families of basic \emph{oriented} strings, namely
\[
\vcenter { \hbox{
\psfrag{A}[Bc][Bc]{\scalebox{1}{\scriptsize{$u$}}}
\psfrag{G}[Bc][Bc]{\scalebox{1}{\scriptsize{$G$}}}
\psfrag{H}[Bc][Bc]{\scalebox{1}{\scriptsize{$P $}}}
\includegraphics[scale=.4]{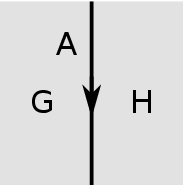}
}}
\quad \textrm{ and } \quad
\vcenter { \hbox{
\psfrag{A}[Bc][Bc]{\scalebox{1}{\scriptsize{$i$}}}
\psfrag{G}[Bc][Bc]{\scalebox{1}{\scriptsize{$P$}}}
\psfrag{H}[Bc][Bc]{\scalebox{1}{\scriptsize{$H$}}}
\includegraphics[scale=.4]{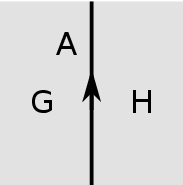}
}}
\]
for every 1-cell $u\colon P\to G$ of $\GG$ and every $(i\colon P\into H)$ in~$\JJ$. The orientation, upwards or downwards, is indicated by the arrow-head placed on the string. Similarly as before, such a string denotes a 1-cell in $\Spanhat (\GG;\JJ)^\str$ from the left-hand region to the right-hand one. The orientation allows us to have \emph{two} possible strings decorated with the label~$i$, for $i\colon P\into H$ in~$\JJ$: one string pointing downward corresponding to~$i^*$ (from $G$ to $P$) and one string pointing upward corresponding to~$i_!=i_*$ (from $P$ to~$H$).
\item A 2-cell (dot) as follows
\[
\vcenter { \hbox{
\psfrag{A}[Bc][Bc]{\scalebox{1}{\scriptsize{$u$}}}
\psfrag{B}[Bc][Bc]{\scalebox{1}{\scriptsize{$v$}}}
\psfrag{G}[Bc][Bc]{\scalebox{1}{\scriptsize{$G$}}}
\psfrag{H}[Bc][Bc]{\scalebox{1}{\scriptsize{$P$}}}
\psfrag{F}[Bc][Bc]{\scalebox{1}{\scriptsize{$\alpha$}}}
\includegraphics[scale=.4]{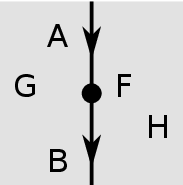}
}}
\]
for every 2-cell $\alpha\colon u\Rightarrow v\colon P\to G$ of~$\GG$, as well as four additional 2-cells
\[
\vcenter { \hbox{
\psfrag{A}[Bc][Bc]{\scalebox{1}{\scriptsize{$i$}}}
\psfrag{G}[Bc][Bc]{\scalebox{1}{\scriptsize{$P$}}}
\psfrag{H}[Bc][Bc]{\scalebox{1}{\scriptsize{$H$}}}
\includegraphics[scale=.4]{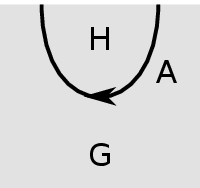}
}}
\quad\quad
\vcenter { \hbox{
\psfrag{A}[Bc][Bc]{\scalebox{1}{\scriptsize{$i$}}}
\psfrag{G}[Bc][Bc]{\scalebox{1}{\scriptsize{$H$}}}
\psfrag{H}[Bc][Bc]{\scalebox{1}{\scriptsize{$P$}}}
\includegraphics[scale=.4]{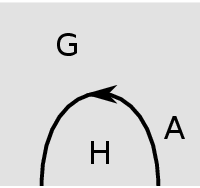}
}}
\quad\quad
\vcenter { \hbox{
\psfrag{A}[Bc][Bc]{\scalebox{1}{\scriptsize{$i$}}}
\psfrag{G}[Bc][Bc]{\scalebox{1}{\scriptsize{$H$}}}
\psfrag{H}[Bc][Bc]{\scalebox{1}{\scriptsize{$P$}}}
\includegraphics[scale=.4]{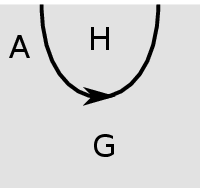}
}}
\quad\quad
\vcenter { \hbox{
\psfrag{A}[Bc][Bc]{\scalebox{1}{\scriptsize{$i$}}}
\psfrag{G}[Bc][Bc]{\scalebox{1}{\scriptsize{$P$}}}
\psfrag{H}[Bc][Bc]{\scalebox{1}{\scriptsize{$H$}}}
\includegraphics[scale=.4]{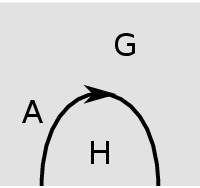}
}}
\]
for every 1-cell $i\colon P\to H$ in~$\JJ$. We leave the latter unlabeled and undotted because the orientation and shape of the string suffices for distinguishing them.
\end{enumerate}
These basic 2-cells can be combined vertically and horizontally according to the usual rules, taking care to also preserve the orientations on all strings. The resulting combined string diagrams are subject to the usual relations of a 2-category (the insertion and deletion of identities as in \Cref{Exa:identity-strings}, the interchange law as in \Cref{Exa:strings-for-exchange}, etc.), as well as the additional three families of relations:
\begin{enumerate}[\rm(1)]
\item All the relations coming from the 2-category~$\GG$.
\item The following `zig-zag' relations
\begin{equation} \label{eq:zigzag-relations}
\vcenter { \hbox{
\psfrag{A}[Bc][Bc]{\scalebox{1}{\scriptsize{$i$}}}
\psfrag{H}[Bc][Bc]{\scalebox{1}{\scriptsize{$P$}}}
\psfrag{G}[Bc][Bc]{\scalebox{1}{\scriptsize{$H$}}}
\includegraphics[scale=.4]{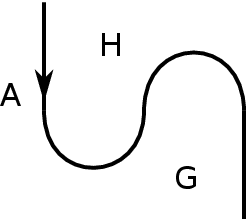}
}}
\quad =\quad
\vcenter { \hbox{
\psfrag{A}[Bc][Bc]{\scalebox{1}{\scriptsize{$i$}}}
\psfrag{H}[Bc][Bc]{\scalebox{1}{\scriptsize{$P$}}}
\psfrag{G}[Bc][Bc]{\scalebox{1}{\scriptsize{$H$}}}
\includegraphics[scale=.4]{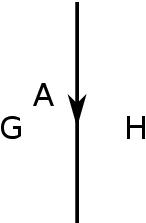}
}}
\quad = \quad
\vcenter { \hbox{
\psfrag{A}[Bc][Bc]{\scalebox{1}{\scriptsize{$i$}}}
\psfrag{H}[Bc][Bc]{\scalebox{1}{\scriptsize{$H$}}}
\psfrag{G}[Bc][Bc]{\scalebox{1}{\scriptsize{$P$}}}
\includegraphics[scale=.4]{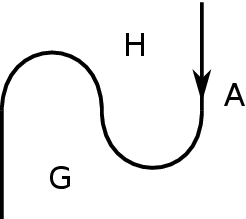}
}}
\end{equation}
and
\begin{equation} \label{eq:zigzag-relation2}
\vcenter { \hbox{
\psfrag{A}[Bc][Bc]{\scalebox{1}{\scriptsize{$i$}}}
\psfrag{H}[Bc][Bc]{\scalebox{1}{\scriptsize{$P$}}}
\psfrag{G}[Bc][Bc]{\scalebox{1}{\scriptsize{$H$}}}
\includegraphics[scale=.4]{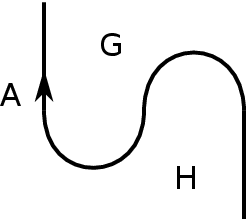}
}}
\quad =\quad
\vcenter { \hbox{
\psfrag{A}[Bc][Bc]{\scalebox{1}{\scriptsize{$i$}}}
\psfrag{H}[Bc][Bc]{\scalebox{1}{\scriptsize{$P$}}}
\psfrag{G}[Bc][Bc]{\scalebox{1}{\scriptsize{$H$}}}
\includegraphics[scale=.4]{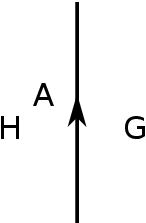}
}}
\quad = \quad
\vcenter { \hbox{
\psfrag{A}[Bc][Bc]{\scalebox{1}{\scriptsize{$i$}}}
\psfrag{H}[Bc][Bc]{\scalebox{1}{\scriptsize{$H$}}}
\psfrag{G}[Bc][Bc]{\scalebox{1}{\scriptsize{$P$}}}
\includegraphics[scale=.4]{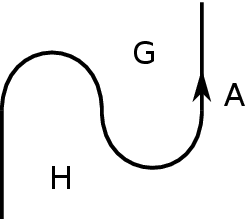}
}}
\end{equation}
for every $(i\colon P\to H)\in \JJ$.

\item And for every Mackey square \eqref{eq:Mackey-square-string-calculus}, the two `pull-over' relations
\begin{equation} \label{eq:string-gluing}
\vcenter { \hbox{
\psfrag{A}[Bc][Bc]{\scalebox{1}{\scriptsize{$i$}}}
\psfrag{B}[Bc][Bc]{\scalebox{1}{\scriptsize{$u$}}}
\psfrag{C}[Bc][Bc]{\scalebox{1}{\scriptsize{$v$}}}
\psfrag{D}[Bc][Bc]{\scalebox{1}{\scriptsize{$j$}}}
\psfrag{X}[Bc][Bc]{\scalebox{1}{\scriptsize{$H$}}}
\psfrag{Y}[Bc][Bc]{\scalebox{1}{\scriptsize{$L$}}}
\psfrag{Z}[Bc][Bc]{\scalebox{1}{\scriptsize{$K$}}}
\psfrag{W}[Bc][Bc]{\scalebox{1}{\scriptsize{$G$}}}
\includegraphics[scale=.4]{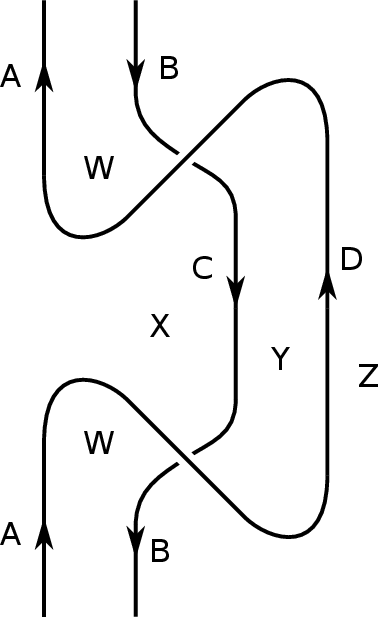}
}}
=
\vcenter { \hbox{
\psfrag{A}[Bc][Bc]{\scalebox{1}{\scriptsize{$i$}}}
\psfrag{B}[Bc][Bc]{\scalebox{1}{\scriptsize{$u$}}}
\psfrag{X}[Bc][Bc]{\scalebox{1}{\scriptsize{$H$}}}
\psfrag{Y}[Bc][Bc]{\scalebox{1}{\scriptsize{$G$}}}
\psfrag{Z}[Bc][Bc]{\scalebox{1}{\scriptsize{$K$}}}
\includegraphics[scale=.4]{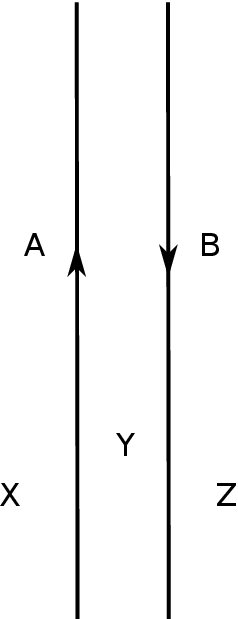}
}}
\textrm{ and } \;\;
\vcenter { \hbox{
\psfrag{A}[Bc][Bc]{\scalebox{1}{\scriptsize{$i$}}}
\psfrag{B}[Bc][Bc]{\scalebox{1}{\scriptsize{$u$}}}
\psfrag{C}[Bc][Bc]{\scalebox{1}{\scriptsize{$v$}}}
\psfrag{D}[Bc][Bc]{\scalebox{1}{\scriptsize{$j$}}}
\psfrag{X}[Bc][Bc]{\scalebox{1}{\scriptsize{$H$}}}
\psfrag{Y}[Bc][Bc]{\scalebox{1}{\scriptsize{$G$}}}
\psfrag{Z}[Bc][Bc]{\scalebox{1}{\scriptsize{$K$}}}
\psfrag{W}[Bc][Bc]{\scalebox{1}{\scriptsize{$L$}}}
\includegraphics[scale=.4]{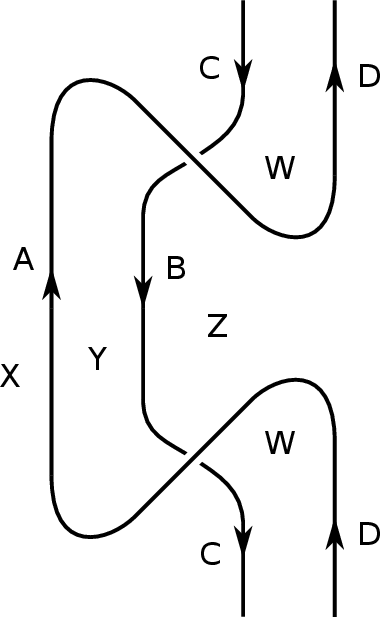}
}}
=
\vcenter { \hbox{
\psfrag{C}[Bc][Bc]{\scalebox{1}{\scriptsize{$v$}}}
\psfrag{D}[Bc][Bc]{\scalebox{1}{\scriptsize{$j$}}}
\psfrag{X}[Bc][Bc]{\scalebox{1}{\scriptsize{$H$}}}
\psfrag{W}[Bc][Bc]{\scalebox{1}{\scriptsize{$L$}}}
\psfrag{Z}[Bc][Bc]{\scalebox{1}{\scriptsize{$K$}}}
\includegraphics[scale=.4]{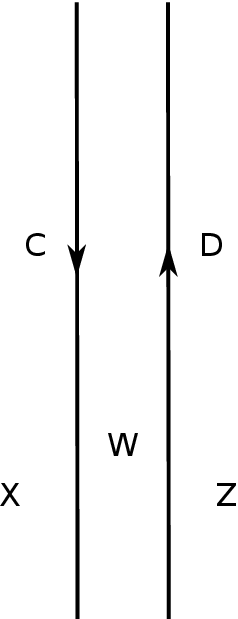}
}}
\end{equation}
where we have depicted the `exchange' 2-cell $\gamma$ of the Mackey square with a crossing~$\vcenter { \hbox{
\includegraphics[scale=.4]{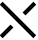}
}}$ and its inverse $\gamma^{-1}$ with the corresponding uncrossing~$\vcenter { \hbox{
\includegraphics[scale=.4]{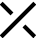}
}}$, rather than with labeled dots~$\bullet$. (Beware that each crossing still indicates a 2-cell, so that in general the string labels change when going past it---although the orientation of each strand is preserved.)
\end{enumerate}
\end{Cons}

\begin{Rem} \label{Rem:memorient}
The choice of string orientations in \Cref{Cons:string_model}\,(b) is meant to evoke the popular arrow notations $M\!\!\downarrow_H$ and $M\!\!\uparrow^G$  used in representation theory for the restriction and induction of modules ($H\leq G$).
\end{Rem}

\begin{Rem} \label{Rem:two-functorialities-agree-strings}
The argument in \Cref{Lem:strict-Mackey-pseudo-funs}, on the agreement of the pseudo-functors $\cat F_!$ and~$\cat F_*$, implies that the equation
\[
\alpha_*\; :=\quad
\vcenter { \hbox{
\psfrag{A}[Bc][Bc]{\scalebox{1}{\scriptsize{$j$}}}
\psfrag{B}[Bc][Bc]{\scalebox{1}{\scriptsize{$i$}}}
\psfrag{H}[Bc][Bc]{\scalebox{1}{\scriptsize{$P$}}}
\psfrag{G}[Bc][Bc]{\scalebox{1}{\scriptsize{$H$}}}
\psfrag{F}[Bc][Bc]{\scalebox{1}{\scriptsize{$\alpha$}}}
\includegraphics[scale=.4]{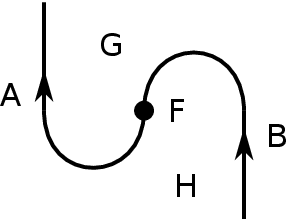}
}}
\quad =\quad
\vcenter { \hbox{
\psfrag{A}[Bc][Bc]{\scalebox{1}{\scriptsize{$j$}}}
\psfrag{B}[Bc][Bc]{\scalebox{1}{\scriptsize{$i$}}}
\psfrag{H}[Bc][Bc]{\scalebox{1}{\scriptsize{$H$}}}
\psfrag{G}[Bc][Bc]{\scalebox{1}{\scriptsize{$P$}}}
\psfrag{F}[Bc][Bc]{\scalebox{1}{\scriptsize{$\alpha$}}}
\includegraphics[scale=.4]{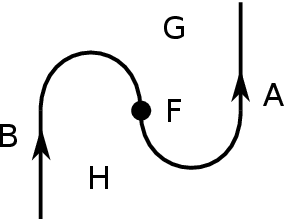}
}}
\quad =: \; \alpha_!
\]
holds in $\Spanhat(\GG;\JJ)^\str$ for every 2-cell $\alpha\colon i\Rightarrow j\colon P\into H$ between 1-cells $i,j$ in~$\JJ$. We invite the reader to (re)check this by themselves, by noticing how it immediately follows from the pull-over relations \eqref{eq:string-gluing} for either of the Mackey squares
\[
\vcenter { \hbox{
\xymatrix@C=14pt@R=14pt{
& \ar[ld]_-{\Id} \ar[dr]^-{j} & \\
 \ar[dr]_i \ar@{}[rr]|{\Ecell\, \alpha} && \ar[dl]^\Id \\
& &
}
}}
\quad\quad \textrm{ or } \quad\quad
\vcenter { \hbox{
\xymatrix@C=14pt@R=14pt{
& \ar[ld]_-{\Id} \ar[dr]^-{i} & \\
 \ar[dr]_j \ar@{}[rr]|{\Ecell\, \alpha^{-1}} && \ar[dl]^\Id \\
& &
}
}}
\]
together with the fact that, on such squares, taking mates $(-)_!$ and $(-)_*$ commute with taking inverses.
By viewing the above equation as a generalization of \eqref{eq:zigzag-relation2}, it makes sense to introduce the simple notation
\[
\vcenter { \hbox{
\psfrag{A}[Bc][Bc]{\scalebox{1}{\scriptsize{$j$}}}
\psfrag{B}[Bc][Bc]{\scalebox{1}{\scriptsize{$i$}}}
\psfrag{G}[Bc][Bc]{\scalebox{1}{\scriptsize{$P$}}}
\psfrag{H}[Bc][Bc]{\scalebox{1}{\scriptsize{$H$}}}
\psfrag{F}[Bc][Bc]{\scalebox{1}{\scriptsize{$\alpha$}}}
\includegraphics[scale=.4]{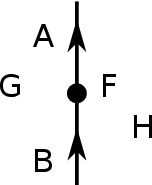}
}}
\]
for the 2-cell $\alpha_!=\alpha_*$, suggesting that we may twist an oriented string upside down (provided it lives in~$\JJ$) while preserving the labeled dots it carries.
\end{Rem}

We end this section with an example of a nontrivial computation by means of the above string calculus.

\begin{Exa}[Special Frobenius via strings]
\label{Exa:special_Frobenius_strings}
Let us provide a complete proof of the special Frobenius property of the ambidextrous adjunction $i_*\dashv i^* \dashv i_*$ in~$\Spanhat$ for any $i\colon H\into G$ in~$\JJ$, which we had only sketched in Proposition~\ref{Prop:Bhat-ambidextre}\,\eqref{it:special-Frob}.
Thus we claim that $\reps \, \leta = \id_{\Id_H}$, or equivalently, in string diagrams:
\[
\vcenter {\hbox{
\psfrag{A}[Bc][Bc]{\scalebox{1}{\scriptsize{$i$}}}
\psfrag{G}[Bc][Bc]{\scalebox{1}{\scriptsize{$G$}}}
\psfrag{H}[Bc][Bc]{\scalebox{1}{\scriptsize{$H$}}}
\includegraphics[scale=.4]{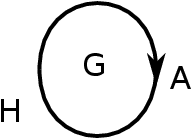}
}}
\quad = \quad
\vcenter {\hbox{
\psfrag{X}[Bc][Bc]{\scalebox{1}{\scriptsize{$H$}}}
\includegraphics[scale=.4]{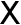}
}}
\quad .
\]
In order to prove this, we recall from \emph{loc.\,cit.\ }that in any (2,1)-category~$\GG$ the square
\begin{equation} \label{eq:Frob-isocomma}
\vcenter{\xymatrix@C=14pt@R=14pt{
& H \ar[ld]_-{\Id} \ar[rd]^-{\Id}
\\
H \ar[rd]_-{\Delta}
 \ar@{}[rr]|{\oEcell{\id_\Delta}}
&& H \ar[ld]^-{\Delta}
\\
& (i/i)
}}
\end{equation}
is a Mackey square whenever~$i$ is faithful, where $\Delta=\langle \Id,\Id,\id_{i}\rangle$ is the unique 1-cell $H\to (i/i)$ such that $\pr_1\Delta=\Id_H$, $\pr_2\Delta=\Id_H$
and $\gamma \Delta =\id_{i}$:
\begin{align} \label{eq:def-Delta}
\vcenter { \hbox{
\xymatrix@C=14pt@R=14pt{
& H \ar[d]^\Delta \ar@/_4ex/[ldd]_-{\Id} \ar@/^4ex/[rdd]^-{\Id} & \\
& (i/i) \ar[ld]_-{\pr_1} \ar[dr]^-{\pr_2} & \\
H \ar[dr]_i \ar@{}[rr]|{\oEcell{\gamma}} && H \ar[dl]^i \\
&G &
}
}}
\quad \quad = \quad \quad
\vcenter { \hbox{
\xymatrix@C=14pt@R=14pt{
& H \ar@/_3ex/[ldd]_-{\Id} \ar@/^3ex/[rdd]^-{\Id} & \\
&& \\
H \ar@{}[rru]|{\oEcell{\id_{i}}} \ar[dr]_i && H \ar[dl]^i \\
&G &
}
}}
\end{align}
The point here is not to redo the latter computation, which already happens in~$\GG$, but to show that string diagrams allow us to rigorously keep track of \emph{all} the information in the 2-cells of~$\Spanhat$, which can otherwise be quite overwhelming if written out in full using cell diagrams.
We begin by translating \eqref{eq:def-Delta} into oriented strings (later, the identity 1-cells $\Id_H$ will be mostly omitted or denoted by a dotted line):
\begin{equation} \label{eq:or-string-transl}
\vcenter{\hbox{
\psfrag{A}[Bc][Bc]{\scalebox{1}{\scriptsize{$i$}}}
\psfrag{B}[Bc][Bc]{\scalebox{1}{\scriptsize{$i$}}}
\psfrag{C}[Bc][Bc]{\scalebox{1}{\scriptsize{$\gamma$}}}
\psfrag{D}[Bc][Bc]{\scalebox{1}{\scriptsize{$\Delta$}}}
\psfrag{P}[Bc][Bc]{\scalebox{1}{\scriptsize{$\pr_1$}}}
\psfrag{Q}[Bc][Bc]{\scalebox{1}{\scriptsize{$\pr_2$}}}
\psfrag{X}[Bc][Bc]{\scalebox{1}{\scriptsize{$G$}}}
\psfrag{Z}[Bc][Bc]{\scalebox{1}{\scriptsize{$H$}}}
\psfrag{W}[Bc][Bc]{\scalebox{1}{\scriptsize{$i/i$}}}
\psfrag{id}[Bc][Bc]{\scalebox{1}{\scriptsize{$\id$}}}
\psfrag{Id}[Bc][Bc]{\scalebox{1}{\scriptsize{$\Id$}}}
\includegraphics[scale=.4]{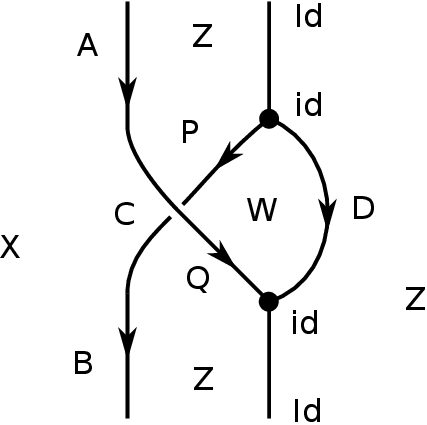}
}}
\quad\quad = \quad\quad
\vcenter{\hbox{
\psfrag{A}[Bc][Bc]{\scalebox{1}{\scriptsize{$i$}}}
\psfrag{X}[Bc][Bc]{\scalebox{1}{\scriptsize{$G$}}}
\psfrag{Z}[Bc][Bc]{\scalebox{1}{\scriptsize{$H$}}}
\includegraphics[scale=.4]{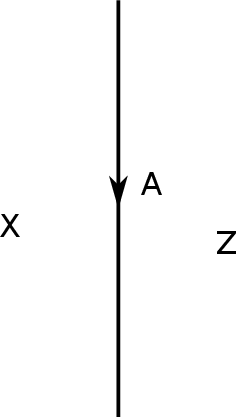}
}}
\end{equation}
Note that there is also a version of this identity with $\gamma^{-1}$ instead of~$\gamma$.
Now we compute as follows, where as usual we highlight at each step the areas where the action is about to happen:
\[
\vcenter {\hbox{
\psfrag{A}[Bc][Bc]{\scalebox{1}{\scriptsize{$i$}}}
\psfrag{G}[Bc][Bc]{\scalebox{1}{\scriptsize{$G$}}}
\psfrag{H}[Bc][Bc]{\scalebox{1}{\scriptsize{$H$}}}
\includegraphics[scale=.4]{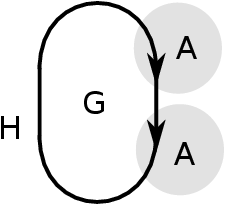}
}}
\;\stackrel{\textrm{\eqref{eq:or-string-transl}}}{=}\;
\vcenter{\hbox{
\psfrag{A}[Bc][Bc]{\scalebox{1}{\scriptsize{$i$}}}
\psfrag{B}[Bc][Bc]{\scalebox{1}{\scriptsize{$i$}}}
\psfrag{D}[Bc][Bc]{\scalebox{1}{\scriptsize{$\Delta$}}}
\psfrag{P}[Bc][Bc]{\scalebox{1}{\scriptsize{$\pr_1$}}}
\psfrag{Q}[Bc][Bc]{\scalebox{1}{\scriptsize{$\pr_2$}}}
\psfrag{X}[Bc][Bc]{\scalebox{1}{\scriptsize{$H$}}}
\psfrag{Y}[Bc][Bc]{\scalebox{1}{\scriptsize{$G$}}}
\psfrag{Z}[Bc][Bc]{\scalebox{1}{\scriptsize{$H$}}}
\psfrag{W}[Bc][Bc]{\scalebox{1}{\scriptsize{$i/i$}}}
\psfrag{id}[Bc][Bc]{\scalebox{1}{\scriptsize{$\id$}}}
\includegraphics[scale=.4]{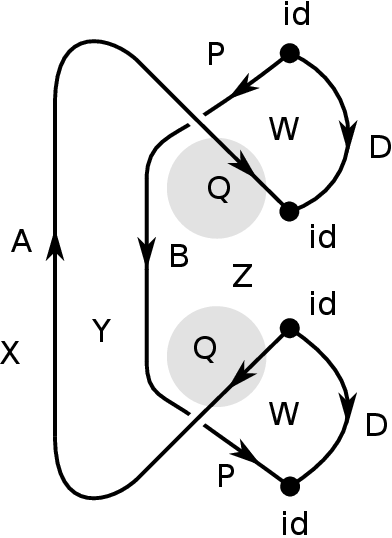}
}}
\;\stackrel{\textrm{\eqref{eq:zigzag-relations}}}{=}\;
\vcenter{\hbox{
\psfrag{A}[Bc][Bc]{\scalebox{1}{\scriptsize{$i$}}}
\psfrag{B}[Bc][Bc]{\scalebox{1}{\scriptsize{$i$}}}
\psfrag{D}[Bc][Bc]{\scalebox{1}{\scriptsize{$\Delta$}}}
\psfrag{P}[Bc][Bc]{\scalebox{1}{\scriptsize{$\pr_1$}}}
\psfrag{Q}[Bc][Bc]{\scalebox{1}{\scriptsize{$\pr_2$}}}
\psfrag{X}[Bc][Bc]{\scalebox{1}{\scriptsize{$H$}}}
\psfrag{Y}[Bc][Bc]{\scalebox{1}{\scriptsize{$G$}}}
\psfrag{Z}[Bc][Bc]{\scalebox{1}{\scriptsize{$H$}}}
\psfrag{W}[Bc][Bc]{\scalebox{1}{\scriptsize{$i/i$}}}
\psfrag{id}[Bc][Bc]{\scalebox{1}{\scriptsize{$\id$}}}
\includegraphics[scale=.4]{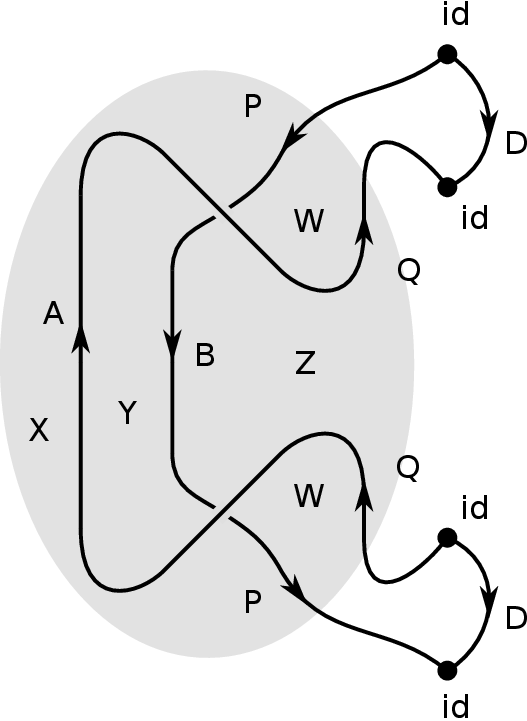}
}}
\underset{\textrm{for }\eqref{eq:def-Delta}}{\overset{\textrm{\eqref{eq:string-gluing}}}{=}}
\]
\[
\vcenter{\hbox{
\psfrag{A}[Bc][Bc]{\scalebox{1}{\scriptsize{$i$}}}
\psfrag{B}[Bc][Bc]{\scalebox{1}{\scriptsize{$i$}}}
\psfrag{D}[Bc][Bc]{\scalebox{1}{\scriptsize{$\Delta$}}}
\psfrag{P}[Bc][Bc]{\scalebox{1}{\scriptsize{$\pr_1$}}}
\psfrag{Q}[Bc][Bc]{\scalebox{1}{\scriptsize{$\pr_2$}}}
\psfrag{X}[Bc][Bc]{\scalebox{1}{\scriptsize{$H$}}}
\psfrag{Y}[Bc][Bc]{\scalebox{1}{\scriptsize{$G$}}}
\psfrag{Z}[Bc][Bc]{\scalebox{1}{\scriptsize{$H$}}}
\psfrag{W}[Bc][Bc]{\scalebox{1}{\scriptsize{$i/i$}}}
\psfrag{id}[Bc][Bc]{\scalebox{1}{\scriptsize{$\id$}}}
\includegraphics[scale=.4]{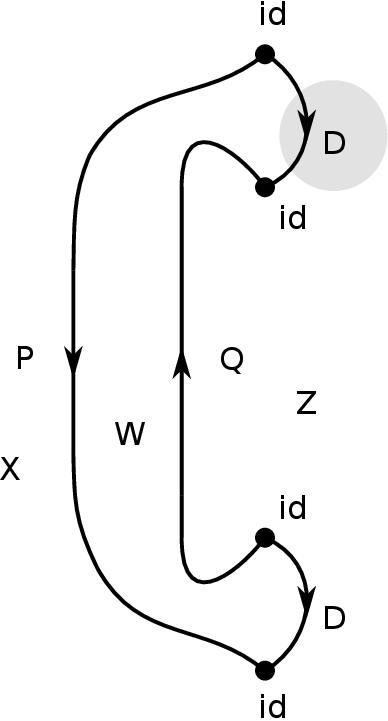}
}}
\stackrel{\textrm{\eqref{eq:zigzag-relations}}}{=}\;
\vcenter{\hbox{
\psfrag{A}[Bc][Bc]{\scalebox{1}{\scriptsize{$i$}}}
\psfrag{B}[Bc][Bc]{\scalebox{1}{\scriptsize{$i$}}}
\psfrag{D}[Bc][Bc]{\scalebox{1}{\scriptsize{$\Delta$}}}
\psfrag{P}[Bc][Bc]{\scalebox{1}{\scriptsize{$\pr_1$}}}
\psfrag{Q}[Bc][Bc]{\scalebox{1}{\scriptsize{$\pr_2$}}}
\psfrag{X}[Bc][Bc]{\scalebox{1}{\scriptsize{$H$}}}
\psfrag{Y}[Bc][Bc]{\scalebox{1}{\scriptsize{$G$}}}
\psfrag{Z}[Bc][Bc]{\scalebox{1}{\scriptsize{$H$}}}
\psfrag{W}[Bc][Bc]{\scalebox{1}{\scriptsize{$i/i$}}}
\psfrag{id}[Bc][Bc]{\scalebox{1}{\scriptsize{$\id$}}}
\psfrag{ID}[Bc][Bc]{\scalebox{1}{\scriptsize{$\Id$}}}
\includegraphics[scale=.4]{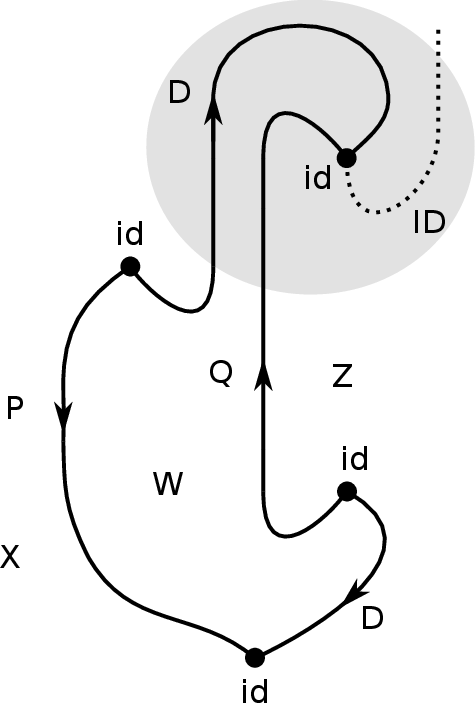}
}}
\overset{\eqref{Rem:two-functorialities-agree-strings}}{=}\;
\vcenter{\hbox{
\psfrag{A}[Bc][Bc]{\scalebox{1}{\scriptsize{$i$}}}
\psfrag{B}[Bc][Bc]{\scalebox{1}{\scriptsize{$i$}}}
\psfrag{D}[Bc][Bc]{\scalebox{1}{\scriptsize{$\Delta$}}}
\psfrag{P}[Bc][Bc]{\scalebox{1}{\scriptsize{$\pr_1$}}}
\psfrag{Q}[Bc][Bc]{\scalebox{1}{\scriptsize{$\pr_2$}}}
\psfrag{X}[Bc][Bc]{\scalebox{1}{\scriptsize{$H$}}}
\psfrag{Y}[Bc][Bc]{\scalebox{1}{\scriptsize{$G$}}}
\psfrag{Z}[Bc][Bc]{\scalebox{1}{\scriptsize{$H$}}}
\psfrag{W}[Bc][Bc]{\scalebox{1}{\scriptsize{$i/i$}}}
\psfrag{id}[Bc][Bc]{\scalebox{1}{\scriptsize{$\id$}}}
\psfrag{ID}[Bc][Bc]{\scalebox{1}{\scriptsize{$\Id$}}}
\includegraphics[scale=.4]{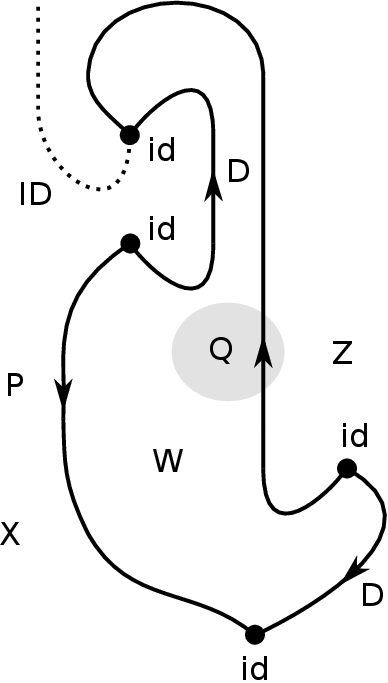}
}}
\stackrel{\textrm{\eqref{eq:zigzag-relations}}}{=}
\]
\[
\vcenter{\hbox{
\psfrag{A}[Bc][Bc]{\scalebox{1}{\scriptsize{$i$}}}
\psfrag{B}[Bc][Bc]{\scalebox{1}{\scriptsize{$i$}}}
\psfrag{D}[Bc][Bc]{\scalebox{1}{\scriptsize{$\Delta$}}}
\psfrag{P}[Bc][Bc]{\scalebox{1}{\scriptsize{$\pr_1$}}}
\psfrag{Q}[Bc][Bc]{\scalebox{1}{\scriptsize{$\pr_2$}}}
\psfrag{X}[Bc][Bc]{\scalebox{1}{\scriptsize{$H$}}}
\psfrag{Y}[Bc][Bc]{\scalebox{1}{\scriptsize{$G$}}}
\psfrag{Z}[Bc][Bc]{\scalebox{1}{\scriptsize{$H$}}}
\psfrag{W}[Bc][Bc]{\scalebox{1}{\scriptsize{$i/i$}}}
\psfrag{id}[Bc][Bc]{\scalebox{1}{\scriptsize{$\id$}}}
\psfrag{ID}[Bc][Bc]{\scalebox{1}{\scriptsize{$\Id$}}}
\includegraphics[scale=.4]{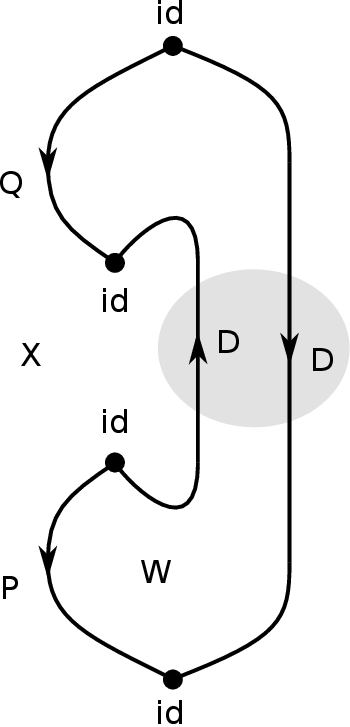}
}}
\quad \underset{\textrm{for }\eqref{eq:Frob-isocomma}}{\overset{\textrm{\eqref{eq:string-gluing}}}{=}}
\vcenter{\hbox{
\psfrag{A}[Bc][Bc]{\scalebox{1}{\scriptsize{$i$}}}
\psfrag{B}[Bc][Bc]{\scalebox{1}{\scriptsize{$i$}}}
\psfrag{D}[Bc][Bc]{\scalebox{1}{\scriptsize{$\Delta$}}}
\psfrag{P}[Bc][Bc]{\scalebox{1}{\scriptsize{$\pr_1$}}}
\psfrag{Q}[Bc][Bc]{\scalebox{1}{\scriptsize{$\pr_2$}}}
\psfrag{X}[Bc][Bc]{\scalebox{1}{\scriptsize{$H$}}}
\psfrag{Y}[Bc][Bc]{\scalebox{1}{\scriptsize{$G$}}}
\psfrag{Z}[Bc][Bc]{\scalebox{1}{\scriptsize{$H$}}}
\psfrag{W}[Bc][Bc]{\scalebox{1}{\scriptsize{$i/i$}}}
\psfrag{id}[Bc][Bc]{\scalebox{1}{\scriptsize{$\id$}}}
\psfrag{ID}[Bc][Bc]{\scalebox{1}{\scriptsize{$\Id$}}}
\includegraphics[scale=.4]{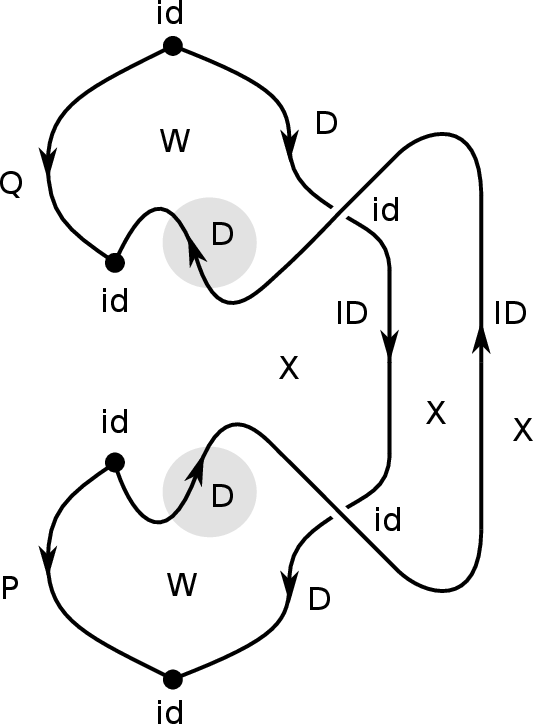}
}}
\;\;\stackrel{\textrm{\eqref{eq:zigzag-relations}}}{=}\;\;
\vcenter{\hbox{
\psfrag{A}[Bc][Bc]{\scalebox{1}{\scriptsize{$i$}}}
\psfrag{B}[Bc][Bc]{\scalebox{1}{\scriptsize{$i$}}}
\psfrag{D}[Bc][Bc]{\scalebox{1}{\scriptsize{$\Delta$}}}
\psfrag{P}[Bc][Bc]{\scalebox{1}{\scriptsize{$\pr_1$}}}
\psfrag{Q}[Bc][Bc]{\scalebox{1}{\scriptsize{$\pr_2$}}}
\psfrag{X}[Bc][Bc]{\scalebox{1}{\scriptsize{$H$}}}
\psfrag{Y}[Bc][Bc]{\scalebox{1}{\scriptsize{$G$}}}
\psfrag{Z}[Bc][Bc]{\scalebox{1}{\scriptsize{$H$}}}
\psfrag{W}[Bc][Bc]{\scalebox{1}{\scriptsize{$i/i$}}}
\psfrag{id}[Bc][Bc]{\scalebox{1}{\scriptsize{$\id$}}}
\psfrag{ID}[Bc][Bc]{\scalebox{1}{\scriptsize{$\Id$}}}
\includegraphics[scale=.4]{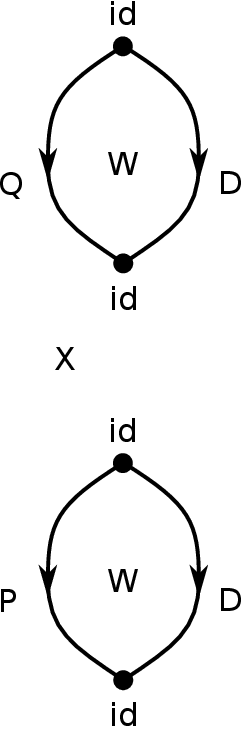}
}}
\]
Then we conclude with the relations $\pr_1\Delta = \Id_H = \pr_2\Delta$ of \eqref{eq:def-Delta}.
\end{Exa}

\bigbreak
\section{The bicategory of Mackey 2-functors}
\label{sec:bicat_2Mack}%
\medskip

Let us investigate morphisms of Mackey 2-functors, which only made a brief appearance in \Cref{sec:sub-quotients} (see \Cref{Def:morphism-of-Mackey-2-functors}). The following proposition characterizes morphisms in four equivalent ways.

\begin{Prop} \label{Prop:Spanhat_trans}
Let $t\colon \MM\Rightarrow \cat N$ be a pre-morphism of (rectified) Mackey 2-functors $\MM,\cat N\colon \GG^\op \to \ADD$ (\ie $t$ is a pseudo-natural transformation of the underlying 2-functors).
The following are equivalent:
\begin{enumerate}[\rm(i)]
\item
\label{it:trans-a}%
$t$ extends (in a unique way) to a pseudo-natural transformation
$\hat t\colon \widehat{\MM} \Rightarrow \widehat{\cat N}$
between the two extended pseudo-functors $\widehat{\MM}, \widehat{\cat N}\colon \Spanhat(\GG;\JJ)\to \ADD$ obtained as in \Cref{Thm:UP-Spanhat}.
\item
\label{it:trans-b}%
For each $(i\colon H\into G)\in \JJ$, the two mates
\[
\vcenter { \hbox{
\xymatrix{
\MM (G) \ar@{<-}[d]_-{i_!} \ar[r]^-{t_G} & \cat{N}(G) \ar@{<-}[d]^{i_!} \\
\MM (H) \ar[r]_-{t_H} \ar@{}[ur]|{\NWcell\; (t_i^{-1})_!} & \cat{N}(H)
}
}}
\quad \textrm{ and } \quad
\vcenter { \hbox{
\xymatrix{
\MM (G) \ar@{<-}[d]_-{i_*} \ar[r]^-{t_G} & \cat{N}(G) \ar@{<-}[d]^{i_*} \\
\MM (H) \ar[r]_-{t_H} \ar@{}[ur]|{\SEcell\; (t_i)_*} & \cat{N}(H)
}
}}
\]
obtained from the component $t_i \colon i^* \circ t_G \Rightarrow t_H \circ i^*$ are each other's inverses.
\item
\label{it:trans-c}%
For each $(i\colon H\into G)\in \JJ$, the left mate $(t_i\inv)_!$ of~\eqref{it:trans-b} is invertible.
\item
\label{it:trans-d}%
For each $(i\colon H\into G)\in \JJ$, the right mate $(t_i)_*$ of~\eqref{it:trans-b} is invertible.
\end{enumerate}
\end{Prop}

\begin{proof}
As was already mentioned in \Cref{sec:sub-quotients}, the equivalence between~\eqref{it:trans-b} and the weaker~\eqref{it:trans-c} or~\eqref{it:trans-d} is immediate because of the commutative square
\begin{equation*}
\vcenter{\xymatrix@C=4em@L=1ex{
t_G i_!^{\scriptscriptstyle \MM} \ar@{=>}[r]^-{\Displ t_G \Theta_i^{\MM}}
& t_G i_*^{\scriptscriptstyle \MM} \ar@{=>}[d]^-{\Displ(t_i)_*}
\\
i_!^{\scriptscriptstyle \NN} t_H \ar@{=>}[r]^-{\Displ\Theta_i^{\NN} t_H} \ar@{=>}[u]^-{\Displ (t_i\inv)_!}
& i_*^{\scriptscriptstyle \NN} t_H
}}
\end{equation*}
of \Cref{Prop:Theta-nat}, which already gives us (in the rectified setting where $\Theta=\Id$) that one composition is the identity:
\[
(t_i)_*\circ (t_i\inv)_!=\id\,.
\]

Now assume~\eqref{it:trans-a}, \ie that we are given a (strong, oplax-oriented) pseudo-natural transformation
$\hat t\colon \widehat{\MM}\Rightarrow \widehat{\cat N}$, and let $\MM:=\widehat{\MM}\circ (-)^*$, $\cat N:=\widehat{\cat N}\circ (-)^*$ and $t:=\hat t\circ (-)^*\colon \MM\to \cat N$ be the restictions to~$\GG$.
\Cref{Lem:tr_inv_mate} (applied to the adjunction $(\ell,r):=(i_!,i^*)$ in $\Spanhat$) shows that the components $\hat t(i_!)$ at each $i\in \JJ$ are uniquely determined by the components $t(i)=\hat t(i^*)$ and the adjunctions $i_!\dashv i^*$, as the mates
\begin{equation} \label{eq:left_computation_of_t}
\hat t(i_!)=(\hat t(i^*)^{-1})_! = (t(i)^{-1})_!
\end{equation}
for all $i\in \JJ$.
Similarly, we may use \Cref{Lem:tr_inv_mate} with the other adjunctions $i^*\dashv i_*$, which yields
\begin{equation} \label{eq:right_computation_of_t}
\hat t(i_*)=(\hat t(i^*)_*)^{-1} = (t(i)_*)^{-1} \,.
\end{equation}
Of course in $\Spanhat$ we have the equality $i_*=i_!$ of 1-cells, hence~\eqref{eq:left_computation_of_t} and~\eqref{eq:right_computation_of_t} together imply that
\begin{equation} \label{eq:equality_of_t's}
(t(i)^{-1})_! = (t(i)_*)^{-1}
\end{equation}
for all $i\in \JJ$. Thus~\eqref{it:trans-a} implies~\eqref{it:trans-b}.

It remains only to show that~\eqref{it:trans-b} implies~(\ref{it:trans-a}).
Thus assume that $t\colon \MM\Rightarrow \cat N$ satisfies~\eqref{it:trans-b}.
Recall that, as in the proof of \Cref{Thm:UP-PsFun-Span}, the naturality and functoriality axioms of $\hat t$ and the decomposition
\begin{align*}
\xymatrix{
& \ar[rr]^-{\hat t} \ar[d]_-{\widehat{\MM} u^*}
 \ar@/_9ex/[dd]_-{\widehat{\MM} (i_! \circ u^*)}
 \ar@{}[ddl]|{\SWcell\;\simeq} &&
 \ar[d]^-{\widehat{\cat{N}} u^*}
 \ar@/^9ex/[dd]^-{\widehat{\cat{N}}(i_!\circ u^*)}
 \ar@{}[dll]|{\SWcell\, \hat t (u^*)} &
\ar@{}[ddl]|{\SWcell\; \simeq} \\
& \ar[rr]^-{\hat t} \ar[d]_-{\widehat{\MM} i_!} &&
 \ar[d]^-{\widehat{\cat{N}} i_!}
 \ar@{}[dll]|{\SWcell\, \hat t (i_!)} & \\
& \ar[rr]_-{\hat t} && &
}
\end{align*}
show that \emph{all} components $\hat t(i_!u^*)$ of $\hat t$ are dictated by~$t$, so we must only show that these components satisfy the axioms of a pseudo-natural transformation $\widehat{\MM}\Rightarrow \widehat{\cat N}$.

Recall that $\widehat{\MM}$ is (re)constructed by gluing together two extensions of $\MM$, namely $\MM_\star$ and~$\MM^\star$, given by \Cref{Cons:UP-Span} and \Cref{Cons:UP-Span-co} respectively:
\[
\vcenter{
\xymatrix@C=3em{
& \Span(\GG;\JJ) \ar[d]^-{(-)_\star} \ar@/^3ex/[rd]^-{\MM_\star}
\\
\GG^\op \ar@/^3ex/[ru]^(.4){(-)^*} \ar@/_3ex/[rd]_-{(-)^{\costar}}
 \ar@{}[r]|-{\eqref{eq:two-embeddings}}
& \Spanhat(\GG;\JJ) \ar[r]^-{\widehat{\MM}}
& \cat{C}
\\
& \Span(\GG;\JJ)^{\co} \ar[u]_-{(-)^\star} \ar@/_3ex/[ru]_-{\MM^\star}
}}
\]
Indeed, by definition (see~\eqref{eq:def-Fhat}) $\widehat{\MM}$ is given by $\widehat{\MM}([\beta,\alpha])=\MM_\star(\alpha)\MM^\star (\beta) $ on arbitrary 2-cells $[\beta,\alpha]$ of $\Spanhat$ as in~\eqref{eq:2-cell-of-Spanhat}. The composite diagonal on the left is the canonical `motivic' pseudo-functor $(-)^*\colon \GG^\op\to \Spanhat$ of \Cref{Not:motive}. Similar remarks hold for~$\cat N$.

As we have observed above, the equality~\eqref{eq:equality_of_t's} is necessary if we want to extend~$t$ to a pseudo-natural~$\hat t\colon \widehat{\MM}\Rightarrow \widehat{\cat N}$, and the \emph{data} of the latter is then uniquely determined by setting $\hat t(i_!):= \hat t (i_*):= (t(i)^{-1})_! = (t(i)_*)^{-1}$.

Note that~\eqref{it:trans-b} implies in particular that $t\colon \MM\Rightarrow \cat N$ is a $\JJ_!$-strong transformation (\Cref{Def:J_!-strong}). Hence by \Cref{Thm:UP-PsFun-Span} we already know that $\hat t$ is a transformation $\MM_\star \Rightarrow \cat N_\star$.
Similarly,~\eqref{it:trans-b} also implies that $t$ is $\JJ_*$-strong and thus, by the dual \Cref{Thm:UP-PsFun-Span-co}, extends to a transformation $\hat t\colon \MM^\star \Rightarrow \cat N^\star$ (see also \Cref{Rem:inverse}).

Since $\widehat{\MM}$ and $\widehat{\cat N}$ coincide with $\MM_\star$ and~$\cat N_\star$ on 0- and 1-cells and have the same structural isomorphisms $\fun$ and $\un$, this already implies that $\hat t$ satisfies the functoriality axioms of a transformation $\widehat{\MM}\Rightarrow \widehat{\cat N}$, so it only remains to check that it satisfies the naturality axiom.
Consider again an arbitrary 2-cell of~$\Spanhat$
\[
[\beta,\alpha] = \left[
\xymatrix@1@L=1ex{ i_!u^* & k_!w^* \ar@{=>}[l]_-{\beta} \ar@{=>}[r]^-{\alpha} & j_!v^* }
\right]
\quad \colon \quad i_!u^* \Rightarrow j_!v^*
\]
with $\alpha=[a,\alpha_1,\alpha_2]$ and $\beta=[b,\beta_1,\beta_2]$ as in~\eqref{eq:2-cell-of-Spanhat}.
The following computation
\[
\vcenter {\hbox{
\psfrag{A}[Bc][Bc]{\scalebox{1}{\scriptsize{$t$}}}
\psfrag{B}[Bc][Bc]{\scalebox{1}{\scriptsize{$\widehat{\MM} (i_!u^*)$}}}
\psfrag{C}[Bc][Bc]{\scalebox{1}{\scriptsize{$\widehat{\cat N} (j_!v^*)$}}}
\psfrag{D}[Bc][Bc]{\scalebox{1}{\scriptsize{$t$}}}
\psfrag{E}[Bc][Bc]{\scalebox{1}{\scriptsize{$\;\;\widehat{\MM} (j_!v^*)$}}}
\psfrag{F}[Bc][Bc]{\scalebox{1}{\scriptsize{$\hat t(j_!v^*)$}}}
\psfrag{G}[Bc][Bc]{\scalebox{1}{\scriptsize{$\widehat{\MM} [\beta, \alpha]$}}}
\includegraphics[scale=.4]{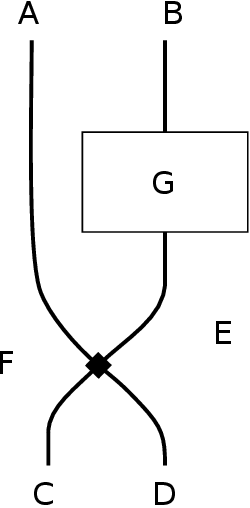}
}}
\; = \;
\vcenter {\hbox{
\psfrag{A}[Bc][Bc]{\scalebox{1}{\scriptsize{$t$}}}
\psfrag{B}[Bc][Bc]{\scalebox{1}{\scriptsize{$\widehat{\MM} (i_!u^*)$}}}
\psfrag{C}[Bc][Bc]{\scalebox{1}{\scriptsize{$\widehat{\cat N} (j_!v^*)$}}}
\psfrag{D}[Bc][Bc]{\scalebox{1}{\scriptsize{$t$}}}
\psfrag{E}[Bc][Bc]{\scalebox{1}{\scriptsize{$\;\;\widehat{\MM} (j_!v^*)$}}}
\psfrag{F}[Bc][Bc]{\scalebox{1}{\scriptsize{$\hat t(j_!v^*)$}}}
\psfrag{G}[Bc][Bc]{\scalebox{1}{\scriptsize{$\MM_\star \alpha$}}}
\psfrag{H}[Bc][Bc]{\scalebox{1}{\scriptsize{$\MM^\star\beta$}}}
\includegraphics[scale=.4]{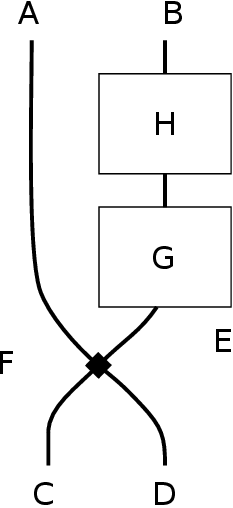}
}}
\; \overset{(\diamondsuit)}{=}
\vcenter {\hbox{
\psfrag{A}[Bc][Bc]{\scalebox{1}{\scriptsize{$t$}}}
\psfrag{B}[Bc][Bc]{\scalebox{1}{\scriptsize{$\widehat{\MM} (i_!u^*)$}}}
\psfrag{C}[Bc][Bc]{\scalebox{1}{\scriptsize{$\widehat{\cat N} (j_!v^*)$}}}
\psfrag{D}[Bc][Bc]{\scalebox{1}{\scriptsize{$t$}}}
\psfrag{E}[Bc][Bc]{\scalebox{1}{\scriptsize{$\widehat{\MM} (j_!v^*)$}}}
\psfrag{F}[Bc][Bc]{\scalebox{1}{\scriptsize{$\;\;\hat t(k_!w^*)$}}}
\psfrag{G}[Bc][Bc]{\scalebox{1}{\scriptsize{$\cat N_\star \alpha$}}}
\psfrag{H}[Bc][Bc]{\scalebox{1}{\scriptsize{$\MM^\star\beta$}}}
\includegraphics[scale=.4]{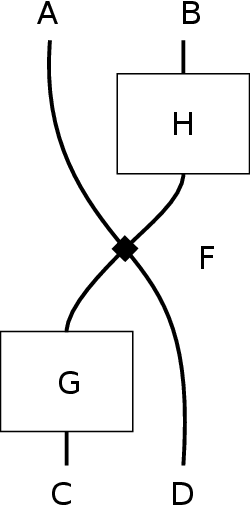}
}}
\;\; \overset{(\heartsuit)}{=} \;
\vcenter {\hbox{
\psfrag{A}[Bc][Bc]{\scalebox{1}{\scriptsize{$t$}}}
\psfrag{B}[Bc][Bc]{\scalebox{1}{\scriptsize{$\widehat{\MM} (i_!u^*)$}}}
\psfrag{C}[Bc][Bc]{\scalebox{1}{\scriptsize{$\widehat{\cat N} (j_!v^*)$}}}
\psfrag{D}[Bc][Bc]{\scalebox{1}{\scriptsize{$ t$}}}
\psfrag{E}[Bc][Bc]{\scalebox{1}{\scriptsize{$\widehat{\cat N} (i_!u^*)$}}}
\psfrag{F}[Bc][Bc]{\scalebox{1}{\scriptsize{$\hat t(i_!u^*)$}}}
\psfrag{G}[Bc][Bc]{\scalebox{1}{\scriptsize{$\cat N_\star \alpha$}}}
\psfrag{H}[Bc][Bc]{\scalebox{1}{\scriptsize{$\cat N^\star\beta$}}}
\includegraphics[scale=.4]{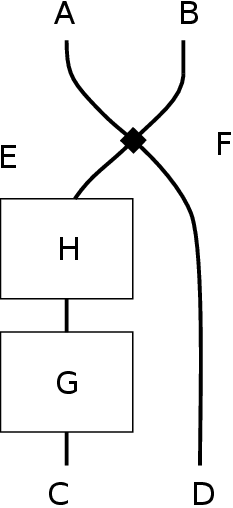}
}}
\; = \;
\vcenter {\hbox{
\psfrag{A}[Bc][Bc]{\scalebox{1}{\scriptsize{$ t$}}}
\psfrag{B}[Bc][Bc]{\scalebox{1}{\scriptsize{$\widehat{\MM} (i_!u^*)$}}}
\psfrag{C}[Bc][Bc]{\scalebox{1}{\scriptsize{$\widehat{\cat N} (j_!v^*)$}}}
\psfrag{D}[Bc][Bc]{\scalebox{1}{\scriptsize{$t$}}}
\psfrag{E}[Bc][Bc]{\scalebox{1}{\scriptsize{$\widehat{\cat N} (i_!u^*)\;$}}}
\psfrag{F}[Bc][Bc]{\scalebox{1}{\scriptsize{$\hat t(i_!u^*)$}}}
\psfrag{G}[Bc][Bc]{\scalebox{1}{\scriptsize{$\widehat{\cat N} [\beta, \alpha]$}}}
\includegraphics[scale=.4]{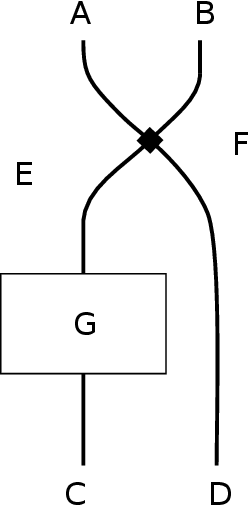}
}}
\]
shows that the naturality of $\hat t$ reduces to the naturality ($\diamondsuit$) as a transformation $\MM_\star\Rightarrow \cat N_\star$ and the naturality ($\heartsuit$) as a transformation $\MM^\star\Rightarrow \cat N^\star$, which we know hold true.
\end{proof}

\begin{Not}
\label{Def:biMack}%
\index{mack@$\biMack$} \index{$mack$@$\biMack$ \, Mackey 2-functors}%
\index{mackic@$\bicMack$} \index{$mackic$@$\bicMack$ \, idempotent-complete Mackey 2-functors}%
We denote by
\[
\biMack(\GG;\JJ)
\]
the (2-full) sub-2-category of $\PsFun(\GG^\op,\ADD)$ consisting of Mackey 2-functors, morphisms between them (\ie transformations satisfying the equivalent conditions of \Cref{Prop:Spanhat_trans}) and modifications.
Note that, in order to talk about Mackey 2-functors, we need the additivity axiom and thus we are now making use of Hypothesis~\ref{Hyp:G_and_I_for_Span}\,\eqref{Hyp:G-I-d}.
From the perspective of additive decompositions coming up in \Cref{ch:additive-motives},
it is natural to also consider the 1-full and 2-full sub-2-category
\[
\bicMack(\GG;\JJ)
\]
of~$\biMack(\GG;\JJ)$ consisting of those Mackey functors $\MM\colon \GG^\op\to \ICADD$ taking values in \emph{idempotent-complete} additive categories.
\end{Not}

\begin{Thm} \label{Thm:identif_2Macks}
By precomposing with the embedding $(-)^*\colon \GG^\op\hookrightarrow \Spanhat(\GG;\JJ)$, we obtain a biequivalence
\[
\PsFun_{\amalg} ( \Spanhat(\GG;\JJ) , \ADD) \stackrel{\sim}{\longrightarrow} \biMack(\GG;\JJ)
\]
where on the left we have the bicategory of additive pseudo-functors on $\Spanhat(\GG;\JJ)$, all strong pseudo-natural transformations between them, and modifications. Similarly, with idempotent-complete values we have a biequivalence
\[
\PsFun_{\amalg} ( \Spanhat(\GG;\JJ) , \ICADD) \stackrel{\sim}{\longrightarrow} \bicMack(\GG;\JJ)\,.
\]
\end{Thm}

\begin{proof}
This is now an immediate consequence of \Cref{Prop:Spanhat_trans} and the universal property of Mackey 2-motives, once we have checked that modifications also correspond.
But the latter was already proved as part of \Cref{Thm:UP-PsFun-Span}. Indeed, as was observed above, the data of a transformation $\hat t\colon \widehat{\MM}\Rightarrow \widehat{\cat N}\colon \Spanhat\to \cat C$ is the same as that of its restriction to $\Span$, and moreover is uniquely determined by its further restriction $t=\hat t\circ (-)^*$ on~$\GG^\op$ by taking mates. Similarly, the data of a modification at the levels of $\GG^\op$, $\Span$ or $\Spanhat$ are all the same, hence it is only a question of verifying the modification axiom with respect to the various classes of 1-cells. Hence the verification done in the proof of \Cref{Thm:UP-PsFun-Span} (for $\Span$) works here as well (for $\Spanhat$) and yields the required bijections.
\end{proof}

\begin{Rem}
The above is analogous to the situation with ordinary Mackey 1-functors.
Indeed, the category of Mackey functors (over~$\mathbb Z$) for a finite group $G$ can be defined as $\Mackey(G):=\Funplus(\widehat{G\sset}, \Ab)$, the category of additive functors on the ordinary span category of $G$-sets; see \Cref{sec:ordinary-spans}.
The canonical embedding $(-)^\star\colon(G\sset)^\op \hookrightarrow \widehat{G\sset}$ induces a functor $\Mackey(G) \to \Funplus (G\sset^\op , \Ab)$, that is, any natural transformation $t\colon M\Rightarrow N$ of Mackey functors also defines a natural transformation $s=t\circ (-)^\star\colon M^\star = M \circ (-)^\star \Rightarrow N \circ (-)^\star = N^\star$ (a `pre-morphism').
But conversely, not every natural transformation $s\colon M^\star \Rightarrow N^\star$ is a natural transformation $M\Rightarrow N$: We are still missing the commutativity of the squares involving the induction maps.
\end{Rem}

\begin{Exa} \label{Exa:add-der-mor}
Let $\cat D$ and $\cat E$ be additive derivators $\Cat^\op\to \ADD$ (\Cref{sec:add-der-Mackey}) and let $F\colon \cat D\to \cat E$ be a morphism of additive derivators, that is, a pseudo-natural transformation where the functors $F_J\colon \cat D(J)\to \cat E(J)$ are additive for all $J\in \Cat$.
Then the restricted 2-functors $\MM := \cat D \restr{\gpd}$ and $\cat N:= \cat E \restr{\gpd}$ are Mackey 2-functors (Theorem~\ref{Thm:ambidex-der}), and the restricted pseudo-natural transformation $t:=F\restr{\gpd} \colon \MM\to \cat N$ is obviously a pre-morphism between them.
If $F$ moreover preserves (homotopy) limits and colimits then $t$ is even a \emph{morphism} of Mackey 2-functors.
Indeed, for a morphism of derivators the preservation of limits and colimits is equivalent to the \emph{a priori} stronger preservation of all homotopy Kan extensions (see \cite[\S2.2]{Groth13}), which for $t=F \restr{\gpd}$ specializes to the preservation of $i_!$ and $i_*$ in the sense of \Cref{Prop:Spanhat_trans}\,\eqref{it:trans-c} and~\eqref{it:trans-d}.

Thus restriction along the inclusion $\gpd\hookrightarrow \Cat$ induces a pseudo-functor from the bicategory of additive derivators with continuous and co-continuous morphisms to the bicategory of Mackey 2-functors $\biMack(\gpd;\faithful)$.
\end{Exa}

\end{chapter-six}
%
\chapter{Additive Mackey 2-motives and decompositions}
\label{ch:additive-motives}%
\bigbreak
\begin{chapter-seven}

In this chapter we take a closer look at the additivity properties of Mackey 2-functors. With an eye on decomposition results, we introduce the bicategory $\mathbb Z \Spanhat$ of \emph{additive} Mackey 2-motives (\Cref{Def:additive-2motives}), where it is possible to take differences of 2-cells and to use idempotent 2-cells in order to split 1-cells and 0-cells into direct sums. We refer the reader to \Cref{sec:additive-sedative} and~\ref{sec:additive-bicats} for recollections on additivity and semi-additivity in categories and bicategories.

\bigbreak
\section{Additive Mackey 2-motives}
\label{sec:additive-motives}%
\medskip

%
\begin{Hyp}
\label{Hyp:G_and_I_for_ZSpan}%
For simplicity, we revert to the simplest set of hypotheses, as stated in \Cref{Hyp:GG}, namely~$\GG$ is a sub-2-category of the (2,1)-category~$\groupoid$ of finite groupoids, and $\JJ$ consists of all faithful 1-cells of~$\GG$. Furthermore, we assume that $\GG$ is closed under finite coproducts~$\amalg$ in~$\groupoid$.
\end{Hyp}

\begin{Rem}
The interested reader will replace \Cref{Hyp:G_and_I_for_ZSpan} by \Cref{Hyp:G_and_I_for_Span} plus a requirement that coproducts in~$\GG$ be sufficiently `disjoint'; \cf \Cref{Exa:ordinary_span_is_sad} and (the proof of)~\Cref{Lem:Span-prods}. We avoid this generalization, since it does not yet present a good ratio of `new examples per added abstract nonsense'.
\end{Rem}

Let $\Spanhat(\GG;\JJ)$ be the associated bicategory of Mackey 2-motives, as in \Cref{Def:Spanhat-bicat}.
We begin with the following observation.

\begin{Prop}
\label{Prop:semi-add-2motives}%
The finite coproducts of $\GG$ make $\Spanhat(\GG;\JJ)$ into a \emph{locally semi-additive} bicategory in the sense of \Cref{Def:Sad-enriched-etc}. In other words, each Hom category $\Spanhat(\GG;\JJ)(G,H)$ is semi-additive: It admits direct sums
\[
(i_!u^*) \oplus (j_!v^*)= (i, j)_!(u, v)^*
\]
of its objects, as well as a zero object~$0= (G \leftarrow \varnothing \to H)$, and therefore inherits a (unique) sum operation for parallel morphisms which makes each Hom-set into an abelian monoid, and composition bilinear. And the composition functors $-\circ -$ are additive in both variables.

Moreover, the finite coproducts of objects in~$\GG$ become in $\Spanhat(\GG;\JJ)$ \emph{direct sums} in the sense of \Cref{Def:sums_in_bicats}.
\end{Prop}

\begin{proof}
This is all rather straightforward from the definitions. See \Cref{Exa:ordinary_span_is_sad} for details on the sum of 1-cells and 2-cells. Let us just say a word on the direct sum of objects in $\Spanhat(\GG;\JJ)$, concretely. The empty groupoid~$\varnothing \in \GG$ clearly becomes a zero object. Given two groupoids $X_1,X_2\in \GG_0$ with coproduct
$i_1\colon X_1\to X_1\sqcup X_2 \gets X_2 :\! i_2$ in~$\GG$, the resulting direct sum in $\Span$ has structure 1-cells given by the following spans:
\[
\xymatrix{
 X_1 \ar@<2pt>[r]^-{(i_1)_!} \ar@{<-}@<-2pt>[r]_-{(i_1)^*}
& X_1\sqcup X_2 \ar@{<-}@<2pt>[r]^-{(i_2)_!} \ar@<-2pt>[r]_-{(i_2)^*}
& X_2 \,.
}
\]
These diagrams remain direct sums in the double-span bicategory $\Spanhat$. Half of this claim follows from \Cref{Lem:Span-prods} below. The other verifications are similar and left to the reader.
\end{proof}

\begin{Lem} \label{Lem:Span-prods}
If $j_1\colon H_1 \to H_1\sqcup H_2 \gets H_2:\!j_2$ is a coproduct of finite groupoids contained in $\GG\subseteq \gpd$, then $j_1^* \colon H_1\gets H_1\sqcup H_2 \to H_2 :\! j_2^*$ is a product in $\Span(\GG;\JJ)$.
\end{Lem}

\begin{proof}
Let $G\in \GG_0$ be a finite groupoid.
We must show that the functor
\[
\Span(\GG;\JJ) (G, H_1 \sqcup H_2)
\stackrel{\Phi}{\longrightarrow}
\Span(\GG;\JJ) (G, H_1) \times \Span(\GG;\JJ) (G, H_2)
\]
induced by composition with $j_1^*$ and~$j_2^*$ is an equivalence. On objects, $\Phi$ sends a span $G \stackrel{u}{\gets} P \stackrel{i}{\to} H_1\sqcup H_2$ to the pair
\[
\vcenter{\xymatrix@C=14pt@R=14pt{
& & (i/j_1) \ar[ld]_-{} \ar[dr]^-{}
 \ar@{}[dd]|-{\Ecell} &
\\
& P \ar[dr]_(.4){i} \ar[dl]_u
&& H_1 \ar[dl]^(.4){j_1} \ar@{=}[dr] & \\
G \ar@{..>}[rr]_-{i_!u^*} && H_1\sqcup H_2 \ar@{..>}[rr]_-{j_1^*} && H_1
}}
\quad
\vcenter{\xymatrix@C=14pt@R=14pt{
& & (i/j_2) \ar[ld]_-{} \ar[dr]^-{}
 \ar@{}[dd]|-{\Ecell} &
\\
& P \ar[dr]_(.4){i} \ar[dl]_u
&& H_2 \ar[dl]^(.4){j_2} \ar@{=}[dr] & \\
G \ar@{..>}[rr]_-{i_!u^*} && H_1\sqcup H_2 \ar@{..>}[rr]_-{j_2^*} && H_2
}}
\]
of composite spans. On maps, it is similarly induced by whiskering with $j_\ell^*$ for~$\ell \in \{1,2\}$. Write $P=P_1\sqcup P_2$ for the decomposition of $P$ induced by that of its image~$i(P)\subseteq H_1\sqcup H_2$, that is $P_\ell:= i^{-1}(i(P)\cap H_\ell)$. It is easy to check that the comparison functor $P_\ell\to (i/j_\ell)$
\[
\xymatrix@R=14pt@C14pt{
& P_\ell \ar[d]^{\cong} \ar@{_{(}->}@/_2ex/[ldd] \ar@/^2ex/[ddr]^{i|_{P_\ell}} & \\
& (i/ j_\ell) \ar[dl] \ar[dr] & \\
P \ar[dr]_(.4)i \ar@{}[rr]|{\Ecell} && H_\ell \ar@{_{(}->}[dl]^(.4){j_\ell} \\
& H_1\sqcup H_2 &
}
\]
is an equivalence for $\ell\in \{1,2\}$, from which we deduce an isomorphism
\begin{equation} \label{eq:iso_PhiPsi=Id}
\Phi(G \stackrel{u}{\gets} P \stackrel{i}{\to} H_1\sqcup H_2) \cong \big( (G \stackrel{u|_{P_1}}{\longleftarrow} P_1 \stackrel{i|_{P_1}\;}{\too} H_1) , (G \stackrel{u|_{P_2}}{\longleftarrow} P_2 \stackrel{i|_{P_2}\;}{\too} H_2) \big) \,.
\end{equation}
There is also a `summation' functor
\[
\Span(\GG;\JJ) (G, H_1) \times \Span(\GG;\JJ) (G, H_2)
\stackrel{\Psi}{\longrightarrow}
\Span(\GG;\JJ) (G, H_1 \sqcup H_2)
\]
which sends a pair
$( (G \stackrel{u_1}{\gets} Q_1 \stackrel{i_1\;}{\to} H_1) , (G \stackrel{u_2}{\gets} Q_2 \stackrel{i_2\;}{\to} H_2))$ to the span
\[
\big(
\xymatrix{
G & \ar[l]_-{(u_1,u_2)} Q_1 \sqcup Q_2 \ar[r]^-{i_1 \sqcup i_2} & H_1\sqcup H_2
}
\big)
\]
and similarly on morphisms. Using~\eqref{eq:iso_PhiPsi=Id}, it is now straightforward to check that $\Psi$ and $\Phi$ are mutually inverse equivalences.
\end{proof}

\begin{Rem} \label{Rem:motivation-for-additive-motives}
The 2-cells in the bicategory $\Spanhat$ of Mackey 2-motives do not have additive inverses, though, which can be inconvenient.
For instance, if $e=e^2\colon \Id_X\Rightarrow \Id_X$ is an idempotent 2-cell, we may wish to form the complementary idempotent $1_X - e$ (here we write $1_X:=\id_{\Id_X}$ for short). Moreover, we may be tempted to use the decomposition $1_X= e+(1_X-e)$ in order to split the 1-cell $\Id_X$ or even the object $X$ itself. Unfortunately, none of this can be done in the bicategory~$\Spanhat$. Fortunately, it is not too hard to enlarge such a bicategory so as to accommodate opposites and splittings. Indeed, in \Cref{sec:additive-bicats} we construct the \emph{block-completion} $\cat B^\flat$ of any locally semi-additive bicategory~$\cat B$ with direct sums. The bicategory $\cat B^\flat$ is \emph{block-complete}, that is: (1) it is locally idempotent-complete, \ie its idempotent 2-cells split 1-cells into direct sums, and (2) its idempotent 2-cells similarly give rise to direct sum decompositions of 0-cells (see \Cref{Rem:object-decompositions} and \Cref{Def:block-complete}). Moreover, there is an embedding $\cat B\to \cat B^\flat$ which is the universal pseudo-functor to a block-complete bicategory (\Cref{Thm:UP-flat}).

By \Cref{Prop:semi-add-2motives} this construction can be applied to~$\Spanhat$, leading to:
\end{Rem}

\begin{Def}[{Additive and $\kk$-linear Mackey 2-motives}]
\label{Def:additive-2motives}%
\index{Mackey 2-motives!additive --}\index{additive Mackey 2-motives}%
\index{Mackey 2-motives!$\kk$-linear --}%
\index{Mackey 2-motives!semi-additive --}%
\index{additive Mackey 2-motives}%
\index{semi-additive Mackey 2-motives}%
\index{bicategory!-- of additive Mackey 2-motives $\mathbb Z\Spanhat(\GG;\JJ)$}%
\index{zspanhatg@$\mathbb Z\Spanhat(\GG;\JJ)$}%
\index{$zspanhat$@$\mathbb Z \Spanhat$ \, additive Mackey 2-motives}%
\index{bicategory!-- of $\kk$-linear Mackey 2-motives $\kk\Spanhat(\GG;\JJ)$}%
\index{kSpanhatG@$\kk\Spanhat(\GG;\JJ)$}%
\index{klinear Mackey@$\kk$-linear Mackey 2-motives}%
\index{$kspanhat$@$\kk\Spanhat$ \, $\kk$-linear Mackey 2-motives}%
We define the bicategory of \emph{additive Mackey 2-motives} to be
\[
\mathbb Z\Spanhat(\GG;\JJ) := \Spanhat(\GG;\JJ)^\flat\,,
\]
the block-completion (as in \Cref{Cons:block-completion} and \Cref{Rem:combined-envelopes}) of the locally semi-additive bicategory $\Spanhat(\GG;\JJ)$.
For contrast, we refer to the unadorned $\Spanhat(\GG;\JJ)$ as the bicategory of \emph{semi-additive Mackey 2-motives}.
Concretely, the bicategory $\mathbb Z\Spanhat(\GG;\JJ)$ consists of:
\begin{enumerate}[{$\bullet$}]
\item As objects, all pairs $(G,e)$ where $G\in \GG_0$ is a finite groupoid and $e=e^2\colon \Id_G\Rightarrow \Id_G$ is an idempotent element of the (group-completed) commutative ring
\[
\Spanhat(\GG;\JJ)(G,G)(\Id_G,\Id_G)_+ \,.
\]
\item For Hom categories, the subcategories
\[
\mathbb Z\Spanhat(\GG;\JJ)((G,e),(G',e')) \; \subseteq \; \big((\Spanhat(\GG;\JJ)(G,G'))_+ \big)^\natural
\]
of the idempotent-completion $(-)^\natural$ of the group completion $(-)_+$ of the semi-additive category $\Spanhat(\GG;\JJ)(G,G')$, consisting of those 1-cells $u$ and 2-cells $\alpha$ which \emph{absorb} the idempotents $e$ and~$e'$, in the sense that $e\circ u \circ e' \cong u$ and $e\circ \alpha \circ e' \cong \alpha$ (see \Cref{Lem:absorption} for more details).
\end{enumerate}
More generally, if $\kk$ is any commutative ring we define the bicategory of \emph{$\kk$-linear Mackey 2-motives}
\[
\kk \Spanhat(\GG;\JJ) := \big( \kk \otimes_\mathbb Z \Spanhat(\GG;\JJ)_+ \big)^\flat
\]
to be the block-completion of the \emph{$\kk$-linearization} of (the group completion of) semi-additive Mackey 2-motives. The latter is obtained simply by tensoring all Hom-groups of 2-cells with~$\kk$ and extending the structural functors $\kk$-linearly.
\end{Def}

\begin{Rem}
\label{Rem:can_add_+}%
Additive Mackey 2-motives inherit from their semi-additive progenitors the property of classifying all Mackey 2-functors whose values are \emph{idem\-po\-tent-com\-plete} additive categories. In other words, precomposition with the canonical pseudo-functor $\GG^\op\to \mathbb Z\Spanhat(\GG;\JJ)$ induces a biequivalence
\begin{equation}
\label{eq:UP-ZSpanhat}%
\PsFun_\amalg (\mathbb Z\Spanhat(\GG;\JJ), \ICADD)\isotoo\bicMack(\GG;\JJ)\,.
\end{equation}
In order to see this, it suffices to break the canonical embedding down to its composition factors (see \Cref{Rem:combined-envelopes}):
\[
\vcenter{ \hbox{
\xymatrix@R=14pt@C=6pt{
& \GG^\op \ar[d] \ar@/_3ex/[ddddl] \\
& \Spanhat \ar[d] \\
& \Spanhat_+ \ar[d] \\
& (\Spanhat_+)^\natural \ar[d] \\
\mathbb Z\Spanhat \ar@{=}[r]^-{\overset{\textrm{def.}}{\phantom{m}}} & {((\Spanhat_+)^\natural)^\flat }
}
}}
\;\; \overset{\underset{\phantom{M}}{\PsFun_\amalg(-,\ICADD)}}{\longmapsto}
\!\!\!\!\!\!\! \!\!\!\!\! \!\!\!\!\! \!\!\!\!\!
\vcenter{ \hbox{
\xymatrix@R=14pt{
{\PsFun_\amalg(\GG^\op,\ICADD) \supseteq \bicMack \phantom{\PsFun_\amalg(\GG^\op,\ICADD) \supseteq }} \\
\PsFun_\amalg(\Spanhat , \ICADD) \ar[u]_{\simeq}_{\qquad \textrm{by \Cref{Thm:identif_2Macks}}} \\
\PsFun_\amalg(\Spanhat_+ , \ICADD) \ar[u]_{\simeq}_{\qquad \textrm{by \Cref{Rem:add_reflections_bicats}\,\eqref{it:additive-envelope-bicat} }} \\
\PsFun_\amalg(\Spanhat_+^\natural , \ICADD) \ar[u]_{\simeq}_{\qquad \textrm{by \Cref{Rem:add_reflections_bicats}\,\eqref{it:ic-envelope-bicat} }} \\
\PsFun_\amalg(\Spanhat^\flat , \ICADD) \ar[u]_{\simeq}_{\qquad \textrm{by \Cref{Thm:UP-flat}}}
}
}}
\]
Besides the universal properties of the various constructions, this also uses the convenient fact that additive pseudo-functors are automatically \emph{locally} additive (\Cref{Prop:(locally)-additive-pseudo-functors}), as well as the fact that the target 2-category $\ICADD$ is itself block-complete (\Cref{Exa:ICAdd-is-lic}).
\end{Rem}

\begin{Rem} \label{Rem:canonical-embedding-is-inj}
We note that each of the canonical pseudo-functors on the left-hand side of the above picture (and hence their composite $\GG^\op\to \mathbb Z\Spanhat(\GG;\JJ)$) does deserve the name `embedding', because it is injective on 0-, 1- and 2-cells.
Injectivity may fail in general on 2-cells for the group completion $\cat B\to \cat B_+$, but holds in this particular case where the Hom monoids of 2-cells in $\Spanhat(\GG;\JJ)$ with $\GG\subseteq\gpd$ are all free; see \Cref{Exa:ordinary_span_is_sad} for details.
\end{Rem}

\begin{Def}
\label{Def:realization}%
If $\MM\colon \GG^\op\to \ICADD$ is any Mackey 2-functor with idempotent-complete values, its essentially unique extension to an additive functor
\[
\widehat{\MM}\colon \mathbb Z \Spanhat(\GG;\JJ)\too \ICADD
\]
will be referred to as the \emph{realization} of additive 2-motives via~$\MM$.
\end{Def}

\begin{Rem}
\label{Rem:can_add_+_strict}%
\index{$2fun$@$\twoFun_\amalg$ \, additive 2-functors}  \index{$2category$@2-category!-- of additive 2-functors}%
\index{Fun@$\twoFun_\amalg$}%
\index{strictification!-- of $\mathbb Z\Spanhat(\GG;\JJ)$}%
\index{strictification!-- of realizations}%
If so wished, there is a way to obtain realizations as \emph{strict 2-functors} rather than mere pseudo-functors, just like Mackey \emph{2-functors} are 2-functors.
Indeed, recall that we may replace $\Spanhat(\GG;\JJ)$ with a biequivalent strict version $\Spanhat(\GG;\JJ)^\str$, for instance the one explicitly presented in Corollary~\ref{Cor:strict-Spanhat}, which is an essentially small strict 2-category. Then we obtain biequivalences
\[
\twoFun_\amalg \big(\mathbb Z\Spanhat(\GG;\JJ)^\str, \ICADD\big)
\stackrel{\sim}{\to}
\PsFun_\amalg \big(\mathbb Z\Spanhat(\GG;\JJ)^\str, \ICADD\big)
\stackrel{\sim}{\to}
\bicMack(\GG;\JJ)
\]
where $\mathbb Z\Spanhat(\GG;\JJ)^\str$ is defined as in \Cref{Def:additive-2motives}, and $\twoFun_\amalg(\cat{C}, \cat{C}')$ denotes the 2-category of additive 2-functors between 2-categories $\cat{C}$ and~$\cat{C}'$, pseudo-natural transformations between them and modifications (\cf \Cref{Not:Fun}).
The first biequivalence holds by Power's strictification theorem for pseudo-functors, recalled in Remark~\ref{Rem:coh_pseudofun}, and the second one is induced by the biequivalence \eqref{eq:UP-ZSpanhat} and the strictification $\Spanhat(\GG;\JJ)\stackrel{\sim}{\to}\Spanhat(\GG;\JJ)^\str$ (see also Remark~\ref{Rem:PsFun_bieq}).
\end{Rem}

\bigbreak
\section{The Yoneda 2-embedding}
\label{sec:Yoneda-2-motives}%
\medskip

%
Our next theorem shows that Mackey 2-motives could be simply thought of as a particular type of representable Mackey 2-functors:

\begin{Thm} \label{Thm:biYoneda_for_Mackey}
There is a `Yoneda' pseudo-functor
\[
\mathbb Z\Spanhat(\GG;\JJ)^{\op} \longrightarrow \bicMack(\GG;\JJ) \,, \quad
(G,e)\mapsto \mathbb Z\Spanhat((G,e),-) |_{\GG^\op}
\]
which is a biequivalence on its 1- and 2-full image.
\end{Thm}

\begin{proof}
Since $\mathbb Z\Spanhat$ is block-complete by construction, as explained in \Cref{Rem:add-bicat-Yoneda} we have a contravariant Yoneda embedding
\[
\mathbb Z\Spanhat(\GG;\JJ)^{\op} \longrightarrow \PsFun_\amalg(\mathbb Z\Spanhat(\GG;\JJ), \ICAdd)
\]
into the category of additive pseudo-functors. In full generality, this is a biequivalence onto its 1- and 2-full image. We may also replace $\ICAdd$ with $\ICADD$ without changing the conclusion. In order to prove the theorem, it now suffices to compose the above with the biequivalence of \eqref{eq:UP-ZSpanhat}.
\end{proof}

Thanks to this result, we may now add a new large family to our list of Mackey 2-functors in \Cref{Exa:Mackey-2-functors}:
\begin{Not}
\label{Not:Muniv}%
\index{muniv@$\Muniv$}%
\index{$muniv$@$\Muniv$ \, represented Mackey 2-functor}%
For every finite groupoid $G_0\in \GG$ and every idempotent 2-cell $e_0=e_0^2\colon \Id_G\Rightarrow \Id_G$ in $\mathbb Z\Spanhat(\GG;\JJ)$ (for instance $e_0=1_G$), we obtain from \Cref{Thm:biYoneda_for_Mackey} a canonical Mackey 2-functor
\[
\Muniv^{(G_0,e_0)}:= \mathbb Z\Spanhat(\GG;\JJ)((G_0, e_0),-)
\]
on~$\GG$.
\end{Not}

The simplest case of this construction, where we set $G_0=1$ the trivial group and $e_0=1_{G_0}$ the identity, is already quite interesting because it recovers ordinary spans of $G$-sets, repackaged into a single Mackey 2-functor:

\begin{Thm} \label{Thm:1-Mack-is-2-Mack}
Let $\GG = \groupoid$ and $\JJ=\faithful$. Then the Mackey 2-functor $\Muniv^1$ associated to the trivial group~$1=(1,1_1)$ is canonically given by
\[
\Muniv^1(G) \overset{\simeq}{\longleftarrow} (\widehat{G\sset})_+^\natural
\]
for every finite group~$G$, and if $u\colon H\to G$ is any group homomorphism then $u^*$ is induced by the restriction functor $u^*\colon G\sset \to H\sset$.
In other words, the (group-completed and idempotent-completed) ordinary span categories of $G$-sets extend, by varying~$G$, to a Mackey 2-functor.
\end{Thm}

\begin{proof}
Fix a finite group~$G$. As explained in \Cref{app:old-Mackey}, the category of $G$-sets can be viewed as a comma category of groupoids over~$G$. More precisely, we prove in \Cref{Cor:Gset_vs_gpdG} that the transport groupoid functor $G\ltimes (-)\colon G\sset \to \groupoid$ lifts to an equivalence
\[
G\sset\stackrel{\simeq}{\longrightarrow} \pih(\gpdG)
\]
between the category of $G$-sets and the truncated comma 2-category of faithful functors in $\GG=\gpd$ over~$G$. There is also an evident forgetful functor
\[
\Span(\GG;\JJ)(1,G) \longrightarrow \pih(\gpdG)
\]
defined as follows (on 1-cells)
\[
\vcenter{\hbox{
\xymatrix@R=10pt{
& H
 \ar[dd]_a
 \ar@/_2ex/[dl]
 \ar@/^2ex/[dr]^-{i}
 \ar@{}[ddl]|{\SEcell^{\exists!}}
 \ar@{}[ddr]|{\NEcell\alpha\;\;\;} & \\
1 && G \\
& H'
 \ar@/_2ex/[ur]_j
 \ar@/^2ex/[ul] &
}
}}
\quad\quad \mapsto \quad\quad
\vcenter{\hbox{
\xymatrix@R=10pt{
 H
 \ar[dd]_a
 \ar@/^2ex/[dr]^-{i}
 \ar@{}[ddr]|{\NEcell\alpha\;\;\;} & \\
& G.\!
\\
 H'
 \ar@/_2ex/[ur]_j &
}
}}
\]
This functor is clearly an isomorphism, since there is always exactly one way to complete any object $(H,i)$ or morphism $[a,\alpha]$ of $\pih (\gpdG)$ to an object or a morphism of $\Span(\GG;\JJ)(1,G)$. By combining the above two equivalences, taking 1-categories of spans, and group- and idempotent-completing, we get a canonical equivalence
\[
(\widehat{G\sset})_+^\natural
\stackrel{\sim}{\longrightarrow}
({\widehat{\pih(\gpdG)}})_+^\natural
\cong (\Spanhat(\GG;\JJ)(1,G))_+^\natural
\stackrel{\textrm{def.}}{=}\Muniv^1(G)
\]
as claimed. It is now straightforward to determine the restrictions~$u^*$, when $G$ varies.
\end{proof}

\bigbreak
\section{Presheaves over a Mackey 2-functor}
\label{sec:presheaves-2Mack}%
\medskip

We recall (\Cref{Rem:old-Burnside}) that there is an equivalence
\begin{equation} \label{eq:eq-ordinary-Mack-span}
\Mackey_\mathbb Z(G) \cong \Funplus((\widehat{G\sset})_+^\natural, \Ab)
\end{equation}
between the category of ordinary Mackey functors for the group $G$ and the additive functor category over spans of $G$-sets.
Thus $(\widehat{G\sset})_+^\natural$ can be identified with the category of \emph{representable} Mackey functors for~$G$.

Our next goal is to extend the result of \Cref{Thm:1-Mack-is-2-Mack} to show that we can also form a Mackey 2-functor by assembling the abelian categories of \emph{all} Mackey functors (not just the representable ones). For this we can use the following general construction:

\begin{Prop} \label{Prop:extend-2Mack-via-Yoneda}
Let $\cat S$ be a Mackey 2-functor on $(\GG;\JJ)$ such that each additive category $\cat S(G)$ is essentially small, and let $\cat A$ be a cocomplete abelian category. Then $G\mapsto \MM(G):=\Funplus(\cat S(G)^\op, \cat A)$ defines a Mackey 2-functor $\MM\colon \GG^\op\to \ICADD$ on $(\GG;\JJ)$, with 1- and 2-functoriality extended from that of $\cat S$ along the additive Yoneda embeddings $y_G\colon \cat S(G)\hookrightarrow \MM(G)$, $X\mapsto \cat S(G)(-,X)$.
\end{Prop}

\begin{proof}
Let us describe the structure of $\MM$ in some details.
Given a 1-cell $u\colon H\to G$ of~$\GG$, we can extend $u^*\colon \cat S(G)\to \cat S(H)$ to the functor categories by forming the left Kan extension of $y_H\circ u^*$ along~$y_G$:
\begin{align} \label{eq:basic-square-u*}
\vcenter { \hbox{
\xymatrix{
\cat S(G) \ar[d]_{u^*} \ar[r]^-{y_G} & \MM(G) \ar@{-->}[d]^{u^*} \ar@{}[dl]|{\SWcell\; \simeq} \\
\cat S(H) \ar[r]^-{y_H} & \MM(H)\,.\!\!
}
}}
\end{align}
Since $y_G$ is fully faithful, the canonical natural transformation $u^*y_G\Rightarrow y_Hu^*$ is an isomorphism. In fact, the extension is given by the formula
\begin{align} \label{eq:first-u*}
u^*M= \colim_{(X,y_G X\to M)} y_Hu^*X
\end{align}
where $M\in \MM(G)$, $Y\in \cat S(H)$, and the colimit is taken over the comma category of $y_G$ over~$M$. (We may even define the extension $u^*$ so that the square~\eqref{eq:basic-square-u*} commutes strictly.)

The extended functor $u^*\colon \MM(G)\to \MM(H)$ has a right adjoint $u_*$ given by restriction along $u^*\colon \cat S(G)^\op\to \cat S(H)^\op$, that is:
\begin{equation} \label{eq:right-adj-Kan}
u_*N = N\circ u^* \quad \textrm{for all } N\in \MM(H)\,.
\end{equation}

It follows also from the standard properties of Kan extensions that for every 2-cell $\alpha\colon u\Rightarrow v$ of~$\GG$, the natural transformation $\alpha^*\colon u^*\Rightarrow v^*$ given by the 2-functoriality of $\cat S$ extends uniquely between the extended functors $u^*$ and $v^*$ on~$\MM(G)$. This makes $\MM(-)$ into a 2-functor $\GG^\op\to \ICADD$. Let us verify that it satisfies the axioms of a Mackey 2-functor as given in \Cref{Def:Mackey-2-functor-intro}.

Clearly $\MM$ satisfies the additivity axiom (Mack\,\ref{Mack-1-intro}) because $\cat S$ does:
\begin{align*}
\MM(G_1\sqcup G_2)
&\;=\; \Fun_+(\cat S(G_1\sqcup G_2)^\op, \cat A) \\
&\;\cong\; \Fun_+(\cat S(G_1)^\op \oplus \cat S(G_2)^\op, \cat A) \\
&\;\cong\; \Fun_+(\cat S(G_1)^\op, \cat A) \oplus \Fun_+(\cat S(G_2)^\op, \cat A) \\
&\;=\; \MM(G_1)\oplus \MM( G_2) \,.
\end{align*}

For the existence of the adjoints (Mack\,\ref{Mack-2-intro}), let $i\in \JJ$. We have already seen that a right adjoint $i_*$ exists. We claim that in this case it is also \emph{left} adjoint to restriction~$i^*$. In order to see this, we first prove the remarkably simple formula
\begin{align} \label{eq:second-i^*}
i^*M \cong M \circ i_* \qquad (M\in \MM(G))
\end{align}
for the restriction functor. Indeed, for all $Y\in \cat S(H)$ we have:
\begin{align*}
(i^*M)(Y)
&\;=\; \left( \colim_{(X,y_G X\to M)} y_Hi^*X \right) (Y) && \textrm{ by }~\eqref{eq:first-u*} \\
&\;=\; \colim_{(X,y_G X\to M)} (y_Hi^*X) (Y) &&  \\
&\;\cong\; \colim_{(X,y_G X\to M)} \cat S(H)(Y, i^*X) \\
&\;\cong\; \colim_{(X,y_G X\to M)} \cat S(G)(i_*Y, X) && \textrm{ by } i_*=i_! \dashv i^* \textrm{ for } \cat S \\
& \;=\; \colim_{(X,y_G X\to M)} (y_G X)(i_*Y) && \\
& \;=\; \left( \colim_{(X,y_G X\to M)} y_G X \right) (i_*Y) && \\
&\;\cong\; M(i_*Y) &&
\end{align*}
which proves~\eqref{eq:second-i^*}.
We will also need the isomorphism
\begin{align} \label{eq:i_*-on-representables}
i_*\circ y_H \cong y_G \circ i_*
\end{align}
which follows from the computation (for $Y\in \cat S(H)$)
\begin{align*}
i_*(y_H(Y))
&\;\cong\; y_H(Y) \circ i^* && \textrm{by }~\eqref{eq:right-adj-Kan} \\
&\;=\; \cat S(H) (i^* (-), Y) && \\
&\;\cong\; \cat S(G) (-, i_*Y) && \textrm{by } i^*\dashv i_* \textrm{ for } \cat S \\
&\;=\; y_G (i_*(Y)) &&
\end{align*}
where we have now used the other adjunction, $i^*\dashv i_*$, for the Mackey 2-functor~$\cat S$. To find the claimed natural isomorphism
\[
\MM(G)(i_*N, M)\cong \MM(H)(N,i^*M)
\]
consider first the case where $M\in \MM(G)$ is arbitrary and $N= y_HY\in \MM(H)$ is representable. We compute:
\begin{align*}
\MM(G) (i_* y_H(Y), M)
& \;\cong\; \MM(G) (y_G i_*(Y), M) && \textrm{by }~\eqref{eq:i_*-on-representables} \\
& \;\cong\; M(i_* Y) && \textrm{by Yoneda} \\
& \;\cong\; (i^*M)(Y) && \textrm{by }~\eqref{eq:second-i^*} \\
& \;\cong\; \MM(H) (y_H Y , i^*M) && \textrm{by Yoneda.}
\end{align*}
This extends to arbitrary~$N$, because we can rewrite $N$ as a colimit of representable objects and because $i_*$ commutes with colimits, the latter being an immediate consequence of~\eqref{eq:right-adj-Kan}. This concludes the proof of (Mack\,\ref{Mack-2-intro}) and, incidentally, also of the ambidexterity axiom (Mack\,\ref{Mack-4-intro}).

To verify the Base-Change formulas (Mack\,\ref{Mack-3-intro}), consider the mates
\[
q_!\circ p^* \overset{\gamma_!}{\Longrightarrow} u^*\circ i_!
\qquadtext{and}
u^*\circ i_* \overset{(\gamma\inv)_*}{\Longrightarrow} q_*\circ p^*
\]
of the natural transformation $\MM(\gamma)=\gamma^*\colon p^*i^*\Rightarrow q^*u^*$ induced by a suitable iso-comma square of~$\GG$. By~\eqref{eq:basic-square-u*} and~\eqref{eq:i_*-on-representables}, the restriction and induction functors of $\MM$ --- and therefore also the adjunctions $i_*\dashv i^* \dashv i_*$ --- restrict to those of~$\cat S$ via the Yoneda embeddings. Since $\cat S$ is a Mackey 2-functor, it follows that the components of the two mates at all representable objects are invertible. By writing an arbitrary object as a colimit of representables, and because the source and target functors commute with colimits, we deduce that the component at an arbitrary object is also invertible.
This ends the proof that $\MM$ is a Mackey 2-functor.
\end{proof}

\begin{Rem}
The above proof illustrates the advantage of defining Mackey 2-functors (\Cref{Def:Mackey-2-functor}) \emph{without} requiring the Strict Mackey Formula~\Mack{7}. The latter could be difficult to verify when the adjoints are provided up to a series of natural isomorphisms. However, the Rectification \Cref{Thm:rectification} still guarantees that one can modify units and counits, if necessary, to guarantee \Mack{7} as well.
\end{Rem}

\begin{Cor} \label{Cor:Mackey-functors-abelian-example}
The assignment $G\mapsto \Mackey_\mathbb Z(G)$ extends to a Mackey 2-functor defined on $\GG=\gpd$ and $\JJ=\faithful$. The result also holds over any base ring~$\kk$ instead of the integers.
\end{Cor}

\begin{proof}
Apply \Cref{Prop:extend-2Mack-via-Yoneda} to $\cat S:= (G \mapsto (\widehat{G\sset})_+^\natural)^\op$, the dual (as in \Cref{Rem:dual-Mackey}) of the Mackey 2-functor of \Cref{Thm:1-Mack-is-2-Mack}, and combine the result with~\eqref{eq:eq-ordinary-Mack-span}. By replacing $\Ab$ with $\kk\MMod$ in the latter equation, we obtain the more general result for $\kk$-linear Mackey functors.
\end{proof}

\bigbreak
\section{Crossed Burnside rings and Mackey 2-motives}
\label{sec:B(G)}%
\medskip

Let $G$ be any finite group.
In this section we show that the endomorphism ring of $\Id_G$ in the bicategory of Mackey 2-motives can be identified with the so-called crossed Burnside ring $\xBur(G)$. This works not only integrally, but equally well over any commutative ring $\kk$ of coefficients; see \Cref{Thm:xBur-End} below.

More precisely, in the following $\kk\Spanhat$ will denote the bicategory of $\kk$-linear Mackey 2-motives $\kk \Spanhat:=\kk \Spanhat(\GG;\JJ)$ as defined in \Cref{Def:additive-2motives}, where we take $\GG:= \gpd$ to be the whole 2-category of finite groupoids, functors, and natural transformations, and $\JJ:=\faithful$ to be the class of all faithful functors.
(More generally, we could also take $(\GG,\JJ)$ to be any admissible pair as in \ref{Hyp:G_and_I_for_ZSpan}, as long as $\GG$ contains the group $G$ of interest as well as all of its subroups.)

Let us recall the other side of the equation:

\begin{Def}[{\cite{Yoshida97} \cite{Bouc03}}]
\label{Def:xBurk}
\index{$bcg$@$\xBur(G)$ \, crossed Burnside ring}%
\index{bcg@$\xBur(G)$}%
\index{crossed Burnside}%
The \emph{crossed Burnside ring} is the Grothendieck ring
\[ \xBur(G) := K_0( G\sset/G^c , \sqcup , \otimes) \]
of the comma category $G\sset/G^c$ of finite $G$-sets over~$G^c$, where the latter is the set $G$ equipped with the conjugation $G$-action. Its sum and multiplication are induced by coproducts (disjoint unions) and by the monoidal structure on $G\sset/G^c$
\[ (X\overset{a}{\to} G^c) \otimes (Y\overset{b}{\to} G^c) := ( X \times Y\xrightarrow{a\times b} G^c\times G^c \overset{\cdot}{\to} G^c)
\]
induced by the group multiplication. Note that, even for non-abelian groups, the ring multiplication is always commutative thanks to the natural isomorphism
\[
(X,a)\otimes (Y,b) \overset{\sim}{\longrightarrow} (Y,b) \otimes (X,a) ,  \quad (x,y) \mapsto (a(x)\cdot y, x)
\]
(this is in fact a braiding for the above monoidal structure, \cf \cite[(1.7)]{Yoshida97}).

The \emph{crossed Burnside $\kk$-algebra} is simply obtained by extension of scalars: $\xBurk(G):= \kk\otimes_\mathbb Z \xBur(G)$. It is a \emph{commutative} $\kk$-algebra.
\end{Def}

\begin{Rem}\label{Rem:Bur-vs-xBur}
\index{$bg$@$\Bur(G)$ \, Burnside ring}%
\index{bg@$\Bur(G)$}%
\index{Burnside ring}%
Recall that the ordinary \emph{Burnside ring}~$\Bur(G)$ of a finite group~$G$ is defined as the Grothendieck group of finite $G$-sets
\[
\Bur(G):=K_0(G\sset, \sqcup, \times) \,.
\]
(We use here the representation theorists' notation. Topologists typically denote this ring by~$A(G)$.) The \emph{Burnside $\kk$-algebra}~$\Burk(G)$ is defined as the $\kk$-linear version:
\[
\Burk(G):=\kk\otimes_\mathbb Z \Bur(G) \,.
\]
We see immediately that the forgetful additive tensor functor $G\sset/G^c\to G\sset$, $(a\colon X\to G^c)\mapsto X$, induces a surjective algebra morphism $\pi\colon \xBurk (G) \to \Bur(G)$. Indeed, there is an additive tensor functor $G\sset \to G\sset/G^c$ mapping a $G$-set $X$ to~$(1\colon X\to G^c)$ and inducing a morphism of algebras $\iota\colon \Burk (G) \to \xBurk(G)$. The latter is clearly a section of the former: $\pi \circ \iota = \Id_{\Burk(G)}$.
\end{Rem}

\begin{Rem} \label{Rem:Bur-vs-End}
The Burnside algebra is also the endomorphism ring of $G/G$ in the $\kk$-linear category of ordinary spans (see \cite{Bouc97} and \cite{Lewis80}):
\[
\Burk(G)
\cong \kk \otimes_\mathbb Z \End_{\widehat{G\sset}_+} (G/G)
= \End_{ (\kk\otimes \widehat{G\sset}_+)^\natural} (G/G) \,.
\]
Here $ (\kk\otimes \widehat{G\sset}_+)^\natural$ is the idempotent-completion (\ref{Rem:completions}\eqref{it:ic}) of the $\kk$-linearization of $\widehat{G\sset}$, the ordinary category of spans of $G$-sets; see \ref{Rem:old-Burnside}.
By tracing the $G$-set $G/G$ through the equivalence of \Cref{Thm:1-Mack-is-2-Mack}, we see easily that it maps to the span $1\gets G=G$, as an object in the Hom category~$\kk\Spanhat(\GG;\JJ)(1,G)$. Therefore we can identify $\Burk(G)$ inside of $\kk$-linear Mackey 2-motives $\kk\Spanhat$ as the endomorphism $\kk$-algebra of the 1-cell~$1\lto G=G$:
\begin{equation} \label{eq:first-B(G)}
\sigma\colon \Burk(G)\overset{\sim}{\longrightarrow} \End_{\kk \Spanhat(1,G)} (1\gets G = G)\,.
\end{equation}
\end{Rem}

Our main result here is a description of the endomorphism ring of the identity 1-cell $\Id_G\colon G\to G$ in $\kk\Spanhat(\GG;\JJ)$, for which we use the obvious abbreviated notation:
\[
\End_{\kk\Spanhat}(\Id_G):=\End_{\kk\Spanhat(G,G)}(\Id_G)\,.
\]

\begin{Thm} \label{Thm:xBur-End}
For every finite group $G$ and every commutative ring~$\kk$, there is a canonical isomorphism of $\kk$-algebras
\[
\sigma^c \colon \xBurk(G) \overset{\sim}{\longrightarrow} \End_{\kk \Spanhat} (\Id_G)
\]
between the crossed Burnside $\kk$-algebra (\Cref{Def:xBurk}) and the endomorphism algebra of the 1-cell~$\Id_G=(G=G=G)$ of $\kk \Spanhat$, the bicategory of $\kk$-linear Mackey 2-motives.  Moreover, the isomorphisms $\sigma^c$ and $\sigma$ (\Cref{Rem:Bur-vs-End}) identify the projection~$\pi$ and the inclusion~$\iota$ of \Cref{Rem:Bur-vs-xBur} with explicit maps~$\phi$ and~$\psi$, respectively, as in the following commutative diagram:
\[
\xymatrix{
\Burk \ar[rr]_-\simeq^-{\sigma} \ar[d]^{\iota} \ar@<-.5em>@/_2em/@{=}[dd]
&& \End_{\kk\Spanhat}\big(1 \!\gets\! G \!=\! G\big) \ar[d]_{\psi} \ar@<4em>@/^2em/@{=}[dd]
\\
\xBurk
 \ar[rr]^-{\sigma^c}_-\simeq \ar[d]^{\pi}
&& \End_{\kk\Spanhat}\big(G \!=\! G \!=\! G\big)
 \ar[d]_-{\phi}
\\
\Burk \ar[rr]_-\simeq^-{\sigma}
&& \End_{\kk\Spanhat}\big(1 \!\gets\! G \!=\! G\big)
}
\]
The map $\phi$ is induced by the whiskering functor
\[
(-) \circ (1\gets G=G) \colon \kk\Spanhat (G,G) \longrightarrow \kk\Spanhat (1,G)
\]
and the map $\psi$ sends
\begin{equation}
\label{eq:End(1<-G=G)-to-End(G=G=G)}%
\vcenter {
\xymatrix@C=4em{
1 \ar@{=}[d] \ar@{}[dr]|{\NEcell\,\exists!}
& G \ar[l] \ar@{=}[r] \ar@{}[dr]|{\SEcell\,\beta}
& G \ar@{=}[d]
\\
1 \ar@{}[dr]|{\SEcell\,\exists!} \ar@{=}[d]
& R \ar[u]_-{b} \ar[d]^-{a} \ar[l]_-{} \ar[r]^-{k} \ar@{}[dr]|{\NEcell\,\alpha}
& G \ar@{=}[d]
\\
1
& G \ar[l] \ar@{=}[r]
& G
}}
\qquad\overset{\underset{}{\psi}}\mapsto \qquad
\vcenter {
\xymatrix@C=4em{
G \ar@{=}[d] \ar@{}[dr]|{\NEcell\,\beta\inv}
& G \ar@{=}[l] \ar@{=}[r] \ar@{}[dr]|{\SEcell\,\beta}
& G \ar@{=}[d]
\\
G \ar@{}[dr]|{\SEcell\,\alpha\inv} \ar@{=}[d]
& R \ar[u]_-{b} \ar[d]^-{a} \ar[l]_-{k} \ar[r]^-{k} \ar@{}[dr]|{\NEcell\,\alpha}
& G \ar@{=}[d]
\\
G
& G \ar@{=}[l] \ar@{=}[r]
& G\,.\!\!
}}
\end{equation}
on representative diagrams of groupoids.
\end{Thm}

\begin{Rem}
The map $\phi$ induced by whiskering is even easier to describe in terms of representative diagrams. It maps
\begin{equation}
\label{eq:End(G=G=G)-to-End(1<-G=G)}%
\vcenter {
\xymatrix@C=4em{
G \ar@{=}[d] \ar@{}[dr]|{\NEcell\,\beta_1}
& G \ar@{=}[l] \ar@{=}[r] \ar@{}[dr]|{\SEcell\,\beta_2}
& G \ar@{=}[d]
\\
G \ar@{}[dr]|{\SEcell\,\alpha_1} \ar@{=}[d]
& R \ar[u]_-{b} \ar[d]^-{a} \ar[l]_-{w} \ar[r]^-{k} \ar@{}[dr]|{\NEcell\,\alpha_2}
& G \ar@{=}[d]
\\
G
& G \ar@{=}[l] \ar@{=}[r]
& G
}}
\qquad\overset{\underset{}{\phi}}\mapsto \qquad
\vcenter {
\xymatrix@C=4em{
1 \ar@{=}[d] \ar@{}[dr]|{\NEcell\,\exists!}
& G \ar[l] \ar@{=}[r] \ar@{}[dr]|{\SEcell\,\beta_2}
& G \ar@{=}[d]
\\
1 \ar@{}[dr]|{\SEcell\,\exists!} \ar@{=}[d]
& R \ar[u]_-{b} \ar[d]^-{a} \ar[l]_-{} \ar[r]^-{k} \ar@{}[dr]|{\NEcell\,\alpha_2}
& G \ar@{=}[d]
\\
1
& G \ar[l] \ar@{=}[r]
& G\,.\!\!
}}
\end{equation}
that is, it forgets the left-hand side of the double span.
\end{Rem}

In order to prove the theorem, we provide sufficiently explicit descriptions of the two rings. For the crossed Burnside ring, this is a well-known alternative picture:

\begin{Lem} \label{Lem:xBur-presentation}
The crossed Burnside $\kk$-algebra~$\xBurk(G)$ has the following presentation. As a $\kk$-module, it is free and generated by the (finitely many) $G$-conjugation classes
\[
[H,a]_G \quad \quad (\textrm{for } H\leq G, a \in \mathrm C_G(H))
\]
of pairs $(H,a)$ where $H\leq G$ is a subgroup, $a$ is an element of the centralizer of~$H$ in~$G$; two such pairs $(H,a)$ and $(K,b)$ are $G$-conjugate iff there exists some $g\in G$ with $H={}^{g\!}K$ and $a = {}^g b$.
The multiplication is defined by the formula
\begin{equation} \label{eq:product-formula-xBurk}
[K,b]_G \cdot [H,a]_G  = \sum_{[g] \in K \backslash G / H}  [K \cap {}^{g\!} H, b \cdot {}^g a ]_G
\end{equation}
on basis elements, where $g$ runs through a full set of representatives for the double cosets $K \backslash G / H$ as usual. Specifically, the element of~$\xBurk(G)$ corresponding to~$[H,a]_G$ is the isomorphism class of the object $G/H\to G^c, xH\mapsto {}^xa$, in $G\sset/G^c$.
\end{Lem}

\begin{proof}
This is straightforward and is explained in \cite[\S2.2]{Bouc03}.
\end{proof}

It is now very easy to connect our two rings.

\begin{Def} \label{Def:iso-map}
Let us denote by
\[
\sigma^c \colon \xBurk(G) \longrightarrow \End_{\kk\Spanhat} (\Id_G)
\]
(for `spanification') the map sending the basis element $[H,a]_G$ of the crossed Burnside algebra, as in \Cref{Lem:xBur-presentation}, to the 2-cell $\Id_G \Rightarrow \Id_G$ in $\kk\Spanhat$ represented by the following (vertical) double span diagram in groupoids
\begin{equation} \label{eq:loop-generator}
\sigma^c(H,a)=\qquad
\vcenter {\hbox{
\xymatrix{
G  & G \ar@{=}[l] \ar@{=}[r] & G  \\
G \ar@{=}[u] \ar@{=}[d] & H \ar[u]_i \ar[l]_-i \ar[r]^-i \ar[d]^i & G \ar@{=}[u] \ar@{=}[d] \\
G \ar@{}[ur]|{\SEcell \; \gamma_a} & G \ar@{=}[l] \ar@{=}[r] & G
}
}}
\end{equation}
where $i\colon H\to G$ is the inclusion homomorphism, $\gamma_a\colon i\Rightarrow i$ is the natural transformation with (sole) component $a \in G$, and the other three squares are commutative. Note that, for an element $a\in G$, the condition $a\in \mathrm C_G(H)$ is precisely the naturality of $\gamma_a\colon i\Rightarrow i\colon H\to G$, hence $\sigma^c$ makes sense. One easily verifies that it is well-defined: a $G$-conjugation $(H,a)={}^{g\!}(K,b)$ yields an equality of 2-cells
\[
\vcenter {\hbox{
\xymatrix{
&& K
  \ar@/_3ex/[dll]_j \ar@/^3ex/[ddl]^j
   \ar@{}[dll]|{\SEcell \gamma_g}
    \ar@{}[ddl]|(.6){\SEcell \gamma_g^{-1}} \\
G
\ar@{=}[d]  &
 H
  \ar@{}[dl]|{\SEcell \;\gamma_a }
  \ar@{<-}[ur]_-{c_g}^-{\simeq}
   \ar[l]_-i
    \ar[d]^i & \\
G & G \ar@{=}[l] &
}
}}
\quad  = \quad
\vcenter{ \hbox{
\xymatrix{
&& K \ar@/_3ex/[dll]_j \ar@/^3ex/[ddl]^j \ar@{}[ddll]|{\SEcell\,\gamma_b} \\
G \ar@{=}[d] && \\
G & G \ar@{=}[l] &
}
}}
\]
in groupoids, where $j\colon K\to G$ is the inclusion homomorphism and the conjugation 1-cell $c_g\colon K\isoto H$ and 2-cell $\gamma_g\colon j\Rightarrow i c_g$ are as in~\Cref{Not:Fun} and \Cref{Rem:not_groups}.
This data yields an isomorphism between the diagrams $\sigma^c(H,b)$ and~$\sigma^c(K,b)$; see details in~\Cref{Rem:2-cells-in-Spanhat}.
\end{Def}

The next lemma already shows that $\sigma^c$ is an isomorphism of $\kk$-modules.

\begin{Lem} \label{Lem:freegeneration}
As a $\kk$-module, $\End_{\kk\Spanhat}(\Id_G)$ is freely generated by the equivalence classes of diagrams of the form \eqref{eq:loop-generator}, indexed by the set $\{[H,a]_G \mid H\leq G , a\in \mathrm C_G(H)\}$, where $[H,a]_G$ denote the $G$-conjugation classes as in \Cref{Lem:xBur-presentation}.
\end{Lem}

\begin{proof}
In any 2-cell $\Id_G\Rightarrow \Id_G$ in the semi-additive bicategory $\Spanhat(G,G)$
\begin{equation}
\label{eq:monodromy}%
\vcenter {
\xymatrix@C=3em@R=2em{
G \ar@{=}[d] \ar@{}[dr]|{\NEcell\,\beta_1}
& G \ar@{=}[l] \ar@{=}[r] \ar@{}[dr]|{\SEcell\,\beta_2}
& G \ar@{=}[d]
\\
G \ar@{}[dr]|{\SEcell\,\alpha_1} \ar@{=}[d]
& R \ar[u]_-{b} \ar[d]^-{a} \ar[l]_-{w} \ar[r]^-{k} \ar@{}[dr]|{\NEcell\,\alpha_2}
& G \ar@{=}[d]
\\
G
& G \ar@{=}[l] \ar@{=}[r]
& G
}}
\end{equation}
we can use \Cref{Prop:2-cells-of-Spanhat} to replace every 1-cell~$a,b,k,w$ by our favorite one, say~$k$. However, one needs to decide \emph{which} 2-cell to use to apply \Cref{Prop:2-cells-of-Spanhat}\,\eqref{it:iso-2-cells-2}. Each of the three replacements (of~$a,b$ and $w$) allows us to replace \emph{one} 2-cell among $\alpha_1,\alpha_2,\beta_1,\beta_2$ by an identity. However, the fourth one will survive. Specifically in the 2-cell~\eqref{eq:monodromy}, we replace the bottom 1-cell~$a$ by~$k$ by using~$\alpha_2$ and we replace the other two 1-cells $b$ and~$w$ by~$k$ as well but by using $\beta_2$ and $\beta_2\beta_1$ respectively. Then we obtain the isomorphic 2-cell
\begin{equation}
\label{eq:loop-monodromy}%
\vcenter {
\xymatrix@C=3em@R=2em{
G \ar@{=}[d] \ar@{}[dr]|{\NEcell\,\id}
& G \ar@{=}[l] \ar@{=}[r] \ar@{}[dr]|{\SEcell\,\id}
& G \ar@{=}[d]
\\
G \ar@{}[dr]|{\SEcell\,\Displ\gamma} \ar@{=}[d]
& R \ar[u]_-{k} \ar[d]^-{k} \ar[l]_-{k} \ar[r]^-{k} \ar@{}[dr]|{\NEcell\,\id}
& G \ar@{=}[d]
\\
G
& G \ar@{=}[l] \ar@{=}[r]
& G
}}
\end{equation}
where the automorphism $\gamma:=\alpha_2\alpha_1\beta_1\inv\beta_2\inv\colon k\isoEcell k$ is the `monodromy' in~\eqref{eq:monodromy}. In particular the latter is not necessarily the identity. Decomposing $R$ into connected components, it is a straightforward exercise to see that the abelian monoid $\End_{\Spanhat(G,G)}(\Id_G)$ is free over 2-cells of the form $\sigma^c(H,a)$ as in~\eqref{eq:loop-generator}; \cf \Cref{Exa:ordinary_span_is_sad}. It remains to observe that the notion of $G$-conjugation relation on pairs~$(H,a)$ as in \Cref{Lem:xBur-presentation} is exactly the same as the isomorphism relation between the corresponding 2-cells $\sigma^c(H,a)$ as in~\eqref{eq:loop-generator}. The result then follows by $\kk$-linearization.
\end{proof}

\begin{Rem}
\label{Rem:sigmac-canonical}%
We can also describe the ring homomorphism $\sigma^c\colon \xBurk(G)\to \End_{\kk\Spanhat}(\Id_G)$ from the original definition of~$\xBurk(G)$, \ie by defining it on arbitrary $G$-sets over~$G^c$, not only the ones corresponding to subgroups (orbits). Recall from Remarks~\ref{Rem:trans-gpd} and~\ref{Rem:trans-gpd-plus} the transport groupoid $G\ltimes X$ associated to a $G$-set~$X$, which comes with a 1-morphism $\pi_X\colon G\ltimes X\to G$ in~$\gpd$. To a morphism of $G$-sets~$f\colon X\to G^c$ we can associate a 2-automorphism~$\gamma_f$ of the 1-cell $\pi_X\colon G\ltimes X\to G$ by the formula $(\gamma_f)_x=f(x)$, for all~$x\in X=\Obj(G\ltimes X)$. Here, $f(x)\in G$ is viewed in~$\Aut_G(\pi_X(x))=G$ in the one-object groupoid~$G$. The assignment
\[
(X\oto{f}G^c)
\qquad\mapsto \qquad
\vcenter {
\xymatrix@C=3em@R=2em{
G \ar@{=}[d] \ar@{}[dr]|{\NEcell\,\id}
& G \ar@{=}[l] \ar@{=}[r] \ar@{}[dr]|{\SEcell\,\id}
& G \ar@{=}[d]
\\
G \ar@{}[dr]|{\SEcell\,\Displ\gamma_f} \ar@{=}[d]
& G\ltimes X \ar[u]_-{\pi_X} \ar[d]^-{\pi_X} \ar[l]_-{\pi_X} \ar[r]^-{\pi_X} \ar@{}[dr]|{\NEcell\,\id}
& G \ar@{=}[d]
\\
G
& G \ar@{=}[l] \ar@{=}[r]
& G
}}
\]
yields the choice-free description of $\sigma^c$.

It remains to check that $\sigma^c\colon\xBurk(G)\to \End_{\kk\Spanhat}(\Id_G)$ identifies the two multiplicative structures, which is precisely the content of the next lemma. Again, we could describe this property without using the particular basis but the multiplication in $\xBurk(G)$ might be more familiar to some readers in terms of the basis, so we present it that way, in a small breach of our no-double-cosets philosophy.
\end{Rem}

\begin{Lem} \label{Lem:compare-products}
The map $\sigma^c$ of \Cref{Def:iso-map} sends the product of $\xBurk(G)$ to the vertical composition of 2-morphisms in $\End_{\kk\Spanhat}(\Id_G)$.
\end{Lem}

\begin{proof}
The computation will be done in strings, after a little preparation involving an iso-comma square of~$\gpd$ (see \Cref{Exa:comma-in-Cat}). At the end of the day, we must vertically compose in the semi-additive $\Spanhat$ two 2-morphisms of the form \eqref{eq:loop-generator}, say for two pairs $(H,a)$ and~$(K,b)$, with $H,K\le G$ and $a\colon G/H\to G^c$ and $b\colon G/K\to G^c$ in~$G\sset$ (that is $a\in \mathrm C_G(H)$ and $b\in\mathrm C_G(K)$), and then compare the result with \eqref{eq:product-formula-xBurk}. Recall that the vertical composition of $\sigma^c(H,a)$ and~$\sigma^c(K,b)$ as in~\eqref{eq:loop-generator} involves, in the `middle' column, the following construction in~$\gpd$:
\begin{equation}
\label{eq:iso-comma-sigma^c}%
\vcenter{\xymatrix@R=10pt{
&& (i/j) \ar[dl]_-{p} \ar[dr]^-{q}
\\
& H \ar[dl]^-i \ar[dr]_-{i} \ar@{}[rr]|{\isocell{\gamma}} && K \ar[dr]_-{j} \ar[dl]^-{j}
\\
G && G && G
}}
\end{equation}
(here we write $i:=\incl_H^G \colon H\to G$ and $j:=\incl_K^G\colon K\to G$ for the two inclusion homomorphisms), with an iso-comma in the middle.
Then, precisely as in the proof of \Cref{Prop:test-on-gps} (for $f:=j$), we can easily decompose its top object $(i/j)$ by the equivalence
\begin{equation}
\label{eq:w-sigma^c}%
w\colon \coprod_{[g] \in K \backslash G/ H} K \cap {}^{g\!}H \overset{\sim}{\too} (i/j)
\end{equation}
which, on the $g$-component, sends the unique object $\bullet$ of $K \cap \,{}^{g\!}H$ to the object $(\bullet, \bullet, g)$ of $(i/j)$, and sends $k\in K \cap \,{}^{g\!}H$ to the morphism $(k^g, k)\colon (\bullet, \bullet, g) \to (\bullet, \bullet, g)$.
The functor $w$ is the unique one fitting in the following diagram on the left
\begin{equation} \label{eq:two-2-cells}
\vcenter{\xymatrix@C=14pt@R=20pt{
& { \coprod_{[g]} K\cap {}^{g\!}H }
 \ar@/_2ex/[ddl]_{\coprod_{[g]} \incl^g}
  \ar[d]^w_{\simeq}
    \ar@/^2ex/[ddr]^{\coprod_{[g]} \incl} & \\
& (i/j) \ar[dl]_-{p} \ar@{ >->}[dr]^-{q}
 \ar@{}[dd]|(.5){\isocell{\gamma}}
\\
H \ar@{ >->}[dr]_-{i}
&& K \ar[dl]^-{j}
\\
&G
}}
\quad = \quad
\vcenter{\xymatrix@C=14pt@R=20pt{
& { \coprod_{[g]} K\cap {}^{g\!}H }
 \ar@/_2ex/[ddl]_{\coprod_{[g]} \incl^g}
    \ar@/^2ex/[ddr]^{\coprod_{[g]} \incl}
     \ar@{}[ddd]|(.5){\isocell{\coprod_{[g]}\gamma_g}} & \\
&
\\
H \ar@{ >->}[dr]_-{i}
&& K \ar[dl]^-{j}
\\
&G
}}
\end{equation}
in such a way that the two triangles commute and the whiskered natural transformation $\gamma \circ w$ is equal to $\coprod_{[g]}\gamma_g$, that is, its component at $[g]$ is the 2-morphism
\[
 \gamma_g\colon (\incl^G_{K\cap {}^{g\!}H})^g = i \circ (\incl_{K\cap {}^{g\!}H}^{{}^{g\!}H})^g \Rightarrow j \circ \incl^K_{K \cap {}^{g\!}H} = \incl^G_{K\cap {}^{g\!}H}
\]
between 1-morphisms $K\cap {}^{g\!}H \to G$, associated to conjugation by~$g$ as in~\Cref{Not:Fun}.

In the string notation of~\Cref{sec:string-presentation}, the map $\sigma^c$ of \Cref{Def:iso-map} sends $[H,a]_G$ to the diagram
\[
\vcenter {\hbox{
\psfrag{I}[Bc][Bc]{\scalebox{1}{\scriptsize{$i$}}}
\psfrag{A}[Bc][Bc]{\scalebox{1}{\scriptsize{$\gamma_a$}}}
\psfrag{G}[Bc][Bc]{\scalebox{1}{\scriptsize{$G$}}}
\psfrag{H}[Bc][Bc]{\scalebox{1}{\scriptsize{$H$}}}
\includegraphics[scale=.4]{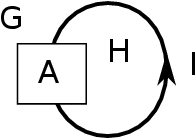}
}}
\]
\ie an anti-clockwise-oriented loop labeled with~$i$, carrying a box labeled with~$\gamma_a$, and separating two plane regions labeled with $G$ and~$H$.
The vertical composite $\sigma([K,b]_G) \circ \sigma([H,a]_G)$ can now be computed as follows (see explanations below; the shaded areas indicate where an interesting change is about to happen):
\[
\vcenter {\hbox{
\psfrag{I}[Bc][Bc]{\scalebox{1}{\scriptsize{$i$}}}
\psfrag{J}[Bc][Bc]{\scalebox{1}{\scriptsize{$j$}}}
\psfrag{A}[Bc][Bc]{\scalebox{1}{\scriptsize{$\gamma_a$}}}
\psfrag{B}[Bc][Bc]{\scalebox{1}{\scriptsize{$\gamma_b$}}}
\psfrag{G}[Bc][Bc]{\scalebox{1}{\scriptsize{$G$}}}
\psfrag{H}[Bc][Bc]{\scalebox{1}{\scriptsize{$H$}}}
\psfrag{K}[Bc][Bc]{\scalebox{1}{\scriptsize{$K$}}}
\includegraphics[scale=.4]{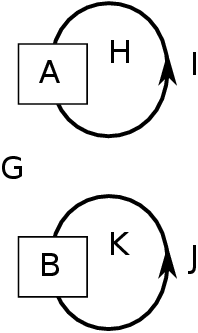}
}}
\; \overset{(1)}{=} \;\;
\vcenter {\hbox{
\psfrag{I}[Bc][Bc]{\scalebox{1}{\scriptsize{$i$}}}
\psfrag{J}[Bc][Bc]{\scalebox{1}{\scriptsize{$j$}}}
\psfrag{A}[Bc][Bc]{\scalebox{1}{\scriptsize{$\gamma_a$}}}
\psfrag{B}[Bc][Bc]{\scalebox{1}{\scriptsize{$\gamma_b$}}}
\psfrag{G}[Bc][Bc]{\scalebox{1}{\scriptsize{$G$}}}
\psfrag{H}[Bc][Bc]{\scalebox{1}{\scriptsize{$H$}}}
\psfrag{K}[Bc][Bc]{\scalebox{1}{\scriptsize{$K$}}}
\includegraphics[scale=.4]{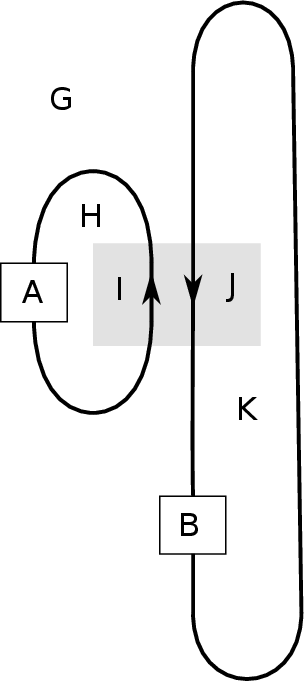}
}}
\;\; \overset{(2)}{=} \;\;
\vcenter {\hbox{
\psfrag{I}[Bc][Bc]{\scalebox{1}{\scriptsize{$i$}}}
\psfrag{J}[Bc][Bc]{\scalebox{1}{\scriptsize{$j$}}}
\psfrag{A}[Bc][Bc]{\scalebox{1}{\scriptsize{$\gamma_a$}}}
\psfrag{B}[Bc][Bc]{\scalebox{1}{\scriptsize{$\gamma_b$}}}
\psfrag{G}[Bc][Bc]{\scalebox{1}{\scriptsize{$G$}}}
\psfrag{H}[Bc][Bc]{\scalebox{1}{\scriptsize{$H$}}}
\psfrag{K}[Bc][Bc]{\scalebox{1}{\scriptsize{$K$}}}
\psfrag{P}[Bc][Bc]{\scalebox{1}{\scriptsize{$p$}}}
\psfrag{Q}[Bc][Bc]{\scalebox{1}{\scriptsize{$q$}}}
\psfrag{S}[Bc][Bc]{\scalebox{1}{\scriptsize{$(i/j)$}}}
\includegraphics[scale=.4]{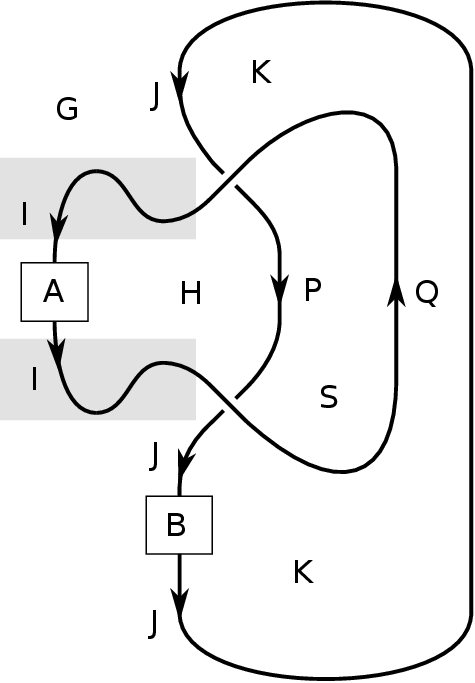}
}}
\;\; \overset{(3)}{=} \;\;
\vcenter {\hbox{
\psfrag{I}[Bc][Bc]{\scalebox{1}{\scriptsize{$i$}}}
\psfrag{J}[Bc][Bc]{\scalebox{1}{\scriptsize{$j$}}}
\psfrag{A}[Bc][Bc]{\scalebox{1}{\scriptsize{$\gamma_a$}}}
\psfrag{B}[Bc][Bc]{\scalebox{1}{\scriptsize{$\gamma_b$}}}
\psfrag{G}[Bc][Bc]{\scalebox{1}{\scriptsize{$G$}}}
\psfrag{H}[Bc][Bc]{\scalebox{1}{\scriptsize{$H$}}}
\psfrag{K}[Bc][Bc]{\scalebox{1}{\scriptsize{$K$}}}
\psfrag{P}[Bc][Bc]{\scalebox{1}{\scriptsize{$p$}}}
\psfrag{Q}[Bc][Bc]{\scalebox{1}{\scriptsize{$q$}}}
\psfrag{S}[Bc][Bc]{\scalebox{1}{\scriptsize{$(i/j)$}}}
\includegraphics[scale=.4]{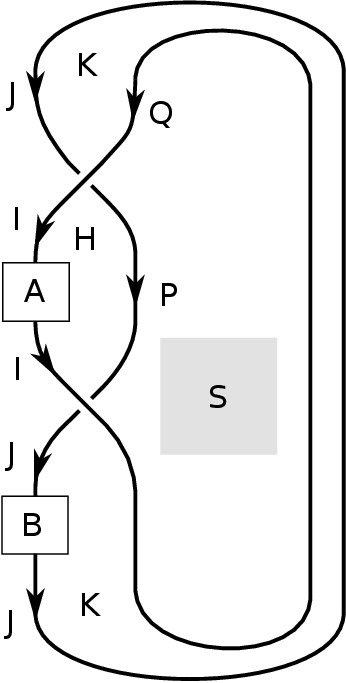}
}}
\]
\vspace{.7cm}
\[
\overset{(4)}{=}
\vcenter {\hbox{
\psfrag{I}[Bc][Bc]{\scalebox{1}{\scriptsize{$i$}}}
\psfrag{J}[Bc][Bc]{\scalebox{1}{\scriptsize{$j$}}}
\psfrag{A}[Bc][Bc]{\scalebox{1}{\scriptsize{$\gamma_a$}}}
\psfrag{B}[Bc][Bc]{\scalebox{1}{\scriptsize{$\gamma_b$}}}
\psfrag{G}[Bc][Bc]{\scalebox{1}{\scriptsize{$G$}}}
\psfrag{H}[Bc][Bc]{\scalebox{1}{\scriptsize{$H$}}}
\psfrag{K}[Bc][Bc]{\scalebox{1}{\scriptsize{$K$}}}
\psfrag{P}[Bc][Bc]{\scalebox{1}{\scriptsize{$p$}}}
\psfrag{Q}[Bc][Bc]{\scalebox{1}{\scriptsize{$q$}}}
\psfrag{W}[Bc][Bc]{\scalebox{1}{\scriptsize{$w$}}}
\psfrag{S}[Bc][Bc]{\scalebox{1}{\scriptsize{$\coprod_{[g]} \!K \! \cap \! {}^g\!H$}}}
\includegraphics[scale=.4]{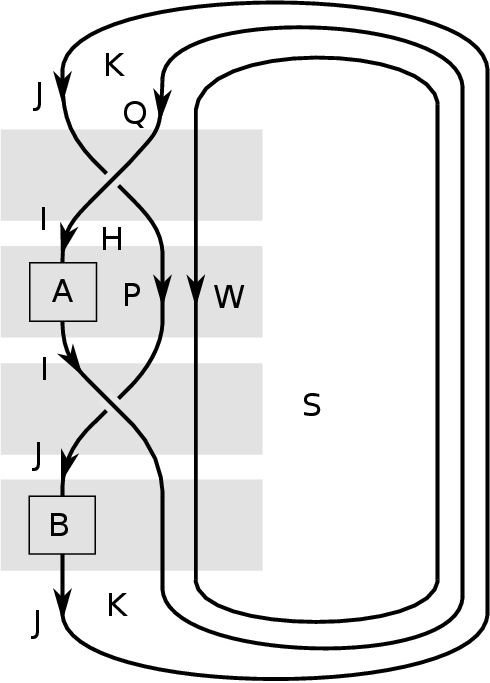}
}}
\overset{(5)}{=}  \coprod_{[g] \in K \backslash G / H}
\vcenter {\hbox{
\psfrag{I}[Bc][Bc]{\scalebox{1}{\scriptsize{$i$}}}
\psfrag{J}[Bc][Bc]{\scalebox{1}{\scriptsize{$j$}}}
\psfrag{A}[Bc][Bc]{\scalebox{1}{\scriptsize{$\gamma_a$}}}
\psfrag{B}[Bc][Bc]{\scalebox{1}{\scriptsize{$\gamma_b$}}}
\psfrag{C}[Bc][Bc]{\scalebox{1}{\scriptsize{$\gamma_{g^{-1}}$}}}
\psfrag{D}[Bc][Bc]{\scalebox{1}{\scriptsize{$\gamma_{g}$}}}
\psfrag{G}[Bc][Bc]{\scalebox{1}{\scriptsize{$G$}}}
\psfrag{H}[Bc][Bc]{\scalebox{1}{\scriptsize{$H$}}}
\psfrag{K}[Bc][Bc]{\scalebox{1}{\scriptsize{$K$}}}
\psfrag{U}[Bc][Bc]{\scalebox{1}{\scriptsize{$\;\incl$}}}
\psfrag{V}[Bc][Bc]{\scalebox{1}{\scriptsize{$\;\incl$}}}
\psfrag{W}[Bc][Bc]{\scalebox{1}{\scriptsize{$\;\;\;\incl^g$}}}
\psfrag{T}[Bc][Bc]{\scalebox{1}{\scriptsize{$K \! \cap \! {}^g\!H$}}}
\includegraphics[scale=.4]{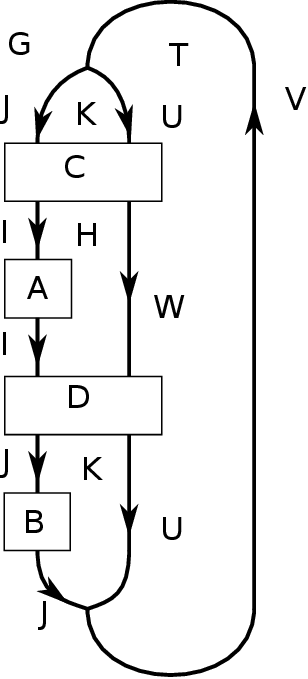}
}}
\!\!\! \overset{(6)}{=}  \coprod_{[g] \in K \backslash G / H} \!\!\!\!
\vcenter {\hbox{
\psfrag{I}[Bc][Bc]{\scalebox{1}{\scriptsize{$\incl$}}}
\psfrag{A}[Bc][Bc]{\scalebox{1}{\scriptsize{${\Displ\gamma}_{b\cdot {}^g\!a}$}}}
\psfrag{G}[Bc][Bc]{\scalebox{1}{\scriptsize{$G$}}}
\psfrag{H}[Bc][Bc]{\scalebox{1}{\scriptsize{$K \! \cap \! {}^g\!H$}}}
\includegraphics[scale=.4]{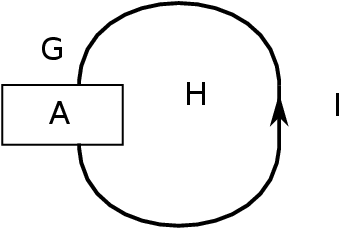}
}}
\]
\vspace{.3cm}

The moving around at~(1) is simply by the interchange law and insertion of identities. At~(2), we use one of the two `pull-over' relations for the iso-comma square~\eqref{eq:iso-comma-sigma^c}, see \eqref{eq:string-gluing}. The crossing ~$\vcenter { \hbox{
\includegraphics[scale=.4]{anc/just-crossing.eps}
}}$ and uncrossing ~$\vcenter { \hbox{
\includegraphics[scale=.4]{anc/just-uncrossing.eps}
}}$ stand for $\gamma$ and~$\gamma^{-1}$, respectively.
The straightening of the $i$-strands at~(3) is by two of the zig-zag relations for the adjunctions $i_* \dashv i^* \dashv i_*$, see~\eqref{eq:zigzag-relations}. The insertion of the $w$-loop at~(4) is the rewriting of the (invisible) identity 2-morphism of the 1-morphism $\Id_{(i/j)}$ as the composite $\varepsilon \circ \eta$, where $\varepsilon$ and $\eta$ are the counit and unit of the adjoint equivalence $w^* \dashv w_*=(w^*)^{-1}$ from \eqref{eq:w-sigma^c}.  At~(5), we use $jqw=\coprod_{[g]}\incl_{K\cap{}^{g\!}H}^G$ as well as the other identities displayed in \eqref{eq:two-2-cells}, in order to rewrite the four shaded 2-morphisms.  Finally, for (6) we simply compute the composite 2-morphisms in~$\gpd$. As $\coprod$ is the sum in $\Spanhat$, the above computation yields exactly the formula~\eqref{eq:product-formula-xBurk} as was to be shown.
\end{proof}

We leave to the reader the straightforward determination of $\varphi$ and $\psi$ given in~\eqref{eq:End(1<-G=G)-to-End(G=G=G)} and~\eqref{eq:End(G=G=G)-to-End(1<-G=G)}. This finishes the proof of \Cref{Thm:xBur-End}.
\qed

\bigbreak
\section{Motivic decompositions of Mackey 2-functors}
\label{sec:decomposition}%
\medskip

We conclude by explaining how the crossed Burnside ring~$\xBur(G)$ of a finite group~$G$ (\Cref{Def:xBurk}) acts on the additive category~$\MM(G)$, for any Mackey 2-functor~$\MM$. As a consequence, ring decompositions of~$\xBur(G)$ induce decompositions of the category~$\MM(G)$. Everything we say about $\xBur(G)$ applies similarly with the ordinary Burnside ring~$\Bur(G)$, via the inclusion~$\iota\colon \Bur(G)\hook \xBur(G)$ of \Cref{Rem:Bur-vs-xBur}, and also $\kk$-linearly over any commutative ring~$\kk$ after the evident notational changes.

\begin{Prop}
\label{Prop:B(G)-acts-on-M(G)}%
Let $\MM\colon\GG^\op\to \ICADD$ be a Mackey 2-functor and consider $\widehat{\MM}\colon \bbZ\Spanhat\to\ICADD$ the associated realization of Mackey 2-motives via~$\MM$ (\Cref{Def:realization}). Then $\widehat{\MM}$ induces ring homomorphisms for every groupoid~$G\in \GG_0$
\[
\End_{\bbZ\Spanhat(\GG;\JJ)(G,G)}(\Id_G) \xrightarrow{\widehat{\MM}} \End_{\Funplus(\MM(G),\MM(G))}(\Id_{\MM(G)})
\]
between rings of endomorphisms of the identity 1-cells. In particular, precomposition with the isomorphism $\sigma^c\colon \xBur(G)\to \End_{\mathbb Z \Spanhat(G,G)}(\Id_G)$ of \Cref{Thm:xBur-End} yields a ring homomorphism
\begin{equation}
\label{eq:B(G)-vs-End(Id)}%
\xBur(G) \xrightarrow{\widehat{\MM}\sigma^c} \End_{\Funplus(\MM(G),\MM(G))}(\Id_{\MM(G)})\,.
\end{equation}
\end{Prop}

\begin{proof}
This is simply the 2-functoriality of~$\widehat{\MM}$ on the endomorphism ring of the 1-cell $\Id_G\colon G\to G$ in~$\bbZ\Spanhat$, whose image is~$\widehat{\MM}(\Id_G)=\Id_{\widehat{\MM} (G)}=\Id_{\MM (G)}$.
\end{proof}

\begin{Rem}
\label{Rem:action}%
We can think of the ring homomorphism $\rho:=\widehat{\MM} \sigma^c$ of~\eqref{eq:B(G)-vs-End(Id)} as an \emph{action} of the ring~$\xBur(G)$ on the category~$\MM(G)$, since every $a\in \xBur(G)$ yields a natural transformation $\rho(a)=\widehat{\MM}(\sigma^c(a))\colon \Id_{\MM(G)}\to \Id_{\MM(G)}$ and therefore an actual endomorphism $\rho(a)_x\colon x\to x$ for every object~$x\in \MM(G)$.
In fact, in this way $\MM(G)$ becomes a \emph{$\xBur(G)$-linear category}, a.k.a.\ a category enriched in $\xBur(G)$-modules: the Hom sets are $\xBur(G)$-modules and composition is $\xBur(G)$-bilinear.

This is a very general fact: If $A$ is any commutative ring and $\cat{A}$ is any additive category, to give an $A$-enrichment on $\cat{A}$ is the same thing as to give a ring homomorphism $\rho\colon A\to \End_{\Fun(\cat{A}, \cat{A})}(\Id_\cat{A})$. Given~$\rho$, we obtain the actions by $a\cdot f := \rho(a)_y \circ f = f \circ \rho(a)_x$ (for $a\in A$ and $f\in \cat{A}(x,y)$). Conversely, given the $A$-enrichment we recover the ring homomorphism $\rho$ by letting $\rho(a)$ ($a\in A$) be the natural transformation with components $\rho(a)_x:= a\cdot \id_x$ for all $x\in \Obj \cat{A}$.
\end{Rem}

\begin{Cor}
\label{Cor:decomposition}%
With notation as in \Cref{Prop:B(G)-acts-on-M(G)}, any ring decomposition
\[ \xBur(G)\cong B_1\times \cdots \times B_n \]
yields a corresponding decomposition of the additive category~$\MM(G)$ as
\[
\MM(G)\cong \cat{N}_1\oplus \cdots \oplus \cat{N}_n
\]
in such a way that for every $i\neq j$, the ring $B_i$ acts as zero on~$\cat{N}_{j}$ and acts on~$\cat{N}_i$ via the homomorphism $\widehat{\MM}\sigma^c$ of~\eqref{eq:B(G)-vs-End(Id)}.
\end{Cor}

\begin{proof}
The decomposition $1=e_1+\ldots + e_n$ in idempotents associated to the isomorphism $\xBur(G)\cong B_1\times \ldots \times B_n$ yields under~$\widehat{\MM}\sigma^c$ a similar decomposition of $\id_{\Id_{\MM(G)}}=f_1+\ldots +f_n$ by \Cref{Prop:B(G)-acts-on-M(G)}. Since $\MM(G)$ is idempotent-complete, the idempotents $(f_1)_x, \ldots , (f_n)_x$ yield decompositions of every object~$x\in\MM(G)$, and consequently a decomposition of the category~$\MM(G)$ as announced. Compare \Cref{Exa:ICADD} for $\cat{A}=\cat{B}=\MM(G)$ and $F=\Id_{\MM(G)}$.

Equivalently (and more abstractly), use that $\widehat{\MM}$ is additive, \ie it sends biproducts of objects in the block-complete bicategory $\mathbb Z\Spanhat$ to biproducts in $\ICADD$, so in particular it sends the motivic decomposition $G\simeq (G,e_1) \oplus \cdots \oplus (G,e_n)$ to an equivalence of additive categories $\MM(G)\simeq \widehat{\MM}(G,e_1)\oplus\ldots \oplus \widehat{\MM}(G,e_n)$.
\end{proof}

\begin{Rem} \label{Rem:the-other-comp}
Let $\cat M$ be any (rectified) Mackey 2-functor, and restrict it to groups as in \Cref{sec:more-examples}. In particular, we have adjunctions $(\Ind^G_H \dashv \Res^G_H, \leta, \leps)$ and $(\Res^G_H \dashv \Ind^G_H, \reta, \reps)$ for all subgroups $H\leq G$ and we have natural isomorphisms $\gamma^*_a\colon \Res^G_H \Rightarrow \Res^G_H$ for all $a\in \mathrm C_G(H)$. The ring map $\widehat{\MM} \sigma^c\colon \xBurk(G)\to \End(\Id_{\cat M(G)})$ of \eqref{eq:B(G)-vs-End(Id)} is then given quite explicitly in terms of this structure by sending the basis element $[H,a]_G\in \xBurk(G)$ to the natural transformation
\[
\xymatrix{
\Id_{\cat M(G)}  \ar@{=>}[r]^-{\reta} & \Ind_H^G \Res^G_H \ar@{=>}[rr]^-{\Ind^G_H \gamma_a^*} && \Ind_H^G \Res^G_H \ar@{=>}[r]^-{\leps}  & \Id_{\cat M(G)}
}
\]
(see \Cref{Def:iso-map}). For example the element $[H,1]_G$, which already comes from the element $[G/H]\in \Burk(G)$ of the ordinary Burnside algebra, is mapped to the (typically non-trivial) composite
\[
\xymatrix{
\Id_{\cat M(G)}  \ar@{=>}[r]^-{\reta} & \Ind_H^G \Res^G_H \ar@{=>}[r]^-{\leps}  & \Id_{\cat M(G)} \;.
}
\]
Recall for contrast that the `other' unit-counit composite
\[
\xymatrix{
\Id_{\cat M(H)}  \ar@{=>}[r]^-{\leta} & \Res^G_H \Ind_H^G  \ar@{=>}[r]^-{\reps}  & \Id_{\cat M(H)}
}
\]
is always just the identity, by the special Frobenius property \Mack{9}.
\end{Rem}

\begin{Exa}
Decompositions of (crossed) Burnside rings~$\Burk(G)$ and~$\xBurk(G)$ have been variously described in the literature, so it is possible to apply \Cref{Cor:decomposition} concretely. The simpler case of the classical Burnside ring is due to Dress \cite{Dress69} and says that the primitive idempotents of the integral Burnside ring $\Bur(G)$ are in bijection with the conjugacy classes of perfect subgroups of~$G$. See also \cite{Yoshida83} for a different approach. Bouc has analyzed decompositions of~$\xBurk(G)$ for various rings~$\kk$, obtaining in particular a complete answer in the rational case in~\cite{Bouc03} (see also \cite{OdaYoshida01}). He also describes the decomposition of the principal idempotent of~$\Bur(G)$ inside the larger ring~$\xBur(G)$.

These decompositions can be applied, via \Cref{Cor:decomposition}, to any example of Mackey 2-functor from \Cref{ch:Examples} which takes values in idempotent complete $\kk$-linear categories, for the ring~$\kk$.
\end{Exa}

\end{chapter-seven}
\begin{appendix}
\begin{chapter-appendix}
\chapter{Categorical reminders}
\label{app:categorical-reminders}%
\bigbreak
\bigbreak
\section{Bicategories and 2-categories}
\label{sec:2-recollections}
\medskip

As some excellent sources on bicategories and 2-categories are readily available, we do not provide a full discussion but limit ourselves to fixing ideas and notations. The original reference is B\'enabou~\cite{Benabou67} and popular sources include Kelly-Street~\cite{KellyStreet74}. In more recent literature, the reader can find a very short treatment in Leinster~\cite{Leinster98pp} or in Street~\cite{Street00}, a textbook chapter in \cite[\S7.7]{Borceux94a}, and a longer discussion in Lack~\cite{Lack10}.

\begin{Rem} \label{Rem:set-theory}
Inevitably, some of the definitions and examples considered here raise set-theoretical issues. Apart from occasionally making some smallness hypothesis (\eg in \Cref{Hyp:G_and_I_for_Span}), we mostly ignore such difficulties in this book as they really are orthogonal to its concerns. We simply trust that any reader knowledgeable enough to spot such pitfalls will also be able to resolve them to their satisfaction, for instance by introducing Grothendieck universes.
\end{Rem}

\begin{Ter}
\label{Ter:bicats}%
\index{bicategory}%
We use the following standard terminology:
\begin{enumerate}[(1)]
\smallbreak
\item
\index{$b0$@$\cat{B}_0$ \, class of objects} \index{B@$\cat{B}_0$}%
\index{$ass$@$\ass$ \, associator} \index{ass@$\ass$}%
\index{$lun$@$\lun$ \, left unitor} \index{lun@$\lun$}%
\index{$run$@$\run$ \, right unitor} \index{run@$\run$}%
A \emph{bicategory} $\cat{B}$ is a weak 2-category, \ie a category `weakly enriched in categories' in the precise sense of B\'enabou~\cite{Benabou67}. It consists of a class of object~$\cat{B}_0$, categories $\cat{B}(X,Y)$ for all pairs $X,Y\in \cat{B}_0$, composition and unit functors (below $1$ denotes the category with one object and one arrow)
\[
\circ = \circ_{XYZ} \colon \cat{B}(Y,Z) \times \cat{B}(X,Y)\to \cat{B}(X,Z)
\quad\quad\quad
1_X \textrm{ or } \Id_X \colon 1\to \cat{B}(X,X)
\]
as well as natural isomorphisms (called \emph{associators} and \emph{left and right unitors})
\[
\xymatrix{
\cat{B}(Z,W) \times \cat{B}(Y,Z) \times \cat{B}(X,Y)
 \ar[rr]^-{\Id \times \circ}
 \ar[d]_-{\circ \times \Id} &&
 \cat{B}(Z,W) \times \cat{B}(X,Z)
 \ar[d]^-{\circ} \\
 \cat{B}(Y,W) \times \cat{B}(X,Y)
 \ar[rr]_-{\circ}
 \ar@{}[urr]|{\NEcell \;\; \ass_{XYZW}} &&
 \cat{B}(X,W)
}
\]
\[
\vcenter { \hbox{
\xymatrix{
1 \times \cat{B}(X,Y) \ar@/^4ex/[dr]^-\cong
 \ar[d]_-{\Id_Y \times \Id} & \\
\cat{B}(Y,Y) \times \cat{B}(X,Y)
 \ar[r]_-{\circ}
 \ar@{}[ur]|{\NEcell \;\; \lun_{XY}} &
 \cat{B}(X,Y)
}
}}
\quad \quad
\vcenter { \hbox{
\xymatrix{
\cat{B}(X,Y) \times 1 \ar@/^4ex/[dr]^-\cong
 \ar[d]_-{\Id \times \Id_X} & \\
\cat{B}(X,Y) \times \cat{B}(X,X)
 \ar[r]_-{\circ}
 \ar@{}[ur]|{\NEcell \;\; \run_{XY}} &
 \cat{B}(X,Y)
}
}}
\]
expressing the up-to-isomorphism associativity and unitality of composition, which are required to satisfy two commutativity conditions (see Remark~\ref{Rem:strictification}):
\begin{equation}
\label{eq:bicategory}%
\vcenter{\xymatrix@C=.5em@L=1ex{
& \kern-1em ((fg)h)k \ar@{=>}[rr]^-{\ass\,\circ\, k} \ar@{=>}[ld]_-{\ass}
&& (f(gh))k \kern-1em \ar@{=>}[rd]^-{\ass}
\\
(fg)(hk) \ar@{=>}[rrd]_-{\ass} \kern-1em
&&&& \kern-1em f((gh)k) \ar@{=>}[lld]^-{f\,\circ\,\ass}
\\
&& \kern-1em f(g(hk)) \kern-1em
}}
\qquad
\vcenter{\xymatrix@C=.5em@L=1ex{
(f\Id)g \ar@{=>}[rr]^-{\ass\;\;} \ar@{=>}[rd]_-{\run\,\circ\, g}
&& f(\Id g) \ar@{=>}[ld]^-{f\,\circ\,\lun}
\\
& \kern-1em fg\kern-1em
}}
\end{equation}
\smallbreak
\item
\index{category@2-category} \index{$2category$@2-category}%
A \emph{2-category} $\cat{B}$ is a bicategory whose associators and unitors are all identities. It is the same as a category enriched over categories in the strict sense of~\cite{Kelly05}.
\smallbreak
\item
\index{horizontal composition}%
\index{vertical composition}%
\index{$b1$@$\cat{B}_1$ \, class of 1-cells} \index{b1@$\cat{B}_1$}%
The objects of a bicategory $\cat{B}$ are also called \emph{0-cells}, the objects of each $\cat{B}(X,Y)$ are called \emph{1-cells} and the arrows of each $\cat{B}(X,Y)$ are called \emph{2-cells}.
By analogy with~$\cat{B}_0$, we sometimes denote by $\cat{B}_1$ the collection of all 1-cells of~$\cat{B}$. The composition of 2-cells within each Hom category is \emph{vertical composition} and the effect of the composition functors $\circ_{X,Y,Z}$ on 1- or 2-cells is \emph{horizontal composition},
as suggested by the following picture where a $k$-cell appears with dimension~$k\in\{0,1,2\}$:
\[
\xymatrix{
& \ar@{}[d]|{\Scell \; \alpha_1} & & \ar@{}[d]|{\Scell \; \beta_1} & \\
X \ar[rr]|-{f_2} \ar@/^7ex/[rr]^-{f_1} \ar@/_7ex/[rr]_-{f_3} &&
 Y \ar[rr]|-{g_2} \ar@/^7ex/[rr]^-{g_1} \ar@/_7ex/[rr]_-{g_3} && Z\,. \\
& \ar@{}[u]|{\Scell \; \alpha_2} & & \ar@{}[u]|{\Scell \; \beta_2} &
}
\]
\index{exchange law}%
Both compositions are often denoted simply by juxtaposition. For example the \emph{exchange law}
\begin{equation} \label{eq:exchange_law}
(\beta_2 \alpha_2)(\beta_1 \alpha_1) =(\beta_2 \beta_1)(\alpha_2 \alpha_1)
\end{equation}
equates the 2-cell $g_1f_1 \Rightarrow g_3f_3$ obtained by composing the four ones in the above diagram first horizontally and then vertically, with the one obtained by composing first vertically then horizontally.
\smallbreak
\item
\index{whiskering}%
With notation as above, the composition functors $\circ$ restrict \eg to functors
\[
g_1\circ - \colon \cat{B}(X,Y) \to \cat{B}(X,Z)
\quad\quad
-\circ f_1\colon \cat{B}(Y,Z) \to \cat{B}(X,Z)
\]
sending $\alpha_1\mapsto g_1\alpha_1$ and $\beta_1\mapsto \beta_1 f_1$. Such operations on 2-cells are referred to as \emph{whiskering} (and were already used in~\eqref{eq:bicategory} above).
\end{enumerate}
\end{Ter}

\begin{Not} \label{Not:co-op}
\index{$bco$@$\cat{B}^\co$ \, revert 2-cells}%
\index{$bop$@$\cat{B}^\op$ \, revert 1-cells}%
\index{bopandbco@$\cat{B}^{\op}$ and $\cat{B}^\co$}%
\index{$fop$@$f^\op$}%
\index{$alphaco$@$\alpha^\co$}%
Given a bicategory~$\cat{B}$, the bicategories $\cat{B}^{\op}$ and $\cat{B}^\co$ are obtained by reverting 1-cells and 2-cells respectively. Reverting both gives~$\cat{B}^{\op,\co}$.
We write $f^\op\colon Y\to X$ for the 1-cell of $\cat B^{\op}$ corresponding to the 1-cell $f\colon X\to Y$ of~$\cat B$, and $\alpha^\co\colon v^{(\op)}\Rightarrow u^{(\op)}$ for the 2-cell of $\cat B^{(\op,)\co}$ corresponding to $\alpha\colon u\Rightarrow v$ in~$\cat B$.
\end{Not}

\begin{Exa}
\index{$cat$@$\CAT$ \, 2-category of categories}\index{cat@$\CAT$}
\index{$cat$@$\Cat$ \, 2-category of small categories} \index{Cat@$\Cat$}%
We denote by $\CAT$ the 2-category of categories, whose objects are typically denoted $\mathcal{A},\mathcal{B},\mathcal{C}$, etc. We denote by $\Cat$ the 2-subcategory of (essentially) small categories, whose objects are typically denoted $I,J$, etc.
\end{Exa}

\begin{Ter} \label{Ter:internal}
Recall that many categorical notions, familiar from the 2-category $\CAT$ of categories, can be internalized into any bicategory~$\cat{B}$.
\begin{enumerate}[(1)]
\smallbreak
\item
\index{equivalence in bicategory}%
An \emph{equivalence} $X\simeq Y$ between two objects of~$\cat{B}$ is a 1-cell $f\in\cat{B}(X,Y)$ such that there exists a 1-cell $g\in\cat{B}(Y,X)$ together with isomorphisms (invertible 2-cells)
$\id_{X}\cong g\circ f$ in $\cat{B}(X,X)$ and $\id_{Y}\cong f \circ g$
in $\cat{B}(Y,Y)$.
\smallbreak
\item
\label{it:adjunction}%
\index{adjunction in bicategory}%
\index{triangle equalities}%
\index{unit}%
\index{counit}%
\index{$\dashv$ \, adjunction}%
Similarly, an \emph{adjunction} in $\cat{B}$, often written $f\dashv g$, consists of 1-cells $f\in\cat{B}(X,Y)$ and $g\in\cat{B}(Y,X)$ together with 2-cells (the \emph{unit} and \emph{counit} of adjunction) $\eta\colon \Id_{X}\Rightarrow g\circ f$ and $\eps\colon f \circ g \Rightarrow \Id_{Y}$ satisfying the usual \emph{triangle equalities}
$(\varepsilon f)(f \eta)=\id_f$ and $(g \varepsilon)(\eta g)=\id_g$.
\smallbreak
\item
\index{adjoint equivalence}%
An \emph{adjoint equivalence} is an adjunction where the unit and counit are invertible.
\smallbreak
\item
\label{it:faithful}%
\index{faithful 1-cell in a bicategory}%
A 1-cell $i\colon X\to Y$ is \emph{faithful} if whiskering with~$i$ on the left is injective, \ie for every $Z\in \cat{B}_0$ the induced functor $i_*\colon \cat{B}(Z,X)\to \cat{B}(Z,Y)$ is faithful in~$\Cat$.
\end{enumerate}
\end{Ter}

\begin{Def}
\label{Def:2-1-category}%
\index{category@(2,1)-category} \index{$21category$@(2,1)-category}%
A 2-category in which every 2-cell is invertible, \ie whose Hom categories $\cat{B}(X,Y)$ are all groupoids, is called a \emph{(2,1)-category}.
\end{Def}

For instance, $\groupoid$ is a (2,1)-category, see \Cref{Not:groupoid}.

\begin{Rem}
\label{Rem:locally}%
\index{locally}%
When speaking of bicategories, we use the adverb `locally' to mean `Hom-wise', in an enriched sense. So, we say that a bicategory~$\cat{B}$ is `locally~P', where P is some property or type of category, if every Hom-category $\cat{B}(X,Y)$ is P and if the composition functors are P-compatible.
\end{Rem}

\begin{Exa} \label{Exa:trivial_bicats}
\index{locally discrete 2-category} \index{$2category$@2-category!locally discrete --}%
\index{category@1-category} \index{$1category$@1-category}%
Every ordinary category, or \emph{1-category}, can be considered the same as a \emph{locally discrete 2-category}, \ie one whose only 2-cells are identities. Locally discrete 2-categories are in particular (2,1)-categories.
\end{Exa}

\begin{Exa}
Some bicategories playing a central role in this work are \emph{locally additive}: their Hom categories are additive and the composition functors are additive in both variables.
See \Cref{Def:Sad-enriched-etc} for this and closely related notions.
\end{Exa}

\begin{Ter} \label{Ter:pseudofun}
In order to compare bicategories, we use:
\begin{enumerate}[(1)]
\smallbreak
\item
\index{pseudo-functor}%
\index{$fun$@$\fun$ \, pseudo-functor structure map}%
\index{$un$@$\un$ \, pseudo-functor structure map}%
A \emph{pseudo-functor} (in B\'enabou's terminology, \emph{homomorphism}) $\cat{F} \colon \cat{B}\to \cat{B}'$ between two bicategories $\cat{B}$ and $\cat{B'}$ consists of a map on objects $\cat{F}=\cat{F}_0\colon \cat{B}_0\to \cat{B}'_0$ together with functors $\cat{F}=\cat{F}_{XY}\colon \cat{B}(X,Y)\to \cat{B}'(\cat{F}X,\cat{F}Y)$ on the Hom categories and natural isomorphisms
\[
\vcenter { \hbox{
\xymatrix{
\cat{B}(Y,Z) \times \cat{B}(X,Y)
 \ar[r]^-\circ
 \ar[d]_-{\cat{F} \times \cat{F}} &
\cat{B}(X,Z)
 \ar[d]^-{\cat{F}} \\
\cat{B}'(\cat{F} Y,\cat{F} Z) \times \cat{B}'(\cat{F} X,\cat{F} Y)
 \ar[r]_-\circ
 \ar@{}[ur]|{\NEcell \; \fun_{XYZ}} &
 \cat{B}(\cat{F} X,\cat{F} Z)
}
}}
\quad\quad
\vcenter { \hbox{
\xymatrix@R=5pt{
& \cat{B}(X,X) \ar[dd]^-{\cat{F}} \\
1 \ar@/^2ex/[ur]^<<<<{1_X}
 \ar@/_2ex/[dr]_<<<<{1_{\cat{F} X \!\!}} & \\
\ar@{}[uur]|{\NEcell\; \un_X} &
 \cat{B}'(\cat{F}X , \cat{F}X)
}
}}
\]
expressing the up-to-isomorphism, or pseudo-, functoriality of $\cat{F}$; the latter are required to make the following diagrams commute (see Remark~\ref{Rem:coh_pseudofun} below):
\[
\vcenter { \hbox{
\xymatrix@C=16pt@L=1ex{
{(\cat F h \circ \cat F g) \circ \cat F f}
 \ar@{=>}[rr]^{\fun \,\circ\, \cat F f}
 \ar@{=>}[d]_{\ass} &&
 {\cat F (h \circ g) \circ \cat Ff} \ar@{=>}[rr]^{\fun} &&
 {\cat F((h\circ g)\circ f)}
 \ar@{=>}[d]^{\cat F \ass} \\
{\cat F h \circ (\cat F g \circ \cat F f)}
 \ar@{=>}[rr]^{\cat F h \,\circ\, \fun} &&
 {\cat F h \circ \cat F (g\circ f)}
 \ar@{=>}[rr]^{\fun} &&
 {\cat F (h \circ (g \circ f))}
}
}}
\]
\[
\vcenter { \hbox{
\xymatrix@C=14pt@L=1ex{
{\Id_{\cat FY} \circ \cat F f}
 \ar@{=>}[d]_\lun
 \ar@{=>}[rr]^-{\un \, \circ \, \cat Ff} &&
 {\cat F(\Id_Y) \circ \cat Ff}
 \ar@{=>}[d]^{\fun}
 \\
{\cat F f} &&
 {\cat F(\Id_Y \circ f)}
 \ar@{=>}[ll]^-{\cat F \,\lun}
}
}}
\quad\quad
\vcenter { \hbox{
\xymatrix@C=14pt@L=1ex{
{\cat F f \circ \Id_{\cat FX}}
 \ar@{=>}[d]_\run
 \ar@{=>}[rr]^-{\cat Ff \, \circ \, \un} &&
 {\cat Ff \circ \cat F(\Id_X)}
 \ar@{=>}[d]^{\fun}
 \\
{\cat F f} &&
 {\cat F(f \circ \Id_X)}
 \ar@{=>}[ll]^-{\cat F \,\run}
}
}}
\]
for all composable 1-cells $X \stackrel{f}{\to} Y \stackrel{g}{\to} Z \stackrel{h}{\to} W$.
\smallbreak
\item
\label{it:strict-2-functor}%
\index{$2functor$@2-functor} \index{functor@2-functor}%
\index{strict pseudo-functor} \index{pseudo-functor!strict --}%
A \emph{(strict) 2-functor} always refers to a \emph{strict} pseudo-functor, \ie one where $\fun_{XYZ}$ and $\un_{X}$ are all identities.
\end{enumerate}
\end{Ter}

\begin{Exa} \label{Exa:trivial_pseudofun}
\index{functor from 2-category to ordinary category}%
A pseudo-functor whose target is locally discrete is necessarily strict; we simply call it a \emph{functor}, the prime examples being (ordinary) functors between (ordinary) categories.
Note however that a pseudo-functor $\cat{F}\colon \cat{B}\to \cat{B}'$ whose source $\cat{B}$ is locally discrete need not be strict if $\cat{B}'$ is not locally discrete, even, say, when $\cat{B}$ is a group.
\end{Exa}

\begin{Not} \label{Not:htpy_cat}
\index{truncation~$\pih$}%
\index{$t1$@$\pih$ \, 1-truncation}%
Every bicategory $\cat{B}$ has an associated
\emph{1-truncation}~$\pih{\cat{B}}$, defined to be the 1-category obtained from $\cat{B}$ by identifying any two isomorphic 1-cells and then discarding the 2-cells.
If $\cat{B}$ is a (2,1)-category,  there is an obvious functor $\cat{B}\to \pih{\cat{B}}$, which is initial among functors from $\cat{B}$ to 1-categories.
\end{Not}

\begin{Ter} \label{Ter:Hom_bicats}
Pseudo-functors between bicategories can be naturally assembled into bicategories, as follows. Fix two bicategories $\cat B$ and~$\cat B'$.
\begin{enumerate}[(1)]
\smallbreak
\item
\label{it:transformation}%
\index{transformation!pseudo-natural --}%
If $\cat{F}_1,\cat{F}_2 \colon \cat B\to \cat B'$ are parallel pseudo-functors, a \emph{(pseudo-natural) transformation} $t\colon \cat{F}_1\Rightarrow \cat{F}_2$ between them consists of a 1-cell
$t_X \colon \cat{F}_1 X \to \cat{F}_2X$ for every object $X$ of $\cat B$ and of an invertible 2-cell
\[
\xymatrix{\cat{F}_1 X
 \ar[r]^-{t_X}
 \ar[d]_-{\cat{F}_1 u} &
\cat{F}_2 X
 \ar[d]^-{\cat{F}_2 u} \\
\cat{F}_1 Y \ar[r]_-{t_Y}
 \ar@{}[ur]|{\SWcell\; t_u} &
 \cat{F}_2Y
}
\]
for every 1-cell $u\colon X\to Y$ of $\cat B$, subject to reasonable compatibility conditions with the vertical and horizontal composition of~$\cat B$ which are detailed in~\cite{Leinster98pp}. In pasting diagrams (thus omitting associators), the latter require
\[
\vcenter{\hbox{
\xymatrix{
& \cat{F}_1 X
\ar@/_6ex/[d]_{\Id}
\ar@{}[dl]|{\;\;\;\;\;\;\stackrel{\scriptstyle \un^{\!-1}}{\simeq}}
 \ar[r]^-{t_X}
 \ar[d]^-{\cat{F}_1 \Id} &
\cat{F}_2 X
 \ar[d]^-{\cat{F}_2 \Id} \\
& \cat{F}_1 X \ar[r]_-{t_X}
 \ar@{}[ur]|{\;\;\;\;\;\;\;\;\SWcell t_{\Id}} &
 \cat{F}_2X
}
}}
\quad = \quad
\vcenter{\hbox{
\xymatrix{
\cat{F}_1 X
 \ar[r]^-{t_X}
 \ar[d]_-{\Id} &
\cat{F}_2 X
 \ar[d]_-{\Id}
 \ar@/^6ex/[d]^{\cat{F}_2 \Id}
 \ar@{}[dr]|{\stackrel{\scriptstyle \un^{\!-1}}{\simeq}\;\;\;\;\;\;}
 & \\
\cat{F}_1 X \ar[r]_-{t_X}
 &
 \cat{F}_2X &
}
}}
\]
and
\[
\vcenter {\hbox{
\xymatrix{
& \cat{F}_1 X
\ar@/_8ex/[dd]_{\cat F_1 vu}
 \ar[r]^-{t_X}
 \ar[d]_-{\cat{F}_1 u} &
\cat{F}_2 X
 \ar[d]^-{\cat{F}_2 u} \ar@/^8ex/[dd]^{F_2vu} & \\
& \cat{F}_1 Y
\ar[d]_{\cat F_1 v}
\ar[r]^-{t_Y}
 \ar@{}[ur]|{\SWcell\; t_u} &
 \cat{F}_2Y
 \ar[d]^{\cat F_2 v} & \\
 \ar@{}[uur]|{\stackrel{\scriptstyle \fun}{\simeq}} &
 \cat F_1 Z
 \ar@{}[ur]|{\SWcell\;t_v}
 \ar[r]_-{t_Z} &
 \cat F_2 Z
 \ar@{}[uur]|{\stackrel{\scriptstyle \;\fun^{\!\!-1}}{\simeq}} &
}
}}
\quad = \quad
\vcenter{\hbox{
\xymatrix{\cat{F}_1 X
 \ar[r]^-{t_X}
 \ar[dd]_-{\cat{F}_1 vu} &
\cat{F}_2 X
 \ar[dd]^-{\cat{F}_2 vu} \\
 & \\
\cat{F}_1 Z \ar[r]_-{t_Z}
 \ar@{}[uur]|{\SWcell\; t_{vu}} &
 \cat{F}_2Z
}
}}
\]
for all composable pair of 1-cells $X\stackrel{u}{\to} Y \stackrel{v}{\to} Z$ (\emph{functoriality}) as well as
\[
\vcenter{\hbox{
\xymatrix@R=16pt{
\cat{F}_1 X
 \ar[r]^-{t_X}
 \ar[dd]_-{\cat{F}_1 u} &
\cat{F}_2 X
 \ar[dd]_-{\cat{F}_2 u}
 \ar@/^8ex/[dd]^{\cat F_2 u'} & \\
 & & \\
\cat{F}_1 Y \ar[r]_-{t_Y}
 \ar@{}[uur]|{\SWcell t_{u} \;\;\;\;} &
 \cat{F}_2Y \ar@{}[uur]|{\stackrel{\cat F_2 \alpha\;}{\Wcell}} &
}
}}
\quad=\quad
\vcenter{\hbox{
\xymatrix@R=16pt{
& {\cat F_1 X}
 \ar@/_8ex/[dd]_{\cat F_1 u}
 \ar[r]^{t_X}
 \ar[dd]^{\cat F_1 u'} &
 {\cat F_2 X} \ar[dd]^{\cat F_2 u'} \\
&& \\
\ar@{}[uur]|{\stackrel{\cat F_1 \alpha\;}{\Wcell}} &
 {\cat F_1 Y}
 \ar@{}[uur]|{\;\;\;\;\;\;\;\SWcell t_{u'}}
 \ar[r]_-{t_Y} & {\cat F_2 Y}
}
}}
\]
for every 2-cell $\alpha:u'\Rightarrow u$ (\emph{naturality}).
Note that the orientation of the 2-cells $t_u$ is a matter of convention, since they are invertible.
\smallbreak
\item
\label{it:lax-transformation}%
\index{lax transformation} \index{transformation!lax --}%
\index{oplax transformation} \index{transformation!oplax --}%
If the 2-cells $t_u$ in the definition of a pseudo-natural transformation are not required to be invertible, one speaks of an \emph{oplax transformation}. If they are oriented in the opposite direction (\ie $t_u\colon t_Y(\cat{F}_1u)\Rightarrow (\cat{F}_2u)t_X$), one obtains the notion of a \emph{lax transformation}.\,(\footnote{\,Some authors prefer to swap the use of `lax' and `oplax'.})
For contrast, a pseudo-natural transformation as in~\eqref{it:transformation} above is also referred to as a \emph{strong} transformation. It is a \emph{strict} transformation if the 2-cells $t_u$ are actually all identities.
\smallbreak
\item
\index{modification}%
If $t, s\colon \cat{F}_1 \Rightarrow \cat{F}_2$ are two parallel (oplax or strong) transformations of pseudo-functors, a \emph{modification} $M\colon t \Rrightarrow s$ between them consists of a 2-cell $M_X\colon t_X\Rightarrow s_X$ of $\cat B'$ for every object $X$ of~$\cat B$, such that the square
\[
\xymatrix@L=1ex{
(\cat{F}_2 u) \, t_X \ar@{=>}[r]^-{1\; M_X \;\;}
 \ar@{=>}[d]_-{t_u} &
 (\cat{F}_2 u) \, s_X \ar@{=>}[d]^-{s_u} \\
t_Y \, (\cat{F}_1 u) \ar@{=>}[r]^-{M_Y \; 1\;\;} &
 s_Y\, (\cat{F}_1 u)
}
\]
is commutative for every 1-cell $u\colon X\to Y$ in $\cat B$.
\smallbreak
\item We denote by $\PsFun(\cat B, \cat B')$ the bicategory with pseudo-functors $\cat B\to \cat B'$ as 0-cells, pseudo-natural transformations between them as 1-cells and modifications as 2-cells. The horizontal composition of transformations and the vertical composition of modifications is performed in the evident way by composing 1- and 2-cells of~$\cat B'$. In particular, if $\cat B'$ is a 2-category then so is $\PsFun(\cat B, \cat B')$.
\smallbreak
\item
\index{biequivalence}%
A \emph{biequivalence} between bicategories is a pseudo-functor $\cat{F}\colon \cat{B}\to \cat{B}'$ such that there exists a pseudo-functor $\cat{G}\colon \cat{B}'\to \cat{B}$ and equivalences $1_\cat B \simeq \cat{G}\cat{F}$ in $\PsFun(\cat B,\cat B)$ and $\cat{F}\cat{G} \simeq 1_{\cat B'}$ in $\PsFun(\cat B', \cat B')$.
Similarly to ordinary equivalences of categories, a pseudo-functor $\cat{F}\colon \cat{B}\to \cat{B}'$ is a biequivalence if and only if each functor $\cat{F}\colon \cat{B}(X,Y)\to \cat{B}'(\cat{F}X,\cat{F}Y)$ is an equivalence of categories (expressing `essential full-faithfulness') and such that $\cat{F}$ is `essentially surjective' meaning that for every $Y\in\cat{B}'$ there exists $X\in\cat{B}$ and an equivalence $\cat{F}X\simeq Y$ in~$\cat{B}'$. This uses the axiom of choice for proper classes and the possibility of correcting any equivalence to an adjoint equivalence (see~\cite[Prop.\,1.5.13]{Leinster04}).
\end{enumerate}
\end{Ter}

\begin{Rem} \label{Rem:PsFun_bieq}
The previous terminology makes clear what it means for two pseudo-functors to be equivalent or isomorphic. Pseudo-functors can be composed in the evident way (\cite[Thm.\,4.3.1]{Benabou67}), and indeed by allowing their source and target to vary we get a \emph{tricategory} of bicategories, pseudo-functors, transformations and modifications -- but we do not need to go into that. Note nonetheless that $\PsFun(-,-)$ preserves biequivalences in each variable.
\end{Rem}

\begin{Rem} \label{Rem:strictification}
\index{strictification!-- for bicategories}%
\index{pasting diagram} \index{diagram!pasting --}%
The coherence axioms~\eqref{eq:bicategory} satisfied by the associators and unitors of a bicategory~$\cat B$ generalize those for a monoidal category. They guarantee that the analogue of Mac~Lane's coherence theorem holds: There exists a (canonical) biequivalence $\cat{F}\colon \cat B\stackrel{\sim}{\to} \cat B'$ where $\cat B'$ is a 2-category (see \eg~\cite[Thm.\,1.5.15]{Leinster04}). An elegant way to view this result is as a bicategorical Yoneda lemma: By the coherence axioms, the assignment $X\mapsto \cat B(-,X)$ induces a well-defined pseudo-functor $\cat B\to \PsFun(\cat B^{\op},\Cat)$ which restricts to a biequivalence $\cat B\stackrel{\sim}{\to} \cat B'$ onto its 1- and 2-full image~$\cat B'$ which, like its ambient bicategory $\PsFun(\cat B^{\op},\Cat)$, is strict since $\Cat$ is.
An important consequence of this is that 2-cells of $\cat B$ can be unambiguously presented by a \emph{pasting diagram} (\cite{KellyStreet74}), that is, by a 2-dimensional display of composable 2-cells, without having to worry about choosing a sequence of evaluation moves of the diagram by vertical and horizontal composition, since all choices will yield the same 2-cell as end result.
\end{Rem}

\begin{Rem} \label{Rem:coh_pseudofun}
\index{strictification!-- for pseudo-functors}%
The coherence axioms satisfied by the structural isomorphisms $\fun$ and $\un$ of a pseudo-functor $\cat{F}\colon \cat{B}\to \cat{B}'$ (detailed in~\cite[\S4]{Benabou67}~\cite{Leinster98pp}) are again reminiscent of those of a (strong) monoidal functor between monoidal categories. These axioms guarantee that the internal notions of \Cref{Ter:internal}, among many others, are preserved by pseudo-functors, hence in particular that they are invariant under biequivalence.
Similarly to bicategories, pseudo-functors can often be strictified (though not always, \cf~\cite[Lemma~2]{Lack07}). \Eg if $\cat{C}$ is any small 2-category, every pseudo-functor $\cat{C}\to \Cat$ is equivalent in $\PsFun(\cat{C},\Cat)$ to some 2-functor; see~\cite[\S4.2]{Power89}.
\end{Rem}

\begin{Cor} \label{Cor:2cat_Yoneda_equivs}
A 1-cell $f\colon X\to Y$ in a bicategory $\cat B$ is an equivalence if and only if $\cat B(T,f)\colon \cat B(T,X)\to \cat B(T,Y)$ is an equivalence of categories for every object $T\in \cat B_0$.
\end{Cor}

\begin{proof}
Since biequivalences preserve internal equivalences (\Cref{Rem:coh_pseudofun}), we may replace $\cat B$ with its image in $\PsFun(\cat B^{\op}, \Cat)$ under the bicategorical Yoneda embedding $X \mapsto \cat B(-,X)$ (Remark~\ref{Rem:strictification}) and the 1-cell $f$ with the pseudo-natural transformation $\cat B(-,f)\colon \cat B(-,X)\to \cat B(-,Y)$  (\Cref{Ter:Hom_bicats}\,\eqref{it:transformation}). Hence it suffices to prove the following: A pseudo-natural transformation $t\colon \cat F\to \cat G$ between pseudo-functors $\cat F, \cat G\in \PsFun(\cat B^\op, \Cat)_0$ is an equivalence if and only if for each $T\in \cat B_0$ the component $t_T\colon \cat FT\to \cat GT$ is an equivalence of categories.
One implication is immediate; for the other, assume that each $t_T$ is an equivalence. Choose an adjoint equivalence $s_T\dashv t_T$ for each~$T$, with unit $\eta_T\colon \Id_{\cat G T}\stackrel{\sim}{\Rightarrow} t_Ts_T$ and counit $\varepsilon_T\colon s_Tt_T\stackrel{\sim}{\Rightarrow} \Id_{\cat FT}$, and use them to define (invertible!) natural transformations~$s_u$
\[
\vcenter{\hbox{
\xymatrix{
\cat GX \ar[r]^-{s_X} \ar[d]_-{\cat Gu} & \cat FX \ar[d]^-{\cat Fu} \\
\cat GY \ar[r]_-{s_Y} \ar@{}[ur]|{\SWcell\; s_u} & \cat FY
}
}}
\quad\quad := \quad\quad
\vcenter{\hbox{
\xymatrix{
\ar@{}[dr]|{\NWcell\;\eta_X\inv} & \cat GX \ar[d]^-{s_X} \\
\cat GX \ar@{}[dr]|{\NWcell\; t_u\inv} \ar[d]_-{\cat Gu} \ar@/^4ex/@{=}[ur] & \cat FX \ar[d]^-{\cat Fu} \ar[l]_-{t_X} \\
\cat GY \ar[d]_-{s_Y} \ar@{}[dr]|{\NWcell\;\varepsilon_Y\inv} & \cat FY \ar[l]_-{t_Y} \\
\cat FY \ar@/_4ex/@{=}[ur] &
}
}}
\]
for all 1-cells $u\colon Y\to X$ in~$\cat B$. It follows from the triangle identities of the adjunctions that $s:=\{s_T,s_u\}_{T,u}$ is a pseudo-natural transformation $\cat G\to\cat F$ and that the units and counits form invertible modifications $\Id_\cat G\cong ts$ and $st \cong \Id_\cat F$. This shows that $t$ is an equivalence in the bicategory $\PsFun(\cat B^{\op},\Cat)$.
\end{proof}

\begin{Rem} \label{Rem:2cat_Yoneda}
Specifically for 2-categories there is another form of the Yoneda lemma, which comes from viewing 2-categories as categories (strictly) enriched over the cartesian closed category $\Cat$ and by specializing the results of~\cite[\S2]{Kelly05}. This result however is not so useful for us, because for any given 2-categories $\cat C$ and $\cat D$ it only concerns the 2-category $\Cat\textrm{-}\!\Fun(\cat C,\cat D)$ of 2-functors $\cat{F}\colon \cat{C}\to \cat D$, \emph{strict} natural transformations, and modifications (\ie the `functor category' of~\cite{Kelly05}); hence it says nothing about non-strict pseudo-natural transformations.
\end{Rem}

We encounter the following 2-categorical variant of usual comma (or `slice') categories:
\begin{Def}
\label{Def:2-comma}%
\index{comma 2-category}%
\index{$bc$@$\cat{B}_{\smallslash C}$ \, comma 2-category}%
Let $\cat{B}$ be a 2-category and $B\in \cat{B}_0$ be a 0-cell. We denote by
\[
\cat{B}_{\smallslash B}
\]
the following \emph{comma 2-category} over~$B$. By definition, its objects are pairs $(X,i_X)$ where $X\in \cat{B}_0$ is a 0-cell and $i_X\colon X\to B$ is a 1-cell of~$\cat{B}$. A 1-cell $(X,i_X)\to (X',i_{X'})$ consists of a pair $(f,\theta_f)$ where $f\colon X\to X'$ is a 1-cell and $\theta_f\colon i_{X'} f \Rightarrow i_X$ a 2-cell (\footnote{\,One can also consider a version in which $\theta_f$ is requested to be invertible. When we apply this construction to a (2,1)-category, like $\cat{B}=\groupoid$, this choice is irrelevant.})
in~$\cat{B}$. A 2-cell $(f,\theta_f)\Rightarrow (g,\theta_g)$ in~$\cat{B}_{\smallslash B}$ is a 2-cell $\alpha\colon f\Rightarrow g$ of $\cat{B}$ such that $\theta_g\,(i_{X'}\alpha) = \theta_f$:
\[
\vcenter{\vbox{
\xymatrix@R=10pt{
X
 \ar@/^2ex/[rrd]^-{i_X}
 \ar@/_3ex/[dd]_f
 \ar@/^3ex/[dd]^g
 \ar@{}[dd]|{\overset{\scriptstyle\alpha}\Ecell}
 \ar@{}[ddrr]|{\qquad\NEcell\;\theta_g} && \\
&& B \\
X' \ar@/_2ex/[rru]_-{i_{X'}} &&
}
}}
\quad = \quad
\vcenter{\vbox{
\xymatrix@R=10pt{
X
 \ar@/^2ex/[rrd]^-{i_X}
 \ar@/_2ex/[dd]_f
 \ar@{}[ddr]|{\NEcell\;\theta_f} && \\
&& B\,. \\
X' \ar@/_2ex/[rru]_-{i_{X'}} &&
}
}}
\]
The vertical and horizontal compositions of~$\cat{B}_{\smallslash B}$ are induced by those of $\cat{B}$ in the evident way.
There is an obvious forgetful 2-functor $\cat{B}_{\smallslash B}\to \cat{B}$ which sends $(X,i_X)$ to $X$ and $(f,\theta_f)$ to~$f$.
\end{Def}

\bigbreak
\section{Mates}
\label{sec:mates}%
\medskip

\index{mate}%
Most familiar results on adjunctions generalize to general bicategories~$\cat{B}$. Let us in particular recall some basic facts about \emph{mates}, that is, about the 2-cell correspondences which are induced (in various ways) by adjunctions.
Consider an adjunction $\ell \adj r$ in $\cat{B}$, for instance in~$\cat{B}=\CAT$, with unit $\eta\colon \Id\Rightarrow r \ell$ and counit $\eps\colon \ell r\Rightarrow \Id$. Then, for any two 1-cells $f_1$ and $f_2$ (with suitable source and target), there are natural bijections between classes of 2-cells (here all simply denoted $[-,-]$)
\[
[f_1, r f_2] \isoto [\ell f_1,f_2]
\qquadtext{and}
[f_1 r , f_2] \isoto [f_1,f_2 \ell]
\]
given by
\[
\vcenter{\xymatrix@C=1em@R=1em{
\ar@/_1em/[dd]_-{f_1} \ar@/^0.5em/[rd]^-{f_2} \ar@{}[dd]|-{\;\;\oEcell{\alpha}}
&&&&& \ar@/_0.5em/[ldd]_-{f_1} \ar@/^0.5em/[rd]^-{f_2} \ar@{}[d]|(.8){\SEcell\alpha}
\\
& \ar@/^0.5em/[ld]^-{r}
& \ar@{|->}[r]
&&&& \ar[lld]|-{\ r\ } \ar@/^0.5em/@{=}[ldd]
\\
&&&& \ar@/_0.5em/[rd]_-{\ell}
& \ar@{}[d]|(.3){\SEcell\eps\;\;}
\\
&&&&& }}
\qquadtext{and}
\vcenter{\xymatrix@C=1em@R=1em{
& \ar@/^1em/[dd]^-{f_2} \ar@/_0.5em/[ld]_-{r} \ar@{}[dd]|-{\oEcell{\beta}\;\;}
&&&&& \ar@/_0.5em/@{=}[ldd] \ar@/^0.5em/[rd]^-{\ell} \ar@{}[d]|(.6){\;\;\SEcell\eta}
\\
\ar@/_0.5em/[rd]_-{f_1}
&&& \ar@{|->}[r]
&&&& \ar[lld]|-{\ r\ } \ar@/^0.5em/[ldd]^-{f_2}
\\
&&&&& \ar@/_0.5em/[rd]_-{f_1}
&\ar@{}[d]|(.3){\SEcell\beta}
\\
&&&&&& }}
\]
respectively.

\begin{Rem}
\label{Rem:_!_*-for-id}%
Given a morphism $\alpha\colon k\Rightarrow \ell$ between 1-cells with left adjoints $k_!\adj k$ and $\ell_!\adj \ell$, we obtain (\eg by Yoneda) a canonical morphism $\alpha_!\colon \ell_!\Rightarrow k_!$. This is nothing but the mate of
\[
\xymatrix@C=14pt@R=14pt{
& \ar@{=}[rd] \ar[ld]_-{k} \ar@{}[dd]|-{\oEcell{\alpha}}
\\
\ar@{=}[dr] && \ar[dl]^{\ell}
\\
& }
\]
which we also denote $\alpha_!$ anyway. Compatibility of mates with pasting becomes the 2-contravariant functoriality of~$(-)_!$, namely $(\beta\alpha)_!=\alpha_!\beta_!$ and $\id_!=\id$.
\end{Rem}

\begin{Rem}
There are situations were the notation $(-)_!$ can be slightly ambiguous, particularly if the `ambient square' is not made explicit.
For instance, if we are given a 2-cell $\gamma\colon k \ell' \Rightarrow \ell k'$ in a 2-category it might happen that both $\ell$ and $\ell'$ have left adjoints $\ell_!$ and $\ell'_!$, in which case the mate of $\gamma$ would be $\gamma_!\colon \ell_! k\Rightarrow k'\ell'_!$. But it can simultaneously happen that $(k\ell')$ and $(\ell k')$ have left adjoints, in which case the mate of~$\gamma$ can be understood as the 2-cell also denoted $\gamma_!\colon (k'\ell)_!\Rightarrow(k\ell')_!$, as in \Cref{Rem:_!_*-for-id}. This issue is purely notational and context should usually make clear what is meant. Note that providing the source and target 1-cells avoids any confusion and we always try to do so.
\end{Rem}

\begin{Rem}
\label{Rem:mates-and-2-functors}%
Throughout the work, we use mates for functors of the form $u^*, v^*, i^*, j^*$ in $\CAT$ appearing through various 2-functors $\MM$ from (2-subcategories of) $\Cat^\op$ with values in~$\CAT$, that is, with $u^*=\MM(u)$, etc. Let us phrase the results we need about mates in this setting.
\end{Rem}

Suppose given four 1-cells and their image under such a 2-functor~$\MM$
\[
\vcenter{\xymatrix@C=2em@R=2em{
& \ar[ld]_-{v} \ar[rd]^-{j}
\\
\ar[rd]_-{i} && \ar[ld]^-{u}
\\
&
}}
\qquad \overset{\MM}{\longmapsto} \qquad
\vcenter{\xymatrix@C=2em@R=2em{
& \ar@{<-}[ld]_-{v^*} \ar@{<-}[rd]^-{j^*}
\\
\ar@{<-}[rd]_-{i^*} && \ar@{<-}[ld]^-{u^*}
\\
&
}}
\]
(the squares are not assumed to commute). Playing the mate construction on both sides in the presence of adjunctions $i_!\adj i^* \adj i_*$ and $j_!\adj j^* \adj j_*$ yields bijections
\[
[v^* i^* \,,\, j^* u^*]  \isotoo  [j_! v^* \,,\, u^* i_!]
\qquad \textrm{and} \qquad
 [j^* u^*\,,\,v^* i^*]  \isotoo [u^* i_* \,,\, j_* v^* ]\,.
\]
On 2-cells $\alpha^*$ and $\beta^*$ coming via~$\MM$, we use the standard notation
\[
\xymatrix@C=2em@R=0em{
[iv,uj] \ar[r]^-{\MM}
& [v^* i^* , j^* u^*] \ar[r]^-{\sim}
& [j_! v^* , u^* i_!]
\\
\alpha \ar@{|->}[r]
& \alpha^* \ar@{|->}[r]
& \alpha_!
}
\]
and
\begin{equation}
\label{eq:mate-isos}%
\vcenter{\xymatrix@C=2em@R=0em{
[uj,iv] \ar[r]^-{\MM}
& [j^* u^*,v^* i^*] \ar[r]^-{\sim}
& [u^* i_* , j_* v^* ]
\\
\beta \ar@{|->}[r]
& \beta^* \ar@{|->}[r]
& \beta_*
}}
\end{equation}
which matches the usual one for derivators. Explicitly, these mates of $\alpha^*$ are:
\begin{align}
\label{eq:alpha_!-def}%
\alpha_! & =
\vcenter{\xymatrix@C=1.5em@R=1em@L=1ex{
&&&&& \ar@{=}@/_.5em/[ldd] \ar@/^.2em/[rd]^-{i_!} \ar@{}[dd]|(.4){\oEcell{\eta}}
\\
&&&&&& \ar[lld]|-{\ i^*\ } \ar[rd]^-{u^*} \ar@{}[ldd]|-{\oEcell{\alpha^*}}
\\
{\bigg(} j_! v^* \ar@{=>}[r]^-{\eta}_-{i_! \adj i^*}
& j_! v^*i^*i_! \ar@{=>}[r]^-{\alpha^*}
& j_! j^*u^*i_! \ar@{=>}[r]^-{\eps}_-{j_! \adj j^*}
& u^*i_! {\bigg)} =
& \ar[rd]_-{v^*}
&&& \ar[lld]|-{\ j^*\ } \ar@{=}@/^.5em/[ldd]
\\
&&&&& \ar@/_.2em/[rd]_-{j_!}
\\
&&&&&& \ar@{}[uu]|(.4){\oEcell{\eps}}
}}
\\
\label{eq:beta_*-def}%
\beta_* & =
\vcenter{\xymatrix@C=1.5em@R=1em@L=1ex{
&&&&& \ar@{=}@/_.5em/[ldd] \ar@/^.2em/[rd]^-{i_*} \ar@{}[dd]|(.4){\overset{\scriptstyle\eps}\Wcell}
\\
&&&&&& \ar[lld]|-{\ i^*\ } \ar[rd]^-{u^*} \ar@{}[ldd]|-{\overset{\scriptstyle\beta^*}\Wcell}
\\
{\bigg(} u^* i_* \ar@{=>}[r]^-{\eta}_-{j^* \adj j_*}
& j_* j^* u^* i_* \ar@{=>}[r]^-{\beta^*}
& j_* v^* i^* i_* \ar@{=>}[r]^-{\eps}_-{i^* \adj i_*}
& j_* v^* {\bigg)}=
& \ar[rd]_-{v^*}
&&& \ar[lld]|-{\ j^*\ } \ar@{=}@/^.5em/[ldd]
\\
&&&&& \ar@/_.2em/[rd]_-{j_*}
\\
&&&&&& \ar@{}[uu]|(.4){\overset{\scriptstyle\eta}\Wcell}
}}
\end{align}
For the details of the so-called \emph{calculus of mates}, the reader is invited to consult~\cite[\S\,1.2]{Groth13} or~\cite{KellyStreet74}. We invoke in some places the \emph{compatibility of mates with pasting}, which can be found in~\cite[Lem.\,1.14]{Groth13}.
\index{mate!compatibility of mate with pasting}

\begin{Exa}
\label{Exa:units-as-mates}%
For any 1-cell $i$, consider the commutative squares
\[
\vcenter{\xymatrix@C=2em@R=2em{
& \ar@{=}[ld]_-{} \ar@{=}[rd]^-{} \ar@{}[dd]|-{\oEcell{\id_{i}}}
\\
\ar[rd]_-{i} && \ar[ld]^-{i}
\\
&
}}
\qquadtext{and}
\vcenter{\xymatrix@C=2em@R=2em{
& \ar[ld]_-{i} \ar[rd]^-{i} \ar@{}[dd]|-{\oEcell{\id_{i}}}
\\
\ar@{=}[rd]_-{} && \ar@{=}[ld]^-{}
\\
&
}}
\]
Using the left-hand and right-hand diagrams, respectively, to form the left mate of~$\id_i$, we obtain the unit $(\id_i)_! = \eta \colon \Id \Rightarrow i^*i_!$ and counit $(\id_i)_! = \eps \colon i_!i^* \Rightarrow \Id$ of the adjunction $i_! \adj i^*$.
Similarly, taking right mates respectively yields the counit $(\id_i)_* = \eps \colon i^*i_* \Rightarrow \Id$ and unit $(\id_i)_* = \eta \colon \Id \Rightarrow i_*i^*$ of $i^* \adj i_*$.\end{Exa}

Compatibility of mates with pasting gives in particular:
\begin{Prop}
\label{Prop:mates-under-top-functor}%
Consider the left-hand diagram and its whiskered 2-cell
\[
\vcenter{\xymatrix{
& \ar[d]^-{d}
\\
& \ar[ld]_-{v} \ar[rd]^-{j} \ar@{}[dd]|-{\oEcell{\gamma}}
\\
\ar[rd]_-{i} && \ar[ld]^-{u}
\\
&
}}
\qquadtext{and}
\vcenter{\xymatrix{
& \ar[ld]_-{v d} \ar[rd]^-{j d} \ar@{}[dd]|-{\oEcell{\gamma d}}
\\
\ar[rd]_-{i} && \ar[ld]^-{u}
\\
&
}}
\]
Suppose that $d^*$ has a left adjoint~$d_!$. Then $(\gamma d)_!=\gamma_!\,\eps$, or more precisely the following diagram commutes
\[
\xymatrix@L=1ex{
j_! d_! d^* v^* \ar@{=>}[d]_-{j_! {\eps} v^*} \ar@{=}[r]^-{\sim}
& (jd)_! (vd)^* \ar@{=>}[d]^-{(\gamma d)_!}
\\
j_! v^* \ar@{=>}[r]_-{\gamma_!}
& u^* i_!
}
\]
where $\eps\colon d_! d^*\Rightarrow\id$ is the counit of $d_!\adj d^*$. Dually, for right mates.
\end{Prop}

\begin{proof}
The 2-cell on the right is the result of the pasting of the 2-cell on the left with the 2-cells of \Cref{Exa:units-as-mates}.
\end{proof}

\begin{Rem}
\label{Rem:mate-natural}%
The above proposition applies in particular when $d$ is an equivalence, allowing us to replace the top object of a square up to equivalence. The reader can furthermore verify that the mating isomorphisms $[v^*i^*,j^*u^*]\cong [j_!v^*,u^*i_!]$ and $[j^*u^*,i^*v^*]\cong [u^*i_*, j_*v^*]$ of~\eqref{eq:mate-isos} are natural in $i,u,j,k$, with respect to 2-cells $i\Rightarrow i'$, $u\Rightarrow u'$, $j\Rightarrow j'$ and $v\Rightarrow v'$.
\end{Rem}

\begin{Rem}
\label{Rem:pseudo-func-of-adjoints}%
Suppose we have a pseudo-functor $\cat{F}\colon \cat B\to \cat B'$ with the property that every 1-cell $\cat{F}u$ admits in~$\cat B'$ a left adjoint $(\cat{F}u)_!$. Then a choice of adjunctions $\cat (Fu)_! \dashv \cat{F}u$ for every $u$ defines a pseudo-functor $\cat{F}_!\colon \cat B^{\op,\co} \to \cat B'$ which agrees with $\cat{F}$ on objects, sends a 1-cell $u$ to $(\cat{F}u)_!$ and a 2-cell $\alpha$ to the mate $(\cat{F} \alpha)_! $ defined as in \Cref{Rem:_!_*-for-id}.
The coherent structure isomorphisms $\fun$ and $\un$ of $\cat{F}_!$ are provided by the mates (for the adjunctions $(\cat Fu)_!\dashv \cat Fu$) of the images in $\cat B'$ of those of $\cat{F}$, together with the unique invertible 2-cells induced by the uniqueness property of adjunctions.
The latter also implies that different choices of adjunctions would yield canonically isomorphic pseudo-functors.
Note that, even if we start out with a strict 2-functor $\cat{F}$, there is no reason in general for $\cat{F}_!$ to be strict.

Similarly, a choice of right adjoints $\cat{F}u\adj (\cat{F}u)_*$ for all 1-cells $u$ of $\cat B$ defines a pseudo-functor $\cat{F}_*\colon \cat B^{\op,\co} \to \cat B'$.

Finally, if the 2-functor $\cat{F}\colon \cat{B}\to \cat{B}'$ is such that every 1-cell $\cat{F}u$ admits an ambidextrous adjoint $(\cat{F}u)_!=(\cat{F}u)_*$ we obtain from the above discussion \emph{two} pseudo-functors $\cat{F}_!$ and~$\cat{F}_*\colon \cat B^{\op,\co} \to \cat B'$ which agree on 0-cells and 1-cells but \apriori\ are different on 2-cells. What happens on 2-cells depends on the choices of the units and counits for the left adjunctions $(\cat{F}u)_!\adj \cat{F}u$ versus the choices of the units and counits for the right adjunctions $\cat{F}u\adj (\cat{F}u)_*$. This is the situation we encounter with our Mackey 2-functors $\MM\colon \groupoid^\op \to \CAT$. The requirement that those units and counits can be chosen so that $\cat{F}_!=\cat{F}_*$ on 2-cells is property~\Mack{8}, which itself rests on the Strict Mackey Formula~\Mack{7}, in the Rectification \Cref{Thm:rectification-intro}.
\end{Rem}

The next facts can be established `one object at a time' or in the following more functorial form.
\begin{Lem}
\label{Lem:detect-iso}%
Let $\ell\adj r$ be an adjunction (in a bicategory, see \ref{Ter:internal}).
\begin{enumerate}[\rm(a)]
\item
Horizontal composition $\ell\circ-$ induces an injection on 2-cells $[s,r]\hook [\ell s,\ell r]$.
\smallbreak
\item
If $\theta\colon r\Rightarrow r$ is such that $\ell\theta\colon \ell r\Rightarrow \ell r$ is the identity then $\theta=\id_{r}$ as well.
\smallbreak
\item
If $\theta\colon r\Rightarrow r$ is such that $\ell\theta\colon \ell r\Rightarrow \ell r$ admits a left inverse then so does $\theta$.
\smallbreak
\item
If $\theta\colon r\Rightarrow r$ is such that $\ell\theta\colon \ell r\Rightarrow \ell r$ is an isomorphism then so is $\theta$.
\end{enumerate}
\end{Lem}

\begin{proof}
The mating isomorphism $[s,r]\isoto [\ell s,\Id]$ decomposes as
\[
\xymatrix{[s,r] \ar[r]^-{\ell}
& [\ell s,\ell r] \ar[r]^-{\eps\circ-}
& [\ell s,\Id]
}
\]
hence the first map is a split monomorphism. This proves~(a), from which (b) immediately follows by taking $s=r$ and considering $\id_r,\theta\in [r,r]$. Let us prove~(c). Let $\varphi\colon \ell r\Rightarrow \ell r$ be a left inverse of~$\ell \theta$, meaning $\varphi\circ (\ell\theta)=\id_{\ell r}$. Define $\phi\colon r\Rightarrow r$ as the following mate of~$\varphi$
\[
\phi\colon
\xymatrix@C=15pt@L=1ex{ r \ar@{=>}[r]^-{\eta r} & r\ell r \ar@{=>}[r]^-{r\varphi} & r\ell r \ar@{=>}[r]^-{r\eps} & r} \,.
\]
Then consider $\phi\theta\colon r\Rightarrow r$ and compute $\ell(\phi\theta)$
\[
\xymatrix@C=5em@L=1ex{
\ell r \ar@{=>}[r]^-{\ell\theta} \ar@{=>}[d]_-{\ell \eta r}
&
\ell r \ar@{=>}[r]^-{\ell\phi} \ar@{=>}[d]_-{\ell \eta r}
&
\ell r
\\
\ell r\ell r \ar@{=>}[r]_-{\ell r \ell\theta}
& \ell r\ell r \ar@{=>}[r]_-{\ell r\varphi}
& \ell r\ell r \ar@{=>}[u]_-{\ell r\eps}
}
\]
by unpacking the definition of~$\phi$ (right-hand square) and using naturality for the left-hand square. The bottom composes to the identity by choice of~$\varphi$ and therefore the top composite is the identity as well by the unit-counit relation $(r\eps)\,(\eta r)=\id_r$. Hence $\ell(\phi\theta)=\id_{\ell r}$ and we obtain~(c) thanks to~(b). Now for~(d), we can apply~(c) to find $\phi\colon r\Rightarrow r$ such that $\phi\theta=\id_r$. Hence $\ell\phi=(\ell\theta)\inv$ and therefore $\ell(\theta\phi)=\id_{\ell r}$ and (b) gives us again $\theta\phi=\id_r$. So $\phi$ is the inverse of~$\theta$.
\end{proof}

\begin{Cor}
\label{Cor:detect-iso}%
Let $r$ and $r'$ be two right adjoints (in a bicategory) of the same~$\ell$ and let $\theta\colon r\Rightarrow r'$ be a 2-cell such that $\ell\theta$ is an isomorphism~$\ell r\stackrel{\sim}{\Rightarrow} \ell r'$. Then $\theta\colon r\stackrel{\sim}{\Rightarrow} r'$ is an isomorphism.
\end{Cor}

\begin{proof}
Let $\chi\colon r\stackrel{\sim}{\Rightarrow} r'$ be any isomorphism and apply \Cref{Lem:detect-iso}\,(d) to the 2-cell $\chi\inv\theta\colon r\Rightarrow r$.
\end{proof}

\bigbreak
\section{String diagrams}
\label{sec:string_diagrams}%
\medskip

Instead of the usual cellular or globular pasting diagrams (Remark~\ref{Rem:strictification}), where a $k$-cell is depicted by an oriented `arrow' of dimension~$k$ ($k=0,1,2$), one can compute in a bicategory $\cat B$ by using their planar duals, \emph{string diagrams}.
References for string diagrams include~\cite{Street96}, \cite{JoyalStreet91},~\cite{TuraevVirelizier17}.

They are dual diagrams, in that they represent 0-cells as regions of the plane, 1-cells as lines separating regions, and 2-cells as dots (or boxes) separating lines. So for instance the pasting diagram on the left
\[
\vcenter{\hbox{
\xymatrix@C=3pt@R=14pt{
&& & \ar@{}[d]|{\Scell\;\alpha} & && \\
X \ar@/^6ex/[rrrrrr]^-f
 \ar@/_1ex/[rrd]_k \ar[rrr]^-g && & Y \ar[rrr]^-h \ar@{}[d]|{\Scell\;\beta} & && Z\\
&& U \ar@/_1ex/[rr]_-\ell && V \ar@/_1ex/[rru]_-{m} &&
}
}}
\quad\quad\quad \leftrightsquigarrow \quad\quad\quad
\vcenter {\hbox{
\psfrag{A}[Bc][Bc]{\scalebox{1}{\scriptsize{$\alpha$}}}
\psfrag{B}[Bc][Bc]{\scalebox{1}{\scriptsize{$\beta$}}}
\psfrag{X}[Bc][Bc]{\scalebox{1}{\scriptsize{$X$}}}
\psfrag{Y}[Bc][Bc]{\scalebox{1}{\scriptsize{$Y$}}}
\psfrag{Z}[Bc][Bc]{\scalebox{1}{\scriptsize{$Z$}}}
\psfrag{U}[Bc][Bc]{\scalebox{1}{\scriptsize{$U$}}}
\psfrag{V}[Bc][Bc]{\scalebox{1}{\scriptsize{$V$}}}
\psfrag{F}[Bc][Bc]{\scalebox{1}{\scriptsize{$f$}}}
\psfrag{G}[Bc][Bc]{\scalebox{1}{\scriptsize{$g$}}}
\psfrag{H}[Bc][Bc]{\scalebox{1}{\scriptsize{$h$}}}
\psfrag{K}[Bc][Bc]{\scalebox{1}{\scriptsize{$k$}}}
\psfrag{L}[Bc][Bc]{\scalebox{1}{\scriptsize{$\ell$}}}
\psfrag{M}[Bc][Bc]{\scalebox{1}{\scriptsize{$m$}}}
\includegraphics[scale=.4]{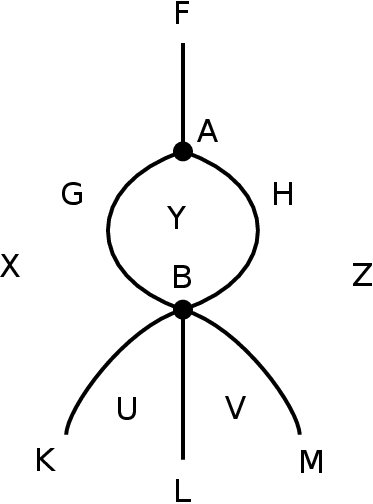}
}}
\]
can be replaced by the string diagram on the right to represent the same composite 2-cell of~$\cat B$. Instead of indicating sources and targets by orienting cells ($\to$, $\Rightarrow$), string diagrams typically rely on the left-right and top-down directions of the page. As above, we choose to orient 1-cells left-to-right and 2-cells top-to-bottom.
Just as for pasting diagrams, the consistency in general bicategories of the calculus of string diagrams --- in which horizontal composition necessarily appears to be strictly associative and unital --- relies on the coherence theorem (\Cref{Rem:strictification}).

An advantage of computing with string diagrams is that identity 1-cells can be safely omitted most of the time.
Moreover, with strings, many compatibility and coherence axioms simply say that certain dots may slide along past certain others.
Thus relations between string diagrams often take an intuitive geometric form.

\begin{Exa} \label{Exa:identity-strings}
An identity 2-cell $\id_f$ (for $f\colon X\to Y$) and an identity 1-cell $\Id_X$ may take any of the following successively more inconspicuous string forms:
\[
\vcenter { \hbox{
\psfrag{A}[Bc][Bc]{\scalebox{1}{\scriptsize{$f$}}}
\psfrag{B}[Bc][Bc]{\scalebox{1}{\scriptsize{$f$}}}
\psfrag{X}[Bc][Bc]{\scalebox{1}{\scriptsize{$X$}}}
\psfrag{Y}[Bc][Bc]{\scalebox{1}{\scriptsize{$Y$}}}
\psfrag{F}[Bc][Bc]{\scalebox{1}{\scriptsize{$\id_f$}}}
\includegraphics[scale=.4]{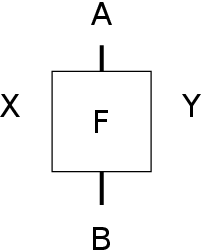}
}}
\quad = \quad
\vcenter { \hbox{
\psfrag{A}[Bc][Bc]{\scalebox{1}{\scriptsize{$f$}}}
\psfrag{B}[Bc][Bc]{\scalebox{1}{\scriptsize{$f$}}}
\psfrag{F}[Bc][Bc]{\scalebox{1}{\scriptsize{\;\;\;$\id_f$}}}
\includegraphics[scale=.4]{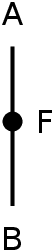}
}}
\quad = \quad
\vcenter { \hbox{
\psfrag{A}[Bc][Bc]{\scalebox{1}{\scriptsize{$f$}}}
\psfrag{B}[Bc][Bc]{\scalebox{1}{\scriptsize{$f$}}}
\psfrag{F}[Bc][Bc]{\scalebox{1}{\scriptsize{\;\;$\id_f$}}}
\includegraphics[scale=.4]{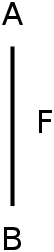}
}}
\quad\quad\quad\quad\quad\quad
\vcenter { \hbox{
\psfrag{A}[Bc][Bc]{\scalebox{1}{\scriptsize{$\Id_X$}}}
\psfrag{B}[Bc][Bc]{\scalebox{1}{\scriptsize{$\Id_X$}}}
\psfrag{X}[Bc][Bc]{\scalebox{1}{\scriptsize{$X$}}}
\psfrag{Y}[Bc][Bc]{\scalebox{1}{\scriptsize{$X$}}}
\includegraphics[scale=.4]{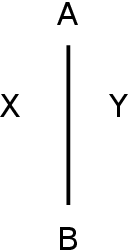}
}}
\quad = \quad
\vcenter { \hbox{
\psfrag{A}[Bc][Bc]{\scalebox{1}{\scriptsize{$\Id_X$}}}
\psfrag{B}[Bc][Bc]{\scalebox{1}{\scriptsize{$\Id_X$}}}
\psfrag{X}[Bc][Bc]{\scalebox{1}{\scriptsize{$X$}}}
\psfrag{Y}[Bc][Bc]{\scalebox{1}{\scriptsize{$X$}}}
\includegraphics[scale=.4]{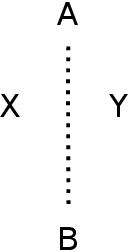}
}}
\quad =
\quad\quad\quad
\]
\end{Exa}

\begin{Exa}[Exchange law] \label{Exa:strings-for-exchange}
A special case of~\eqref{eq:exchange_law} yields the relation
\[
\vcenter { \hbox{
\psfrag{A1}[Bc][Bc]{\scalebox{1}{\scriptsize{$f_1$}}}
\psfrag{B1}[Bc][Bc]{\scalebox{1}{\scriptsize{$g_1$}}}
\psfrag{A2}[Bc][Bc]{\scalebox{1}{\scriptsize{$f_2$}}}
\psfrag{B2}[Bc][Bc]{\scalebox{1}{\scriptsize{$g_2$}}}
\psfrag{F}[Bc][Bc]{\scalebox{1}{\scriptsize{$\alpha$}}}
\psfrag{G}[Bc][Bc]{\scalebox{1}{\scriptsize{$\beta$}}}
\includegraphics[scale=.4]{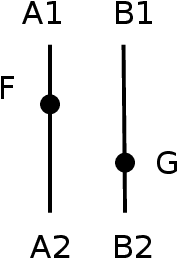}
}}
\quad = \quad
\vcenter { \hbox{
\psfrag{A1}[Bc][Bc]{\scalebox{1}{\scriptsize{$f_1$}}}
\psfrag{B1}[Bc][Bc]{\scalebox{1}{\scriptsize{$g_1$}}}
\psfrag{A2}[Bc][Bc]{\scalebox{1}{\scriptsize{$f_2$}}}
\psfrag{B2}[Bc][Bc]{\scalebox{1}{\scriptsize{$g_2$}}}
\psfrag{F}[Bc][Bc]{\scalebox{1}{\scriptsize{$\alpha$}}}
\psfrag{G}[Bc][Bc]{\scalebox{1}{\scriptsize{$\beta$}}}
\includegraphics[scale=.4]{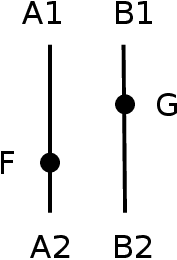}
}}
\]
which suggests that parallel blocks may slide past each other.
\end{Exa}

\begin{Exa}[Adjoint functors] \label{Exa:strings-for-adjoints}
For instance, the unit $\eta\colon \Id_X\Rightarrow r\ell$ and counit $\varepsilon \colon \ell r \Rightarrow \Id_Y$ of an adjunction $\ell : X \leftrightarrows Y: r$ may be depicted by either of these successively simpler diagrams:
\[
\vcenter {\hbox{
\psfrag{F}[Bc][Bc]{\scalebox{1}{\scriptsize{$\eta$}}}
\psfrag{Id}[Bc][Bc]{\scalebox{1}{\scriptsize{$\Id_X$}}}
\psfrag{X}[Bc][Bc]{\scalebox{1}{\scriptsize{$X$}}}
\psfrag{Y}[Bc][Bc]{\scalebox{1}{\scriptsize{$Y$}}}
\psfrag{L}[Bc][Bc]{\scalebox{1}{\scriptsize{$\ell$}}}
\psfrag{R}[Bc][Bc]{\scalebox{1}{\scriptsize{$r$}}}
\includegraphics[scale=.4]{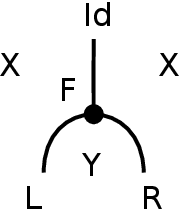}
}}
\;=\;
\vcenter {\hbox{
\psfrag{F}[Bc][Bc]{\scalebox{1}{\scriptsize{$\eta$}}}
\psfrag{Id}[Bc][Bc]{\scalebox{1}{\scriptsize{$\Id_X$}}}
\psfrag{X}[Bc][Bc]{\scalebox{1}{\scriptsize{$X$}}}
\psfrag{Y}[Bc][Bc]{\scalebox{1}{\scriptsize{$Y$}}}
\psfrag{L}[Bc][Bc]{\scalebox{1}{\scriptsize{$\ell$}}}
\psfrag{R}[Bc][Bc]{\scalebox{1}{\scriptsize{$r$}}}
\includegraphics[scale=.4]{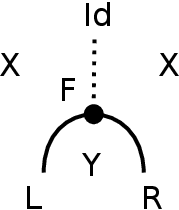}
}}
\;=\;
\vcenter {\hbox{
\psfrag{A}[Bc][Bc]{\scalebox{1}{\scriptsize{$\ell$}}}
\psfrag{B}[Bc][Bc]{\scalebox{1}{\scriptsize{$r$}}}
\psfrag{F}[Bc][Bc]{\scalebox{1}{\scriptsize{$\eta$}}}
\includegraphics[scale=.4]{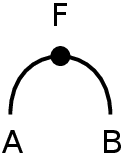}
}}
\quad\quad\quad
\vcenter {\hbox{
\psfrag{F}[Bc][Bc]{\scalebox{1}{\scriptsize{$\varepsilon$}}}
\psfrag{Id}[Bc][Bc]{\scalebox{1}{\scriptsize{$\Id_Y$}}}
\psfrag{X}[Bc][Bc]{\scalebox{1}{\scriptsize{$X$}}}
\psfrag{Y}[Bc][Bc]{\scalebox{1}{\scriptsize{$Y$}}}
\psfrag{L}[Bc][Bc]{\scalebox{1}{\scriptsize{$\ell$}}}
\psfrag{R}[Bc][Bc]{\scalebox{1}{\scriptsize{$r$}}}
\includegraphics[scale=.4]{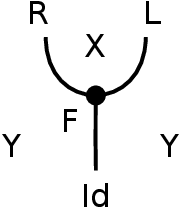}
}}
\;=\;
\vcenter {\hbox{
\psfrag{F}[Bc][Bc]{\scalebox{1}{\scriptsize{$\varepsilon$}}}
\psfrag{Id}[Bc][Bc]{\scalebox{1}{\scriptsize{$\Id_Y$}}}
\psfrag{X}[Bc][Bc]{\scalebox{1}{\scriptsize{$X$}}}
\psfrag{Y}[Bc][Bc]{\scalebox{1}{\scriptsize{$Y$}}}
\psfrag{L}[Bc][Bc]{\scalebox{1}{\scriptsize{$\ell$}}}
\psfrag{R}[Bc][Bc]{\scalebox{1}{\scriptsize{$r$}}}
\includegraphics[scale=.4]{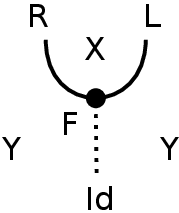}
}}
\;=\;
\vcenter {\hbox{
\psfrag{A}[Bc][Bc]{\scalebox{1}{\scriptsize{$r$}}}
\psfrag{B}[Bc][Bc]{\scalebox{1}{\scriptsize{$\ell$}}}
\psfrag{F}[cc][Bc]{\scalebox{1}{\scriptsize{$\varepsilon$}}}
\includegraphics[scale=.4]{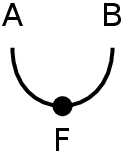}
}}
\]
The two triangular equations for this adjunction become
\[
\vcenter { \hbox{
\psfrag{A}[Bc][Bc]{\scalebox{1}{\scriptsize{$r$}}}
\psfrag{B}[Bc][Bc]{\scalebox{1}{\scriptsize{$r$}}}
\psfrag{C}[Bc][Bc]{\scalebox{1}{\scriptsize{$\ell$}}}
\psfrag{F}[Bc][Bc]{\scalebox{1}{\scriptsize{$\varepsilon$}}}
\psfrag{G}[Bc][Bc]{\scalebox{1}{\scriptsize{$\eta$}}}
\includegraphics[scale=.4]{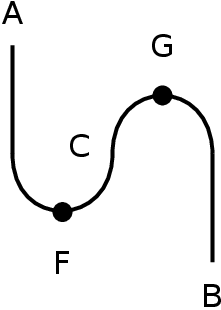}
}}
\quad =\quad
\vcenter { \hbox{
\psfrag{A}[Bc][Bc]{\scalebox{1}{\scriptsize{$r$}}}
\includegraphics[scale=.4]{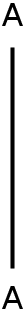}
}}
\quad\quad
\quad \quad
\quad\quad
\vcenter { \hbox{
\psfrag{A}[Bc][Bc]{\scalebox{1}{\scriptsize{$\ell$}}}
\psfrag{B}[Bc][Bc]{\scalebox{1}{\scriptsize{$\ell$}}}
\psfrag{C}[Bc][Bc]{\scalebox{1}{\scriptsize{$r$}}}
\psfrag{F}[Bc][Bc]{\scalebox{1}{\scriptsize{$\varepsilon$}}}
\psfrag{G}[Bc][Bc]{\scalebox{1}{\scriptsize{$\eta$}}}
\includegraphics[scale=.4]{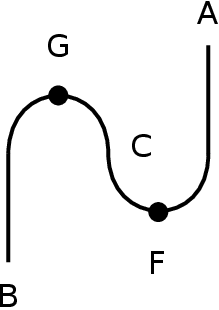}
}}
\quad = \quad
\vcenter { \hbox{
\psfrag{A}[Bc][Bc]{\scalebox{1}{\scriptsize{$\ell$}}}
\includegraphics[scale=.4]{anc/identity-snakelength.eps}
}}
\]
which suggest that unit-counit pairs may be straightened by pulling the string.
\end{Exa}

\begin{Exa}[Mates]
The calculus of mates recalled in \Cref{sec:mates} has a nice formulation in terms of string diagrams. For a 2-cell
\[
\vcenter { \hbox{
\xymatrix@C=14pt@R=14pt{
& & \\
 \ar@{}[rr]|{\oEcell{\alpha}}
 \ar[ur]^{v^*} && \ar[ul]_{j^*} \\
&\ar[ur]_{u^*} \ar[ul]^{i^*} &
}
}}
\quad\quad \rightsquigarrow \quad\quad
\vcenter { \hbox{
\psfrag{A}[Bc][Bc]{\scalebox{1}{\scriptsize{$i^*$}}}
\psfrag{B}[Bc][Bc]{\scalebox{1}{\scriptsize{$v^*$}}}
\psfrag{C}[Bc][Bc]{\scalebox{1}{\scriptsize{$u^*$}}}
\psfrag{D}[Bc][Bc]{\scalebox{1}{\scriptsize{$j^*$}}}
\psfrag{F}[Bc][Bc]{\scalebox{1}{\scriptsize{\;\;\;$\alpha$}}}
\includegraphics[scale=.4]{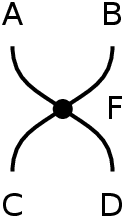}
}}
\]
and adjunctions $i_!\dashv i^* \dashv i_*$ and $j_!\dashv j^* \dashv j_*$, as typically appear in this work, the mates $\alpha_!$ and $\alpha_*$ are depicted as
\[
\alpha_!
\quad=\quad
\vcenter { \hbox{
\psfrag{A}[Bc][Bc]{\scalebox{1}{\scriptsize{$v^*$}}}
\psfrag{B}[Bc][Bc]{\scalebox{1}{\scriptsize{$j_!$}}}
\psfrag{C}[Bc][Bc]{\scalebox{1}{\scriptsize{$i_!$}}}
\psfrag{D}[Bc][Bc]{\scalebox{1}{\scriptsize{$u^*$}}}
\psfrag{F}[Bc][Bc]{\scalebox{1}{\scriptsize{$\varepsilon$}}}
\psfrag{G}[Bc][Bc]{\scalebox{1}{\scriptsize{$\eta$}}}
\psfrag{H}[Bc][Bc]{\scalebox{1}{\scriptsize{$\alpha$}}}
\includegraphics[scale=.4]{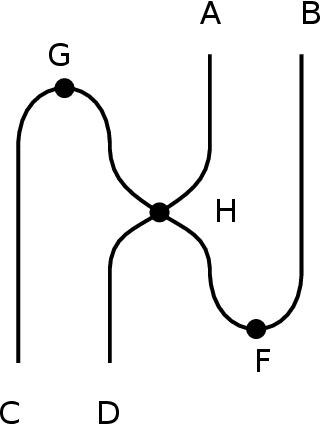}
}}
\quad\quad\quad
\textrm{and}
\quad\quad\quad
(\alpha\inv)_*
\quad=\quad
\vcenter { \hbox{
\psfrag{A}[Bc][Bc]{\scalebox{1}{\scriptsize{$i_*$}}}
\psfrag{B}[Bc][Bc]{\scalebox{1}{\scriptsize{$u^*$}}}
\psfrag{C}[Bc][Bc]{\scalebox{1}{\scriptsize{$v^*$}}}
\psfrag{D}[Bc][Bc]{\scalebox{1}{\scriptsize{$j_*$}}}
\psfrag{F}[Bc][Bc]{\scalebox{1}{\scriptsize{$\varepsilon$}}}
\psfrag{G}[Bc][Bc]{\scalebox{1}{\scriptsize{$\eta$}}}
\psfrag{H}[Bc][Bc]{\scalebox{1}{\scriptsize{$\;\;\;\alpha^{-1}$}}}
\includegraphics[scale=.4]{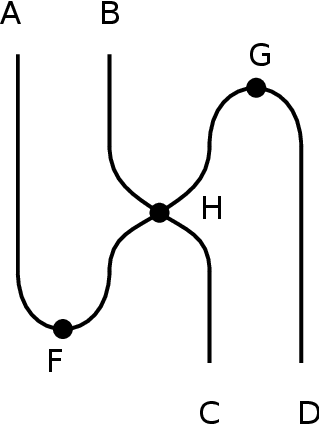}
}}
\]
In particular, the special case with $u=\Id$ and $v=\Id$ (which gives rise for instance to the pseudo-functoriality of $(-)_!$ and $(-)_*$ as in \Cref{Rem:pseudo-func-of-adjoints}) becomes simply:
\[
\vcenter { \hbox{
\psfrag{B}[Bc][Bc]{\scalebox{1}{\scriptsize{$j_!$}}}
\psfrag{C}[Bc][Bc]{\scalebox{1}{\scriptsize{$i_!$}}}
\psfrag{F}[Bc][Bc]{\scalebox{1}{\scriptsize{$\varepsilon$}}}
\psfrag{G}[Bc][Bc]{\scalebox{1}{\scriptsize{$\eta$}}}
\psfrag{H}[Bc][Bc]{\scalebox{1}{\scriptsize{$\alpha$}}}
\includegraphics[scale=.4]{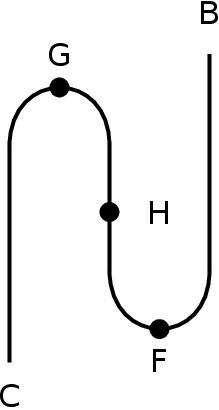}
}}
\quad\quad\quad
\textrm{and}
\quad\quad\quad
\vcenter { \hbox{
\psfrag{A}[Bc][Bc]{\scalebox{1}{\scriptsize{$i_*$}}}
\psfrag{D}[Bc][Bc]{\scalebox{1}{\scriptsize{$j_*$}}}
\psfrag{F}[Bc][Bc]{\scalebox{1}{\scriptsize{$\varepsilon$}}}
\psfrag{G}[Bc][Bc]{\scalebox{1}{\scriptsize{$\eta$}}}
\psfrag{H}[Bc][Bc]{\scalebox{1}{\scriptsize{$\;\alpha^{-1}$}}}
\includegraphics[scale=.4]{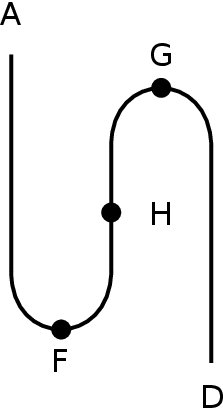}
}}
\]
It is an instructive exercise to verify all the claims in \Cref{sec:mates} using string notation.
\end{Exa}

\begin{Exa}[Pseudo-functors] \label{Exa:strings-for-fun}
The strings for the two structural isomorphisms of a pseudo-functor $\cat F\colon \cat B\to \cat B'$ (see \Cref{Ter:pseudofun}) take the following form:
\[
\vcenter {\hbox{
\psfrag{F}[Bc][Bc]{\scalebox{1}{\scriptsize{$\un \;\;$}}}
\psfrag{A}[Bc][Bc]{\scalebox{1}{\scriptsize{$\Id_{\cat F X}$}}}
\psfrag{B}[Bc][Bc]{\scalebox{1}{\scriptsize{$\cat F (\Id_X)$}}}
\includegraphics[scale=.4]{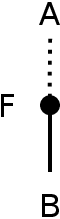}
}}
\quad\quad\quad\quad\quad\quad\quad
\vcenter {\hbox{
\psfrag{F}[Bc][Bc]{\scalebox{1}{\scriptsize{$\fun \;\;\;$}}}
\psfrag{C}[Bc][Bc]{\scalebox{1}{\scriptsize{$\cat F(g f)$}}}
\psfrag{B}[Bc][Bc]{\scalebox{1}{\scriptsize{$\cat F g$}}}
\psfrag{A}[Bc][Bc]{\scalebox{1}{\scriptsize{$\cat F f$}}}
\includegraphics[scale=.4]{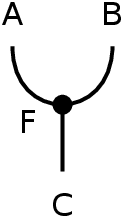}
}}
\]
Of the coherence axioms, the unital relations for a 1-cell $f\colon X\to Y$ of~$\cat B$ become
\[
\vcenter {\hbox{
\psfrag{F}[Bc][Bc]{\scalebox{1}{\scriptsize{$\un$}}}
\psfrag{G}[Bc][Bc]{\scalebox{1}{\scriptsize{$\;\;\;\;\fun$}}}
\psfrag{H}[Bc][Bc]{\scalebox{1}{\scriptsize{$\;\;\;\;\;\;\cat F(\run)$}}}
\psfrag{A}[Bc][Bc]{\scalebox{1}{\scriptsize{$\cat F f$}}}
\psfrag{B}[Bc][Bc]{\scalebox{1}{\scriptsize{$\Id_{\cat FX}$}}}
\psfrag{C}[Bc][Bc]{\scalebox{1}{\scriptsize{$\cat F (\Id_X)\;\;\;\;\;\;\;$}}}
\psfrag{D}[Bc][Bc]{\scalebox{1}{\scriptsize{$\cat F (f \Id_X)\;\;\;\;\;\;\;\;\;\;$}}}
\psfrag{E}[Bc][Bc]{\scalebox{1}{\scriptsize{$\cat F f$}}}
\includegraphics[scale=.4]{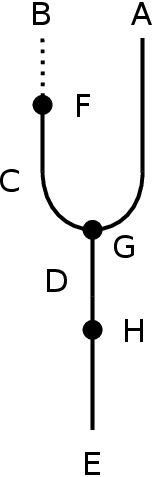}
}}
\quad \;=\;
\vcenter {\hbox{
\psfrag{F}[Bc][Bc]{\scalebox{1}{\scriptsize{$\;\;\run$}}}
\psfrag{A}[Bc][Bc]{\scalebox{1}{\scriptsize{$\cat F f$}}}
\psfrag{B}[Bc][Bc]{\scalebox{1}{\scriptsize{$\Id_{\cat FX}$}}}
\psfrag{C}[Bc][Bc]{\scalebox{1}{\scriptsize{$\cat F f$}}}
\includegraphics[scale=.4]{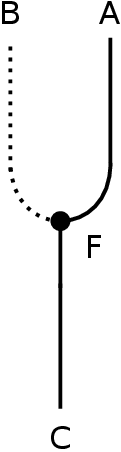}
}}
\quad\quad\quad\quad\quad\quad\quad
\vcenter {\hbox{
\psfrag{F}[Bc][Bc]{\scalebox{1}{\scriptsize{$\lun\;\;$}}}
\psfrag{A}[Bc][Bc]{\scalebox{1}{\scriptsize{$\cat F f$}}}
\psfrag{B}[Bc][Bc]{\scalebox{1}{\scriptsize{$\Id_{\cat FY}$}}}
\psfrag{C}[Bc][Bc]{\scalebox{1}{\scriptsize{$\cat F f$}}}
\includegraphics[scale=.4]{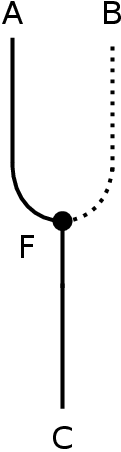}
}}
\;=\; \quad
\vcenter {\hbox{
\psfrag{F}[Bc][Bc]{\scalebox{1}{\scriptsize{$\un$}}}
\psfrag{G}[Bc][Bc]{\scalebox{1}{\scriptsize{$\fun\;\;\;$}}}
\psfrag{H}[Bc][Bc]{\scalebox{1}{\scriptsize{$\cat F(\lun)\;\;\;\;\;\;$}}}
\psfrag{A}[Bc][Bc]{\scalebox{1}{\scriptsize{$\cat F f$}}}
\psfrag{B}[Bc][Bc]{\scalebox{1}{\scriptsize{$\Id_{\cat FY}$}}}
\psfrag{C}[Bc][Bc]{\scalebox{1}{\scriptsize{$\;\;\;\;\;\;\;\cat F (\Id_Y)$}}}
\psfrag{D}[Bc][Bc]{\scalebox{1}{\scriptsize{$\;\;\;\;\;\;\;\;\;\;\cat F (\Id_Y f)$}}}
\psfrag{E}[Bc][Bc]{\scalebox{1}{\scriptsize{$\cat F f$}}}
\includegraphics[scale=.4]{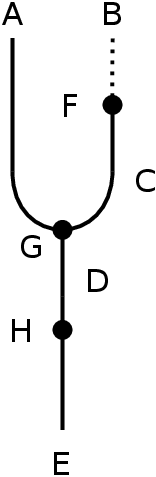}
}}
\]
and the associativity relation for three composable 1-cells $\stackrel{f}{\to} \stackrel{g}{\to} \stackrel{h}{\to}$ is:
\[
\vcenter {\hbox{
\psfrag{F}[Bc][Bc]{\scalebox{1}{\scriptsize{$\fun \;\;\;$}}}
\psfrag{G}[Bc][Bc]{\scalebox{1}{\scriptsize{$\;\;\fun$}}}
\psfrag{A}[Bc][Bc]{\scalebox{1}{\scriptsize{$\cat F f$}}}
\psfrag{B}[Bc][Bc]{\scalebox{1}{\scriptsize{$\cat F g$}}}
\psfrag{C}[Bc][Bc]{\scalebox{1}{\scriptsize{$\cat F h$}}}
\psfrag{D}[Bc][Bc]{\scalebox{1}{\scriptsize{$\cat F (gf)\;\;\;\;\;$}}}
\psfrag{E}[Bc][Bc]{\scalebox{1}{\scriptsize{$\cat F (h(gf))$}}}
\includegraphics[scale=.4]{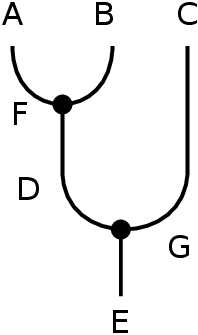}
}}
\quad\;=\;\quad
\vcenter {\hbox{
\psfrag{F}[Bc][Bc]{\scalebox{1}{\scriptsize{$\;\;\;\;\fun$}}}
\psfrag{G}[Bc][Bc]{\scalebox{1}{\scriptsize{$\fun\;\;$}}}
\psfrag{H}[Bc][Bc]{\scalebox{1}{\scriptsize{$\cat F(\ass)\;\;\;\;\;\;\;\;$}}}
\psfrag{A}[Bc][Bc]{\scalebox{1}{\scriptsize{$\cat F f$}}}
\psfrag{B}[Bc][Bc]{\scalebox{1}{\scriptsize{$\cat F g$}}}
\psfrag{C}[Bc][Bc]{\scalebox{1}{\scriptsize{$\cat F h$}}}
\psfrag{D}[Bc][Bc]{\scalebox{1}{\scriptsize{$\;\;\;\;\;\cat F (hg)$}}}
\psfrag{E}[Bc][Bc]{\scalebox{1}{\scriptsize{$\cat F (h(gf))$}}}
\psfrag{E'}[Bc][Bc]{\scalebox{1}{\scriptsize{$\;\;\;\;\;\;\;\;\;\;\cat F ((hg)f)$}}}
\includegraphics[scale=.4]{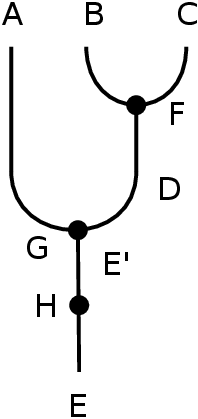}
}}
\]
Here we have taken the trouble to explicitly write the images $\cat F(\run)$, $\cat F(\lun)$ and $\cat F(\ass)$ of the unitors and associators of~$\cat B$ (as well as the unitors of~$\cat B'$). But in string notation it is safe to omit them, just as one would typically omit identity 1-cells, unitors and associators of the ambient bicategory~$\cat B'$: it is always straightforward to reintroduce them explicitly when necessary. Then the relations simplify to
\begin{equation} \label{eq:fun-unit-simpler}
\vcenter {\hbox{
\psfrag{F}[Bc][Bc]{\scalebox{1}{\scriptsize{$\un$}}}
\psfrag{G}[Bc][Bc]{\scalebox{1}{\scriptsize{$\;\;\;\;\fun$}}}
\psfrag{A}[Bc][Bc]{\scalebox{1}{\scriptsize{$\cat F f$}}}
\psfrag{C}[Bc][Bc]{\scalebox{1}{\scriptsize{$\cat F (\Id_X)\;\;\;\;\;\;\;$}}}
\psfrag{D}[Bc][Bc]{\scalebox{1}{\scriptsize{$\cat F f$}}}
\includegraphics[scale=.4]{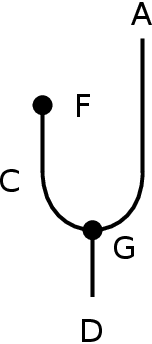}
}}
\quad\;=\;\quad
\vcenter {\hbox{
\psfrag{A}[Bc][Bc]{\scalebox{1}{\scriptsize{$\cat F f$}}}
\includegraphics[scale=.4]{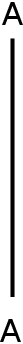}
}}
\quad\;=\;\quad
\vcenter {\hbox{
\psfrag{F}[Bc][Bc]{\scalebox{1}{\scriptsize{$\un$}}}
\psfrag{G}[Bc][Bc]{\scalebox{1}{\scriptsize{$\fun\;\;\;$}}}
\psfrag{A}[Bc][Bc]{\scalebox{1}{\scriptsize{$\cat F f$}}}
\psfrag{C}[Bc][Bc]{\scalebox{1}{\scriptsize{$\;\;\;\;\;\;\;\cat F (\Id_Y)$}}}
\psfrag{D}[Bc][Bc]{\scalebox{1}{\scriptsize{$\cat F f$}}}
\includegraphics[scale=.4]{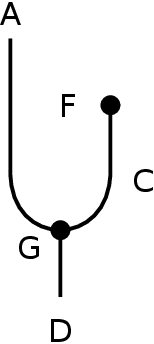}
}}
\end{equation}
and
\begin{equation} \label{eq:fun-ass-simpler}
\vcenter {\hbox{
\psfrag{F}[Bc][Bc]{\scalebox{1}{\scriptsize{$\fun \;\;\;$}}}
\psfrag{G}[Bc][Bc]{\scalebox{1}{\scriptsize{$\;\;\fun$}}}
\psfrag{A}[Bc][Bc]{\scalebox{1}{\scriptsize{$\cat F f$}}}
\psfrag{B}[Bc][Bc]{\scalebox{1}{\scriptsize{$\cat F g$}}}
\psfrag{C}[Bc][Bc]{\scalebox{1}{\scriptsize{$\cat F h$}}}
\psfrag{D}[Bc][Bc]{\scalebox{1}{\scriptsize{$\cat F (gf)\;\;\;\;\;$}}}
\psfrag{E}[Bc][Bc]{\scalebox{1}{\scriptsize{$\cat F (hgf)$}}}
\includegraphics[scale=.4]{anc/mult-ass-dot-right.eps}
}}
\quad\;=\;\quad
\vcenter {\hbox{
\psfrag{F}[Bc][Bc]{\scalebox{1}{\scriptsize{$\;\;\;\;\fun$}}}
\psfrag{G}[Bc][Bc]{\scalebox{1}{\scriptsize{$\fun\;\;$}}}
\psfrag{A}[Bc][Bc]{\scalebox{1}{\scriptsize{$\cat F f$}}}
\psfrag{B}[Bc][Bc]{\scalebox{1}{\scriptsize{$\cat F g$}}}
\psfrag{C}[Bc][Bc]{\scalebox{1}{\scriptsize{$\cat F h$}}}
\psfrag{D}[Bc][Bc]{\scalebox{1}{\scriptsize{$\;\;\;\;\;\cat F (hg)$}}}
\psfrag{E}[Bc][Bc]{\scalebox{1}{\scriptsize{$\cat F (hgf)$}}}
\includegraphics[scale=.4]{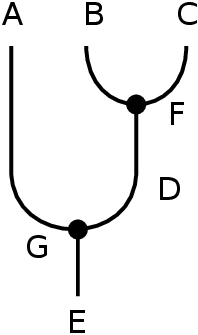}
}}
\end{equation}
(Of course, whenever $\cat B$ and $\cat B'$ are 2-categories the axioms take the latter form on the nose.) Because of these axioms, for every equality $f_1\circ \ldots\circ f_n = g_1 \circ \ldots\circ g_m$ (for some bracketings) between composite 1-cells of $\cat B$ there is only one way to go from $\cat F(f_1) \circ \ldots \circ \cat F(f_n)$ to $\cat F(g_1) \circ \ldots \circ \cat F(g_m)$ by combining instances of $\fun$ and $\un$, so we may as well represent this canonical 2-cell by a single dot:
\begin{equation} \label{eq:general-ass}
\vcenter {\hbox{
\psfrag{A}[Bc][Bc]{\scalebox{1}{\scriptsize{$\cat F f_1$}}}
\psfrag{B}[Bc][Bc]{\scalebox{1}{\scriptsize{$\cdots$}}}
\psfrag{C}[Bc][Bc]{\scalebox{1}{\scriptsize{$\cat F f_n$}}}
\psfrag{A'}[Bc][Bc]{\scalebox{1}{\scriptsize{$\cat F g_1$}}}
\psfrag{C'}[Bc][Bc]{\scalebox{1}{\scriptsize{$\cat F g_m$}}}
\includegraphics[scale=.4]{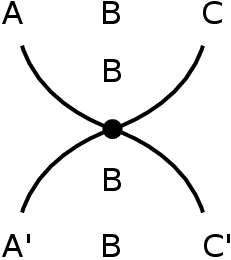}
}}
\end{equation}
\end{Exa}

\begin{Exa} \label{Exa:strings-for-natural-fun}
Consider again a pseudo-functor $\cat F = (\cat F, \fun, \un)\colon \cat B \to \cat B'$. The fact that $\fun$ is a natural transformation translates into the string equation
\[
\vcenter {\hbox{
\psfrag{F}[Bc][Bc]{\scalebox{1}{\scriptsize{$\fun\;\;$}}}
\psfrag{G}[Bc][Bc]{\scalebox{1}{\scriptsize{$\cat F(\psi \, \varphi)$}}}
\psfrag{A}[Bc][Bc]{\scalebox{1}{\scriptsize{$\cat F f$}}}
\psfrag{B}[Bc][Bc]{\scalebox{1}{\scriptsize{$\cat F g$}}}
\psfrag{C}[Bc][Bc]{\scalebox{1}{\scriptsize{$\;\;\;\;\;\;\;\cat F (g f)$}}}
\psfrag{D}[Bc][Bc]{\scalebox{1}{\scriptsize{$\cat F (g' f')$}}}
\includegraphics[scale=.4]{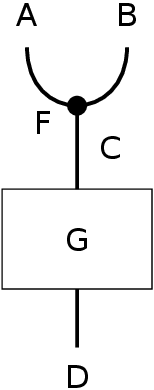}
}}
\quad=\quad
\vcenter {\hbox{
\psfrag{F}[Bc][Bc]{\scalebox{1}{\scriptsize{$\cat F \varphi$}}}
\psfrag{G}[Bc][Bc]{\scalebox{1}{\scriptsize{$\cat F \psi$}}}
\psfrag{H}[Bc][Bc]{\scalebox{1}{\scriptsize{$\fun\;\;$}}}
\psfrag{A}[Bc][Bc]{\scalebox{1}{\scriptsize{$\cat F f$}}}
\psfrag{B}[Bc][Bc]{\scalebox{1}{\scriptsize{$\cat F g$}}}
\psfrag{C}[Bc][Bc]{\scalebox{1}{\scriptsize{$\cat F f'$\;\;\;}}}
\psfrag{D}[Bc][Bc]{\scalebox{1}{\scriptsize{$\;\;\;\cat F g'$}}}
\psfrag{E}[Bc][Bc]{\scalebox{1}{\scriptsize{$\cat F (g' f')$}}}
\includegraphics[scale=.4]{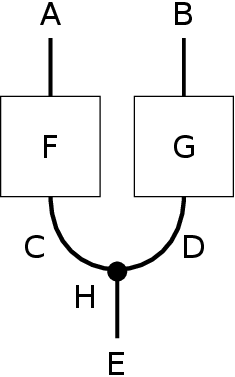}
}}
\]
for any pair of 2-cells $\varphi\colon f \Rightarrow f'$ and $\psi\colon g \Rightarrow g'$ in~$\cat B$.
\end{Exa}

\begin{Exa}[Transformations]
\label{Exa:trans}%
Consider an oplax transformation $t\colon \cat F\Rightarrow \cat G$ between pseudo-functors, as in \Cref{Ter:Hom_bicats}. Its components look as
\[
\vcenter { \hbox{
\psfrag{A}[Bc][Bc]{\scalebox{1}{\scriptsize{$t_X$}}}
\psfrag{B}[Bc][Bc]{\scalebox{1}{}}
\psfrag{X}[Bc][Bc]{\scalebox{1}{\scriptsize{$\cat F X$}}}
\psfrag{Y}[Bc][Bc]{\scalebox{1}{\scriptsize{$\cat G X$}}}
\includegraphics[scale=.4]{anc/line-identity-objects.eps}
}}
\quad\qquadtext{and}\qquad
\vcenter {\hbox{
\psfrag{F}[Bc][Bc]{\scalebox{1}{\scriptsize{$t(u)$}}}
\psfrag{A}[Bc][Bc]{\scalebox{1}{\scriptsize{$t_X$}}}
\psfrag{B}[Bc][Bc]{\scalebox{1}{\scriptsize{$\cat G u$}}}
\psfrag{C}[Bc][Bc]{\scalebox{1}{\scriptsize{$\cat F u$}}}
\psfrag{D}[Bc][Bc]{\scalebox{1}{\scriptsize{$t_Y$}}}
\includegraphics[scale=.4]{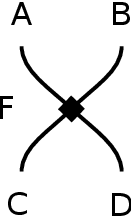}
}}
\]
for each 0-cell~$X$ and for each 1-cell $u\colon X\to Y$.
The functoriality axioms say that we may slide $\un$ and $\fun$ past~$t$
\[
\vcenter {\hbox{
\psfrag{F}[Bc][Bc]{\scalebox{1}{\scriptsize{$t(\Id)$}}}
\psfrag{G}[Bc][Bc]{\scalebox{1}{\scriptsize{$\un^{\!-1}$}}}
\psfrag{A}[Bc][Bc]{\scalebox{1}{\scriptsize{$t_X$}}}
\psfrag{B}[Bc][Bc]{\scalebox{1}{\scriptsize{$\cat G \Id$}}}
\psfrag{C}[Bc][Bc]{\scalebox{1}{\scriptsize{$\Id$}}}
\psfrag{D}[Bc][Bc]{\scalebox{1}{\scriptsize{$t_X$}}}
\psfrag{E}[Bc][Bc]{\scalebox{1}{\scriptsize{$\cat F \Id$}}}
\includegraphics[scale=.4]{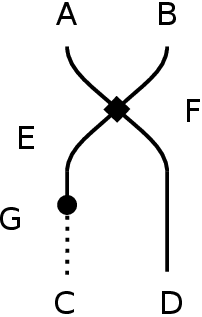}
}}
\quad = \quad
\vcenter {\hbox{
\psfrag{F}[Bc][Bc]{\scalebox{1}{\scriptsize{$t(\Id)$}}}
\psfrag{G}[Bc][Bc]{\scalebox{1}{\scriptsize{$\un^{\!-1}$}}}
\psfrag{A}[Bc][Bc]{\scalebox{1}{\scriptsize{$t_X$}}}
\psfrag{B}[Bc][Bc]{\scalebox{1}{\scriptsize{$\cat G \Id$}}}
\psfrag{C}[Bc][Bc]{\scalebox{1}{\scriptsize{$\Id$}}}
\psfrag{D}[Bc][Bc]{\scalebox{1}{\scriptsize{$t_X$}}}
\includegraphics[scale=.4]{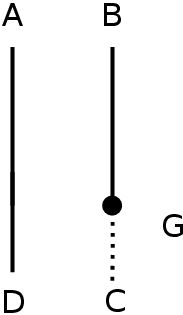}
}}
\quad \quad \textrm{and} \quad\quad
\vcenter {\hbox{
\psfrag{F}[Bc][Bc]{\scalebox{1}{\scriptsize{$t(u)$}}}
\psfrag{H}[Bc][Bc]{\scalebox{1}{\scriptsize{$t(v)$}}}
\psfrag{G}[Bc][Bc]{\scalebox{1}{\scriptsize{$\fun^{\!-1}$}}}
\psfrag{A}[Bc][Bc]{\scalebox{1}{\scriptsize{$t_X$}}}
\psfrag{B}[Bc][Bc]{\scalebox{1}{\scriptsize{$\cat G vu$}}}
\psfrag{B1}[Bc][Bc]{\scalebox{1}{\scriptsize{$\cat G u$}}}
\psfrag{B2}[Bc][Bc]{\scalebox{1}{\scriptsize{$\cat G v$}}}
\psfrag{C}[Bc][Bc]{\scalebox{1}{\scriptsize{$\cat Fu$}}}
\psfrag{D}[Bc][Bc]{\scalebox{1}{\scriptsize{$\cat Fv$}}}
\psfrag{E}[Bc][Bc]{\scalebox{1}{\scriptsize{$t_Z$}}}
\psfrag{I}[Bc][Bc]{\scalebox{1}{\scriptsize{$t_Y$}}}
\includegraphics[scale=.4]{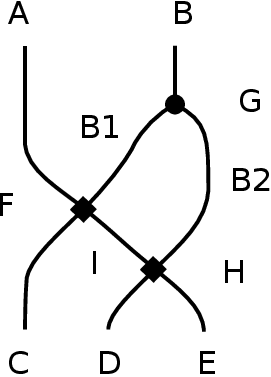}
}}
\quad\quad = \quad\quad
\vcenter {\hbox{
\psfrag{F}[Bc][Bc]{\scalebox{1}{\scriptsize{$t(vu)$}}}
\psfrag{G}[Bc][Bc]{\scalebox{1}{\scriptsize{$\fun^{\!-1}$}}}
\psfrag{A}[Bc][Bc]{\scalebox{1}{\scriptsize{$t_X$}}}
\psfrag{B}[Bc][Bc]{\scalebox{1}{\scriptsize{$\cat G vu$}}}
\psfrag{C}[Bc][Bc]{\scalebox{1}{\scriptsize{$\cat Fu$}}}
\psfrag{D}[Bc][Bc]{\scalebox{1}{\scriptsize{$\cat Fv$}}}
\psfrag{E}[Bc][Bc]{\scalebox{1}{\scriptsize{$t_Z$}}}
\psfrag{H}[Bc][Bc]{\scalebox{1}{\scriptsize{$\cat F vu$\;\;\;}}}
\includegraphics[scale=.4]{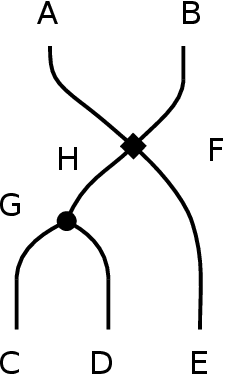}
}}
\]
for all composable $X \stackrel{u}{\to} Y \stackrel{v}{\to} Z$, and the naturality axiom says that we may do the same with the image of every 2-cell $\alpha\colon u'\Rightarrow u$:
\begin{equation*}
\vcenter {\hbox{
\psfrag{F}[Bc][Bc]{\scalebox{1}{\scriptsize{$t(u)$}}}
\psfrag{G}[Bc][Bc]{\scalebox{1}{\scriptsize{$\cat G \alpha$}}}
\psfrag{A}[Bc][Bc]{\scalebox{1}{\scriptsize{$t_X$}}}
\psfrag{B}[Bc][Bc]{\scalebox{1}{\scriptsize{$\cat G u'$}}}
\psfrag{C}[Bc][Bc]{\scalebox{1}{\scriptsize{$\cat Fu$}}}
\psfrag{D}[Bc][Bc]{\scalebox{1}{\scriptsize{$t_Y$}}}
\psfrag{E}[Bc][Bc]{\scalebox{1}{\scriptsize{$t_Z$}}}
\psfrag{H}[Bc][Bc]{\scalebox{1}{\scriptsize{$\cat G u$}}}
\includegraphics[scale=.4]{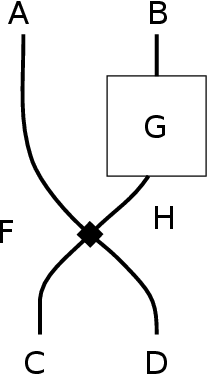}
}}
\quad\quad = \quad\quad
\vcenter {\hbox{
\psfrag{F}[Bc][Bc]{\scalebox{1}{\scriptsize{$t(u')$}}}
\psfrag{G}[Bc][Bc]{\scalebox{1}{\scriptsize{$\cat F \alpha$}}}
\psfrag{A}[Bc][Bc]{\scalebox{1}{\scriptsize{$t_X$}}}
\psfrag{B}[Bc][Bc]{\scalebox{1}{\scriptsize{$\cat G u'$}}}
\psfrag{C}[Bc][Bc]{\scalebox{1}{\scriptsize{$\cat Fu$}}}
\psfrag{D}[Bc][Bc]{\scalebox{1}{\scriptsize{$t_Y$}}}
\psfrag{E}[Bc][Bc]{\scalebox{1}{\scriptsize{$t_Z$}}}
\psfrag{H}[Bc][Bc]{\scalebox{1}{\scriptsize{$\cat F u'$}}}
\includegraphics[scale=.4]{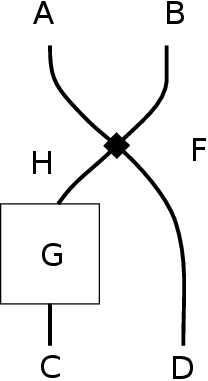}
}}
\end{equation*}
\end{Exa}

\begin{Lem}
\label{Lem:tr_inv_mate}%
Consider an oplax transformation $t\colon \cat{F}\to \cat{G}$ between pseudo-functors $\cat{F}, \cat{G}\colon \cat B\to \cat B'$, and let $\ell\dashv r$ be an adjunction in~$\cat B$. Then the component $t_r$ is invertible, with inverse the right mate of $t_\ell$ with respect to (the images under $\cat F$ and~$\cat G$ of) the given adjunction.
Equivalently, $t_\ell$ is the left mate of $t_r\inv$.
\end{Lem}

\begin{proof}
This is~\cite[Lemma~1.9]{DawsonParePronk04}.
(To be precise, our statement follows by applying \emph{loc.\ cit.\ }to the 2-dual adjunction $f:=r^\co\dashv \ell^\co=:u$ and the \emph{lax} transformation $t^\co\colon \cat{F}^{\co}, \cat{G}^{\co} \colon \cat B^{\co} \to \cat{C}^{\co}$. Beware that in \emph{loc.\,cit.\ }the use of `lax' and `oplax' is inverted with respect to ours). We provide here a transparent proof with strings.

The equivalence of the two conclusions is the mate correspondence and is immediately verified.
Writing $\eta$ and $\varepsilon$ for the unit and counit of the adjunction $\ell\dashv r$, recall that the unit and counit of the adjunction $\cat F \ell \dashv \cat Fr$ are given by
\[
\vcenter {\hbox{
\psfrag{Id}[Bc][Bc]{\scalebox{1}{\scriptsize{$\Id$}}}
\psfrag{X}[Bc][Bc]{\scalebox{1}{\scriptsize{$\un$}}}
\psfrag{Y}[Bc][Bc]{\scalebox{1}{\scriptsize{$\fun^{-1}$}}}
\psfrag{L}[Bc][Bc]{\scalebox{1}{\scriptsize{$\cat F\ell$}}}
\psfrag{R}[Bc][Bc]{\scalebox{1}{\scriptsize{$\cat Fr$}}}
\psfrag{F}[Bc][Bc]{\scalebox{1}{\scriptsize{$\cat F\eta$}}}
\includegraphics[scale=.4]{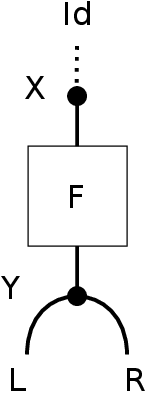}
}}
\quad\quad \textrm{and} \quad\quad
\vcenter {\hbox{
\psfrag{Id}[Bc][Bc]{\scalebox{1}{\scriptsize{$\Id$}}}
\psfrag{X}[Bc][Bc]{\scalebox{1}{\scriptsize{$\un^{-1}$\;\;\;}}}
\psfrag{Y}[Bc][Bc]{\scalebox{1}{\scriptsize{$\fun$}}}
\psfrag{L}[Bc][Bc]{\scalebox{1}{\scriptsize{$\cat F\ell$}}}
\psfrag{R}[Bc][Bc]{\scalebox{1}{\scriptsize{$\cat Fr$}}}
\psfrag{F}[Bc][Bc]{\scalebox{1}{\scriptsize{$\cat F\varepsilon$}}}
\includegraphics[scale=.4]{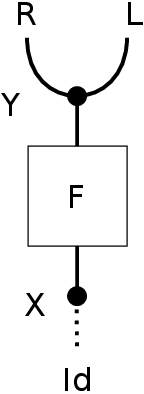}
}}
\]
respectively, and similarly for~$\cat G$. Now compute using exchange (\Cref{Exa:strings-for-exchange}) followed by the functoriality and naturality of~$t$ (\Cref{Exa:trans}):
\[
t(r)\circ t(\ell)_*
\quad = \quad
\vcenter {\hbox{
\psfrag{A}[Bc][Bc]{\scalebox{1}{\scriptsize{$\cat Fr$}}}
\psfrag{B}[Bc][Bc]{\scalebox{1}{\scriptsize{$t$}}}
\psfrag{C}[Bc][Bc]{\scalebox{1}{\scriptsize{$\cat Fr$}}}
\psfrag{D}[Bc][Bc]{\scalebox{1}{\scriptsize{$t$}}}
\psfrag{X1}[Bc][Bc]{\scalebox{1}{\scriptsize{$\eta$}}}
\psfrag{X2}[Bc][Bc]{\scalebox{1}{\scriptsize{$\varepsilon$}}}
\psfrag{F}[Bc][Bc]{\scalebox{1}{\scriptsize{$t(\ell)$}}}
\psfrag{G}[Bc][Bc]{\scalebox{1}{\scriptsize{$t(r)$}}}
\psfrag{T}[Bc][Bc]{\scalebox{1}{\scriptsize{$t$}}}
\psfrag{R}[Bc][Bc]{\scalebox{1}{\scriptsize{$\cat Gr$}}}
\psfrag{R1}[Bc][Bc]{\scalebox{1}{\scriptsize{$\cat G\ell$\;}}}
\psfrag{R2}[Bc][Bc]{\scalebox{1}{\scriptsize{$\cat F\ell$}}}
\includegraphics[scale=.4]{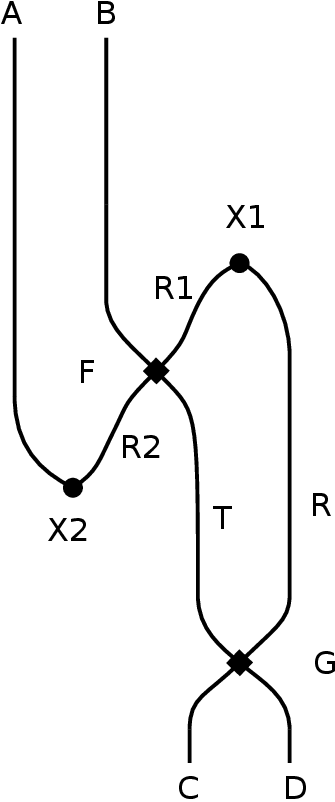}
}}
\quad = \quad
\vcenter {\hbox{
\psfrag{A}[Bc][Bc]{\scalebox{1}{\scriptsize{$\cat Fr$}}}
\psfrag{B}[Bc][Bc]{\scalebox{1}{\scriptsize{$t$}}}
\psfrag{C}[Bc][Bc]{\scalebox{1}{\scriptsize{$\cat Fr$}}}
\psfrag{D}[Bc][Bc]{\scalebox{1}{\scriptsize{$t$}}}
\psfrag{F1}[Bc][Bc]{\scalebox{1}{\scriptsize{$\cat G\eta$}}}
\psfrag{F2}[Bc][Bc]{\scalebox{1}{\scriptsize{$\cat F\varepsilon$}}}
\psfrag{F}[Bc][Bc]{\scalebox{1}{\scriptsize{$t(\ell)$}}}
\psfrag{G}[Bc][Bc]{\scalebox{1}{\scriptsize{$t(r)$}}}
\psfrag{X1}[Bc][Bc]{\scalebox{1}{\scriptsize{$\un$}}}
\psfrag{X2}[Bc][Bc]{\scalebox{1}{\scriptsize{$\un^{-1}$}}}
\psfrag{Y1}[Bc][Bc]{\scalebox{1}{\scriptsize{$\fun^{-1}$}}}
\psfrag{Y2}[Bc][Bc]{\scalebox{1}{\scriptsize{$\fun$}}}
\includegraphics[scale=.4]{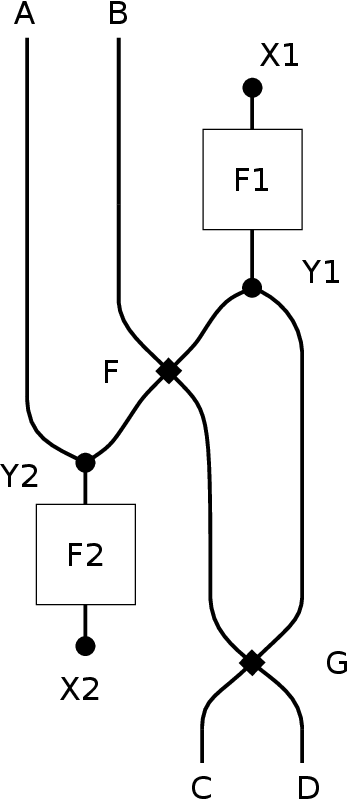}
}}
\quad = \quad
\vcenter {\hbox{
\psfrag{A}[Bc][Bc]{\scalebox{1}{\scriptsize{$\cat Fr$}}}
\psfrag{B}[Bc][Bc]{\scalebox{1}{\scriptsize{$t$}}}
\psfrag{C}[Bc][Bc]{\scalebox{1}{\scriptsize{$\cat Fr$}}}
\psfrag{D}[Bc][Bc]{\scalebox{1}{\scriptsize{$t$}}}
\psfrag{F1}[Bc][Bc]{\scalebox{1}{\scriptsize{$\cat G\eta$}}}
\psfrag{F2}[Bc][Bc]{\scalebox{1}{\scriptsize{$\cat F\varepsilon$}}}
\psfrag{F}[Bc][Bc]{\scalebox{1}{\scriptsize{$t(\ell)$}}}
\psfrag{G}[Bc][Bc]{\scalebox{1}{\scriptsize{$t(r)$}}}
\psfrag{X1}[Bc][Bc]{\scalebox{1}{\scriptsize{$\un$}}}
\psfrag{X2}[Bc][Bc]{\scalebox{1}{\scriptsize{$\un^{-1}$}}}
\psfrag{Y1}[Bc][Bc]{\scalebox{1}{\scriptsize{$\fun^{-1}$}}}
\psfrag{Y2}[Bc][Bc]{\scalebox{1}{\scriptsize{$\fun$}}}
\includegraphics[scale=.4]{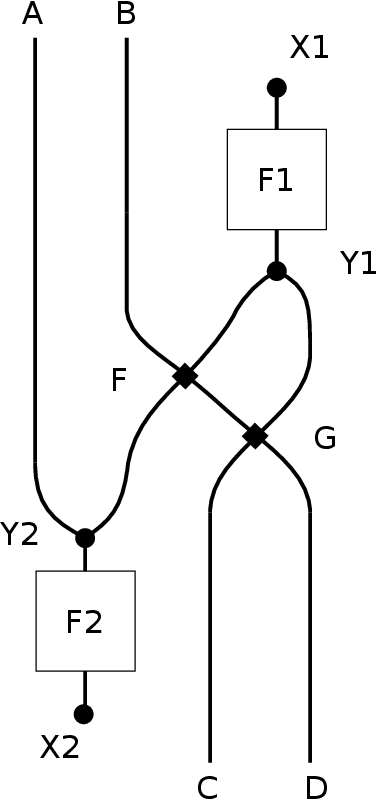}
}}
\]
\[
= \quad
\vcenter {\hbox{
\psfrag{A}[Bc][Bc]{\scalebox{1}{\scriptsize{$\cat Fr$}}}
\psfrag{B}[Bc][Bc]{\scalebox{1}{\scriptsize{$t$}}}
\psfrag{C}[Bc][Bc]{\scalebox{1}{\scriptsize{$\cat Fr$}}}
\psfrag{D}[Bc][Bc]{\scalebox{1}{\scriptsize{$t$}}}
\psfrag{F1}[Bc][Bc]{\scalebox{1}{\scriptsize{$\cat G\eta$}}}
\psfrag{F2}[Bc][Bc]{\scalebox{1}{\scriptsize{$\cat F\varepsilon$}}}
\psfrag{F}[Bc][Bc]{\scalebox{1}{\scriptsize{$t(r\ell)$\;\;}}}
\psfrag{G}[Bc][Bc]{\scalebox{1}{\scriptsize{$\cat Gr\ell$}}}
\psfrag{X1}[Bc][Bc]{\scalebox{1}{\scriptsize{$\un$}}}
\psfrag{X2}[Bc][Bc]{\scalebox{1}{\scriptsize{$\un^{-1}$}}}
\psfrag{Y1}[Bc][Bc]{\scalebox{1}{\scriptsize{$\;\fun^{-1}$}}}
\psfrag{Y2}[Bc][Bc]{\scalebox{1}{\scriptsize{$\fun$}}}
\includegraphics[scale=.4]{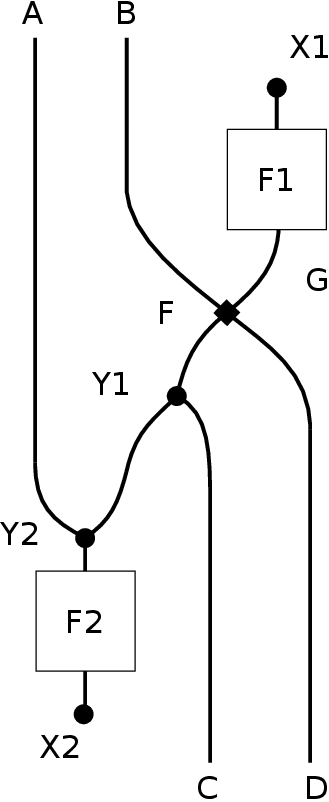}
}}
\quad = \quad
\vcenter {\hbox{
\psfrag{A}[Bc][Bc]{\scalebox{1}{\scriptsize{$\cat Fr$}}}
\psfrag{B}[Bc][Bc]{\scalebox{1}{\scriptsize{$t$}}}
\psfrag{C}[Bc][Bc]{\scalebox{1}{\scriptsize{$\cat Fr$}}}
\psfrag{D}[Bc][Bc]{\scalebox{1}{\scriptsize{$t$}}}
\psfrag{F1}[Bc][Bc]{\scalebox{1}{\scriptsize{$\cat F\eta$}}}
\psfrag{F2}[Bc][Bc]{\scalebox{1}{\scriptsize{$\cat F\varepsilon$}}}
\psfrag{F}[Bc][Bc]{\scalebox{1}{\scriptsize{$t(\Id)$\;\;}}}
\psfrag{G}[Bc][Bc]{\scalebox{1}{\scriptsize{$\;\;\;\cat G\Id$}}}
\psfrag{X1}[Bc][Bc]{\scalebox{1}{\scriptsize{$\un$}}}
\psfrag{X2}[Bc][Bc]{\scalebox{1}{\scriptsize{$\un^{-1}$}}}
\psfrag{Y1}[Bc][Bc]{\scalebox{1}{\scriptsize{$\;\fun^{-1}$}}}
\psfrag{Y2}[Bc][Bc]{\scalebox{1}{\scriptsize{$\fun$}}}
\includegraphics[scale=.4]{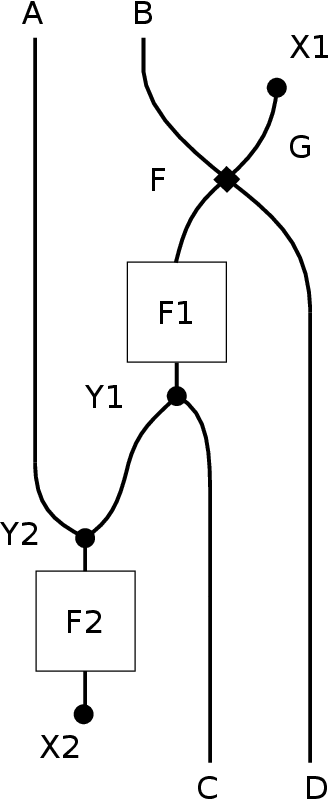}
}}
\quad = \quad
\vcenter {\hbox{
\psfrag{A}[Bc][Bc]{\scalebox{1}{\scriptsize{$\cat Fr$}}}
\psfrag{B}[Bc][Bc]{\scalebox{1}{\scriptsize{$t$}}}
\psfrag{C}[Bc][Bc]{\scalebox{1}{\scriptsize{$\cat Fr$}}}
\psfrag{D}[Bc][Bc]{\scalebox{1}{\scriptsize{$t$}}}
\psfrag{F1}[Bc][Bc]{\scalebox{1}{\scriptsize{$\cat F\eta$}}}
\psfrag{F2}[Bc][Bc]{\scalebox{1}{\scriptsize{$\cat F\varepsilon$}}}
\psfrag{X1}[Bc][Bc]{\scalebox{1}{\scriptsize{$\un$}}}
\psfrag{X2}[Bc][Bc]{\scalebox{1}{\scriptsize{$\un^{-1}$}}}
\psfrag{Y1}[Bc][Bc]{\scalebox{1}{\scriptsize{$\fun^{-1}$}}}
\psfrag{Y2}[Bc][Bc]{\scalebox{1}{\scriptsize{$\fun$}}}
\includegraphics[scale=.4]{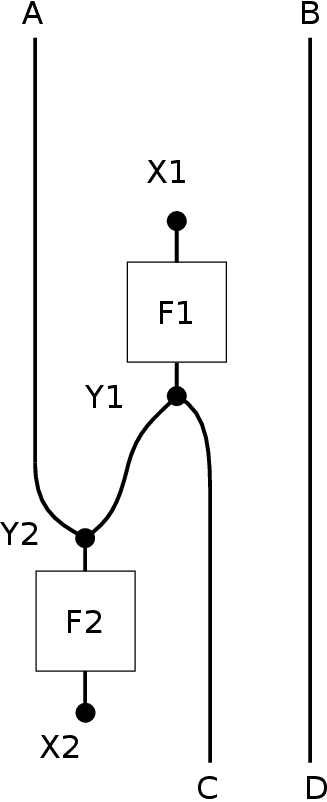}
}}
 = \quad
\id_{t \circ \cat F(r)}
\]
The verification of $(t_\ell)_*\circ t_r =\id$ is symmetrical and is left to the reader.
\end{proof}

\begin{Rem}
\label{Rem:tr_inv_mate}%
The statement of \Cref{Lem:tr_inv_mate} is asymmetric: $t$ is \emph{op}lax and the conclusion does not guarantee invertibility of~$t_\ell$. The dual statement says that if $t$ is a \emph{lax} transformation then $t_\ell$ is indeed invertible, with inverse $(t_r)_!$.
\end{Rem}

We are now able to provide a full proof of \Cref{Lem:comm-square-comp} and thereby repay a debt from that early section. Let us recall the statement:

\begin{Lem} \label{Lem:full-details-comp}
The following square
\begin{equation}
\label{eq:delta_i-M-N-appendix}%
\vcenter{\xymatrix@C=6em@L=1ex{
\cat N(p_1) t_H \ar@{=>}[r]^-{\Displ \delta_i^{\scriptscriptstyle \NN}\,t_H} \ar@{=>}[d]^-{\simeq}_-{\Displ t_{p_1}}
& \cat N(p_2) t_H \ar@{=>}[d]_-{\simeq}^-{\Displ t_{p_2}}
\\
t_{(i/i)} \cat M(p_1) \ar@{=>}[r]^-{\Displ t_{(i/i)}\delta_i^{\scriptscriptstyle \MM}}
& t_{(i/i)} \cat M( p_2 )
}}
\end{equation}
of natural transformations between functors $\MM(H)\to \NN(i/i)$ is commutative.
Recall that
\[
\vcenter{
\xymatrix@C=14pt@R=14pt{
& (i/i) \ar[dl]_-{p_1} \ar[dr]^-{p_2}
\\
H \ar[dr]_-{i} \ar@{}[rr]|-{\isocell{\lambda}}
&& H \ar[dl]^-{i}
\\
&G
}}
\]
is the self-iso-comma construction on a faithful functor $i\colon H\to G$ of finite groupoids; $t\colon \cat M \to \cat N$ is a pseudo-natural transformation between strict 2-functors $\MM,\NN\colon \GG^\op\to \ADD$ taking values in additive categories (for $\GG \subseteq \gpd$ a sub-2-category of groupoids); finally, the natural transformations $\delta_i^\MM\colon \cat M(p_1)\Rightarrow \cat M(p_2)$ and $\delta_i^\cat N\colon \cat N(p_1)\Rightarrow \cat N(p_2)$ are defined in \Cref{Prop:delta} (this is recalled below).
\end{Lem}

 \begin{proof}
 Let $C$ be either the essential full image $\Delta_i(H)$ of $\Delta_i$ in~$\cat (i/i)$, or its complement $C= (i/i)\smallsetminus \Delta_i(H)$, and write $j\colon C\hookrightarrow (i/i)$ for the inclusion functor. In order to show the commutativity of \eqref{eq:delta_i-M-N-appendix}, it suffices to show that it commutes after applying the functor $\cat N(j)$ to it for both choices of~$C$.
 Recall that \Cref{Prop:delta} characterizes $\delta_i^\MM$ and $\delta_i^\cat N$ by the property that if we apply $\cat M(j)\colon \MM(i/i)\to \MM(C)$ (respectively $\cat N(j)\colon \cat N(i/i)\to \cat N(C)$) to it, we obtain the identity or zero natural transformation, according as to whether $C=\Delta_i(H)$ or $C= (i/i)\smallsetminus \Delta_i(H)$.

Let us first consider $C:= \Delta_i(H)$. After applying $\cat N(j)$ to the bottom arrow $t_{(i/i)} \delta_i^\cat M$ in \eqref{eq:delta_i-M-N-appendix}, we can rewrite the result as the following string diagram:
 \[
 \vcenter {\hbox{
\psfrag{A}[Bc][Bc]{\scalebox{1}{\scriptsize{$\cat M p_1$}}}
\psfrag{B}[Bc][Bc]{\scalebox{1}{\scriptsize{$t_{(i/i)}$}}}
\psfrag{C}[Bc][Bc]{\scalebox{1}{\scriptsize{$\cat N j$}}}
\psfrag{E}[Bc][Bc]{\scalebox{1}{\scriptsize{$\cat M p_2$}}}
\psfrag{D}[Bc][Bc]{\scalebox{1}{\scriptsize{$\delta_i^\cat M$\;\;}}}
\psfrag{P}[Bc][Bc]{\scalebox{1}{\scriptsize{$t_C$}}}
\includegraphics[scale=.4]{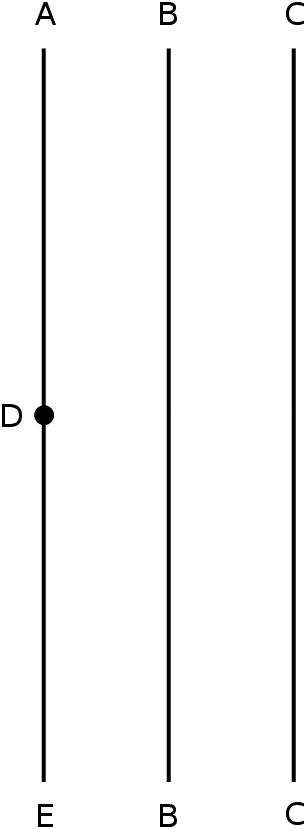}
}}
 \quad = \quad\quad
\vcenter {\hbox{
\psfrag{A}[Bc][Bc]{\scalebox{1}{\scriptsize{$\cat M p_1$}}}
\psfrag{B}[Bc][Bc]{\scalebox{1}{\scriptsize{$t_{(i/i)}$}}}
\psfrag{C}[Bc][Bc]{\scalebox{1}{\scriptsize{$\cat N j$}}}
\psfrag{D}[Bc][Bc]{\scalebox{1}{\scriptsize{$\delta_i^\cat M$\;\;}}}
\psfrag{E}[Bc][Bc]{\scalebox{1}{\scriptsize{$\cat M p_2$}}}
\psfrag{H}[Bc][Bc]{\scalebox{1}{\scriptsize{$t_j$}}}
\psfrag{L}[Bc][Bc]{\scalebox{1}{\scriptsize{$t_j^{-1}$}}}
\psfrag{O}[Bc][Bc]{\scalebox{1}{\scriptsize{$\cat Mj$\;\;}}}
\psfrag{P}[Bc][Bc]{\scalebox{1}{\scriptsize{\;\;$t_C$}}}
\includegraphics[scale=.4]{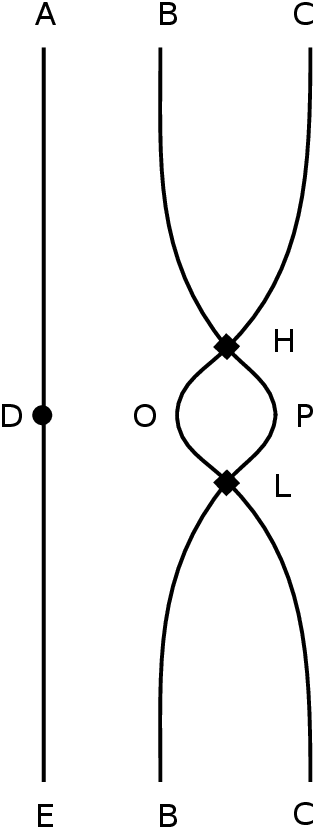}
}}
\quad = \quad
\vcenter {\hbox{
\psfrag{A}[Bc][Bc]{\scalebox{1}{\scriptsize{$\cat M p_1$}}}
\psfrag{B}[Bc][Bc]{\scalebox{1}{\scriptsize{$t_{(i/i)}$}}}
\psfrag{C}[Bc][Bc]{\scalebox{1}{\scriptsize{$\cat N j$}}}
\psfrag{D}[Bc][Bc]{\scalebox{1}{\scriptsize{$\delta_i^\cat M$}}}
\psfrag{E}[Bc][Bc]{\scalebox{1}{\scriptsize{$\cat M p_2$}}}
\psfrag{F}[Bc][Bc]{\scalebox{1}{\scriptsize{$t_{p_1}^{-1}$}}}
\psfrag{G}[Bc][Bc]{\scalebox{1}{\scriptsize{$t_{p_1}$}}}
\psfrag{H}[Bc][Bc]{\scalebox{1}{\scriptsize{$t_j$}}}
\psfrag{L}[Bc][Bc]{\scalebox{1}{\scriptsize{$t_j^{-1}$}}}
\psfrag{M}[Bc][Bc]{\scalebox{1}{\scriptsize{$t_{p_2}^{-1}$\;\;}}}
\psfrag{N}[Bc][Bc]{\scalebox{1}{\scriptsize{$t_{p_2}\;\;$}}}
\psfrag{O}[Bc][Bc]{\scalebox{1}{\scriptsize{$\cat Mj$\;\;}}}
\includegraphics[scale=.4]{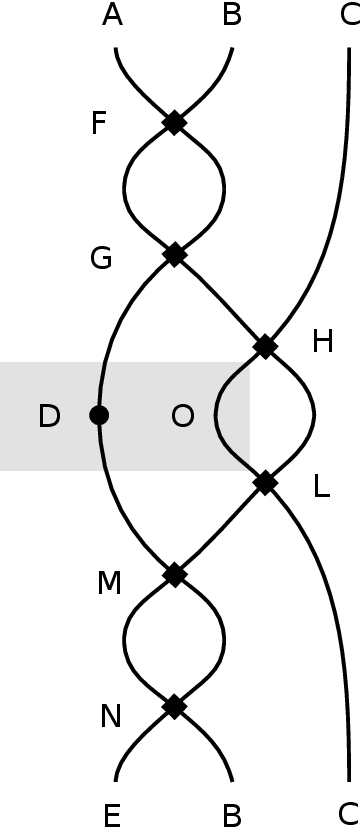}
}}
 \]
 In this case, the shaded box is equal to the identity natural transformation by \Cref{Prop:delta}\,(1). Using the functoriality of $t$ (applied in this case to $\fun=\id$; see \Cref{Exa:trans}) and \Cref{Prop:delta}\,(1) once more, we proceed as follows:
 \[
\overset{\textrm{\ref{Prop:delta}\,(1)}}{=} \quad
\vcenter {\hbox{
\psfrag{A}[Bc][Bc]{\scalebox{1}{\scriptsize{$\cat M p_1$}}}
\psfrag{B}[Bc][Bc]{\scalebox{1}{\scriptsize{$t_{(i/i)}$}}}
\psfrag{C}[Bc][Bc]{\scalebox{1}{\scriptsize{$\cat N j$}}}
\psfrag{E}[Bc][Bc]{\scalebox{1}{\scriptsize{$\cat M p_2$}}}
\psfrag{F}[Bc][Bc]{\scalebox{1}{\scriptsize{$t_{p_1}^{-1}$}}}
\psfrag{G}[Bc][Bc]{\scalebox{1}{\scriptsize{$t_{p_1}$}}}
\psfrag{H}[Bc][Bc]{\scalebox{1}{\scriptsize{$t_j$}}}
\psfrag{L}[Bc][Bc]{\scalebox{1}{\scriptsize{$t_j^{-1}$}}}
\psfrag{M}[Bc][Bc]{\scalebox{1}{\scriptsize{$t_{p_2}^{-1}$\;\;}}}
\psfrag{N}[Bc][Bc]{\scalebox{1}{\scriptsize{$t_{p_2}\;\;$}}}
\psfrag{O}[Bc][Bc]{\scalebox{1}{\scriptsize{$t_C$}}}
\psfrag{id}[Bc][Bc]{\scalebox{1}{\scriptsize{$\id$}}}
\psfrag{Id}[Bc][Bc]{\scalebox{1}{\scriptsize{$\Id$}}}
\includegraphics[scale=.4]{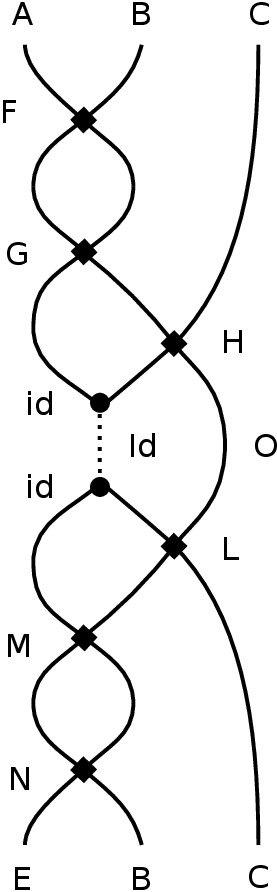}
}}
\quad \overset{2 \times \textrm{\ref{Exa:trans}} }{=} \quad
\vcenter {\hbox{
\psfrag{A}[Bc][Bc]{\scalebox{1}{\scriptsize{$\cat M p_1$}}}
\psfrag{B}[Bc][Bc]{\scalebox{1}{\scriptsize{$t_{(i/i)}$}}}
\psfrag{C}[Bc][Bc]{\scalebox{1}{\scriptsize{$\cat N j$}}}
\psfrag{E}[Bc][Bc]{\scalebox{1}{\scriptsize{$\cat M p_2$}}}
\psfrag{F}[Bc][Bc]{\scalebox{1}{\scriptsize{$t_{p_1}^{-1}$}}}
\psfrag{N}[Bc][Bc]{\scalebox{1}{\scriptsize{$t_{p_2}\;\;$}}}
\psfrag{O}[Bc][Bc]{\scalebox{1}{\scriptsize{$t_H$}}}
\psfrag{id}[Bc][Bc]{\scalebox{1}{\scriptsize{$\id$}}}
\psfrag{Id}[Bc][Bc]{\scalebox{1}{\scriptsize{$\Id$}}}
\includegraphics[scale=.4]{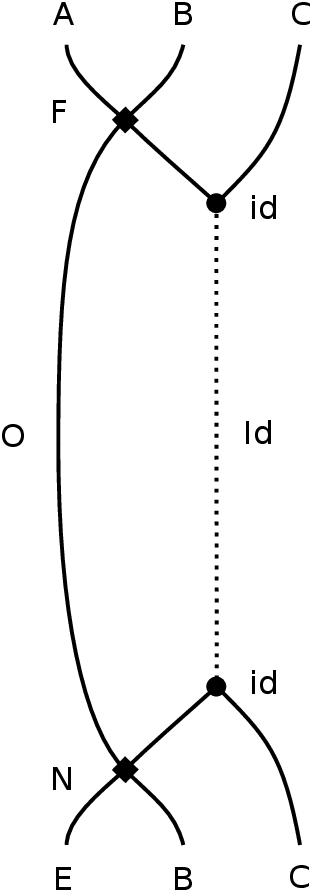}
}}
\quad \overset{\textrm{\ref{Prop:delta}\,(1)}}{=} \quad
\vcenter {\hbox{
\psfrag{A}[Bc][Bc]{\scalebox{1}{\scriptsize{$\cat M p_1$}}}
\psfrag{B}[Bc][Bc]{\scalebox{1}{\scriptsize{$t_{(i/i)}$}}}
\psfrag{C}[Bc][Bc]{\scalebox{1}{\scriptsize{$\cat N j$}}}
\psfrag{D}[Bc][Bc]{\scalebox{1}{\scriptsize{$\delta_i^\cat N$}}}
\psfrag{E}[Bc][Bc]{\scalebox{1}{\scriptsize{$\cat M p_2$}}}
\psfrag{F}[Bc][Bc]{\scalebox{1}{\scriptsize{$t_{p_1}^{-1}$}}}
\psfrag{N}[Bc][Bc]{\scalebox{1}{\scriptsize{$t_{p_2}\;\;$}}}
\psfrag{O}[Bc][Bc]{\scalebox{1}{\scriptsize{$t_H$}}}
\psfrag{U}[Bc][Bc]{\scalebox{1}{\scriptsize{\;\;$\cat Np_1$}}}
\psfrag{V}[Bc][Bc]{\scalebox{1}{\scriptsize{\;$\cat Np_2$}}}
\includegraphics[scale=.4]{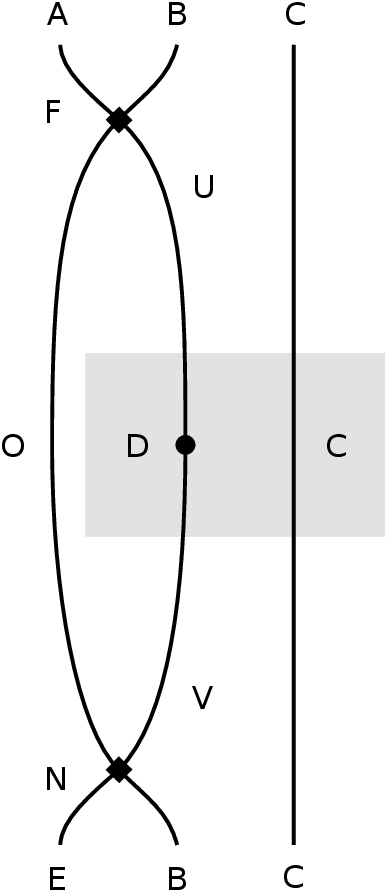}
}}
 \]
 The latter natural transformation is precisely the result of applying $\cat N(j)$ to the up-then-right-then-down composite $(t_{p_2})(\delta^\cat N_i t_H)(t_{p_1}^{-1})$ in \eqref{eq:delta_i-M-N-appendix}, proving commutativity.

With $C:= (i/i)\smallsetminus \Delta_i(H)$ the proof is even simpler, as we can replace the two shaded boxes with zero by \Cref{Prop:delta}\,(2), from which we see that both sides of the above computation are zero by the bilinearity of the vertical and horizontal composition of natural transformations in~$\ADD$.
 \end{proof}


\bigbreak
\section{How to read string diagrams in this book}
\label{sec:strings-here}%
\medskip

In the previous section we have recalled the general yoga of string diagrams and have provided a healthy dose of examples. Now let us say a few more words on how they are used in this book.

A prominent feature of string diagrams over cellular pasting diagrams is that it is actually possible to prove that their `evaluation' into 2-cells is invariant under deformation of diagrams, in a suitable sense (ambient isotopy, sequences of moves, etc.; see \cite[\S4]{Street96} \cite[\S2]{BarrettMeusburgerSchaumann13pp} \cite[\S2]{TuraevVirelizier17}). This incorporates the strictification theorem but is stronger, and it allows to give presentations of certain (structured) monoidal categories or bicategories in terms of equivalence classes of string diagrams modulo deformation. This use of strings is taken up in \Cref{sec:string-presentation}, where we present a strictification of the bicategory of Mackey 2-motives.
Nonetheless, in this work we mostly use strings more informally, starting in \Cref{ch:bicat-spans}, as a conveniently compact notation for complicated 2-cells which, in usual cellular notation, would take up too much space on the page and would be hard to read.

In practice, and again in order to save space, we slightly cheat and present our diagrams in a compressed form by suppressing some of the white space. The remaining of this section explains how to translate such diagrams, with the help of an example taken from the proof of \Cref{Lem:Gwelldef}.

\begin{Exa} \label{Rem:how-to-read}
Say we encounter a string diagram like the one on the left:
\[
\vcenter { \hbox{
\psfrag{A}[Bc][Bc]{\scalebox{1}{\scriptsize{$\cat F u$}}}
\psfrag{B}[Bc][Bc]{\scalebox{1}{\scriptsize{$(\cat F i)_!$}}}
\psfrag{C}[Bc][Bc]{\scalebox{1}{\scriptsize{\;\;\;$\cat F ja$}}}
\psfrag{C'}[Bc][Bc]{\scalebox{1}{\scriptsize{$(\cat F ja)_!$}}}
\psfrag{D}[Bc][Bc]{\scalebox{1}{\scriptsize{$\cat F j$}}}
\psfrag{E}[Bc][Bc]{\scalebox{1}{\scriptsize{$\cat F a$}}}
\psfrag{R}[Bc][Bc]{\scalebox{1}{\scriptsize{$\cat F va$\;\;}}}
\psfrag{S}[Bc][Bc]{\scalebox{1}{\scriptsize{$\cat F a$\;}}}
\psfrag{S'}[Bc][Bc]{\scalebox{1}{\scriptsize{$(\cat F a)_!$\;}}}
\psfrag{T}[Bc][Bc]{\scalebox{1}{\scriptsize{$\cat F v$}}}
\psfrag{U}[Bc][Bc]{\scalebox{1}{\scriptsize{\;\;\;\;$(\cat F j)_!$}}}
\psfrag{F}[Bc][Bc]{\scalebox{1}{\scriptsize{$\cat F \alpha_1$}}}
\psfrag{G}[Bc][Bc]{\scalebox{1}{\scriptsize{$\cat F \alpha_2$}}}
\includegraphics[scale=.4]{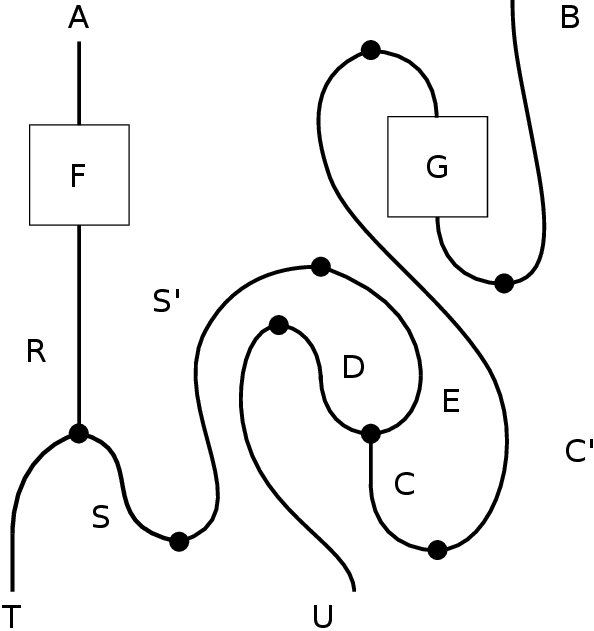}
}}
\quad\quad\quad \rightsquigarrow \quad\quad\quad
\vcenter { \hbox{
\psfrag{A}[Bc][Bc]{\scalebox{1}{\scriptsize{$\cat F u$}}}
\psfrag{B}[Bc][Bc]{\scalebox{1}{\scriptsize{$(\cat F i)_!$}}}
\psfrag{B'}[Bc][Bc]{\scalebox{1}{\scriptsize{$\cat Fi$}}}
\psfrag{C}[Bc][Bc]{\scalebox{1}{\scriptsize{\;\;\;$\cat F ja$}}}
\psfrag{C'}[Bc][Bc]{\scalebox{1}{\scriptsize{$(\cat F ja)_!$}}}
\psfrag{D}[Bc][Bc]{\scalebox{1}{\scriptsize{$\cat F j$}}}
\psfrag{E}[Bc][Bc]{\scalebox{1}{\scriptsize{$\cat F a$}}}
\psfrag{R}[Bc][Bc]{\scalebox{1}{\scriptsize{$\cat F va$\;\;}}}
\psfrag{S}[Bc][Bc]{\scalebox{1}{\scriptsize{$\cat F a$\;}}}
\psfrag{S'}[Bc][Bc]{\scalebox{1}{\scriptsize{$(\cat F a)_!$\;}}}
\psfrag{T}[Bc][Bc]{\scalebox{1}{\scriptsize{$\cat F v$}}}
\psfrag{U}[Bc][Bc]{\scalebox{1}{\scriptsize{\;\;\;\;$(\cat F j)_!$}}}
\psfrag{F}[Bc][Bc]{\scalebox{1}{\scriptsize{$\cat F \alpha_1$}}}
\psfrag{G}[Bc][Bc]{\scalebox{1}{\scriptsize{$\cat F \alpha_2$}}}
\psfrag{N'}[Bc][Bc]{\scalebox{1}{\scriptsize{$\eta$}}}
\psfrag{E'}[Bc][Bc]{\scalebox{1}{\scriptsize{$\varepsilon$}}}
\psfrag{F'}[Bc][Bc]{\scalebox{1}{\scriptsize{$\fun$}}}
\psfrag{F''}[Bc][Bc]{\scalebox{1}{\scriptsize{$\fun^{-1}$\;\;\;\;\;}}}
\includegraphics[scale=.4]{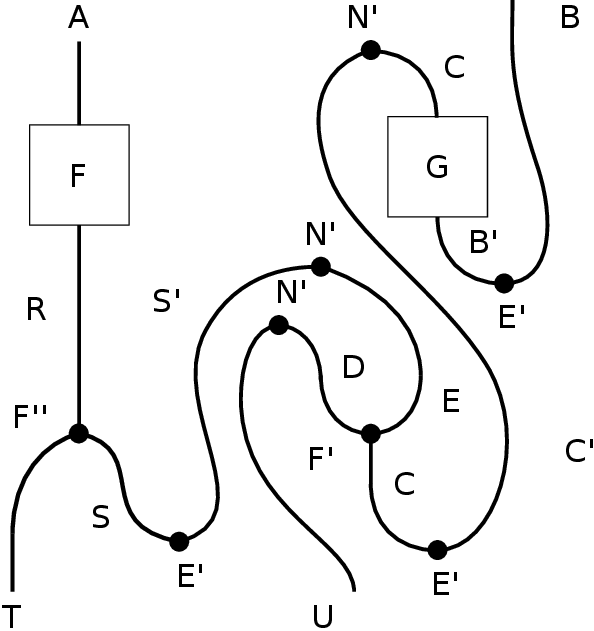}
}}
\]
It stands for a 2-cell $\alpha$ of some bicategory~$\cat B$. Concretely, how do we determine this 2-cell? For a start, context should allow us to add labels to all the dots and boxes (representing specified 2-cells) and all strings (1-cells) that make it up, as in the right-hand side above. (We omit the labels for the planar regions -- \ie objects -- but it could also be done.) Then we add some white space in the vertical and horizontal directions (by an `ambient isotopy') until we are able to cut up the diagram into horizontal stripes, each of which consists of a horizontal juxtaposition (\ie a horizontal composition in~$\cat B$) of recognizable dots/boxes/identities. While doing this, we take care not to change the up-down orientation of each piece of string. We obtain something like this:
 \[
 \vcenter { \hbox{
\psfrag{A}[Bc][Bc]{\scalebox{1}{\scriptsize{$\cat F u$}}}
\psfrag{B}[Bc][Bc]{\scalebox{1}{\scriptsize{$(\cat F i)_!$}}}
\psfrag{B'}[Bc][Bc]{\scalebox{1}{\scriptsize{$\cat Fi$}}}
\psfrag{C}[Bc][Bc]{\scalebox{1}{\scriptsize{\;\;\;$\cat F ja$}}}
\psfrag{C'}[Bc][Bc]{\scalebox{1}{\scriptsize{$(\cat F ja)_!$}}}
\psfrag{D}[Bc][Bc]{\scalebox{1}{\scriptsize{$\cat F j$}}}
\psfrag{E}[Bc][Bc]{\scalebox{1}{\scriptsize{$\cat F a$\;}}}
\psfrag{R}[Bc][Bc]{\scalebox{1}{\scriptsize{$\cat F va$\;\;}}}
\psfrag{S}[Bc][Bc]{\scalebox{1}{\scriptsize{$\cat F a$\;}}}
\psfrag{S'}[Bc][Bc]{\scalebox{1}{\scriptsize{$(\cat F a)_!$\;}}}
\psfrag{T}[Bc][Bc]{\scalebox{1}{\scriptsize{$\cat F v$}}}
\psfrag{U}[Bc][Bc]{\scalebox{1}{\scriptsize{\;\;\;\;$(\cat F j)_!$}}}
\psfrag{F}[Bc][Bc]{\scalebox{1}{\scriptsize{$\cat F \alpha_1$}}}
\psfrag{G}[Bc][Bc]{\scalebox{1}{\scriptsize{$\cat F \alpha_2$}}}
\psfrag{N'}[Bc][Bc]{\scalebox{1}{\scriptsize{$\eta$}}}
\psfrag{E'}[Bc][Bc]{\scalebox{1}{\scriptsize{$\varepsilon$}}}
\psfrag{F'}[Bc][Bc]{\scalebox{1}{\scriptsize{$\fun$}}}
\psfrag{F''}[Bc][Bc]{\scalebox{1}{\scriptsize{$\fun^{-1}$\;\;\;\;\;}}}
\psfrag{S1}[Bl][Bl]{\scalebox{1}{$\sigma_1=\id_{(\cat F i)_!} \circ \eta \circ \id_{\cat F u}$}}
\psfrag{S2}[Bl][Bl]{\scalebox{1}{$\sigma_2=\id_{(\cat Fi)_!} \circ \cat F \alpha_2 \circ \id_{(\cat Fja)_!} \circ \cat F\alpha_1 $}}
\psfrag{S3}[Bl][Bl]{\scalebox{1}{$\sigma_3=\varepsilon \circ \id_{(\cat Fja)_!} \circ \eta \circ \id_{\cat Fva} $}}
\psfrag{S4}[Bl][Bl]{\scalebox{1}{$\sigma_4=\id_{(\cat Fja)_!} \circ \id_{\cat Fa} \circ \eta \circ \id_{(\cat Fa)_!} \circ \fun^{-1} $}}
\psfrag{S5}[Bl][Bl]{\scalebox{1}{$\sigma_5= \id_{(\cat Fja)_!} \circ \fun \circ (\cat Fj)_! \circ \varepsilon \circ \id_{\cat Fv} $}}
\psfrag{S6}[Bl][Bl]{\scalebox{1}{$\sigma_6= \varepsilon \circ \id_{(\cat Fj)_!} \circ \id_{\cat Fv}$}}
\includegraphics[scale=.4]{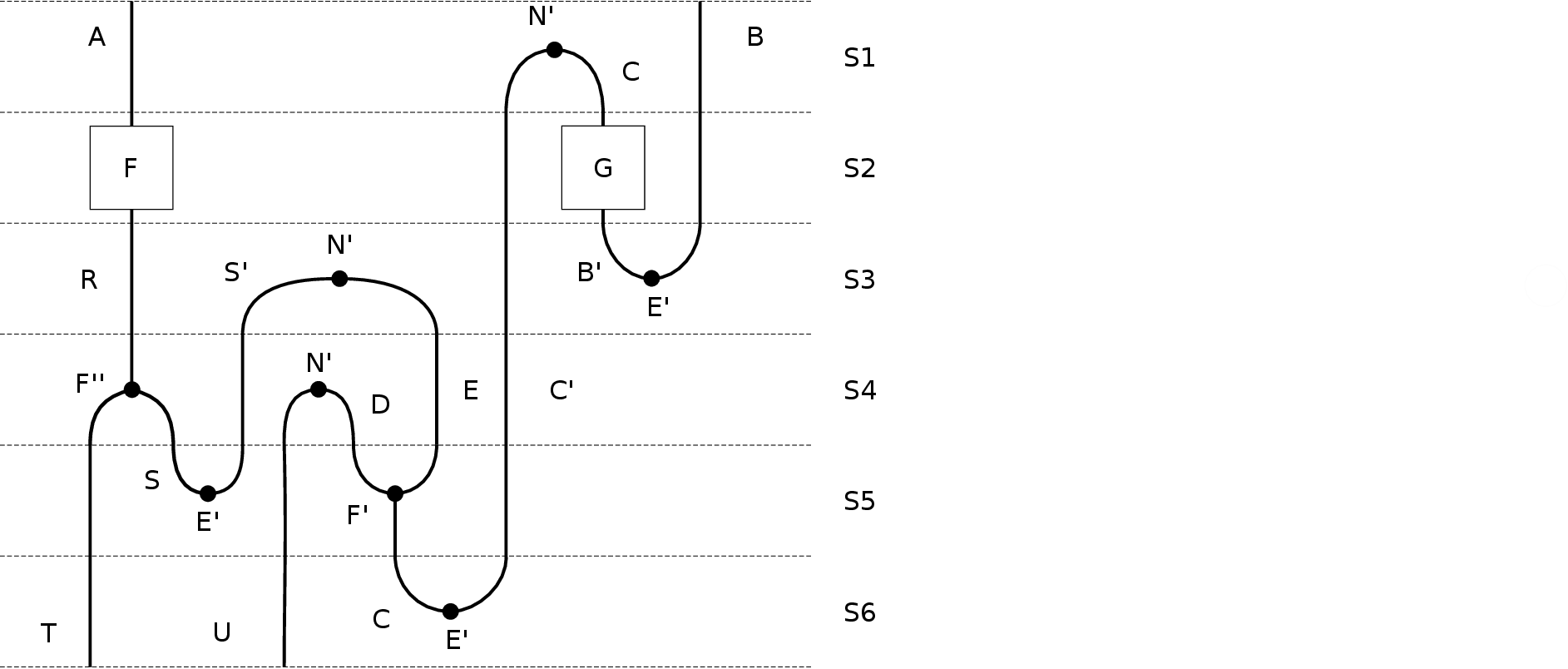}
}}
\]
If $\cat B$ happens to be a strict 2-category, we can already evaluate each stripe separately into a 2-cell~$\sigma_i$, by horizontally composing the 2-cells we encounter when scanning the stripe from left to right. The results are depicted on the right-hand side. Note that, by scanning each line from left to right, we can read off the source and target 1-cells of each~$\sigma_i$. They are composable by construction, and their (vertical) composite in $\cat B$ is the final result: $\alpha=\sigma_6\sigma_5\sigma_4\sigma_3\sigma_2\sigma_1$.

If $\cat B$ is not strict, then we must first insert some identity 1-cells so that every dot/box has an explicit domain and codomain in the diagram:
 \begin{equation} \label{eq:app-example}
 \vcenter { \hbox{
\psfrag{A}[Bc][Bc]{\scalebox{1}{\scriptsize{$\cat F u$}}}
\psfrag{B}[Bc][Bc]{\scalebox{1}{\scriptsize{$(\cat F i)_!$}}}
\psfrag{B'}[Bc][Bc]{\scalebox{1}{\scriptsize{$\cat Fi$}}}
\psfrag{C}[Bc][Bc]{\scalebox{1}{\scriptsize{\;\;\;$\cat F ja$}}}
\psfrag{C'}[Bc][Bc]{\scalebox{1}{\scriptsize{$(\cat F ja)_!$}}}
\psfrag{D}[Bc][Bc]{\scalebox{1}{\scriptsize{$\cat F j$}}}
\psfrag{E}[Bc][Bc]{\scalebox{1}{\scriptsize{$\cat F a$\;}}}
\psfrag{R}[Bc][Bc]{\scalebox{1}{\scriptsize{$\cat F va$\;\;}}}
\psfrag{S}[Bc][Bc]{\scalebox{1}{\scriptsize{$\cat F a$\;}}}
\psfrag{S'}[Bc][Bc]{\scalebox{1}{\scriptsize{$(\cat F a)_!$\;}}}
\psfrag{T}[Bc][Bc]{\scalebox{1}{\scriptsize{$\cat F v$}}}
\psfrag{U}[Bc][Bc]{\scalebox{1}{\scriptsize{\;\;\;\;$(\cat F j)_!$}}}
\psfrag{F}[Bc][Bc]{\scalebox{1}{\scriptsize{$\cat F \alpha_1$}}}
\psfrag{G}[Bc][Bc]{\scalebox{1}{\scriptsize{$\cat F \alpha_2$}}}
\psfrag{N'}[Bc][Bc]{\scalebox{1}{\scriptsize{$\eta$}}}
\psfrag{E'}[Bc][Bc]{\scalebox{1}{\scriptsize{$\varepsilon$}}}
\psfrag{F'}[Bc][Bc]{\scalebox{1}{\scriptsize{\;\;$\fun$}}}
\psfrag{F''}[Bc][Bc]{\scalebox{1}{\scriptsize{$\fun^{-1}$\;\;\;\;\;}}}
\includegraphics[scale=.4]{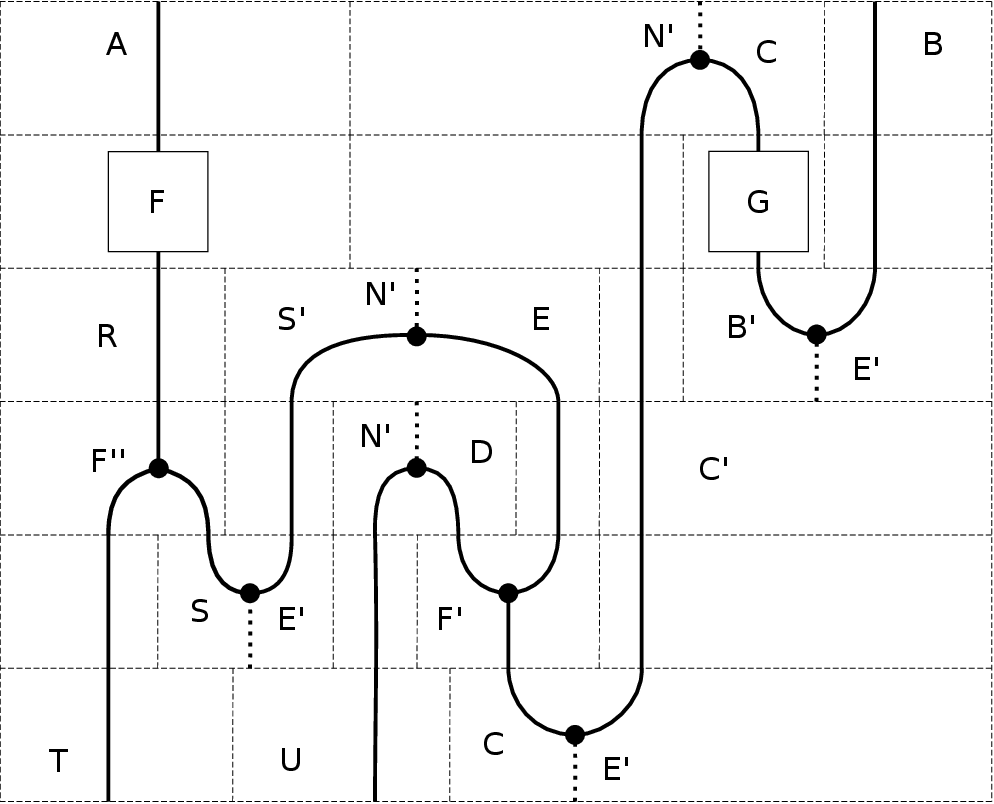}
}}
\end{equation}
Here we have also drawn, within each stripe, vertical boundaries separating the dot/boxes.
If the diagram happens to contain dot/boxes such as~\eqref{eq:general-ass} having more then two strings in either its domain or codomain, then they must all be broken down into more elementary components (this can create new stripes).

Then we choose, for each stripe, a bracketing of the dotted boxes within the stripe, and use it to horizontally compose the 2-cells corresponding to the dot/boxes in the stripe so as to form a 2-cell~$\sigma_i$.
For instance, for $i=3$ we may choose
$\sigma_i= (((\varepsilon \circ \id_{(\cat Fja)_!}) \circ \eta) \circ \id_{\cat Fva})$. Indeed, we may choose the canonical left bracketing for all~$i$.
Note that each boundary line between stripes inherits now two bracketed lists of 1-cells, corresponding to the codomain of the stripe above it and the domain of the one below it. These two bracketed lists may differ in general, by some missing or extra identity 1-cells and in the bracketing itself.
Fortunately we can connect the two associated 1-cells by inserting instances of the associativity and unitality isomorphisms of~$\cat B$, in order to vertically compose the $\sigma_i$'s. By the strictification theorem for bicategories, this is always possible and the resulting composite $\alpha$ does not depend on the choice of bracketing or of connecting isomorphisms. More generally, as mentioned above, the axioms of bicategories guarantee that this procedure (suitably formalized) evaluates to a unique $\alpha$ even under a wide class of deformations of the diagram.

If we want to translate the original string diagram into a cellular pasting diagram, this can be done in a straightforward way by forming the planar dual diagram of~\eqref{eq:app-example} with respect to the dotted boxes:
\[
\xymatrix{
\ar@{=}[d] \ar[rr]^-{\cat Fu} &&
 \ar@{=}[d] \ar@{=}[rrrr] &&&&
 \ar@{=}[d] \ar@{=}[rr] \ar@{}[drr]|{\Scell \,\eta\;\;} &&
 \ar@{=}[d] \ar[r]^-{(\cat Fi)_!} &
 \ar@{=}[d] \\
\ar@{=}[d] \ar[rr]^-{\cat Fu} \ar@{}[drr]|{\Scell\,\cat F\alpha_1} &&
 \ar@{=}[d] \ar@{=}[rrrr] &&&&
 \ar@{=}[d] \ar[r]_-{(\cat Fja)_!} &
 \ar@{=}[d] \ar[r]^-{\cat Fja} \ar@{}[dr]|{\Scell\,\cat F\alpha_2} &
 \ar@{=}[d] \ar[r]^-{(\cat Fi)_!} &
 \ar@{=}[d] \\
\ar@{=}[d] \ar[rr]_{\cat Fva} &&
 \ar@{=}[d] \ar@{=}[rrrr] \ar@{}[rrrrd]|{\Scell\,\eta} &&&&
 \ar@{=}[d] \ar[r]_-{(\cat Fja)_!} &
 \ar@{=}[d] \ar[r]_-{\cat Fi} \ar@{}[drr]|{\Scell\,\varepsilon} &
 \ar[r]_-{(\cat Fi)_!} &
 \ar@{=}[d] \\
\ar@{=}[d] \ar[rr]^-{\cat Fva} \ar@{}[drr]|{\Scell\,\fun^{-1}} &&
 \ar@{=}[d] \ar[r]^-{(\cat Fa)_!} &
 \ar@{=}[d] \ar@{=}[rr] \ar@{}[drr]|{\Scell\,\eta} &&
 \ar@{=}[d] \ar[r]^-{\cat Fa} &
 \ar@{=}[d] \ar[r]_-{(\cat Fja)_!} &
 \ar@{=}[d] \ar@{=}[rr] &&
 \ar@{=}[d] \\
\ar@{=}[d] \ar[r]_-{\cat Fv} &
 \ar@{=}[d] \ar[r]^-{\cat Fa} \ar@{}[drr]|{\Scell\,\varepsilon} &
 \ar[r]^-{(\cat Fa)_!} &
 \ar@{=}[d] \ar[r]_-{(\cat Fj)_!} &
 \ar@{=}[d] \ar[r]^-{\cat Fj} \ar@{}[drr]|{\Scell\, \fun} &
 \ar[r]^-{\cat Fa} &
 \ar@{=}[d] \ar[r]_-{(\cat Fja)_!} &
 \ar@{=}[rr] &&
 \ar@{=}[d] \\
\ar@{=}[d] \ar[r]_-{\cat Fv} &
 \ar@{=}[d] \ar@{=}[rr] &&
 \ar@{=}[d] \ar[r]_-{(\cat Fj)_!} &
 \ar@{=}[d] \ar[rr]_-{\cat Fja} \ar@{}[rrrrrd]|{\Scell\,\varepsilon} &&
 \ar[rrr]_-{(\cat Fja)_!} &&&
 \ar@{=}[d] \\
\ar[r]_-{\cat Fv} &
 \ar@{=}[rr] &&
 \ar[r]_-{(\cat Fj)_!} &
 \ar@{=}[rrrrr] &&&&&
}
\]
It is also straightforward to produce a string diagram by dualizing such a grid-like cellular pasting.

Note that in order to evaluate a cellular diagram as above in a non-strict bicategory we are equally obliged to insert coherence isomorphisms and make choices.
\end{Exa}

\bigbreak
\section{The ordinary category of spans}
\label{sec:ordinary-spans}%
\medskip
%
\begin{Def}
\label{Def:ordinary-spans}%
\index{category of spans}\index{span}%
\index{$(^$@$\widehat{(\ldots)}$ \, (ordinary) category of spans}
Let $\cat{C}$ be an essentially small category with pullbacks. We write
\[
\widehat{\cat{C}}
\]
for the \emph{ordinary category of spans}, where the objects are the same as those of $\cat{C}$, a morphism $X\to Y$ is an isomorphism class of spans $X \leftarrow S \rightarrow Y$ and composition is induced by taking pullbacks in the standard way of correspondences (or calculus of fractions). See \Cref{Rem:hat=truncate(Span)}. If a morphism $\varphi\colon X\to Y$ of $\widehat{\cat{C}}$ is represented by
\[
\xymatrix@C=12pt@R=12pt{
& S \ar[dl]_f \ar[dr]^g \\
X & & Y
}
\]
we also write $\varphi=[f,g]$. We use the notation
\[
(-)_\star \colon \cat{C} \longrightarrow \widehat{\cat{C}}
\quad \quad \textrm{and} \quad \quad
(-)^\star \colon \cat{C}^{\op} \longrightarrow \widehat{\cat{C}}
\]
for the two canonical functors given on morphisms by
\[
f \mapsto f_\star := [\id, f]
\quad \quad \textrm{and} \quad \quad
f \mapsto f^\star := [f, \id]\,.
\]
\end{Def}

We record a few basic properties of the construction $\widehat{\cat{C}}$.

\begin{Rem} \label{Rem:ordinary-spans}
\index{transposition!-- of ordinary spans}%
The operation $\varphi=[f,g]\mapsto \varphi^t:= [g,f]$ of switching the two legs of a span provides an involutive isomorphism of categories
\[
(-)^t\colon \widehat{\cat{C}}{\,}^{\op}\stackrel{\cong}{\longrightarrow} \widehat{\cat{C}}
\]
which is the identity on objects. We call it \emph{transposition}.
\end{Rem}

\begin{Prop}
\label{Prop:UP-ordinary-spans}%
The embeddings $(-)_\star$ and $(-)^\star$ are fully faithful and jointly satisfy the following property: To give a functor $F\colon \widehat{\cat{C}}\to \cat D$ to any other category~$\cat D$ is the same as to give a pair of functors $F_\star\colon \cat{C}\to \cat D$ and $F^\star\colon \cat{C}^{\op}\to \cat D$ such that
\begin{enumerate}[\rm(a)]
\item
\label{it:UP-ordinary-spans-a}%
$F_\star$ and $F^\star$ take the same values on objects, and
\item
\label{it:UP-ordinary-spans-b}%
for any pullback square in $\cat{C}$
\begin{align*}
\xymatrix@C=12pt@R=12pt{
& \ar[ld]_-{a} \ar[dr]^-{b} & \\
 \ar[dr]_c && \ar[dl]^d \\
& &
}
\end{align*}
the equation $F_\star(b)F^\star(a)= F^\star(d) F_\star(c)$ holds in~$\cat D$.
\end{enumerate}
The functor $F$ is the unique one such that $F\circ (-)_\star=F_\star$ and $F\circ (-)^\star=F^\star$:
\[
\xymatrix{
\cat{C} \ar[d]_-{(-)_{\star}} \ar@/^2ex/[rrd]^-{F_\star} && \\
{\widehat{\cat{C}}} \ar[rr]^-{F} && \cat D \\
\cat{C}^{\op} \ar[u]^-{(-)^{\star}} \ar@/_2ex/[urr]_-{F^\star} &&
}
\]
\end{Prop}

\begin{proof}
This is ultimately a (well-known) special case of the universal property in \Cref{Thm:UP-PsFun-Span}, which is also quite easy to verify directly.
\end{proof}

\begin{Lem}
\label{Lem:functoriality_ordinary-spans}%
The construction $\cat{C}\mapsto \widehat{\cat{C}}$ satisfies the following properties:
\begin{enumerate}[\rm(a)]
\item %
Every pullback-preserving functor $F\colon \cat{C}_1\to \cat{C}_2$ extends uniquely to a functor $\widehat{F}\colon \widehat{\cat{C}_1} \to \widehat{\cat{C}_2}$ commuting with the canonical embeddings $(-)_\star$ and $(-)^\star$, and given by the formula $\widehat{F}([f,g]) = [Ff, Fg]$.

\item
\label{it:funct-ordinary-spans-b}%
Each isomorphism $\alpha\colon F\overset{\sim}\Rightarrow G$ of pullback-preserving functors $F,G\colon \cat{C}_1\to \cat{C}_2$ yields an isomorphism $\hat\alpha\colon \widehat F\overset{\sim}\Rightarrow \widehat G$ defined by $\hat\alpha_X=(\alpha_X)_\star$.

\item %
Taking $\widehat{(-)}$ commutes with products of categories:
$\widehat{\cat{C}_1\!\times\! \cat{C}_2} = \widehat{\cat{C}_1} \times \widehat{\cat{C}_2}$.
\end{enumerate}
\end{Lem}

\begin{proof}
Part (a) follows from the universal property of Proposition~\ref{Prop:UP-ordinary-spans}, and the rest is equally straightforward. Note that there is no reason for (b) to be true in general if $\alpha$ is not invertible, because $(\alpha_X)_\star$ in~$\what{C}_2$ may not be natural in~$X$.
\end{proof}

\begin{Lem}\label{Lem:isos-in-spans}
Let $\cat{C}$ be a category with pullbacks. The only isomorphisms $\varphi$ in~$\widehat{\cat{C}}$ are those coming from~$\cat{C}$, that is, those of the form~$\varphi=f_\star$ for an isomorphism $f$ in~$\cat{C}$. It follows that every isomorphism $\varphi$ in~$\widehat{\cat{C}}$ satisfies $\varphi\inv=\varphi^t$ in the notation of Remark~\ref{Rem:ordinary-spans}.
\end{Lem}

\begin{proof}
Let $\varphi=[a,b]$ and $\psi=[c,d]$ be mutually inverse spans. Then $\psi\varphi=\id$ implies that $a$ and $d$ are split epis and the pullback maps $\tilde c$ and $\tilde b$ are split monos.
\[
\xymatrix@L=1pt@C=12pt@R=12pt{
&& \ar[ld]_-{\tilde c} \ar[dr]^-{\tilde b} && \ar[ld]_-{\tilde a} \ar[dr]^-{\tilde d} && \\
& \ar[dl]_a \ar[dr]^b && \ar[dl]_c \ar[dr]^d && \ar[dl]_a \ar[dr]^b & \\
\ar@{..>}[rr]_-{\varphi} && \ar@{..>}[rr]_-{\psi} && \ar@{..>}[rr]_-{\varphi} &&
}
\]
Similarly, $\varphi\psi=\id$ implies that $c$ and $b$ are split epis and the pullback maps $\tilde a$ and $\tilde d$ are split monos. Since pullbacks preserve split epis, we see that each of the maps $\tilde a$, $\tilde b$, $\tilde c$ and $\tilde d$ is both a split mono and a split epi, hence is invertible. Then the maps $a,b,c$ and~$d$, being their right or left inverses, are also isomorphisms. It follows that $\varphi=[a,b]=[\id,ba\inv]=f_\star$ for $f:=ba\inv$.

The second part follows immediately since $(f_\star)\inv=f^\star=(f_\star)^t$.
\end{proof}

We conclude with a few words relating the 1-category of spans~$\widehat{\cat{C}}$ to bicategories.

\begin{Rem}
\label{Rem:hat=truncate(Span)}%
We expand details about spans and their composition in \Cref{ch:2-motives}, even at the bicategorical level. In particular, the above $\widehat{\cat{C}}$ is nothing but the 1-truncation~$\pih$ (\Cref{Not:htpy_cat}) of the bicategory $\Span(\cat{C},\Mor\cat{C})$:
\[
\widehat{\cat{C}} = \pih{\Span(\cat{C}, \Mor \cat{C})}\,.
\]
Indeed, every ordinary category $\cat{C}$ can be viewed as a `locally discrete' $2$-category, \ie one whose only 2-cells are the identities. In such a (2,1)-category, iso-comma squares and weak pullbacks coincide with ordinary pullbacks. Moreover, every 1-cell (morphism of~$\cat{C}$) is trivially faithful. Thus if $\cat{C}$ is an essentially small category admitting arbitrary pullbacks, viewed as a (2,1)-category $\GG:=\cat{C}$ equipped with $\JJ = \Mor \cat{C}$ the class of all morphisms, we may construct the bicategory of spans $\Span(\cat{C}, \Mor \cat{C})=\Span(\GG;\JJ)$ as in \Cref{Def:Span-bicat}. This is the special case originally considered by B\'enabou~\cite{Benabou67} (note that B\'enabou credits the first bicategory of spans to Yoneda~\cite{Yoneda60}).
\end{Rem}

\bigbreak
\section{Additivity for categories}
\label{sec:additive-sedative}%
\medskip

The purpose of this section is to fix our terminology in relation to additivity, semi-additivity and idempotent-completion. All the results mentioned are standard.

\begin{Ter}
\label{Ter:additive_cat_etc}%
We consider the following additive notions for categories:
\begin{enumerate}[(1)]
\smallbreak
\item
\label{it:semi-add}
A pointed category~$\cat{A}$, with zero object $0$ (both initial and final), is called \emph{semi-additive} if for any two objects~$X,Y\in\cat{A}$, the coproduct $X\sqcup Y$ and the product $X\times Y$ exist and coincide, \ie are isomorphic via the canonical map $X\sqcup Y\to X\times Y$ with components $(\id_X, 0_{X,Y}, 0_{Y,X}, \id_Y)$, where $0_{U,V}\colon U\to 0\to V$ denotes the unique map factoring through~$0$. We denote by~$X\oplus Y$ this \emph{biproduct}. In that case every set $\cat{A}(X,Y)$ canonically becomes an abelian monoid with neutral element $0:=0_{X,Y}$ and with addition defined for all $f,g\in\cat{A}(X,Y)$ by
\begin{equation}
\label{eq:f+g}%
f+g=\big(X\otoo{\smat{1\\1}} X\times X=X\oplus X \otoo{f\oplus g\;\,} Y\oplus Y =Y\sqcup Y \otoo{(1\ 1)} Y\big)\,.
\end{equation}
Consequently $\cat A$ is a category \emph{enriched over abelian monoids}, in the strict sense of~\cite{Kelly05}, meaning that composition is bilinear: $h(f+g)=hf + hg$, $(f+g)h=fh + gh$, $0f=0$ and $g\,0=0$. Compare~\cite[VIII.2]{MacLane98}.
\smallbreak
\item
\label{it:direct-sum}%
\index{direct sum!-- in a category}%
In a category enriched over abelian monoids, a \emph{direct sum} of two objects~$X_1$ and~$X_2$ is a diagram
\begin{equation} \label{eq:direct-sum-def}
\xymatrix{
X_1 \ar@<2pt>[r]^-{i_1} \ar@{<-}@<-2pt>[r]_-{p_1} &
 X_1\oplus X_2 \ar@{<-}@<2pt>[r]^-{i_2} \ar@<-2pt>[r]_-{p_2} &
 X_2
}
\end{equation}
such that $p_\alpha i_\alpha=\id_{X_\alpha}$, $p_\alpha i_\beta = 0 $ (for $\{\alpha,\beta\}=\{1,2\}$) and $i_1p_1 + i_2p_2 = \id_{X_1\oplus X_2}$. Then $(X_1\oplus X_2,p_1,p_2)$ is the product $X_1\times X_2$ and $(X_1\oplus X_2,i_1,i_2)$ is the coproduct~$X_1\sqcup X_2$. Thus direct sums are the same as biproducts.
\smallbreak
\item
An \emph{additive} category~$\cat{A}$ is a semi-additive category such that every morphism $f\colon X\to Y$ has an opposite $-f\colon X\to Y$, \ie the abelian monoids $\cat{A}(X,Y)$ are abelian groups.
\end{enumerate}
\end{Ter}

\begin{Rem}
One can equivalently define a semi-additive (resp.\ additive) category as a category enriched over abelian monoids (resp.\ abelian groups) that admits all finite direct sums, including the empty sum which is~$0$.
\end{Rem}
\begin{Rem}
\label{Rem:matrix-notation}%
\index{matrix notation!-- for categories}%
In a semi-additive category~$\cat{A}$, if $X=X_1\oplus\ldots\oplus X_m$ and $Y=Y_1\oplus \ldots \oplus Y_n$ are two direct sums, we obtain an isomorphism of abelian monoids
\[
\cat{A}(X,Y) \cong \bigoplus_{\alpha,\beta} \cat{A}(X_\beta, Y_\alpha)
\qquad
f \mapsto (p_\alpha f i_\beta )_{\alpha,\beta}
\]
with inverse $(g_{\alpha,\beta})_{\alpha,\beta}\mapsto \sum_{\alpha,\beta} i_\alpha g_{\alpha,\beta} p_\beta$, which turns composition of maps into matrix multiplication. In particular,~\eqref{eq:f+g} reads $f+g = (1\;\;1)\big({}^f_0\;\; {}^0_g\big)\big({}^1_1\big)$.
\end{Rem}

There is only one relevant notion of additivity for functors:
\begin{Def}
\label{Def:additive_fun}%
\index{additive functor} \index{functor!additive --}%
Let $\cat{A}$ and $\cat{B}$ be semi-additive (\eg additive) categories. A functor $F\colon \cat{A}\to \cat{B}$ is called \emph{additive} if it preserves sums of morphisms, $F(f+g)=F(f)+F(g)$, and zero maps, $F(0_{X,Y})=0_{FX,FY}$. That is, $F$ is a functor of categories enriched over abelian monoids, in the sense of~\cite{Kelly05}.
\end{Def}

\begin{Rem} \label{Rem:autom-add}
An additive functor automatically preserves direct sums, up to a unique canonical isomorphism, by \Cref{Ter:additive_cat_etc}\,\eqref{it:direct-sum}. Conversely, if $F\colon \cat{A}\to \cat{B}$ sends the directs sums of $\cat{A}$ to direct sums of~$\cat{B}$ (\eg if $F$ is an equivalence) then $F$ must be additive because of~\eqref{eq:f+g}.
\end{Rem}

\begin{Not}
\label{Not:ADD}%
\index{$add$@$\ADD$ \, 2-category of additive categories} \index{ADD@$\ADD$}%
\index{$add$@$\Add$ \, 2-category of small additive categories} \index{Add@$\Add$}%
\index{$sad$@$\SAD$ \, 2-category of semi-additive categories} \index{SAD@$\SAD$}%
\index{$sad$@$\Sad$ \, 2-category of small semi-additive categories} \index{Sad@$\Sad$}%
\index{$fun+$@$\Funplus$ \, category of additive functors} \index{Funplus@$\Funplus$}%
Together with all natural transformations between them, the additive functors $\cat{A}\to \cat{B}$ form a category
\[
\Funplus(\cat{A}, \cat{B}) \,.
\]
Let $\Sad$ be the 2-category of all small categories which are semi-additive and whose Hom categories are the above $\Funplus(\cat{A},\cat{B})$. Let $\Add$ be the full sub-2-category of~$\Sad$ consisting of additive categories. Compositions are the obvious restrictions of those of~$\Cat$ in both cases.

It is also useful to allow all large semi-additive and all large additive categories, which gives rise to the (`very large') 2-categories $\SAD$ and $\ADD$, respectively, with all additive functors as 1-cells and all their natural transformations as 2-cells.
\end{Not}

\begin{Def}
\label{Def:idempotent-complete}%
\index{idempotent-complete additive category} \index{category!idempotent-complete additive --}%
One says that an additive category~$\cat{A}$ is \emph{idempotent-complete} if every idempotent endomorphism $e=e^2\colon X\to X$ in~$\cat{A}$ splits, \ie if it yields a decomposition $X\simeq X_1\oplus X_2$ under which $e$ becomes the projection~$\smat{1&0\\0&0}$ on~$X_1$. This decomposition yields $X=\img(e)\oplus \img(1-e)$, that is, both $e$ and $1-e$ have an image even if~$\cat{A}$ is not abelian. In particular, the direct summand~$\img(e)$ is unique up to unique isomorphism, and is functorial in~$e$.
We denote by~$\ICAdd$ the full sub-2-category of~$\Add$ consisting of idempotent-complete additive categories, and similarly for $\ICADD\subset\ADD$.
\end{Def}

\begin{Rem}
\label{Rem:completions}%
The fully faithful 2-functors $\ICAdd\hook \Add \hook \Sad$ are reflexive inclusions in the 2-categorical sense:
\[
\xymatrix@R=1.5em{
\Sad \ar@/_2em/@{..>}[d]_-{(-)_+}
\\
\Add \vcorrect{1} \ar@{_(->}[u] \ar@/_2em/@{..>}[d]_-{(-)^\natural}
\\
\ICAdd\!\!\vcorrect{1} \ar@{_(->}[u]
}
\]
In other words, for every semi-additive category~$\cat{A}$, there exists an additive functor $\cat{A}\to \cat{A}_+$, with $\cat{A}_+$ additive, that induces by pre-composition an equivalence
\begin{equation}
\label{eq:envelope_equiv_+}%
\Funplus(\cat{A}_+,\cat{B})\overset{\sim}{\too} \Funplus(\cat{A},\cat{B})
\end{equation}
for every additive category~$\cat{B}$. This construction $\cat{A}\mapsto \cat{A}_+$ uniquely extends to a 2-functor~$(-)_+\colon\Sad\to \Add$. Similarly, for every additive category~$\cat{C}$, there exists an additive functor $\cat{C}\to \cat{C}^\natural$, with $\cat{C}^\natural$ idempotent-complete, that induces by pre-composition an equivalence
\begin{equation}
\label{eq:envelope_equiv_ic}%
\Funplus(\cat{C}^\natural,\cat{D})\overset{\sim}{\too} \Funplus(\cat{C},\cat{D})
\end{equation}
for every idempotent-complete additive category~$\cat{D}$. This construction $\cat{C}\mapsto \cat{C}^\natural$ uniquely extends to a 2-functor~$(-)^\natural\colon\Add\to \ICAdd$.

\smallbreak

We recall the details for the reader's convenience:
\begin{enumerate}[(1)]
\item
\label{it:_+}%
\index{$(+$@$(\ldots)_+$ \, group-completion}%
\index{group-completion $(\ldots)_+$!-- for categories}%
\emph{Group-completion}~$\cat{A}_+$\,: This is a purely `enriched-category' construction. Recall that every abelian monoid~$M$ maps to the associated (Grothendieck) group $M_+$ obtained by formally adding opposites $-m$ for every $m\in M$ and declaring $m-n=m'-n'$ in~$M_+$ if $m+n'+\ell=m'+n+\ell$ in~$M$ for some~$\ell\in M$. This functor $M\mapsto M_+$ provides a left adjoint to the inclusion of abelian groups into abelian monoids, which is moreover monoidal (\ie $(M\times M')_+\cong M_+\times M'_+$). We can then group-complete any category~$\cat{A}$ enriched over abelian monoids by declaring $\cat{A}_+(X,Y):=\big(\cat{A}(X,Y)\big)_+$. By the universal property, there is an additive functor $\cat{A}\to \cat{A}_+$ inducing the required equivalence~\eqref{eq:envelope_equiv_+}, which in this case is in fact an \emph{isomorphism} of categories.
\smallbreak
\item
\label{it:ic}%
\index{$(n$@$(\ldots)^\natural$ \, idempotent-completion}%
\index{idempotent-completion $(\ldots)^\natural$!-- for categories}%
\emph{Idempotent-completion}~$\cat{C}^\natural$\,: Let~$\cat{C}$ be an additive category. Its \emph{idempotent-completion} (\aka \emph{Karoubi envelope})~$\cat{C}^\natural$ has objects given by pairs~$(X,e)$, where $X$ is an object of~$\cat{C}$ and $e=e^2\colon X\to X$ is an idempotent, and morphisms $f\colon (X,e)\to (X',e')$ given by morphisms $f\colon X\to X'$ in~$\cat{C}$ such that $f=e'fe$. The embedding $\cat{C}\to \cat{C}^\natural$ maps an object~$X$ to~$(X,\id)$ and a morphism~$f$ to~$f$. For every idempotent $e=e^2\colon X\to X$, we have $(X,e)\oplus (X,1-e)\cong(X,\id)$ in~$\cat{C}^\natural$. To prove the equivalence~\eqref{eq:envelope_equiv_ic}, note that every additive $F\colon \cat{C}\to \cat{D}$ extends uniquely to~$\cat{C}^\natural$ by mapping~$(X,e)$ to the summand $\img(F(e))$ of~$F(X)$.
\end{enumerate}
\end{Rem}

\begin{Exa}[Span categories] \label{Exa:ordinary_span_is_sad}
Let $\cat C$ be a category with pull-backs and consider the category $\widehat{\cat C}$ of spans in $\cat C$ (see \Cref{Def:ordinary-spans}). Assume moreover that $\cat C$ is \emph{extensive}, \ie it admits all finite coproducts and the functor
\[
\cat C_{/X} \times \cat C_{/Y} \stackrel{\sim}{\too} \cat C_{/X\sqcup Y}
\quad \quad
(f\colon S \to X, g\colon T \to Y) \mapsto (f\sqcup g\colon S\sqcup T \to X\sqcup Y)
\]
comparing comma categories is an equivalence for all objects $X,Y$. Extensive categories include finite sets, finite $G$-sets for a group~$G$, the (1-)category of finite groupoids, or indeed any (elementary) topos. Then its category of spans~$\widehat{\cat C}$ is semi-additive, as explained in \cite[\S3]{PanchadcharamStreet07}. Concretely, the empty coproduct (initial object) $\varnothing$ of $\cat C$ is a zero object in $\widehat{\cat C}$ and a coproduct $i_1:X_1\to X_1\sqcup X_2 \gets X_2:i_2$ provides the following direct sum:
\[
\xymatrix{
 X_1 \ar@<2pt>[r]^-{(i_1)_\star} \ar@{<-}@<-2pt>[r]_-{(i_1)^\star}
& X_1\sqcup X_2 \ar@{<-}@<2pt>[r]^-{(i_2)_\star} \ar@<-2pt>[r]_-{(i_2)^\star}
& X_2\,.
}
\]
The resulting addition~\eqref{eq:f+g} of two spans $[X\gets S_1 \to Y]$ and $[X\gets S_2 \to Y]$ is simply given by $[X \gets S_1 \sqcup S_2 \to Y]$, and the zero map $0_{X,Y}$ is $[X \gets \varnothing \to Y]$. If $F\colon \cat C\to \cat C'$ is a pullback- and coproduct-preserving functor, the induced functor $\widehat{F}\colon \widehat{\cat C}\to \widehat{\cat C'}$ (\Cref{Lem:functoriality_ordinary-spans}) is clearly additive. It is also common to consider the \emph{additive} category of spans $\widehat{\cat C}_+$, obtained by group-completion as in \Cref{Rem:completions}\,\eqref{it:_+}. Note that the canonical functor $\widehat{\cat C} \to \widehat{\cat C}_+$ is faithful, at least if $\cat C$ has the property that each object decomposes into a (up to isomorphism) unique finite coproduct of $\amalg$-indecomposable objects (as happens with finite $G$-sets, groupoids etc.). Indeed, the latter property is inherited by the comma categories $\cat C/U$ and implies that each Hom abelian monoid $\widehat{\cat C}(X,Y)=\Obj(\cat C/(X\times Y))/_{\cong}$ is free and therefore cancellable, so that the canonical homomorphism $\widehat{\cat C}(X,Y)\to \widehat{\cat C}(X,Y)_+$ is injective.
\end{Exa}

\begin{Exa}
The most common example of idempotent-completion~$\cat{C}^\natural$ might be the category of projective $R$-modules $R\textsf{-}\mathrm{Proj}\cong(R\textsf{-}\mathrm{Free})^\natural$ which is the idempotent-completion of the additive category of free $R$-modules, for every ring~$R$.
\end{Exa}

\bigbreak
\section{Additivity for bicategories}
\label{sec:additive-bicats}%
\medskip

We now extend the 1-categorical ideas of \Cref{sec:additive-sedative} to the realm of bicategories, insofar as needed in this work. Most notably, we introduce the notion of `block-completion' (\Cref{Def:block-complete} and \Cref{Cons:block-completion}) which reflects the possibility of decomposing 0-cells and 1-cells by way of idempotents. But first we apply the definitions and constructions of \Cref{sec:additive-sedative} locally, \ie `Hom-wise'.

\begin{Def}
\label{Def:Sad-enriched-etc}%
\index{bicategory!locally additive --}%
\index{bicategory!locally semi-additive --}%
\index{bicategory!locally idempotent-complete --}%
\index{pseudo-functor!locally additive --}
We say that a bicategory $\cat B$ is \emph{locally semi-additive} (resp.\ \emph{locally additive}, resp.\ \emph{locally idempotent-complete}), if all its Hom categories $\cat B(X,Y)$ are semi-additive (resp.\ additive, resp.\ idempotent-complete) and all its horizontal composition functors are additive functors of both variables. A pseudo-functor $\cat F\colon \cat B\to \cat B'$ between locally semi-additive bicategories is \emph{locally additive} if each component $\cat F\colon \cat B(X,Y)\to \cat B'(\cat FX,\cat FY)$ is an additive functor (\Cref{Def:additive_fun}).
\end{Def}

\begin{Rem} \label{Rem:coherent-additivity}
The coherent structure maps of any locally semi-additive bicategory are compatible with direct sums of 1-cells, that is, we have $\run_{f_1\oplus f_2} = \run_{f_1}\oplus \run_{f_2}$ for the right unitors, and similarly for left unitors and associators. This follows from naturality, and is a completely general fact: The components of any natural transformation $\alpha\colon F\Rightarrow G\colon \cat A\to \cat A'$ of additive functors between semi-additive categories decompose diagonally on direct sums: $\alpha_{x_1\oplus x_2}= \alpha_{x_1}\oplus \alpha_{x_2}$.
\end{Rem}

\begin{Rem}
\label{Rem:End-is-comm}%
If $\cat B$ is a locally additive bicategory, then for each object $X$ we have a 2-cell endomorphism \emph{ring} $\End_{\cat B(X,X)}(\Id_X)$. For short, we simply denote by $1_X$ its multiplicative unit $\id_{\Id_X}$. This ring is commutative by the standard commutative diagram:
\[
\xymatrix@=10pt{
\Id_X \ar[rrrr]^-{\alpha} \ar[dddd]_{\beta} &&&& \Id_X \ar[dddd]^{\beta} \\
& \Id_X \circ \Id_X \ar [ul]_{\simeq} \ar[rr]^-{\alpha \circ \id} \ar[dd]_{\id \circ \beta} \ar[ddrr]|{\alpha \circ \beta} && \Id_X \circ \Id_X \ar[ur]^{\simeq} \ar[dd]^{\id \circ \beta} & \\
&&&& \\
& \Id_X \circ \Id_X \ar[dl]^{\simeq} \ar[rr]^-{\alpha \circ \id} && \Id_X \circ \Id_X \ar[dr]_{\simeq} & \\
\Id_X \ar[rrrr]^-{\alpha} &&&& \Id_X\,.\!\!
}
\]
This is a form of Eckmann-Hilton argument: The two operations given by horizontal and vertical composition `mutually distribute' hence must coincide on~$\End(\Id_X)$.
\end{Rem}

\begin{Rem}
\label{Rem:add_reflections_bicats}%
By locally applying the constructions of \Cref{Rem:completions}, we obtain canonical forgetful and completion pseudo-functors comparing the three kinds of enriched bicategories of \Cref{Def:Sad-enriched-etc}. Let us be specific.
\begin{enumerate}[(1)]
\smallbreak
\item \label{it:additive-envelope-bicat}
\index{$(+$@$(\ldots)_+$ \, group-completion}%
\index{group-completion $(\ldots)_+$!-- for bicategories}%
If the bicategory~$\cat B$ is locally semi-additive there is a canonical pseudo-functor $\cat B\to \cat B_+$, where $\cat B_+$ is locally additive and through which every other pseudo-functor $\cat B \to \cat B'$ to some locally additive bicategory $\cat B'$ must factor (essentially) uniquely. The 0-cells of~$\cat{B}$ and~$\cat{B}_+$ are the same and the Hom categories of $\cat B_+$ are simply the group completions $\cat B_+(X,Y):=\cat B(X,Y)_+$. The pseudo-functor $\cat B\to \cat B_+$ has components given by the canonical functors $\cat B(X,Y)\to \cat B(X,Y)_+$; see \Cref{Rem:completions}\,\eqref{it:_+}. Note that $\cat{B}_+$ has the same 0-cells and 1-cells as~$\cat{B}$ and that 2-cells of~$\cat{B}_+$ are (formal) differences of 2-cells of~$\cat{B}$. It is immediate to see that the canonical embedding $\cat B\to \cat B_+$ induces, for any locally additive category~$\cat C$, a bi-equivalence (actually an isomorphism)
\[
\PsFun_+ (\cat B_+, \cat C) \stackrel{\simeq}{\too} \PsFun_+(\cat B,\cat C)
\]
of bicategories of \emph{locally additive} pseudo-functors, pseudo-natural transformations and modifications.

\smallbreak
\item \label{it:ic-envelope-bicat}
\index{$(n$@$(\ldots)^\natural$ \, idempotent-completion}%
\index{idempotent-completion $(\ldots)^\natural$!-- for bicategories}%
Similarly, if $\cat{C}$ is a locally additive bicategory there is a canonical pseudo-functor $\cat{C}\to \cat{C}^\natural$, where $\cat{C}^\natural$ is locally idempotent-complete and through which every other pseudo-functor $\cat{C} \to \cat{D}$ to some locally idempotent-complete bicategory~$\cat{D}$ must factor (essentially) uniquely. The 0-cells of~$\cat{C}$ and~$\cat{C}^\natural$ are the same and the Hom categories of $\cat{C}^\natural$ are simply the idempotent-completions $\cat{C}^\natural(X,Y):=\cat{C}(X,Y)^\natural$. The pseudo-functor $\cat{C}\to \cat{C}^\natural$ has components given by the canonical functors $\cat{C}(X,Y)\to \cat{C}(X,Y)^\natural$; see \Cref{Rem:completions}\,\eqref{it:ic}. Note that the 0-cells of~$\cat{C}^\natural$ are the same as those of~$\cat{C}$, that 1-cells of~$\cat{C}^\natural$ are direct summands of 1-cells of~$\cat{C}$ (\ie 1-cells together with an idempotent 2-cell), and that 2-cells of~$\cat{C}^\natural$ are 2-cells of~$\cat{C}$ which are compatible with the relevant idempotents. For every locally idempotent-complete bicategory~$\cat D$, the embedding $\cat C\to \cat C^\natural$ induces a biequivalence
\index{$psfun+$@$\PsFun_+$ \, locally additive pseudo-functors} \index{PsFun@$\PsFun_+$}%
\[
\PsFun_+ (\cat C^\natural, \cat D) \stackrel{\simeq}{\too} \PsFun_+(\cat C,\cat D)
\]
of bicategories of locally additive pseudo-functors. (We leave this as an easy exercise. The existence and uniqueness of the extension of pseudo-natural transformations~$t$ uses that their 2-cell components $t_f$ decompose diagonally for direct sums of 1-cells: $t_f=t_{f_1}\oplus t_{f_2}$ if $f=f_1\oplus f_2$. This is a consequence of the naturality axiom similarly to \Cref{Rem:coherent-additivity}.)
\end{enumerate}
\end{Rem}

\begin{Exa}
\label{Exa:ICADD}%
The 2-category~$\ICADD$ of idempotent-complete additive categories is itself locally idempotent-complete as a bicategory. In other words, we have $\ICADD\stackrel{\sim}{\to}(\ICADD)^\natural$.
Indeed, for additive categories~$\cat{A}$ and~$\cat{B}$ with ($\cat{A}$ and) $\cat{B}$ idempotent-complete, the category $\Funplus(\cat{A},\cat{B})$ of additive functors from~$\cat{A}$ to~$\cat{B}$ is idempotent-complete. To see this, observe that if $e=e^2\colon F\Rightarrow F$ is an idempotent natural transformation of~$F\colon \cat{A}\to\cat{B}$ then for every $x\in \Obj\cat{A}$ the idempotent $e_x\colon F(x)\to F(x)$ yields a decomposition $F(x)\cong \img(e_x)\oplus \img(1-e_x)$ in~$\cat{B}$; we can then decompose $F\cong F_1\oplus F_2$ where one defines $F_1\colon \cat{A}\to \cat{B}$ by mapping an object~$x$ to~$\img(e_x)$ and a morphism $f\colon x\to x'$ to $F(f)e_x=e_{x'}F(f)$, which restricts to~$\img(e_x)\to \img(e_{x'})$, and similarly for~$F_2$ with~$1-e_x$ instead of~$e_x$.
\end{Exa}

After the above discussion of Hom categories in bicategories~$\cat{B}$, we now turn to constructions involving 0-cells:
\begin{Def}
\label{Def:sums_in_bicats}%
Let $\cat B$ be a bicategory.
\begin{enumerate}[(1)]
\smallbreak
\item
A \emph{final object} of~$\cat{B}$ is an object~$\final$ with the property that $\cat B(X,\final)\stackrel{\sim}{\to}1$ is an equivalence for all $X\in \cat B_0$, where as before~$1$ denotes the final category, which has one object and one morphism. Dually, an initial object of~$\cat{B}$ is a $\varnothing\in \cat B_0$ such that $\cat B(\varnothing , Y)\stackrel{\sim}{\to}1$ is an equivalence, and a \emph{zero} object~$0$ is one which is both initial and final.
\smallbreak
\item
A \emph{product} of 0-cells~$X_1$ and~$X_2$ is a pair of 1-cells $p_1\colon X_1\leftarrow X_1\times X_2 \to X_2:\!p_2$ inducing an equivalence
\[
({p_1}_*, {p_2}_*)\colon \cat B(Y,X_1\times X_2) \stackrel{\sim}{\too} \cat B(Y,X_1)\times \cat B(Y,X_2)
\]
of Hom categories for all~$Y$. Dually, a diagram $X_1\to X_1\sqcup X_2\lto X_2$ is a \emph{coproduct} if it induces equivalences $\cat B(X_1\sqcup X_2, Y)\stackrel{\sim}{\to} \cat B(X_1,Y)\times \cat B(X_2,Y)$.
\smallbreak
\item
Assume that the Hom categories of $\cat B$ admit zero objects $0=0_{X,Y}\in \cat B(X,Y)$ and that they are preserved by horizontal composition. Then if $\cat B$ has a product $X_1\times X_2$ and a coproduct $X_1\sqcup X_2$ we may define a (unique up to isomorphism) comparison 1-cell
$X_1\sqcup X_2 \rightarrow X_1 \times X_2$ determined by the four components $(\Id_{X_2}, 0_{X_1,X_2},0_{X_2,X_1}, \Id_{X_2})$. If the latter is an equivalence, we may equip $X_1\sqcup X_2$ (or equivalently $X_1\times X_2$) with the structure both of a product and of a coproduct and call it a \emph{biproduct} of $X_1$ and~$X_2$.
\smallbreak
\item
Assume that $\cat B$ is locally semi-additive (\Cref{Def:Sad-enriched-etc}).
A \emph{direct sum} in~$\cat B$ is a diagram of 1-cells as in~\eqref{eq:direct-sum-def} for which there exist isomorphisms
\begin{equation*}
p_1\circ i_1 \simeq \Id_{X_1}, \quad
p_2\circ i_2 \simeq \Id_{X_2}, \quad
p_2\circ i_1 \simeq 0_{X_2,X_1}, \quad
p_1\circ i_2 \simeq 0_{X_1,X_2},
\end{equation*}
\begin{equation*}
i_1p_1 \oplus i_2p_2 \simeq \Id_{X_1\oplus X_2}
\end{equation*}
where the latter uses the direct sum in the category $\cat B(X_1\oplus X_2,X_1\oplus X_2)$.
\end{enumerate}
Each of the above notions is called \emph{strict} if the equivalences are actually isomorphisms (in the latter, if the isomorphisms are equalities).
We extend the definitions as usual to finite products, coproducts, biproducts and direct sums, the empty case being defined to be an initial, final, and both initial and final (\ie zero) object.
\end{Def}

The next lemma provides a strong link between (co)products of 0-cells and (co)products of 1-cells in each Hom category (\cf \Cref{Ter:additive_cat_etc}\,\eqref{it:direct-sum}).

\begin{Lem}
\label{Lem:0-1-direct-sums}%
In any locally semi-additive bicategory~$\cat B$, biproducts and direct sums are equivalent notions.
\end{Lem}

\begin{proof}
We leave this as a straightforward exercise for the reader, which makes crucial use of the fact that the horizontal composition functors of $\cat B$ preserve direct sums (\ie biproducts) of 1-cells in both variables.
\end{proof}

\begin{Exa}
The 2-category $\groupoid$ of all finite groupoids admits all finite (strict) coproducts, provided by the usual disjoint unions of categories, and $\groupoid$ also admits (strict) finite products provided by the usual product of categories.
\end{Exa}

\begin{Exa}
\label{Exa:ADD-dir-sums}%
The bicategory $\SAD$ of semi-additive categories admit all finite direct sums. Indeed, the usual (strict) product $\cat{A}_1\oplus \cat{A}_2:=\cat{A}_1\times \cat{A}_2$ of two semi-additive categories is again so, with objectwise direct sums $(x_1,x_2)\oplus (y_1,y_2)=(x_1\oplus y_1, x_2\oplus y_2)$ and zero object $0=(0,0)$. For such categories the product is also a (non~strict) coproduct, with structural 1-cells given by the embedding functors $\cat{A}_1\hook\cat{A}_1\oplus \cat{A}_2$, $x_1\mapsto (x_1,0)$ and $\cat{A}_2\hook\cat{A}_1\oplus \cat{A}_2$, $x_2\mapsto (0,x_2)$. If $\cat{A}_1$ and $\cat{A}_2$ happen to be additive or idempotent-complete then so is $\cat{A}_1\oplus \cat{A}_2$, hence $\ADD$ and $\ICADD$ also admit direct sums. The zero category $0:=1$ is a zero object in each.
\end{Exa}

\begin{Rem}
\label{Rem:add-adjoint}%
\index{matrix notation!-- for bicategories}%
As in ordinary categories, we can use the standard matrix notation (\Cref{Rem:matrix-notation}) for 1-cells into, or out of, a direct sum in a bicategory, as well as 2-cells between such 1-cells. For instance in $\ADD$, if the components $F_1$ and $F_2$ of a functor $F=\smat{F_1\\ F_2}\colon \cat{A}\too \cat{B}_1\oplus \cat{B}_2$ admit left adjoints ${F_1}_!\adj F_1$ and ${F_2}_!\adj F_2$ then these are the components of the left adjoint of~$F$, that is, $F_!=\smat{{F_1}_! & {F_2}_!} \adj F$ with compatible units and counits:
\begin{align*}
&\xymatrix@C=4em{
\eta\colon \;\; \Id_{\cat{B}_1\oplus \cat{B}_2}=\smat{\Id_{\cat{B}_1}& 0\\0&\Id_{\cat{B}_2}}
 \ar[r]^-{\smat{\eta_1&0\\0&\eta_2}}
& \ \smat{F_1{F_1}_!&F_1{F_2}_!\\F_2{F_1}_!&F_2{F_2}_!} =\smat{F_1\\F_2}\smat{{F_1}_!&{F_2}_!}=FF_!
}
\\
&\xymatrix@C=4em{
\eps\colon \;\; F_!F=\smat{{F_1}_!&{F_2}_!}\smat{F_1\\F_2}={F_1}_!F_1\oplus {F_2}_!F_2 \
 \ar[r]^-{\smat{\eps_1& \eps_2}}
& \ \Id_{\cat{A}}\,.
}
\end{align*}
\end{Rem}

\begin{Def}
\label{Def:additive-pseudofunctor}%
\index{$psfun$@$\PsFun_\amalg$ \, additive pseudo-functors}%
\index{PsFun@$\PsFun_\amalg$}%
In this work, we say that a pseudo-functor $\cat F\colon \cat B\to \cat B'$ between bicategories with products is \emph{additive} if it preserves products, \ie if the canonical comparison 1-cells
\[
\cat F(\final) \stackrel{\sim}{\to} \final
\quad \textrm{ and } \quad
\cat F(X_1 \times X_2) \stackrel{\sim}{\to} \cat F(X_1) \times \cat F(X_2)
\]
are equivalences for all objects\,(\footnote{\,One can check that the latter implies the former provided there exists a 1-cell $\final\to \cat F(\final)$, which must be the case for instance when the target bicategory is pointed (\eg $\cat B'=\ADD$).}). Typically, we consider functors $\MM\colon \GG^\op\to \ADD$ where $\cat B=\GG^\op$ for $\GG$ a bicategory of finite groupoids closed under finite coproducts in~$\groupoid$, which then become products in the opposite bicategory~$\cat{B}$. Thus additivity for $\MM$ translates into the by~now familiar axiom~\Mack{1} for Mackey 2-functors or (Der\,\ref{Der-1}) for derivators. Because of this context, we use the notation
\[
\PsFun_\amalg(\cat B, \cat B') \subseteq \PsFun(\cat B,\cat B')
\]
for the 1- and 2-full sub-bicategory of additive pseudo-functors, decorated with `$\amalg$' rather than the more logical~`$\Pi$'.
\end{Def}

\begin{Lem}
\label{Lem:diagonal-dir-sum}%
Let $\cat B$ be a locally semi-additive bicategory and assume that the direct sums $X\oplus X$ and $Y\oplus Y$ exist for some $X,Y\in \cat B_0$. Then the direct sums in the Hom category $\cat B(X,Y)$ are given by the composite functor
\begin{equation} \label{eq:dir-sums-rewritten}
\xymatrix@C=4em{
\cat B(X,Y) \times \cat B(X,Y) \ar[r]^-{-\oplus -} & \cat B(X\oplus X, Y \oplus Y) \ar[r]^-{(\Delta^*,\nabla_*)} & \cat B(X,Y) \,,
}
\end{equation}
where $\Delta=\smat{1 \\ 1} \colon X\to X\oplus X$ and $\nabla=\smat{1\;\; 1} \colon X\oplus X\to X$ are the diagonal and co-diagonal 1-cells, and where $-\oplus -$ is induced by the direct sums of 0-cells.
\end{Lem}
\begin{proof}
The definition of direct sums of 0-cells gives us an equivalence
\[
\cat B(X\oplus X,Y\oplus Y) \stackrel{\sim}{\too} \cat B(X,Y) \oplus \cat B(X,Y) \oplus \cat B(X,Y) \oplus \cat B(X,Y) \,.
\]
In particular, any pair of 1-cells $f_1\colon X\to Y$ and $f_2\colon X\to Y$ gives rise to the 1-cell $\smat{f_1 \;\; 0\\ 0 \;\;\; f_2} \colon X\oplus X\to Y\oplus Y$ corresponding to the four components $(f_1,0_{X,Y},0_{X,Y},f_2)$, and similarly for 2-cells. This is what the first functor $-\oplus -$ in \eqref{eq:dir-sums-rewritten} does. Note that $\smat{f_1 \;\; 0\\ 0 \;\;\; f_2}$ coincides with the 1-cell direct sum $f_1\oplus f_2 = \smat{f_1 \;\; 0\\ 0 \;\;\;\; 0} \oplus \smat{0 \;\; 0\\ \; 0 \;\; f_2}$ in $\cat B(X\oplus X,Y\oplus Y)$, since the two have the same 2-cell Hom groups
\[
\Hom(f_1,g)\oplus \Hom(f_2,g) \quad \textrm{and} \quad \Hom(g,f_1)\oplus \Hom(g,f_2)
\]
to and from any other 1-cell $g\colon X\oplus X\to Y\oplus Y$ (as they are calculated component\-wise).
Since the horizontal pre- and post-composition whiskering functors $-\circ \Delta$ and $\nabla \circ -$ are additive, they preserve direct sums of 1-cells, hence the second functor in \eqref{eq:dir-sums-rewritten} must send $\smat{f_1 \;\; 0\\ 0 \;\;\; f_2}$ to the direct sum $f_1\oplus f_2$ in $\cat B(X,Y)$, as claimed.
\end{proof}

The next result is as amusing as it is useful:

\begin{Prop} \label{Prop:(locally)-additive-pseudo-functors}
Let $\cat F$ be any pseudo-functor between locally additive bicategories with finite direct sums.
Then $\cat F$ is additive (\Cref{Def:additive-pseudofunctor}) if and only if it is locally additive (\Cref{Def:Sad-enriched-etc}). In other words: $\cat F$ preserves direct sums of 0-cells iff it preserves direct sums of 1-cells iff it preserves sums of 2-cells.
\end{Prop}

\begin{proof}
The last claimed equivalence follows by applying \Cref{Rem:autom-add} to all the functors $\cat F_{X,Y}\colon \cat B(X,Y)\to \cat B'(\cat FX,\cat FY)$.
If $\cat F$ is locally additive then it preserves all direct sum diagrams, which are expressed in terms of isomorphisms between composites and direct sums of 1-cells. Since direct sums are (bi)products by \Cref{Lem:0-1-direct-sums}, the pseudo-functor $\cat F$ is additive.
Conversely, an additive pseudo-functor must preserve the direct sums in each Hom category by \Cref{Lem:diagonal-dir-sum}.
\end{proof}

Next, we consider the link between direct sum decompositions and idempotents. The local idempotent-completion $\cat C\mapsto \cat C^\natural$ of \Cref{Rem:add_reflections_bicats} is somewhat unsatisfactory, because it does not account for decompositions of 0-cells:

\begin{Rem} \label{Rem:object-decompositions}
Let $\cat B$ be a locally additive bicategory, and suppose we have an equivalence $X\simeq X_1\oplus X_2$ decomposing an object $X$ into a direct sum of two objects. We obtain an induced equivalence of additive categories
\[
\cat B(X,X) \simeq \cat B(X_1,X_1) \oplus \cat B(X_1,X_2) \oplus \cat B(X_2,X_1) \oplus \cat B(X_2,X_2)
\]
which allows us to write 1-cells $X\to X$ and their morphisms in matrix notation (\cf \Cref{Rem:add-adjoint}). In particular, the identity 1-cell $\Id_X$ has components $(\Id_{X_1}, 0, 0, \Id_{X_2})$ and its (by \Cref{Rem:End-is-comm}, commutative) endomorphism ring decomposes into a product as follows:
\begin{align*}
\End_{\cat B(X,X)}(\Id_X)
& \;\simeq\; \End \left( \begin{array}{cc} \Id_{X_1} & 0 \\ 0 & \Id_{X_2} \end{array} \right) \\
& \;\simeq\; \End_{\cat B(X_1,X_1)}(\Id_{X_1}) \times \End_{\cat B(X_2,X_2)}(\Id_{X_2}) \,.
\end{align*}
Moreover, such a ring decomposition corresponds to a decomposition of its unit $1_X:=\id_{\Id_X}$ as a sum of two orthogonal idempotents:
\begin{equation} \label{eq:orthog-decomp}
1_{X}=e_1 + e_2 \quad\textrm{ with } \quad e_1^2= e_1, \quad e_2^2 = e_2 \quad \textrm{ and } \quad e_1e_2=0 \,.
\end{equation}
In the opposite direction, however, if we are given a sum decomposition $1_{X}=e_1 + e_2$ in orthogonal idempotents, nothing guarantees the existence of a direct sum decomposition $X\simeq X_1\oplus X_2$ in~$\cat B$ giving rise to it. This motivates the following.
\end{Rem}

\begin{Def} \label{Def:block-complete}
\index{block-complete} \index{bicategory!block-complete --}%
A locally idempotent-complete bicategory $\cat B$ (\Cref{Def:Sad-enriched-etc}) is \emph{block-complete} if
\begin{enumerate}[(1)]
\item it admits all finite direct sums (\Cref{Def:sums_in_bicats}), and
\item
\index{block-decomposition}%
it admits \emph{block-decompositions}. By the latter we simply mean that, whenever the identity 2-cell $1_X=\id_{\Id_X}$ of an object $X$ decomposes as a sum of two orthogonal idempotents as in \eqref{eq:orthog-decomp}, then there exist objects $X_1,X_2$ and an equivalence $X\simeq X_1\oplus X_2$ identifying the idempotents $e_1$ and $e_2$ with the 2-cells $\big({}^1_0\;\; {}^0_0\big)$ and $\big({}^0_0\;\; {}^0_1\big)$, respectively.
\end{enumerate}
The second condition is of course equivalent to its analogue with $n$ rather than two summands, for $n\geq2$ arbitrary.
\end{Def}

\begin{Exa} \label{Exa:ICAdd-is-lic}
The 2-category $\ICAdd$ of idempotent-complete additive categories, and its very large version $\ICADD$, are block-complete. Indeed, they are locally idempotent-complete by \Cref{Exa:ICADD} and admit finite direct sums by \Cref{Exa:ADD-dir-sums}. As for block-decompositions, consider a sum decomposition $1_{\cat{A}}= e_1 + e_2$ in orthogonal idempotent natural transformations of the identity $1_\cat{A}=\id_{\Id_\cat{A}}$ of some idempotent-complete category~$\cat{A}$.
 For each $i=1,2$ and each object $x \in \cat{A}$, the component $e_{i,x}\colon x\to x$ is an idempotent in~$\cat{A}$. We therefore obtain a splitting $x\cong \img(e_{1,x})\oplus \img(e_{2,x})$ in~$\cat{A}$ identifying $e_{1,x}$ and $e_{2,x}$ with $\big({}^1_0\;\; {}^0_0\big)$ and $\big({}^0_0\;\; {}^0_1\big)$, respectively (\cf \Cref{Def:idempotent-complete}).
By the functoriality of images, this defines two endofunctors (for $i=1,2$)
\[
E_i\colon \cat{A} \to \cat{A}, \quad x \mapsto E_i(x):= \img(e_{i,x})
\]
with the property that the identity functor $\Id_\cat{A}$ is isomorphic to the direct sum $E_1\oplus E_2$ in $\ADD(\cat{A},\cat{A})$. We now define (for $i=1,2$)
\[
\cat{A}_i := \mathrm{Im}(E_i) \subseteq \cat{A}
\]
to be the full replete image of $E_i$ in $\cat{A}$. Using the orthogonality of $\cat{A}_1$ and $\cat{A}_2$ within~$\cat{A}$, which follows from that of~$e_1$ and~$e_2$, it is now straightforward to verify that the two functors
\[
\cat{A} \to \cat{A}_1 \oplus \cat{A}_2 , \quad x \mapsto (E_1x, E_2 x)
\]
and
\[
\cat{A}_1 \oplus \cat{A}_2 \to \cat{A} , \quad (x_1,x_2) \mapsto x_1\oplus x_2
\]
are mutually inverse equivalences. This shows that $1_\cat{A}= e_1+e_2$ is realized by the block-decomposition $\cat{A}\simeq \cat{A}_1\oplus \cat{A}_2$, as wished.
\end{Exa}

Now we take a closer look at the image 1-cells $E$ of idempotent 2-cells~$e$.

\begin{Not} \label{Not:idemp-E}
Let $e=e^2\colon \Id_X \Rightarrow \Id_X$ be an idempotent 2-cell in a locally additive bicategory~$\cat B$ such that $\cat B(X,X)=\cat B(X,X)^\natural$ is idempotent-complete. Then $e$ splits: There exists a 1-cell $E=\img(e)\colon X\to X$ together with 2-cells $r_e\colon \Id_X\Rightarrow E$ and $i_e\colon E\Rightarrow \Id_X$ such that
$i_er_e=e$ and $r_ei_e=\id_E$, this data being unique up to a unique isomorphism.
For any 1-cell $F\colon X\to Y$, we can define two 2-cells $\rho_{e,F}\colon F\circ E\Rightarrow F$ and $\overline{\rho}_{e,F}\colon F\Rightarrow F\circ E$ by the following two pastings:
\[
\rho_{e,F}:=\quad
\vcenter { \hbox{
\xymatrix{
& X
 \ar@/^3ex/[dr]^-F
 \ar@{}[dd]|{\;\;\quad \simeq\Scell\; \run_F} & \\
X
 \ar@/^3ex/[ur]^-{E}
 \ar@/_3ex/[ur]_>>>>{\Id_X}
 \ar@/_5ex/[rr]_-{F}
 \ar@{}[ur]|{i_e\;\Scell} &&
 Y \\
&&
}
}}
\quad \quad \quad \quad
\overline{\rho}_{e,F} :=\quad
\vcenter { \hbox{
\xymatrix{
&& \\
X
 \ar@/^5ex/[rr]^-{F}
 \ar@/_3ex/[dr]_-{E}
 \ar@/^3ex/[dr]^>>>{\Id_X}
 \ar@{}[dr]|{r_e\;\Scell}
 \ar@{}[rr]|{\quad\; \simeq \Scell \; \run_F^{-1}} && Y \\
& X
 \ar@/_3ex/[ur]_-{F} &
}
}}
\]
Similarly, for any 1-cell $G\colon Z\to X$ we define $\lambda_{e,G}\colon E\circ G\Rightarrow G$ and $\overline{\lambda}_{e,G}\colon G\Rightarrow E\circ G$ by the following two pastings:
\[
\lambda_{e,G}:=\quad
\vcenter{ \hbox{
\xymatrix{
& X
 \ar@/^3ex/[dr]^-{E}
 \ar@/_3ex/[dr]_<<<{\Id_X}
 \ar@{}[dr]|{\Scell\; i_e}
 & \\
Z
 \ar@/_5ex/[rr]_-{G}
 \ar@/^3ex/[ur]^-{G}
 \ar@{}[rr]|{\lun_G\;\Scell\simeq \quad} && X \\
&&
}
}}
\quad \quad \quad \quad
\overline{\lambda}_{e,G} :=\quad
\vcenter{ \hbox{
\xymatrix{
&& \\
Z
\ar@/^5ex/[rr]^-{G}
 \ar@{}[rr]|{\lun_G^{-1} \Scell\simeq \;\quad }
 \ar@/_3ex/[dr]_-{G} && X \\
& X
 \ar@/^3ex/[ur]^<<<{\Id_X}
 \ar@/_3ex/[ur]_-{E}
 \ar@{}[ur]|{\Scell\; r_e} &
}
}}
\]
\end{Not}

\begin{Lem} \label{Lem:absorption}
Retaining \Cref{Not:idemp-E}, we have:
\begin{enumerate}[\rm(a)]
\item For every $F\colon X\to Y$, the 2-cell $\rho_{e,F}\colon F\circ \img(e)\Rightarrow F$ is invertible if and only if \emph{$F$ absorbs~$e$}, in the sense that $F e=\id_F$, \ie modulo right unitors we can identify the 2-cells
\[
\id_F \cong
\left(
\xymatrix@1{ X \ar@/^3ex/[r]^-{\Id_X} \ar@/_3ex/[r]_-{\Id_X} \ar@{}[r]|{\Scell\; e} & X \ar[r]^-F & Y }
\right) \,.
\]
The latter is further equivalent to the condition $F\circ (1_X-e)\cong 0$.
\item
\index{absorbs}%
For every $G\colon Z\to X$, the 2-cell $\lambda_{e,G}\colon \img(e)\circ G\Rightarrow G$ is invertible if and only if \emph{$G$ absorbs~$e$}, in the sense that $e G=\id_G$, \ie modulo left unitors we can identify the 2-cells
\[
\id_G \cong
\left(
\xymatrix@1{ Z \ar[r]^-G & X \ar@/^3ex/[r]^-{\Id_X} \ar@/_3ex/[r]_-{\Id_X} \ar@{}[r]|{\Scell\; e} & X }
\right) \,.
\]
The latter is further equivalent to the condition $(1_X-e)\circ G\cong 0$.
\end{enumerate}
\end{Lem}

\begin{proof}
By the equation $r_ei_e=\id_E$, it is always the case that
$\overline{\rho}_{e,F}\rho_{e,F}=\id_{F\circ E}$ and
$\overline{\lambda}_{e,G}\lambda_{e,G}=\id_{E\circ G}$.
Moreover, we see by the equation $i_er_e=e$ that $\rho_{e,F} \overline{\rho}_{e,F} =\id_{F}$ precisely when $F$ absorbs~$e$; and similarly for~$G$. The equivalence with the vanishing statements is trivial since $1_X=e+(1_X-e)$.
\end{proof}

\begin{Rem} \label{Rem:E-absorbs-e}
Note that $E=\img(e)\colon X\to X$ itself absorbs~$e$ on both sides.
This follows from the vanishing conditions for absorption in \Cref{Lem:absorption}, the vanishing of the vertical composite $(1_X-e)e$, and the fact that the unitors identify the vertical and horizontal compositions of 2-cells $\Id_X\Rightarrow \Id_X$ (\cf \Cref{Rem:End-is-comm}).
\end{Rem}

We are now ready for the main construction of this section.

\begin{Cons}[Block completion] \label{Cons:block-completion}
\index{$(f$@$(\ldots)^\flat$ \, block completion}%
\index{block completion $(\ldots)^\flat$}%
Let $\cat{B}$ be a locally idempotent-complete additive bicategory with all finite direct sums of 0-cells (\Cref{Def:sums_in_bicats}). Let us construct a new bicategory~$\cat{B}^\flat$ called the \emph{block-completion} of~$\cat{B}$, as follows.
\begin{enumerate}[{$\bullet$}]
\item
The 0-cells of~$\cat{B}^\flat$ consist of pairs~$(X,e)$ where $X$ is a 0-cell of~$\cat{B}$ and $e=e^2\colon \Id_{X}\to \Id_{X}$ is an idempotent 2-cell of the identity 1-cell of~$X$ in the category~$\cat{B}(X,X)$. (Since we assume $\cat{B}(X,X)$ idempotent-complete, this idempotent~$e$ corresponds to a decomposition $\Id_{X} \cong E_1\oplus E_2$ where $E_1=\img(e)$ and~$E_2=\img(1-e)$.)
\smallbreak
\item
The Hom-category $\cat B^\flat((X,e),(X',e'))$ is the full subcategory of~$\cat{B}(X,X')$ of those 1-cells~$F\colon X\to X'$ which absorb the idempotents $e$ and $e'$ in the equivalent senses of \Cref{Lem:absorption} (for instance: $e'\circ F\cong \id_{F} \cong F\circ e$).
\smallbreak
\item
The composition functors of~$\cat{B}^\flat$ are simply restricted from those of~$\cat{B}$. The identity 1-cell~$\Id_{(X,e)}$ is given by the direct summand~$E=\img(e)\leq \Id_{X}$, which exists because $\cat{B}(X,X)$ is idempotent-complete and which belongs to $\cat B^\flat$ by \Cref{Rem:E-absorbs-e}.
\smallbreak
\item
The associators $\ass_{X,Y,Z}$ of $\cat B^\flat$ are those of~$\cat B$.
For any 1-cell $F\colon (X,e)\to (X',e')$, we define the right and left unitors $\run_F\colon F\circ \Id_{(X,e)} \Rightarrow F$ and $\lun_F\colon \Id_{(X',e')}\circ F \Rightarrow F$ to be the 2-cells $\rho_{e,F}$ and $\lambda_{e',F}$ of \Cref{Not:idemp-E}. They are invertible by \Cref{Lem:absorption} since $F$ absorbs~$e$ and~$e'$.
\end{enumerate}
It is straightforward to verify that the above data defines a bicategory. The only possible issue concerns the unitor coherence axioms, which follow from the idempotency of $e$ and an identification of vertical and horizontal compositions as in \Cref{Rem:E-absorbs-e}.

We have a (rather strict) pseudo-functor $\cat{B}\to \cat{B}^\flat$ mapping 0-cells $X$ to~$(X,1)$ and which is the identity on 1-cells and 2-cells. (Here we use the canonical identification $\img(1_X)=\Id_X$ in~$\cat{B}(X,X)$.)
\end{Cons}

\begin{Thm}[Universal property of block-completion]
\label{Thm:UP-flat}%
Let $\cat{B}$ be a locally \break
idempotent-complete additive bicategory, with all finite direct sums of 0-cells (\Cref{Def:sums_in_bicats}). The bicategory~$\cat{B}^\flat$ of \Cref{Cons:block-completion} is block-complete and the 2-functor $\cat{B}\to \cat{B}^\flat$ induces by pre-composition a biequivalence
\[
\PsFun_\amalg(\cat{B}^\flat,\cat{C})\overset{\sim}{\too}\PsFun_\amalg(\cat{B},\cat{C})
\]
of bicategories of additive pseudo-functors, for every block-complete bicategory~$\cat{C}$.\end{Thm}

\begin{proof}
The pattern of proof is similar to the universal property of the idem\-potent-comp\-le\-tion of additive 1-categories. All unproved claims below are straightforward verifications, most easily done after strictifying~$\cat B$ and~$\cat C$.

Given two idempotents $(X,e\colon \Id_{X}\to \Id_{X})$ and $(X',e'\colon \Id_{X'}\to \Id_{X'})$, one needs to verify that the subcategory $\cat{B}^\flat((X,e),(X',e'))\subseteq \cat{B}(X,X')$ is idempotent complete, for which it suffices to see that it is closed under taking direct summands. This follows easily from the compatibility of the unitors with direct sums (\Cref{Rem:coherent-additivity}) and guarantees that $\cat{B}^\flat$ remains locally idempotent-complete.

The direct sums of objects in~$\cat{B}^\flat$ are directly inherited from those of~$\cat{B}$ as expected: $(X,e)\oplus (X',e')=(X\oplus X',\smat{e&0\\0&e'})$ and $\cat{B}^\flat$ admits block-decompositions because we have equivalences
\[
\xymatrix{
{(X,e)\oplus (X,1-e) \phantom{mmmmm}} \ar@/^3ex/[r]^-{\smat{\img(e)\; & \; \img(1-e)}} \ar@{}[r]|{\simeq} & \ar@/^3ex/[l]^-{\smat{\img(e)\\ \img(1-e)}} (X,1)
}
\]
(use \Cref{Lem:absorption}).
More generally, for an object~$(X,e)$ of~$\cat{B}^\flat$ and an idempotent $f=f^2\colon \Id_{(X,e)}\to \Id_{(X,e)}$ of its identity 1-cell, which itself is the summand~$\img(e)\le\Id_X$, we have $(X,e)\cong (X,f)\oplus (X,e-f)$ in~$\cat{B}^\flat$, with the equivalence similarly given by the 1-cells of~$\cat{B}$ consisting of~$\img(e)$, $\img(f)$ and~$\img(e-f)$. Thus $\cat B^\flat$ is block-complete, as claimed.

Now let $\cat C$ be another block-complete bicategory.
Given a pseudo-functor $\cat{F}\colon \cat{B}\to \cat{C}$ and an idempotent $e=e^2\colon \Id_{X}\to \Id_{X}$ in~$\cat{B}$, we obtain under the isomorphism~$\cat{F}(\Id_X)\cong \Id_{\cat{F}(X)}$ an idempotent~$\cat{F}(e)\colon \Id_{\cat{F}(X)}\to \Id_{\cat{F}(X)}$. Since~$\cat{C}$ is block-complete, we get a decomposition $\cat{F}(X)=Y_1\oplus Y_2$ such that the constructed idempotent is the projection onto~$Y_1$ (on the identity) and one simply sends~$(X,e)$ to that 0-cell~$Y_1$. This construction extends to a pseudo-functor $\hat{\cat{F}}\colon \cat{B}^\flat\to \cat{C}$ which agrees with~$\cat{F}$ on~$\cat{B}$.
It is easy to verify that, up to a unique isomorphism, this is the only way to extend~$\cat{F}\colon \cat{B}\to\cat{C}$ into an additive functor~$\hat{\cat{F}}\colon \cat{B}^\flat\to\cat{C}$.

Consider now two such extensions $\hat{\cat{F}},\hat{\cat{G}}\colon \cat{B}^\flat\to \cat{C}$ of pseudo-functors $\cat F,\cat G\colon \cat B\to \cat C$, and let $t\colon \cat F\Rightarrow \cat G$ be a pseudo-natural transformation. We claim that $t$ extends in a unique way to a pseudo-natural transformation $\hat t\colon \hat{\cat{F}}\Rightarrow\hat{\cat{G}}$.

To see the uniqueness of such a~$\hat t$, we can reason similarly to \Cref{Rem:coherent-additivity}. Let $X= X_1 \oplus X_2$ be a 0-cell direct sum in~$\cat B^\flat$. Then the structure 1-cells $i_k$ and $p_\ell$ (for $k,\ell \in \{1,2\} $) of the direct sum $X_1\oplus X_2$ give rise to diagrams
\[
\vcenter {\hbox{
\xymatrix{
& \cat{F} X_k
\ar@/_12ex/[dd]
 \ar[r]^-{\hat{t}_{X_k}}
 \ar[d]_{\cat{F}i_k} &
\cat{G} X_k
 \ar[d]^{\cat{G}i_k}
 \ar@/^12ex/[dd] & \\
& \cat{F} (X_1\oplus X_2)
\ar[d]_{\cat{F}p_\ell}
\ar[r]^-{\hat{t}_{X_1\oplus X_2}}
 \ar@{}[ur]|{\simeq \; \SWcell\; \hat{t}_{i_k}} &
 \cat{G}(X_1\oplus X_2)
 \ar[d]^{\cat{G}p_\ell} & \\
 \ar@{}[uur]^{\simeq} &
 \cat F X_\ell
 \ar@{}[ur]|{\simeq \; \SWcell\;\hat{t}_{p_\ell}}
 \ar[r]_-{\hat{t}_{X_\ell}} &
 \cat G X_\ell
 \ar@{}[uur]_{\simeq} &
}
}}
\]
where the vertical composites are either (isomorphic to) identity or zero 1-cells. Together with the naturality and functoriality of~$\hat t$, these diagrams show that $\hat{t}_{X_1\oplus X_2}$ has components $\smat{\hat{t}_{X_1} & 0 \\ 0 & \hat{t}_{X_2}}$.
A similar reasoning shows that the 2-cell components $\hat{t}_f$ for $f\colon X_1\oplus X_2\to Y_1\oplus Y_2$ decompose according to the matrix coordinates of~$f$.
In particular, since all objects and 1-cells of $\cat B^\flat$ arise as direct summands of objects and 1-cells of~$\cat B$, it follows that the 1- and 2-cell components of~$\hat t$ are determined by those of~$t$ (together with the chosen direct sum decompositions of the source and target objects $\cat FX$ and~$\cat GX$ for every idempotent $e$ on~$X$). Thus indeed $t$ determines~$\hat{t}$.

For the construction of the extension $\hat t$ from~$t$, we just reason backwards. Namely, for every object $(X,e)$ of $\cat B^\flat$, we set $\hat{t}_{(X,e)}$ to be the 1-cell
\[
\hat{t}_{(X,e)} \;:=\; \cat G(\img(e))\circ t_X\circ \cat F(\img(e))
\]
of~$\cat C$ (\ie the $\cat F(X,e)$-$\cat G(X,e)$-component of $t_X$ with respect to the direct sum decompositions of source and target).
For every 1-cell $f\colon (X,e)\to (X',e')$ of~$\cat B^\flat$ we define $\hat{t}_f$ by the pasting
\[
\xymatrix{
\hat{\cat{F}}(X,e)
 \ar[rrr]^{\hat{t}_{(X,e)}}
 \ar[dr]|{\cat F\img(e)}
 \ar[ddd]_{\cat Ff}
 \ar@{}[drrr]|{\textrm{def.}} &&&
\hat{\cat{G}}(X,e)
 \ar[ddd]^{\cat Gf} \\
& \cat{F}X
 \ar[r]^-{t_X}
 \ar[d]_{\cat Ff}
 \ar@{}[dr]|{\SWcell\; t_f} &
 \cat{G}X
 \ar[ur]|{\cat G \img(e)}
 \ar[d]^{\cat Gf} & \\
& \cat FX'
 \ar[r]_-{t_{X'}}
 \ar[dl]|{\cat F\img(e')} &
\cat{G}X'
 \ar[dr]|{\cat G \img(e')} & \\
\hat{\cat{F}}(X',e')
 \ar[rrr]_-{\hat{t}_{(X',e')}}
 \ar@{}[uuur]|{\cong}
 \ar@{}[urrr]|{\textrm{def.}} &&&
\hat{\cat{G}}(X',e')
 \ar@{}[uuul]|{\cong}
}
\]
where the left and right outmost invertible 2-cells are given by the absorbency of~$f$.

Finally, consider a modification $M=(M_X)_{X\in \cat B_0}\colon t \Rrightarrow s$, where $t,s\colon \cat F\Rightarrow \cat G$ are any two parallel transformations.
For every idempotent $e=e^2\colon \Id_X\to \Id_X$ on~$X$, we define a 2-cell $\hat M_{(X,e)}$ by the following pasting:
\[
\xymatrix{
& \hat{\cat{F}}(X,e)
 \ar[d]|{\cat{F}\img(e)}
 \ar@/_12ex/[ddd]_{\hat{t}_{(X,e)}}
 \ar@/^12ex/[ddd]^{\hat{s}_{(X,e)}} & \\
& \cat{F}X
 \ar@/_5ex/[d]_{t_X}
 \ar@/^5ex/[d]^{s_X}
 \ar@{}[d]|{\overset{\underset{}{M_X}}{\Ecell}}
 & \\
& \cat{G}X
 \ar[d]|{\cat G\img(e)} & \\
& \hat{\cat{G}}(X,e) &
}
\]
The resulting collection $\hat M:= (\hat{M}_{(X,e)}\colon \hat{t}_{(X,e)}\Rightarrow \hat{s}_{(X,e)})_{(X,e)\in \cat B^\flat_0}$ defines a modification $\hat M\colon \hat{t}\Rrightarrow \hat{s}$ extending~$M$ in the unique possible way.
\end{proof}

\begin{Rem} \label{Rem:combined-envelopes}
One can of course combine the Hom-wise group completion and idempotent-completion of \Cref{Rem:add_reflections_bicats} and the block-completion of \Cref{Cons:block-completion}: If $\cat{B}$ is a locally semi-additive bicategory with all finite sums, it can first be made into a locally additive bicategory~$\cat B_+$, which can be made into a locally idempotent-complete one $(\cat{B}_+)^\natural$, which can then be block-completed into $((\cat{B}_+)^\natural)^\flat$. We still get a universal pseudo-functor
\[
\cat B\to \cat{B}_+\hook (\cat{B}_+)^\natural \hook ((\cat{B}_+)^\natural)^\flat=:\cat{B}^\flat
\]
into a block-complete bicategory, and we can still denote the latter by~$\cat{B}^\flat$ and call it the \emph{block-completion} of~$\cat{B}$.
\end{Rem}

\tristars

We end by recording an additive version of the bicategorical Yoneda lemma:

\begin{Rem}
\label{Rem:add-bicat-Yoneda}%
Let $\cat B$ be any locally additive bicategory with all finite direct sums of objects. Then there are covariant and contravariant Yoneda pseudo-functors
\[
\cat B\too \PsFun_\amalg (\cat B^\op, \Add), \quad X \mapsto \cat B(-,X)
\]
and
\[
\cat B^\op \too \PsFun_\amalg (\cat B, \Add), \quad X \mapsto \cat B(X,-)
\]
which are biequivalences on their 1- and 2-full images. If $\cat B$ is locally idempotent-complete, we may of course replace $\Add$ with~$\ICAdd$.
The non-additive version is well-known and is essentially equivalent to the strictification theorem (see \Cref{Rem:strictification}). The additive version is then an easy consequence: Each pseudo-functor $\cat B(-,X)\colon \cat B^\op\to \Cat$ takes values in $\Add$ because $\cat B$ is locally additive, and preserves direct sums simply because they are, in particular, coproducts; and similarly for the dual embedding.

Since in \Cref{sec:additive-motives} we use the latter version, let us be more explicit on the construction of the contravariant Yoneda pseudo-functor. It sends an object $X\in \cat B_0$ to the pseudo-functor $\cat B(X,-)\colon \cat B\to \Add$; a 1-cell $u\colon X\to Y$ to the pseudo-natural transformation $u^*= \cat B(u,-)\colon \cat B(Y,-)\to \cat B(X,-)$ with components $u^*_T=\cat B(u,T)\colon \cat B(Y,T)\to \cat B(X,T)$, $v\mapsto vu$ (for $T\in \cat B_0$) and $u^*_w=\ass^{-1}_{XYTS} \colon w(vu) \Rightarrow (wv)u$ (for $w\colon T\to S$); and a 2-cell $\alpha\colon u\Rightarrow u'$ to the modification $u^* \Rightarrow u'^*$ with components $\cat B(\alpha,T)=(-)\circ \alpha \colon \cat B(u,T)\Rightarrow \cat B(u',T)$ (for $T\in \cat B_0$). Note that this is still \emph{co}variant on 2-cells~$\alpha$!
\end{Rem}

\bigbreak
\chapter{Ordinary Mackey functors on a given group}
\label{app:old-Mackey}
\medskip

In this section, we fix an `ambient group'~$G$. We recall the classical definitions of Mackey functors on~$G$. The first one is the original definition, due to Green~\cite{Green71}.

\begin{Def}
\label{Def:Mackey-functor-on-G}%
\index{Mackey functor!-- on a finite group}%
A \emph{Mackey functor~$M$ on the finite group~$G$} consists of the data of an abelian group $M(H)$ for each subgroup $H\le G$, together with restriction homomorphisms $R^H_K\colon M(H)\to M(K)$ and induction (or transfer) homomorphisms $I_K^H\colon M(K)\to M(H)$ for all~$K\le H$, and conjugation homomorphisms $c_g\colon M(H)\to M({{}^{g\!}}H)$ for all~$g\in G$; this data must satisfy a series of rather obvious compatibilities (with obvious quantifiers):
\begin{enumerate}[(a)]
\item $R^H_H=\id$, $I_H^H=\id$ and $c_h=\id_{M(H)}$ when $h\in H$; (\footnote{\,The latter condition $c_h=\id_{M(H)}$ is not really `obvious'; it is a defining feature related to the fact that $M(H)$ should really only depend on the $G$-set $G/H$. See \Cref{Rem:old-Burnside}.})
\item $c_{gh}=c_gc_h$ and whenever $J\le K\le H$ then $R^H_J=R^K_J\,R^H_K$ and $I_J^H=I_K^H\,I_J^K$;
\item $c_g \, R^H_K=R^{\,{{}^{g\!}}H}_{{}^{g\!}K}\, c_g$ and $c_g \, I^H_K=I^{\,{{}^{g\!}}H}_{{}^{g\!}K}\, c_g$;
\end{enumerate}
as well as the following non-trivial \emph{Mackey (double-coset) formula} for all $K,J\le H$
\begin{equation}
\label{eq:old-Mackey}%
\index{double-coset formula}%
R^H_J\circ I_K^H = \sum_{[x]\in J\bs H/K}I_{J\cap \,{}^{x\!}K}^J \circ c_x \circ R^K_{J^x\cap K}\,.
\end{equation}
See details in Lewis~\cite{Lewis80}, Th\'evenaz-Webb~\cite{ThevenazWebb95}, Bouc~\cite{Bouc97}, or Webb's survey~\cite{Webb00} whose notation we adopted above. A \emph{morphism} of Mackey functors $F\colon M\to M'$ consists of homomorphisms $F(H)\colon M(H)\to M'(H)$ for all~$H\le G$, commuting with restriction, induction and conjugation maps.
\end{Def}

\begin{Rem}\label{Rem:old-Burnside}
There is a well-known `motivic' approach, due to Dress~\cite{Dress73} and Lindner~\cite{Lindner76}, to the above ordinary Mackey functors through the \emph{Burnside category}~$\BurnG$. The objects of the additive category~$\BurnG$ are finite $G$-sets~$X$ and the morphism group $\Hom_{\BurnG}(X,Y)$ is the group-completion of the abelian monoid of isomorphism classes of spans \mbox{$X\lto Z \to Y$} of $G$-maps. Every subgroup $H\le G$ defines a $G$-set $G/H$ and this assignment $H\mapsto G/H$ satisfies the same two variances and compatibilities as a Mackey functor (it defines a Mackey functor with values in~$\BurnG$). The Burnside category~$\BurnG$ is `motivic' in that every Mackey functor factors uniquely through an additive functor from $\BurnG$ to~$\Ab$.

Note that the above assignment $H\mapsto G/H$, from subgroups of~$G$ to objects in~$\BurnG$, factors via the category $G\sset$. Each conjugation $c_g\colon H\to {{}^{g\!}}H$ yields a map of $G$-sets $G/H\to G/\,{{}^{g\!}}H$ given by $[x]_H\mapsto [x g\inv]_{{{}^{g\!}}H}$. In particular $c_h\colon H\to H$ becomes the identity $G/H\to G/H$. This explains $c_h=\id_{M(H)}$ in \Cref{Def:Mackey-functor-on-G}.

In the language of \Cref{sec:ordinary-spans}, we can consider $\widehat{G\sset}$, the span category for the 1-category $G\sset$ of finite $G$-sets. The above Burnside category $\BurnG$ is obtained from $\widehat{G\sset}$ by group-completing the (abelian monoids of) morphisms. Hence there is no difference between additive functors $\BurnG\to \Ab$ and additive functors $\widehat{G\sset}\to \Ab$ where additivity of $F\colon \widehat{G\sset}\to \Ab$ simply means that $F(X\sqcup Y)\cong F(X)\oplus F(Y)$ via the natural map (which implies $F(\varnothing)=0$ as usual).

Consequently, we have an equivalence of categories
\begin{equation}
\label{eq:G-Mackey-span}%
\left\{{\textrm{Mackey functors on~}G
\atop
\textrm{(\Cref{Def:Mackey-functor-on-G})}}\right\}
\overset{\sim}{\longleftarrow}
\left\{{\textrm{additive functors }
\atop
\widehat{G\sset}\to \Ab}\right\}\,.
\end{equation}
\index{Mackey functor!-- in the sense of Dress-Lindner}%
Since $\widehat{G\sset}$ is a semi-additive category (\Cref{Exa:ordinary_span_is_sad}) and $\Ab$ is idempotent complete, in the above equivalence we may also replace the former with its idempotent complete additive envelope $(\widehat{G\sset})_+^\natural$ (\Cref{Rem:completions}).
When we wish to emphasize the above description as additive functors on $\widehat{G\sset}$, or $(\widehat{G\sset})_+^\natural$, we refer to the latter as Mackey functors for~$G$ \emph{in the sense of Dress-Lindner}.
\end{Rem}

We now want to rephrase the above with \emph{groupoids} instead of $G$-sets. This involves a well-known construction:

\begin{Rem} \label{Rem:trans-gpd}
\index{transport groupoid $G\ltimes X$} \index{groupoid!transport --}%
\index{$gx@$G\ltimes X$ \, transport groupoid}%
Recall that the \emph{transport groupoid}, or \emph{translation groupoid}, of a $G$-set $X$ is the groupoid $G\ltimes X$ having $X$ as object-set and with an arrow $(g,x)\colon x\to x'$ (also written $g\colon x\to gx$ for simplicity) for every $g\in G$ such that $gx=x'$. If \mbox{$f\colon X\to Y$} is a $G$-equivariant map, there is an evident (faithful!) functor \mbox{$G\ltimes f\colon $}\mbox{$G\ltimes X\to G\ltimes Y$} defined as $x\mapsto f(x)$ on objects and $(g,x)\mapsto (g,f(x))$ on maps. This defines a faithful strict 2-functor
\[
G\ltimes - \colon G\sset \longrightarrow \groupoidf
\]
from finite $G$-sets to the 2-category of finite groupoids and faithful functors.
An explicit computation (see \eg Ganter~\cite[Prop.\,2.9]{Ganter13pp}) shows that $G\ltimes -$ preserves weak pullbacks, \ie sends pullbacks of $G$-sets to what we call Mackey squares in \Cref{sec:comma} (\ie squares equivalent to iso-comma squares of groupoids).
\end{Rem}

Recall also that, if we let the group~$G$ vary, we can obtain every (finite) groupoid as a (finite) coproduct of such transporter groupoids.
But now we rather want to fix~$G$ and refine the target of the above construction $G\ltimes-$:

\begin{Def}
\label{Def:gpdG}%
\index{groupoid@$\gpdG$}%
\index{$gpdf$@$\gpdG$}%
\index{comma 2-category over a groupoid~$G$}%
For our fixed finite group~$G$, we denote by
\[
\gpdG
\]
the \emph{comma 2-category} of finite groupoids and faithful functors, $\groupoid^\smallfaithful$, over~$G$. By \Cref{Def:2-comma}, its objects are pairs $(H,i_H)$ where $H \in \groupoid_0$ is a finite groupoid and $i_H\colon H\into G$ is a faithful functor in~$\groupoid$. A 1-cell $(H,i_H)\to (H',i_{H'})$ consists of a pair $(u,\theta_u)$ where $u\colon H\into H'$ is a (necessarily faithful) functor and $\theta_u\colon i_{H'}\,u\Rightarrow i_H$ a natural isomorphism.
A 2-cell $(u,\theta_u)\Rightarrow (v,\theta_v)$ is a 2-cell $\alpha\colon u\Rightarrow v$ of $\groupoid$ compatible with $\theta_u$ and $\theta_v$ in that $\theta_v\,(i_{H'}\alpha) = \theta_u$:
\[
\vcenter{\vbox{
\xymatrix@R=10pt{
H
 \ar@/^2ex/[rrd]^-{i_H}
 \ar@/_3ex/[dd]_u
 \ar@/^3ex/[dd]^v
 \ar@{}[dd]|{\overset{\scriptstyle\alpha}\Ecell}
 \ar@{}[ddrr]|{\;\;\;\NEcell\;\theta_v} && \\
&& G \\
H' \ar@/_2ex/[rru]_-{i_{H'}} &&
}
}}
\quad = \quad
\vcenter{\vbox{
\xymatrix@R=10pt{
H
 \ar@/^2ex/[rrd]^-{i_H}
 \ar@/_2ex/[dd]_u
 \ar@{}[ddr]|{\NEcell\;\theta_u} && \\
&& G\,. \\
H' \ar@/_2ex/[rru]_-{i_{H'}} &&
}
}}
\]
The vertical and horizontal compositions of~$\gpdG$ are induced by those of $\groupoid$ in the evident way. There is an obvious forgetful 2-functor $\gpdG\to \groupoidf \subset \groupoid$ which sends $(H,i_H)$ to $H$ and $(u,\theta_u)$ to~$u$. We call $\gpdG$ the \emph{2-category of finite groupoids embedded into~$G$}.
\end{Def}

\begin{Rem}
\label{Rem:trans-gpd-plus}%
The transport groupoid $X\mapsto G\ltimes X$ canonically lifts to~$\gpdG$ if we endow the groupoid $G\ltimes X$ of a $G$-set $X$ with the structure functor
\[
\pi_X\colon G\ltimes X \to G, \quad x \mapsto \bullet, \quad (g,x) \mapsto g
\]
and if we send a $G$-map $f\colon X\to Y$ to the pair $(G\ltimes f, \id)$; this is well-defined because $\pi_X$ is faithful and because $\pi_{Y}\circ (G\ltimes f)$ and $\pi_X$ are equal functors.
\end{Rem}

\begin{Prop} \label{Prop:Gset_vs_gpdG}
The canonical lift $X\mapsto (G\ltimes X,\pi_X)$ of the transport groupoid (as in \Cref{Rem:trans-gpd-plus}) is a biequivalence $G\sset\isoto \gpdG$
\[
\xymatrix{
&& \gpdG \ar[d]^-{\mathrm{forget}} \\
{G\sset} \ar[rru]^-{\simeq} \ar[rr]_-{G\ltimes -} && \groupoid
}
\]
where the 1-category $G\sset$ is viewed as a discrete 2-category. In particular, the Hom categories of $\gpdG$ are all equivalent to discrete categories (sets).
\end{Prop}

\begin{proof}
We must verify that the 2-functor $G\ltimes - \colon G\sset \to \gpdG$ is surjective on objects up to equivalence, and that for every pair of $G$-sets $X,Y$ the functor
$
G\sset (X,Y) \to \gpdG (G\ltimes X , G\ltimes Y)
$
is an equivalence of 1-categories.

For the first point, let $i_H\colon H\into G$ be a faithful functor in~$\groupoid$. Note that $G\ltimes -$ commutes with coproducts, so that we may easily reduce to the case where $H$ is a group, and up to equivalence in~$\gpdG$, we can even assume that $i_H\colon H\into G$ corresponds to the inclusion of a subgroup. We then obtain a commutative triangle
\[
\xymatrix@R=5pt{
H \ar@/^2ex/[rrd]^-{i_H}
 \ar[dd]_u^\simeq && \\
&& G \\
{G \ltimes (G/H)} \ar@/_2ex/[rru]_-{\;\;\;\; \pi_{G/H}} &&
}
\]
where $u$ is the equivalence of groupoids which sends $\bullet\mapsto eH$ on objects and $h\mapsto (h,e H)$ on maps. Therefore $H$ is equivalent in $\gpdG$ to a transport groupoid, as claimed.

For the second point, fix two $G$-sets $X$ and~$Y$. Since the category $G\sset (X,Y)$ is discrete, what we must show is that for every 1-cell $(u,\theta_u)\colon (G\ltimes X,\pi_X)\to (G\ltimes Y, \pi_Y)$ in~$\gpdG$ there exist exactly one $G$-map $f\colon X\to Y$ and one (invertible) 2-cell $\alpha\colon u\Rightarrow G\ltimes f$. To see why this is the case, unfold the definitions: A 1-cell $(u,\theta_u)$ consists precisely of a functor $u\colon G\ltimes X\to G\ltimes Y$, whose components on objects $x\in X$ and maps $(g,x)$ we denote respectively by
\[
x \mapsto u(x)
\quad\quad \textrm{ and } \quad\quad
(g,x)\colon x\to gx \quad \mapsto \quad u(g,x) \colon u(x) \to u(gx) = u(g,x)u(x)
\]
and a natural transformation $\theta_u\colon \pi_Y u \Rightarrow \pi_X$, which is precisely the same as a collection $(\theta_{u,x})_{x\in X}$ of elements $\theta_{u,x}\in G$ such that
\begin{equation} \label{eq:naturality-of-alpha}
 g \cdot \theta_{u,x} = \theta_{u,gx}\cdot u(g,x)
\end{equation}
in $G$ for all $g\in G$ and $x\in X$.
Now define a map $f\colon X\to Y$ by setting $f(x):= \theta_{u,x} \cdot u(x)$. It is $G$-equivariant by~\eqref{eq:naturality-of-alpha}:
\[
f(gx)
= \theta_{u,gx} \cdot u(gx)
= \theta_{u,gx} \cdot u(g,x) \cdot u(x)
= g \cdot \theta_{u,x} \cdot u(x)
= g\cdot f(x)\,.
\]
Moreover the collection $\alpha:= (\theta_{u,x})_{x\in X}$ forms a natural transformation $\alpha \colon u \Rightarrow G\ltimes f$ (again by~\eqref{eq:naturality-of-alpha}) such that $\pi_Y\alpha = \theta_u$ (so it defines a 2-cell in $\gpdG$), and it is immediate to see that these are the unique possible choices for such $f$ and~$\alpha$.
\end{proof}

Recall 1-truncation $\pih{(-)}$ from \Cref{Not:htpy_cat}, which produces a 1-category out of a 2-category by identifying isomorphic 1-cells (and then dropping 2-cells).

\begin{Cor}
\label{Cor:Gset_vs_gpdG}%
The functor $X\mapsto (G\ltimes X,\pi_X)$ of \Cref{Def:gpdG} induces an equivalence of categories
\[
G\sset\isoto \pih{(\gpdG)}\,.
\]
In particular, the 1-truncated category $\pih{(\gpdG)}$ admits pullbacks.
\end{Cor}

\begin{proof}
Apply 1-truncation $\pih{(-)}$ to the biequivalence $G\sset\isoto \gpdG$ of \Cref{Prop:Gset_vs_gpdG}. Here $G\sset$ is a discrete bicategory so $\pih{(G\sset)}=G\sset$.
\end{proof}

\begin{Def}
\label{Def:span-gpdG}%
\index{gpdz@$\spanG$}%
Since the category $\cat{C}:=\pih{(\gpdG)}$ admits pullbacks, let
\[
\spanG:=\widehat{\cat{C}}=\widehat{(\pih{(\gpdG)})}
\]
be the corresponding category of spans in the sense of \Cref{sec:ordinary-spans}. We call $\spanG$ the \emph{category of spans of groupoids faithful over~$G$}.
\end{Def}

Explicitly, an object of $\spanG$ consists of a finite groupoid $H$ with a chosen faithful functor $i_H\colon H\into G$ to~$G$ and morphisms are equivalence classes of `spans'
\[
\xymatrix{H \ar@{ >->}[rd]_-{i_H}
& S _\vcorrect{.1} \ar@{ >->}[d]|-{i_{S_{\vphantom{j}}}} \ar[l]_-{u} \ar[r]^-{v}
 \ar@{}[ld]|(.3){\simeq} \ar@{}[rd]|(.3){\simeq}
& K \ar@{ >->}[ld]^-{i_K}
\\
& G &}
\]
(in which the triangles commute up to isomorphisms), where two such spans are equivalent if there exists an equivalence between the middle objects making everything commute up to isomorphism.

This 1-category of spans $\spanG$ is also the 1-truncation $\pih{\Span(\gpdG,\all)}$ of the bicategory of spans (\Cref{Def:Span-bicat}) for the 2-category~$\GG=\gpdG$ with respect to~$\JJ=\all$ (which are all faithful anyway). We return to this observation in \Cref{Rem:alt_descr_Span}, where we explain it in a more general setting.

\begin{Thm} \label{Thm:Dress_Mackey_via_gpd}
The category of ordinary Mackey functors on~$G$ in the sense of \Cref{Def:Mackey-functor-on-G} is equivalent to the category of additive functors $\spanG\to \Ab$ on the above category $\spanG$ of spans of groupoids faithful over~$G$.
\end{Thm}

\begin{proof}
By \Cref{Cor:Gset_vs_gpdG}, we have an equivalence of categories of spans $\widehat{G\sset}\isoto\spanG$. The result now follows from the Dress-Lindner description~\eqref{eq:G-Mackey-span} of Mackey functors.
\end{proof}

\end{chapter-appendix}
\end{appendix}


\printindex
\end{document}